\documentclass{acta_2022_clean}

\usepackage{macros}

\title[Distributionally Robust Optimization]{Distributionally Robust Optimization}

\author[D.~Kuhn, S.~Shafiee, and W.~Wiesemann]{Daniel Kuhn\\
	Risk Analytics and Optimization Chair,\\ \'Ecole Polytechnique F\'ed\'erale de Lausanne, Lausanne, Switzerland\\
	E-mail: {daniel.kuhn@epfl.ch} \\
	\and
	Soroosh Shafiee\\
	School of Operations Research and Information Engineering,\\ Cornell University, Ithaca, NY, USA\\
	E-mail: {shafiee@cornell.edu} \\
	\and
	Wolfram Wiesemann\\
	Imperial College Business School,\\ Imperial College London, London, United Kingdom\\
	E-mail: {ww@imperial.ac.uk}}

\begin{document}
	
	\label{firstpage}
	\maketitle
	
	\vspace{7.5em}
	
	\begin{abstract}
		Distributionally robust optimization (DRO) studies decision problems under uncertainty where the probability distribution governing the uncertain problem parameters is itself uncertain. A key component of any DRO model is its ambiguity set, that is, a family of probability distributions consistent with any available structural or statistical information. DRO seeks decisions that perform best under the worst distribution in the ambiguity set. This worst case criterion is supported by findings in psychology and neuroscience, which indicate that many decision-makers have a low tolerance for distributional ambiguity. DRO is rooted in statistics, operations research and control theory, and recent research has uncovered its deep connections to regularization techniques and adversarial training in machine learning. This survey presents the key findings of the field in a unified and self-contained manner.
	\end{abstract}
	
	\vspace{-1.25em}
	\tableofcontents 
	
	\section{Introduction}
	\label{sec:intro}
	
	Traditionally, mathematical optimization studies problems of the form
	\begin{align*}
		\inf_{x \in \cX} \; \ell (x),
	\end{align*}
	where a decision $x$ is sought from the set $\cX \subseteq \R^n$ of feasible solutions that minimizes a loss function $\ell : \R^n \rightarrow \overline \R$. With its early roots in the development of calculus by Isaac Newton, Gottfried Wilhelm Leibniz, Pierre de Fermat and others in the late 17th century, mathematical optimization has a rich history that involves contributions from numerous mathematicians, economists, engineers, and scientists. The birth of modern mathematical optimization is commonly credited to George Dantzig, whose simplex algorithm developed in 1947 solves linear optimization problems where $\ell$ is affine and $\cX$ is a polyhedron \citep{dantzig1949programming}. Subsequent milestones include the 
	development of the rich theory of convex analysis \citep{rockafellar1970convex} as well as the discovery of polynomial-time solution methods for linear \citep{khachiyan1979polynomial, karmarkar1984new} and broad classes of nonlinear convex optimization problems \citep{doi:10.1137/1.9781611970791}.
	
	Classical optimization problems are \emph{deterministic}, that is, all problem data are assumed to be known with certainty. However, most decision problems encountered in practice depend on parameters that are corrupted by measurement errors or that are revealed only {\em after} a decision must be determined and committed. A na\"ive approach to model uncertainty-affected decision problems as deterministic optimization problems would be to replace all uncertain parameters with their expected values or with appropriate point predictions. However, it has long been known and well-documented that decision-makers who replace an uncertain parameter of an optimization problem with its mean value fall victim to the `flaw of averages' \citep{savage-2006-probability-management, savage-2012-flaw}. In order to account for uncertainty realizations that deviate from the mean value, \citet{beale1955minimizing} and \citet{dantzig1955linear} independently introduced \emph{stochastic programs} of the form
	\begin{align}\label{eq:sp}
		\inf_{x \in \cX} \; \E_\P \left[ \ell (x, Z) \right],
	\end{align}
	which explicitly model the uncertain problem parameters $Z$ as a random vector that is governed by a probability distribution $\P$, and where a decision is sought that performs best in expectation (or, subsequently, according to some risk measure). 
	Since then, stochastic programming has grown into a mature field \citep{birge2011introduction, shapiro2009lectures}, and it provides the theoretical underpinnings of the empirical risk minimization principle in machine learning \citep{bishop2006pattern, hastie2009elements}.
	
	Despite their success in theory and practice, stochastic programs suffer from at least two shortcomings. Firstly, the assumption that the probability distribution $\P$ is known precisely is unrealistic in many practical settings, and stochastic programs can be sensitive to mis-specifications of this distribution. This effect has been described by different communities as the optimizer's curse \citep{smith2006optimizer}, the error-maximization effect of optimization \citep{michaud1989markowitz, demiguel2009portfolio}, the optimization bias \citep{shapiro2003monte} or overfitting \citep{bishop2006pattern, hastie2009elements}. Secondly, evaluating the expected loss of a fixed decision requires computing a multi-dimensional integral, which is provably hard already for embarrassingly simple loss functions and distributions. Hence, stochastic programs suffer from a curse of dimensionality, that is, their computational complexity generically displays an exponential dependence on the dimension of the random vector~$Z$. To alleviate both shortcomings, \citet{soyster1973convex} proposed to model uncertainty-affected decision problems as \emph{robust optimization problems} of the form
	\begin{align*}
		\inf_{x \in \cX} \; \sup_{z \in \cZ} \, \ell (x, z).
	\end{align*}
	Robust optimization replaces the probabilistic description of the uncertain problem parameters with a set-based description and seeks for decisions that perform best in view of the worst anticipated parameter realization $z$ from within an \emph{uncertainty set} $\cZ$. After an extended period of neglect, the ideas of \citet{soyster1973convex} have been revisited and substantially extended in the late nineties onwards by \citet{kouvelis1997robust, elghaoui1997robust}, \citet{elghaoui1998robust, elghaoui1998least, ben-tal1999robust, ben1998robust, ben1999robust, bertsimas2004price} and others. For reviews of the robust optimization literature, we refer to \citet{ben2009robust}, \citet{rh03} and \citet{bertsimas2022robust}. We point out that similar ideas have been developed independently in the areas of robust stability \citep{horn1985optimal, doyle1989robust, green1995tutorial}, which investigates whether a system remains stable in the face of parameter variations, and robust control \citep{zames1966robust, khalil1996control, zhou1996robust}, which designs systems that maintain a desirable performance in the presence of parameter variations. For textbook introductions to robust stability and control, we refer to \citet{zhou-1999-robust-control} and \citet{dullerud-2001-robust-control}. \citet{hansen2008robustness} adapt robust control techniques to economic problems affected by model uncertainty, where they design policies that perform well across a range of possible model mis-specifications.
	
	While robust optimization reduces the informational and computational burden that plagues stochastic programs, its equal treatment of all parameter realizations within the uncertainty set and its exclusive focus on worst-case scenarios can make it overly conservative for practical applications. These concerns prompted researchers to study \emph{distributionally robust optimization problems} of the form
	\begin{align}
		\label{eq:primal:dro}
		\inf_{x \in \cX} \; \sup_{\P \in \cP} \; \E_\P \left[ \ell (x, Z) \right],
	\end{align}
	which model the uncertain problem parameters $Z$ as a random vector that is governed by some distribution $\P$ from within an \emph{ambiguity set} $\cP$, and where a decision is sought that performs best in view of its expected value under the worst distribution $\P \in \cP$. Distributionally robust optimization (DRO) thus blends the distributional perspective of stochastic programming with the worst-case focus of robust optimization. Herbert E.~Scarf is commonly credited with pioneering this approach in his study on newsvendor problems where the uncertain demand distribution is only characterized through its mean and variance \citep{Scarf:58}. Subsequently, \citet{dupavcova1966minimax, dupacova1987minimax, dupacova1994applications} and \citet{shapiro2002minimax} have studied DRO problems whose ambiguity sets specify the support, some lower-order moments, independence patterns or other structural properties of the unknown probability distribution. \citet{ermoliev1985stochastic} and \citet{gaivoronski1991numerical} have developed early solution approaches for DRO problems over moment ambiguity sets. The advent of modern DRO is often attributed to the works of \citet{bertsimas2002relation, bertsimas2005optimal}, who derive probability inequalities under partial distributional information and apply their techniques to option pricing problems, of \citet{ghaoui2003worst} and \citet{calafiore2006distributionally}, who study DRO problems where a quantile of the objective function should be minimized, or a set of uncertainty-affected constraints should be satisfied with high probability, across all probability distributions with known moment bounds, and of \citet{delage2010distributionally}, who study similar DRO problems with a worst-case expected value~objective.
	
	Early research on DRO has primarily focused on moment ambiguity sets, which contain all distributions on a prescribed support set~$\cZ$ that satisfy finitely many moment constraints. In contrast to stochastic programs, DRO problems with moment ambiguity sets sometimes exhibit favorable scaling with respect to the dimension of the random vector~$Z$. However, strikingly different distributions can share identical moments. As a consequence, moment ambiguity sets always include a wide range of distributions, including some implausible ones that can safely be ruled out when ample historical data is available. This prompted \citet{ben2013robust} and \citet{wang2016likelihood} to introduce ambiguity sets that contain all distributions in some neighborhood of a prescribed reference distribution (typically the empirical distribution that is formed from historical data). These neighborhoods can be defined with respect to a discrepancy function between probability distributions such as a $\phi$-divergence \citep{csiszar1964informationstheoretische} or a Wasserstein distance \citep{villani2008optimal}. Unlike moment ambiguity sets, discrepancy-based ambiguity sets have a tunable size parameter ({\em e.g.}, a radius) and can thus be shrunk to a singleton that contains only the reference distribution. If the reference distribution converges to the unknown true distribution and the size parameter decays to~$0$ as more historical data becomes available, then the DRO problem eventually reduces to the classical stochastic program under the true distribution. Early work on discrepancy-based ambiguity sets relies on the assumption that~$Z$ is a {\em discrete} random vector with a finite support set~$\cZ$. Extensions to discrepancy-based DRO problems with generic (possibly {\em continuous}) random vectors are due to \citet{mohajerin2018data, zhao2018data, blanchet2019quantifying, zhang2022simple} and \citet{gao2016distributionally}, who construct ambiguity sets using optimal transport discrepancies. We refer to \citet{kuhn2019wasserstein} and \citet{rahimian2022frameworks} for prior surveys of the DRO literature.
	
	Historically, the term `distributional robustness' has its roots in robust statistics. The term was coined by \citet{huber1981robust} to describe methods aimed at making robust decisions in the presence of outlier data points. This idea expanded upon earlier works by \citet{box1953non, box1979robustness}, who explores robustness in situations where the underlying distribution deviates from normality, a common assumption underlying many statistical models. To address the challenges posed by outliers, statisticians have developed several contamination models, each offering a unique approach to mitigating data irregularities. The Huber contamination model, introduced by \citet{huber1964robust, huber1968robust} and further developed by \citet{hampel1968contributions, hampel1971general}, assumes that the observed data is drawn from a mixture of the true distribution and an arbitrary contaminating distribution. Neighborhood contamination models define deviations from the true distribution in terms of statistical distances such as the total variation \citep{donoho1988automatic} or Wasserstein distances \citep{zhu2022generalized, liu2023robust}. More recently, data-dependent adaptive contamination models allow for a fraction of the observed data points to be replaced with points drawn from an arbitrary distribution \citep{diakonikolas2019robust, zhu2022generalized}. Interestingly, the optimistic counterpart of a DRO model, which optimizes in view of the best (as opposed to the worst) distribution in the ambiguity set, recovers many estimators from robust statistics \citep{jose2024favor, xie2024favor}. For a survey of recent advances in algorithmic robust statistics we refer to \citet{diakonikolas2023algorithmic}.
	
	Robust and distributionally robust optimization have found manifold applications in machine learning. For example, popular regularizers from the machine learning literature are known to admit a robustness interpretation, which offers theoretical insights into the strong empirical performance of regularization in practice \citep{xu2009regularization, shafieezadeh2019regularization, li2022tikhonov, gao2024wasserstein}. Likewise, optimistic counterparts of DRO models that optimize in view of the best (as opposed to the worst) distribution in the ambiguity set give rise to upper confidence bound algorithms that are ubiquitous in the bandit and reinforcement learning literature \citep{jose2024favor, xie2024favor}. DRO is also related to adversarial training, which aims to improve the generalization performance of a machine learning model by training it in view of adversarial examples \citep{ian15adversarial}. Adversarial examples are perturbations of existing data points that are designed to mislead a model into making incorrect predictions. 
	
	There are also deep connections between DRO and extensions of stochastic (dynamic) programming that replace the expected value with coherent risk measures. Similar to the expected value, a risk measure maps random variables to extended real numbers. In contrast to the expected value, which is risk-neutral since it weighs positive and negative outcomes equally, risk measures most commonly assign greater weights to negative outcomes and thus account for the risk aversion frequently observed among decision-makers. \citet{artzner1999coherent} and \citet{delbaen2002coherent} show that risk measures satisfying the axioms of coherence as well as a Fatou property can be equivalently represented as worst-case expectations over specific sets of distributions. In other words, there is a direct link between optimizing worst-case expectations (as done in DRO) and optimizing coherent risk measures. A similar representation theorem has been developed for a class of nonlinear expectations, the so-called $G$-expectations that are based on the solution of a backward stochastic differential equation, in the financial mathematics literature \citep{peng1997backward, peng2007gB, peng2007g, peng2019nonlinear}. \citet{peng2011gg} shows that sublinear $G$-expectations are equivalent to worst-case expectations over specific families of distributions, thus creating a bridge between the theory of $G$-expectations and~DRO.

	Philosophically, DRO is related to the principle of \emph{ambiguity aversion}, under which individuals prefer known risks over unknown risks even when the unknown risks promise potentially higher rewards. In the economics literature, the distinction between risky outcomes whose probabilities are known and ambiguous outcomes whose probabilities are (partially) unknown goes back to at least \citet{keynes1921treatise} and \citet{knight1921risk}. The concept of ambiguity aversion has been widely popularized through the Ellsberg paradox \citep{ellsberg1961risk}, a thought experiment under which people are asked to choose between betting on an urn with a known distribution of colored balls (\emph{e.g.}, 50 red and 50 blue) and an urn with an unknown distribution of the same colored balls (\emph{i.e.}, the proportion of red to blue is unknown). Despite the potential for equal or better odds, many people prefer to bet on the urn with the known distribution, that is, they display ambiguity aversion. The Ellsberg paradox challenges classical expected utility theory, and it has led to extensions such as the maxmin expected utility theory \citep{gilboa1989maxmin} that serve as theoretical underpinnings of DRO. Ambiguity aversion has subsequently been identified in countless empirical economic studies across financial markets \citep{epstein2003equilibrium, bossaerts2010ambiguity}, insurance markets \citep{cabantous2007ambiguity}, individual decision-making \citep{dimmock2016ambiguity}, macroeconomic policy \citep{HANSEN20101097}, auctions \citep{salo1995auctions} and games of trust \citep{li2019trust}.
	
	
	There is also substantial medical and neuroscientific evidence that supports the presence of ambiguity aversion. \citet{hsu2005neural} found that the amygdala, a key emotional processing center in the brain, becomes more active when individuals are confronted with ambiguity compared to situations with known probabilities, indicating its role in driving ambiguity aversion. A meta-analysis by \citet{krain2006distinct} highlights the involvement of the prefrontal cortex, which is responsible for higher-order cognitive control, rational decision-making, and emotional regulation, in processing ambiguity. In addition, a meta-analysis of \citet{wu2021devil} shows that processing risk and ambiguity both rely on the anterior insula. Risk processing additionally activates the dorsomedial prefrontal cortex and ventral striatum, whereas ambiguity processing specifically engages the dorsolateral prefrontal cortex, inferior parietal lobe, and right anterior insula. This supports the notion that distinct neural mechanisms are engaged when individuals face ambiguous versus risky decisions. Genetic factors may influence an individual's tendency toward ambiguity aversion. \citet{he2010serotonin} link certain genetic polymorphisms to the performance of individuals in decision-making under risk and ambiguity. In a separate study, \citet{buckert14stress} examine how hormonal changes, such as  higher cortisol levels which are linked to stress and anxiety, affect decision-making under risk and ambiguity. These findings collectively suggest that perceptions of risk and ambiguity are not just a cognitive phenomenon but also influenced by brain structures and genetic and hormonal factors that shape individual differences in decision-making under ambiguity. Finally, we mention \citet{hartley2015adolescent} and \citet{blankenstein2016adolescent}, who examine how ambiguity aversion differs between children, adolescents and adults, and \citet{monkeys}, who observed that rhesus macaques monkeys also exhibit ambiguity aversion when offered the choice between risky and ambiguous games of large and small juice outcomes.

	The remainder of this survey is structured as follows. A significant part of our analysis is dedicated to studying the worst-case expectation $\sup_{\P \in \cP} \E_\P \left[ \ell (x, Z) \right]$, which constitutes the objective function of the DRO problem~\eqref{eq:primal:dro}. Evaluating this expression typically requires the solution of a semi-infinite optimization problem over infinitely many variables that characterize the probability distribution $\mathbb{P}$, subject to finitely many constraints imposed by the ambiguity set $\mathcal{P}$. This problem, which we refer to as \emph{nature's subproblem}, is the key feature that distinguishes the DRO problem~\eqref{eq:primal:dro} from deterministic, stochastic, and robust optimization problems. Sections~\ref{sec:ambiguity-sets} and~\ref{sec:topology} review commonly studied ambiguity sets $\mathcal{P}$ and their topological properties, focusing especially on conditions under which nature's subproblem attains its optimal value. Sections~\ref{sec:duality-wc-expectation} and~\ref{sec:duality-wc-risk} develop a duality theory for nature's subproblem that allows us to upper bound or equivalently reformulate the worst-case expectation with a semi-infinite optimization problem over finitely many dual decision variables that are subjected to infinitely many constraints. This duality framework lays the foundations for the analytical solution of nature’s subproblem in Section~\ref{sec:analytical-wc}, which relies on constructing primal and dual feasible solutions that yield the same objective value and thus enjoy strong duality. Sections~\ref{sec:finite-convex-reformulations} and~\ref{sec:approximations-of-nature} leverage the same duality theory to develop equivalent reformulations and conservative approximations of nature's subproblem as well as the overall DRO problem~\eqref{eq:primal:dro}. Section~\ref{sec:numerical-solution-methods} demonstrates how the duality theory gives rise to numerical solution techniques for nature's subproblem and the full DRO problem. Finally, Section~\ref{sec:statistics} reviews the statistical guarantees enjoyed by different ambiguity sets.
	
	Length restrictions dictated difficult trade-offs in the choice of topics covered by this survey. We decided to focus on the most commonly used ambiguity sets and to only briefly review other possible choices, such as marginal ambiguity sets, ambiguity sets with structural constraints (including, \emph{e.g.}, symmetry and unimodality), Sinkhorn ambiguity sets or conditional relative entropy ambiguity sets. Likewise, we do not cover the important but somewhat more advanced topics of distributionally favourable optimization and decision randomization. Finally, we focus on single-stage problems where the uncertainty is fully resolved after the here-and-now decision $x \in \mathcal{X}$ is taken; two-stage and multi-stage DRO problems, where uncertainty unfolds over time and recourse decisions are possible, are reviewed by \citet{DI15:robust_multistage} and \citet{YANIKOGLU2019799}.
	
	
	

	\subsection{Notation} 
	All vector spaces considered in this paper are defined over the real numbers. For brevity, we simply refer to them as `vector spaces' instead of `real vector spaces.' We use $\overline \R = \R \cup \{ -\infty, \infty \}$ to denote the extended reals. The effective domain of a function $f: \R^d \to \overline \R$ is defined as $\dom(f) = \{ z \in \R^d : f(z) < \infty \}$, and the epigraph of $f$ is defined as $\epi(f)=\{(z,\alpha) \in\R^d\times\R :f(z) \leq\alpha\}$. We say that~$f$ is proper if $\dom(f) \neq \emptyset$ and $f(z) > -\infty$ for all $z \in \R^d$. The convex conjugate of $f$ is the function $f^*:\R^d\to\overline\R$ defined through $f^*(y) = \sup_{z \in \R^d} y^\top z - f(z)$. A convex function $f$ is called closed if it is proper and lower semicontinuous or if it is identically equal to~$+\infty$ or to~$-\infty$. One can show that~$f$ is closed if and only if it coincides with its bi-conjugate $f^{**}$, that is, with the conjugate of~$f^*$. If~$f$ is proper, convex and lower semicontinuous, then its recession function $f^\infty: \R^d \to \overline \R$ is defined through $f^\infty(z) = \lim_{\alpha \to \infty} \alpha^{-1}(f(z_0 + \alpha z) - f(z_0))$, where~$z_0$ is any point in~$\dom(f)$~\citep[Theorem~8.5]{rockafellar1970convex}. The perspective of~$f$ is the function $f^\pi: \R^d \times \R \to \overline \R$ defined through $f^\pi(z, t) = t f(z / t)$ if $t > 0$, $f^\pi(z, t) = f^\infty(z)$ if $t = 0$ and $f^\pi(z, t) = \infty$ if $t < 0$. One can show that $f^\pi$ is proper, convex and lower semicontinuous \citep[page~67]{rockafellar1970convex}. When there is no risk of confusion, we occasionally use $t f(z / t)$ to denote $f^\pi(z, t)$ even if $t=0$. The indicator function~$\delta_\cZ:\R^d\to\overline\R$ of a set $\cZ\subseteq\R^d$ is defined through $\delta_\cZ(z)=0$ if $z\in\cZ$ and $\delta_\cZ(z)=\infty$ if $z\notin\cZ$. The conjugate~$\delta^*_\cZ$ of~$\delta_\cZ$ is called the support function of~$\cZ$. Thus, it satisfies $\delta^*_\cZ(y)=\sup_{z\in\cZ} y^\top z$. Random objects are denoted by capital letters ({\em e.g.},~$Z$) and their realizations are denoted by the corresponding lowercase letters ({\em e.g.},~$z$). For any closed set~$\cZ\subseteq\R^d$, we use $\cM(\cZ)$ to denote the space of all finite signed Borel measures on~$\cZ$, while $\cM_+(\cZ)$ stands for the convex cone of all (non-negative) Borel measures in $\cM(\cZ)$, and $\cP(\cZ)$ stands for the convex set of all probability distributions in $\cM_+(\cZ)$. The expectation operator with respect to $\P\in\cP(\cZ)$ is defined through $\E_\P[f(Z)]=\int_\cZ f(z)\,\diff\P(z)$ for any Borel function $f:\cZ\to\overline \R$. If the integrals of the positive and the negative parts of~$f$ both evaluate to~$\infty$, then we define $\E_\P[f(Z)]$ `adversarially.' That is, we set $\E_\P[f(Z)] = \infty$ ($-\infty$) if the integral appears in the objective function of a minimization (maximization) problem. The Dirac probability distribution that assigns unit probability to~$z\in\cZ$ is denoted as~$\delta_z$. The Dirac distribution~$\delta_z$ should not be confused with the indicator function~$\delta_{\{z\}}$ of the singleton $\{z\}$. For any $\P \in\cP(\cZ)$ and any Borel measurable transformation $f:\cZ\to\cZ'$ between Borel sets $\cZ\subseteq \R^d$ and $\cZ'\subseteq \R^{d'}$, we denote by $\P\circ f^{-1}$ the pushforward distribution of~$\P$ under~$f$. Thus, if~$Z$ is a random vector on~$\cZ$ governed by~$\P$, then $f(Z)$ is a random vector on~$\cZ'$ governed by $\P\circ f^{-1}$. The closure, the interior and the relative interior of a set~$\cZ\subseteq \R^d$ are denoted by $\cl(\cZ)$, $\text{int}(\cZ)$ and $\rint(\cZ)$, respectively. We use $\R^d_+$ and~$\R_{++}^d$ to denote the non-negative orthant in~$\R^d$ and its interior. In addition, we use~$\S^d$ to denote the space of all symmetric matrices in~$\R^{d\times d}$. The cone of positive semidefinite matrices in~$\S^d$ is denoted by~$\S_+^d$, and~$\S^d_{++}$ stands for its interior, that is, the set of all positive definite matrices in~$\S^d$. The truth value~$\ds 1_\cE$ of a logical statement evaluates to~$1$ if~$\cE$ is true and to~$0$ otherwise. The set of all natural numbers $\{1,2,3,\ldots\}$ is denoted by~$\N$, and $[n] = \{1, \dots, n\}$ stands for the set of all integers up to $n \in \N$.

	\section{Ambiguity Sets}
	\label{sec:ambiguity-sets}
	An ambiguity set~$\cP$ is a family of probability distributions on a common measurable space. Throughout this paper we assume that~$\cP\subseteq\cP(\cZ)$, where~$\cP(\cZ)$ denotes the entirety of {\em all} Borel probability distributions on a closed set~$\cZ\subseteq\R^d$. This section reviews popular classes of ambiguity sets. For each class, we first give a formal definition and provide historical background information. Subsequently, we exemplify important instances of ambiguity sets and highlight how they are used. 
	
	\subsection{Moment Ambiguity Sets}
	\label{sec:moment-ambiguity-sets}
	A moment ambiguity set is a family of probability distributions that satisfy finitely many (generalized) moment conditions. Formally, it can thus be represented as
	\begin{align}
		\label{eq:moment-ambiguity-set}
		\cP = \left\{ \P \in \cP(\cZ) \, : \, \E_\P \left[ f (Z) \right] \in \cF \right\},
	\end{align}
	where $f: \cZ \to \R^m$ is a Borel measurable moment function, and~$\cF \subseteq \R^m$ is an uncertainty set. By definition, the moment ambiguity set~\eqref{eq:moment-ambiguity-set} thus contains all probability distributions~$\P$ supported on~$\cZ$ whose generalized moments $\E_\P[f(Z)]$ are well-defined and belong to the uncertainty set~$\cF$. Ambiguity sets of the type~\eqref{eq:moment-ambiguity-set} were first studied by \citet{isii1960extrema,isii1962sharpness} and \citet{karlin1966tchebycheff} to establish the sharpness of generalized Chebyshev inequalities. The following subsections review popular instances of the moment ambiguity set.
	
	\subsubsection{Support-Only Ambiguity Sets}
	\label{sec:support}
	The support-only ambiguity set contains all probability distributions supported on $\cZ \subseteq \R^d$, that is, $\cP = \cP(\cZ)$. It can be viewed as an instance of~\eqref{eq:moment-ambiguity-set} with $f(z)=1$ and $\cF=\{1\}$. Any DRO problem with ambiguity set $\cP(\cZ)$ is ostensibly equivalent to a classical robust optimization problem with uncertainty set~$\cZ$, that is,
	\begin{align*}
		\inf_{x \in \cX} ~ \sup_{\P \in \cP(\cZ)} ~ \E_{\P} \left[ \ell(x, Z) \right] = \inf_{x \in \cX} ~ \sup_{z \in \cZ} \ell(x, z).
	\end{align*}
	For a comprehensive review of the theory and applications of robust optimization we refer to \citep{ben1998robust,ben1999robust,ben2000robust,ben2002robust,bertsimas2004price,ben2009robust,bertsimas2011theory,ben2015deriving,bertsimas2022robust}.
	
	If the uncertainty set~$\cZ$ covers a fraction of
	$1-\varepsilon$ of the total probability mass of some distribution~$\P$, then the worst-case loss $\sup_{z\in\cZ}\ell(x,z)$ is guaranteed to exceed the $(1-\varepsilon)$-quantile of~$\ell(x,Z)$ under~$\P$. This can be achieved by leveraging prior structural information or statistical data from~$\P$. For example, $\P(Z\in\cZ)\geq 1-\varepsilon$ may hold (with certainty) if~$\cZ$ is an appropriately sized intersection of halfspaces and ellipsoids and if~$Z$ has independent, symmetric, unimodal and/or sub-Gaussian components under~$\P$ \citep{bertsimas2004price,janak2007new,ben2009robust,li2011comparative,bertsimas2021probabilistic}. Alternatively, it may hold (with high confidence) if~$\cZ$ is constructed from independent samples from~$\P$ by using statistical hypothesis tests \citep{postek2016computationally,bertsimas2018robust,bertsimas2018data}, quantile estimation \citep{hong2021learning}, or learning-based methods \citep{han2021multiple,goerigk2023data,wang2023learning}.

	\subsubsection{Markov Ambiguity Sets}
	\label{sec:Markov}
	Markov's inequality provides an upper bound on the probability that a non-negative univariate random variable~$Z$ with mean~$\mu\geq 0$ exceeds a positive threshold~$\tau>0$. Formally, it states that $\P(Z \geq \tau) \leq \mu/\tau$ for every possible probability distribution of~$Z$ in the ambiguity set $\cP = \{\P \in  \cP(\R_+): \E_\P[Z] = \mu\}$. If $\mu\leq \tau$, then Markov's inequality is sharp, that is, there exists a probability distribution $\P^\star\in\cP$ for which the inequality holds as an equality. Indeed, the distribution $\P^\star = (1-\mu/\tau) \delta_{0} + \mu/\tau \delta_{\tau}$, where $\delta_z$ is the Dirac distribution that places point mass as $z \in \R$, is an element of~$\cP$ and satisfies $\P(Z \geq \tau) =\mu/\tau$. These insights imply that $\sup_{\P\in\cP} \P(Z\geq \tau)=\mu/\tau$ and that the supremum is attained by~$\P^\star$ whenever~$\mu\leq\tau$. Thus, Markov's bound can be interpreted as the optimal value of a DRO problem. It is therefore common to refer to~$\cP$ as a Markov ambiguity set. More generally, we define the Markov ambiguity set corresponding to a closed support set~$\cZ \subseteq \R^d$ and a mean vector~$\mu \in \R^d$ as a family of multivariate distributions of the form
	\begin{align}
		\label{eq:Markov}
		\cP =\left\{ \P\in \cP(\cZ): \E_\P[Z] = \mu \right\}.
	\end{align}
	Thus, the Markov ambiguity set~\eqref{eq:Markov} contains all distributions supported on~$\cZ$ that share the same mean vector~$\mu$. However, these distributions may have dramatically different shapes and higher-order moments. Worst-case expectations over Markov ambiguity sets are sometimes used as efficiently computable upper bounds on the expected cost-to-go functions in stochastic programming. If the cost-to-go functions are concave in the uncertain problem parameters, then these worst-case expectations are closely related to Jensen's inequality \citep{jensen1906fonctions}; see also Section~\ref{sec:jensen}. If the cost-to-go functions are convex and~$\cZ$ is a polyhedron, on the other hand, then these worst-case expectations are related to the Edmundson-Madansky inequality \citep{edmundson:56, madansky:59}; see also Section~\ref{sec:edmundson-madansky}.

	\subsubsection{Chebyshev Ambiguity Sets}
	\label{sec:Chebyshev}
	Chebyshev's inequality provides an upper bound on the probability that a univariate random variable~$Z$ with finite mean~$\mu\in\R$ and variance~$\sigma^2> 0$ deviates from its mean by more than~$k>0$ standard deviations. Formally, it states that $\P\left(|Z - \mu| \geq k \sigma \right) \leq 1/k^2$ for every possible probability distribution of~$Z$ in the ambiguity set $\cP=\{\P \in  \cP(\R): \E_\P[Z] = \mu, ~ \E_\P [Z^2] = \sigma^2 + \mu^2 \}$. Chebyshev's inequality is sharp if~$k \geq 1$. Indeed, one readily verifies that the distribution
	\begin{align*}
		\P^\star = \frac{1}{2k^2} \delta_{\mu - k \sigma} + \left( 1 - \frac{1}{k^2} \right) \delta_{\mu} + \frac{1}{2k^2} \delta_{\mu + k \sigma}
	\end{align*} 
	is an element of~$\cP$ and satisfies $\P(|Z - \mu| \geq k \sigma) = 1/k^2$. These insights imply that $\sup_{\P\in\cP} \P(|Z - \mu| \geq k \sigma) = 1/k^2$ and that the supremum is attained for~$k\geq 1$. Thus, Chebyshev's bound can be interpreted as the optimal value of a DRO problem. It is therefore common to refer to~$\cP$ as a Chebyshev ambiguity set. More generally, we define the Chebyshev ambiguity set corresponding to a closed support set~$\cZ \subseteq \R^d$, mean vector $\mu \in \R^d$ and second-order moment matrix $M \in \S_+^d$, $M\succeq \mu\mu^\top$, as
	\begin{align}
		\label{eq:Chebyshev}
		\cP = \left\{ \P \in \cP(\cZ): \E_\P[Z] = \mu, ~\E_\P[Z Z^\top] = M \right\}.
	\end{align}
	Thus, the Chebyshev ambiguity set~\eqref{eq:Chebyshev} contains all distributions supported on~$\cZ$
	that share the same mean vector~$\mu$ and second-order moment matrix~$M$ (and thus also the same covariance matrix $\Sigma=M-\mu \mu^\top\in \S_+^d$). However, these distributions may have dramatically different shapes and higher-order moments.
	
	The Chebyshev ambiguity set~\eqref{eq:Chebyshev} captures the distributional information relevant for multivariate Chebyshev inequalities  \citep{lal1955note,marshall1960one,tong1980probability,rujeerapaiboon2018chebyshev}. In operations research, Chebyshev ambiguity sets are routinely used since the seminal work of \citet{Scarf:58} on the distributionally robust newsvendor, which is widely perceived as the first paper on DRO. 
	Since then a wealth of DRO models with Chebyshev ambiguity sets have emerged in the context of newsvendor and portfolio selection problems. These models involve a wide range of different decision criteria such as the expected value \citep{gallegomoon:93,natarajan2007mean,popescu2007robust}, the value-at-risk \citep{ghaoui2003worst, xu2012optimization,zymler2013distributionally,zymler2013worst,rujeerapaiboon2016robust,yang2016distributionally,zhang2018ambiguous}, the conditional value-at-risk~\citep{natarajan2010utility, chen2011tight, zymler2013worst,hanasusanto2015distributionally}, spectral risk measures \citep{li2018closed} and distortion risk measures \citep{cai2023distributionally, pesenti2020optimizing}, as well as minimax regret criteria~\citep{yue2006expected,perakis2008regret}. Besides this, Chebyshev ambiguity sets have found numerous applications in option and stock pricing \citep{bertsimas2002relation}, statistics and machine learning \citep{lanckriet2001minimax,lanckriet2002robust,strohmann2002formulation,huang2004minimum,bhattacharyya2004second,farnia2016minimax,nguyen2019optimistic,rontsis2020distributionally}, stochastic programming \citep{birgewets:86,dula1992tchebysheff,dokov2005second,bertsimas2010models,natarajan2011mixed}, control \citep{van2015distributionally,yang2018dynamic,xin2021time,xin2022distributionally}, the operation of power systems \citep{xie2017distributionally,zhao2017distributionally}, complex network analysis \citep{van2021robust,brugman2022sharpest}, queuing systems \citep{van2023second}, healthcare \citep{mak2015appointment,shehadeh2020distributionally}, and extreme event analysis \citep{lam2017tail}, among others.
	
	\subsubsection{Chebyshev Ambiguity Sets with Uncertain Moments}
	\label{sec:chebyshev-with-moment-uncertainty}
	Working with Chebyshev ambiguity sets is appropriate when the first- and second-order moments of~$\P$ are known, while all higher-order moments are unknown. In practice, however, even the first- and second-order moments are never known with absolute certainty. Instead, they must be estimated from statistical data and are thus subject to estimation errors. This prompted \citet{ghaoui2003worst} to introduce a Chebyshev ambiguity set with uncertain moments, which can be represented as
	\begin{align}
		\label{eq:chebyshev-with-moments-in-F}
		\cP & = \left\{ \P \in \cP(\cZ): \left(\E_\P[Z], \E_\P[Z Z^\top] \right)\in\cF \right\}.
	\end{align}
	Here, $\cF \subseteq \R^d \times \S_+^d$ is a convex set that captures the moment uncertainty. Clearly, $\cP$ can be expressed as a union of crisp Chebyshev ambiguity sets, that is, we have
	\begin{align*}
		\cP  = \bigcup_{(\mu,M)\in\cF}\left\{ \P \in \cP(\cZ): \E_\P[Z] = \mu, ~\E_\P[Z Z^\top] = M \right\}.
	\end{align*}
	Note that the Chebyshev ambiguity set with uncertain moments encapsulates the support-only ambiguity set, the Markov ambiguity set, and the Chebyshev ambiguity set as special cases. They are recovered by setting $\cF=\R^d\times\S^d_+, \cF=\{\mu\}\times \S^d_+$, and $\cF=\{\mu\}\times\{M\}$, respectively. 
	
	\citet{ghaoui2003worst} capture the uncertainty in the moments using the box
	\begin{align*}
		\cF = \left\{(\mu, M) \in \R^d \times \S_+^d: \underline{\mu} \leq \mu \leq \overline{\mu}, ~ \underline{M} \preceq M \preceq \overline{M} \right\}
	\end{align*}
	parametrized by the moment bounds $\underline{\mu}, \overline{\mu}\in \R^d$ and $\underline{M}, \overline{M} \in \S_+^d$. 
	
	Given noisy estimates~$\hat\mu$ and~$\hat\Sigma$ for the unknown mean vector and covariance matrix of~$\P$, respectively, \citet{delage2010distributionally} propose the ambiguity set
	\begin{align*}
		\cP  = \left\{ \P \in \cP(\cZ)\;: \!\begin{array}{l}
			(\E_\P[Z]-\hat\mu)^\top\hat\Sigma^{-1} (\E_\P[Z]-\hat\mu)\leq\gamma_1\\
			\E_\P[(Z-\hat\mu)(Z-\hat\mu)^\top ]\preceq\gamma_2 \hat\Sigma
		\end{array} \right\}.
	\end{align*}
	By construction, $\cP$ contains all distributions on~$\cZ$ whose first-order moments reside in an ellipsoid with center~$\hat\mu$ and whose second-order moments (relative to~$\hat\mu$) reside in a semidefinite cone with apex~$\gamma_2\hat\Sigma$. An elementary calculation reveals that
	\[
	\E_\P[(Z-\hat\mu)(Z-\hat \mu)^\top] = \E_\P[ ZZ^\top] -\E_\P[Z]\hat\mu^\top-\hat\mu \E_\P[Z]^\top +\hat \mu\hat\mu^\top.
	\]
	Thus, $\cP$ can be viewed as a Chebyshev ambiguity set with uncertain moments. Indeed, $\cP$ is an instance of~\eqref{eq:chebyshev-with-moments-in-F} if we define the moment uncertainty set as
	\begin{align*}
		\cF = \left\{(\mu, M) \in \R^d \times \S_+^d\;:\! \begin{array}{l}
			(\mu - \hat \mu)^\top \hat \Sigma (\mu - \hat \mu) \leq \gamma_1 \\
			M - \mu \hat \mu^\top- \hat \mu \mu^\top + \hat \mu \hat \mu^\top \preceq \gamma_2 \hat \Sigma 
		\end{array} \right\}.
	\end{align*}
	\citet{delage2010distributionally} show that if~$\hat\mu$ and~$\hat\Sigma$ are set to the sample mean and the sample covariance matrix constructed from a finite number of independent samples from~$\P$, respectively, then one can tune the size parameters~$\gamma_1\geq 0$ and $\gamma_2\geq 1$ to ensure that~$\P$ belongs to~$\cP$ with any desired confidence. 
	
	Chebyshev as well as Markov ambiguity sets with uncertain moments have found various applications ranging from control \citep{nakao2021distributionally} to integer stochastic programming \citep{bertsimas2004probabilistic,cheng2014distributionally}, portfolio optimization \citep{natarajan2010utility}, extreme event analysis \citep{bai2023distributionally} and mechanism design and pricing \citep{bergemann2008pricing,bandi2014optimal,koccyiugit2020distributionally,koccyiugit2021robust,chen2023screening,bayrak2022distributionally,anunrojwong2024robustness}, among many others.
	
	The uncertainty set~$\cF$ for the first- and second-order moments of~$\P$ often corresponds to a neighborhood of a nominal mean-covariance pair $(\hat\mu, \hat\Sigma)$ with respect to some measure of discrepancy. For example, matrix norms such as the Frobenius norm, the spectral norm or the nuclear norm \citep[\ts~9]{bernstein2009matrix} provide natural measures to quantify the dissimilarity of covariance matrices. The discrepancy between two mean-covariance pairs $(\mu, \Sigma)$ and $(\hat\mu, \hat\Sigma)$ can also be defined as the discrepancy between the normal distributions $\cN(\mu, \Sigma)$ and $\cN(\hat\mu, \hat\Sigma)$ with respect to a probability metric or an information-theoretic divergence such as the Kullback-Leibler divergence \citep{kullback1959information}, the Fisher-Rao distance \citep{atkinson1981rao} or other spectral divergences \citep{zorzi2014multivariate}.
	
	As we will discuss in more detail in Section~\ref{sec:optimal:transport}, the 2-Wasserstein distance between two normal distributions $\cN(\mu, \Sigma)$ and $\cN(\hat\mu, \hat\Sigma)$ coincides with the Gelbrich distance between the underlying mean-covariance pairs $(\mu, \Sigma)$ and $(\hat\mu, \hat\Sigma)$. In the following, we first provide a formal definition of the Gelbrich distance and then exemplify how it can be used to define a moment uncertainty set~$\cF$.
	
	
	\begin{definition}[Gelbrich Distance]
		\label{def:Gelbrich}
		The Gelbrich distance between two mean-covariance pairs~$(\mu, \Sigma)$ and $(\hat \mu, \hat \Sigma)$ in~$\R^d \times \S_+^d$ is given by
		\begin{align*}
			\G \big( (\mu, \Sigma), (\hat \mu, \hat \Sigma) \big) = \sqrt{\| \mu - \hat \mu \|_2^2 + \Tr \Big( \Sigma + \hat \Sigma - 2 \big( \hat \Sigma^\half \Sigma \hat \Sigma^\half \big)^\half \Big)}.
		\end{align*}
	\end{definition}
	The Gelbrich distance is non-negative, symmetric and subadditive, and it vanishes if and only if $(\mu, \Sigma) = (\hat \mu, \hat \Sigma)$. Thus, it represents a metric on $\R^d \times \S_+^d$ \citep[pp.~239]{givens1984class}. 
	When $\mu = \hat \mu$, then the Gelbrich distance collapses to the Bures distance between~$\Sigma$ and~$\hat\Sigma$, which was conceived as a measure of dissimilarity between density matrices in quantum information theory. The Bures distance is known to induce a Riemannian metric on the space of positive semidefinite matrices~\citep{bhatia2018strong, bhatia2019bures}. When~$\Sigma$ and~$\hat\Sigma$ are simultaneously diagonalizable, then their Bures distance coincides with the Hellinger distance between their spectra. The Hellinger distance is closely related to the Fisher-Rao metric ubiquitous in information theory~\citep{liese1987convex}. Even though the Gelbrich distance is nonconvex, the {\em squared} Gelbrich distance is jointly convex in both of its arguments. This is an immediate consequence of the following proposition, which can be found in \citep{olkin1982distance,dowson1982frechet,givens1984class,panaretos2020invitation}.
	
	\begin{proposition}[SDP Representation of the Gelbrich Distance]
		\label{prop:Gelbrich:SDP}
		For any mean-covariance pairs $(\mu, \Sigma)$ and $(\hat \mu, \hat \Sigma)$ in~$\R^d \times \S_+^d$, we have
		\begin{align} 
			\label{eq:Gelbrich:SDP}
			\G^2 \big( (\mu, \Sigma), (\hat \mu, \hat \Sigma) \big) = 
			\left\{
			\begin{array}{cl}
				\displaystyle \min_{C \in \R^{d \times d}} & \| \mu - \hat \mu \|_2^2 + \Tr (\Sigma + \hat \Sigma - 2C) \\[1ex]
				\st & \begin{bmatrix} \Sigma & C \\ C^\top & \hat \Sigma \end{bmatrix} \succeq 0. 
			\end{array}
			\right.
		\end{align} 
	\end{proposition}
	\begin{proof}
		Throughout the proof we keep $\mu$, $\hat\mu$ and $\Sigma$ fixed and treat~$\hat\Sigma$ as a parameter. We also use $f(\hat\Sigma)$ as a shorthand for the left hand side of~\eqref{eq:Gelbrich:SDP} and $g(\hat \Sigma)$ as a shorthand for the right hand side of~\eqref{eq:Gelbrich:SDP}. Elementary manipulations show that
		\begin{align}
			\label{eq:max:C}
			g(\hat\Sigma) = \|\mu - \hat \mu \|_2^2 + \Tr(\Sigma + \hat \Sigma) - \left\{\begin{array}{cl}
				\displaystyle \max_{C \in \R^{d \times d}} & \Tr (2C) \\[1ex]
				\st & \begin{bmatrix} \Sigma & C \\ C^\top & \hat \Sigma \end{bmatrix} \succeq 0.
			\end{array}\right.
		\end{align}
		The maximization problem in~\eqref{eq:max:C} is dual to the following minimization problem.
		\begin{align*}
			\begin{array}{cl}
				\displaystyle \inf_{A_{11}, A_{22}\in\S^d} & \Tr (A_{11} \Sigma) + \Tr(A_{22} \hat \Sigma) \\[1ex]
				\st & \begin{bmatrix} A_{11} & I_d \\ I_d & A_{22} \end{bmatrix} \succeq 0
			\end{array}
		\end{align*}
		Strong duality holds because $A_{11} = A_{22} = 2I_d$ constitutes a Slater point for the dual problem \cite[Theorem~2.4.1]{ben2001lectures}. The existence of a Slater point further implies that the primal maximization problem in~\eqref{eq:max:C} as well as the minimization problem in~\eqref{eq:Gelbrich:SDP} are solvable. By \citep[Corollary~8.2.2]{bernstein2009matrix}, both $A_{11}$ and $A_{22}$ must be positive definite in order to be dual feasible. Thus, they are invertible. We can therefore employ a Schur complement argument \citep[Lemma~4.2.1]{ben2001lectures} to simplify the dual problem to
		\begin{align}
			\label{eq:min:A22}
			\inf_{A_{11} \succeq A_{22}^{-1} \succ 0} \, \Tr(A_{11} \Sigma) + \Tr(A_{22} \hat \Sigma) = \inf_{A_{22} \succ 0} \, \Tr(A_{22}^{-1} \Sigma) + \Tr(A_{22} \hat \Sigma),
		\end{align}
		where the equality holds because $\Sigma \succeq 0$. The optimal value of the resulting minimization problem is concave and upper semicontinuous in~$\hat\Sigma$ because it constitutes a pointwise infimum of affine functions of~$\hat\Sigma$. Thus, $g(\hat\Sigma)$ is convex and lower semicontinuous. We now show that if $\hat \Sigma \succ 0$, then the convex minimization problem over~$A_{22}$ in~\eqref{eq:min:A22} can be solved in closed form. To this end, we construct a positive definite matrix $A^\star_{22}$ that satisfies the problem's first-order optimality condition
		\begin{align*}
			\hat \Sigma - A_{22}^{-1} \Sigma A_{22}^{-1} = 0 \quad \iff \quad A_{22} \hat \Sigma A_{22} - \Sigma = 0.
		\end{align*}
		Indeed, multiplying the quadratic equation on the right from both sides with $\hat \Sigma^\half$ yields the equivalent equation $(\hat \Sigma^\half A_{22} \hat \Sigma^\half)^2 = \hat \Sigma^\half \Sigma \hat \Sigma^\half$. As $\hat \Sigma \succ 0$, this equation is uniquely solved by $A^\star_{22} = \hat \Sigma^{-\half}( \hat \Sigma^\half \Sigma \hat \Sigma^\half )^\half \hat \Sigma^{-\half}$. Substituting~$A^\star_{22}$ into~\eqref{eq:min:A22} reveals that the optimal value of the dual minimization problem is given by $2 \Tr(( \hat \Sigma^\half \Sigma \hat \Sigma^\half)^\half)$. Substituting this value into~\eqref{eq:max:C} then shows that $g(\hat\Sigma)=f(\hat \Sigma)$ whenever $\hat\Sigma\succ 0$.
		
		It remains to be shown that $g(\hat\Sigma)=f(\hat \Sigma)$ if $\hat\Sigma$ is singular. To this end, we recall from \cite[Lemma~A.2]{nguyen2019bridging} that the matrix square root is continuous on~$\S^d_+$, which implies that~$f(\hat \Sigma)$ is continuous on~$\S^d_+$. For any singular~$\hat\Sigma\succeq 0$, we thus have
		\[
		f(\hat\Sigma)=\liminf_{\hat\Sigma'\to\hat\Sigma,\,\hat\Sigma'\succ 0} f(\hat\Sigma') = \liminf_{\hat\Sigma'\to\hat\Sigma,\,\hat\Sigma'\succ 0} g(\hat\Sigma') = g(\hat\Sigma).
		\]
		Here, the first equality exploits the continuity of~$f$, and the second equality holds because $f(\hat\Sigma') =g(\hat\Sigma')$ for every~$\hat\Sigma'\succ 0$. The third equality follows from the convexity and lower semicontinuity of~$g$, which imply that the limit inferior can neither be larger nor smaller than~$g(\hat \Sigma)$, respectively. This completes the proof.
	\end{proof}
	
	Proposition~\ref{prop:Gelbrich:SDP} shows that the squared Gelbrich distance coincides with the optimal value of a tractable semidefinite program. This makes the Gelbrich distance attractive for computation. As a byproduct, the proof of Proposition~\ref{prop:Gelbrich:SDP} reveals that the squared Gelbrich distance is convex as well as continuous on its domain. 
	
	Following \citet{nguyen2021mean}, we can now introduce the Gelbrich ambiguity set as an instance of the Chebyshev ambiguity set~\eqref{eq:chebyshev-with-moments-in-F} with uncertain moments. The corresponding moment uncertainty set is given by
	\begin{align}
		\label{eq:Gelbrich:F}
		\cF = \left\{(\mu, M) \in \R^d \times \S_+^d~:\, \begin{array}{l}
			\exists \Sigma\,\in\S^d_+ \text{ with } M = \Sigma + \mu \mu^\top, \\ \G \left( (\mu, \Sigma), (\hat \mu, \hat \Sigma) \right) \leq r 
		\end{array}
		\right\},
	\end{align}
	where~$(\hat \mu, \hat \Sigma)$ is a nominal mean-covariance pair, and the radius~$r \geq 0$ serves as a tunable size parameter. Below we refer to~$\cF$ as the Gelbrich uncertainty set. The next proposition establishes basic topological and computational properties of~$\cF$.
	
	\begin{proposition}[Gelbrich Uncertainty Set]
		\label{prop:U:Gelbrich:SDP}
		The uncertainty set $\cF$ defined in \eqref{eq:Gelbrich:F} is convex and compact. In addition, it admits the semidefinite representation
		\begin{align*}
			\cF = \left\{(\mu, M) \in \R^d \times \S_+^d~:\, \begin{array}{l}
				\exists C\,\in\R^{d\times d},~U\in\S^d_+ \text{ with } \\[0.5ex]
				\| \hat \mu\|_2^2 - 2 \mu^\top \hat \mu + \Tr (M + \hat \Sigma - 2C) \leq r^2, \\[1ex] 
				\begin{bmatrix} M - U & C \\ C^\top & \hat \Sigma \end{bmatrix} \succeq 0, ~ \begin{bmatrix} U & \mu \\ \mu^\top & 1 \end{bmatrix} \succeq 0
			\end{array}
			\right\}.
		\end{align*}
	\end{proposition}
	
	\begin{proof}
		The proof exploits the semidefinite representation of the squared Gelbrich distance established in Proposition~\ref{prop:Gelbrich:SDP}. Note first that if $M = \Sigma + \mu \mu^\top$, then
		\begin{align*}
			\| \mu - \hat \mu \|_2^2 + \Tr (\Sigma + \hat \Sigma - 2C) = \| \hat \mu \|_2^2 - 2 \mu^\top \hat \mu + \Tr (M + \hat \Sigma - 2C).
		\end{align*}
		By Proposition~\ref{prop:Gelbrich:SDP}, the Gelbrich uncertainty set $\cF$ can thus be represented as
		\begin{align*}
			\cF = \left\{(\mu, M) \in \R^d \times \S_+^d~:\,
			\begin{array}{l}
				\exists C\,\in\R^{d\times d} \text{ with } \\
				\| \hat \mu \|_2^2 - 2 \mu^\top \hat \mu + \Tr (M + \hat \Sigma - 2C) \leq r^2, \\[1ex] 
				\begin{bmatrix} M - \mu \mu^\top & C \\ C^\top & \hat \Sigma \end{bmatrix} \succeq 0
			\end{array}
			\right\}.
		\end{align*}
		A standard Schur complement argument further reveals that
		\begin{align*}
			\begin{bmatrix} M - \mu \mu^\top & C \\ C^\top & \hat \Sigma \end{bmatrix} \succeq 0 \iff \exists U \in \S_+^d \text{ with } \begin{bmatrix} M - U & C \\ C^\top & \hat \Sigma \end{bmatrix} \succeq 0, ~ \begin{bmatrix} U & \mu \\ \mu^\top & 1 \end{bmatrix} \succeq 0.
		\end{align*}
		Hence, the Gelbrich uncertainty set~$\cF$ admits the semidefinite representation given in the proposition statement. Convexity is evident from this representation, which expresses $\cF$ as the projection of a set defined by conic inequalities in a lifted space.
		
		It remains to be shown that~$\cF$ is compact. To this end, we define
		\begin{align*}
			\cV = \left\{ (\mu, \Sigma) \in \R^d \times \S_+^d: \G \left( (\mu, \Sigma), (\hat \mu, \hat \Sigma) \right) \leq r \right\}
		\end{align*}
		as the ball of radius~$r$ around~$(\hat \mu, \hat \Sigma)$ with respect to the Gelbrich distance. Note that~$\cF = f(\cV)$, where the transformation $f:\R^d\times\S^d_+\to \R^d\times \S^d_+$ is defined through $f(\mu,\Sigma)=(\mu,\Sigma+\mu\mu^\top)$. We will now prove that~$\cV$ is compact. As~$f$ is continuous and as compactness is preserved under continuous transformations, this will readily imply that~$\cF$ is compact. Clearly, $\cV$ is closed because the Gelbrich distance is continuous. To show that~$\cV$ is also bounded, fix any~$(\mu, \Sigma)\in\cV$. By the definition of the Gelbrich distance, we have $\| \mu - \hat \mu \| \leq r^2$. In addition, we find
		\begin{align*}
			\Tr \Big( \big( \hat \Sigma^\half \Sigma \hat \Sigma^\half \big)^\half \Big)
			& =  \max_{C \in \R^{d \times d}} \left\{ \Tr (C) : \begin{bmatrix} \Sigma & C \\ C^\top & \hat \Sigma \end{bmatrix} \succeq 0 \right\}\\
			& \leq \max_{C \in \R^{d \times d}} \left\{\Tr(C):\displaystyle C_{ij}^2 \leq \Sigma_{ii} \hat \Sigma_{jj} ~ \forall i,j \in [d] \right\} \leq \sqrt{\Tr(\Sigma) \Tr(\hat \Sigma)}\,,
		\end{align*}
		where the equality has been established in the proof of Proposition~\ref{prop:Gelbrich:SDP}. The two inequalities follow from a relaxation of the linear matrix inequality, which exploits the observation that all second principal minors of a positive semidefinite matrix are non-negative, and from the Cauchy-Schwarz inequality. Thus, $\Sigma$ satisfies
		\begin{equation*}
			\left( \Tr(\Sigma)^\half - \Tr(\hat \Sigma)^\half \, \right)^2 \leq \Tr \left( \Sigma + \hat \Sigma - 2 \big( \hat \Sigma^\half \Sigma \hat \Sigma^\half \big)^\half) \right) \leq r^2 ,
		\end{equation*}
		where the second inequality holds because $(\mu,\Sigma)\in\cV$. We may therefore conclude that $\Tr(\Sigma) \leq (r + \Tr(\hat \Sigma)^\half )^2$, which in turn implies that $0 \preceq \Sigma \preceq (r + (\Tr(\hat \Sigma))^\half )^2 I_d$. In summary, we have shown that both~$\mu$ and~$\Sigma$ belong to bounded sets. As $(\mu,\Sigma)\in\cV$ was chosen arbitrarily, this proves that~$\cV$ is indeed bounded and thus compact. 
	\end{proof}
	
	Proposition~\ref{prop:Gelbrich:SDP} shows that the uncertainty set~$\cF$ is convex. This is surprising because $\cF=f(\cV)$, where the Gelbrich ball~$\cV$ in the space of mean-covariance pairs is convex thanks to Proposition~\ref{prop:Gelbrich:SDP} and where~$f$ is a {\em quadratic} bijection. Indeed, convexity is usually only preserved under {\em affine} transformations.
	
	Gelbrich ambiguity sets were introduced by \citet{nguyen2021mean} in the context of robust portfolio optimization. They have also found use in machine learning \citep{bui2021counterfactual,vu2021distributionally,nguyen2022distributionally}, estimation \citep{nguyen2019bridging}, filtering \citep{shafieezadeh2018wasserstein,kargin2024distributionally} and control \citep{mcallister2023distributionally,al2023distributionally,hajar2023wasserstein,hakobyan2024wasserstein,taskesen2024distributionally,kargin2024lqr,kargin2024infinite,kargin2024wasserstein}.

	\subsubsection{Mean-Dispersion Ambiguity Sets}
	\label{sec:mean-dispersion}
	
	
	If~$\cK\subseteq \R^k$ is a proper convex cone and $v_1,v_2\in\R^k$, then the inequality $v_1\preceq_\cK v_2$ means that $v_2-v_1\in\cK$. Also, a function $G:\R^m\to\R^k$ is called $\cK$-convex if 
	\[
	G(\theta v_1+(1-\theta)v_2)\preceq_\cK \theta G(v_1)+(1-\theta)G(v_2) \quad \forall v_1,v_2\in\R^m, ~ \forall \theta \in [0, 1].
	\]
	The mean-dispersion ambiguity set corresponding to a convex closed support set~$\cZ\subseteq\R^d$, a mean vector~$\mu\in\R^d$, a $\cK$-convex dispersion function $G:\R^m\to\R^k$ and a dispersion bound $g\in\R^k$ is defined as
	\begin{align}
		\label{eq:mean-dispersion-ambiguity-set}
		\cP = \left\{ \P \in \cP(\cZ) \, : \, \E_\P[Z]=\mu,~\E_\P[G(Z) ] \preceq_\cK g \right\}.
	\end{align}
	The mean-dispersion ambiguity set is highly expressive, that is, it can model various stylized features of the unknown probability distribution. For example, if~$\|\cdot\|$ is a norm on~$\R^d$, $G(z)=\|z-\mu\|$ is convex in the usual sense, and $g=\sigma\in\R_+$, then all distributions~$\P\in\cP$ have a mean absolute deviation from the mean that is bounded by~$\sigma$. Alternatively, if $G(z)= (z-\mu)(z-\mu)^\top$ is $\S_+^d$-convex and $g=\Sigma\in\S_+^d$, then~$\cP$ reduces to a Chebyshev ambiguity set with moment uncertainty. Specifically, the covariance matrix of any~$\P\in\cP$ is bounded by~$\Sigma$ in Loewner order. \citet{wiesemann2014distributionally} show that the ambiguity set~$\cP$, which contains distributions of the $d$-dimensional random vector~$Z$, is closely related to the lifted ambiguity set
	\begin{equation*}
		\cQ = \left\{ \Q \in \cP(\cC) \, : \, \E_\Q[ Z ] = \mu,~\E_\Q[ U ] = g \right\}
	\end{equation*}
	with support set $\cC=\{(z,u)\in \cZ\times \R^k: G(z)\preceq_\cK u\}$, which contains {\em joint} distributions of~$Z$ and an auxiliary $m$-dimensional random vector~$U$. Indeed, one can prove that~$\cP=\{\Q_Z:\Q\in\cQ\}$, where~$\Q_Z$ denotes the marginal distribution of~$Z$ under~$\Q$. As the loss function depends only on~$Z$ but not on~$U$, this reasoning implies that the inner worst-case expectation problem in~\eqref{eq:primal:dro} satisfies 
	\begin{equation*}
		\sup_{\P \in \cP} \; \E_\P \left[ \ell (x, Z) \right] = \sup_{\Q \in \cQ} \; \E_\Q \left[ \ell (x, Z)\right] .
	\end{equation*}
	Hence, one can replace the original ambiguity set~$\cP$ with the lifted ambiguity set~$\cQ$. This is useful because~$\cQ$ constitutes a simple Markov ambiguity set that specifies only the support set~$\cC$ and the mean~$(\mu, g)$ of the joint random vector~$(Z,U)$. In addition, one can show that~$\cZ$ is convex because~$\cZ$ is convex and~$G$ is $\cK$-convex. In summary, DRO problems with mean-dispersion ambiguity sets of the form~\eqref{eq:mean-dispersion-ambiguity-set} can systematically be reduced to DRO problems with Markov ambiguity sets.

	A more general class of mean-dispersion ambiguity sets can be used to shape the moment generating function of~$Z$ under~$\P$. Specifically, \citet{chen2019distributionally} introduce the \emph{entropic dominance} ambiguity set 
	\begin{align*}
		\cP = \left\{ \P \in \cP(\cZ): \E_\P [Z] = \mu, ~ \log \left( \E_\P [\exp(\theta^\top (Z - \mu))] \right) \leq g(\theta) ~~ \forall \theta \in \R^d \right\},
	\end{align*}
	where $g: \R^d \to \R$ is a convex and twice continuously differentiable function satisfying $g(0) = 0$ and $\nabla g(0) = 0$. The constraints parametrized by~$\theta$ impose a continuum of upper bounds on the cumulant generating function (that is, the logarithmic moment generating function) of the centered random variable $Z - \mu$ under~$\P$. The choice of~$g$ determines the specific class of distributions included in the ambiguity set. For example, if $g(\theta) = \sigma^2 \theta^\top \theta/2$ for some~$\sigma > 0$, then the ambiguity set contains only sub-Gaussian distributions with variance proxy~$\sigma^2$. Sub-Gaussian distributions are probability distributions whose tails are bounded by the tails of a Gaussian distribution. They play a significant role in probability theory and statistics, particularly in the study of concentration inequalities and high-dimensional phenomena \citep{vershynin2018high,wainwright2019high}.
	
	The entropic dominance ambiguity set imposes infinitely many  constraints on~$\P$. \citet{chen2019distributionally} show that worst-case expectation problems over this ambiguity set can be reformulated as semi-infinite conic programs. They propose a cutting plane algorithm to solve these conic programs efficiently. The entropic dominance ambiguity set has also found applications in the study of nonlinear and PDE-constrained DRO problems \citep{milz2020approximation,milz2022approximation}. Generalized entropic dominance ambiguity sets are considered by \citet{chen2023robust}.

	\subsubsection{Higher-order Moment Ambiguity Sets}
	Markov and Chebyshev ambiguity sets only impose conditions on the first- and/or second-order moments of~$\P$. DRO problems with such ambiguity sets are often tractable. In this section we briefly comment on moment ambiguity sets that impose conditions on higher-order (polynomial) moments of~$\P$, which generically lead to NP-hard DRO problems \citep[Propositions~4.5 and~4.6]{popescu2005semidefinite}. 
	
	Assume now that~$\cZ$ is a closed semialgebraic set defined as the feasible set of finitely many polynomial inequalities. In addition, define the monomial of order~$\alpha \in \Z_+^d$ in~$z\in\R^d$ as the function $\prod_{i=1}^d z_i^{\alpha_i}$, which we denote more compactly as~$z^\alpha$. The higher-order moment ambiguity set induced by a finite index set~$\cA \subseteq \Z_+^d$ and the moment bounds~$m_\alpha\in\R$, $\alpha\in\cA$, is then given by 
	\begin{align*}
		\cP = \left\{ \P \in \cP(\cZ): \E_\P [ Z^\alpha ] \leq m_\alpha ~ \forall \alpha \in \cA \right\}.
	\end{align*}
	Evaluating the worst-case expectation of a polynomial function (or the characteristic function of a semialgebraic set) over all distributions in~$\cP$ thus amounts to solving a generalized moment problem. This moment problem as well as its dual constitute semi-infinite linear programs, which can be recast as finite-dimensional conic optimization problems over certain moment cones and the corresponding dual cones of non-negative polynomials \citep{karlin1966tchebycheff, zuluaga2005conic}. Even though NP-hard in general, these conic problems can be approximated by increasingly tight sequences of tractable semidefinite programs by using tools from polynomial optimization \citep{parrilo2000structured,parrilo2003semidefinite, lasserre2001global,lasserre2009moments}. This general technique gives rise to worst-case expectation bounds and generalized Chebyshev inequalities with respect to the ambiguity set~$\cP$ \citep{bertsimas2000moment,lasserre2002bounds,popescu2005semidefinite,lasserre2008semidefinite}. In addition, it leads to tight bounds on worst-case risk measures \citep{natarajan2009constructing}.

	\subsection{$\phi$-Divergence Ambiguity Sets}
	\label{sec:phi-divergence-amgibuity-sets}
	The dissimilarity between two probability distributions is often quantified in terms of a $\phi$-divergence, which is uniquely determined by an entropy function~$\phi$.
	
	\begin{definition}[Entropy Functions]
		\label{def:phi}
		An entropy function $\phi: \R \to \overline{\R}$ is a lower semicontinuous convex function with $\phi(1) =0$ and $\phi(s) =+\infty$ for all $s<0$. 
	\end{definition}
	
	Note that any entropy function~$\phi$ is continuous relative to its domain. In fact, this is true for any {\em univariate} convex lower semicontinuous function. We emphasize, however, that {\em multivariate} convex lower semicontinuous functions can have points of discontinuity within their domains \cite[Example~2.38]{rockafellar2009variational}. The notion of a $\phi$-divergence relies on the perspective $\phi^\pi$ of the entropy function~$\phi$.
	
	\begin{definition}[$\phi$-Divergences \citep{csiszar1964informationstheoretische,csiszar1967information,ali1966general}]
		\label{def:D_phi}
		The (generalized) $\phi$-divergence of $\P\in \cP (\cZ)$ with respect to $\hat{\P}\in \cP (\cZ)$ is given by
		\begin{align*}
			\D_\phi(\P, \hat{\P}) = \int_{\cZ}
			\phi^\pi \left( \frac{\diff \P}{\diff \rho}(z), \frac{\diff \hat{\P}}{\diff \rho}(z) \right) \diff \rho(z), 
		\end{align*}
		where $\phi$ is an entropy function and~$\rho \in \cM_+(\cZ)$ is any dominating measure. This means that~$\P$ and~$\hat\P$ are absolutely continuous with respect to~$\rho$, that is, $\P , \hat{\P} \ll \rho$. 
	\end{definition}
	
	By the definition of $\phi^\pi$ and our convention that $0 \phi(s/0)$ should be interpreted as the recession function~$\phi^\infty(s)$, $\D_\phi(\P, \hat{\P})$ can be recast as
	\begin{align*}
		\D_\phi(\P, \hat{\P}) = \int_{\cZ}
		\frac{\diff \hat{\P}}{\diff \rho}(z)  \cdot  \phi\left( \frac{\frac{\diff \P}{\diff \rho}(z) }{ \frac{\diff \hat{\P}}{\diff \rho}(z)}\right) \diff \rho(z).
	\end{align*}
	A dominating measure $\rho$ always exists, but it must depend on~$\P$ and~$\hat{\P}$. For example, one may set $\rho = \P + \hat{\P}$. The absolute continuity conditions $\P \ll \rho$ and $\hat{\P} \ll \rho$ ensure that the Radon-Nikodym derivatives $\frac{\diff \P}{\diff \rho}$ and $\frac{\diff \hat{\P}}{\diff \rho}$ are well-defined, respectively. The following proposition derives a dual representation of a generic $\phi$-divergence, which reveals that $\D_\phi(\P, \hat{\P})$ is in fact independent of the choice of~$\rho$.
	
	\begin{proposition}[Dual Representation of $\phi$-Divergences]
		\label{prop:dual-phi-divergences}
		We have
		\begin{align*}
			\D_\phi(\P, \hat{\P}) = \sup_{f \in \cF} \; \int_{\cZ}
			f(z)\, \diff\P(z)  -\int_\cZ  \phi^*(f(z)) \, \diff \hat{\P}(z),
		\end{align*}
		where $\cF$ denotes the family of all bounded Borel functions $f: \cZ \rightarrow \dom(\phi^*)$.
	\end{proposition}
	
	\begin{proof}
		As the entropy function~$\phi(s)$ is proper, convex and lower semicontinuous on~$\R$ and as $0\phi(s/0)$ is interpreted as the recession function~$\phi^\infty(s)$, the perspective function~$\phi^\pi(s,t)=t\phi(s/t)$ is proper, convex and lower semicontinuous on~$\R\times\R_+$. By \cite[Theorem~12.2]{rockafellar1970convex}, $\phi^\pi(s,t)$ can therefore be expressed as the conjugate of its conjugate. Note that the conjugate of $\phi^\pi(s,t)$ satisfies
		\begin{align*}
			(\phi^\pi)^*(f,g) & = \sup_{s\in \R,\,t \in \R_+} fs + gt - t\phi(s/t) \\
			& = \;\sup_{t \in  \R_+}\; gt +t\phi^*(f) = \left\{ \begin{array}{cl}
				0 & \text{if } f \in \dom(\phi^*) \text{ and } g + \phi^*(f) \leq 0, \\
				+\infty & \text{otherwise,}
			\end{array}\right.
		\end{align*}
		for all $f,g \in \R$. The second equality in the above expression follows from \cite[Theorem~16.1]{rockafellar1970convex}. As $\phi^\pi(s,t)=\sup_{f,g\in\R} sf+tg-(\phi^\pi)^*(f,g)$ by virtue of \cite[Theorem~12.2]{rockafellar1970convex}, the $\phi$-divergence is thus given by
		\begin{align*}
			\D_\phi(\P, \hat{\P}) & = \int_{\cZ}
			\sup_{f,g\in\R} \left\{ \frac{\diff \P}{\diff \rho}(z)\cdot f + \frac{\diff \hat{\P}}{\diff \rho}(z)\cdot g - (\phi^\pi)^* (f,g) \right\} \diff \rho(z) \\
			& = \int_{\cZ}
			\sup_{f\in\dom(\phi^*)} 
			\left\{ \frac{\diff \P}{\diff \rho}(z)\cdot f - \frac{\diff \hat{\P}}{\diff \rho}(z)\cdot \phi^*(f) \right\} \diff \rho(z)\\
			&= \sup_{f\in\cF} \; \int_{\cZ}
			\left\{\frac{\diff \P}{\diff \rho}(z)\cdot f(z) - \frac{\diff \hat{\P}}{\diff \rho}(z)\cdot \phi^*(f(z)) \right\} \diff \rho(z),
		\end{align*}
		where the second equality exploits our explicit formula for~$(\phi^\pi)^*$ derived above, while the third equality follows from
		\cite[Theorem~14.60]{rockafellar2009variational}. This theorem applies because the function $h:\dom(\phi^*)\times \cZ\rightarrow \R$ defined through
		\[
		h(f,z)=\frac{\diff \P}{\diff \rho}(z)\cdot f - \frac{\diff \hat{\P}}{\diff \rho}(z)\cdot \phi^*(f)
		\]
		is continuous in~$f$ and Borel measurable in~$z$, thus constituting a Carath\'eodory integrand in the sense of~\cite[Example~14.29]{rockafellar2009variational}. The claim then follows immediately from the definition of Radon-Nikodym derivatives. 
	\end{proof}
	
	Proposition~\ref{prop:dual-phi-divergences} reveals that~$\D_\phi(\P, \hat{\P})$ is jointly convex in~$\P$ and~$\hat\P$. If $\phi(s)$ grows superlinearly with~$s$, that is, if the asymptotic growth rate~$\phi^\infty(1)$ is infinite, then $\D_\phi(\P, \hat{\P})$ is finite if and only if $\frac{\diff \P}{\diff \rho}(z)=0$ for $\rho$-almost all $z\in\cZ$ with $\frac{\diff \hat{\P}}{\diff \rho}(z)=0$. Put differently, $\D_\phi(\P, \hat{\P})$ is finite if and only if $\P\ll\hat\P$. In this special case, the chain rule for Radon-Nikodym derivatives implies that $\frac{\diff \P}{\diff \rho} / \frac{\diff \hat{\P}}{\diff \rho} = \frac{\diff \P}{\diff \hat{\P}}$. If~$\phi^\infty(1)=\infty$, the $\phi$-divergence thus admits the more common (but less general) representation
	\begin{align*}
		\D_\phi(\P, \hat{\P}) = \left\{\begin{array}{ll}
			\int_{\cZ}
			\phi\left( \frac{\diff \P}{\diff \hat{\P}}(z)\right) \diff \hat{\P}(z) & \text{if } \P\ll \hat{\P},\\
			+\infty & \text{otherwise.}    
		\end{array}\right.
	\end{align*}
	
	We are now ready to define the $\phi$-divergence ambiguity set as
	\begin{align}
		\label{eq:phi-divergence-ambiguity-set}
		\cP = \left\{ \P \in \cP(\cZ) \, : \, \D_\phi(\P, \hat \P) \leq r \right\}.
	\end{align}
	This set contains all probability distributions~$\P$ supported on~$\cZ$ whose $\phi$-divergence with respect to some prescribed reference distribution~$\hat \P$ is at most~$r \geq 0$.
	
	\begin{remark}[Csisz{\'a}r Duals]
		The family of generalized $\phi$-divergences (which may adopt finite values even if $\P\not\ll \hat\P$) is invariant under permutations of~$\P$ and~$\hat \P$. Formally, we have $\D_\phi(\P, \hat{\P}) = \D_\psi(\hat{\P}, \P)$, where $\psi$ denotes the Csisz{\'a}r dual of~$\phi$ defined through $\psi(s) = \phi^\pi(1, s)=s\phi(1/s)$ \citep[Lemma~2.3]{ben1991certainty}. One readily verifies that if~$\phi$ is a valid entropy function in the sense of Definition~\ref{def:phi}, then~$\psi$ is also a valid entropy function. This relationship shows that, even though $\phi$-divergences are generically asymmetric, we do not sacrifice generality by focusing on divergence ambiguity sets of the form~\eqref{eq:phi-divergence-ambiguity-set}, with the nominal distribution~$\hat\P$ being the second argument of the divergence. From the discussion after Proposition~\ref{prop:dual-phi-divergences} it is clear that if $\phi^\infty(1) = \infty$, then all distributions~$\P$ in the $\phi$-divergence ambiguity set~\eqref{eq:phi-divergence-ambiguity-set} satisfy~$\P\ll\hat\P$. 
		Similarly, if the Csisz{\'a}r dual of $\phi$ satisfies $\psi^\infty(1) = \infty$, then all distributions~$\P$ in the $\phi$-divergence ambiguity set satisfy~$\hat\P\ll\P$. Table~\ref{tab:phi-divergence} lists common entropy functions and their Csisz{\'a}r duals. We emphasize that the family of Cressie-Read divergences includes the (scaled) Pearson $\chi^2$-divergence for $\beta = 2$, the Kullback-Leibler divergence for $\beta \to 1$ and the likelihood divergence for $\beta \to 0$ as special cases. 
	\end{remark}
	
	\begin{table}[!t]
		\setlength\tabcolsep{3pt}
		\footnotesize
		\centering
		\begin{tabular}{l @{\qquad} l l l l}
			\toprule
			Divergence & $\phi(s)~(s\geq 0)$ & $\psi(s) ~(s\geq 0)$ &$\phi^\infty(1)$ &$\psi^\infty(1)$ \\ \hline
			Kullback-Leibler& $s \log(s) - s + 1$ &$-\log(s) + s - 1$ & $\infty$ & $1$ \\[1ex]
			Likelihood & $-\log(s) + s - 1$ &  $s \log(s) - s + 1$ & $1$ & $\infty$ \\[1ex]
			Total variation &$\half|s-1|$ &$\half|s-1|$& $\frac{1}{2}$ & $\frac{1}{2}$ \\[1ex]
			Pearson $\chi^2$ &$(s-1)^2$ &$\frac{1}{s}(s-1)^2$ & $\infty$ & $1$ \\[1ex]
			Neyman $\chi^2$ &$\frac{1}{s}(s-1)^2$ & $(s-1)^2$ & $1$ &$\infty$ \\[1ex]
			Cressie-Read for $\beta \in (0, 1)$ &$\frac{s^\beta - \beta s + \beta - 1}{\beta(\beta-1)}$ &$\frac{s^{1-\beta} - \beta + \beta s - s}{\beta(\beta-1)}$ & $\frac{1}{1-\beta}$ & $\frac{1}{\beta}$ \\[1ex]
			Cressie-Read for $\beta > 1$ &$\frac{s^\beta - \beta s + \beta - 1}{\beta(\beta-1)}$ &$\frac{s^{1-\beta} - \beta + \beta s - s}{\beta(\beta-1)}$ & $\infty$ & $\frac{1}{\beta - 1}$ \\[1ex]
			\bottomrule
		\end{tabular}
		\caption{Examples of entropy functions and their Csisz{\'a}r duals. }
		\label{tab:phi-divergence}
	\end{table}
	
	The DRO literature often focuses on the \emph{restricted} $\phi$-divergence ambiguity set
	\begin{align}
		\label{eq:restricted-phi-divergence-ambiguity-set}
		\cP = \left\{ \P \in \cP(\cZ) \, : \, \P \ll \hat \P, ~ \D_\phi(\P, \hat \P) \leq r \right\}
	\end{align}
	introduced by \citet{ben2013robust}. Unlike the standard $\phi$-divergence ambiguity set~\eqref{eq:phi-divergence-ambiguity-set}, it contains only distributions that are absolutely continuous with respect to the reference distribution~$\hat\P$ even if $\phi^\infty(1)< \infty$. \citet{ben2013robust} study DRO problems over restricted $\phi$-divergence ambiguity sets under the assumption that the reference distribution~$\hat \P$ is discrete. In this case, the absolute continuity constraint $\P\ll\hat\P$ ensures that the ambiguity set contains only discrete distributions supported on the atoms of~$\hat \P$, and thus nature's worst-case expectation problem reduces to a finite convex program. \citet{ben2013robust} further develop a duality theory for this problem class. \citet{shapiro2017distributionally} extends this duality theory to general reference distributions~$\hat \P$ that are not necessarily discrete. \citet{hu2013ambiguous} and \citet{jiang2016data} show that any distributionally robust individual chance constraint with respect to a restricted $\phi$-divergence ambiguity set is equivalent to a classical chance constraint under the reference distribution~$\hat\P$ but with a rescaled confidence level. A classification of various $\phi$-divergences and an analysis of the structural properties of the corresponding $\phi$-divergence ambiguity sets is provided by \citet{bayraksan2015data} under the assumption that~$\cZ$ is finite. Below we review popular instances of the standard and restricted $\phi$-divergence ambiguity~sets.

	\subsubsection{Kullback-Leibler Ambiguity Sets}
	\label{section:KL}
	The Kullback-Leibler divergence is the $\phi$-divergence corresponding to the entropy function that satisfies $\phi(s) = s \log(s) - s + 1$ for all $s\geq 0$; see also Table~\ref{tab:phi-divergence}. As $\phi^\infty(1) = +\infty$, it thus admits the following equivalent definition.
	
	\begin{definition}[Kullback-Leibler Divergence]
		\label{def:KL}
		The Kullback-Leibler divergence of~$\P\in \cP (\cZ)$ with respect to~$\hat{\P}\in \cP (\cZ)$ is given by
		\begin{align*}
			\KL(\P, \hat \P) =
			\begin{cases}
				\displaystyle \int_\cZ \log\left(\frac{\diff \P}{\diff \hat \P}(z)\right) \diff \P(z) & \text{if } \P \ll \hat \P, \\[2ex]
				+\infty & \text{otherwise.}
			\end{cases}
		\end{align*}
	\end{definition}
	
	We now review a famous variational formula for the Kullback-Leibler divergence. 
	
	\begin{proposition}[\citet{donsker1983asymptotic}]
		\label{prop:donsker:KL}
		The Kullback-Leibler divergence of~$\P$ with respect to~$\hat \P$ satisfies
		\begin{align}
			\label{eq:donsker}
			\KL(\P, \hat{\P}) = \sup_{f \in \cF} \; \int_{\cZ} f(z) \, \diff\P(z) - \log \left( \int_\cZ  \text{e}^{f(z)} \, \diff \hat{\P}(z) \right),
		\end{align}
		where $\cF$ denotes the family of all bounded Borel functions $f: \cZ \rightarrow \R^d$.
	\end{proposition}
	\begin{proof}
		The convex conjugate of the entropy function~$\phi$ inducing the Kullback-Leibler divergence satisfies $\phi^*(t)=\exp(t)-1$ with $\dom(\phi^*)=\R$. Thus, the dual representation of generic $\phi$-divergences established in Proposition~\ref{prop:dual-phi-divergences} implies that
		\begin{align*}
			\KL(\P, \hat{\P}) 
			&= \sup_{f \in \cF} \; \int_{\cZ} f(z)\, \diff\P(z) - \int_\cZ \left( \text{e}^{f(z)} - 1 \right) \, \diff \hat{\P}(z),
		\end{align*}
		where $\cF$ denotes the family of all bounded Borel functions $f:\cZ \rightarrow \R$. Note that~$\cF$ is invariant under constant shifts. That is, if $f(z)$ is a bounded Borel function, then so is $f(z)+c$ for any constant~$c\in\R$. Without loss of generality, we may thus optimize over both~$f\in\cF$ and~$c\in\R$ in the above maximization problem to obtain
		\begin{align*}
			\KL(\P, \hat{\P})
			&= \sup_{f \in \cF} \; \sup_{c \in \R} \; \int_{\cZ} \left( f(z) + c \right) \, \diff\P(z) - \int_\cZ  \left( \text{e}^{f(z) + c} - 1 \right) \diff \hat{\P}(z).
		\end{align*}
		For any fixed~$f\in\cF$, the inner maximization problem over~$c$ is uniquely solved by 
		\[
		c^\star=-\log\left(\int_\cZ e^{f(z)}\,\diff \hat{\P}(z)\right) .
		\]
		Substituting this expression back into the objective function yields~\eqref{eq:donsker}.
	\end{proof}
	
	Proposition~\ref{prop:donsker:KL} establishes a link between the Kullback-Leibler divergence and the entropic risk measure. This connection will become useful in Section~\ref{sec:phi:duality}.
	
	The Kullback-Leibler ambiguity set of radius~$r\geq 0$ around $\hat\P\in\cP(\cZ)$ is given~by
	\begin{align}
		\label{eq:KL-ambiguity-set}
		\cP = \left\{ \P \in \cP(\cZ) \, : \, \KL (\P, \hat \P) \leq r \right\}.
	\end{align}
	As $\phi^\infty(1) = +\infty$, all distributions $\P\in\cP$ are absolutely continuous with respect to~$\hat\P$. Thus, $\cP$ coincides with the {\em restricted} Kullback-Leibler ambiguity set. \citet{ghaoui2003worst} derive a closed-form expression for the worst-case value-at-risk of a linear loss function when~$\hat \P$ is a {\em Gaussian} distribution. \citet{hu2013kullback} use similar techniques to show that any distributionally robust individual chance constraint with respect to a Kullback-Leibler ambiguity set is equivalent to a classical chance constraint with a rescaled confidence level. \citet{calafiore2007ambiguous} studies worst-case mean-risk portfolio selection problems when $\hat\P$ is a {\em discrete} distribution. The Kullback-Leibler ambiguity set has also found applications in least-squares estimation \citep{levy2004robust}, hypothesis testing \citep{levy2008robust,gul2017minimax}, filtering \citep{levy2012robust,zorzi2016robust,zorzi2017convergence,zorzi2017robustness}, the theory of risk measures \citep{ahmadi2012entropic,postek2016computationally} and
	extreme value analysis \citep{blanchet2020distributionally}, among many others. 
	
	\subsubsection{Likelihood Ambiguity Sets}
	\label{sec:likelihood-ambiguity-sets}
	As the Kullback-Leibler divergence fails to be symmetric, it gives rise to two strictly different ambiguity sets. The Kullback-Leibler ambiguity set from Section~\ref{section:KL} is obtained by fixing the {\em second} argument of the Kullback-Leibler divergence to the reference distribution~$\hat\P$ and considering all distributions~$\P$ with $\KL (\P, \hat \P) \leq r$. An alternative ambiguity set is obtained by using~$\hat\P$ as the {\em first} argument and setting
	\begin{align}
		\label{eq:likelihood-ambiguity-set}
		\cP = \left\{ \P \in \cP(\cZ) \, : \, \KL (\hat \P, \P) \leq r \right\}.
	\end{align}
	We refer to~$\cP$ as the likelihood ambiguity set centered at~$\hat\P\in\cP(\cZ)$. Indeed, the likelihood or Burg-entropy divergence of $\P\in\cP(\cZ)$ with respect to~$\hat\P\in\cP(\cZ)$ is usually defined as the reverse Kullback-Leibler divergence $\KL(\hat \P, \P)$. This terminology is based on the following intuition. If~$\cZ$ is a discrete set and $\hat\P=\frac{1}{N}\sum_{i=1}^N\delta_{\hat z_i}$ is the empirical distribution corresponding to~$N$ independent samples $\{\hat z_i\}_{i=1}^N$ from an unknown distribution on~$\cZ$, then it is natural to construct the family of all distributions on~$\cZ$ that make the observed data achieve a prescribed level of likelihood. This distribution family corresponds to a superlevel set of the likelihood function $\cL(\P)=\prod_{i=1}^N\P(Z=\hat z_i)$ over~$\cP(\cZ)$. One can show that any such {\em superlevel} set coincides with a {\em sublevel} set of the likelihood divergence~$\KL (\hat \P, \P)$. Thus, it constitutes a likelihood ambiguity set of the form~\eqref{eq:likelihood-ambiguity-set}. We emphasize that this correspondence does not easily carry over to situations where~$\cZ$ fails to be discrete. 
	
	Likelihood ambiguity sets were originally introduced by \citet{wang2016likelihood} in the context of static DRO, and they were used by \citet{WKR13:rmdps} in the context of robust Markov decision processes. \citet{bertsimas2018data,bertsimas2018robust} show that the likelihood ambiguity set contains all distributions that pass a G-test of goodness-of-fit at a prescribed significance level.
	
	Likelihood ambiguity sets display several statistical optimality properties even if~$\cZ$ is uncountable. To explain these properties, we consider the task of evaluating a $(1-\eta)$-upper confidence bound on the expected value of some loss function under an unknown distribution~$\P$ when~$N$ independent samples from~$\P$ are given. Leveraging the empirical likelihood theorem by \citet{owen1988empirical}, \citet{lam2019recovering} shows a desirable property of the likelihood ambiguity set centered around the empirical distribution~$\hat\P$: The associated worst-case expected loss provides the least conservative confidence bound for a constant significance level~$\eta$ asymptotically when the radius $r$ decays at the rate $1/N$. Similar guarantees for a broader class of $\phi$-divergences are reported by \citet{duchi2021statistics}. In addition, \citet{van2021data} leverage Sanov's large deviation principle \cite[Theorem 11.4.1]{cover2006elements} to prove that the worst-case expected loss with respect to a likelihood ambiguity set of constant radius~$r$ around~$\hat\P$ provides the least conservative confidence bound for a decaying  significance level $\eta\propto e^{-rN}$ asymptotically for large~$N$. \citet{gupta2019near} further shows that a likelihood ambiguity set of radius $r\propto N^{-1/2}$ around~$\hat \P$ represents the smallest convex ambiguity set that satisfies a Bayesian robustness guarantee. 
	
	\subsubsection{Total Variation Ambiguity Sets}
	\label{section:TV}
	The total variation distance of two distributions $\P,\hat\P\in\cP(\cZ)$ is the maximum absolute difference between the probabilities assigned to any event by~$\P$ and~$\hat \P$.
	
	\begin{definition}[Total Variation Distance]
		\label{def:TV}
		The total variation distance is the function $\TV:\cP(\cZ) \times \cP(\cZ) \to [0, 1]$ defined through
		\begin{align*}
			\TV(\P, \hat{\P}) = \sup \left\{\left|\P(\cB)  - \hat\P(\cB)\right| : \cB\subseteq\cZ \text{ is a Borel set}\right\}.
		\end{align*}
	\end{definition}
	
	The total variation distance is ostensibly symmetric and satisfies the identity of indiscernible as well as the triangle inequality. Thus, it constitutes a metric on~$\cP(\cZ)$. In addition, the total variation distance is an instance of a $\phi$-divergence.
	\begin{proposition}
		\label{prop:tv}
		The total variation distance coincides with the $\phi$-divergence induced by the the entropy function with $\phi(s)=\half|s-1|$ for all~$s\ge 0$.
	\end{proposition}
	\begin{proof}
		The conjugate entropy function evaluates to $\phi^*(t)=\max\{t,-\half\}$ if $t \le \half$ and to $\phi^*(t)=+\infty$ if $t> \half$. By Proposition~\ref{prop:dual-phi-divergences}, the $\phi$-divergence corresponding to the given entropy function thus admits the dual representation
		\begin{align}
			\label{eq:TV:IPM}
			\D_\phi(\P, \hat{\P}) = \sup_{f \in \cF} \; \int_{\cZ} f(z)\, \diff\P(z) - \int_\cZ  \max\left\{f(z),-\half \right\} \, \diff \hat{\P}(z),
		\end{align}
		where $\cF$ denotes the family of all bounded Borel functions $f:\cZ \rightarrow (-\infty,\half]$. As clipping any~$f\in\cF$ from below at $-\half$ creates a new function in $\cF$ with a non-inferior objective value, we can in fact restrict attention to Borel functions $f:\cZ\to[-\half, \half]$. The objective function in~\eqref{eq:TV:IPM} then simplifies to $\int_{\cZ} f(z) \diff\P(z) - \int_\cZ  f(z) \diff \hat{\P}(z)$. This simplified objective function remains unchanged when~$f$ is shifted by a constant. In summary, we may therefore conclude that \eqref{eq:TV:IPM} is equivalent to
		\begin{align}
			\label{eq:TV:IPM2}
			\D_\phi(\P, \hat{\P}) = \sup_{f \in \cF'} \; \int_{\cZ} f(z)\, \diff\P(z) - \int_\cZ  f(z) \, \diff \hat{\P}(z),
		\end{align}
		where $\cF'$ denotes the family of all Borel functions $f:\cZ \rightarrow [0,1]$. Moreover, as the objective function of the maximization problem in~\eqref{eq:TV:IPM2} is linear in~$f$, we can further restrict~$\cF'$ to contain only binary Borel functions $f:\cZ\to\{0,1\}$ without sacrificing optimality. As there is a one-to-one correspondence between Borel sets and their characteristic functions, we finally obtain the desired identity
		\begin{align*}
			\D_\phi(\P, \hat{\P}) = \sup \left\{\left|\P(\cB)  - \hat\P(\cB)\right| : \cB\subseteq\cZ \text{ is a Borel set}\right\}.
		\end{align*}
		Hence, the claim follows.
	\end{proof}
	
	The total variation ambiguity set of radius~$r\geq 0$ around $\hat\P\in\cP(\cZ)$ is given~by
	\begin{align*}
		\cP = \left\{ \P \in \cP(\cZ) \, : \, \TV (\P, \hat \P) \leq r \right\}.
	\end{align*}
	Most of the existing literature focuses on the {\em restricted} total variation ambiguity set, which contains all distributions $\P\in\cP$ that satisfy~$\P\ll\hat \P$. \citet[Theorem~1]{jiang2018risk} and \citet[Example~3.7]{shapiro2017distributionally} show that the worst-case expected loss with respect to a restricted total variation ambiguity set coincides with a combination of a conditional value-at-risk and the essential supremum of the loss with respect to~$\hat\P$, see also Section~\ref{sec:TV-bound}. \citet{rahimian2019controlling,rahimian2019identifying,rahimian2022effective} study the worst-case distributions of DRO problems over unrestricted total variation ambiguity sets when~$\cZ$ is finite. The total variation ambiguity set is related to Huber's contamination model from robust statistics \citep{huber1981robust}, which assumes that a fraction $r \in (0,1)$ of all samples in a statistical dataset are drawn from an arbitrary contaminating distribution. Hence, the total variation distance between the target distribution to be estimated and the contaminated data-generating distribution is at most~$r$. It is thus natural to use a total variation ambiguity set of radius~$r$ around some estimated distribution as the search space for the target distribution \citep{nishimura2004search,nishimura2006axiomatic,bose2009dynamic,duchi2023distributionally,tsanga2024trade}.

	\subsubsection{$\chi^2$-Divergence Ambiguity Set}
	
	The $\chi^2$-divergence is the $\phi$-divergence corresponding to the entropy function that satisfies $\phi(s)=(s-1)^2$ for all~$s\ge 0$; see also Table~\ref{tab:phi-divergence}. As $\phi^\infty(1) = +\infty$, it thus admits the following equivalent definition.
	
	\begin{definition}[$\chi^2$-Divergence]
		The $\chi^2$-divergence of~$\P\in \cP (\cZ)$ with respect to~$\hat{\P}\in \cP (\cZ)$ is given by
		\begin{align*}
			\chi^2(\P, \hat \P) =
			\begin{cases}
				\displaystyle \int_\cZ \left(\frac{\diff \P}{\diff \hat \P}(z) - 1 \right)^2 \diff \hat \P(z) & \text{if } \P \ll \hat \P, \\[2ex]
				+\infty & \text{otherwise.}
			\end{cases}
		\end{align*}
	\end{definition}
	The $\chi^2$-divergence admits the following dual representation.
	\begin{proposition}
		\label{prop:variance:chi}
		The $\chi^2$-divergence of~$\P$ with respect to~$\hat \P$ satisfies
		\begin{align*}
			\chi^2(\P, \hat{\P})
			= \sup_{f \in \cF} \frac{\left(\E_\P [f(Z)] - \E_{\hat \P} [f(Z)] \right)^2}{\V_{\hat \P} [f(Z)]},
		\end{align*}
		where $\cF$ is a shorthand for the family of all bounded Borel functions $f:\cZ \rightarrow \R$, and $\V_{\hat \P} [f(Z)]$ stands for the variance of $f(Z)$ under~$\hat \P$. If $\V_{\hat \P} [f(Z)]=0$, then the above fraction is interpreted as~0 if $\E_{\P} [f(Z)]=\E_{\hat \P} [f(Z)]$ and as~$+\infty$ otherwise.
	\end{proposition}
	
	\begin{proof}
		The convex conjugate of the entropy function inducing the $\chi^2$-divergence satisfies $\phi^*(t)=\frac{t^2}{4}+t$ if $t \geq -2$ and $\phi^*(t)=-1$ if $t < -2$, and its domain is given by $\dom(\phi^*)=\R$. Consequently, Proposition~\ref{prop:dual-phi-divergences} implies that
		\begin{align*}
			\chi^2(\P, \hat \P) 
			&= \sup_{f \in \cF} \; \int_{\cZ} f(z)\, \diff \P(z)  -\int_\cZ  \left( \frac{f(z)^2}{4} + f(z) \right) \, \diff \hat{\P}(z),
		\end{align*}
		where $\cF$ denotes the family of all bounded Borel functions $f:\cZ \rightarrow \R$. Note that we have replaced $\phi^*(f(z))$ with $f(z)^2/4+f(z)$ in the second integral. This may be done without loss of generality. Indeed, if the function~$f(z)$ adopts values below~$-2$, then it is (weakly) dominated by the function~$f'(z)=\max\{f(z),-2\}$. Note also that~$\cF$ is invariant under constant shifts. That is, if $f(z)$ is a bounded Borel function, then so is $f(z)+c$ for any constant~$c\in\R$. An elementary calculation reveals that, for any fixed~$f\in\cF$, the optimal shift is~$c^\star=-\E_{\hat\P}[f(Z)]$. Hence, we may replace $f(z)$ with $f(z)-\E_{\hat\P}[f(Z)]$ in the above expression, which yields
		\begin{align*}
			\chi^2(\P, \hat \P) 
			&= \sup_{f \in \cF} \; \E_{\P}[f(Z)]  -\E_{\hat \P}[f(Z)] -\frac{\V_{\hat \P}[f(Z)]}{4}.
		\end{align*}
		Note that the set~$\cF$ is also invariant under scaling. That is, if $f(z)$ is a bounded Borel function, then so is $cf(z)$ for any constant~$c\in\R$. We may thus optimize separately over $f\in\cF$ and $c\in\R$ in the above maximization problem to obtain
		\begin{align*}
			\chi^2(\P, \hat{\P})
			&= \sup_{f \in \cF} \; \sup_{c \in \R} \; \left(\E_{\P}[f(Z)]  -\E_{\hat \P}[f(Z)]\right)c -\frac{\V_{\hat \P}[f(Z)]}{4}c^2\\
			&= \sup_{f \in \cF} \; \frac{\left(\E_\P [f(Z)] - \E_{\hat \P} [f(Z)] \right)^2}{\V_{\hat \P} [f(Z)]}.
		\end{align*}
		Note that the inner maximization problem over~$c$ simply evaluates the conjugate of the convex quadratic function $\V_{\hat \P}[f(Z)] c^2/4$ at $\E_\P[f(Z)]- \E_{\hat \P}[f(Z)]$, which is available in closed form. Thus, the claim follows.
	\end{proof}
	
	As the $\chi^2$-divergence fails to be symmetric, it give rise to two complementary ambiguity sets, which differ according to whether the reference distribution~$\hat\P\in\cP(\cZ)$ is used as the first or the second argument of the $\chi^2$-divergence. \citet{lam2018sensitivity} defines the \emph{Pearson} $\chi^2$-ambiguity set of radius $r\geq0$ around $\hat\P$ as
	\begin{align}
		\label{eq:pearson-chi-squared-ambiguity-set}
		\cP = \left\{ \P \in \cP(\cZ) \, : \, \chi^2 (\P, \hat \P) \leq r \right\}
	\end{align}
	in order to analyze operations and service systems with dependent data. \citet{philpott2018distributionally} develop a stochastic dual dynamic programming algorithm for solving distributionally robust multistage stochastic programs with a Pearson ambiguity set. In the context of static DRO,
	\citet{duchi2019variance} show that robustification with respect to a Pearson ambiguity set is closely related to variance regularization. Note that as  $\phi^\infty(1) = +\infty$, the Pearson ambiguity set coincides with its {\em restricted} version, which contains only distributions~$\P\ll\hat\P$.
	
	\citet{klabjan2013robust} define the \emph{Neyman} $\chi^2$-ambiguity set as
	\begin{align*}
		\cP = \left\{ \P \in \cP(\cZ) \, : \, \chi^2 (\hat \P, \P) \leq r \right\}
	\end{align*}
	in order to formulate robust lot-sizing problems. \citet{hanasusanto2013robust} use a Neyman ambiguity set with finite~$\cZ$ in the context of robust data-driven dynamic programming. Finally, \citet{hanasusanto2015distributionally} use the same ambiguity set to model the uncertainty in the mixture weights of multimodal demand distributions.
	
	\subsection{Optimal Transport Ambiguity Sets}
	\label{sec:optimal:transport}
	
	Optimal transport theory offers a natural way to quantify the difference between probability distributions and gives rise to a rich family of ambiguity sets. To explain this, we first introduce the notion of a transportation cost function. 
	
	\begin{definition}[Transportation Cost Function]
		\label{def:cost}
		A lower semicontinuous function $c: \cZ \times \cZ \to [0, +\infty]$ with $c(z, z) = 0$ for all~$z \in \cZ$ is a transportation cost function.
	\end{definition}
	
	Every transportation cost function induces an optimal transport discrepancy.
	
	\begin{definition}[Optimal Transport Discrepancy]
		\label{def:OT}
		The optimal transport discrepancy $\OT_c:\cP(\cZ) \times \cP(\cZ) \to [0, +\infty]$ associated with any given transportation cost function~$c$ is defined through 
		\begin{align}
			\label{eq:ot-discrepancy}
			\OT_c(\P, \hat{\P}) = \inf_{\gamma \in \Gamma(\P, \hat{\P})} \E_{\gamma} [c(Z, \hat Z)],
		\end{align}
		where $\Gamma(\P, \hat{\P})$ represents the set of all couplings~$\gamma$ of $\P$ and $\hat \P$, that is, all joint probability distributions of~$Z$ and~$\hat Z$ with marginals~$\P$ and~$\hat{\P}$, respectively. 
	\end{definition}
	
	By definition, we have~$\gamma \in \Gamma(\P, \hat \P)$ if and only if $\gamma((Z,\hat Z)\in\cB \times\cZ) = \P(Z\in\cB)$ and $\gamma((Z,\hat Z)\in\cZ \times \hat \cB) = \hat \P(\hat Z\in \hat\cB)$ for all Borel sets $\cB,\hat\cB \subseteq \cZ$. If the probability distributions~$\P$ and~$\hat \P$ are visualized as two piles of sand, then any coupling~$\gamma \in \Gamma(\P, \hat \P)$ can be interpreted as a transportation plan, that is, an instruction for morphing~$\hat \P$ into the shape of~$\P$ by moving sand between various origin-destination pairs in~$\cZ$. Indeed, for any fixed origin~$\hat z\in\cZ$, the conditional probability $\gamma(z\leq Z\leq z+\diff z|\hat Z=\hat z)$ determines the proportion of the sand located at~$\hat z$ that should be moved to (an infinitesimally small rectangle at) the destination~$z$. If the cost of moving one unit of probability mass from~$\hat z$ to~$z$ amounts to~$c(z,\hat z)$, then $\OT_c(\P, \hat{\P})$ is the minimal amount of money that is needed to morph~$\hat\P$ into~$\P$. 
	We now provide a dual representation for generic optimal transport discrepancies.
	
	\begin{proposition}[Kantorovich Duality I]
		\label{prop:kanorovich-duality}
		We have
		\begin{align}
			\label{eq:dual-OT-problem}
			\OT_c(\P, \hat \P) = \left\{ \begin{array}{cl}
				\displaystyle \sup_{f\in\cL^1(\P),\,g\in\cL^1(\hat\P)} & \displaystyle \int_{\cZ} f(z) \, \diff \P(z) - \int_{\cZ} g(\hat z) \, \diff \hat \P(\hat z) \\[3ex]
				\text{\em s.t.} &  f(z) - g(\hat z)\leq c(z,\hat z) \quad \forall z,\hat z\in\cZ,
			\end{array}\right.
		\end{align}
		where~$\cL^1(\P)$ and~$\cL^1(\hat\P)$ denote the sets of all Borel functions from~$\cZ$ to~$\R$ that are integrable with respect to~$\P$ and~$\hat\P$, respectively. 
	\end{proposition}
	
	The dual problem ~\eqref{eq:dual-OT-problem} represents the profit maximization problem of a third party that redistributes the sand from~$\hat\P$ to~$\P$ on behalf of the problem owner by buying sand at the origin~$\hat z$ at unit price~$g(\hat z)$ and selling sand at the destination~$z$ at unit price~$f(z)$. The constraints ensure that it is cheaper for the problem owner to use the services of the third party instead of moving the sand without external help at the transportation cost~$c(z,\hat z)$ for every origin-destination pair~$(\hat z,z)$. The optimal price functions~$f^\star$ and~$g^\star$, if they exist, are termed Kantorovich potentials.
	
	\begin{proof}[Proof of Proposition~\ref{prop:kanorovich-duality}]
		For a general proof we refer to~\citep[Theorem~5.10\,(i)]{villani2008optimal}.
		We prove the claim under the simplifying assumption that~$\cZ$ is compact. 
		In this case, the family $\cC(\cZ \times \cZ)$ of all continuous (and thus bounded) functions $f:\cZ \times \cZ\to \R$ equipped with the supremum norm constitutes a Banach space. Its topological dual is the space $\cM(\cZ \times \cZ)$ of all finite signed Borel measures on $\cZ \times \cZ$ equipped with the total variation norm \citep[Corollary~7.18]{folland1999real}. This means that for every continuous linear functional $\varphi:\cC(\cZ\times\cZ)\to\R$ there exists $\gamma\in\cM(\cZ\times\cZ)$ such that $\varphi(f)=\int_{\cZ\times\cZ}f(z,\hat z)\,\diff\gamma(z,\hat z)$ for all $f\in\cC(\cZ\times\cZ)$.
		
		We first use the Fenchel–Rockafellar duality theorem to show that
		\begin{align}
			\label{eq:OT:Cb}
			\OT_c(\P, \hat \P) = \left\{ \begin{array}{cl}
				\displaystyle \sup_{f,g \in \cC(\cZ)} & \displaystyle \int_{\cZ} f(z) \, \diff \P(z) - \int_{\cZ} g(\hat z) \, \diff \hat \P(\hat z) \\[2ex]
				\text{s.t.} &  f(z) - g(\hat z)\leq c(z,\hat z) \quad \forall z,\hat z\in\cZ,
			\end{array}\right.
		\end{align}
		that is, we prove that strong duality holds if the price functions~$f$ and~$g$ in the dual problem are restricted to the space~$\cC(\cZ)$ of continuous functions from~$\cZ$ to~$\R$. To this end, we re-express the maximization problem in~\eqref{eq:OT:Cb} more compactly as
		\begin{equation}
			\label{eq:compact-dual-OT-problem}
			\sup_{h \in \cC(\cZ \times \cZ)}  -\phi(h) - \psi(h),
		\end{equation}
		where the convex functions $\phi, \psi : \cC(\cZ \times \cZ) \to (-\infty,+\infty]$ are defined through 
		\[
		\phi(h) = \left\{ \begin{array}{ll}
			0 & \text{if }-h(z, \hat z) \leq c(z, \hat z)~\forall z,\hat z\in\cZ,\\
			+\infty & \text{otherwise,}
		\end{array}\right.
		\]
		and 
		\[
		\psi(h) = \left\{ \begin{array}{ll}
			\displaystyle \int_{\cZ}\int_\cZ h(z,\hat z)\, \diff \P(z) \,\diff \hat \P(\hat z) & \left\{\begin{array}{l} \text{if }\exists f,g \in\cC(\cZ) \text{ with} \\
				h(z, \hat z) = g(\hat z) - f(z)~\forall z,\hat z\in\cZ, \end{array}\right.\\[2ex]
			+\infty & \text{otherwise.}
		\end{array}\right.
		\]
		Note that~\eqref{eq:compact-dual-OT-problem} can be viewed as the conjugate of $\phi+\psi$ with respect to the pairing of~$\cC(\cZ\times\cZ)$ and~$\cM(\cZ\times\cZ)$ evaluated at the zero measure. Note also that~$\phi$ is continuous at the constant function~$h_0\equiv 1$ because the transportation cost function~$c$ is non-negative. In addition, $h_0$ belongs to the domain of~$\psi$. The Fenchel–Rockafellar duality theorem \citep[Theorem~1.12]{brezis2011functional} thus ensures that the conjugate of the sum of the proper convex functions~$\phi$ and~$\psi$ coincides with the infimal convolution of their conjugates~$\phi^*$ and~$\psi^*$. Hence, \eqref{eq:compact-dual-OT-problem} equals
		\begin{align}
			\label{eq:fenchel:duality}
			(\phi+\psi)^*(0) = \inf_{\gamma \in \cM(\cZ\times\cZ)}  \phi^*(-\gamma) + \psi^*(\gamma).
		\end{align}
		It remains to evaluate the conjugates of~$\phi$ and~$\psi$. For any $\gamma\in\cM(\cZ\times\cZ)$ we have
		\begin{align*}
			\phi^*(-\gamma) 
			&= \sup_{h \in \cC(\cZ \times \cZ)} \left\{ - \int_{\cZ\times\cZ} h(z, \hat z) \, \diff \gamma(z, \hat z): -h(z, \hat z) \leq c(z, \hat z) ~ \forall z, \hat z \in \cZ \right\} \\
			&= \begin{cases}
				\displaystyle \int_{\cZ \times \cZ} c(z, \hat z) \, \diff \gamma(z, \hat z) & \text{if~} \gamma \in \cM_+(\cZ \times \cZ), \\
				+\infty & \text{otherwise,}
			\end{cases}
		\end{align*}
		where $\cM_+(\cZ \times \cZ)$ stands for the cone of finite Borel measures on~$\cZ\times\cZ$. Indeed, if $\gamma \in \cM_+(\cZ \times \cZ)$, then the second equality follows from the monotone convergence theorem, which applies because~$c$ is lower semicontinuous and can thus be written as the pointwise limit of a non-decreasing sequence of continuous functions (see also Lemma~\ref{lem:baire-semicontinuity} below). On the other hand, if $\gamma \not\in \cM_+(\cZ \times \cZ)$, then the second equality holds because every~$\gamma\in\cM(\cZ\times\cZ)$ is a Radon measure, which ensures that the measure of any Borel set can be approximated with the integral of a continuous function. Similarly, for any $\gamma\in\cM(\cZ\times\cZ)$ one readily verifies that
		\begin{align*}
			\psi^*(\gamma) = \begin{cases}
				0 &  \text{if~} \gamma\in\Gamma(\P,\hat\P), 
				\\
				+\infty & \text{otherwise.}
			\end{cases}
		\end{align*}
		Substituting the above formulas for~$\phi^*$ and~$\psi^*$ into~\eqref{eq:fenchel:duality} yields~\eqref{eq:OT:Cb}.
		
		Relaxing the requirement $f,g \in \cC(\cZ)$ to $f \in \cL^1(\P)$ and $g \in \cL^1(\hat \P)$ on the right hand side of~\eqref{eq:OT:Cb} immediately leads to the upper bound
		\begin{align}
			\label{eq:OT:upper}
			\OT_c(\P, \hat \P) \leq \left\{ \begin{array}{cl}
				\displaystyle \sup_{f\in\cL^1(\P),\,g\in\cL^1(\hat\P)} & \displaystyle \int_{\cZ} f(z) \, \diff \P(z) - \int_{\cZ} g(\hat z) \, \diff \hat \P(\hat z) \\[2ex]
				\text{s.t.} &  f(z) - g(\hat z)\leq c(z,\hat z) \quad \forall z,\hat z\in\cZ.
			\end{array}\right.
		\end{align}
		On the other hand, it is clear that
		\begin{align*}
			\OT_c(\P, \hat \P) = \inf_{\gamma \in \cM_+(\cZ \times \cZ)} \, \sup_{f \in \cL^1(\P), g \in \cL^1(\hat \P)} \,  \int_{\cZ \times \cZ} & \big( c(z, \hat z) - f(z) + g(\hat z) \big) \, \diff \gamma(z, \hat z) \\
			& + \int_{\cZ} f(z) \, \diff \P(z) - \int_{\cZ} g(\hat z) \, \diff \hat \P(\hat z).
		\end{align*}
		Interchanging the order of minimization and maximization in the above expression and then evaluating the inner infimum in closed form yields
		\begin{align}
			\label{eq:OT:lower}
			\OT_c(\P, \hat \P) \geq \left\{ \begin{array}{cl}
				\displaystyle \sup_{f\in\cL^1(\P),\,g\in\cL^1(\hat\P)} & \displaystyle \int_{\cZ} f(z) \, \diff \P(z) - \int_{\cZ} g(\hat z) \, \diff \hat \P(\hat z) \\[2ex]
				\text{s.t.} &  f(z) - g(\hat z)\leq c(z,\hat z) \quad \forall z,\hat z\in\cZ.
			\end{array}\right.
		\end{align}
		Combining~\eqref{eq:OT:upper} with~\eqref{eq:OT:lower} proves~\eqref{eq:dual-OT-problem}, and thus the claim follows.
	\end{proof}
	
	The dual optimal transport problem~\eqref{eq:dual-OT-problem} constitutes a linear program over the price functions~$f \in \cL^1(\P)$ and~$g \in \cL^1(\hat\P)$, and its objective function is linear in~$\P$ and~$\hat \P$. As pointwise suprema of linear functions are convex, $\OT_c(\P, \hat \P)$ is thus jointly convex in~$\P$ and~$\hat\P$. Problem~\eqref{eq:dual-OT-problem} can be further simplified by invoking the $c$-transform $f^c:\cZ\to(-\infty,+\infty]$ of the price function~$f$, which is defined through
	\begin{align}
		\label{eq:c-transofrm-of-f}
		f^c(\hat z) = \sup_{z \in \cZ} \, f(z) - c(z, \hat z).
	\end{align}
	The constraints of the dual problem~\eqref{eq:dual-OT-problem} can now be re-expressed as
	\begin{align*}
		g(\hat z) \geq f(z) - c(z, \hat z) \quad \forall z,\hat z \in \cZ\quad \iff \quad g(\hat z) \geq f^c(\hat z)\quad \forall \hat z \in \cZ.
	\end{align*}
	Note that problem~\eqref{eq:dual-OT-problem} seeks a price function~$g$ that is as {\em small} as possible. As~$g$ is lower bounded by~$f^c$, this suggests that~$g = f^c$  at optimality. Conversely, defining the $c$-transform $g^c:\cZ\to[-\infty,+\infty)$ of the price function~$g$ through
	\begin{align}
		\label{eq:c-transofrm-of-g}
		g^c(z) = \inf_{\hat z \in \cZ} \, g(\hat z) + c(z, \hat z),
	\end{align}
	the constraint of problem~\eqref{eq:dual-OT-problem} can be re-expressed as
	\begin{align*}
		f(z) \leq g(\hat z) + c(z, \hat z) \quad \forall z,
		\hat z\in \cZ\quad\iff \quad f(z) \leq g^c(z)\quad\forall z\in\cZ.
	\end{align*}
	This suggests that~$f = g^c$ at optimality. Note that $f^c$ and $g^c$ may fail to be integrable with respect to~$\hat\P$ and~$\P$, respectively. If~$f\in\cL^1(\P)$ and~$g\in\cL^1(\hat \P)$, however, then one can verify that the integrals $\int_{\cZ} f^c(\hat z)\, \diff \hat \P(\hat z) <+\infty$ and $\int_{\cZ} g^c(z) \,\diff \P(z)>-\infty$ exist as extended real numbers. The above insights culminate in the following corollary, which we state without proof. For details see \citet[Theorem~5.10\,(i)]{villani2008optimal}.
	
	\begin{corollary}[Kantorovich Duality II]
		\label{cor:kanorovich-duality}
		We have
		\begin{align*}
			\OT_c(\P, \hat \P) 
			&= \sup_{f \in \cL^1(\P)} ~ \int_\cZ f(z) \, \diff \P(z) - \int_\cZ f^c(\hat z) \, \diff \hat \P(\hat z) \\
			&= \sup_{g \in \cL^1(\hat \P)} ~ \int_\cZ g^c(z) \, \diff \P(z) - \int_\cZ g(\hat z) \, \diff \hat \P(\hat z), 
		\end{align*}
		where the $c$-transforms $f^c$ and $g^c$ are defined in~\eqref{eq:c-transofrm-of-f} and \eqref{eq:c-transofrm-of-g}, respectively. In addition, the first (second) supremum does not change if we require that $f=g^c$ ($g=f^c$) for some function $g:\cZ\to(-\infty,+\infty]$ ($f:\cZ\to[-\infty,+\infty)$).
	\end{corollary}
	
	Given any transportation cost function~$c$, reference distribution~$\hat\P\in\cP(\cZ)$ and transportation budget~$r\geq 0$, the optimal transport ambiguity set is defined as
	\begin{align}
		\label{eq:OT-ambiguity-set}
		\cP = \left\{ \P \in \cP(\cZ) \, : \, \OT_c(\P, \hat \P) \leq r \right\}.
	\end{align}
	By construction, $\cP$ contains all probability distributions~$\P$ that can be obtained by reshaping the reference distribution~$\hat \P$ at a finite cost of at most~$r \geq 0$. The optimal transport ambiguity set was first studied by \citet{pflug2007ambiguity}, who propose a successive linear programming algorithm to solve robust mean-risk portfolio selection problems when~$\cZ$ is {\em finite}. \citet{postek2016computationally} leverage tools from conjugate duality theory to develop an exact solution method for the same problem class. \citet{wozabal2012framework} and \citet[\ts~7.1]{pflug2014multistage} reformulate DRO problems with optimal transport ambiguity sets over {\em uncountable} support sets~$\cZ\subseteq\R^d$ as finite-dimensional {\em non}convex programs and address them with methods from global optimization. \citet{mohajerin2018data} and \citet{zhao2018data} use specialized duality results to show that these DRO problems are in fact equivalent to generalized moment problems that admit exact reformulations as finite-dimensional convex programs. \citet{blanchet2019quantifying,gao2016distributionally} as well as \citet{zhang2022simple} show that the underlying duality results remain valid even when~$\cZ$ is a Polish space. For recent surveys of the theory and applications of DRO with optimal transport ambiguity sets we refer to \citet{kuhn2019wasserstein} and \citet{blanchet2021statistical}.

	\subsubsection{$p$-Wasserstein Ambiguity Sets}
	\label{sec:p-Wasserstein-balls}
	It is common to set the transportation cost function~$c$ in Definition~\ref{def:OT} to the $p$-th power of some metric on~$\cZ$. In this case, the $p$-th root of the optimal transport discrepancy is termed the $p$-Wasserstein distance.
	
	\begin{definition}[$p$-Wasserstein Distance]
		\label{def:p-Wassertein}
		Assume that~$d(\cdot,\cdot)$ is a metric on~$\cZ$ and $p \in [1, +\infty)$ is a prescribed exponent. Then, the $p$-Wasserstein distance $\W_p:\cP(\cZ) \times \cP(\cZ) \to [0,+\infty]$ corresponding to~$d$ and~$p$ is defined via
		\begin{align*}
			\W_p(\P, \hat \P) = \inf_{\gamma \in \Gamma(\P, \hat{\P})} \left( \E_{\gamma} [d(Z, \hat Z)^p]  \right)^{\frac{1}{p}}.
		\end{align*}
	\end{definition}
	
	Definition~\ref{def:p-Wassertein} implies that if~$c(z, \hat z) = d(z, \hat z)^p$, then $W_p^p(\P, \hat \P) = \OT_c(\P, \hat \P)$. In the following we use $\cP_p(\cZ) = \{\P \in \cP(\cZ): \E_{\P}[d(Z,\hat z_0)^p] < \infty \}$ to denote the family of all distributions on~$\cZ$ with finite $p$-th moment. As~$d$ is a metric, $\cP_p(\cZ)$ is independent of the choice of the reference point~$\hat z_0\in\cZ$. The $p$-Wasserstein distance constitutes a metric on~$\cP_p(\cZ)$. Indeed, it is evident that~$W_p(\P, \hat \P)$ is symmetric and vanishes if and only if $\P=\hat\P$. The proof that $W_p(\P, \hat \P)$ obeys the triangle inequality requires a gluing lemma for transportation plans and is therefore more intricate; see, {\em e.g.}, \cite[\ts~1]{villani2008optimal}. The $p$-Wasserstein distance further metrizes the weak convergence of distributions {\em and} the convergence of their $p$-th moments. This means that $W_p(\P, \hat \P_N)$ converges to~$0$ if and only if $\hat\P_N$ converges weakly to~$\P$ {\em and} $\E_{\hat\P_N}[d(Z,\hat z_0)^p]$ converges to $\E_{\P}[d(Z,\hat z_0)^p]$ as~$N$ grows \citep[Theorem~6.9]{villani2008optimal}. Furthermore, the $p$-Wasserstein distance enjoys attractive measure concentration properties. Specifically, if~$\hat\P_N$ represents the empirical distribution obtained from~$N$ independent samples from~$\P$, then the rate at which~$\hat\P_N$ converges to~$\P$ in $p$-Wasserstein distance admits sharp asymptotic and finite-sample bounds \citep{fournier2015rate,weed2019sharp}. 
	
	As the $p$-Wasserstein distance constitutes the $p$-th root of an optimal transport discrepancy, Proposition~\ref{prop:kanorovich-duality} and Corollary~\ref{cor:kanorovich-duality} readily imply that it admits a dual representation. For $p=1$ this dual representation becomes particularly simple. Indeed, one can show that the $1$-Wasserstein distance coincides with the integral probability metric generated by all test functions that are Lipschitz continuous with respect to the metric~$d$ and have Lipschitz modulus at most~$1$.
	
	\begin{corollary}[Kantorovich-Rubinstein Duality]
		\label{cor:kanorovich-rubinstein-duality}
		We have
		\begin{align*}
			\W_1(\P, \hat \P) = \sup_{\substack{f \in \cL^1(\P),\,\lip(f) \leq 1}}  \; \int_{\cZ} f(z) \, \diff \P(z) - \int_{\cZ} f(\hat z) \, \diff \hat \P(\hat z).
		\end{align*}
	\end{corollary}
	\begin{proof}
		Corollary~\ref{cor:kanorovich-duality} implies that
		\begin{align*}
			\W_1(\P, \hat \P) = \sup_{\substack{f \in \cL^1(\P)}}  \; \int_{\cZ} f(z) \, \diff \P(z) - \int_{\cZ} f^c(\hat z) \, \diff \hat \P(\hat z).
		\end{align*}
		In addition, it ensures that the supremum does not change if we restrict the search space to functions that are representable as $f=g^c$ for some $g:\cZ\to(-\infty,+\infty]$. By~\eqref{eq:c-transofrm-of-g}, we thus have $f(z)=\inf_{\hat z\in\cZ} g(\hat z)+d(z,\hat z)$. For any fixed~$\hat z\in\cZ$, the auxiliary function $f_{\hat z}(z)=g(\hat z)+d(z,\hat z)$ is ostensibly $1$-Lipschitz with respect to the metric~$d$. As infima of $1$-Lipschitz functions remain $1$-Lipschitz, we thus find $\lip(f)\leq 1$. In summary, we have shown that restricting attention to $1$-Lipschitz functions does not reduce the supremum of the dual optimal transport problem. Next, we prove that~$\lip(f) \leq 1$ implies that~$f^c=f$. Indeed, for any $\hat z\in\cZ$ we have
		\begin{align*}
			f(\hat z)\leq \sup_{z\in\cZ} f(z)-d(z,\hat z)\leq \sup_{z\in\cZ} f(\hat z)+d(z,\hat z)-d(z,\hat z)=f(\hat z),
		\end{align*}
		where the two inequalities hold because $d(\hat z,\hat z)=0$ and~$\lip(f) \leq 1$, respectively. This implies via~\eqref{eq:c-transofrm-of-f} that $f(\hat z) = \sup_{z\in\cZ} f(z)-d(z,\hat z)=f^c(\hat z)$ for all~$\hat z\in\cZ$. Hence, $f^c$ coincides with~$f$ whenever $\lip(f) \leq 1$, and thus the claim follows. 
	\end{proof}

	The $p$-Wasserstein ambiguity set of radius~$r\geq 0$ around~$\hat\P\in\cP(\cZ)$ is defined as
	\begin{align}
		\label{eq:p-wasserstein-ball}
		\cP = \left\{ \P \in \cP(\cZ) \, : \, \W_p(\P, \hat \P) \leq r \right\}.
	\end{align}
	\citet{pflug2012} study robust portfolio selection problems, where the uncertainty about the asset return distribution is captured by a $p$-Wasserstein ball. They prove that---as~$r$ approaches infinity---it becomes optimal to distribute one's capital equally among all available assets. Hence, this result reveals that the popular $1/N$-investment strategy widely used in practice \citep{demiguel-1-over-n} is optimal under extreme ambiguity. \citet{pflug2012,pichler:2013} and \citet{wozabal2014robustifying} further show that, for a broad range of convex risk measures, the worst-case portfolio risk across all distributions in a $p$-Wasserstein ball equals the nominal risk under~$\hat\P$ plus a regularization term that scales with the Wasserstein radius~$r$; see also Section~\ref{sec:Lipschitz-continuous-risk-measures}. 
	
	The Wasserstein ambiguity set corresponding to $p=1$ enjoys particular prominence in DRO.
	The Kanthorovich-Rubinstein duality can be used to construct a simple upper bound on the worst-case expectation of a Lipschitz continuous loss function across all distributions in a $1$-Wasserstein ball. This upper bound is given by the sum of the expected loss under the nominal distribution~$\hat \P$ plus a regularization term that consists of the Lipschitz modulus of the loss function weighted by the radius~$r$ of the ambiguity set. \citet{shafieezadeh2015distributionally} demonstrate that this upper bound is exact for distributionally robust logistic regression problems. However, this exactness result extends in fact to many linear prediction models with convex \citep{ chen2018robust,chen2019selecting,blanchet2019robust,  shafieezadeh2019regularization,wu2022generalization} and even nonconvex loss functions \citep{gao2024wasserstein,ho2020adversarial}. More generally, $1$-Wasserstein ambiguity sets have found numerous applications in diverse areas such as two-stage and multi-stage stochastic programming \citep{zhao2018data,hanasusanto2018conic,duque2020distributionally, bertsimas2023data}, chance constrained programming \citep{chen2018data,xie2019distributionally,ho2020distributionally,shen2020chance}, inverse optimization \citep{esfahani2018inverse}, statistical learning \citep{blanchet2019multivariate,zhu2022distributionally}, hypothesis testing \citep{gao2018robust}, contextual stochastic optimization~\citep{zhang2024optimal}, transportation \citep{sun2023distributionally}, control \citep{cherukuri2020cooperative,yang2020wasserstein,boskos2020data,li2020data,coulson2021distributionally,aolaritei2022uncertainty,terpin2022trust,terpin2024dynamic}, and power systems analysis \citep{wang2018risk,ordoudis2021energy}, among others.
	
	The Wasserstein ambiguity set corresponding to $p=2$ also enjoys wide popularity. Before reviewing its various uses, we highlight an interesting connection between the $2$-Wasserstein distance and the Gelbrich distance introduced in Section~\ref{sec:chebyshev-with-moment-uncertainty} (see Definition~\ref{def:Gelbrich}). As pointed out by \citet[Theorem~2.1]{gelbrich1990formula}, the $2$-Wasserstein distance between two probability distributions provides an upper bound on the Gelbrich distance between their mean-covariance pairs.

	\begin{theorem}[Gelbrich Bound]
		\label{theorem:gelbrich}
		Assume that $\cZ$ is equipped with the Euclidean metric $d(z, \hat z) = \| z - \hat z \|_2$.
		For any distributions $\P,\hat \P\in\cP(\cZ)$ with finite mean vectors $\mu,\hat\mu \in \R^d$ and covariance matrices $\Sigma$, $\hat \Sigma \in \S_+^d$, respectively, we have 
		$$\W_2(\P, \hat \P) \geq \G ((\mu, \Sigma), (\hat \mu, \hat \Sigma)).$$
	\end{theorem}
	\begin{proof}
		By definition, the squared $2$-Wasserstein distance satisfies
		\begin{align*}
			& \W_2^2(\P, \hat \P) =  \inf_{\gamma \in \Gamma(\P, \hat \P)}~ \int_{\cZ \times \cZ} \| z - \hat z \|_2^2 \, \diff \gamma(z, \hat z) \\[0.5em]
			& = \left\{
			\begin{array}{cl}
				\inf & \| \mu - \hat \mu \|_2^2 + \Tr[\Sigma + \hat \Sigma - 2C] \\[1ex]
				\st &  \gamma \in \Gamma(\P, \hat \P),\; C \in \R^{d \times d} \\[1ex]
				& \displaystyle \int_{\cZ \times \cZ} \begin{bmatrix} z-\mu \\ \hat z-\hat \mu \end{bmatrix} \begin{bmatrix} z-\mu \\ \hat z -\hat\mu\end{bmatrix}^\top \diff \gamma( z,  \hat z) = \begin{bmatrix} \Sigma & C \\ C^\top & \hat \Sigma \end{bmatrix},\quad  \begin{bmatrix} \Sigma & C \\ C^\top & \hat \Sigma \end{bmatrix}\succeq 0.
			\end{array} \right. 
		\end{align*}
		Note that the new decision variable~$C$ is uniquely determined by the transportation plan~$\gamma$, that is, it represents the cross-covariance matrix of~$Z$ and~$\hat Z$ under~$\gamma$. Thus, its presence does not enlarge the feasible set. Note also that the linear matrix inequality in the last expression is redundant because the second-order moment matrix of~$\gamma$ is necessarily positive semidefinite. Thus, its presence does not reduce the feasible set. Finally, note that the integral of the quadratic function
		\begin{align*}
			\| z - \hat z \|_2^2 = & \| \mu - \hat \mu \|_2^2 + \| z - \mu \|_2^2 + \| \hat z - \hat \mu \|_2^2 -2(z-\mu)^\top(\hat z-\hat \mu) \\
			& + 2(\mu-\hat\mu)^\top (z-\mu) - 2(\mu-\hat\mu)^\top (\hat z-\hat\mu)
		\end{align*}
		with respect to~$\gamma$ is uniquely determined by the first- and second-order moments of~$\gamma$ and evaluates to $\| \mu - \hat \mu \|_2^2 + \Tr[\Sigma + \hat \Sigma - 2C]$. Relaxing the last optimization problem by removing all constraints that involve~$\gamma$ then yields
		\begin{align*}
			\W_2^2(\P, \hat \P) \geq 
			\left\{
			\begin{array}{cl}
				\displaystyle \min_{C \in \R^{d \times d}} & \| \mu - \hat \mu \|_2^2 + \Tr[\Sigma + \hat \Sigma - 2C]\\
				\st & \begin{bmatrix} \Sigma & C \\ C^\top & \hat \Sigma \end{bmatrix} \succeq 0. 
			\end{array}
			\right.
		\end{align*}
		By Proposition~\ref{prop:Gelbrich:SDP}, the optimal value of the resulting semidefinite program amounts to $\G^2( (\mu, \Sigma), (\hat \mu, \hat \Sigma))$. The claim follows by taking square roots on both sides.
	\end{proof}
	
	The proof of Theorem~\ref{theorem:gelbrich} reveals that the squared Gelbrich distance coincides with the minimum of a relaxed optimal transport problem, which only requires the marginals of the transportation plan~$\gamma$ to have the same first- and second-order moments as~$\P$ and~$\hat\P$, respectively. Gelbrich's inequality may be useful when the exact $2$-Wasserstein distance is inaccessible. Indeed, computing the $2$-Wasserstein distance between a discrete and a continuous distribution is $\#$P-hard already when the discrete distribution has only two atoms \citep{taskesen2021semi}. Computing the $2$-Wasserstein distance may even be $\#$P-hard when both distributions are discrete \citep{taskesen2023discrete}. If both~$\P$ and~$\hat\P$ are Gaussian, then Gelbrich's inequality collapses to an equality. Thus, the $2$-Wasserstein distance between two Gaussian distributions matches the Gelbrich distance between their mean vectors and covariance matrices \citep[Proposition~7]{givens1984class}. This classical result, which actually predates Gelbrich's inequality, is nowadays recognized as an immediate consequence of a celebrated optimality condition for optimal transport problems by \citet{brenier1991polar}. Using Brenier's optimality condition, one can prove more generally that if $\hat \P$ is a positive semidefinite affine pushforward of~$\P$, that is, if there exists an affine function $f(z) = A z + b$ with $A \in \S_+^d$ and $b \in \R^d$ such that $\hat \P = \P \circ f^{-1}$, then the $2$-Wasserstein distance between~$\P$ and~$\hat\P$ matches again the Gelbrich distance between their mean vectors and covariance matrices \citep[Theorem~2]{nguyen2021mean}. 
	
	The $2$-Wasserstein ambiguity set has found applications in machine learning \citep{sinha2018certifying,blanchet2019robust,blanchet2019confidence,blanchet2018optimal}, inverse optimization \citep{esfahani2018inverse}, two-stage stochastic programming \citep{hanasusanto2018conic}, estimation and filtering \citep{shafieezadeh2018wasserstein,nguyen2019bridging,kargin2024distributionally}, portfolio optimization~\citep{blanchet2018distributionally,nguyen2021mean} as well as control theory \citep{al2023distributionally,hajar2023wasserstein,hakobyan2024wasserstein,taskesen2024distributionally,kargin2024lqr,kargin2024infinite,kargin2024wasserstein}.
	
	\subsubsection{L\'evy-Prokhorov Ambiguity Sets}
	\label{sec:LP-ambiguity-sets}
	The L\'evy-Prokhorov distance is one of the most widely used probability metrics because it metrizes the topology of weak convergence on~$\cP(\cZ)$. We assume below that~$d(\cdot,\cdot)$ is a continuous metric on~$\cZ$. For any set~$\cB\subseteq\cZ$ and $r\geq 0$, we use 
	\begin{align}
		\label{eq:Br}
		\cB_r = \left\{z \in \cZ: \exists z' \in \cB \text{ with } d(z,z') \leq r \right\}
	\end{align}
	to denote the $r$-neighborhood of~$\cB$. The dependence of~$\cB_r$ on the metric~$d$ is notationally suppressed because~$d$ is usually obvious from the context. With these preparations, we are now ready to define the L\'evy-Prokhorov distance.
	
	\begin{definition}[L\'evy-Prokhorov Distance]
		For any metric $d(\cdot,\cdot)$ on~$\cZ$, the L\'evy-Prokhorov distance $\LP:\cP(\cZ) \times \cP(\cZ) \to [0,1]$ induced by~$d$ is defined via
		\begin{align*}
			\LP (\P, \hat{\P}) = \inf \left\{r\geq 0: \P(\cB)  \leq \hat\P(\cB_r) +r \text{ for all Borel sets } \cB\subseteq\cZ\right\},
		\end{align*}
		where $\cB_r$ is defined in~\eqref{eq:Br}.
	\end{definition}
	
	The L\'evy-Prokhorov distance is bounded by~$1$ and vanishes if and only if its arguments match. In addition, one can easily show that it satisfies the triangle inequality. However, it appears to be asymmetric. The next proposition reveals that the L\'evy-Prokhorov distance is closely linked to the theory of optimal transport. 
	
	\begin{proposition}[\citet{strassen1965existence}]
		\label{prop:dual-levy-prokhorov}
		If the transportation cost function~$c_r$ corresponding to $r\geq 0$ is defined through $c_r(z,\hat z)=\ds 1_{d(z,\hat z)>r}$ for all $z,\hat z\in\cZ$, then
		\begin{align*}
			\LP (\P, \hat{\P}) = \inf \left\{r \geq 0: \OT_{c_r}(\P, \hat{\P})\leq r \right\}.
		\end{align*}
	\end{proposition}
	
	\begin{proof}
		Note that $c_r$ is lower semicontinuous because the metric~$d$ is continuous by assumption.
		By Proposition~\ref{prop:kanorovich-duality}, $\OT_{c_r}(\P, \hat{\P})$ thus admits the dual representation
		\begin{align}
			\label{eq:levi-prokhorov-duality}
			\begin{array}{cl}
				\displaystyle \sup_{f\in\cL^1(\P),\,g\in\cL^1(\hat\P)} & \displaystyle \int_{\cZ} f(z) \, \diff \P(z) - \int_{\cZ} g(\hat z) \, \diff \hat \P(\hat z) \\
				\text{s.t.} & f(z)-g(\hat z)\leq \ds 1_{d(z,\hat z)>r} \quad \forall z,\hat z\in\cZ.
			\end{array}
		\end{align}
		Here, for any fixed~$g$, it is optimal to push~$f$  up such that for all $z\in\cZ$ we have
		\begin{subequations}
			\begin{align}
				\label{eq:f:bounds}
				f(z) = \inf_{\hat z\in\cZ}  g(\hat z) + \ds 1_{d(z,\hat z)>r} \implies  \inf_{\hat z\in\cZ}  g(\hat z) \leq f(z) \leq 1 + \inf_{\hat z\in\cZ}  g(\hat z).   
			\end{align}
			Also, for any fixed $f$, it is optimal to push $g$ down such that for all $\hat z\in\cZ$ we have
			\begin{align}
				\label{eq:g:bounds}
				g(\hat z) = \sup_{z\in\cZ}  f(z) - \ds 1_{d(z,\hat z)>r} \implies \sup_{z\in\cZ}  f(z)-1\leq g(\hat z)\leq \sup_{z\in\cZ}  f(z).   
			\end{align}
		\end{subequations}
		Combining the upper bound on $g(\hat z)$ in~\eqref{eq:g:bounds} with the upper bound on $f(z)$ in~\eqref{eq:f:bounds} further implies that $g(\hat z) \leq \sup_{z \in \cZ} f(z) \leq 1 + \inf_{z' \in \cZ}  g(z')$. At optimality, \eqref{eq:f:bounds} and~\eqref{eq:g:bounds} must hold simultaneously, and thus we have
		\[
		\inf_{ z'\in\cZ}  g( z') \leq f(z) \leq 1 + \inf_{z'\in\cZ}  g(z') \quad \text{and} \quad \inf_{ z'\in\cZ}  g( z') \leq g(\hat z) \leq 1 + \inf_{z'\in \cZ}  g(z')
		\]
		for all $z,\hat z\in\cZ$. Note that, as both~$\P$ and~$\hat \P$ are probability distributions, the objective function of the dual optimal transport problem~\eqref{eq:levi-prokhorov-duality} remains invariant under the substitutions~$f(z)\leftarrow f(z)-\inf_{z'\in\cZ} g(z')$ and~$g(\hat z) \leftarrow g(\hat z) - \inf_{z'\in\cZ}g (z')$. In the following, we may thus assume without loss of generality that $0\leq f(z)\leq 1$ for all~$z\in\cZ$ and that $0\leq g(\hat z)\leq 1$ for all~$\hat z\in\cZ$.
		
		As $f$ and $g$ are now normalized to~$[0,1]$, they admit the integral representations
		\[
		f(z)=\int_0^1 \ds 1_{f(z)\geq \tau}\,\diff\tau \quad\forall z\in\cZ\quad \text{and} \quad g(\hat z)=\int_0^1 \ds 1_{g(\hat z)\geq \tau}\,\diff\tau \quad \forall \hat z\in\cZ.
		\]
		Next, one can show that~$f$ and~$g$ satisfy the constraints in~\eqref{eq:levi-prokhorov-duality} if and only if
		\begin{align}
			\label{eq:levi-prokhorov-duality-constraints}
			\ds 1_{f(z)\geq \tau} -\ds 1_{g(\hat z)\geq \tau} \leq \ds 1_{d(z,\hat z)>r} \quad \forall z,\hat z\in\cZ,~\forall \tau\in[0,1].
		\end{align}
		Note first that~\eqref{eq:levi-prokhorov-duality-constraints} is trivially satisfied unless its left hand side evaluates to~$1$ and its right hand side evaluates to~$0$. This happens if and only if $f(z)\geq\tau$ and $g(\hat z)<\tau$ for some $\tau\in[0,1]$ and $z,\hat z\in\cZ$ with $d(z,\hat z)\leq r$. This is impossible, however, because it implies that $f(z)-g(\hat z)>0$ for some $z,\hat z$ with $d(z,\hat z)\leq r$, thus contradicting the constraints in~\eqref{eq:levi-prokhorov-duality}. Hence, the constraints in~\eqref{eq:levi-prokhorov-duality} imply~\eqref{eq:levi-prokhorov-duality-constraints}. The converse implication follows immediately from the integral representations of~$f$ and~$g$. 
		
		Finally, note that $\ds 1_{f(z)\geq \tau}$ and $\ds 1_{g(\hat z)\geq \tau}$ are the characteristic functions of the Borel sets $\cB=\{z\in\cZ:f(z)\geq \tau\}$ and $\cC=\{\hat z\in\cZ:g(\hat z)\geq\tau\}$, respectively. Note also that~\eqref{eq:levi-prokhorov-duality-constraints} holds if and only if $\cC\supseteq\cB_r$. Recalling their integral representations, we may thus conclude that the functions~$f$ and~$g$ are feasible in~\eqref{eq:levi-prokhorov-duality} if and only if they represent convex combinations of (infinitely many) characteristic functions of the form $\ds 1_{z\in\cB}$ and $\ds 1_{\hat z\in\cC}$ for some Borel sets~$\cB$ and~$\cC$ with~$\cC\supseteq\cB_r$. As the objective function of~\eqref{eq:levi-prokhorov-duality} is linear in~$f$ and~$g$, its supremum does not change if we restrict the feasible set to such characteristic functions. Hence, \eqref{eq:levi-prokhorov-duality} reduces to
		\begin{align*}
			\OT_{c_r}(\P, \hat{\P}) = \sup \left\{ \P(\cB)  -\hat\P(\cC) : \cB,\cC\subseteq\cZ\text{ are Borel sets with } \cC\supseteq\cB_r \right\}.
		\end{align*}
		Clearly, it is always optimal to set $\cC=\cB_r$, and thus the claim follows.
	\end{proof}
	
	While Proposition~\ref{prop:dual-levy-prokhorov} follows from \cite[Theorem~11]{strassen1965existence}, the proof shown here parallels that of \cite[Theorem~1.27]{villani2003topics}. As a byproduct, Proposition~\ref{prop:dual-levy-prokhorov} reveals that the L\'evy-Prokhorov distance is symmetric, which is not evident from its definition. Thus, it constitutes indeed a metric. 
	
	The L\'evy-Prokhorov ambiguity set of radius~$r\geq 0$ around~$\hat\P\in\cP(\cZ)$ is defined~as
	\[
	\cP= \left\{ \P \in \cP(\cZ) \, : \, \LP(\P, \hat \P) \leq r \right\}.
	\]
	For our purposes, the most important implication of Proposition~\ref{prop:dual-levy-prokhorov} is that $\cP$ can be viewed as special instance of an optimal transport ambiguity set, that is, we have
	\[
	\cP= \left\{ \P \in \cP(\cZ) \, : \, \OT_{c_r}(\P, \hat{\P}) \leq r \right\}
	\]
	for any radius $r\geq 0$. L\'evy-Prokhorov ambiguity sets were first introduced in the context of chance-constrained programming \citep{erdougan2006ambiguous}. They also naturally emerge in data-driven decision-making and the training of robust machine learning models \citep{pydi2021adversarial,bennouna2023holistic,bennouna2023certified}. We close this section with a useful corollary, which follows immediately from the last part of the proof of Proposition~\ref{prop:dual-levy-prokhorov}.
	
	\begin{corollary}
		\label{cor:dual:OT:c:r}
		If the transportation cost function~$c_r$ corresponding to $r\geq 0$ is defined through $c_r(z,\hat z)=\ds 1_{d(z,\hat z)>r}$ for all $z,\hat z\in\cZ$, then we have
		\begin{align*}
			\OT_{c_r}(\P, \hat{\P}) = \sup \left\{ \P(\cB) - \hat\P(\cB_r) : \cB \, \subseteq \cZ\text{ is a Borel set} \right\},
		\end{align*}
		where the $r$-neighborhood $\cB_r$ is defined in~\eqref{eq:Br}.
	\end{corollary}
	
	\subsubsection{Total Variation Ambiguity Sets Revisited}
	In Section~\ref{section:TV} we showed that the total variation distance constitutes an instance of a $\phi$-divergence; see Proposition~\ref{prop:tv}. We can now demonstrate that the total variation distance is also an instance of an optimal transport discrepancy.  
	
	\begin{proposition}
		\label{prop:TV=OT}
		If $c(z, \hat z) = \ds{1}_{z\neq \hat z}$ for all $z,\hat z \in \cZ$, then we have
		$$\TV(\P, \hat \P) = \OT_{c}(\P, \hat \P) = \inf_{\gamma \in \Gamma(\P, \hat \P)} \, \gamma(Z \neq \hat Z). $$
	\end{proposition}
	\begin{proof}
		By Definition~\ref{def:TV}, the total variation distance satisfies
		\begin{align*}
			\TV(\P, \hat \P) 
			&= \sup \left\{ \left| \P(\cB) - \hat\P(\cB) \right|: \cB \, \subseteq \cZ\text{ is a Borel set } \right\} \\
			& = \sup \left\{ \; \P(\cB) - \hat\P(\cB)\, : \cB \, \subseteq \cZ\text{ is a Borel set } \right\}= \OT_{c}(\P, \hat \P),
		\end{align*}
		where the second equality holds because the complement of any Borel set is again a Borel set. The third equality follows from Corollary~\ref{cor:dual:OT:c:r} for~$r=0$, which applies because $c(z, \hat z) = \ds{1}_{d(z, \hat z) > 0}$ for any (continuous) metric~$d$ on~$\cZ$. Since $c(z, \hat z) = \ds{1}_{z\neq \hat z}$, we also have
		$$
		\OT_{c}(\P, \hat \P) = \inf_{\gamma \in \Gamma(\P, \hat \P)} \, \E_{\gamma} \left[ \ds{1}_{Z \neq \hat Z} \right] = \inf_{\gamma \in \Gamma(\P, \hat \P)} \, \gamma( Z \neq \hat Z).
		$$
		This observation completes the proof.
	\end{proof}
	
	Proposition~\ref{prop:TV=OT} readily implies that any total variation ambiguity set can also be viewed as a special instance of an optimal transport ambiguity set.
	
	\subsubsection{$\infty$-Wasserstein Ambiguity Sets}
	Section~\ref{sec:p-Wasserstein-balls} focuses exclusively on $p$-Wasserstein distances corresponding to finite exponents~$p\in[1,\infty)$. The $\infty$-Wasserstein distance requires special treatment.
	
	\begin{definition}[$\infty$-Wasserstein Distance]
		\label{def:infty-Wasserstein-distance}
		The $\infty$-Wasserstein distance $\W_\infty:\cP(\cZ) \times \cP(\cZ) \to [0,\infty]$ corresponding to a continuous metric $d(\cdot,\cdot)$ on~$\cZ$ is
		\begin{align}
			\label{prob:W_infty}
			\W_\infty(\P, \hat \P)  = \inf_{\gamma \in \Gamma(\P, \hat \P)} \esssup_\gamma\left[d(Z, \hat Z)\right],
		\end{align}
		where the essential supremum of $d(Z,\hat Z)$ under~$\gamma$ is given by
		\begin{align*}
			\esssup_\gamma\left[d(Z, \hat Z)\right] = \inf_{\tau\in\R} \left\{ \tau: \gamma (d(Z, \hat Z) > \tau) = 0 \right\}.
		\end{align*}
	\end{definition}
	
	Definition~\ref{def:infty-Wasserstein-distance} makes sense because the $\infty$-Wasserstein distance can be obtained from the $p$-Wasserstein distance in the limit when~$p$ tends to infinity.
	
	\begin{proposition}[{\citet{givens1984class}}]
		\label{prop:W:p:infty}
		For any $\P, \hat \P \in \cP(\cZ)$ we have
		\begin{align*}
			\W_\infty(\P, \hat \P) = \lim_{p \to \infty} \, \W_p(\P, \hat \P)
			= \sup_{p \geq 1} \, \W_p(\P, \hat \P).
		\end{align*}
	\end{proposition}
	\begin{proof}
		If $p \geq q \geq 1$, then $f(t)=t^{q/p}$ is concave on~$\R_+$. This implies that
		\begin{align*}
			\W_p(\P, \hat \P) 
			& = \inf_{\gamma \in \Gamma(\P, \hat \P)} \left(\E_\gamma[d(Z, \hat Z)^p]^{\frac{q}{p}}\right)^{\frac{1}{q}} \geq  \inf_{\gamma \in \Gamma(\P, \hat \P)} \left(\E_\gamma[d(Z, \hat Z)^q]\right)^{\frac{1}{q}} = \W_q(\P, \hat \P)
		\end{align*}
		thanks to Jensen's inequality. Hence, $\W_p(\P, \hat \P)$ is non-decreasing in the exponent~$p$ as long as $p \in [1, \infty)$. In addition, for any transportation plan~$\gamma \in \Gamma(\P, \hat \P)$ and exponent~$p \in [1, \infty)$, the definition of the essential supremum readily implies that
		\begin{align*}
			\left( \E_\gamma [d(Z, \hat Z)^p] \right)^{\frac{1}{p}} \leq \esssup_\gamma [d(Z, \hat Z)^p]^{\frac{1}{p}}= \esssup_\gamma [d(Z, \hat Z)].
		\end{align*}
		Minimizing both sides of this inequality across all $\gamma\in\Gamma(\P, \hat \P)$ further implies that $\W_p(\P, \hat \P)\leq \W_\infty(\P, \hat \P)$ for all $p\in[1,\infty)$. In summary, we may thus conclude that
		$$
		\lim_{p \to \infty} \W_p(\P, \hat \P) = \sup_{p \geq 1} \W_p(\P, \hat \P) \leq \W_\infty(\P, \hat \P).
		$$
		It remains to be shown that the last inequality holds in fact as an equality. To see this, 
		fix some tolerance~$\varepsilon > 0$. For any $p\in\N$, let $\gamma_p\in\Gamma(\P,\hat\P)$ be a coupling with $\E_{\gamma_p} [d(Z, \hat Z)^p]^{1/p}=\W_p(\P, \hat \P)$. Note that $\gamma_p$ exists because, as we will see in Corollary~\ref{cor:compact:Gamma} and Proposition~\ref{prop:semicontinuity} below, $\Gamma(\P,\hat\P)$ is weakly compact and $\E_\gamma [d(Z, \hat Z)^p]$ is weakly lower semicontinuous in~$\gamma$. Next, let $\{\gamma_{p(j)}\}_{j\in\N}$ be a subsequence that converges weakly to some coupling~$\gamma_\infty\in \Gamma(\P,\hat\P)$, which exists again because $\Gamma(\P,\hat\P)$ is weakly compact. We proceed by case distinction.
		
		\paragraph{Case~1.} If $\esssup_{\gamma_\infty} [d(Z, \hat Z)]$ is finite, define the open set 
		\[
		\cB = \left\{ (z, \hat z) \in \cZ \times \cZ: d(z, \hat z) > \esssup_{\gamma_\infty} [d(Z, \hat Z)] - \varepsilon \right\},
		\]
		and note that $\gamma_\infty(\cB) > 0$ by the definition of the essential supremum. We then find
		\begin{align*}
			\W_{p(j)}(\P, \hat \P)& \geq \left( \int_{\cB} d(z, \hat z)^{p(j)} \, \diff \gamma_{p(j)}(z, \hat z) \right)^{\frac{1}{p(j)}} \\
			&\geq \gamma_{p(j)}(\cB)^{\frac{1}{p(j)}} \left(\esssup_{\gamma_\infty} [d(Z, \hat Z)] - \varepsilon\right)\\& \geq \gamma_{p(j)}(\cB)^{\frac{1}{p(j)}} \left(\W_\infty(\P, \hat \P) - \varepsilon\right).
		\end{align*}
		Since $\cB$ is open and $\gamma_{p(j)}$ converges weakly to~$\gamma_\infty$ as $j$ grows, the Portmanteau theorem \citep[Theorem~2.1\,(iiv)]{billingsley2013convergence} implies that $\liminf_{j\to\infty} \gamma_{p(j)}(\cB)\geq \gamma_\infty(\cB)>0$. Thus, $\gamma_{p(j)}(\cB)^{1/p(j)}$ converges to~$1$ as $j$ grows, and we obtain
		\[
		\lim_{p \to \infty} \W_p(\P, \hat \P) \geq \W_\infty(\P, \hat \P) - \varepsilon.
		\]
		As this inequality holds for any tolerance~$\varepsilon>0$, the above reasoning finally implies that $\W_p(\P, \hat \P)$ converges indeed to~$\W_\infty(\P, \hat \P)$ for large~$p$.
		
		\paragraph{Case~2.} If $\esssup_{\gamma_\infty} [d(Z, \hat Z)]=\infty$, then we replace $\esssup_{\gamma_\infty} [d(Z, \hat Z)]$ in the definition of the open set~$\cB$ with an arbitrarily large constant. Proceeding as in Case~1 eventually reveals that $\lim_{p \to \infty} \W_p(\P, \hat \P) = \W_\infty(\P, \hat \P)=\infty$.
	\end{proof}
	
	To develop some intuition for Proposition~\ref{prop:W:p:infty}, consider the optimal transport problem in the definition of $\W_p(\P, \hat \P)$. If~$p>1$, then the cost $c(z,\hat z)=d(z,\hat z)^p$ of transporting one unit of probability mass from~$\hat z$ to~$z$ grows superlinearly with the distance~$d(z,\hat z)$. Hence, parts of the distribution~$\hat\P$ that are transported further under an optimal transportation plan contribute more to~$\W_p(\P, \hat \P)$. In addition, as~$p$ tends to infinity, eventually only the portion of the distribution~$\hat\P$ that is transported the furthest has an impact on $\W_\infty(\P, \hat \P)$. Even more, only the largest transportation {\em distance} matters, whereas the {\em amount} of probability mass transported is irrelevant. 
	
	Despite Proposition~\ref{prop:W:p:infty}, the optimal transport problems in the definitions of the Wasserstein distances of order~$p<\infty$ and of order~$p=\infty$ are fundamentally different. Indeed, if $p<\infty$, then the objective function $\E_\gamma[d(Z, \hat Z)^p]$ of the optimal transport problem is linear in the tansportation plan~$\gamma$. If~$p=\infty$, on the other hand, then the objective function $\esssup_{\gamma}[d(Z,\hat Z)]$ is not even convex, but rather quasi-convex, in~$\gamma$ \citep[Lemma~2.2]{jylha2015optimal}; see also \citep{champion2008wasserstein}. Thus, $\infty$-Wasserstein distances require a more subtle treatment. 
	
	The next proposition relates the $\infty$-Wasserstein distance to a standard optimal transport problem. Therefore, it has computational relevance. 
	
	\begin{proposition}
		\label{prop:dual-W_infty}
		If the transportation cost function~$c_r$ corresponding to $r\geq 0$ is defined through $c_r(z,\hat z)=\ds 1_{d(z,\hat z)>r}$ for all $z,\hat z\in\cZ$, then we have
		\begin{align*}
			\W_\infty (\P, \hat{\P}) 
			= \inf \left\{r \geq 0: \OT_{c_r}(\P, \hat{\P}) \leq 0 \right\}.
		\end{align*}
	\end{proposition}
	\begin{proof}
		Recall that $\OT_{c_r}(\P, \hat \P) = \inf \{ \E_\gamma[c_r(Z, \hat Z)]:\gamma \in \Gamma(\P, \hat \P)\}$. Note that the underlying optimal transport problem is solvable because $\Gamma(\P,\hat\P)$ is weakly compact and because $\E_\gamma [d(Z, \hat Z)^p]$ is weakly lower semicontinuous in~$\gamma$ thanks to Corollary~\ref{cor:compact:Gamma} and Proposition~\ref{prop:semicontinuity} below, respectively. Therefore, we have
		\begin{align*}
			& \inf \left\{r \geq 0: \OT_{c_r}(\P, \hat{\P}) \leq 0 \right\} \\
			&\quad = \inf \left\{ r \geq 0: \exists \gamma \in \Gamma(\P, \hat \P) \text{ with } \E_\gamma[c_r(Z, \hat Z)] = 0 \right\} \\
			&\quad = \inf_{\gamma \in \Gamma(\P, \hat \P),\, r \in \R_+} \left\{ r : \gamma [d(Z, \hat Z) > r] = 0 \right\} = \W_\infty(\P, \hat \P),
		\end{align*}
		where the first equality holds because $\OT_{c_r}(\P, \hat \P)$ is non-negative and because the underlying optimal transport problem is solvable. The second equality follows from the definitions of $c_r$ and the $\infty$-Wasserstein distance.
	\end{proof}
	
	Combining Proposition~\ref{prop:dual-W_infty} with Corollary~\ref{cor:dual:OT:c:r} immediately yields the following equivalent characterization of the $\infty$-Wasserstein distance.
	
	\begin{corollary}[{\citet{givens1984class}}]
		\label{cor:dual-W_infty}
		The $\infty$-Wasserstein distance satisfies
		\begin{align*}
			\W_\infty (\P, \hat{\P}) = \inf \left\{r \geq 0: \P(\cB)  \leq \hat\P(\cB_r) \text{ for all Borel sets } \cB\subseteq\cZ\right\},
		\end{align*}
		where the $r$-neighborhood $\cB_r$ is defined in~\eqref{eq:Br}.
	\end{corollary}
	The $\infty$-Wasserstein ambiguity set of radius~$r\geq 0$ around $\hat\P\in\cP(\cZ)$ is defined as
	\begin{align}
		\label{eq:infty-wasserstein-ambiguity-set}
		\cP= \left\{ \P \in \cP(\cZ) \, : \, \W_\infty(\P, \hat \P) \leq r \right\}.
	\end{align}
	Proposition~\ref{prop:dual-W_infty} implies that $\cP$ coincides with an optimal transport ambiguity set with transportation cost function $c_r(z,\hat z)=\ds 1_{d(z,\hat z)>r}$, that is, we have
	\begin{align*}
		\cP= \left\{ \P \in \cP(\cZ) \, : \, \OT_{c_r}(\P, \hat{\P}) \leq 0 \right\}.
	\end{align*}
	DRO with $\infty$-Wasserstein ambiguity sets has strong connections to adversarial machine learning \citep{gao2017wasserstein,garcia2022regularized,trillos2022adversarial,garcia2023analytical,bungert2023geometry,bungert2024mean,gao2024wasserstein,pydi2024many,frank2024adversarial,frank2024existence} and kernel density estimation \citep{xu2012distributional}. In addition, $\infty$-Wasserstein ambiguity sets are used in two- and multi-stage stochastic programming \citep{xie2020tractable,bertsimas2020two,bertsimas2023data}, portfolio optimization~\citep{nguyen2021robustifying}, and robust learning \citep{nguyen2020distributionallyrobust,wang2024wasserstein}.

	\subsection{Other Ambiguity Sets}
	
	There exist several ambiguity sets that cannot be classified as moment, $\phi$-divergence or optimal transport ambiguity sets. In the following we offer a brief overview of these ambiguity sets without providing extensive mathematical details. 
	
	\subsubsection{Marginal Ambiguity Sets}
	\label{sec:marginal-ambiguity-sets}
	Marginal ambiguity sets specify the marginal distributions of multiple subvectors of~$Z$ without detailing their joint distribution. The simplest example of a marginal ambiguity set is the Fr\'echet ambiguity set, which specifies the marginal distributions of all individual components of~$Z$ but provides no information about their copula. Thus, the Fr\'echet ambiguity set is parametrized by $d$ marginal cumulative distribution functions $F_i:\R\rightarrow [0,1]$, $i\in[d]$, and can be represented as
	\begin{align}
		\label{eq:frechet}
		\cP =\left\{ \P \in \cP(\R^d): \P (Z_i \leq z_i) = F_i (z_i) \; \;\forall z_i \in \R, \,\forall i \in [d] \right\}.
	\end{align}
	Here, $F_i$ is an arbitrary cumulative distribution function, that is, a right-continuous, non-decreasing function with $\lim_{z_i\to-\infty}F_i(z_i)=0$ and $\lim_{z_i\to+\infty}F_i(z_i)=1$. Fr\'echet ambiguity sets are relevant for probabilistic logic. Imagine that each~$Z_i$ represents a binary variable that evaluates to~$1$ if a certain event occurs and to~$0$ otherwise, and assume that the probability of each event is known, whereas the joint distribution of all events is unknown. In this setting, \citet{boole1854investigation} was interested in computing bounds on the probability of a composite event encoded by a Boolean function of the variables~$Z_i$, $i\in[d]$. Almost a century later, \citet{frechet1935generalisation} derived explicit inequalities for the probabilities of such composite events, which are now called Fr\'echet inequalities. Note that these Fr\'echet inequalities can be obtained by minimizing or maximizing the probability of the composite event over all distributions in a Fr\'echet ambiguity set with Bernoulli marginals. More recently, there has been growing interest in generalized Fr\'echet inequalities, which bound the risk of general (not necessarily Boolean) functions of~$Z$ with respect to all distributions in a Fr\'echet ambiguity set with general (not necessarily Bernoulli) marginals. For example, a wealth of Fr\'echet inequalities for the risk of a sum of random variables have emerged in finance and risk management \citep{ruschendorf1983solution,ruschendorf1991frechet,embrechts2006bounds,wang2011complete,wang2013bounds,puccetti2013sharp,van2016frechet,blanchet2020convolution}. In addition, \citet{natarajan2009persistency} derive sharp bounds for the worst-case expectation of a piecewise affine functions over a Fr\'echet ambiguity set. We highlight that Fr\'echet ambiguity sets are also relevant because they coincide with the feasible sets of multi-marginal optimal transport problems, which can sometimes be solved in polynomial time \citep{pass2015multi,altschuler2023polynomial,natarajan2023distributionally}. 
	
	General marginal ambiguity sets specify the marginal distributions of several (possibly overlapping) subsets of the set $\{Z_i:i\in[d]\}$ of random variables. However, checking whether such an ambiguity set is non-empty is NP-complete even if each~$Z_i$ is a Bernoulli random variable and each subset accommodates merely two elements \citep{honeyman1980testing,georgakopoulos1988probabilistic}. Computing worst-case expectations over marginal ambiguity sets is thus intractable unless the subsets of random variables with known marginals are disjoint \citep{doan2012complexity} or if the corresponding overlap graph displays a running intersection property \citep{doan2015robustness}.
	
	Marginal ambiguity sets are attractive because, given limited statistical data, it is far easier to estimate low-dimensional marginals than their global dependence structure. However, even univariate marginals cannot be estimated exactly. For this reason, several researchers study marginal ambiguity sets that provide only limited information about the marginals such as bounds on marginal moments or marginal dispersion measures \citep{bertsimas2004probabilistic,bertsimas2006persistence,bertsimas2006tight,chen2010cvar,mishra2012choice,natarajan2018asymmetry}. 
	
	A related stream of literature focuses on ambiguity sets under which the random variables~$Z_i$, $i\in[d]$, are {\em independent} and governed by ambiguous marginal distributions. For example, the Hoeffding ambiguity set contains all joint distributions on a box with independent (and completely unknown) marginals, whereas the Bernstein ambiguity set contains all distributions from within the Hoeffding ambiguity set subject to marginal moment bounds \citep{nemirovski2007convex,hanasusanto2015distributionally}. Bernstein ambiguity sets that constrain the mean as well as the mean-absolute deviation of each marginal are used to derive safe tractable approximations for distributionally robust chance constrained programs \citep{postek2018robust}, two-stage integer programs \citep{postek2018robust,postek2019approximation}, and queueing systems \citep{wang2024distributionally}. 
	
	DRO with marginal ambiguity sets has close connections to submodularity and to the theory of comonotonicity in risk management \citep{tchen1980inequalities,ruschendorf2013mathematical,bach2013learning,bach2019submodular,natarajan2023distributionally,long2024supermodularity}. It has a broad range of diverse applications ranging from discrete choice modeling \citep{natarajan2009persistency,mishra2014theoretical,chen2022distributionally,ruan2022nonparametric}, queuing theory \citep{van2022mad}, transportation \citep{wang2020distributionally,shehadeh2023distributionally}, chance constrained programming \citep{xie2022optimized}, scheduling \citep{mak2015appointment}, inventory management \citep{liu2024newsvendor}, the analysis of complex networks \citep{chen2020correlation,van2021robust,brugman2022sharpest} and mechanism design \citep{carroll2017robustness,gravin2018separation,chen2023screening,wang2024minimax,wang2024power}, etc. For further details we refer to the comprehensive monograph by \citet{natarajan2021optimization}.

	\subsubsection{Mixture Ambiguity Sets and Structural Ambiguity Sets}
	
	Let~$\Theta\subseteq \R^m$ be a Borel set and~$\P_\theta\in\cP(\cZ)$ a parametric distribution that is uniquely determined by~$\theta\in\Theta$. Assume that $\P_\theta(Z\in\cB)$ is a Borel measurable function of~$\theta$ for every fixed Borel set~$\cB\subseteq \cZ$. The parametric distribution family $\{\P_\theta:\theta\in\Theta\}$ can then be used as a mixture family, which induces the mixture ambiguity set
	\begin{align}
		\label{eq:mixture-ambiguity-set}
		\cP = \left\{ \int_{\Theta} \P_\theta \, \diff \Q(\theta): \Q\in \cP(\Theta)  \right\}.
	\end{align}
	Thus, $\cP$ contains all distributions that can be represented as mixtures of the distributions~$\P_\theta$, $\theta\in\Theta$. Put differently, for every $\P\in\cP$ there exists a mixture distribution~$\Q\in\cP(\Theta)$ with $\P(Z\in\cB) = \int_{\Theta} \P_\theta(Z\in\cB) \, \diff \Q(\theta)$ for all Borel sets~$\cB\subseteq \cZ$. This construction ensures that $\cP\subseteq \cP(\cZ)$ is convex. For example, if $\P_\theta$ is a Gaussian distribution whose mean and covariance matrix are encoded by~$\theta$, then~$\cP$ contains (possibly continuous) mixtures of Gaussians. Mixture ambiguity sets corresponding to compact parameter sets~$\Theta$ are studied by \citet{lasserre2021distributionally}, who develop a semidefinite programming-based hierarchy of increasingly tight inner approximations for the feasible set of a distributionally robust chance constraint. 
	
	Note that~$\cP$ can be viewed as the convex hull of the parametric distribution family $\{\P_\theta:\theta\in\Theta\}$. A classical result in convex analysis due to Minkowski asserts that any compact convex subset of a Euclidean vector space coincides with the convex hull of its extreme points. Choquet theory \citep{phelps2001lectures} seeks similar extreme point representations for convex compact subsets of topological vector spaces. For example, if~$\{\P_\theta:\theta\in\Theta\}$ is the set of all extreme distributions of a weakly compact convex ambiguity set~$\cP$, then~\eqref{eq:mixture-ambiguity-set} constitutes a Choquet representation of~$\cP$.
	
	Families of distributions that share certain structural properties sometimes admit a Choquet representation of the form~\eqref{eq:mixture-ambiguity-set}. For example, let~$\cP$ be the family of all distributions~$\P\in\cP(\R^d)$ that are point symmetric about the origin. This means that $\P(Z\in\cB)=\P(-Z\in\cB)$ for every Borel set~$\cB\subseteq\R^d$. One can then show that all extreme distributions of~$\cP$ are representable as~$\P_\theta=\half\delta_{+\theta}+\half\delta_{-\theta}$ for some~$\theta\in\R^d$. Thus, $\cP$ admits a Choquet representation of the form~\eqref{eq:mixture-ambiguity-set}. As another example, let~$\cP$ be the family of all distributions~$\P\in\cP(\R^d)$ that are $\alpha$-unimodal about the origin for some~$\alpha>0$. This means that $t^\alpha\P(Z\in\cB/t)$ is non-decreasing in~$t>0$ for every Borel set~$\cB\subseteq\R^d$. One can then show that every extreme distribution of~$\cP$ is a distribution~$\P_\theta$ supported on the line segment from~$0$ to~$\theta\in\R^d$ with the property that $\P_\theta(\|Z\|_2\leq t\|\theta\|_2)=t^\alpha$ for all~$t\in[0,1]$. Thus, $\cP$ admits again a Choquet representation of the form~\eqref{eq:mixture-ambiguity-set}. We remark that $d$-unimodal distributions on~$\R^d$ are also called star-unimodal. One readily verifies that a distribution with a continuous probability density function is star-unimodal if and only if the density function is non-increasing along each ray emanaging from the origin. In addition, one can show that the family of all $\alpha$-unimodal distributions converges---in a precise sense---to the family of {\em all} possible distributions on~$\R^d$ as~$\alpha$ tends to infinity. For more information on structural distribution families and their Choquet representations we refer to 
	\citep{dharmadhikari1988unimodality}.

	The moment ambiguity sets of Section~\ref{sec:moment-ambiguity-sets} are known to contain discrete distributions with only very few atoms;  see Section~\ref{sec:finite-convex-reformulations}. However, uncertainties encountered in real physical, technical or economic systems are unlikely to follow such discrete distributions. Instead, they are often expected to be unimodal. Hence, an effective means to eliminate the pathological discrete distributions from a moment ambiguity set is to intersect it with the structural ambiguity set of all $\alpha$-unimodal distributions for some~$\alpha>0$. \citet{popescu2005semidefinite} combines ideas from Choquet theory and sums-of-squares polynomial optimization to approximate worst-case expectations over the resulting intersection ambiguity sets by a hierarchy of increasingly accurate bounds, each of which is computed by solving a tractable semidefinite program. \citet{van2016generalized} and \citet{van2019distributionally} extend this approach and establish {\em exact} semidefinite programming reformulations for the worst-case probability of a polyhedron and the worst-case conditional value-at-risk of a piecewise linear convex loss function across all $\alpha$-unimodal distributions in a Chebyshev ambiguity set; see also \citep{hanasusanto2015perspective}. \citet{li2019ambiguous} demonstrate that these semidefinite programming reformulations can sometimes be simplified to highly tractable second-order cone programs. Complementing moment information with structural information generally leads to less conservative DRO models as \citet{li2016distributionally} demonstrate in the context of a power system application. \citet{lam2021orthounimodal} consider another basic notion of distributional shape known as orthounimodality and build a corresponding Choquet representation to address multivariate extreme event estimation. More recently, \citet{lam2024shape} combine Choquet theory with importance sampling and likelihood ratio techniques for modeling distribution shapes.

	\subsubsection{Non-Standard $\phi$-Divergence and Optimal Transport Ambiguity Sets}
	A wealth of non-standard $\phi$-divergences and optimal transport discrepancies have been proposed to measure the dissimilarity between probability distributions. They offer great flexibility in designing ambiguity sets with complementary computational and statistical properties. Non-standard distance measures notably include smoothed $\phi$-divergences \citep{zeitouni1991universal,yang2018robust,liu2023smoothed} as well as combinations of $\phi$-divergences and optimal transport discrepancies \citep{reid2011information,dupuis2022formulation,vanparys2024efficient}. In addition, they include coherent Wasserstein distances \citep{li2022general} and Sinkhorn divergences \citep{wang2021sinkhorn} as well as divergences based on causal optimal transport
	\citep{analui2014distributionally,pflug2014multistage,yang2022decision,aroradata2022data,jiang2024sensitivity}, outlier-robust optimal transport \citep{nietert2024outlier,nietert2024robust}, mixed-feature optimal transport \citep{selvi2022wasserstein,belbasi2023s}, cluster-based optimal transport \citep{wang2022mean}, partial optimal transport \citep{esteban2022distributionally}, sliced optimal transport \citep{olea2022out}, multi-marginal optimal transport \citep{lau2022wasserstein,trillos2023multimarginal,rychener2024wasserstein}, and constrained conditional moment optimal transport \citep{li2022tikhonov,blanchet2023unifying,sauldubois2024first}.   
	
	\subsubsection{Ambiguity Sets Based on Integral Probability Metrics}
	
	Let~$\cF$ be a family of Borel measurable test functions $f:\cZ\to\R$ such that $f\in\cF$ if and only if~$-f\in\cF$. The integral probability metric generated by~$\cF$ is defined~via
	\begin{align*}
		\D_{\cF}(\P, \hat \P) = \sup_{f \in \cF} \; \int_{\cZ} f(z) \, \diff \P(z) - \int_{\cZ} f(\hat z) \, \diff \hat \P(\hat z)
	\end{align*}
	for all distributions~$\P,\hat\P\in\cP(\cZ)$ under which all test functions~$f\in\cF$ are integrable. The underlying maximization problem probes how well the test functions can distinguish~$\P$ from~$\hat\P$. By construction, $\D_{\cF}$ constitutes a pseudo-metric, that is, it is non-negative and symmetric (because $\cF=-\cF$), vanishes if its arguments match, and satisfies the triangle inequality. In addition, $\D_{\cF}$ becomes a proper metric if~$\cF$ separates distributions, in which case $\D_{\cF}(\P, \hat \P)$ vanishes only if $\P=\hat\P$. The ambiguity set of radius~$r\geq 0$ around~$\hat\P\in\cP(\cZ)$ with respect to~$\D_{\cF}$ is defined as
	\[
	\cP=\left\{ \P\in\cP(\cZ) : \D_{\cF}(\P, \hat \P)\leq r\right\}.
	\]
	The proof of Proposition~\ref{prop:tv} reveals that the total variation distance is the integral probability metric generated by all Borel functions $f: \cZ \to [-1/2, 1/2]$; see~\eqref{eq:TV:IPM2}. The Kantorovich-Rubinstein duality established in Corollary~\ref{cor:kanorovich-rubinstein-duality} further shows that the $1$-Wasserstein distance is the integral probability metric generated by all Lipschitz continuous  functions $f :\cZ\to\R$ with $\lip(f) \leq 1$. In addition, if~$\cH$ is a reproducing kernel Hilbert space of Borel functions $f:\cZ\to\R$ with Hilbert norm~$\|\cdot\|_\cH$, then the maximum mean discrepancy distance corresponding to~$\cH$ is the integral probability metric generated by the standard unit ball $\cF = \{ f \in \cH : \| f \|_\cH \leq 1 \}$ in~$\cH$. Maximum mean discrepancy ambiguity sets are studied in \citep{staib2019distributionally,zhu2020mmd,zhu2021kernel,zeng2022generalization,iyengar2022hedging}. \citet{husain2020distributional} uncovers a deep connection between DRO problems and regularized empirical risk minimization problems, which holds whenever the ambiguity set is defined via an integral probability metric.

	\section{Topological Properties of Ambiguity Sets}
	\label{sec:topology}
	A fundamental question of theoretical as well as practical interest is whether nature's subproblem in~\eqref{eq:primal:dro} is solvable or, in other words, whether the inner supremum in~\eqref{eq:primal:dro} is attained. In this section we will investigate under what conditions the Weierstrass extreme value theorem applies to nature's subproblem. That is, we will develop easily checkable conditions under which the ambiguity set~$\cP$ is weakly compact and the expected loss~$\E_\P[\ell(x,Z)]$ is weakly upper semicontinuous in~$\P$. Throughout this discussion, we assume that~$\cZ$ is a closed subset of~$\R^d$.
	
	A classical result by Baire asserts that a function on the real line is lower semicontinuous if and only if it can be represented as the pointwise supremum of a non-decreasing sequence of continuous functions \citep{baire1905semicontinuity}. Below we will use the following multivariate generalization of this result.
	
	\begin{lemma}[{\citet[p.~132]{stromberg2015introduction}}]
		\label{lem:baire-semicontinuity}
		A function $f:\cZ\rightarrow (-\infty, +\infty]$ is lower semicontinuous if and only if there is a non-decreasing sequence of continuous functions~$f_i: \cZ\rightarrow \R$, $i\in\N$, with $f(z)=\sup_{i\in\N} f_i(z)$ for all~$z\in\cZ$.
	\end{lemma}
	If~$f$ is bounded from below, then the continuous functions~$f_i$ can be assumed to be uniformly bounded. Indeed, if $f(z)\geq 0$, say, then the continuous function $f_i(z)$ can be replaced with the bounded continuous function $f'_i(z)=\min\{\max\{f_i(z),0\}, i \}$. The sequence $f'_i$, $i\in\N$, is still non-decreasing and converges pointwise to~$f$. 
	
	\begin{definition}[Weak Convergence of Probability Distributions]
		\label{def:weak-convergence}
		A sequence of probability distributions $\P_j\in\cP(\cZ)$, $j\in\N$, converges weakly to~$\P\in\cP(\cZ)$ if for every bounded and continuous function $f:\cZ\rightarrow\R$ we have
		\[
		\lim_{j\in\N} \E_{\P_j} \left[ f (Z) \right] = \E_{\P} \left[ f (Z) \right].
		\]
	\end{definition}
	
	There is a close link between the continuity properties of the expected value of~$f(Z)$ with respect to the distribution~$\P$ and the continuity properties of~$f$. Recall that a function $F:\cP(\cZ)\to\overline\R$ is weakly continuous if  $\lim_{i\to\infty}F(\P_i)=F(\P)$ for every sequence $\P_i\in\cP(\cZ)$, $i\in\N$, that converges weakly to~$\P$. Weak lower and upper semicontinuity are defined analogously in the obvious way.
	
	\begin{proposition}[Continuity of Expected Values]
		\label{prop:semicontinuity}
		If $f:\cZ\rightarrow [-\infty,+\infty]$  is lower semicontinuous and bounded from below, then $\E_\P [ f (Z)]$ is weakly lower semicontinuous in~$\P\in\cP(\cZ)$. Conversely, if $f$ is upper semicontinuous and bounded from above, then $\E_\P [ f (Z)]$ is weakly upper semicontinuous in~$\P\in\P(\cZ)$. Finally, if~$f$ is continuous and bounded, then $\E_\P [ f (Z)]$ is weakly continuous in~$\P\in\P(\cZ)$.
	\end{proposition}
	\begin{proof}
		Assume first that $f$ is lower semicontinuous and bounded from below. In the following, we assume without loss of generality that~$f$ is in fact non-negative. Then, by Lemma~\ref{lem:baire-semicontinuity}, there is a non-decreasing sequence of bounded, continuous and non-negative functions~$f_i$, $i\in\N$, with $f(z)=\sup_{i\in\N} f_i(z)$. If~$\P_j\in\cP(\cZ)$, $j\in\N$, is any sequence of distributions that converges weakly to~$\P$, then we find
		\begin{align*}
			\liminf_{j\in\N} \E_{\P_j} \left[ f (Z) \right]&= \sup_{k\in\N} \inf_{j\geq k} \E_{\P_j} \left[ \sup_{i\in\N} f_i(Z) \right] \\
			&= \sup_{k\in\N} \inf_{j\geq k}  \sup_{i\in\N} \E_{\P_j} \left[ f_i(Z) \right] \\
			&\geq \sup_{i\in\N}  \sup_{k\in\N} \inf_{j\geq k}  \E_{\P_j} \left[ f_i(Z) \right] \\
			&= \sup_{i\in\N} \E_{\P} \left[ f_i(\xi) \right]  = \E_{\P} \left[ f(Z) \right].
		\end{align*}
		Here, both the second and the last equality follow from the monotone convergence theorem, which applies because each~$f_i$ is bounded and thus integrable with respect to any probability distribution and because the~$f_i$, $i\in\N$, form a non-decreasing sequence of non-negative functions. The inequality follows from the interchange of the supremum over~$i$ and the infimum over~$j$, and the third equality holds because~$\P_j$ converges weakly to~$\P$ and because~$f_i$ is continuous and bounded. This shows that~$\E_\P [ f (Z)]$ is weakly lower semicontinuous in~$\P$. 
		
		The proofs of the assertions regarding weak upper semicontinuity and weak continuity are analogous and therefore omitted for brevity.
	\end{proof}
	
	In the following we equip the family $\cP(\cZ)$ of all probability distributions on~$\cZ$ with the weak topology, which is generated by the open sets
	\[
	U_{f,\delta}=\left\{ \P\in \cP(\cZ):  \left|\E_{ \P} [f(Z)]\right|<\delta\right\}
	\]
	encoded by any continuous bounded function~$f:\cZ\to\R$ and tolerance~$\delta>0$. The weak topology on $\cP(\cZ)$ is metrized by the Prokhorov metric \cite[Theorem~6.8]{billingsley2013convergence}, and therefore the notions of sequential compactness and compactness are equivalent on $\cP(\cZ)$; see, {\em e.g.}, \cite[Theorem~28.2]{munkres2000topology}.
	
	\begin{definition}[Tightness]
		A family $\cP \subseteq \cP(\cZ)$ of distributions is tight if for any tolerance $\varepsilon > 0$ there is a compact set ${\cC} \subseteq \cZ$ with $\P(Z\not\in {\cC}) \leq\varepsilon$ for all $\P\in\cP$.
	\end{definition}
	
	A classical result by Prokhorov asserts that a distribution family is weakly compact if and only if it is tight and weakly closed. Prokhorov's theorem is the key tool to show that an ambiguity set is weakly compact. We state it without~proof.
	
	\begin{theorem}[{\citet[Theorem~5.1]{billingsley2013convergence}}]
		\label{thm:prokhorov}
		A family $\cP \subseteq \cP(\cZ)$ of distributions is weakly compact if and only if it is tight as well as weakly closed.
	\end{theorem}
	
	In the following we revisit the ambiguity sets of Section~\ref{sec:ambiguity-sets} one by one and determine under what conditions they are tight, weakly closed and weakly compact. 
	
	\subsection{Moment Ambiguity Sets}
	
	The support-only ambiguity sets arguably form the simplest class of moment ambiguity sets because they impose no moment conditions at all. In fact, {\em all} other ambiguity sets considered in this paper are subsets of a support-only ambiguity set.
	
	\begin{proposition}[Support-Only Ambiguity Sets]
		\label{prop:support:only:compact}
		The set $\cP(\cZ)$ of all distributions supported on~$\cZ\subseteq \R^d$ is weakly compact if and only if $\cZ$ is compact.
	\end{proposition}
	\begin{proof}
		Note first that $\cP(\cZ)$ is tight if and only if~$\cZ$ is bounded. Indeed, if~$\cZ$ is bounded, then it is compact because~$\cZ$ is closed thanks to our blanket assumption. Given any~$\varepsilon>0$, we may thus set $\cC = \cZ$, which ensures that $\P( Z \notin \cC) = 0 \leq \varepsilon$ for all $\P \in \cP(\cZ)$. Hence, $\cP(\cZ)$ is tight. If~$\cZ$ is unbounded, on the other hand, then $\cP(\cZ)$ trivially fails to be tight. Indeed, for any compact set~$\cC \subseteq \cZ$, the complement $\cZ \backslash \cC$ is non-empty because~$\cC$ is bounded and~$\cZ$ is not. Hence, there exists a probability distribution $\P \in \cP(\cZ)$ supported on $\cZ\backslash \cC$ such that $\P(Z \notin \cC) = 1$.
		
		Next, note that~$\cP(\cZ)$ is weakly closed if and only if~$\cZ$ is closed. To see this, assume first that $cZ$ is closed. Recall that the indicator function $\delta_\cZ$ is defined by $\delta_\cZ(z)=0$ if $z\in\cZ$ and $\delta_\cZ(z)=+\infty$ if $z\notin\cZ$. Thus, it is lower semicontinuous and bounded below. By Proposition~\ref{prop:semicontinuity}, $\E_\P[\delta_\cZ(Z)]$ is therefore weakly lower semicontinuous in~$\P$. If $\P_j\in\cP(\cZ)$, $j\in\N$, converges weakly to~$\P$, we then have
		\begin{align*}
			0=\liminf_{j\in\N} \E_{\P_j}[\delta_\cZ(Z)] \geq \E_\P[\delta_\cZ(Z)] \geq 0,
		\end{align*}
		where the equality holds because~$\P_j$ is supported on~$\cZ$ for every~$j\in \N$, and the first inequality follows from weak lower semicontinuity. This implies that~$\P \in \cP(\cZ)$, and thus~$\cP(\cZ)$ is weakly closed. Conversely, assume that~$\cP(\cZ)$ is weakly closed, and consider a sequence~$z_j\in\cZ$, $j\in\N$, converging to~$z$. Then, the sequence of Dirac distributions $\delta_{z_j}$, $j\in\N$, converges weakly to~$\delta_z$, and thus we find
		\begin{align*}
			0=\liminf_{j\in\N} \E_{\delta_{z_j}}[\delta_\cZ(Z)] \geq \E_{\delta_z}[\delta_\cZ(Z)] \geq 0.
		\end{align*}
		Here, the first inequality holds again because $\E_\P[\delta_\cZ(Z)]$ is weakly lower semicontinuous in~$\P$. This implies that $\E_{\delta_z}[\delta_\cZ(Z)]=0$, which holds if and only if~$z\in\cZ$. Thus, $\cZ$ is closed. Given these insights, the claim follows from Theorem~\ref{thm:prokhorov}.
	\end{proof}
	
	By using Proposition~\ref{prop:support:only:compact}, we can now show that a moment ambiguity set of the form~\eqref{eq:moment-ambiguity-set} is weakly compact whenever the underlying support set~$\cZ$ is compact, the moment function~$f$ is continuous and the uncertainty set~$\cF$ is closed.
	
	\begin{proposition}[Moment Ambiguity Sets]   \label{prop:moment-ambiguity-sets-compact}
		If~$\cZ\subseteq\R^d$ is a compact support set, $f:\cZ\to\R^m$ is a continuous moment function and $\cF \subseteq \R^m$ is a closed uncertainty set, then the moment ambiguity set~$\cP$ defined in~\eqref{eq:moment-ambiguity-set} is weakly compact.
	\end{proposition}
	
	\begin{proof}
		As~$\cZ$ is compact, the support-only ambiguity set~$\cP(\cZ)$ is weakly compact by virtue of Proposition~\ref{prop:support:only:compact}. Consequently, $\cP(\cZ)$ is tight and weakly closed. This readily implies that~$\cP$ is tight as a subset of a tight set remains tight. Proposition~\ref{prop:semicontinuity} further implies that $\E_\P[f(Z)]$ is weakly continuous in~$\P$. As~$\cF$ is closed and as the pre-image of any closed set under a continuous transformation is closed, we may conclude that $\cP_f=\{\P\in\cP(\R^d):\E_\P[f(Z)]\in\cF\}$ is weakly closed. Hence, $\cP=\cP(\cZ)\cap\cP_f$ is weakly closed as the intersection of two weakly closed sets. Given these insights, the claim follows readily from Theorem~\ref{thm:prokhorov}.
	\end{proof}
	
	The conditions of Proposition~\ref{prop:moment-ambiguity-sets-compact} are only sufficient but not necessary for weak compactness. The next examples show that moment ambiguity sets can be tight or weakly compact even if the support set~$\cZ$ or the moment function~$f$ are unbounded.
	
	\begin{example}[Markov Ambiguity Sets]
		The Markov ambiguity set~\eqref{eq:Markov} fails to be tight if~$\cZ=\R^d$. For example, if~$\cZ=\R$ and $\mu=0$, then for every compact set~${\cC}\subseteq\R$ there is a constant~$R>0$ such that the two-point distribution $\P=\half\delta_{-R}+ \half \delta_R$ is fully supported on the complement of~${\cC}$. However, the Markov ambiguity set~$\cP$ becomes tight if $\cZ=\R_+$ and $\mu=1$. Indeed, in this case Markov's inequality implies that $\P(Z\not\in {\cC}) \leq\varepsilon$ for every $\P\in\cP$ and $\varepsilon>0$ if we define ${\cC}$ as the compact interval $[0,1/\varepsilon]$. Even in this case, however, $\cP$ fails to be weakly closed. Indeed, the distributions $\P_i=\frac{i}{i+1}\delta_0+\frac{1}{i+1}\delta_{i+1}$ belong to~$\cP$ for all~$i\in\N$, but their weak limit~$\P=\delta_0$ is no member of~$\cP$. If~$\cZ$ is convex, one can extend this reasoning in the obvious way to show that~$\cP$ is weakly compact if and only if~$\cZ$ is compact.
	\end{example}
	
	The next example shows that Chebyshev ambiguity sets are tight irrespective of~$\cZ$. Nevertheless, they are not always weakly compact.
	
	\begin{example}[Chebyshev Ambiguity Sets]
		\label{ex:chebyshev-not-compact}
		The Chebyshev ambiguity set~$\cP$ defined in~\eqref{eq:Chebyshev} is always tight. To see this, assume without loss of generality that~$\mu=0$ and $M=I_d$, which can always be enforced by applying an affine coordinate transformation. Given any $\varepsilon>0$, we can define a compact set ${\cC} =\{z\in\cZ: \|z\|_2\leq \sqrt{d/\varepsilon}\}$. It is then easy to see that any distribution $\P\in\cP$ satisfies
		\begin{align*}
			\P(Z\not\in {\cC}) = \P\left(\|Z\|_2 > \sqrt{d/\varepsilon} \right) \leq \E_\P \left[ \|Z\|_2^2\cdot  \varepsilon/d \right] = \varepsilon,
		\end{align*}
		where the inequality holds because the quadratic function $q(z)= \| z\|_2^2\cdot \varepsilon/d$ majorizes the characteristic function of~$\cZ\backslash \cC$. Hence, $\cP$ is indeed tight. However, $\cP$ is not necessarily weakly closed. To see this, suppose that~$d=1$ and that~$\cZ=\R$. In this case the distributions $\P_i= \frac{1}{2i^2}\delta_{-i} + \frac{i^2-1}{i^2}\delta_0 + \frac{1}{2i^2}\delta_i$ have zero mean and unit variance for all~$i\in\N$. That is, they all belong to~$\cP$. However, they converge weakly to~$\P=\delta_0$, which is not an element of~$\cP$. Thus, $\cP$ fails to be weakly compact. 
		%
	\end{example}
	
	The family of all distributions on~$\R^d$ with bounded $p$-th-order moments is always weakly compact even though ambiguity sets that fix the $p$-th-order moments to prescribed values ({\em e.g.}, the Chebyshev ambiguity set) may {\em not} be weakly compact. 
	
	\begin{example}[$p$-th-Order Moment Ambiguity Sets]
		\label{ex:moment-ambiguity-set-compact}
		The ambiguity set
		\[
		\cP =\left\{ \P\in \cP(\cZ):   \E_\P[\|Z\|^p] \leq R \right\}
		\]
		induced by any norm~$\|\cdot\|$ on~$\R^d$ and two parameters $p,R>0$ is weakly compact. Using a similar reasoning as in Example~\ref{ex:chebyshev-not-compact}, one can show that for any $\varepsilon>0$ there exists a compact set, namely ${\cC}=\{z\in\cZ: \|z\|\leq (R/\varepsilon)^{1/p}\}$, which satisfies $\P(Z\not\in {\cC}) \leq\varepsilon$. Thus, $\cP$ is tight. To see that $\cP$ is also weakly closed, note that $f(z)=\|z\|^p$ is continuous and bounded below. By Proposition~\ref{prop:semicontinuity}, the expected value $\E_\P[\|Z\|^p]$ is therefore weakly lower semicontinuous in~$\P$ and has weakly closed sublevel sets. Therefore, $\cP$ is weakly compact by virtue of Theorem~\ref{thm:prokhorov}.
	\end{example}
	
	\subsection{$\phi$-Divergence Ambiguity Sets}
	In this section we show that $\phi$-divergence ambiguity sets of the form~\eqref{eq:phi-divergence-ambiguity-set} are weakly compact whenever the entropy function~$\phi$ grows superlinearly. Otherwise, if~$\phi$ grows at most linearly, then the corresponding $\phi$-divergence ambiguity sets generically fail to be weakly compact. Recall that an entropy function~$\phi$ in the sense of Definition~\ref{def:phi} grows superlinearly if and only if $\phi^\infty(1)=\infty$; see also Table~\ref{tab:phi-divergence}.
	
	\begin{lemma}[Worst-Case Probability Maps]
		\label{prop:worst-case-prob-map}
		Let~$\cP$ be the $\phi$-divergence ambiguity set of radius~$r> 0$ around~$\hat{\P}\in\cP(\cZ)$ defined in~\eqref{eq:phi-divergence-ambiguity-set}, and assume that~$\phi$ is continuous at~$1$ and that $\phi^\infty(1)=\infty$. Then, there is a continuous, concave and sur\-jective function $p:[0,1]\rightarrow[0,1]$ that depends only on~$\phi$ and $r$ such that
		\begin{align*}
			\sup_{\P\in \cP} \P(Z\in \cB) = p(\hat{\P}(Z\in\cB))
		\end{align*}
		for every Borel set~$\cB\subseteq\cZ$.
	\end{lemma}
	\begin{proof}
		The proof is constructive. That is, we define the function~$p$ through
		\begin{align*}
			p(t)= \inf_{\lambda_0\in\R, \lambda\in\R_+} \lambda_0 + \lambda r + t\cdot (\phi^*)^\pi \left( 1 - \lambda_0, \lambda\right) + (1-t) \cdot (\phi^*)^\pi\left(-\lambda_0, \lambda\right)
		\end{align*}
		for all $t\in[0,1]$. In the remainder we show that~$p$ satisfies all desired properties. By construction, $p$ depends only on~$\phi$ and~$r$ and coincides with the lower envelope of infinitely many linear functions in~$t$. Hence, $p$ is concave as well as upper semicontinuous. By the definition of~$\cP$ and by Theorem~\ref{thm:duality:restricted:phi} below, we also have
		\begin{align}
			\sup_{\P\in \cP} \P(Z\in \cB)=& \sup_{\P\in \cP(\cZ) } \left\{ \E_\P[ \ds 1_\cB(Z)] : \D_\phi(\P, \hat {\P}) \leq r\right\} \nonumber \\[-1ex] 
			= & \inf_{\lambda_0\in\R, \lambda\in\R_+} \lambda_0 + \lambda r + \E_{\hat{\P}}\left[ (\phi^*)^\pi \left( \ds 1_\cB(Z) - \lambda_0, \lambda\right) \right] \label{eq:p-identity}\\
			= &p(\hat{\P}(Z\in\cB)), \nonumber
		\end{align}
		for any Borel set~$\cB$, where the last equality follows from the definition of~$p$. As the worst-case probability on the left hand side of~\eqref{eq:p-identity} falls within~$[0,1]$ and as $\hat{\P}(Z\in\cB)$ can adopt any value in~$[0,1]$, it is clear that the range of~$p$ is a subset of~$[0,1]$. Next, we show that~$p$ is continuous. To this end, note that the concavity and finiteness of~$p$ on~$[0,1]$ imply via \cite[Theorem~10.1]{rockafellar1970convex} that $p$ is continuous on~$(0,1)$. In addition, its upper semicontinuity prevents~$p$ from jumping at~0 or at~1. Thus, $p$ is indeed continuous throughout~$[0,1]$. Finally, setting $\cB=\emptyset$ or $\cB=\cZ$ in~\eqref{eq:p-identity} shows that~$p(0)=0$ and~$p(1)=1$, respectively. Consequently, we may conclude that~$p$ is surjective. This observation completes the proof. 
	\end{proof}
	
	As $\hat{\P}\in\cP$, the worst-case probability map~$p$ from Lemma~\ref{prop:worst-case-prob-map} satisfies $p(t)\geq t$ for all~$t\in[0,1]$, that is, the worst-case probability is never smaller than the nominal probability. We remark that the map~$p$ also emerges in the study of distributionally robust chance constraints over $\phi$-divergence ambiguity sets with $\phi^\infty(1)=\infty$. Indeed, any such distributionally robust chance constraint with violation probability $\varepsilon\in(0,1)$ is equivalent to a classical chance constraint under the reference distribution~$\hat\P$ with (smaller) violation probability~$p^{-1}(\varepsilon)$; see \citep{ghaoui2003worst, jiang2016data, shapiro2017distributionally}. We can now show that divergence ambiguity sets corresponding to superlinear entropy functions are weakly compact. 
	
	\begin{proposition}[$\phi$-Divergence Ambiguity Sets] 
		\label{prop:compactness-phi-divergence-ambiguity-sets}
		If~$\phi$ is an entropy function with $\phi^\infty(1)=\infty$, then the corresponding $\phi$-divergence ambiguity set~$\cP$ defined in~\eqref{eq:phi-divergence-ambiguity-set} is weakly compact for any closed set $\cZ\subseteq\R^d$, distribution $\hat{\P}\in \cP(\cZ)$ and~$r\geq 0$.
	\end{proposition}
	\begin{proof}
		We first show that~$\cP$ is tight. To this end, select any~$\varepsilon\in(0,1)$, and define~$p^{-1}(\varepsilon)$ as the unique~$t\in(0,1]$ satisfying $p(t)=\varepsilon$, where~$p$ represents the worst-case probability map from Lemma~\ref{prop:worst-case-prob-map}. Note that~$p^{-1}(\varepsilon)$ is well-defined because~$p$ is concave and surjective and because~$p(0)=0$ and~$p(1)=1$. Note also that $p^{-1}(\varepsilon) \leq \varepsilon$ because $p(t)\geq t$. Next, select a sufficiently large~$R>0$ such that $\hat{\P}(\|Z\|_2>R)\leq p^{-1}(\varepsilon)$, and define a compact set $\cC= \{z\in\cZ:\|z\|_2\leq R\}$. Lemma~\ref{prop:worst-case-prob-map} applied to $\cB=\cZ\backslash \cC$ then allows us to conclude that
		\[
		\sup_{\P\in\cP}\P(Z\notin\cC) = p(\hat{\P}(Z\notin \cC)) \leq p(p^{-1}(\varepsilon)) =\varepsilon,
		\]
		where the inequality follows from the monotonicity of~$p$ and choice of~$R$. We have thus shown that $\P(Z\notin\cC)\leq\varepsilon$ for all $\P\in\cP$, and thus~$\cP$ is tight.
		
		It remains to be shown that~$\cP$ is weakly closed. To this end, recall first that~$\cP(\cZ)$ is weakly closed because~$\cZ$ is closed; see Proposition~\ref{prop:support:only:compact}. Next, recall from Proposition~\ref{prop:dual-phi-divergences} that any $\phi$-divergence admits a dual representation of the form
		\begin{align}
			\label{eq:dual-phi-divergence}
			\D_\phi(\P, \hat{\P}) = \sup_{f\in\cF} \; \int_{\cZ}
			f(z)\, \diff\P(z)  -\int_\cZ  \phi^*(f(z)) \, \diff \hat{\P}(z),
		\end{align}
		where $\cF$ denotes the family of all bounded Borel functions $f:\cZ\rightarrow \dom(\phi^*)$. In fact, $\cF$ can be restricted to the space~$\cF^c$ of all {\em continuous} bounded functions without reducing the supremum in~\eqref{eq:dual-phi-divergence}. This is a direct consequence of Lusin's theorem, which ensures that for any~$\delta>0$ and~$f\in\cF$ there exists a compact set~$\cA\subseteq\cZ$ with~$\hat\P(Z\notin\cA)\leq \delta$ and a bounded continuous function~$f_\delta\in\cF^c$ that coincides with~$f$ on~$\cA$ and satisfies $\sup_{z\in\cZ} |f_\delta(z)| \leq \sup_{z\in\cZ} |f(z)| = \| f \|_\infty$. As the convex lower semicontinuous function~$\phi^*$ is continuous on its domain, both
		\[
		\phi^*_l= \inf_{s\in\dom(\phi^*)} \left\{ \phi^*(s) : |s|\leq \|f\|_\infty \right\} \quad \text{and}\quad \phi^*_u= \sup_{s\in\dom(\phi^*)} \left\{ \phi^*(s) : |s|\leq \|f\|_\infty \right\}
		\]
		are finite. Therefore, we have
		\begin{align*}
			& \int_{\cZ}
			f_\delta (z)\, \diff\P(z) - \int_\cZ  \phi^*(f_\delta(z)) \, \diff \hat{\P}(z) \\
			& \geq \int_{\cZ}
			f(z)\, \diff\P(z) - \int_\cZ  \phi^*(f(z)) \, \diff \hat{\P}(z)  -2\|f\|_\infty \,\P(Z\notin\cA)-(\phi^*_u- \phi^*_l)\, \hat\P(Z\notin\cA).
		\end{align*}
		As $\phi^\infty(1)=\infty$ implies $\P \ll \hat\P$ and as $\hat\P(Z\notin\cZ)\leq \delta$, both $\P(Z\notin\cA)$ and $\hat\P(Z\notin\cA)$ decay to~$0$ as~$\delta$ is reduced. Thus, the objective function value of~$f_\delta$ in problem~\eqref{eq:dual-phi-divergence} is asymptotically non-inferior to that of~$f$. This confirms that restricting~$\cF$ to~$\cF^c$ has no impact on the supremum in~\eqref{eq:dual-phi-divergence}. Recall now from Proposition~\ref{prop:semicontinuity} that, for any bounded continuous function~$f\in\cF^c$, the first integral in~\eqref{eq:dual-phi-divergence} is weakly continuous in~$\P$. Thus, $\D_\phi(\P, \hat{\P})$ is weakly lower semicontinuous in~$\P$ as a pointwise supremum of weakly continuous functions. This implies that any sublevel set of the function $f(\P)=\D_\phi(\P, \hat{\P})$ is weakly closed. We thus conclude that the divergence ambiguity set is weakly closed. The claim then follows from Theorem~\ref{thm:prokhorov}.
	\end{proof}
	
	The proof of Proposition~\ref{prop:compactness-phi-divergence-ambiguity-sets} critically relies on the assumption that~$\phi^\infty(1)=\infty$, which ensures that the divergence ambiguity set contains only distributions that are absolutely continuous with respect to~$\hat\P$. Below we show that if the entropy function~$\phi$ grows at most linearly (that is, if~$\phi^\infty(1)<\infty$) and~$\cZ$ is unbounded, then the corresponding divergence ambiguity set fails to be weakly compact. As a preparation, we first establish an upper bound on any $\phi$-divergence on~$\cP(\cZ)\times\cP(\cZ)$.
	
	\begin{lemma}[Upper Bounds on $\phi$-Divergences]
		\label{lem:maximum-radius-divergence-ambiguity-sets}
		If~$\phi$ is an entropy function and $\cZ\subseteq\R^d$ a closed set, then we have $\D_\phi(\P,\hat\P)\leq \phi(0)+\phi^\infty(1)$ for all~$\P,\hat\P\in\cP(\cZ)$. This upper bound is attained if~$\P$ and~$\hat\P$ are mutually singular, that is, if $\P\perp\hat\P$.
	\end{lemma}
	
	\begin{proof}
		In the first part of the proof we derive the desired upper bound. To this end, assume that~$\phi(0)<\infty$ and~$\phi^\infty(1)<\infty$ for otherwise the upper bound is trivially satisfied. As the entropy function is convex, we then have
		\[
		\phi(s)\leq \frac{\Delta}{s+\Delta}\phi(0)+\frac{s}{s+\Delta}\phi(s+\Delta)~\iff~ \phi(s)\leq\phi(0)+s\, \frac{\phi(s+\Delta)-\phi(0)}{s+\Delta}
		\]
		for every~$s,\Delta\geq 0$. Letting~$\Delta$ tend to infinity, this implies that $\phi(s)\leq \phi(0)+s\, \phi^\infty(1)$ for all~$s\geq 0$. The $\phi$-divergence between any $\P,\hat\P\in\cP(\cZ)$ thus satisfies
		\begin{align*}
			\D_\phi(\P,\hat\P) & =\int_\cZ \frac{\diff \hat{\P}}{\diff \rho}(z)  \,  \phi\left( \frac{\frac{\diff \P}{\diff \rho}(z) }{ \frac{\diff \hat{\P}}{\diff \rho}(z)} \right) \diff \rho(z) \\
			&\leq \int_\cZ \frac{\diff \hat{\P}}{\diff \rho}(z)  \,  \phi(0) \, \diff \rho(z) + \int_\cZ \frac{\diff \P}{\diff \rho}(z)  \,  \phi^\infty(1) \, \diff \rho(z) = \phi(0)+\phi^\infty(1),
		\end{align*}
		where we may assume without loss of generality that the dominating measure~$\rho\in\cM_+(\cZ)$ is given by $\rho=\P+\hat\P$. This establishes the desired upper bound. It remains to be shown that this bound is attained even if~$\phi(0)$ or~$\phi^\infty(1)$ evaluate to infinity. To this end, suppose that~$\P$ and~$\hat\P$ are mutually singular. This means that there exist disjoint Borel sets~$\cB,\hat\cB\subseteq \cZ$ with $\P(Z\in\cB)=1$ and $\hat\P(Z\in\hat\cB)=1$. We thus have
		\begin{align*}
			\D_\phi(\P,\hat\P) & =\int_{\hat\cB} \frac{\diff \hat{\P}}{\diff \rho}(z)  \,  \phi( 0 )\, \diff \rho(z) + \int_{\cB} 0 \,  \phi\left( \frac{\frac{\diff \P}{\diff \rho}(z) }{ 0} \right) \diff \rho(z) \\
			&= \phi(0)+ \int_{\cB} \phi^\infty\left(\frac{\diff \P}{\diff \rho}(z) \right) \diff \rho(z)= \phi(0)+\phi^\infty(1).
		\end{align*}
		The first equality holds because $\frac{\diff \P}{\diff \rho}(z)=0$ for $\rho$-almost all~$z\in\hat\cB$ and $\frac{\diff \hat{\P}}{\diff \rho}(z)=0$ for $\rho$-almost all~$z\in\cB$. The second equality follows from the definition of the perspective function and exploits that the restriction of~$\rho$ to~$\hat\cB$ coincides with~$\hat\P$. The third equality, finally, holds because the restriction of~$\rho$ to~$\cB$ coincides with~$\P$. Note that the upper bound is attained even if $\phi(0)=\infty$ or $\phi^\infty(1)=\infty$.
	\end{proof}
	
	The following example reveals that $\phi$-divergence ambiguity sets fail to be weakly compact if~$\phi^\infty(1)<\infty$ and if the set~$\cZ$ without the atoms of~$\hat\P$ is unbounded.
	
	\begin{example}[$\phi$-Divergence Ambiguity Sets]
		Consider an entropy function~$\phi$ with~$\phi^\infty(1)<\infty$. By Lemma~\ref{lem:maximum-radius-divergence-ambiguity-sets}, $\D_\phi(\P,\hat\P)$ is bounded above by~$\overline r=\phi(0)+\phi^\infty(1)$ for all $\P,\hat\P\in\cP$. In addition, let~$\cP$ be the $\phi$-divergence ambiguity set with center~$\hat\P\in\cP(\cZ)$ and radius~$r\in(0,\overline r)$ defined in~\eqref{eq:phi-divergence-ambiguity-set}. Assume that for every $R>0$ there exists~$z_0\in\cZ$ with $\|z_0\|_2\geq R$ and~$\hat\P(Z=z_0)=0$. This assumption holds, for example, whenever~$\cZ$ is unbounded and convex, and it implies that~$\cP$ fails to be tight. To see this, fix an arbitrary compact set~$\cC\subseteq\cZ$, and select any point~$z_0\in\cZ\backslash\cC$ with $\hat \P(Z=z_0)=0$. Such a point exists by assumption. Next, consider the distributions $\P_\theta=(1-\theta)\,\hat\P+\theta\,\delta_{z_0}$ parametrized by~$\theta\in[0,1]$. Note that~$\hat\P$ and~$\delta_{z_0}$ are mutually singular and that $f(\theta)= \D_\phi(\P_\theta,\hat\P)$ is a convex continuous bijective function from~$[0,1]$ to~$[0,\overline r]$. Set now~$\varepsilon = \half f^{-1}(r)$. For $\theta=f^{-1}(r)$, the distribution~$\P_\theta$ satisfies $ \D_\phi(\P_\theta,\hat\P)=f(f^{-1}(r))=r$ and thus belongs to~$\cP$. In addition, $\P_\theta(Z\notin\cC) \geq f^{-1}(r)>\varepsilon$ because~$z_0\notin\cC$. Note that~$\varepsilon$ is independent of~$\cC$ and~$z_0$ as long as~$\hat\P(Z=z_0)=0$. As the compact set~$\cC$ was chosen arbitrarily, this implies that~$\cP$ fails to be tight and weakly compact.
	\end{example}


	\subsection{Marginal Ambiguity Sets}
	\label{sec:topology-marginal-ambiguity-sets}
	As a preparation towards exploring the topological properties of optimal transport ambiguity sets, we first study marginal ambiguity sets. The following proposition shows that Fr\'echet ambiguity sets, which prescribe the marginal distributions of all~$d$ individual components of $Z$, are always weakly compact.
	
	\begin{proposition}[Fr\'echet Ambiguity Sets]
		\label{prop:freched-compact}
		The Fr\'echet ambiguity set~$\cP$ defined in~\eqref{eq:frechet} is weakly compact for any cumulative distribution functions~$F_i$, $i\in [d]$.
	\end{proposition}
	
	\begin{proof}
		We first show that the Fr\'echet ambiguity set is tight. For any $\varepsilon>0$ and $i\in[d]$, we can set $\underline z_i$ and $\overline z_i$ to the $\varepsilon/(2d)$-quantile and the $(1-\varepsilon/(2d))$-quantile of the distribution function~$F_i$, respectively. Setting ${\cC}=\times_{i\in[d]}[\underline z_i,\overline z_i]$ yields
		\begin{align*}
			\P(Z\not\in {\cC}) \leq \sum_{i\in [d]}\P(Z_i\not\in[\underline z_i,\overline z_i]) = \sum_{i\in [d]} \varepsilon/d=\varepsilon,
		\end{align*}
		where the inequality follows from the union bound. Thus, $\cP$ is tight. It remains to be shown that~$\cP$ is weakly closed. Note that the distribution function of~$Z_i$ under~$\P$ matches $F_i$ if and only if for every bounded continuous function~$f$ we have
		\[
		\E_\P[f(Z_i)] = \int_{-\infty}^{+\infty} f(z_i)\,\diff F_i(z_i).
		\]
		This is true because every Borel distribution on~$\R$ constitutes a Radon measure. The set of all~$\P\in\cP(\R^d)$ satisfying the above equality for any fixed bounded and continuous function~$f$ and any fixed index~$i\in[d]$ is weakly closed by Proposition~\ref{prop:semicontinuity}. Hence, $\cP$ is weakly closed because closedness is preserved by intersection.
	\end{proof}
	
	It is straightforward to generalize Proposition~\ref{prop:freched-compact} from Fr\'echet ambiguity sets to generic marginal ambiguity sets as discussed in Section~\ref{sec:marginal-ambiguity-sets}, which prescribe {\em multivariate} marginal distributions. Details are omitted for brevity.

	\subsection{Optimal Transport Ambiguity Sets}
	
	Recall that~$\Gamma(\P,\hat\P)$ denotes the family of all transportation plans linking the probability distributions~$\P,\hat\P\in\cP(\cZ)$. Thus, $\Gamma(\P, \hat{\P})$ contains all joint distributions~$\gamma$ of~$Z$ and~$\hat Z$ with marginals~$\P$ and~$\hat{\P}$, respectively. The set~$\Gamma(\P,\hat\P)$ appears in the definition of the optimal transport discrepancy~$\OT_c(\P, \hat{\P})$; see Definition~\ref{def:OT}. The reasoning in Section~\ref{sec:topology-marginal-ambiguity-sets} immediately implies that~$\Gamma(\P,\hat\P)$ is weakly compact because it constitutes a marginal ambiguity set. This insight is formalized in the following simple corollary of Proposition~\ref{prop:freched-compact}. Its proof is omitted for brevity.
	
	\begin{corollary}[Transportation Plans]
		\label{cor:compact:Gamma}
		The set of all transportation plans $\Gamma(\P, \hat \P)$ with marginal distributions $\P, \hat \P \in \cP(\cZ)$ is weakly compact.
	\end{corollary}
	
	Corollary~\ref{cor:compact:Gamma} enables us to show that the optimal transport problem in~\eqref{eq:ot-discrepancy} is solvable as the transportation cost function is assumed to be lower semicontinuous.
	
	\begin{lemma}[Solvability of Optimal Transport Problems]
		\label{lem:OT-solvability}
		The infimum in~\eqref{eq:ot-discrepancy} is attained.
	\end{lemma}
	\begin{proof}
		By Corollary~\ref{cor:compact:Gamma}, the set $\Gamma(\P, \hat{\P})$ is weakly compact. In addition, the transportation cost function $c(z,\hat z)$ is lower semicontinuous and bounded below. By Proposition~\ref{prop:semicontinuity}, the expected value $ \E_\gamma[ c(Z,\hat Z)]$ is therefore weakly lower semicontinuous in~$\gamma$. Thus, the optimal transport problem in~\eqref{eq:ot-discrepancy} is solvable thanks to Weierstrass' theorem, and its infimum is attained. 
	\end{proof}
	
	Lemma~\ref{lem:OT-solvability} allows us to prove that the optimal transport discrepancy~$\OT_c(\P, \hat{\P})$ constitutes a weakly lower semicontinuous function of its inputs~$\P$ and~$\hat{\P}$.
	
	\begin{lemma}[Weak Lower Semicontinuity of Optimal Transport Discrepancies]
		\label{lem:Wasserstein-lsc}
		The optimal transport discrepancy $\OT_c(\P,\hat{\P})$ is weakly lower semicontinuous jointly in~$\P$ and~$\hat\P$. 
	\end{lemma}
	
	\begin{proof}
		Assume that $\P_j$ and $\hat{\P}_j$, $j\in\N$, converge weakly to $\P$ and~$\hat{\P}$, respectively, and define the countable ambiguity sets~$\cP = \{\P_j\}_{j\in\N}$ and~$\hat{\cP} = \{\hat {\P}_j\}_{j\in\N}$. By the definition of sequential compactness, the weak closures of $\cP$ and $\hat{\cP}$ are weakly compact. Prokhorov's theorem (see Theorem~\ref{thm:prokhorov}) thus implies that both $\cP$ and $\hat{\cP}$ are tight. Hence, for any~$\varepsilon>0$ there exist two compact sets ${\cC}, \hat {\cC}\subseteq \R^d$ with
		\[
		\P_j(Z\not\in {\cC}) \leq \varepsilon/2 \quad \text{and}\quad \hat{\P}_j(\hat Z\not\in \hat {\cC}) \leq \varepsilon/2 \quad \forall j\in\N.
		\]
		Whenever $\gamma\in \Gamma(\P_j, \,\hat{\P}_j)$ for some $j\in\N$, we thus have
		\[
		\gamma\big((Z,\hat Z) \notin {\cC}\times\hat {\cC}\big) \leq \P_j(Z\notin {\cC}) + \hat{\P}_j(Z\notin \hat {\cC})\leq \varepsilon.
		\]
		As ${\cC} \times \hat {\cC}$ is compact and as~$\varepsilon$ was chosen arbitrarily, this reveals that the union
		\begin{align}
			\label{eq:pi-union}
			\bigcup_{j \in \N} \Gamma(\P_j, \,\hat{\P}_j)
		\end{align}
		is tight, which in turn implies via Prokhorov's theorem that its closure is weakly compact. Let now~$\gamma_j^\star$ be an optimal coupling of~$\P_j$ and~$\hat{\P}_j$, which solves problem~\eqref{eq:ot-discrepancy}, and which exists thanks to Lemma~\ref{lem:OT-solvability}. As all these optimal couplings belong to some weakly compact set ({\em i.e.}, the weak closure of~\eqref{eq:pi-union}), we may assume without loss of generality that $\gamma_j^\star$, $j\in\N$, converges weakly to some distribution~$\gamma$. Otherwise, we can pass to a subsequence. Clearly, we have $\gamma\in\Gamma(\P, \hat {\P})$. For~$\gamma^\star$ an optimal coupling of~$\P$ and~$\hat{\P}$, we then find
		\begin{align*}
			\liminf_{j\rightarrow \infty} \OT_c(\P_j,\hat{\P}_j) & = \liminf_{j\rightarrow \infty} \E_{\gamma_j^\star}[c(Z,\hat Z)] \\
			& \geq \E_{\gamma}[c(Z,\hat Z)] \geq \E_{\gamma^\star}[c(Z,\hat Z)] = \OT_c(\P,\hat{\P}),
		\end{align*}
		where the two equalities follow from the definitions of~$\gamma_j^\star$ and~$\gamma^\star$, respectively. The first inequality holds because $\E_{\gamma}[c(Z,\hat Z)]$ is weakly lower semicontinuous in~$\gamma$ thanks to Proposition~\ref{prop:semicontinuity}, and the second inequality follows from the suboptimality of~$\gamma$ in~\eqref{eq:ot-discrepancy}. Thus, $\OT_c(\P,\hat{\P})$ is weakly lower semicontinuous in~$\P$ and~$\hat{\P}$.
	\end{proof}
	
	Lemma~\ref{lem:Wasserstein-lsc} is inspired by \cite[Lemma~5.2]{PhilippeClement2008} and \cite[Theorem~1]{yue2020linear}. Next, we prove that Wasserstein ambiguity sets are weakly compact. Throughout this discussion we assume that the metric underlying the transportation cost function is induced by a norm~$\|\cdot\|$ on~$\R^d$. This assumption simplifies our derivations but could be relaxed. Recall that the $p$-Wasserstein distance~$\W_p(\P,\hat\P)$ for~$p\geq 1$ is the $p$-th root of~$\OT_c(\P,\hat\P)$, where the transportation cost function is set to~$c(z,\hat z)=\|z-\hat z\|^p$; see Definition~\ref{def:p-Wassertein}.
	
	\begin{theorem}[$p$-Wasserstein Ambiguity Sets]
		\label{thm:wasserstein-compactness}
		Assume that the metric~$d(\cdot,\cdot)$ on~$\cZ$ is induced by some norm~$\|\cdot\|$ on the ambient space~$\R^d$. If~$\hat{\P}\in\cP(\cZ)$ has finite $p$-th moments (i.e., $\E_{\hat {\P}}[\|Z\|^p]<\infty$) for some exponent~$p\geq 1$, then the $p$-Wasserstein ambiguity set~$\cP$ defined in~\eqref{eq:p-wasserstein-ball} 
		is weakly compact.
	\end{theorem}
	\begin{proof}
		We first show that all distributions~$\P\in\cP$ have uniformly bounded $p$-th moments. To this end, set $\hat r = \E_{\hat {\P}}[\|Z\|^p]<\infty$, and note that any $\P\in\cP$ satisfies
		\begin{align*}
			\big(  \E_{\P}[\|Z\|^p]\big)^{\frac{1}{p}} = \W_p(\P,\delta_0) & \leq \W_p(\P,\hat{\P}) + \W_p(\hat{\P}, \delta_0) \\
			& = \W_p(\P,\hat{\P}) + \left(  \E_{\hat{\P}}[\|Z\|^p]\right)^{\frac{1}{p}} \leq r + \hat r.
		\end{align*}
		Here, the first inequality holds because the $p$-Wasserstein distance is a metric and thus satisfies the triangle inequality, and the second inequality holds because~$\P\in\cP$. We therefore have $\E_{\P}[\|Z\|^p]\leq (r+\hat r)^p$ for every $\P\in\cP$. In other words, the Wasserstein ball~$\cP$ is a subset of the $p$-th-order moment ambiguity set discussed in Example~\ref{ex:moment-ambiguity-set-compact}. This implies that~$\cP$ is tight. Note further that $\cP$ is defined as a sublevel set of the function $f(\P) = \W_p(\P,\hat{\P})$, which is weakly lower semicontinuous thanks to Lemma~\ref{lem:Wasserstein-lsc}. Hence, $\cP$ is weakly closed. 
	\end{proof}
	
	Finally, we prove that the $\infty$-Wasserstein ambiguity set is always weakly compact. 
	
	\begin{corollary}[$\infty$-Wasserstein Ambiguity Sets]
		Assume that the metric~$d(\cdot,\cdot)$ on~$\cZ$ is induced by some norm~$\|\cdot\|$ on the ambient space~$\R^d$. Then, the $\infty$-Wasserstein ambiguity set defined in~\eqref{eq:infty-wasserstein-ambiguity-set} is weakly compact for every~$\hat\P\in\cP(\cZ)$.
	\end{corollary}
	\begin{proof}
		We first show that~$\cP$ is tight. To this end, select any~$\varepsilon>0$ and any compact set~$\hat\cC\subseteq\cZ$ with $\hat\P(Z\not\in\hat\cC)\leq \varepsilon$. Note that~$\hat \cC$ is guaranteed to exist because~$\hat\P$ is a probability distribution. Next, define~$\cC$ as the $r$-neighborhood~$\hat\cC_r$ of~$\hat\cC$, that is, set
		\[
		\cC=\left\{z\in\cZ:\exists \hat z\in\hat\cC\text{ with }\|z-\hat z\|\leq r \right\},
		\]
		see also~\eqref{eq:Br}. One readily verifies that~$\cC$ inherits compactness from~$\hat\cC$. Any distribution~$\P\in\cP$ satisfies $\W_\infty(\P,\hat\P)\leq r$. Consequently, we find
		\begin{align*}
			\P(Z\not\in\cC)=\P(Z\in\cZ\backslash\cC)\leq \hat\P(Z\in\cZ\backslash\hat\cC)=\hat\P(Z\not\in\hat\cC)\leq \varepsilon,
		\end{align*}
		where the first inequality follows from Corollary~\ref{cor:dual-W_infty} and the observation that the $r$-neighborhood of~$\cZ\backslash\cC$ coincides with~$\cZ\backslash\hat\cC$. The second inequality follows from the definition of~$\hat\cC$. As~$\varepsilon$ was chosen arbitrarily, $\cP$ is tight. It remains to be shown that~$\cP$ is weakly closed. Proposition~\ref{prop:W:p:infty} readily implies that $\W_\infty(\P, \hat{\P})\leq r$ if and only if $\W_p(\P, \hat{\P})\leq r$ for all~$p\geq 1$. Thus, we may conclude that
		\begin{align*}
			\cP = \bigcap_{p \geq 1} \left\{ \P \in \cP(\R^d): \W_p(\P, \hat{\P}) \leq r \right\}.
		\end{align*}
		That is, the $\infty$-Wasserstein ambiguity set can be expressed as the intersection of all $p$-Wasserstein ambiguity sets for~$p\geq 1$, all of which are weakly closed by Theorem~\ref{thm:wasserstein-compactness}. Hence, $\cP$ is is indeed weakly closed, and the claim follows.
	\end{proof}
	
	\section{Duality Theory for Worst-Case Expectation Problems}
	\label{sec:duality-wc-expectation}
	The DRO problem~\eqref{eq:primal:dro} is often interpreted as a zero-sum game between the decision-maker and a fictitious adversary. The decision-maker moves first and thus selects~$x$ {\em before} seeing~$\P$. Therefore, $x$ is optimized against {\em all} distributions~$\P\in\cP$. In contrast, the adversary moves second and thus selects~$\P$ {\em after} seeing~$x$. Therefore, $\P$ is only optimized against {\em one particular} decision~$x\in\cX$. Put differently, the adversary's choice may adapt to the decision-maker's choice but {\em not} vice versa. 
	
	In this section we develop a duality theory for the adversary's subproblem, which aims to maximize the expected loss of a fixed decision~$x$ across all distributions in a convex ambiguity set~$\cP$. To avoid clutter, we suppress the dependence of the loss function~$\ell$ on the fixed decision~$x$ throughout this discussion, that is, we write~$\ell(z)$ instead of~$\ell(x, z)$. We thus address worst-case expectation problems of the form
	\begin{align}
		\label{eq:worst-case:expectation}
		\sup_{\P \in \cP} \; \E_\P [\ell(Z)].
	\end{align}
	Note that~$\cP$ represents a convex subset of the linear space of all finite signed Borel measures on~$\cZ$. Unless~$\cZ$ is finite, \eqref{eq:worst-case:expectation} thus constitutes an infinite-dimensional convex program with a linear objective function. For this problem to be well-defined, we assume that $\ell : \cZ \to \overline \R$ is a Borel function. In line with \citep[Section~14.E]{rockafellar2009variational}, we define $\E_\P [\ell(Z)]=-\infty$ if $\E_\P [\max\{\ell(Z),0\}]=\infty$ and $\E_\P [\min\{\ell(Z),0\}]=-\infty$. This means that infeasibility trumps unboundedness. More generally, throughout the rest of the paper, we assume that if the objective function of a minimization (maximization) problem can be expressed as the difference of two terms, both of which evaluate to~$\infty$, then the objective function value should be interpreted as $\infty$ ($-\infty$). This convention is in line with the rules of extended arithmetic used in \citep{rockafellar2009variational}.
	
	In the remainder we will show that~\eqref{eq:worst-case:expectation} can be dualized by using elementary tools from finite-dimensional convex analysis \citep{fenchel1953convex, rockafellar1970convex} for a broad class of finitely-parametrized ambiguity sets including all moment ambiguity sets (Section~\ref{sec:moment-ambiguity-sets-duality}), $\phi$-divergence ambiguity sets (Section~\ref{sec:phi:duality}) and optimal transport ambiguity sets (Section~\ref{sec:optimal:transport-duality}). We broadly adopt the proof strategies developed by \citet{shapiro2001duality} and \citet{zhang2022simple} for moment and optimal transport ambiguity sets, respectively, and we extend them to $\phi$-divergence ambiguity sets.
	
	
	\subsection{General Proof Strategy}
	\label{sec:duality-proof-strategy}
	In order to outline the high-level ideas for dualizing~\eqref{eq:worst-case:expectation}, we recall a basic result on the convexity of parametric infima; see, {\em e.g.}, \citep[Theorem~1]{rockafellar1974conjugate}.
	
	\begin{lemma}[Convexity of Optimal Value Functions]
		\label{lem:param:cvx}
		If~$\cU$ and~$\cV$ are arbitrary real vector spaces and $H: \cU \times \cV \to \overline \R$ is a convex function, then the optimal value function~$h:\cU\to\overline \R$ defined through $h(u) = \inf_{v \in \cV} H(u, v)$ is convex.
	\end{lemma}
	
	\begin{proof}
		Note that~$h$ is a convex function if and only if its epigraph $\epi(h)$ is a convex set. By the definitions of the epigraph and the infimum operator, we find
		\begin{align*}
			\epi(h) 
			&= \{ (u, t) \in \cU \times \R : h(u) \leq t \} \\
			&= \{ (u, t) \in \cU \times \R : \exists v \in \cV ~~\text{with}~~ H(u, v) \leq t + \varepsilon ~~ \forall \varepsilon >0 \}\\
			&= \bigcap_{\varepsilon>0} \, \{ (u, t) \in \cU \times \R : \exists v \in \cV ~~\text{with}~~ H(u, v)-\varepsilon \leq t \}.
		\end{align*}
		Thus, $\epi(h)$ can be obtained by projecting $\cap_{\varepsilon >0} \epi(H-\varepsilon)$ to $\cU\times \R$. The claim then follows because~$\epi(H-\varepsilon)$ is convex for every~$\varepsilon>0$ thanks to the convexity of~$H$ and because convexity is preserved under intersections and linear transformations; see, {\em e.g.}, \citep[Theorems~2.1 \& 5.7]{rockafellar1970convex}.
	\end{proof}
	
	The following result marks a cornerstone of convex analysis. It states that the bi\-conjugate~$h^{**}$ (that is, the conjugate of~$h^*$) of a closed convex function~$h$ coincides with~$h$. Here, we adopt the standard convention that~$h$ is closed if it is lower semi\-continuous and either $h(u) > -\infty$ for all~$u\in\cU$ or~$h(u) = -\infty$ for all~$u\in\cU$. We use~$\cl(h)$ to denote the closure of~$h$, that is, the largest closed function below~$h$.
	
	\begin{lemma}[Fenchel–Moreau Theorem]
		\label{lem:bi:coincidence}
		For any convex function $h: \R^d \to \overline \R$, we have $h \geq h^{**}$. The inequality becomes an equality on~$\rint(\dom(h))$.
	\end{lemma}
	
	\begin{proof}
		By \citep[Theorem~12.2]{rockafellar1970convex}, we have~$h^{**} = \cl(h)\leq h$. In addition, \citep[Theorem~10.1]{rockafellar1970convex} ensures that the convex function~$h$ is continuous on $\rint(\dom(h))$ and thus coincides with~$\cl(h)$ there. Hence, the claim follows.
	\end{proof}
	
	The main idea for dualizing the worst-case expectation problem~\eqref{eq:worst-case:expectation} is to represent its optimal value as~$-h(u)$, where~$h(u)=\inf_{v \in \cV} H(u, v)$, $\cU$ is a finite-dimensional space of parameters~$u$ that encode the ambiguity set~$\cP$ (such as a set of prescribed moments or a size parameter), and~$\cV$ is an infinite-dimensional space of finite signed measures on~$\cZ$. In addition, $H(u,v)$ represents the negative expected loss if the signed measure~$v$ happens to be a probability measure in~$\cP\subseteq\cV$ and evaluates to~$\infty$ otherwise. If~$H(u,v)$ is jointly convex on~$u$ and~$v$, then~$h(u)$ is convex by virtue of Lemma~\ref{lem:param:cvx}. A problem dual to~\eqref{eq:worst-case:expectation} can then be constructed from the bi-conjugate~$h^{**}(u)$. Lemma~\ref{lem:bi:coincidence} provides conditions for strong duality.

	\subsection{Moment Ambiguity Sets}
	\label{sec:moment-ambiguity-sets-duality}
	Recall from Section~\ref{sec:moment-ambiguity-sets} that the generic moment ambiguity set~\eqref{eq:moment-ambiguity-set} is defined~as
	\begin{align*}
		\cP = \left\{ \P \in \cP_f(\cZ) \, : \, \E_\P \left[ f (Z) \right] \in \cF \right\},
	\end{align*}
	where $\cZ \subseteq \R^d$ is a closed support set, $f: \cZ \to \R^m$ is a Borel measurable moment function, $\cF \subseteq \R^m$ is a closed moment uncertainty set, and $\cP_f(\cZ)$ denotes the family of all distributions~$\P\in\cP(\cZ)$ for which $\E_\P[f(Z)]$ is finite.\footnote{Clearly, $\E_\P[f(Z)]$ must be finite to belong to the closed set~$\cF$. Therefore, we may replace~$\cP(\cZ)$ with~$\cP_f(\cZ)$ in the definition of~$\cP$ without loss of generality. However, working with~$\cP_f(\cZ)$ is more convenient when we dualize the worst-case expectation problem~\eqref{eq:worst-case:expectation} over~$\cP$.} We may assume without loss of generality that~$\cF$ is covered by the convex set
	\begin{align*}
		\cC = \left\{ \E_\P[f(Z)] : \P \in \cP_f(\cZ) \right\}
	\end{align*}
	of all possible moments of any distribution on~$\cZ$. To rule out trivial special cases, we make the blanket assumption that~$\cZ$ and~$\cF$ are non-empty. 
	
	Clearly, problem~\eqref{eq:worst-case:expectation} over the moment ambiguity set~\eqref{eq:moment-ambiguity-set} can be recast as
	\begin{align}
		\label{eq:decomp}
		\sup_{\P \in \cP} \E_\P \left[ \ell (Z) \right] = \sup_{u \in \cF} \, \sup_{\P \in \cP_f(\cZ)} \big\{ \E_{\P} \left[ \ell(Z) \right]: \E_{\P} \left[ f(Z) \right] = u \big\} 
		= \sup_{u \in \cF} - h(1, u), 
	\end{align}
	where the auxiliary function~$h: \R\times\R^{m} \to \overline \R$ is defined through
	\begin{align}
		\label{eq:h:moment}
		h(u_0, u) = \inf_{v \in \cM_{f,+} (\cZ)} \left\{ -\int_{\cZ} \ell(z) \, \diff v(z): \int_{\cZ} \diff v(z) = u_0, 
		\int_{\cZ} f(z) \, \diff v(z) = u \right\}.
	\end{align}
	Here, the set~$ \cM_{f,+} (\cZ)$ stands for the family of all Borel measures~$v\in \cM_{+} (\cZ)$ for which the integral $\int_{\cZ}f(z)\,\diff v(z)$ is finite. Put differently, $ \cM_{f,+} (\cZ)$ represents the convex cone generated by $\cP_f(\cZ)$. As the objective and constraint functions of the minimization problem in~\eqref{eq:h:moment} are all jointly convex and jointly linear in~$v$, $u_0$ and~$u$, respectively, the equivalent reformulation that incorporates the constraints into the objective via indicator functions remains convex. This implies via Lemma~\ref{lem:param:cvx} that~$h$ is convex. Under a reasonable regularity condition, one can further show that the domain of~$h$ coincides with the cone generated by~$\{1\}\times \cC$. 
	
	\begin{lemma}[Domain of~$h$]
		\label{lem:domain:h:moment}
		If $\E_\P[\ell(Z)]>-\infty$ for every~$\P\in\cP_f(\cZ)$, then we have \[ \dom(h)=\cone(\{1\}\times \cC).\]
	\end{lemma}
	\begin{proof}
		It is clear that $(u_0,u)\in\dom(h)$ if and only if $h(u_0,u)<\infty$, which is the case if and only if the minimization problem in~\eqref{eq:h:moment} is feasible. Thus, it remains to be shown that the problem in~\eqref{eq:h:moment} is feasible if and only if $(u_0,u)\in\cone (\{1\}\times\cC)$. To this end, assume first that the problem in~\eqref{eq:h:moment} is feasible at~$(u_0,u)$. This implies that there is~$v\in\cM_{f,+}(\cZ)$ with $\int_{\cZ} \diff v(z) = u_0$ and $\int_{\cZ} f(z) \, \diff v(z) = u$. Hence, $u_0\geq 0$. If~$u_0=0$, then we must have~$u=0$. If~$u_0>0$, on the other hand, then $v/u_0$ must be a probability measure in~$\cP_f(\cZ)$, which implies that~$u/u_0\in\cC$. In either case, $(u_0,u)$ is a non-negative multiple of a point in~$\{1\}\times\cC$ and thus belongs to~$\cone(\{1\}\times\cC)$. Next, assume that~$(u_0,u)\in \cone(\{1\}\times\cC)$. If~$u_0=0$, then~$u=0$, and indeed, the zero measure in~$\cM_{f,+}(\cZ)$ is feasible in~\eqref{eq:h:moment}. If~$u_0>0$, on the other hand, then~$u/u_0\in\cC$. By the definition of~$\cC$, there exists a distribution~$\P\in\cP_f(\cZ)$ with $\E_\P[f(Z)]=u/u_0$. As $\E_\P[\ell(Z)]>-\infty$, this implies that~$v=u_0\P$ is feasible in~\eqref{eq:h:moment}. We have thus shown that~\eqref{eq:h:moment} is feasible if and only if $(u_0,u)\in\cone (\{1\}\times\cC)$. This observation completes the proof. 
	\end{proof}
	
	The following proposition characterizes the bi-conjugate of~$h$.
	
	\begin{proposition}[Bi-conjugate of~$h$]
		\label{prop:duality:moment}
		The bi-conjugate of~$h$ defined in~\eqref{eq:h:moment} satisfies
		\begin{align*}
			h^{**}(u_0, u) = \sup_{\lambda_0 \in \R, \, \lambda \in \R^m} \big\{ -u_0\lambda_0 - u^\top \lambda : \lambda_0 + f(z)^\top\lambda \geq \ell(z) ~~ \forall z \in \cZ \big\}.
		\end{align*}
		If additionally $\E_\P[\ell(Z)]>-\infty$ for every~$\P\in\cP_f(\cZ)$, then~$h^{**}$ and~$h$ match on the cone generated by~$\{1\}\times \rint(\cC)$ except at the origin.
	\end{proposition}
	\begin{proof}
		For any fixed~$(\lambda_0,\lambda)\in\R\times \R^m$, the convex conjugate of~$h$ satisfies
		\begin{align*}
			h^*(-\lambda_0, -\lambda) 
			&= \sup_{u_0\in\R,\, u \in \R^m} - u_0 \lambda_0 - u^\top \lambda - h(u_0, u) \\
			&= \left\{
			\begin{array}{cl}
				\sup & \displaystyle - u_0 \lambda_0 - u^\top \lambda + \int_{\cZ} \ell(z) \, \diff v(z) \\[2ex]
				\st & u_0 \in \R,\; u \in \R^m, \; v \in \cM_{f,+}(\cZ) \\[0.5ex] 
				& \displaystyle \int_{\cZ}  \diff v(z) = u_0, ~ \int_{\cZ} f(z) \; \diff v(z) = u
			\end{array} \right. \\[1ex]
			&= \sup_{v \in \cM_{f,+}(\cZ)}  \int_{\cZ} \big( \ell(z) - \lambda_0 - f(z)^\top \lambda \big)\, \diff v(z) \\[1ex]
			&= \left\{ \begin{array}{cl}
				0 & \text{~if~} \ell(z) - \lambda_0 - f(z)^\top \lambda \leq 0 ~~ \forall z \in \cZ ,\\
				\infty & \text{otherwise,}
			\end{array}
			\right. 
		\end{align*}
		where the last equality holds because~$\cM_{f,+}(\cZ)$ contains all weighted Dirac measures on~$\cZ$. Thus, for any fixed~$(u_0,u) \in\R\times \R^m$, the conjugate of~$h^*$ satisfies
		\begin{align*}
			h^{**}(u_0, u) 
			&= \sup_{\lambda_0\in\R,\, \lambda \in \R^m} - u_0\lambda_0 - u^\top \lambda - h^*(-\lambda_0, -\lambda) \\
			&=\sup_{\lambda_0\in\R,\, \lambda \in \R^m} \big\{ -u_0\lambda_0 - u^\top\lambda :  \lambda_0 + f(z)^\top \lambda \geq \ell(z) ~~ \forall z \in \cZ \big\}.
		\end{align*}
		This establishes the desired formula for the bi-conjugate of~$h$. Assume now that $\E_\P[\ell(Z)]>-\infty$ for every~$\P\in\cP_f(\cZ)$. It remains to be shown that $h(u_0,u)=h^{**}(u_0,u)$ for all~$(u_0,u)\neq (0,0)$ in the cone generated by~$\{1\}\times \rint(\cC)$. However, this follows immediately from Lemma~\ref{lem:bi:coincidence} and the observation that
		\[
		\rint(\dom(h)) = \rint(\cone(\{1\}\times\cC)) = \cone(\{1\}\times\rint(\cC))\backslash \{(0,0)\},
		\]
		where the two equalities hold because of Lemma~\ref{lem:domain:h:moment} and~\cite[Corollary~6.8.1]{rockafellar1970convex}, respectively. Therefore, the claim follows.
	\end{proof}
	
	Proposition~\ref{prop:duality:moment} implies that $h(1,u)=h^{**}(1,u)$ for all~$u\in\rint(\cC)$. The following main theorem exploits this relation to convert the maximization problem on the right hand side of~\eqref{eq:decomp} to an equivalent dual minimization problem. 
	
	\begin{theorem}[Duality Theory for Moment Ambiguity Sets]
		\label{thm:duality:moment}
		If~$\cP$ is the moment ambiguity set~\eqref{eq:moment-ambiguity-set}, then the following weak duality relation holds.  
		\begin{align}
			\label{eq:weak-duality-moments}
			\sup_{\P \in \cP} ~ \E_\P \left[ \ell(Z) \right] 
			\leq \left\{
			\begin{array}{cl}
				\inf & \lambda_0 + \delta_\cF^*(\lambda) \\[1ex]
				\st & \lambda_0 \in \R, \, \lambda \in \R^m \\ [1ex]
				& \lambda_0 + f(z)^\top\lambda \geq \ell(z) \quad \forall z \in \cZ.
			\end{array}
			\right.
		\end{align}
		If $\E_\P[\ell(Z)]>-\infty$ for all~$\P\in\cP_f(\cZ)$ and~$\cF \subseteq \cC$ is a convex and compact set with $\rint(\cF)\subseteq \rint(\cC)$, then strong duality holds, that is, \eqref{eq:weak-duality-moments} becomes an equality. 
	\end{theorem}
	
	\begin{proof}
		For ease of exposition, we introduce 
		\[
		\cL = \left\{ (\lambda_0, \lambda) \in \R\times\R^m: \lambda_0 + f(z)^\top \lambda \geq \ell(z) ~ \forall z \in \cZ \right\}
		\]
		as a shorthand for the dual feasible set. Using the decomposition~\eqref{eq:decomp}, we find
		\begin{align*}
			\sup_{\P \in \cP} \; \E_\P \left[ \ell(Z) \right]= \sup_{u \in \cF} \; -h(1, u)
			&\leq \sup_{u \in \cF} \inf_{(\lambda_0, \lambda) \in \cL} ~ \lambda_0 + u^\top\lambda \\
			&\leq \inf_{(\lambda_0, \lambda) \in \cL} \sup_{u \in \cF} ~ \lambda_0 + u^\top\lambda \\
			&= \inf_{(\lambda_0, \lambda) \in \cL} ~ \lambda_0 + \delta_\cF^*(\lambda).
		\end{align*}
		Here, the first inequality exploits Proposition~\ref{prop:duality:moment} and Lemma~\ref{lem:bi:coincidence}, which ensures that~$h\geq h^{**}$, and the second inequality holds thanks to the max-min inequality. The last equality follows from the definition of the support function~$\delta^*_\cF$. This establishes the weak duality relation~\eqref{eq:weak-duality-moments}. Next, suppose that~$\cF$ is a convex compact set with~$\rint(\cF)\subseteq \rint(\cC)$. Under this additional assumption, we have
		\begin{align*}
			\sup_{\P \in \cP} \; \E_\P \left[ \ell (Z) \right]
			= \sup_{u \in \cF} \; -h(1, u)
			&= \sup_{u \in \rint(\cF)} -h(1, u) \\ 
			&= \sup_{u \in \rint(\cF)} \inf_{(\lambda_0, \lambda) \in \cL} ~ \lambda_0 + u^\top\lambda \\
			&= \sup_{u \in \cF} \inf_{(\lambda_0, \lambda) \in \cL} ~ \lambda_0 + u^\top\lambda
			= \inf_{(\lambda_0, \lambda) \in \cL} ~ \lambda_0 + \delta_\cF^*(\lambda),
		\end{align*}
		where the first equality exploits~\eqref{eq:decomp}. The second equality follows from two observations. First, $\rint(\cF)$ is non-empty and convex \cite[Theorem~6.2]{rockafellar1970convex}. Second, $-h(1,u)$ is concave in~$u$, which ensures that~$-h(1,u)$ cannot jump up on the boundary of its domain~$\cC$ and---in particular---on the boundary of~$\cF\subseteq\cC$. Taken together, these observations imply that we can restrict~$\cF$ to~$\rint(\cF)$ without reducing the supremum. The third equality follows from Proposition~\ref{prop:duality:moment}, which allows us to replace~$h$ with~$h^{**}$ on~$\rint(\cF) \subseteq \rint(\cC)$. The fourth equality holds because~$-h^{**}(1,u)$ is concave in~$u$, which allows us to change~$\rint(\cF)$ back to~$\cF$. Finally, the fifth equality follows from Sion's minimax theorem~\citep[Theorem~4.2]{sion1958general}, which applies because~$\cF$ is convex and compact, $\cL$ is convex and $\lambda_0+u^\top \lambda$ is biaffine in~$u$ and~$(\lambda_0,\lambda)$. Therefore, strong duality holds.
	\end{proof}
	
	Theorem~\ref{thm:duality:moment} shows that the worst-case expectation problem~\eqref{eq:worst-case:expectation} over the moment ambiguity set~\eqref{eq:moment-ambiguity-set} admits a semi-infinite dual. Indeed, the dual problem on the right hand side of~\eqref{eq:weak-duality-moments} accommodates finitely many decision variables but infinitely many constraints parametrized by the uncertainty realizations~$z\in\cZ$. The dual problem can also be interpreted as a robust optimization problem with uncertainty set~$\cZ$. Note that we did {\em not} assume~$\cZ$ to be convex. In addition, we emphasize that compactness of~$\cF$ is {\em not} a necessary condition for strong duality. Indeed, strong duality can also be established under Slater-type conditions \citep{zhen2023unification}. Finally, the condition $\rint(\cF)\subseteq \rint(\cC)$ is equivalent to the---seemingly weaker---requirement that~$\cF$ intersects~$\rint(\cC)$. Indeed, if~$\cF\cap \rint(\cC)\neq \emptyset$, then~$\cF$ is not entirely contained in the relative boundary of $\cC$, which implies via \citep[Corollary~6.5.2]{rockafellar1970convex} that $\rint(\cF)\subseteq \rint(\cC)$. 
	
	In the remainder of this section, we use Theorem~\ref{thm:duality:moment} to dualize worst-case expectations problems corresponding to popular classes of moment ambiguity sets. Recall from Section~\ref{sec:chebyshev-with-moment-uncertainty} that the Chebyshev ambiguity set~\eqref{eq:chebyshev-with-moments-in-F} is defined~as
	\begin{equation*}
		\cP = \left\{ \P \in \cP_2(\cZ): \E_\P[Z] = \mu, ~~\E_\P[Z Z^\top] = M ~~ \forall (\mu, M) \in \cF \right\},
	\end{equation*}
	where $\cF \subseteq \R^d \times \S_+^d$ is a closed moment uncertainty set, and $ \cP_2(\cZ)$ denotes the set of all distributions in~$\cP(\cZ)$ with finite second moments. Note that~$\cP$ is an instance of the generic moment ambiguity set~\eqref{eq:moment-ambiguity-set} with moment function~$f(z) = (z, z z^\top)$. 
	
	\begin{theorem}[Duality Theory for Chebyshev Ambiguity Sets]
		\label{thm:duality:Chebyshev}
		If~$\cP$ is the Chebyshev ambiguity set~\eqref{eq:chebyshev-with-moments-in-F}, then the following weak duality relation holds.
		\begin{align}
			\label{eq:weak-duality-chebyshev}
			\sup_{\P \in \cP} ~ \E_\P \left[ \ell(Z) \right] 
			\leq \left\{
			\begin{array}{cl}
				\inf & \lambda_0 + \delta_\cF^*(\lambda, \Lambda) \\[1ex]
				\st & \lambda_0 \in \R, \, \lambda \in \R^d, \, \Lambda \in \S^d \\ [1ex]
				& \lambda_0 + \lambda^\top z + z^\top \Lambda z \geq \ell(z) ~~ \forall z \in \cZ.
			\end{array}
			\right.
		\end{align}
		If~$\E_\P [ \ell (Z) ]>-\infty$ for all~$\P\in\cP_2(\cZ)$ and~$\cF$ is a convex compact set with $M\succ\mu\mu^\top$ for all~$(\mu,M)\in\rint(\cF)$, then strong duality holds, that is, \eqref{eq:weak-duality-chebyshev} becomes an~equality. 
	\end{theorem}
	
	Theorem~\ref{thm:duality:Chebyshev} is a direct corollary of Theorem~\ref{thm:duality:moment}. Thus, we omit its proof. Recall that the Chebyshev ambiguity~\eqref{eq:chebyshev-with-moments-in-F} set with uncertain moments encapsulates the support-only ambiguity set~$\cP(\cZ)$, the Markov ambiguity set~\eqref{eq:Markov}, and the Chebyshev ambiguity set~\eqref{eq:Chebyshev} with fixed moments as special cases. They are recovered by setting $\cF=\R^d\times\S^d$, $\cF=\{\mu\}\times \S^d$ and $\cF=\{\mu\}\times\{M\}$, respectively. The following lemma characterizes the support functions of these moment uncertainty sets in closed form. The proof is elementary and is thus omitted.
	
	\begin{lemma}[Support Functions of Elementary Sets]
		\label{lem:support:elementary}
		The following hold.
		\begin{enumerate}[label=(\roman*)]
			\item If $\cF = \R^d\times\S^d$, then $\delta^*_\cF(\lambda, \Lambda) = \delta_{\{ (0, 0) \}}(\lambda, \Lambda)$.
			\item If $\cF = \{\mu\}\times\S^d$, then $\delta^*_\cF(\lambda, \Lambda) = \lambda^\top \mu + \delta_{\{0\}}(\Lambda)$.
			\item If $\cF = \{\mu\}\times\{M\}$, then $\delta^*_\cF(\lambda, \Lambda) = \lambda^\top \mu + \Tr(\Lambda M)$.
		\end{enumerate}
	\end{lemma}
	When combined with Theorem~\ref{thm:duality:moment}, Lemma~\ref{lem:support:elementary} immediately leads to duality theorems for support-only, Markov, and Chebyshev ambiguity sets. For brevity, we omit the details. In Section~\ref{sec:chebyshev-with-moment-uncertainty}, we have also defined the Gelbrich ambiguity set as a Chebyshev ambiguity set with uncertain moments of the form~\eqref{eq:chebyshev-with-moments-in-F} with~$\cF$ representing the Gelbrich uncertainty set~\eqref{eq:Gelbrich:F} defined as
	\begin{align*}
		\cF = \left\{(\mu, M) \in \R^d \times \S_+^d~:\, \begin{array}{l}
			\exists \Sigma\,\in\S^d_+ \text{ with } M = \Sigma + \mu \mu^\top, \\ \G \left( (\mu, \Sigma), (\hat \mu, \hat \Sigma) \right) \leq r 
		\end{array}
		\right\},
	\end{align*}
	where~$\G$ is the Gelbrich distance of Definition~\ref{def:Gelbrich}.
	In the following we derive the support function~$\delta^*_\cF$ of the Gelbrich uncertainty set~$\cF$.

	\begin{lemma}[Support Function of Gelbrich Uncertainty Sets]
		\label{lem:support:Gelbrich}
		Let~$\cF$ be the Gelbrich uncertainty set~\eqref{eq:Gelbrich:F} of radius~$r\geq 0$ around $(\hat \mu, \hat \Sigma) \in \R^d \times \S_+^d$, where~$\G$ is the Gelbrich distance of Definition~\ref{def:Gelbrich}. 
		For any $(\lambda, \Lambda) \in \R^d \times \S^d$, we then have
		\begin{align*}
			\delta^*_{\cF}(\lambda, \Lambda) =
			\left\{ 
			\begin{array}{cl}
				\inf & \gamma \big( r^2 - \| \hat \mu \|^2 - \Tr(\hat \Sigma) \big) + \Tr(A) + \alpha \\[0.5ex]
				\st & \alpha, \gamma \in \R_+, \; A \in \S_+^d \\[1ex]
				& \begin{bmatrix} \gamma I_d- \Lambda & \gamma \hat \Sigma^\half \\[1ex] \gamma \hat \Sigma^\half & A \end{bmatrix} \succeq 0, \; \begin{bmatrix} \gamma I_d- \Lambda & \gamma \hat \mu + \frac{\lambda}{2} \\[1ex] (\gamma \hat \mu + \frac{\lambda}{2})^\top & \alpha \end{bmatrix} \succeq 0.
			\end{array} 
			\right.
		\end{align*}
	\end{lemma}
	
	\begin{proof}
		By Proposition~\ref{prop:U:Gelbrich:SDP}, which provides a semidefinite representation of the Gelbrich uncertainty set~$\cF$, the support function of~$\cF$ satisfies
		\begin{align*}
			\delta^*_{\cF}(\lambda, \Lambda) 
			&= \left\{
			\begin{array}{cl}
				\sup & \mu^\top \lambda + \Tr(M \Lambda) \\[0.5ex]
				\st & \mu \in \R^d,\; M, U \in \S_+^d,\; C \in \R^{d \times d} \\[0.5ex]
				& \Tr(M - 2 \mu \hat \mu^\top - 2 C) \leq r^2 - \| \hat \mu \|^2 - \Tr(\hat \Sigma) \\[0.5ex]
				& \begin{bmatrix} M - U  & C \\ C^\top & \hat \Sigma \end{bmatrix} \succeq 0 ,\; \begin{bmatrix} U & \mu \\ \mu^\top & 1 \end{bmatrix} \succeq 0.
			\end{array}
			\right.
		\end{align*}
		By conic duality \cite[Theorem~1.4.2]{ben2001lectures}, the maximization problem in the above expression admits the dual minimization problem
		\begin{align*}
			\begin{array}{cl}
				\inf & \gamma \big( r^2 - \| \hat \mu \|^2 - \Tr(\hat \Sigma) \big) + \Tr(\hat \Sigma A_{22}) + \alpha \\ [0.5ex]
				\st & \alpha,\gamma \in \R_+,\; A_{11}, A_{22}, B\in\S^d_+ \\[1ex]
				& \begin{bmatrix} A_{11} & \gamma I_d\\ \gamma I_d& A_{22} \end{bmatrix} \succeq 0 ,\; \begin{bmatrix} B & \gamma \hat \mu + \frac{\lambda}{2} \\ (\gamma \hat \mu + \frac{\lambda}{2})^\top & \alpha \end{bmatrix} \succeq 0 ,\; \gamma I_d- \Lambda \succeq A_{11} \succeq B.
			\end{array}
		\end{align*}
		Strong duality holds because $\alpha = \| 2 \gamma \hat \mu + \lambda \|^2$, $\gamma = \max \{ \lambda_{\max}(\Lambda), 0 \} + 4$, $A_{11} = 2I$, $A_{22} = \gamma^2 I$ and $B = I$ represents a Slater point for the dual problem. At optimality, we have $\gamma I_d- \Lambda = A_{11} = B$. Hence, the dual problem can be further simplified to
		\begin{align*}
			\begin{array}{cl}
				\inf & \gamma \big( r^2 - \| \hat \mu \|^2 - \Tr(\hat \Sigma) \big) + \Tr(\hat \Sigma A_{22}) + \alpha \\ [0.5ex]
				\st & \alpha,\gamma \in \R_+, \; A_{22}\in\S^d_+\\[1ex]
				& \begin{bmatrix} \gamma I_d- \Lambda & \gamma I_d\\ \gamma I_d& A_{22} \end{bmatrix} \succeq 0 ,\; \begin{bmatrix} \gamma I_d- \Lambda & \gamma \hat \mu + \frac{\lambda}{2} \\ (\gamma \hat \mu + \frac{\lambda}{2})^\top & \alpha \end{bmatrix} \succeq 0.
			\end{array}
		\end{align*}
		The substitution $A\leftarrow\hat \Sigma^{\frac{1}{2}} A_{22} \hat \Sigma^{\frac{1}{2}}$ and the equivalence
		\[
		\begin{bmatrix} \gamma I_d- \Lambda & \gamma I_d\\ \gamma I_d& A_{22} \end{bmatrix} \succeq 0 \iff \begin{bmatrix} I_d& 0 \\ 0 & \hat\Sigma^\half \end{bmatrix} \begin{bmatrix} \gamma I_d- \Lambda & \gamma I_d\\ \gamma I_d& A_{22} \end{bmatrix} \begin{bmatrix} I_d& 0 \\ 0 & \hat\Sigma^\half \end{bmatrix} \succeq 0
		\]
		then yield the desired semidefinite program. Thus, the optimal value of this semi\-definite program equals indeed~$\delta^*_\cF(\lambda,\Lambda)$.
	\end{proof}
	
	Armed with Theorem~\ref{thm:duality:Chebyshev} and Lemma~\ref{lem:support:Gelbrich}, we are now prepared to dualize the worst-case expectation problem over a Gelbrich ambiguity set.  
	
	\begin{theorem}[Duality Theory for Gelbrich Ambiguity Sets] \label{thm:duality:Gelbrich}
		If~$\cP$ is the Chebyshev ambiguity set~\eqref{eq:chebyshev-with-moments-in-F} with~$\cF$ representing the Gelbrich uncertainty~\eqref{eq:Gelbrich:F}, then the following weak duality relation holds.
		\begin{align}
			\label{eq:weak-duality-gelbrich}
			\sup_{\P \in \cP} \E_\P \left[ \ell(Z) \right] 
			\leq \left\{\!\!
			\begin{array}{c@{~\,}l}
				\inf & \lambda_0 + \gamma \big( r^2 - \| \hat \mu \|^2 - \Tr(\hat \Sigma) \big) + \Tr(A) + \alpha \\[0.5ex]
				\st & \lambda_0 \in \R, \, \alpha, \gamma \in \R_+, \, \lambda \in \R^d, \, \Lambda \in \S^d, \, A \in \S_+^d \\ [0.5ex]
				& \lambda_0 + \lambda^\top z + z^\top \Lambda z \geq \ell(z) ~~ \forall z \in \cZ \\[0.5ex]
				& \begin{bmatrix} \gamma I_d- \Lambda & \gamma \hat \Sigma^\half \\[1ex] \gamma \hat \Sigma^\half & A \end{bmatrix} \succeq 0, \begin{bmatrix} \gamma I_d- \Lambda & \gamma \hat \mu + \frac{\lambda}{2} \\[1ex] (\gamma \hat \mu + \frac{\lambda}{2})^\top & \alpha \end{bmatrix} \succeq 0.
			\end{array}
			\right.
		\end{align}
		If $\E_\P [ \ell (Z) ]>-\infty$ for all~$\P\in\cP_2(\cZ)$ and~$r>0$, then strong duality holds, that is, the inequality~\eqref{eq:weak-duality-gelbrich} becomes an equality.
	\end{theorem}
	
	\begin{proof}
		Weak duality follows immediately from the first claim of Theorem~\ref{thm:duality:Chebyshev} and Lemma~\ref{lem:support:Gelbrich}. To prove strong duality, recall from Proposition~\ref{prop:U:Gelbrich:SDP} that the Gelbrich uncertainty set~$\cF$ is convex and compact. In addition, recall from the proof of Proposition~\ref{prop:Gelbrich:SDP} that the Gelbrich distance is continuous. As~$r>0$, this implies that
		\begin{align*}
			\rint(\cF) = \left\{(\mu, M) \in \R^d \times \S_+^d~: M \succ \mu \mu^\top,\; \G \left( (\mu, M-\mu\mu^\top), (\hat \mu, \hat \Sigma) \right) < r \right\},
		\end{align*}
		which in turn ensures that $M\succ\mu\mu^\top$ for all $(\mu,M)\in\rint(\cF)$. Therefore, strong duality follows from the second claim of Theorem~\ref{thm:duality:Chebyshev}. 
	\end{proof}
	
	We close this section with some historical remarks. The classical problem of moments asks whether there exists a distribution on~$\cZ$ with a given sequence of moments. In the language of this survey, the problem of moments thus seeks to determine whether a given moment ambiguity set of the form~\eqref{eq:moment-ambiguity-set} is non-empty, where~$f$ is a polynomial and~$\cF$ is a singleton. The analysis of moment problems has a long and distinguished history in mathematics dating back to the 19th century. Notable contributions were made by \citet{tchebichef1874valeurs, markov1884certain, stieltjes1894recherches, hamburger1920moment} and \citet{hausdorff1923moment}; see \citep{shohat1950problem} for an early survey. The study of moment problems with tools from mathematical optimization---in particular semi-infinite duality theory---was pioneered by \citet{isii1960extrema, isii1962sharpness}. \citet{shapiro2001duality} formulates the worst-case expectation problem over a family of distributions with prescribed moments as an infinite-dimensional conic linear program and establishes conditions for strong duality.

	\subsection{$\phi$-Divergence Ambiguity Sets}
	\label{sec:phi:duality}
	Recall from Section~\ref{sec:phi-divergence-amgibuity-sets} that the $\phi$-divergence ambiguity set~\eqref{eq:phi-divergence-ambiguity-set} is defined as
	\begin{align*}
		\cP = \left\{ \P \in \cP(\cZ) \, : \, \D_\phi(\P, \hat \P) \leq r \right\}.
	\end{align*}
	Here, $\cZ$ is a closed support set, $r\geq 0$ is a size parameter, $\phi$ is an entropy function in the sense of Definition~\ref{def:phi}, $\D_\phi$ is the corresponding $\phi$-divergence in the sense of Definition~\ref{def:D_phi}, and~$\hat\P\in\cP(\cZ)$ is a reference distribution. It is expedient to extend~$\D_\phi$ to arbitrary measures. By slight abuse of notation, we thus define the $\phi$-divergence of~$v\in\cM_+(\cZ)$ with respect to~$\hat v \in \cM_+(\cZ)$ as
	\begin{align*}
		\D_\phi(v, \hat v) = \int_{\cZ}
		\phi^\pi \left( \frac{\diff v}{\diff \rho}(z), \frac{\diff \hat v}{\diff \rho}(z) \right) \diff \rho(z), 
	\end{align*}
	where~$\rho\in\cM_+(\cZ)$ is a dominating measure with~$v , \hat v \ll \rho$. An obvious generalization of Proposition~\ref{prop:dual-phi-divergences} implies that $\D_\phi(v,\hat v)$ is convex in~$(v,\hat v)$ and independent of the choice of~$\rho$. By using the extension of~$\D_\phi$ to general measures, the worst-case expectation problem~\eqref{eq:worst-case:expectation} over the ambiguity set~\eqref{eq:phi-divergence-ambiguity-set} can now be recast as
	\begin{align*}
		\sup_{\P \in \cP} \; \E_\P[\ell(Z)] = - h(1, r),
	\end{align*}
	where the auxiliary function $h: \R^2 \to \overline \R$ is defined through
	\begin{align}
		\label{eq:h:phi}
		h(u_0, u) = \inf_{v \in \cM_+ (\cZ)} \left\{ -\int_{\cZ} \ell(z) \, \diff v(z): \int_{\cZ} \diff v(z) = u_0, \; 
		\D_\phi(v, \hat \P) \leq u \right\}.
	\end{align}
	As the objective and constraint functions of the minimization problem in~\eqref{eq:h:phi} are jointly convex in~$v$, $u_0$ and~$u$, Lemma~\ref{lem:param:cvx} implies that~$h$ is convex. Clearly, we have~$\dom(h)\subseteq \R_+^2$. Under mild regularity conditions, one can additionally show that $\{1\} \times \R_{++} \subseteq \rint(\dom(h))$.

	\begin{lemma}[Domain of~$h$]
		\label{lem:domain:h:phi-divergence}
		If $\E_{\hat\P}[\ell(Z)]>-\infty$ and~$\phi$ is continuous at~$1$, then \[\{1\} \times \R_{++} \subseteq \rint(\dom(h)).\]
	\end{lemma}
	
	\begin{proof}
		If~$u_0=1$ and~$u>0$, then~$v=\hat\P$ is feasible in~\eqref{eq:h:phi}. Indeed, $\hat\P$ obeys both constraints, and its objective function value satisfies $\E_{\hat\P}[\ell(Z)]>-\infty$. If we perturb~$u_0$ and~$u$ locally, then~$u_0\hat\P$ satisfies the equality constraint, and the objective function does {\em not} evaluate to~$-\infty$ for all~$u_0\geq 0$. The inequality constraint, on the other hand, is satisfied for all~$u>0$ and all~$u_0$ that are sufficiently close to~$1$ because 
		\[
		\D_\phi(u_0\hat \P,\hat\P)=\phi^\pi(u_0,1)=\phi(u_0)<u.
		\]
		Here, the first equality follows from the definition of~$\D_\phi$ with~$\rho=\hat\P$, the second equality follows from the definition of the perspective function~$\phi^\pi$, and the inequality holds because~$\phi(1)=0$, $u>0$ and~$\phi(u_0)$ is continuous at~$u_0=1$. This confirms that $(1,u)\in\rint(\dom(h))$ for every $u>0$, and thus the claim follows.
	\end{proof}
	
	The following two lemmas are instrumental to derive the bi-conjugate of~$h$. 
	
	\begin{lemma}[Conjugates of Scaled Perspective Functions]
		\label{lem:conjugate-of-perspective}
		If $\phi$ is an entropy function in the sense of Definition~\ref{def:phi}, $t\in\R$, $\beta\in \R_+$ and $\lambda\in\R_{++}$, then we have
		\begin{align*}
			\sup_{\alpha\in\R} t\alpha-\lambda\phi^\pi(\alpha,\beta)=\left\{ \begin{array}{ll}
				\beta \lambda\phi^*(t/\lambda) & \text{if }\beta>0, \\
				\lambda\delta_{\cl(\dom(\phi^*))}(t/\lambda) & \text{if }\beta=0.
			\end{array}\right.
		\end{align*}
	\end{lemma}
	\begin{proof}
		If~$\beta>0$, then we have
		\begin{align*}
			\sup_{\alpha\in\R} t\alpha-\lambda\phi^\pi(\alpha,\beta)= \sup_{\alpha \in \R} t \alpha - \lambda \beta \phi( \alpha/\beta)  =  \beta \, \sup_{\alpha \in \R} t \alpha - \lambda \phi( \alpha ) = \beta \lambda \phi^* ( t/\lambda ),
		\end{align*}
		where the three equalities follow from the definition of the perspective function~$\phi^\pi$, the substitution $\alpha\leftarrow \alpha/\beta$ and the replacement of $t$ by $\lambda t / \lambda$, respectively. Note that these manipulations are admissible because~$\beta,\lambda>0$. If $\beta = 0$, then we have 
		\begin{align*}
			\sup_{\alpha\in\R} t\alpha-\lambda\phi^\pi(\alpha,\beta)& = \sup_{\alpha \in \R} t \alpha - \lambda \phi^\infty ( \alpha ) \\
			& = \sup_{\alpha \in \R} t \alpha - \lambda \delta^*_{\dom(\phi^*)} ( \alpha ) = \lambda \delta_{\cl(\dom(\phi^*))}( t/\lambda),
		\end{align*}
		where the first equality holds again because of the definition of~$\phi^\pi$, and the second equality exploits \cite[Theorem~13.3]{rockafellar1970convex}. The third equality replaces $t$ with $\lambda t/\lambda$ and exploits the elementary observation that the conjugate of the support function of a convex set coincides with the indicator function of the closure of this set \citep[Theorem~13.2]{rockafellar1970convex}. Thus, the claim follows.
	\end{proof}
	
	\begin{lemma}[Domain of Conjugate Entropy Functions]
		\label{lem:domain-phi*}
		If $\phi$ is an entropy function in the sense of Definition~\ref{def:phi}, then we have
		\[
		\cl(\dom(\phi^*)) = \begin{cases}
			(-\infty, \phi^\infty(1)] & \text{if } \phi^\infty(1)<\infty,\\
			\R & \text{if } \phi^\infty(1)=\infty.
		\end{cases}
		\]
	\end{lemma}
	\begin{proof}
		As $\phi$ is proper, convex and closed, \citep[Theorem~8.5]{rockafellar1970convex} implies that its recession function~$\phi^\infty$ is positive homogeneous. Recall that $\phi(s)=\infty$ for every~$s<0$. We may thus conclude that $\phi^\infty(t)=t\,\phi^\infty(1)$ for $t>0$, $\phi^\infty(t)=0$ for~$t=0$ and $\phi^\infty(t)=\infty$ for $t<0$. In addition, \citep[Theorem~13.3]{rockafellar1970convex} implies that the support function of $\dom(\phi^*)$ coincides with the recession function~$\phi^\infty$. The indicator function of~$\cl(\dom(\phi^*))$ is known to coincide with the conjugate of the support function of~$\dom(\phi^*)$, and therefore it satisfies
		\[
		\delta_{\cl(\dom(\phi^*))} (s) = \sup_{t\in \R_+} \left( s- \phi^\infty(1)\right) t = \left\{\begin{array}{ll} 0 & \text{if } s\leq \phi^\infty(1), \\ \infty & \text{otherwise.}\end{array}\right.
		\]
		This shows that $\cl(\dom(\phi^*)) =(-\infty, \phi^\infty(1)]$ if $\phi^\infty(1)<\infty$ and that $\cl(\dom(\phi^*)) =\R$ otherwise. Hence, the claim follows.
	\end{proof}
	
	\begin{table}[!t]
		\setlength\tabcolsep{3pt}
		\footnotesize
		\centering
		\begin{tabular}{l @{\qquad} l @{\quad} l @{\quad} l}
			\toprule
			Divergence & $\phi(s)~(s\geq 0)$ & $\phi^\infty(1)$ & $\phi^*(t)$ \\ \hline
			Kullback-Leibler& $s \log(s) - s + 1$ & $\infty$ & $\mathrm{e}^t - 1$ \\[1ex]
			Likelihood & $-\log(s) + s - 1$ & $1$ & $-\log(1-t)$ \\[1ex]
			Total variation & $\half|s-1|$ & $\frac{1}{2}$ & $\max\{t, -1/2 \} + \delta_{(-\infty, 1/2]}(t)$ \\[1ex]
			Pearson $\chi^2$ & $(s-1)^2$ & $\infty$ & $(t/2+1)^2_+ - 1 $  \\[1ex]
			Neyman $\chi^2$ & $\frac{1}{s}(s-1)^2$ & $1$ & $2 - 2 \sqrt{1-t}$  \\[1ex]
			Cressie-Read for $\beta \in (0, 1)$ & $\frac{s^\beta - \beta s + \beta - 1}{\beta(\beta-1)}$ & $1$ & $\frac{ [(\beta - 1) t + 1 ]_+^{{\beta}/{(\beta-1)}}}{\beta} $  \\[1ex]
			Cressie-Read for $\beta > 1$ & $\frac{s^\beta - \beta s + \beta - 1}{\beta(\beta-1)}$ & $\infty$ & $\frac{ [(\beta - 1) t + 1 ]_+^{{\beta}/{(\beta-1)}}}{\beta} $  \\[1ex]
			\bottomrule
		\end{tabular}
		\caption{Examples of entropy functions, their asymptotic slopes and their conjugates. Here, for any $c\in\R$, we use the $[c]_+$ as a shorthand for $\max \{ c, 0 \}$.}
		\label{tab:phi-divergences-semiinfinite-reformulation}
	\end{table}
	
	\begin{proposition}[Bi-conjugate of~$h$]
		\label{prop:duality:divergence}
		Assume that $\E_{\hat\P}[\ell(Z)]>-\infty$. Then, the bi-conjugate of~$h$ defined in~\eqref{eq:h:phi} satisfies
		\begin{align*}
			h^{**}(u_0, u) = \left\{ \begin{array}{cl} \displaystyle \sup_{\lambda_0\in\R, \lambda \in \R_+} & \displaystyle -\lambda_0 u_0 - \lambda u - \E_{\hat \P} \left[ (\phi^*)^\pi \left( \ell(Z) - \lambda_0, \lambda \right) \right] \\
				\st & \displaystyle \lambda_0+ \lambda\,\phi^\infty(1) \geq \sup_{z\in\cZ} \ell(z),
			\end{array}\right.
		\end{align*}
		where the product $\lambda\,\phi^\infty(1)$ is assumed to evaluate to~$\infty$ if~$\lambda=0$ and~$\phi^\infty(1) =\infty$. If~$\phi$ is continuous at~$1$, then $h^{**}$ coincides with~$h$ on~$\{1 \} \times \R_{++}$.
	\end{proposition}
	
	As~$\phi(1)=0$, we have~$\phi^*(\tau) = \sup_{\alpha\in\R} \tau\alpha-\phi(\alpha)\geq\tau$ for all~$\tau\in\R$. This readily implies that $(\phi^*)^\pi(\tau,\lambda)\geq \tau$ for all $\tau, \lambda\in\R$. Hence, $\E_{\hat \P} [ (\phi^*)^\pi( \ell(Z) - \lambda_0, \lambda)]\geq \E_{\hat \P} [\ell(Z)-\lambda_0]$. In addition, $\phi^*$ is non-decreasing because $\dom(\phi)\subseteq \R_+$. Examples of common entropy functions and their conjugates are listed in Table~\ref{tab:phi-divergences-semiinfinite-reformulation}.

	\begin{proof}[Proof of Proposition~\ref{prop:duality:divergence}]
		For any fixed $(\lambda_0,\lambda)\in\R^2$, the conjugate of~$h$ satisfies
		\begin{align*}
			h^*(-\lambda_0, -\lambda) 
			&= \sup_{(u_0, u) \in \R^2} \; -\lambda_0 u_0 - \lambda u - h(u_0, u) \\
			&= \sup_{u \in \R_+, v \in \cM_+(\cZ)} \left\{ -\lambda u + \int_\cZ \left( \ell(z) - \lambda_0 \right) \diff v(z) : \D_\phi(v, \hat\P) \leq u \right\},
		\end{align*}
		where the second equality holds because $\int_\cZ \diff v(z) = u_0$ and~$\D_\phi(v, \hat\P)\geq 0$. As $\E_{\hat\P}[\ell(Z)]>-\infty$, the resulting maximization problem over~$u$ is unbounded whenever~$\lambda < 0$. If~$\lambda > 0$, on the other hand, then we find
		\begin{align*}
			h^*(-\lambda_0, -\lambda) 
			&= \sup_{v \in \cM_+(\cZ)}  \int_\cZ \left( \ell(z) - \lambda_0 \right) \diff v(z) - \lambda \D_\phi(v, \hat\P)\\[0.5ex]
			&= \left\{ \begin{array}{cl} \sup & \displaystyle \int_{\cZ} \left( \ell(z) - \lambda_0 \right) \frac{\diff v}{\diff \rho}(z) - \lambda \phi^\pi\left( \frac{\diff v}{\diff \rho}(z), \frac{\diff \hat \P}{\diff \rho}(z) \right) \diff \rho(z)\\[2ex]
				\st & v, \rho \in \cM_+(\cZ), ~  v\ll\rho,~ \hat\P\ll \rho,
			\end{array}\right.
		\end{align*}
		where the second equality exploits the definition of~$\D_\phi$. Note that~$\frac{\diff v}{\diff \rho}(z)$ and~$\frac{\diff \hat \P}{\diff \rho}(z)$ belong to the space~$\cL_1(\rho)$ of all $\rho$-integrable Borel functions that can be represented as the Radon-Nikodym derivative of some measure in~$\cM_+ (\cZ)$ with respect to~$\rho$. Introducing auxiliary decision variables $\alpha,\beta\in\cL_1(\rho)$ for the Radon-Nikodym derivatives of~$v$ and~$\hat\P$, respectively, and eliminating the measure~$v$ yields
		\begin{align}
			\label{eq:h*-phi-divergnence}
			h^*(-\lambda_0, -\lambda) = \left\{
			\begin{array}{cl}
				\sup & \displaystyle \int_{\cZ} \left( \ell(z) - \lambda_0 \right) \alpha(z) - \lambda \phi^\pi\left( \alpha(z), \beta(z) \right) \diff \rho(z) \\[2ex]
				\st & \rho \in \cM_+(\cZ), ~ \alpha,\beta \in \cL_1(\rho) \\[0.5ex]
				& \displaystyle \frac{\diff \hat \P}{\diff \rho} =\beta~~\rho\text{-a.s.}
			\end{array}
			\right.
		\end{align}
		For any $\rho \in \cM_+(\cZ)$ and $\beta \in \cL_1(\rho)$ with $\beta=\frac{\diff\hat\P}{\diff\rho}$ $\rho$-almost surely, we then find
		\begin{align}
			& \sup_{\alpha \in \cL_1(\rho)} \int_{\cZ} \left( \ell(z) - \lambda_0 \right) \alpha(z) - \lambda \phi^\pi\left( \alpha(z), \beta(z) \right) \diff \rho(z) \nonumber \\
			& = \int_{\cZ} \sup_{\alpha \in \R} \big\{ \left( \ell(z) - \lambda_0 \right) \alpha - \lambda \phi^\pi ( \alpha, \beta(z))\big\} \, \diff \rho(z), \label{eq:magic_equation}
		\end{align}
		where the equality follows from \citep[Theorem~14.60]{rockafellar2009variational}, which applies because the negation of the function in curly brackets in~\eqref{eq:magic_equation} constitutes a normal integrand in the sense of \citep[Definition~14.27]{rockafellar2009variational}. This can be verified by recalling that sums and perspectives of normal integrands are normal integrands \cite[Proposition~14.45 \& Example~14.48]{rockafellar2009variational}. Next, we partition~$\cZ$ into $\cZ_+(\beta) = \{ z \in \cZ: \beta(z) > 0 \}$ and $\cZ_0(\beta) = \{ z \in \cZ: \beta(z) = 0 \}$. By Lemma~\ref{lem:conjugate-of-perspective}, the integral~\eqref{eq:magic_equation} equals
		\begin{align*}
			& \int_{\cZ_+(\beta)} \lambda \phi^* \left( \frac{\ell(z) - \lambda_0}{\lambda} \right) \beta(z) \, \diff \rho(z) + \int_{\cZ_0(\beta)} \lambda \delta_{\cl(\dom(\phi^*))}\left( \frac{\ell(z) - \lambda_0}{\lambda} \right) \diff \rho(z).
		\end{align*}
		As $\beta=\frac{\diff\hat\P}{\diff\rho}$ $\rho$-almost surely, and as $\hat\P(Z\in\cZ_+(\beta))=1$, the first of these integrals simply reduces to an expectation with respect to the reference distribution and is thus independent of~$\beta$. The second integral still depends on~$\beta$ through the integration domain~$\cZ_0(\beta)$. Thus, partially maximizing over~$\alpha$ allows us to recast~\eqref{eq:h*-phi-divergnence} as 
		\begin{align*}
			& h^*(-\lambda_0, -\lambda) =  \E_{\hat \P} \left[ \lambda \phi^* \left( \frac{\ell(Z) - \lambda_0}{\lambda} \right) \right] \\ & \qquad + \sup_{\substack{\rho \in \cM_+(\cZ), \\ \beta\in\cL_1(\rho)}} \left\{\int_{\cZ_0(\beta)} \lambda \delta_{\cl(\dom(\phi^*))}\left( \frac{\ell(z) - \lambda_0}{\lambda} \right) \diff \rho(z): \frac{\diff \hat \P}{\diff \rho} =\beta~\rho\text{-a.s.} \right\}. 
		\end{align*}
		If there exists~$z_0\in\cZ$ with $(\ell(z_0)-\lambda_0)/\lambda\not\in\cl(\dom(\phi^*))$, then~$h^*(-\lambda_0, -\lambda)=\infty$. To see this, assume first that~$z_0$ is an atom of~$\hat\P$. In this case, the expectation in the first line evaluates to~$\infty$. If~$z_0$ is {\em not} an atom of~$\hat\P$, then the supremum in the second line evaluates to~$\infty$ because we may set $\rho=\hat\P+\delta_{z_0}$ and define $\beta\in\cL_1(\rho)$ through $\beta(z)=1$ if $z\neq z_0$ and $\beta(z_0)=0$. Hence, we may conclude that
		\begin{align*}
			h^*(-\lambda_0, -\lambda) 
			= \begin{cases}
				\E_{\hat \P} \left[ \lambda \phi^* \left( \frac{\ell(Z) - \lambda_0}{\lambda} \right) \right] & \text{if} ~ \frac{\ell(z) - \lambda_0}{\lambda} \in \cl(\dom (\phi^*)) ~ \forall z \in \cZ, \\
				\infty & \text{otherwise.}
			\end{cases}
		\end{align*}
		Note that this formula was derived under the assumption that~$\lambda> 0$. Note also that, by Lemma~\ref{lem:domain-phi*}, the condition $(\ell(z) - \lambda_0)/\lambda \in \cl(\dom (\phi^*))$ is equivalent to the requirement that $ \lambda_0+ \lambda\,\phi^\infty(1)$ is larger than or equal to $\sup_{z\in\cZ} \ell(z)$. We claim that
		\begin{align}
			& h^*(-\lambda_0, -\lambda) \nonumber \\
			&= \begin{cases}
				\E_{\hat \P} \left[ (\phi^*)^\pi \left( \ell(Z) - \lambda_0, \lambda \right) \right] & \text{if } \lambda\geq 0\text{ and } \displaystyle \lambda_0+ \lambda\,\phi^\infty(1) \geq \sup_{z\in\cZ} \ell(z), \\
				\infty & \text{otherwise,}
			\end{cases}
			\label{eq:general-h*-phi-divergence}
		\end{align}
		for all~$\lambda_0,\lambda\in\R$. Indeed, the above reasoning and the definition of the perspective function $(\phi^*)^\pi$ ensure that~\eqref{eq:general-h*-phi-divergence} holds whenever $\lambda\neq 0$. Note that $h^*$ is convex and closed thanks to \citep[Theorem~12.2]{rockafellar1970convex}. The expression on the right hand side of~\eqref{eq:general-h*-phi-divergence} is also convex and closed in~$(\lambda_0,\lambda)$. In particular, it is lower semicontinuous thanks to Fatou's lemma, which applies because $\phi(1)=0$ such that $(\phi^*)^\pi(t,\lambda)\geq t$ for all~$t\in\R$ and~$\lambda\in\R_+$ and because $\E_{\hat\P}[\ell(Z)]>-\infty$. Observe also that $(\phi^*)^\pi$ is proper, closed and convex thanks to \citep[page~35, page~67 \& Theorem~13.3]{rockafellar1970convex}. Hence, \eqref{eq:general-h*-phi-divergence} must indeed hold for all $\lambda_0,\lambda\in\R$. 
		
		Given~\eqref{eq:general-h*-phi-divergence}, we finally obtain
		\begin{align*}
			h^{**}(u_0, u)
			&= \sup_{\lambda_0, \lambda \in \R} ~ -\lambda_0 u_0 - \lambda u - h^*(-\lambda_0, -\lambda) \\
			&= \left\{ \begin{array}{cl} \displaystyle \sup_{\lambda_0\in\R, \lambda \in \R_+} & \displaystyle -\lambda_0 u_0 - \lambda u - \E_{\hat \P} \left[ (\phi^*)^\pi \left( \ell(Z) - \lambda_0, \lambda \right) \right] \\
				\st & \displaystyle \lambda_0+ \lambda\,\phi^\infty(1) \geq \sup_{z\in\cZ} \ell(z),
			\end{array}\right.
		\end{align*}
		which establishes the desired formula for the bi-conjugate of~$h$. It remains to be shown that if $\phi$ is continuous at~$1$, then $h(1,u)=h^{**}(1,u)$ for all~$u\in\R_{++}$. However, this follows immediately from Lemmas~\ref{lem:bi:coincidence} and~\ref{lem:domain:h:phi-divergence}.
	\end{proof}
	
	The following main theorem uses Proposition~\ref{prop:duality:divergence} to dualize the worst-case expectation problem~\eqref{eq:worst-case:expectation} with a $\phi$-divergence ambiguity set. 
	
	\begin{theorem}[Duality Theory for $\phi$-Divergence Ambiguity Sets]
		\label{thm:duality:phi}
		Assume that $\E_{\hat\P}[\ell(Z)]>-\infty$. If~$\cP$ is the $\phi$-divergence ambiguity set~\eqref{eq:phi-divergence-ambiguity-set}, then the following weak duality relation holds.
		\begin{align}
			\label{eq:weak-duality-phi-divergence}
			\sup_{\P \in \cP} \; \E_\P[\ell(Z)] \leq \left\{ \begin{array}{cl} \displaystyle \inf_{\lambda_0\in\R, \lambda \in \R_+} & \displaystyle \lambda_0  + \lambda r + \E_{\hat \P} \left[ (\phi^*)^\pi \left( \ell(Z) - \lambda_0, \lambda \right) \right] \\
				\st & \displaystyle  \lambda_0+ \lambda\,\phi^\infty(1) \geq \sup_{z\in\cZ} \ell(z)
			\end{array}\right.
		\end{align}
		Here, the product $\lambda\,\phi^\infty(1)$ is assumed to evaluate to~$\infty$ if~$\lambda=0$ and~$\phi^\infty(1) =\infty$. If additinally~$r>0$ and~$\phi$ is continuous at~$1$, then strong duality holds, that is, the inequality~\eqref{eq:weak-duality-phi-divergence} collapses to an equality.
	\end{theorem}
	
	\begin{proof}
		Recall first that
		\[
		\sup_{\P \in \cP} \E_\P[\ell(Z)] = - h(1, r)\leq -h^{**}(1,r),
		\]
		where the inequality holds because of Lemma~\ref{lem:bi:coincidence}. Weak duality thus follows from the first claim in Proposition~\ref{prop:duality:divergence}. If~$\phi$ is additionally continuous at~$1$, and if~$r>0$, then strong duality follows from the second claim in Proposition~\ref{prop:duality:divergence}.
	\end{proof}
	
	Recall now that the {\em restricted} $\phi$-divergence ambiguity set~\eqref{eq:restricted-phi-divergence-ambiguity-set} is defined as
	\begin{align*}
		\cP = \left\{ \P \in \cP(\cZ) \, : \, \P \ll \hat \P, \; \D_\phi(\P, \hat \P) \leq r \right\}.
	\end{align*}
	That is, $\cP$ contains all distributions from within the (unrestricted) $\phi$-divergence ambiguity set~\eqref{eq:phi-divergence-ambiguity-set} that are absolutely continuous with respect to~$\hat\P$. The worst-case expected loss over~$\cP$ can again be expressed as~$-h(1,r)$, where~$h(u_0,u)$ is now defined as the infimum of the optimization problem~\eqref{eq:h:phi} with the additional constraint~$v\ll\hat \P$. One readily verifies that~$h$ remains convex and that $\{1\}\times\R_{++}$ is still contained in~$\rint(\dom(h))$ despite this restriction. Indeed, the proof of Lemma~\ref{lem:domain:h:phi-divergence} remains valid almost verbatim.

	\begin{theorem}[Duality Theory for Restricted $\phi$-Divergence Ambiguity Sets]
		\label{thm:duality:restricted:phi}
		Assume that $\E_{\hat\P}[\ell(Z)]>-\infty$. If~$\cP$ is the restricted $\phi$-divergence ambiguity set~\eqref{eq:restricted-phi-divergence-ambiguity-set}, then the following weak duality relation holds.
		\begin{align}
			\label{eq:weak-duality-restricted-phi-divergence}
			\sup_{\P \in \cP} \; \E_\P[\ell(Z)] \leq \inf_{\lambda_0\in\R, \lambda \in \R_+} \lambda_0  + \lambda r + \E_{\hat \P} \left[ (\phi^*)^\pi \left( \ell(Z) - \lambda_0, \lambda \right) \right] 
		\end{align}
		If additionally~$r>0$ and~$\phi$ is continuous at~$1$, then strong duality holds, that is, the inequality~\eqref{eq:weak-duality-restricted-phi-divergence} collapses to an equality.
	\end{theorem}
	
	Note that if~$(\lambda_0,\lambda)$ is feasible in~\eqref{eq:weak-duality-restricted-phi-divergence}, then $(\ell(Z)-\lambda_0,\lambda)$ belongs $\hat\P$-almost surely to~$\dom((\phi^*)^\pi)$. Otherwise, its objective function value equals~$\infty$. In view of Lemma~\ref{lem:domain-phi*}, this implies that $\lambda_0+\lambda\,\phi^\infty(1)\geq \esssup_{\hat\P}[\ell(Z)]$. In contrast, if~$(\lambda_0,\lambda)$ is feasible in~\eqref{eq:weak-duality-phi-divergence}, then it satisfies the constraint $\lambda_0+ \lambda\,\phi^\infty(1) \geq \sup_{z\in\cZ} \ell(z)$, which is more restrictive unless $\phi^\infty(1)=\infty$. Hence, the dual problem in~\eqref{eq:weak-duality-restricted-phi-divergence} has a (weakly) larger feasible set and a (weakly) smaller infimum than the dual problem in~\eqref{eq:weak-duality-phi-divergence}. This is perhaps unsurprising because~\eqref{eq:weak-duality-restricted-phi-divergence} corresponds to the worst-case expectation problem over the restricted $\phi$-divergence ambiguity set, which is (weakly) smaller than the corresponding {\em un}restricted $\phi$-divergence ambiguity set. Note also that the solution of a worst-case expectation problem over an unrestricted $\phi$-divergence ambiguity set depends on~$\cZ$ and not just on the support of~$\hat\P$.

	\begin{proof}[Proof of Theorem~\ref{thm:duality:restricted:phi}]
		If $h(u_0,u)$ is defined as the infimum of the optimization problem~\eqref{eq:h:phi} with the additional constraint~$v\ll\hat \P$, then one can show that
		\begin{align*}
			h^{**}(u_0, u) = \sup_{\lambda_0, \lambda \in \R} -\lambda_0 u_0 - \lambda u - \E_{\hat \P} \left[ (\phi^*)^\pi \left( \ell(Z) - \lambda_0, \lambda \right) \right].
		\end{align*}
		Indeed, one can proceed as in the proof of Proposition~\ref{prop:duality:divergence}. However, the reasoning simplifies significantly because the additional constraint~$v\ll\hat\P$ allows us to set the dominating measure~$\rho$ in the definition of~$\D_\phi$ to~$\hat\P$. Thus, the Radon-Nikodym derivative~$\beta=\diff\hat\P/\diff\rho$ is $\hat\P$-almost surely equal to~$1$. This in turn implies that the calculation of~$h^*$ requires no case distinction, that is, the set~$\cZ_0(\beta)$ is empty. 
		
		Given the bi-conjugate of~$h$, both weak and strong duality can then be established exactly as in the proof of Theorem~\ref{thm:duality:phi}. Details are omitted for brevity.
	\end{proof}
	
	\citet[Proposition~5]{van2021data} establish a strong duality result for worst-case expectations over likelihood ambiguity sets as introduced in Section~\ref{sec:likelihood-ambiguity-sets}. Theorem~\ref{thm:duality:phi} extends this result to general $\phi$-divergence ambiguity sets with a significantly shorter proof that only uses tools from convex analysis. \citet{ben2013robust} establish a strong duality result akin to Theorem~\ref{thm:duality:restricted:phi} for restricted $\phi$-divergence ambiguity sets under the assumption that the reference distribution~$\hat \P$ is discrete. \citet{shapiro2017distributionally} extends this result to general reference distributions by using tools from infinite-dimensional analysis.
	In contrast, our proof of Theorem~\ref{thm:duality:restricted:phi} establishes the same duality result using finite-dimensional convex analysis. 
	

	\subsection{Optimal Transport Ambiguity Sets}
	\label{sec:optimal:transport-duality}
	Recall from Section~\ref{sec:optimal:transport} that the optimal transport ambiguity set~\eqref{eq:OT-ambiguity-set} is defined~as
	\begin{align*}
		\cP = \left\{ \P \in \cP(\cZ) \, : \, \OT_c(\P, \hat \P) \leq r \right\}.
	\end{align*}
	Here, $\cZ$ is a closed support set, $r\geq 0$ is a size parameter, $c$ is a transportation cost function in the sense of Definition~\ref{def:cost}, $\OT_c$ is the corresponding optimal transport discrepancy in the sense of Definition~\ref{def:OT}, and~$\hat\P\in\cP(\cZ)$ is a reference distribution. In analogy to Section~\ref{sec:phi:duality}, the worst-case expectation problem~\eqref{eq:worst-case:expectation} over the ambiguity set~\eqref{eq:OT-ambiguity-set} can now be reformulated as 
	\begin{align*}
		\sup_{\P \in \cP} \; \E_\P[\ell(Z)] = - h(r),
	\end{align*}
	where the auxiliary function $h: \R \to \overline \R$ is defined through
	\begin{align}
		\label{eq:h:OT}
		h(u) = \inf_{\P \in \cP(\cZ)} \left\{ - \E_\P [\ell(Z)] \, : \, 
		\OT_c(\P, \hat \P) \leq u \right\}.
	\end{align}
	As the objective and constraint functions of the minimization problem in~\eqref{eq:h:OT} are jointly convex in~$\P$ and~$u$, Lemma~\ref{lem:param:cvx} implies that~$h$ is convex. Recall also that~$c$ is non-negative and satisfies~$c(z,z)=0$ for all~$z\in\cZ$. If~$\E_{\hat \P}[\ell(\hat Z)]>-\infty$, it is therefore easy to show that~$\dom(h)=\R_+$.
	
	The following lemma will be instrumental for deriving the bi-conjugate of~$h$. Recall that~$\Gamma(\P, \hat{\P})$ denotes the set of all couplings of~$\P$ and~$\hat \P$; see Definition~\ref{def:OT}. 
	
	\begin{lemma}[Interchangeability Principle]
		\label{lem:interchangeability}
		If~$c$ is a transportation cost function, $\ell$ is upper semicontinuous and~$\lambda~\geq 0$, then we have
		\begin{align*}
			\sup_{\P \in \cP(\cZ)} \sup_{\gamma \in \Gamma(\P, \hat\P)} \E_{\gamma} \left[\ell(Z) -\lambda c(Z, \hat Z) \right] 
			= \E_{\hat\P} \left[ \sup_{z \in \cZ} \; \ell(z) - \lambda c(z,\hat Z) \right].
		\end{align*}
	\end{lemma}
	
	One can show that Lemma~\ref{lem:interchangeability} remains valid, for example, if~$\cZ$ is a Polish (separable metric) space equipped with its Borel $\sigma$-algebra and even if~$c$ and~$\ell$ fail to be lower and upper semicontinuous, respectively \citep[Proposition~1]{zhang2022simple}. 
	
	\begin{proof}[Proof of Lemma~\ref{lem:interchangeability}] 
		Define $L:\cZ\to\overline\R$ through $L(\hat z) = \sup_{z \in \cZ} \ell(z) - \lambda c(z, \hat z)$. If~$\lambda=1$, then~$L$ reduces to the $c$-transform of~$\ell$ defined in~\eqref{eq:c-transofrm-of-f}. Note first that $\lambda c(z, \hat z) - \ell(z)$ is lower semicontinuous in~$(z,\hat z)$ and thus constitutes a normal integrand thanks to \citep[Example~14.31]{rockafellar2009variational}. This implies via \citep[Theorem 14.37]{rockafellar2009variational} that $L(\hat z)$ is Borel-measurable. 
		
		Observe next that, by the definition of~$L$, we have $\ell(z) - \lambda c(z, \hat z) \leq L(\hat z)$ for all~$z, \hat z \in \cZ$. This inequality persists if we integrate both sides with respect to any coupling~$\gamma \in \Gamma(\P, \hat\P)$ for any distribution~$\P\in\cP(\cZ)$, and therefore we obtain
		\begin{align*}
			\sup_{\P \in \cP(\cZ)} \sup_{\gamma \in \Gamma(\P, \hat\P)} \E_{\gamma} \left[\ell(Z) -\lambda c(Z, \hat Z) \right] \leq \E_{\hat\P} \left[ L(\hat Z) \right].
		\end{align*}
		It remains to prove the reverse inequality. To this end, observe that
		\begin{align*}
			\E_{\hat\P} \left[ L(\hat Z) \right] = \E_{\hat\P} \left[ \sup_{z \in \cZ} \ell(z) - \lambda c(z, \hat Z) \right] &= \sup_{f \in \cF} \E_{\hat\P} \left[ \ell(f(\hat Z)) - \lambda c(f(\hat Z), \hat Z) \right] \\
			&\leq \sup_{\P \in \cP(\cZ)} \sup_{\gamma \in \Gamma(\P, \hat\P)} \E_{\gamma} \left[\ell(Z) -\lambda c(Z, \hat Z) \right],
		\end{align*}
		where $\cF$ denotes the family of all Borel functions $f:\cZ\to\cZ$. The second equality follows from \citep[Theorem~14.60]{rockafellar2009variational}, which applies because $\lambda c(z, \hat z) - \ell(z)$ is a normal integrand. Note that the joint distribution of~$f(\hat Z)$ and~$\hat Z$ under~$\hat\P$ coincides with the pushforward distribution $\gamma=\hat\P\circ g^{-1}$, where $g:\cZ\to \cZ\times\cZ$ is defined through $g(\hat z)=(f(\hat z), \hat z)$. By construction, we have $\gamma\in\Gamma(\hat\P\circ f^{-1}, \hat \P)$. The inequality in the above expression therefore holds because~$\hat\P\circ f^{-1}\in\cP(\cZ)$. This observation completes the proof. 
	\end{proof}
	
	\begin{proposition}[Bi-conjugate of~$h$]
		\label{prop:duality:ot}
		Assume that $\E_{\hat\P}[\ell(\hat Z)]>-\infty$ and that~$\ell$ is upper semicontinuous. Then, the bi-conjugate of~$h$ defined in~\eqref{eq:h:OT} satisfies
		\begin{align*}
			h^{**}(u) = \sup_{\lambda \geq 0} ~ - \lambda r - \E_{\hat \P} \left[\sup_{z \in \cZ} \ell(z) - \lambda c(z, \hat Z) \right].
		\end{align*}
		In addition, $h^{**}$ coincides with~$h$ on~$\R_{++}$.
	\end{proposition}
	
	
	\begin{proof}
		For any fixed~$\lambda\in\R$, the conjugate of~$h$ satisfies
		\begin{align*}
			h^*(-\lambda) 
			&= \sup_{u \in \R} \; - \lambda u - h(u) \\
			&= \sup_{u \in \R_+, \P \in \cP(\cZ)}  \left\{ -\lambda u + \E_{\P} [\ell(Z)] : \OT_c(\P, \hat\P) \leq u \right\},
		\end{align*}
		where the second equality holds because $\OT_c(\P, \hat\P)\geq 0$. As $\E_{\hat\P}[\ell(\hat Z)]>-\infty$, the resulting maximization problem is unbounded if~$\lambda < 0$. If~$\lambda > 0$, then we find
		\begin{align}
			h^*(-\lambda) 
			&= \sup_{\P \in \cP(\cZ)} \; \E_{\P} [\ell(Z)] -\lambda \OT_c(\P, \hat\P) \nonumber \\
			&= \sup_{\P \in \cP(\cZ)} \; \sup_{\gamma \in \Gamma(\P, \hat\P)} \; \E_{\P} [\ell(Z)] -\lambda \E_{\gamma} [c(Z, \hat Z)] \nonumber \\  
			&= \sup_{\P \in \cP(\cZ)} \sup_{\gamma \in \Gamma(\P, \hat\P)} \E_{\gamma} \left[\ell(Z) -\lambda c(Z, \hat Z)
			\right] \nonumber \\
			&= \E_{\hat\P} \left[ \sup_{z \in \cZ} \; \ell(z) - \lambda c(z,\hat Z) \right],
			\label{eq:general-h*-OT}
		\end{align}
		where the second equality follows from Definition~\ref{def:OT}, the third equality holds because the marginal distribution of~$Z$ under~$\gamma$ is given by~$\P$, and the fourth equality exploits Lemma~\ref{lem:interchangeability}. The above reasoning implies that $h^*(-\lambda)$ coincides with~\eqref{eq:general-h*-OT} for all~$\lambda> 0$. However, this formula remains valid at~$\lambda=0$. To see this, note that $h^*$ is convex and closed thanks to \citep[Theorem~12.2]{rockafellar1970convex}. The last expectation in~\eqref{eq:general-h*-OT} is also convex and closed in~$\lambda$ thanks to Fatou's lemma, which applies because $\sup_{z \in \cZ} \ell(z) - \lambda c(z,\hat z)$ is larger than or equal to~$\ell(\hat z)$ and lower semicontinuous in~$\lambda$ for every~$\hat z\in\cZ$ and because $\E_{\hat\P}[\ell(\hat Z)]>-\infty$. Hence, the last expectation in~\eqref{eq:general-h*-OT} is indeed convex and lower-semicontinuous in~$\lambda$, and thus it coincides indeed with $h^*(-\lambda)$ for all~$\lambda\in\R_+$.

		
		Given~\eqref{eq:general-h*-OT}, we finally obtain the following formula for the bi-conjugate of~$h$.
		\begin{align*}
			h^{**}(u)
			&= \sup_{\lambda \geq 0} ~ -\lambda u - h^*(-\lambda) = \sup_{\lambda \geq 0} ~ - \lambda u - \E_{\hat \P} \left[\sup_{z \in \cZ} \ell(z) - \lambda c(z, \hat Z) \right]
		\end{align*}
		Here, the first equality holds because~$h^*(-\lambda)=\infty$ whenever~$\lambda<0$. The second equality follows from~\eqref{eq:general-h*-OT}, which holds for any~$\lambda\geq 0$. This establishes the desired formula for~$h^{**}$. Lemma~\ref{lem:bi:coincidence} and our earlier observation that~$\dom(h)=\R_+$ finlly imply that $h(u)=h^{**}(u)$ for all~$u\in\R_{++}$.
	\end{proof}

	The following main theorem uses Proposition~\ref{prop:duality:ot} to dualize the worst-case expectation problem~\eqref{eq:worst-case:expectation} with an optimal transport ambiguity set. 
	
	\begin{theorem}[Duality Theory for Optimal Transport Ambiguity Sets]
		\label{thm:duality:OT}
		Assume that $\E_{\hat \P}[\ell(\hat Z)]>-\infty$ and $\ell$ is upper semicontinuous. If~$\cP$ is the optimal transport ambiguity set defined in~\eqref{eq:OT-ambiguity-set}, then the following weak duality relation holds.
		\begin{align}
			\label{eq:weak-duality-ot}
			\sup_{\P \in \cP} \; \E_\P[\ell(Z)] 
			\leq \inf_{\lambda \in \R_+} \; \lambda r + \E_{\hat\P} \left[ \sup_{z \in \cZ} \; \ell(z) - \lambda c(z,\hat Z) \right].
		\end{align}
		If~$r > 0$, then strong duality holds, that is, \eqref{eq:weak-duality-ot} collapses to an equality.
	\end{theorem}
	
	\begin{proof}
		Recall first that
		\[
		\sup_{\P \in \cP} \E_\P[\ell(Z)] = - h(r)\leq -h^{**}(r),
		\]
		where the inequality holds because of Lemma~\ref{lem:bi:coincidence}. Weak duality thus follows from the first claim in Proposition~\ref{prop:duality:ot}. If~$r>0$, then strong duality follows from the second claim in Proposition~\ref{prop:duality:ot}. This concludes the proof.
	\end{proof}
	
	\citet{mohajerin2018data} and \citet{zhao2018data} use semi-infinite duality theory to prove Theorem~\ref{thm:duality:OT} in the special case when~$\OT_c$ is the $1$-Wasserstein distance and when the reference distribution~$\hat\P$ is discrete. \citet{blanchet2019quantifying} and \citet{gao2016distributionally} prove a generalization of Theorem~\ref{thm:duality:OT} by leveraging a Fenchel duality theorem in Banach spaces and by devising a constructive argument, respectively. They both allow for arbitrary optimal transport discrepancies as well as arbitrary reference distributions on Polish spaces. The proof shown here, which exploits the interchangeability principle of Lemma~\ref{lem:interchangeability} and elementary tools from convex analysis, is due to \citet{zhang2022simple}.

	\section{Duality Theory for Worst-Case Risk Problems}
	\label{sec:duality-wc-risk}
	The standard DRO problem~\eqref{eq:primal:dro} assumes that the decision-maker is risk-neutral and ambiguity-averse. Risk-neutrality means that if the distribution of~$Z$ is known, then decisions are ranked by their {\em expected} loss. Ambiguity-aversion means that if the distribution of~$Z$ is ambiguous, then expectations are evaluated under a distribution in the ambiguity set~$\cP$ that is {\em most detrimental} to the decision-maker.
	
	If low-probability events have a disproportionate negative impact on the decision-maker, then it is {\em in}appropriate to use the expected loss as a decision criterion even if the distribution of~$Z$ is known. Instead, it is expedient to rank decisions by the {\em risk} of their loss with respect to a law-invariant risk measure. A law-invariant risk measure~$\varrho$ assigns each (univariate) loss distribution in~$\cP(\R)$ a riskiness index. If the loss is representable as~$\ell(Z)$, where $\ell:\R^d\to\R$ is a Borel function and~$Z$ is a $d$-dimensional random vector with probability distribution~$\P$, then the distribution of the loss~$\ell(Z)$ is given by the pushforward distribution $\P\circ\ell^{-1}$. Throughout this paper, we use~$\varrho_\P[\ell(Z)]$ to denote the risk~$\varrho(\P\circ\ell^{-1})$ of such a loss distribution. These conventions are formalized in the following definition. Here and in the remainder we use $\cL(\R^d)$ to denote the family of all Borel functions $\ell:\R^d\to\R$.
	
	\begin{definition}[Law-Invariant Risk Measure]
		\label{def:law-invariant-risk-measure}
		A law-invariant risk measure is a function $\varrho:\cP(\R)\to \overline{\R}$. We use~$\varrho_\P[\ell(Z)]$ to denote~$\varrho(\P\circ\ell^{-1})$ for any Borel function $\ell\in\cL(\R^d)$, Borel distribution~$\P\in\cP(\R^d)$ and dimension~$d\in\N$.
	\end{definition}
	
	
	A law-invariant risk measure~$\varrho$ has the property that if~$\P_1\circ\ell_1^{-1} = \P_2\circ\ell_2^{-1}$ for two different Borel functions~$\ell_1$ and~$\ell_2$ and two different distributions~$\P_1$ and~$\P_2$ on~$\R^{d_1}$ and~$\R^{d_2}$, respectively, then $\varrho_{\P_1}[\ell_1(Z_1)] = \varrho_{\P_2}[\ell_2(Z_2)]$. In fact, this property is the very reason for why~$\varrho$ is called `law-invariant.'
	
	The notation~$\varrho_\P[\ell(Z)]$ is consistent with our usual conventions for the expected value~$\E_\P[\ell(Z)]$, which is a special instance of a law-invariant risk measure. Also, it makes the dependence of the risk on~$\P$ explicit, which is necessary when~$\P$ is ambiguous. We stress that, in contrast to most of the literature on risk measures, our definition of a law-invariant risk measure~$\varrho$ is {\em not} tied to a particular probability space. A prime example of a law-invariant risk measure is the value-at-risk.
	
	\begin{definition}[Value-at-Risk]
		\label{def:var}
		The value-at-risk (VaR) at level~$\beta\in(0,1)$ of an uncertain loss~$\ell(Z)$ with $\ell\in\cL(\R^d)$ and $Z\sim\P\in\cP(\R^d)$ is given by
		\begin{align}
			\label{eq:var}
			\beta\VaR_{\P} [\ell(Z)] = \inf \left\{ \tau \in \R : \P(\ell(Z) \leq \tau) \geq 1-\beta \right\}.
		\end{align}
	\end{definition}
	The VaR is indeed law-invariant because $\P(\ell(Z) \leq \tau)=F(\tau)$ depends on~$\ell$ and~$\P$ only indirectly through the cumulative distribution function~$F$ associated with the pushfoward distribution $\P\circ\ell^{-1}$. Note that the infimum in~\eqref{eq:var} is attained because~$F$ is non-decreasing and right-continuous. By construction, the VaR at level~$\beta$ represents the smallest number~$\tau^\star$ that weakly exceeds the loss with probability~$1-\beta$. Thus, it coincides with the leftmost $(1-\beta)$-quantile of the loss distribution~$F$. For later reference we remark that the $\beta$-VaR can be reformulated~as
	\begin{align}
		\label{eq:var-refomulation}
		\beta\VaR_{\P} [\ell(Z)] = \inf \left\{ \tau \in \R : \P(\ell(Z) \geq \tau) \leq \beta \right\}.
	\end{align}
	However, the infimum in~\eqref{eq:var-refomulation} may {\em not} be attained. Note that the VaR is well-defined and finite for {\em any} loss function~$\ell\in\cL(\R^d)$ and for {\em any} distribution~$\P\in\cP(\R^d)$. Nonetheless, other law-invariant risk measures are finite only for certain sub-classes of loss functions and distributions. In the remainder of this paper we will often study risk measures that display some or all of the following structural properties.
	

	\begin{definition}[Properties of Risk Measures] \label{def:invariance-properties}
		\hspace{-2mm} A law-invariant risk measure~$\varrho$ is 
		\begin{enumerate}[label=(\roman*)]
			\item {\em translation-invariant} if 
			\[
			\varrho_\P[\ell(Z) + c] = \varrho_\P[\ell(Z)] + c\quad \forall \ell\in\cL(\R^d),\; \forall c \in \R,\;\forall \P\in\cP(\R^d);
			\]
			\item {\em scale-invariant} if 
			\[
			\varrho_\P[c\ell(Z)] = c\varrho_\P[\ell(Z)]\quad \forall \ell\in\cL(\R^d),\; \forall c \in \R_+,\; \forall \P\in\cP(\R^d);
			\]
			\item {\em monotone} if 
			\begin{align*}
				& \varrho_\P[\ell_1(Z)] \leq \varrho_\P[\ell_2(Z)]\\
				& \hspace{1cm}\forall \ell_1,\ell_2\in\cL(\R^d)\text{ with }\ell_1(Z)\leq \ell_2(Z)~\P\text{-a.s.},\; \forall \P\in\cP(\R^d);
			\end{align*}
			\item {\em convex} if 
			\begin{align*}
				&\varrho_\P[\theta\ell_1(Z)+(1-\theta)\ell_2(Z)] \leq \theta\varrho_\P[\ell_1(Z)] +(1-\theta)\varrho_\P[\ell_2(Z)]\\
				& \hspace{3cm} \forall \ell_1,\ell_2\in\cL(\R^d),\; \forall \theta\in[0,1],\; \forall \P\in\cP(\R^d).
			\end{align*}
		\end{enumerate}
	\end{definition}
	A {\em coherent} risk measure is translation-invariant, scale-invariant, monotone as well as convex \citep{artzner1999coherent}. In addition, a {\em convex} risk measure is translation-invariant, monotone and convex (but not necessarily scale-invariant).

	
	
	Any law-invariant risk measure~$\varrho$ gives rise to a risk-averse DRO problem
	\begin{align}
		\label{eq:dro-risk}
		\inf_{x \in \cX} \; \sup_{\P \in \cP} \; \varrho_\P \left[ \ell (x, Z) \right].
	\end{align}
	This problem seeks a decision~$x$ that minimizes the worst-case risk of the random loss~$\ell(x,Z)$ with respect to all distributions of~$Z$ in the ambiguity set~$\cP$. Below we will show that the duality theory for worst-case expectation problems developed in Section~\ref{sec:duality-wc-expectation} has ramifications for a broad class of worst-case risk problems of the~form
	\begin{align}
		\label{eq:worst-case:risk}
		\sup_{\P \in \cP} \varrho_{\P}[\ell(Z)].
	\end{align}
	Here, we suppress as usual the dependence of the loss function on~$x$ to avoid clutter. 

	\subsection{Optimized Certainty Equivalents}
	We now describe a class of law-invariant risk measures for which the {\em risk-averse} DRO problem~\eqref{eq:dro-risk} can be converted to an equivalent {\em risk-neutral} DRO problem of the form~\eqref{eq:primal:dro}. This will show that many risk-averse DRO problems are susceptible to methods developed for risk-neutral problems. The risk measures studied in this section are induced by disutility functions in the sense of the following definition.
	
	\begin{definition}[Disutility Function]
		\label{def:generating}
		A disutility function $g : \R \to \R$ is a convex (and therefore continuous) function with $g(0) = 0$ and $g(\tau) > \tau$ for all $\tau \neq 0$.
	\end{definition}
	
	\citet{ben1986expected} use disutility functions to construct a class of law-invariant risk measures, which they term optimized certainty equivalents. Recall that if the objective function of a minimization (maximization) problem can be expressed as the difference of two terms, both of which evaluate to~$\infty$ ({\em e.g.}, the positive and negative parts of an integral), then it should be interpreted as~$\infty$ ($-\infty$). 
	
	\begin{definition}[Optimized Certainty Equivalent]
		\label{def:regular}
		The optimized certainty equivalent induced by the disutility function~$g$ is the law-invariant risk measure~$\varrho$~with
		\begin{align}
			\label{eq:regular}
			\varrho_{\P} [\ell(Z)] = \inf_{\tau \in \R} \; \tau + \E_\P \left[ g(\ell(Z) - \tau) \right].
		\end{align}
	\end{definition}
	
	The expected disutility $\E_\P [ g(\ell(Z))]$ represents a deterministic present loss that the decision-maker considers to be equally (un)desirable as the random future loss~$\ell(Z)$. If it is possible to shift a deterministic portion~$\tau$ of the loss~$\ell(Z)$ to the present, then the decision-maker will solve the minimization problem in~\eqref{eq:regular} in order to strike an optimal trade-off between present and future losses. Hence, it is natural to interpret $\varrho_{\P} [\ell(Z)]$ as an `optimized certainty equivalent.' 
	
	There is also an intimate relation between optimized certainty equivalents and a class of~$\phi$-divergences. To see this, let~$\phi$ be an entropy function in the sense of Definition~\ref{def:phi} with $\phi^\infty(1)=\infty$. Assume also that~$\phi$ is twice continuously differentiable on a neighborhood of~$1$ with $\phi'(1)=0$ and $\phi''(1)>0$. Under these conditions, $\phi^*$ constitutes a disutility function in the sense of Definition~\ref{def:generating}. Indeed, $\phi^*$ is real-valued because~$\phi^\infty(1)=\infty$ and satisfies $\phi^*(t)\geq t$ for all $t\in\R$ because $\phi(1)=0$. Finally, we have $\phi^*(0)=0$ because $\phi'(1)=0$ and $\phi^*(t)>t$ for all $t\neq 0$ because $\phi''(t)>0$. If $\E_{\hat\P}[\ell(Z)]>-\infty$, then the optimized certainty equivalent induced by the disutility function~$g=\phi^*$ satisfies
	\begin{align}
		\label{eq:oce-vs-phi}
		\inf_{\lambda_0 \in \R} \lambda_0 + \E_{\hat \P}[\phi^*(\ell(Z) - \lambda_0)] = \sup_{\P\in\cP(\cZ)} \E_\P[\ell(Z)] - \D_\phi(\P, \hat \P)
	\end{align}
	and thus coincides with the optimal value of a penalty-based distributionally robust optimization model with a $\phi$-divergence penalty. The equality in the above expression follows from \citet[Theorem~4.2]{ben2007old}, which is reminiscent of the strong duality theorem for worst-case expectation problems over restricted $\phi$-divergence ambiguity sets (see Theorem~\ref{thm:duality:restricted:phi}). The assumption that $\phi^\infty(1)=\infty$ ensures indeed that~$\D_\phi(\P, \hat \P)$ is finite only if~$\P \ll \hat \P$. We also remark that if~$g$ is a disutility function in the sense of Definition~\ref{def:generating} and if~$g$ is non-decreasing, then $g^*$ constitutes an entropy function in the sense of Definition~\ref{def:phi}.

	We will see below that the optimized certainty equivalents encapsulate several widely used risk measures as special cases. Notable examples include the mean-variance risk measure, the mean-median risk measure, the conditional value-at-risk or the entropic risk measure. More generally, \citet{rockafellar2006generalized,rockafellar2008risk} show that virtually any regular risk measure admits a representation of the form~\eqref{eq:regular} provided that the expected disutility is replaced with a more general measure of regret; see also \citep{rockafellar2014random,rockafellar2015measures} and the survey papers \citep{rockafellar2013superquantiles,royset2022risk}.

	\begin{definition}[Mean-Variance Risk Measure]
		The mean-variance risk measure with risk-aversion coefficient $\beta \in (0, \infty)$ is the law-invariant risk measure~$\varrho$ with 
		\[
		\varrho_\P[\ell(Z)] = \E_\P[\ell(Z)] + \beta \cdot \V_\P [\ell(Z)],
		\]
		where $\V_\P [\ell(Z)]$ denotes the variance of $\ell(Z)$ under~$\P$.
	\end{definition}
	
	We call a function $f:\R\to\overline\R$ coercive if $\lim_{i\to\infty} f(\tau_i)=\infty$ for every sequence $\{\tau_i\}_{i\in\N}$ with $\lim_{i\to\infty}|\tau_i|=\infty$. Coercivity will play a key role in re-expressing worst-case optimized certainty equivalents in terms of worst-case expectations.

	\begin{proposition}[Mean-Variance Risk Measure]
		\label{prop:mean-variance-risk}
		The mean-variance risk measure~$\varrho$ with risk-aversion coefficient~$\beta\in(0,\infty)$ is the optimized certainty equivalent induced by the disutility function $g(\tau) = \tau + \beta \tau^2$. The objective function of problem~\eqref{eq:regular} is coercive in~$\tau$ and is uniquely minimized by $\tau^\star = \E_\P[\ell(Z)]$.
	\end{proposition}
	\begin{proof}
		The objective function of problem~\eqref{eq:regular} corresponding to the disutility function~$g$ is given by
		$\E_\P[\ell(Z) + \beta(\ell(Z)-\tau)^2]$. This function is ostensibly coercive in~$\tau$ and is minimized by $\tau^\star = \E_\P[\ell(Z)]$. Substituting~$\tau^\star$ back into the objective function shows that the optimized certainty equivalent induced by~$g$ coincides indeed with the mean-variance risk measure with risk-aversion coefficient~$\beta$.
	\end{proof}
	
	\begin{definition}[Mean-MAD Risk Measure]
		The mean-median absolute deviation (MAD) risk measure with risk-aversion coefficient $ \beta \in (0, \infty)$ is the law-invariant risk measure~$\varrho$ with 
		\[
		\varrho_\P[\ell(Z)] = \E_\P[\ell(Z)] + \beta \cdot \E_\P \big[ | \ell(Z) - \M_\P [\ell(Z)] | \big],
		\]
		where $\M_\P [\ell(Z)]$ denotes the median of $\ell(Z)$ under~$\P$.
	\end{definition}
	
	\begin{proposition}[Mean-MAD Risk Measure]
		\label{prop:mean-median-risk}
		The mean-MAD risk measure~$\varrho$ with risk-aversion coefficient~$\beta\in(0,\infty)$ is the optimized certainty equivalent induced by the disutility function $g(\tau) = \tau + \beta |\tau|$. The objective function of problem~\eqref{eq:regular} is coercive in~$\tau$ and is minimized by $\tau^\star = \M_\P [\ell(Z)]$.
	\end{proposition}
	\begin{proof}
		The objective function of problem~\eqref{eq:regular} corresponding to the disutility function~$g$ is given by
		$\E_\P[\ell(Z) + \beta |\ell(Z)-\tau|]$. This function is ostensibly coercive in~$\tau$ and is minimized by $\tau^\star = \M_\P[\ell(Z)]$. Substituting~$\tau^\star$ back into the objective function yields the mean-MAD risk measure with risk-aversion coefficient~$\beta$.
	\end{proof}
	
	\begin{definition}[Conditional Value-at-Risk]
		\label{def:cvar}
		The conditional VaR (CVaR) at level $\beta\in(0,1)$ is the law-invariant risk measure denoted as~$\beta\CVaR$ with
		\begin{align}
			\label{eq:CVaR}
			\beta\CVaR_{\P} [\ell(Z)] = \inf_{\tau\in\R} \tau+\frac{1}{\beta} \E_\P\left[ \max \left\{\ell(Z)-\tau,0 \right\} \right].
		\end{align}
	\end{definition}
	
	Note that $\beta\CVaR_{\P} [\ell(Z)]$ converges to~$\E_\P[\ell(Z)]$ as~$\beta$ tends to~$1$. One can further show that it converges to the essential supremum $\esssup_\P[\ell(Z)]$ as $\beta$ tends to~$0$.
	
	\begin{proposition}[CVaR]
		\label{prop:CVaR}
		The CVaR at level~$\beta\in(0,1)$ is the optimized certainty equivalent induced by the disutility function $g(\tau) = \beta^{-1} \max\{\tau,0\}$. The objective function of problem~\eqref{eq:regular} is coercive in~$\tau$ and is minimized by $\tau^\star = \beta\VaR_\P [\ell(Z)]$.
	\end{proposition}
	\begin{proof}
		It is evident that problem~\eqref{eq:CVaR} is an instance of problem~\eqref{eq:regular} corresponding to the given disutility function~$g$. In addition, as~$\beta\in(0,1)$, it is evident that the objective function of problem~\eqref{eq:CVaR} is coercive in~$\tau$. Finally, one readily verifies that $\tau^\star = \beta\VaR_\P [\ell(Z)]$ solves the first-order optimality condition of the unconstrained convex program~\eqref{eq:CVaR} and thus constitutes a minimizer. 
	\end{proof}
	
	By substituting $\tau^\star = \beta\VaR_\P [\ell(Z)]$ into the objective function of problem~\eqref{eq:CVaR}, it becomes now clear that $\beta\CVaR_{\P}[\ell(Z)] \geq \beta\VaR_{\P}[\ell(Z)]$. If the loss~$\ell(Z)$ has a continuous distribution under~$\P$, then one can further use~\eqref{eq:CVaR} to show that 
	\[
	\beta\CVaR_{\P}[\ell(Z)] = \E_\P\left[\ell(Z)\left|\ell(Z)\geq \beta\VaR_{\P}[\ell(Z)]\right.\right].
	\]
	Hence, the CVaR at level~$\beta$ coincides with the expectation of the upper $\beta$-tail of the loss distribution, which implies that $\beta\CVaR_{\P}[\ell(Z)]$ is generically {\em strictly} larger than $\beta\VaR_{\P}[\ell(Z)]$. For details we refer to \citep{rockafellar2000optimization, rockafellar2002cvar-general}. 
	
	\begin{definition}[Entropic Risk Measure]
		The entropic risk measure with risk-aver\-sion parameter $\beta \in (0, \infty)$ is the law-invariant risk measure denoted as~$\beta\ERM$~with
		\begin{align}
			\label{eq:entropic}
			\beta\ERM_{\P} [\ell(Z)] = \frac{1}{\beta} \log \E_\P\left[ \exp\big(\beta \ell(Z)\big) \right] .
		\end{align}
	\end{definition}
	
	Using a Taylor expansion, one can show that $\beta\ERM_{\P} [\ell(Z)]$ converges to the expected value~$\E_\P[\ell(Z)]$ as~$\beta$ tends to~$0$. Similarly, one can show that $\beta\ERM_{\P} [\ell(Z)]$ converges to the essential supremum $\esssup_\P[\ell(Z)]$ as $\beta$ tends to~$\infty$.

	\begin{proposition}[Entropic Risk Measure]
		\label{prop:ERM}
		The entropic risk measure with risk-aversion parameter~$\beta\in(0,1)$ is the optimized certainty equivalent induced by the disutility function $g(\tau)=\beta^{-1}( \exp(\beta\tau)-1)$. The objective function of problem~\eqref{eq:regular} is coercive in~$\tau$ and is minimized by $\tau^\star = \beta^{-1}\log(\E_\P[\exp(\beta \ell(Z))])$.
	\end{proposition}
	
	\begin{proof}
		By the definition of~$g$, we have
		\begin{align*}
			\inf_{\tau\in\R} \tau+\E_\P\left[ g(\ell(Z)-\tau) \right] & = \inf_{\tau\in\R} \tau + \frac{1}{\beta} \E_\P\left[ \exp\big(\beta (\ell(Z) - \tau)\big) - 1 \right] \\
			& = \frac{1}{\beta} \log \E_\P\left[ \exp\big(\beta \ell(Z)\big) \right] = \beta\ERM_{\P} [\ell(Z)].
		\end{align*}
		The second equality holds because the unconstrained convex minimization problem over~$\tau$ is uniquely solved by $\tau^\star = \beta^{-1}\log(\E_\P[\exp(\beta \ell(Z))])$, which can be verified by inspecting the problem's first-order optimality condition. In addition, as~$\beta\in(0,1)$, it is clear that the problem's objective function is coercive in~$\tau$.
	\end{proof}

	\citet{kupper2009representation} show that, with the exception of the expected value, the entropic risk measure is the only relevant law-invariant risk measure that obeys the tower property. That is, for any random vectors~$Z_1$ and~$Z_2$ it satisfies
	\begin{align*}
		\beta\ERM_{\P} [\ell(Z_2)] = \beta\ERM_{\P}[\beta\ERM_{\P} [\ell(Z_2)|Z_1]],
	\end{align*}
	where the {\em conditional} entropic risk measure $\beta\ERM_{\P} [Z_1|Z_2]$ is defined in the obvious way by replacing the unconditional expectation in~\eqref{eq:entropic} with a conditional expectation. The entropic risk measure is often used for modeling risk-aversion in {\em dynamic} optimization problems, where the dynamic consistency of the decisions taken at different points in time is a concern. For example, it occupies center stage in finance \citep{follmer2008stochastic}, risk-sensitive control \citep{whittle1990risk,bacsar1995h} and economics \citep{hansen2008robustness}.

	\begin{proposition}[Dual Representation of the Entropic Risk Measure]
		\label{prop:dual-erm}
		Assume that $\E_{\hat\P}[\ell(Z)]>-\infty$. Then, the entropic risk measure admits the dual representation
		\begin{align*}
			\beta\ERM_{\hat \P} [\ell(Z)] = \sup_{\P \in \cP(\cZ)} \E_\P [\ell(Z)] - \frac{1}{\beta} \cdot \KL(\P, \hat \P).
		\end{align*}
	\end{proposition}
	
	\begin{proof}
		Let~$\phi$ be the entropy function of the Kullback-Leibler divergence. Thus, we have $\phi^*(t)=\mathrm{e}^t-1$ for all $t\in\R$; see Table~\ref{tab:phi-divergences-semiinfinite-reformulation}. By  Proposition~\ref{prop:ERM}, the entropic value-at-risk is the optimized certainty equivalent induced by the disutility function
		\[
		g(t)=\beta^{-1}( \exp(\beta t)-1) = \beta^{-1}\phi^*(\beta t)= (\beta^{-1}\phi)^*(t),
		\]
		where the last equality uses \citep[Theorem~16.1]{rockafellar1970convex}. This implies that
		\begin{align*}
			\beta\ERM_{\hat \P} [\ell(Z)] & = \inf_{\tau\in\R} \tau+\E_\P\left[ (\beta^{-1}\phi)^*(\ell(Z)-\tau) \right] \\
			& = \sup_{\P\in\cP(\cZ)} \E_\P[\ell(Z)] - \D_{\beta^{-1}\phi}(\P, \hat \P) \\
			& = \sup_{\P \in \cP(\cZ)} \E_\P [\ell(Z)] - \beta^{-1} \KL(\P, \hat \P).
		\end{align*}
		Here, the second equality follows from the strong duality relation~\eqref{eq:oce-vs-phi}, which applies because $\E_{\hat\P}[\ell(Z)]>-\infty$, and the third equality holds because the entropy function~$\phi$ was assumed to induce the Kullback-Leibler divergence. 
	\end{proof}
	
	We remark that Proposition~\ref{prop:dual-erm} can also be proved by leveraging the Donsker-Varadhan formula from Proposition~\ref{prop:donsker:KL} in lieu of the duality relation~\eqref{eq:oce-vs-phi}.


One can show that every optimized certainty equivalent~$\varrho$ is translation-invariant and convex. If the underlying disutility function~$g$ is non-decreasing, then~$\varrho$ is also monotone, and if~$g$ is positive homogeneous, then~$\varrho$ is also scale-invariant.

In the remainder we will show that if~$\varrho$ is any optimized certainty equivalent, then the worst-case risk problem~\eqref{eq:worst-case:risk} can be reduced a worst-case expectation problem of the form~\eqref{eq:worst-case:expectation}. This reduction is predicated on a lopsided minimax theorem to be derived below, and it allows us to extend the duality theory for worst-case expectation problems of Section~\ref{sec:duality-wc-expectation} to a rich class of worst-case risk problems.

\subsection{Lopsided Minimax Theorems}
A generic minimax problem can be represented as
\begin{align*}
	\inf_{u \in \cU} \, \sup_{v \in \cV} \, H(u, v),
\end{align*}
where~$\cU$ and~$\cV$ are arbitrary spaces, and $H:\cU\times\cV\to\overline \R$ is an arbitrary function. A minimax theorem provides conditions under which the infimum and supremum operators can be interchanged without changing the problem's optimal value. The following minimax theorem inspired by \citep[Example~13]{rockafellar1974conjugate} will be essential for solving worst-case risk problems with optimized certainty equivalents. Recall from Section~\ref{sec:duality-proof-strategy} that a convex function is closed if it is either proper and lower semicontinuous or identically equal to~$+\infty$ or to $-\infty$.

\begin{theorem}[Lopsided Minimax Theorem]
	\label{thm:minimax}
	Suppose that~$\cU$ is an arbitrary vector space and~$\cV$ is a locally convex topological vector space. Assume also that the function $H: \cU \times \cV \to \overline\R$ is such that $H(u,v)$ is convex in~$u$ and such that $-H(u,v)$ is convex and closed in~$v$. If $\sup_{v\in\cV}\inf_{u\in\cU} H(u,v)>-\infty$ and for every $\alpha\in\R$ there exists~$u \in \cU$ such that $\{ v \in \cV : H(u, v) \geq \alpha \}$ is compact, then we have
	\begin{align*}
		\inf_{u \in \cU} \, \sup_{v \in \cV} \, H(u,v) = \sup_{v \in \cV} \, \inf_{u \in \cU} \, H(u,v). 
	\end{align*}
\end{theorem}
\begin{proof}
	Let $\cV^*$ be the topological dual of~$\cV$, and define the bilinear form $\inner{\cdot}{\cdot}:\cV^*\times\cV\to\R$ through $\inner{v^*}{v}=v^*(v)$. If we equip~$\cV^*$ with the weak topology induced by~$\cV$, then $\inner{\cdot}{v}$ is a continuous linear functional on~$\cV^*$ for every~$v\in\cV$, and every continuous linear functional on~$\cV^*$ can be represented in this way.
	
	Define $F:\cU\times\cV^* \to\overline\R$ through $F(u,v^*)=\sup_{v\in\cV} H(u,v)-\inner{v^*}{v}$, which is jointly convex in~$u$ and~$v^*$ thanks to Lemma~\ref{lem:param:cvx}. Thus, $F(u,v^*)=(-H)^*(u,-v^*)$, where the conjugate of $-H(u,v)$ is evaluated with respect to its second argument~$v$ only. As $-H(u,v)$ is convex and closed in~$v$, this implies via Lemma~\ref{lem:bi:coincidence} that $F^{*}(u,v)=-H(u,-v)$. Here, again, the conjugate of~$F(u,v^*)$ is evaluated with respect to its second argument~$v^*$ only. In addition, define $h:\cV^*\to\overline \R$ through $h(v^*)= \inf_{u\in\cU} F(u,v^*)$, which is convex in~$v^*$. Thus, we find
	\[
	h(0)= \inf_{u\in\cU} F(u,0)= \inf_{u\in\cU} \sup_{v\in\cV} H(u,v),
	\]
	where the two equalities follow from the definitions of~$h$ and~$F$, respectively. In addition, we also have
	\begin{align*}
		h^{**}(0) &= \sup_{v\in\cV} -h^*(-v)= \sup_{v\in\cV} \inf_{v^*\in\cV^*} \inner{v^*}{v}+h(v^*) \\
		& = \sup_{v\in\cV} \inf_{u\in\cU} \inf_{v^*\in\cV^*} \inner{v^*}{v}+F(u,v^*) \\
		&= \sup_{v\in\cV} \inf_{u\in\cU} -F^{*}(u,-v) =\sup_{v\in\cV} \inf_{u\in\cU} H(u,v),
	\end{align*}
	where the first two equalities follow from the definitions of the bi-conjugate~$h^{**}$ and the conjugate~$h^*$, respectively, and the third equality exploits the definition of~$h$. The fourth equality follows from the definition of the conjugate~$F^*$, and the last equality holds because $F^{*}(u,v)=-H(u,-v)$. Thus, the desired minimax result holds if we manage to prove that $h(0)=h^{**}(0)$.
	
	
	By the definitions of~$h^*$ and~$h$ and by the relation~$F^*(u,v)=-H(u,-v)$, we have
	\begin{align*}
		\{v\in\cV: h^*(v)\leq \alpha\} & = \left\{v\in\cV: \sup_{v^*\in\cV^*} \inner{v^*}{v}- h(v^*)\leq \alpha \right\} \\
		& = \left\{v\in\cV: \sup_{u\in\cU} \sup_{v^*\in\cV^*} \inner{v^*}{v}- F(u,v^*)\leq \alpha \right\} \\
		& = \left\{v\in\cV: \sup_{u\in\cU} -H(u,-v)\leq \alpha \right\} \\ &= - \bigcap_{u\in\cU} \{v\in\cV: H(u,v)\geq -\alpha \}
	\end{align*}
	for any~$\alpha\in\R$. Hence, $\{v\in\cV: h^*(v)\leq \alpha\}$ is representable as an intersection of closed sets, at least one of which is compact. Therefore, the intersection is also compact. Selecting $\alpha > \inf_{v\in\cV} h^*(v)$, which is possible because $\sup_{v\in\cV}\inf_{u\in\cU} H(u,v)>-\infty$ implies that $\inf_{v\in\cV}h^*(v)<\infty$, we further ensure that the compact set $\{v\in\cV: h^*(v)\leq \alpha\}$ is non-empty. This implies via \cite[Theorem~10\,(b)]{rockafellar1974conjugate} that~$h^{**}(v^*)$ and~$h(v^*)$ are both bounded above on a neighborhood of~$0$. By \citep[Theorem~17\,(a)]{rockafellar1974conjugate}, this in turn implies that $h(0)=h^{**}(0)$, which establishes the desired minimax equality.
\end{proof}

Swapping the roles of~$u$ and~$v$ leads to the following immediate corollary.

\begin{corollary}[Reverse Lopsided Minimax Theorem]
	\label{cor:minimax}
	Suppose that~$\cU$ is a locally convex topological vector space and~$\cV$ is an arbitrary vector space. Assume also that the function $H: \cU \times \cV \to \overline\R$ is such that $H(u,v)$ is convex and closed in~$u$ and such that $-H(u,v)$ is convex in~$v$. If $\inf_{u\in\cU} \sup_{v\in\cV} H(u,v)<\infty$ and for every $\alpha\in\R$ there exists $v \in \cV$ such that $\{ u \in \cU : H(u, v) \leq \alpha \}$ is compact, then we have
	\begin{align*}
		\inf_{u \in \cU} \, \sup_{v \in \cV} \, H(u,v) = \sup_{v \in \cV} \, \inf_{u \in \cU} \, H(u,v). 
	\end{align*}
\end{corollary}

A function $h_v(u)=H(u,v)$ whose sublevel sets $\{ u \in \cU : h_v(u) \leq \alpha \}$ are all compact is commonly referred to as {\em inf-compact} \citep{hartung1982extension}. The following lemma provides an easily checkable sufficient condition for the inf-compactness of $h_v(u)$ in case~$\cU$ is a Euclidean space. To this end, recall that a function $h_v$ is {\em coercive} if for every sequence $\{ u_i \}_{i\in \N}$ with $\lim_{i \to \infty} \| u_i \|_2 = \infty$, we have $\lim_{i \to \infty} h_v(u_i) = \infty$. 

\begin{lemma}[Inf-Compactness]
\label{lem:infcomp}
Suppose that $\cU$ is a Euclidean space and $H:\cU \times \cV \to \overline \R$ is lower semicontinuous and coercive in its first argument. Then, the sublevel sets $\{ u \in \cU : H(u, v) \leq \alpha \}$ are compact for all~$v\in\cV$ and~$\alpha \in \R$.
\end{lemma}

\begin{proof}
To show that the sublevel set $\cU_\alpha(v)= \{ u \in \cU : H(u, v) \leq \alpha \}$ is compact, note first that~$\cU_\alpha(v)$ is closed because $H(u,v)$ is lower semicontinuous in~$u$. In order to prove that~$\cU_\alpha(v)$ is also bounded, assume for the sake of contradiction that there exists a sequence $\{ u_i \}_{i \in \N} \in \cU_\alpha(v)$ with $\lim_{i \to \infty} \| u_i \| = \infty$. As $H(u, v)$ is coercive in~$u$, we have $\lim_{i \to \infty} H(u_i, v) = \infty$. However, this contradicts the assumption that $H(u_i,v)\leq \alpha$ for all~$i\in\N$. Thus, $\cU_\alpha(v)$ must be bounded and compact.
\end{proof}

Note that if $H_0:\cU_0\times\cV_0\to\overline\R$ is defined on convex sets~$\cU_0\subseteq\cU$ and~$\cV_0\subseteq\cV$, then it can be extended to a function $H:\cU\times\cV\to\overline\R$ on the underlying vector spaces~$\cU$ and~$\cV$ by setting
\[
H(u,v) =\left\{ \begin{array}{ll}
H_0(u,v) & \text{if $u\in\cU_0$ and $v\in\cV_0$,} \\
+\infty &  \text{if $u\not\in\cU_0$ and $v\in\cV_0$,} \\
-\infty & \text{if $v\not\in\cV_0$.} \\
\end{array} \right.
\]
This construction guarantees that 
\[
\inf_{u\in\cU} \sup_{v\in\cV} H(u,v) = \inf_{u\in\cU_0} \sup_{v\in\cV_0} H_0(u,v) \quad \text{and}\quad \sup_{v\in\cV} \inf_{u\in\cU} H(u,v) = \sup_{v\in\cV_0} \inf_{u\in\cU_0} H_0(u,v).
\]
It also guarantees that if $H_0(u,v)$ is convex and closed in~$u$ and concave in~$v$, then so is~$H(u,v)$. Thus, the feasible sets in any convex-concave minimax problem can always be extended to the underlying vector spaces without changing the problem.

We now leverage Corollary~\ref{cor:minimax} to derive a minimax theorem for optimized certainty equivalents. This result exploits the inf-compactness of the objective function of problem~\eqref{eq:regular} in~$\tau$. \citet{shafiee2024general} establish similar minimax theorems for a more general class of regular risk and deviation measures introduced by \citet{rockafellar2013fundamental}.

\begin{theorem}[Minimax Theorem for Optimized Certainty Equivalents]
\label{thm:duality:regular}
Suppose that $\cP \subseteq \cP(\cZ)$ is non-empty and convex, $\varrho$ is any optimized certainty equivalent induced by a disutility function~$g$, $\sup_{\P\in\cP}\E_\P[g(\ell(Z))]<\infty$, and $\E_\P[\ell(Z)]>-\infty$ for all~$\P\in\cP$. Then, $G(\tau,\P)=\tau + \E_\P[ g(\ell(Z) - \tau)]$ for~$\tau\in\R$ and~$\P\in\cP$ satisfies
\begin{align*}
	\sup_{\P \in \cP}  \varrho_{\P} [\ell(Z)] 
	= \sup_{\P \in \cP}  \inf_{\tau \in \R} G(\tau,\P)
	= \inf_{\tau \in \R} \sup_{\P \in \cP}  G(\tau,\P).
\end{align*}
\end{theorem}

\begin{proof}
Note first that $G(\tau,\P)$ is convex in~$\tau$ and concave (in fact, linear) in~$\P$. In addition, $G(\tau,\P)$ is closed in~$\tau$. To see this, observe that
\begin{align*}
	\liminf_{\tau'\to\tau} G(\tau',\P)& = \liminf_{\tau'\to\tau} \E_\P[\tau' +  g(\ell(Z) - \tau')] \\ 
	& \geq \E_\P[\liminf_{\tau'\to\tau} \tau' +  g(\ell(Z) - \tau')] \\
	& \geq \E_\P[\tau +  g(\ell(Z) - \tau)] =G(\tau,\P),
\end{align*}
where the two inequalities follow from Fatou's lemma and the continuity of~$g$, respectively. Fatou's lemma applies because any disutility function satisfies $g(\tau)\geq\tau$ for all $\tau\in\R$, which implies that $\tau+ g(\ell(z)-\tau)\geq \ell(z)$ for all~$z\in\cZ$ and $\tau\in \R$. Note also that $\E_\P[\ell(Z)]$ is finite by assumption. Next, we show that $G(\tau,\P)$ is inf-compact in~$\tau$. To this end, recall that $g(0)=0$ and $g(\tau)> \tau$ for all~$\tau\neq 0$. As~$g$ is also convex, this implies that $g(\tau)$ must grow faster than~$\tau$ as~$\tau$ tends to~$+\infty$ and that~$g(\tau)$ must decay slower than~$\tau$ as~$\tau$ tends to~$-\infty$. Hence, there exists $\varepsilon>0$ with $g(\tau)\geq (1+\varepsilon)\tau-1$ and $g(\tau)\geq (1-\varepsilon)\tau-1$ for all~$\tau\in\R$. For a formal proof of this assertion we refer to \citep[Lemma~C.10]{zhen2023unification}. This implies that
\begin{align*}
	G(\tau,\P)\geq \tau+(1+\varepsilon)\left( \E_\P[\ell(Z)]-\tau\right) -1 = -\varepsilon\tau + (1+\varepsilon)  \E_\P[\ell(Z)] -1
\end{align*}
and 
\begin{align*}
	G(\tau,\P)\geq \tau+(1-\varepsilon)\left( \E_\P[\ell(Z)]-\tau\right) -1 = \varepsilon\tau + (1+\varepsilon)  \E_\P[\ell(Z)] -1
\end{align*}
for all~$\tau\in\R$, and thus $\{\tau\in\R:G(\tau,\P)\leq \alpha\}$ is compact for every~$\alpha\in\R$.

Next, set~$\cU=\R$, and define~$\cV=\cM(\R^d)$ as the space of all finite signed Borel measures on~$\R^d$. In addition, define the function $H:\cU\times\cV\to\overline\R$ through
\[
H(u,v) = \left\{\begin{array}{ll} 
	G(u,v) & \text{if } v\in\cP,\\
	-\infty & \text{if } v\not\in\cP.
\end{array}\right.
\]
By construction, $H(u,v)$ is convex and closed in~$u$ and concave in~$v$. Recall from Section~\ref{sec:duality-proof-strategy} that a convex function is closed if it is either proper and lower semi\-continuous or identically equal to~$-\infty$ or to ~$+\infty$. In addition, we have
\[
\sup_{v\in\cV} H(0,v)=\sup_{\P\in\cP} G(0,\P)=\sup_{\P\in\cP} \E_\P[g(\ell(Z))]<\infty.
\]
and the sublevel sets $\{u\in\cU:H(u,v)\leq\alpha\}$ are compact for every~$\alpha\in\R$ provided that $v\in\cP$. The claim thus follows from Corollary~\ref{cor:minimax}.
\end{proof}

Theorem~\ref{thm:duality:regular} implies that if~$\beta\in(0,1)$, then the worst-case $\beta$-CVaR satisfies
\begin{align}
\label{eq:worst-case-cvar}
\sup_{\P\in\cP} \beta\CVaR_\P[\ell(Z)] =\inf_{\tau\in\R} \tau+\frac{1}{\beta} \sup_{\P\in\cP}\E_\P\left[ \max\{\ell(Z)-\tau,0\} \right] 
\end{align}
for {\em any} non-empty convex ambiguity set~$\cP\subseteq\cP(\cZ)$ provided that $\E_\P[|\ell(Z)|]<\infty$ for all~$\P\in\cP$. In the extant literature, the interchange of the supremum over~$\P$ and the infimum over~$\tau$ is often justified with Sion's minimax theorem \citep{sion1958general}. However, many studies overlook that Sion's minimax theorem only applies if~$\cP$ is weakly compact and $\E_\P[ \max\{\ell(Z)-\tau,0\}]$ is weakly upper semi\-continuous in~$\P$. As shown in Section~\ref{sec:topology}, unfortunately, many popular ambiguity sets fail to be weakly compact. In addition, $\E_\P[ \max\{\ell(Z)-\tau,0\}]$ fails to be weakly upper semi\-continuous unless the loss function~$\ell$ is upper semi\-continuous and bounded on~$\cZ$; see Proposition~\ref{prop:semicontinuity}. All non-trivial convex loss functions on~$\R^d$ violate this condition.  In contrast, Theorem~\ref{thm:duality:regular} offers a more general result that exploits the inf-compactness in~$\tau$ but obviates any restrictive topological conditions on~$\cP$ or~$\ell$.

\subsection{Moment Ambiguity Sets}

Recall that the generic moment ambiguity set~\eqref{eq:moment-ambiguity-set} is defined as
\begin{align*}
\cP = \left\{ \P \in \cP_f(\cZ) \, : \, \E_\P \left[ f (Z) \right] \in \cF \right\},
\end{align*}
where $\cZ \subseteq \R^d$ is a non-empty closed support set, $f: \cZ \to \R^m$ is a Borel measurable moment function, $\cF \subseteq \R^m$ is a non-empty closed moment uncertainty set, and $\cP_f(\cZ)$ denots the family of all distributions $\P\in\cP(\cZ)$ for which $\E_\P[f(Z)]$ is finite. Recall also that $\cC = \left\{ \E_\P[f(Z)] : \P \in \cP_f(\cZ) \right\}$ represents the family of all possible moments of any distribution on~$\cZ$. The next theorem establishes a duality result for the worst-case risk problem~\eqref{eq:worst-case:risk} with a moment ambiguity set. 

\begin{theorem}[Duality Theory for Moment Ambiguity Sets II]
\label{thm:duality:risk:moment}
If~$\cP$ is the moment ambiguity set~\eqref{eq:moment-ambiguity-set} and $\varrho$ is an optimized certainty equivalent induced by a disutility function $g$, then the following weak duality relation holds.  
\begin{align}
	\label{eq:weak-duality-moments-risk}
	\sup_{\P \in \cP} ~ \varrho_\P \left[ \ell(Z) \right] 
	\leq \left\{
	\begin{array}{cl}
		\inf & \tau + \lambda_0 + \delta_\cF^*(\lambda) \\[1ex]
		\st & \tau, \lambda_0 \in \R, \, \lambda \in \R^m \\ [1ex]
		& \lambda_0 + f(z)^\top \lambda \geq g(\ell(z) - \tau) ~~ \forall z \in \cZ
	\end{array}
	\right.
\end{align}
If $\sup_{\P\in\cP}\E_\P[g(\ell(Z))]<\infty$, $\E_\P[\ell(Z)]>-\infty$ for all~$\P\in\cP_f(\cZ)$, and~$\cF \subseteq \cC$ is a convex and compact set with $\rint(\cF)\subseteq \rint(\cC)$, then strong duality holds, that is, the inequality~\eqref{eq:weak-duality-moments-risk} becomes an equality.
\end{theorem}

\begin{proof}
The max-min inequality implies that
\begin{align*}
	\sup_{\P \in \cP} \varrho_\P \left[ \ell(Z) \right] & = \sup_{\P \in \cP} \inf_{\tau\in\R} \tau+\E_\P\left[ g(\ell(Z)-\tau) \right] \\
	& \leq \inf_{\tau\in\R} \sup_{\P \in \cP} \tau+\E_\P\left[ g(\ell(Z)-\tau) \right].
\end{align*}
The inner maximization problem in the resulting upper bound constitutes a worst-case expectation problem. Hence, it is bounded above by the dual problem derived in Theorem~\ref{thm:duality:moment}. Substituting this dual problem into the above expression yields~\eqref{eq:weak-duality-moments-risk}. Strong duality follows from the minimax theorem for optimized certainty equivalents (Theorem~\ref{thm:duality:regular}) and the strong duality result for worst-case expectation problems (Theorem~\ref{thm:duality:moment}), which apply under the given assumptions.
\end{proof}

The semi-infinite constraint in~\eqref{eq:weak-duality-moments-risk} involves the composite function~$g(\ell(z)-\tau)$, which fails to be concave in~$z$ even if~$g$ is non-decreasing and~$\ell$ is concave. Thus, checking whether a given~$(\tau,\lambda_0,\lambda)$ satisfies the semi-infinite constraint in~\eqref{eq:weak-duality-moments-risk} is generically hard. In fact, \citet[Theorem~1]{chen2024robust} prove that evaluating the worst-case entropic risk is NP-hard even if~$\ell$ is linear and~$\cP$ is a Markov ambiguity set. Hence, while providing theoretical insights, Theorem~\ref{thm:duality:regular} does not necessarily pave the way towards an efficient method for solving worst-case risk problems of the form~\eqref{eq:worst-case:risk}. Nevertheless, Theorem~\ref{thm:duality:regular} provides a concise reformulation for~\eqref{eq:worst-case:risk} that is susceptible to approximate iterative solution procedures.

\subsection{$\phi$-Divergence Ambiguity Sets}
Recall that the $\phi$-divergence ambiguity set~\eqref{eq:phi-divergence-ambiguity-set} is defined~as
\begin{align*}
\cP = \left\{ \P \in \cP(\cZ) \, : \, \D_\phi(\P, \hat \P) \leq r \right\},
\end{align*}
where $\cZ$ is a closed support set, $r\geq 0$ is a size parameter, $\phi$ is an entropy function in the sense of Definition~\ref{def:phi}, $\D_\phi$ is the corresponding $\phi$-divergence in the sense of Definition~\ref{def:D_phi}, and~$\hat\P\in\cP(\cZ)$ is a reference distribution. The next theorem establishes a duality result for worst-case risk problems over $\phi$-divergence ambiguity sets. The proof follows from Theorems~\ref{thm:duality:phi} and~\ref{thm:duality:regular} and is thus omitted.

\begin{theorem}[Duality Theory for $\phi$-Divergence Ambiguity Sets II]
\label{thm:duality:risk:phi}
Assume that $\E_{\hat \P}[\ell(Z)]>-\infty$. If~$\cP$ is the $\phi$-divergence ambiguity set~\eqref{eq:phi-divergence-ambiguity-set}, and $\varrho$ is an optimized certainty equivalent induced by a disutility function~$g$, then the following weak duality relation holds.
\begin{equation*}
	\begin{aligned}
		\sup_{\P \in \cP} \; \varrho_\P[\ell(Z)] \leq \left\{ \begin{array}{cl} \displaystyle \inf_{\tau, \lambda_0\in\R, \lambda \in \R_+} & \displaystyle \tau + \lambda_0  + \lambda r + \E_{\hat \P} \left[ (\phi^*)^\pi \left( g(\ell(Z) - \tau) - \lambda_0, \lambda \right) \right] \\
			\st & \displaystyle  \lambda_0+ \lambda\,\phi^\infty(1) \geq \sup_{z\in\cZ} g(\ell(z)-\tau)
		\end{array}\right.
	\end{aligned}
\end{equation*}
If $\sup_{\P\in\cP}\E_\P[g(\ell(Z))]<\infty$, $\E_\P[\ell(Z)]>-\infty$ for all $\P\in\cP$, $r>0$ and~$\phi$ is con\-tinuous at~$1$, then strong duality holds, that is, the inequality becomes an equality.
\end{theorem}

A duality result akin to Theorem~\ref{thm:duality:risk:phi} also holds for worst-case risk problems over {\em restricted} $\phi$-divergence ambiguity sets of the form
\begin{align*}
\cP = \left\{ \P \in \cP(\cZ) \, : \, \P \ll \hat \P, \; \D_\phi(\P, \hat \P) \leq r \right\}.
\end{align*}
The proof of the next theorem follows immediately from Theorems~\ref{thm:duality:restricted:phi} and~\ref{thm:duality:regular}.

\begin{theorem}[Duality Theory for Restricted $\phi$-Divergence Ambiguity Sets II]
\label{thm:duality:risk:phi-restricted}
Assume that $\E_{\hat \P}[\ell(Z)]>-\infty$. If~$\cP$ is the restricted $\phi$-divergence ambiguity set~\eqref{eq:restricted-phi-divergence-ambiguity-set}, and $\varrho$ is an optimized certainty equivalent induced by a disutility function~$g$, then the following weak duality relation holds.
\begin{align*}
	\sup_{\P \in \cP} \; \varrho_\P[\ell(Z)] \leq \inf_{\tau, \lambda_0 \in \R, \, \lambda \in \R_+} ~ \tau + \lambda_0 + \lambda r + \E_{\hat \P} \left[ (\phi^*)^\pi \left( g(\ell(Z) - \tau) - \lambda_0, \lambda \right) \right].
\end{align*}
If $\sup_{\P\in\cP}\E_\P[g(\ell(Z))]<\infty$, $\E_\P[\ell(Z)]>-\infty$ for all $\P\in\cP$, $r>0$ and~$\phi$ is con\-tinuous at~$1$, then strong duality holds, that is, the inequality becomes an equality.
\end{theorem}

\subsection{Optimal Transport Ambiguity Sets}

Recall that the optimal transport ambiguity set~\eqref{eq:OT-ambiguity-set} is defined as
\begin{align*}
\cP = \left\{ \P \in \cP(\cZ) \, : \, \OT_c(\P, \hat \P) \leq r \right\},
\end{align*}
where $\cZ$ is a closed support set, $r\geq 0$ is a size parameter, $c$ is a transportation cost function in the sense of Definition~\ref{def:cost}, $\OT_c$ is the corresponding optimal transport discrepancy in the sense of Definition~\ref{def:OT}, and~$\hat\P\in\cP(\cZ)$ is a reference distribution. The next theorem establishes a duality result for worst-case risk problems over optimal transport ambiguity sets. Its proof follows immediately from Theorems~\ref{thm:duality:OT} and~\ref{thm:duality:regular} and is thus omitted.

\begin{theorem}[Duality Theory for Optimal Transport Ambiguity Sets II]
\label{thm:duality:risk:OT}
Assume that $\E_{\hat\P}[\ell(\hat Z)]>-\infty$ and $\ell$ is upper semicontinuous. If~$\cP$ is the optimal transport ambiguity set defined in~\eqref{eq:OT-ambiguity-set} and~$\varrho$ is an optimized certainty equivalent induced by a disutility function~$g$, then the following weak duality relation holds.
\begin{align*}
	\sup_{\P \in \cP} \; \varrho_\P[\ell(Z)]
	\leq \inf_{\tau \in \R, \, \lambda \in \R_+} \; \tau + \lambda r + \E_{\hat\P} \left[ \sup_{z \in \cZ} \; g(\ell(z) - \tau) - \lambda c(z,\hat Z) \right].
\end{align*}
If $\sup_{\P\in\cP}\E_\P[g(\ell(Z))]<\infty$, $\E_\P[\ell(Z)]>-\infty$ for all $\P\in\cP$ and~$r > 0$, then strong duality holds, that is, the inequality becomes an equality.
\end{theorem}

Worst-case risk problems with optimal transport ambiguity sets are studied by \citet{pflug2007ambiguity}, \citet{pichler:2013} and \citet{wozabal2014robustifying} in the context of portfolio selection with linear loss functions and by \citet{esfahani2018inverse} in the context of inverse optimization using the CVaR. \citet{sadana2024data} investigate worst-case entropic risk measures over $\infty$-Wasserstein balls and establish tractable reformulations under standard convexity assumptions. \citet{kent2021modified} and \citet{sheriff2023nonlinear} develop customized Frank-Wolfe algorithms in the space of probability distribution to address worst-case risk problems involving generic loss functions and risk measures. Specifically, \citet{kent2021modified} work with Wasserstein gradient flows and use the corresponding notions of smoothness to establish the convergence of their Frank-Wolfe algorithm. In contrast, \citet{sheriff2023nonlinear} work with G\^ateaux derivatives, which leads to a different notion of smoothness and thus to a different convergence analysis. Both algorithms display sublinear convergence rates. When the reference distribution~$\hat \P$ is discrete or when only samples from $\hat \P$ are used, the algorithms' iterates represent discrete distributions with progressively increasing bit sizes. Theorem~\ref{thm:duality:risk:OT} provides a compact, albeit potentially nonconvex, reformulation of the worst-case risk problem. This reformulation is amenable to primal-dual gradient methods in the finite-dimensional space of the dual variables, which are guaranteed to converge to a stationary point. 

Worst-case risk problems represent special instances of optimization problems over spaces of probability distributions. The mainstream methods to address such problems leverage the machinery of Wasserstein gradient flows \citep{ambrosio2008gradient}. Wasserstein gradient flows have recently been used in the context of distributionally robust optimization problems \citep{lanzetti2022first,lanzetti2024variational, xu2024flow}, nonconvex optimization \citep{chizat2018global,chizat2022sparse} or variational inference~\citep{jiang2024algorithms,lambert2022variational,diao2023forward,zhang2020theoretical}. The results of this section are new and complementary to these existing works.

\section{Analytical Solutions of Nature's Subproblem}
\label{sec:analytical-wc}
A key challenge in DRO is to handle the worst-case expectation problem embedded in~\eqref{eq:primal:dro}. This problem is solved by the fictitious adversary---commonly thought of as {\em nature}--- once the decision-maker has committed to an~$x\in\cX$. It maximizes a linear function over a convex subset of an infinite-dimensional space of measures and thus appears to be intractable. Therefore, considerable research effort has been devoted to identifying conditions under which this problem is efficiently solvable. We now show that it can actually be solved {\em analytically} in interesting situations.

The duality theory derived in Section~\ref{sec:duality-wc-expectation} motivates the following simple strategy for finding analytical solutions of nature's subproblem. Construct feasible solutions for the primal worst-case expectation problem and its dual, and show that their objective function values match. If such matching solutions can be found, then both of them must be optimal in their respective optimization problems thanks to weak duality. As we will see below, this simple strategy succeeds surprisingly often. In addition, we will see that analytical solutions for worst-case {\em expectation} problems can sometimes be generalized to analytical solutions for worst-case {\em risk} problems of the form~\eqref{eq:worst-case:risk}. The material reviewed in this section covers several decades of research in DRO from the 1950s until the present day.

\subsection{Jensen Bound}
\label{sec:jensen}
Consider the worst-case expectation problem
\begin{subequations}
\begin{align}
	\label{eq:primal-Markov}
	\sup_{\P \in \cP(\cZ)} \left\{ \E_\P \left[ \ell (Z) \right] \;:\; \E_\P \left[ Z \right] = \mu \right\},
\end{align}
which maximizes the expected value of $\ell(Z)$ over the Markov ambiguity set of all distributions supported on~$\cZ$ with mean~$\mu$. The Markov ambiguity set is a moment ambiguity set of the form~\eqref{eq:moment-ambiguity-set} with~$f(z)=z$ and~$\cF=\{\mu\}$. By Theorem~\ref{thm:duality:moment} and as the support function of $\cF$ is linear, the problem dual to~\eqref{eq:primal-Markov} is given by
\begin{align}
	\label{eq:dual-Markov}
	\inf_{\lambda_0\in\R,\, \lambda\in\R^d} \left\{ \lambda_0+\lambda^\top \mu \;:\; \lambda_0+\lambda^\top z\geq \ell(z) ~~ \forall z\in\cZ \right\}.
\end{align}
\end{subequations}
Intuitively, this dual problem aims to find an affine function~$a(z)=\lambda_0+\lambda^\top z$ that majorizes the loss function~$\ell(z)$ on~$\cZ$ and has minimal expected value~$\E_\P[a(Z)]$ under any distribution~$\P$ feasible in the primal problem~\eqref{eq:primal-Markov}.

\begin{proposition}[Jensen Bound]
\label{prop:jensen}
Suppose that~$\cZ$ is convex, $\mu\in\cZ$, $\ell$ is concave, and~$\lambda^\star$ is any supergradient of~$\ell$ at~$\mu$. Then, the primal problem~\eqref{eq:primal-Markov} is solved by~$\P^\star=\delta_\mu$, and the dual problem~\eqref{eq:dual-Markov} is solved by~$(\lambda_0^\star,\lambda^\star)$, where~$\lambda_0^\star=\ell(\mu)-\mu^\top \lambda^\star$. In addition, the optimal values of~\eqref{eq:primal-Markov} and~\eqref{eq:dual-Markov} both equal~$\ell(\mu)$.    
\end{proposition}

\begin{proof}
By construction, $\P^\star$ is feasible in the primal worst-case expectation problem, and its objective function value amounts to~$\ell(\mu)$. In addition, $(\lambda_0^\star,\lambda^\star)$ is feasible in the dual robust optimization problem because~$\lambda^\star$ is a supergradient of~$\ell$ at~$\mu$, and its objective function value amounts to~$\ell(\mu)$, too. Hence, by weak duality as established in Theorem~\ref{thm:duality:moment}, $\P^\star$ is primal optimal, and~$(\lambda_0^\star,\lambda^\star)$ is dual optimal. 
\end{proof}

Proposition~\ref{prop:jensen} implies Jensen's inequality $\E_\P [ \ell (Z)] \leq \E_{\P^\star} [ \ell (Z)]= \ell (\E_\P [ Z])$, which holds for all distributions~$\P$ feasible in~\eqref{eq:primal-Markov} \citep{jensen1906fonctions}. Proposition~\ref{prop:jensen} further shows that~\eqref{eq:dual-Markov} is solved by any affine function tangent to~$\ell$ at~$\mu$. 

If the loss function~$\ell(x,z)$ in the DRO problem~\eqref{eq:primal:dro} is concave in~$z$ for any fixed~$x\in\cX$, then Proposition~\ref{prop:jensen} implies that the {\em same} distribution~$\P^\star$ solves the inner maximization problem in~\eqref{eq:primal:dro} for every~$x\in\cX$. Hence, the DRO problem~\eqref{eq:primal:dro} reduces to the (non-robust) stochastic program $\inf_{x\in\cX}\E_{\P^\star} [\ell(x,Z)]$.

Jensen's inequality is traditionally used to approximate hard {\em stochastic} optimization problems of the form $\inf_{x\in\cX} \E_\P[\ell(x,Z)]$, where~$\P$ is a known continuous distribution of~$Z$. Proposition~\ref{prop:jensen} implies that if~$\ell(x,z)$ is concave in~$z$ for any~$x\in\cX$, then replacing~$\P$ with~$\P^\star=\delta_{\E_\P[Z]}$ leads to a conservative approximation of this stochastic program. As~$\P^\star$ is discrete (in fact, a Dirac distribution), the resulting approximate problem is much easier to solve. Its approximation quality can be improved by partitioning~$\cZ$ into finitely many convex cells and constructing separate Jensen bounds for all cells \cite[Section~10.1]{birge2011introduction}.

\subsection{Edmundson-Madansky Bound}
\label{sec:edmundson-madansky}
The worst-case expectation problem~\eqref{eq:primal-Markov} over a Markov ambiguity set and its dual~\eqref{eq:dual-Markov} can also be solved in closed form if~$\ell$ is convex and~$\cZ$ is a simplex.

\begin{proposition}[Edmundson-Madansky Bound]
\label{prop:EM-bound}
Suppose that~$\cZ$ is the probability simplex in~$\R^d$ with vertices $e_i$, $i\in[d]$, $\mu\in\rint(\cZ)$, and~$\ell$ is convex and real-valued. Then, the primal problem~\eqref{eq:primal-Markov} is solved by~$\P^\star=\sum_{i=1}^d \mu_i\delta_{e_i}$, and the dual problem~\eqref{eq:dual-Markov} is solved by~$(\lambda_0^\star,\lambda^\star)$, where~$\lambda^\star_0=0$ and~$\lambda_i^\star=\ell(e_i)$ for all $i\in[d]$. In addition, the optimal values of~\eqref{eq:primal-Markov} and~\eqref{eq:dual-Markov} both equal~$\sum_{i=1}^d \E_\P [ Z_i] \ell(e_i)$.
\end{proposition}

\begin{proof}
As~$\mu$ belongs to the probability simplex, $\P^\star$ is feasible in the primal worst-case expectation problem with objective function value $\sum_{i=1}^d \mu_i\ell(e_i)$. Also, as~$\ell$ is convex, Jensen's inequality implies that
\[
\lambda_0^\star +z^\top \lambda^\star = \sum_{i=1}^d z_i\ell(e_i)\geq \ell\left(\sum_{i=1}^d z_ie_i\right)=\ell(z)\quad \forall z\in \cZ.
\]
We conclude that $(\lambda_0^\star,\lambda^\star)$ is feasible in the dual robust optimization problem, and its objective function value amounts to $\sum_{i=1}^d \mu_i\ell(e_i)$, too. Hence, by weak duality as established in Theorem~\ref{thm:duality:moment}, $\P^\star$ is primal optimal, and~$(\lambda_0^\star,\lambda^\star)$ is dual optimal.
\end{proof}

Proposition~\ref{prop:EM-bound} implies the Edmundson-Madansky inequality, which states that $\E_\P [ \ell (Z)] \leq \E_{\P^\star} [ \ell (Z)]= \sum_{i=1}^d \E_\P [ Z_i] \ell(e_i)$ for all distributions~$\P$ feasible in~\eqref{eq:primal-Markov} \citep{edmundson:56, madansky:59}, and it shows that~\eqref{eq:dual-Markov} is solved by an affine function that touches~$\ell$ at the vertices~$e_i$, $i\in[d]$, of~$\cZ$. We emphasize, however, that Proposition~\eqref{prop:EM-bound} remains valid with minor modifications if~$\cZ$ is an arbitrary regular simplex in~$\R^d$, that is, the convex hull of~$d+1$ affinely independent vectors $v_i\in \R^d$, $i\in[d+1]$; see \citep{birgewets:86, gassmannziemba:86}.

If the loss function~$\ell(x,z)$ in~\eqref{eq:primal:dro} is convex in~$z$ for any fixed~$x\in\cX$, then Proposition~\ref{prop:EM-bound} implies that the DRO problem~\eqref{eq:primal:dro} is equivalent to the stochastic program $\inf_{x\in\cX}\E_{\P^\star} [\ell(x,Z)]$, where $\P^\star$ is {\em independent} of~$x$. As~$\P^\star$ is a discrete distribution with~$d$ atoms, this stochastic program is usually easy to solve.

\subsection{Barycentric Approximation}
Consider the worst-case expectation problem
\begin{subequations}
\label{eq:baycentric-pd}
\begin{align}
	\label{eq:primal-barycentric}
	\sup_{\P \in \cP(\cV\times\cW)} \left\{ \E_\P \left[ \ell (V,W) \right] : \E_\P \left[ V \right] = \bar v,\, \E_\P \left[ W \right] = \bar w,\, \E_\P \left[ VW^\top \right] = C\right\},
\end{align}
which maximizes the expected value of~$\ell (V,W)$ across all distributions of $Z=(V,W)$ on~$\cV\times\cW$ under which~$V$ and~$W$ have mean vectors~$\bar v$ and~$\bar w$, respectively, and cross moment matrix~$C$. Note that if~$V$ and~$W$ are uncorrelated, then $C=\bar v\bar w^\top$. Problem~\eqref{eq:primal-barycentric} optimizes over a moment ambiguity set of the form~\eqref{eq:moment-ambiguity-set} with~$f(v,w)=(v, w,vw^\top)$ and~$\cF=\{ \bar v\}\times \{\bar w\} \times \{C\}$. By Theorem~\ref{thm:duality:moment} and as the support function of~$\cF$ is linear, the problem dual to~\eqref{eq:primal-barycentric} is given by
\begin{align}
	\label{eq:dual-barycentric}
	\begin{array}{cl}
		\inf &  \lambda_0+ \lambda_v^\top \bar v + \lambda_w^\top \bar w + \inner{\Lambda}{C}\\
		\st & \lambda_0\in\R, ~\lambda_v\in\R^{d_v},~ \lambda_w\in\R^{d_w},~\Lambda \in\R^{d_v\times d_w} \\
		& \lambda_0+\lambda_v^\top v+ \lambda_w^\top w+ v^\top \Lambda w\geq \ell(v,w) \quad \forall v\in\cV,\; \forall w\in \cW.
	\end{array}
\end{align}
\end{subequations}
This dual problem seeks a bi-affine function~$b(v,w)=\lambda_0+\lambda_v^\top v+\lambda_w^\top w+ v^\top \Lambda w$ that majorizes the loss function~$\ell( v,w)$ on~$\cV\times\cW$ and minimizes~$\E_\P[b( V, W)]$ under any distribution~$\P$ feasible in~\eqref{eq:primal-barycentric}. The following proposition shows that problems~\eqref{eq:primal-barycentric} and~\eqref{eq:dual-barycentric} can be solved in closed form if~$\ell$ is a concave-convex saddle function and~$\cW$ is a simplex. Below, we use~$e_i$ to denote the $i$-th standard basis vector in~$\R^{d_w}$, $i\in[d_w]$, and~$e$ to denote the vector of ones in~$\R^{d_w}$.

\begin{proposition}[Barycentric Approximation]
\label{prop:barycentric-bound}
Suppose that~$\cV\subseteq \R^{d_v}$ is convex and $\cW\subseteq \R^{d_w}$ is the probability simplex with vertices~$e_i$, $i\in[d_w]$. Suppose also that the loss function~$\ell (v,w)$ is concave and superdifferentiable in~$v$ for any fixed~$w$ and convex in~$w$ for any fixed~$v$. In addition, suppose that~$\bar v\in \cV$, $\bar w\in\rint(\cW)$ and $Ce=\bar v$ and that problem~\eqref{eq:primal-barycentric} is feasible. Then~\eqref{eq:primal-barycentric} is solved~by
\[
\P^\star = \sum_{i=1}^{d_w} \bar w_i \, \delta_{(Ce_i/\bar w_i,e_i)}.
\]
If $\Lambda_i^\star$ is any super\-gradient in $\partial_{v} \ell(Ce_i/\bar w_i, e_i)$ for all $i\in[d_w]$ and 
\[
\lambda^\star_{w,i}=\ell(Ce_i/\bar w_i,e_i)-(\Lambda_i^\star)^\top Ce_i/ \bar w_i\quad \forall i\in[d_w],
\]
then the dual problem~\eqref{eq:dual-barycentric} is solved by~$(\lambda^\star_0,\lambda_v^\star, \lambda_w^\star, \Lambda^\star)$, where~$\lambda^\star_0=0$ and $\lambda_v^\star=0$, while~$\lambda_w^\star$ has elements~$\lambda^\star_{w,i}$ and~$\Lambda^\star$ has columns~$\Lambda_i^\star$, $i\in[d_w]$. The optimal values of~\eqref{eq:primal-barycentric} and~\eqref{eq:dual-barycentric} coincide and are both equal to
\[
\sum_{i=1}^{d_w}\mu_{w,i} \,\ell(Ce_i/\bar w_i, e_i).
\]
\end{proposition}

The condition $Ce=\bar v$ is necessary for~\eqref{eq:primal-barycentric} to be feasible. Indeed, if~$\P$ is feasible in~\eqref{eq:primal-barycentric}, then we have $Ce=\E_\P[ VW^\top e]=\E_\P[ V] =\bar v$. Here, the second equality holds because~$\P(W\in\cW)=1$ and $\cW$ is the probability simplex in~$\R^{d_w}$. However, the condition $Ce=\bar v$ is {\em not} sufficient for~\eqref{eq:primal-barycentric} to be feasible. Indeed, if the support set~$\cV=\{\bar v\}$ is a singleton, then~$C=\E_\P[VW^\top]=\bar v\bar w^\top$. That is, $V$ and~$W$ must be uncorrelated. Hence, $\cV$ and~$C$ cannot be selected independently. To circumvent this problem, Proposition~\ref{prop:barycentric-bound} requires~\eqref{eq:primal-barycentric} to be feasible.

\begin{proof}[Proof of Proposition~\ref{prop:barycentric-bound}]
Note that~$\bar w>0$ and~$e^\top \bar w=1$ because~$\bar w$ belongs to the relative interior of the probability simplex~$\cW$. Thus, $\P^\star$ is indeed a well-defined probability distribution, that is, the atoms of~$\P^\star$ have positive probabilities that sum to~$1$. In addition, $\P^\star$ is supported on~$\cV\times\cW$ because 
\[
Ce_i/\bar w_i = \E_\P\left[\frac{VW_i}{\E_\P[W_i]}\right] = \E_\P\left[ V\, \frac{\E_\P[W_i|V]}{\E_\P[W_i]}\right] \in\cV \quad \text{and} \quad e_i\in\cW \quad \forall i\in[d_w],
\]
where $\P$ is any distribution feasible in~\eqref{eq:primal-barycentric}. Note also that if~$V$ and~$W$ are uncorrelated, in which case $C=\bar v\bar w^\top$, then the $i$-th generalized barycenter $Ce_i/\bar w_i$ of~$\cV$ simplifies to~$\bar v$ for every~$i\in [d_w]$. Recalling that $Ce=\bar v$, we further have
\begin{align*}
	\E_{\P^\star}\left[ V\right] = \sum_{i=1}^{d_w} \bar w_i \, Ce_i/\bar w_i=\bar v,\quad \E_{\P^\star}\left[ W\right] = \sum_{i=1}^{d_w} \bar w_i \, e_i=\bar w
\end{align*}
and
\begin{align*}
	\E_{\P^\star}\left[ VW^\top \right] = \sum_{i=1}^{d_w} \bar w_i \, Ce_ie_i^\top/\bar w_i=C.
\end{align*}
In summary, we have shown that~$\P^\star$ is feasible in~\eqref{eq:primal-barycentric}. A similar calculation reveals that the objective function value of~$\P^\star$ in~\eqref{eq:primal-barycentric} is given by the formula in the proposition statement. Details are omitted for brevity.

To show that~$(\lambda^\star_0,\lambda_v^\star, \lambda_w^\star, \Lambda^\star)$ is feasible in~\eqref{eq:dual-barycentric}, note first that
\begin{align*}
	&\lambda^\star_0+(\lambda_v^\star)^\top v+ (\lambda_w^\star)^\top w+v^\top \Lambda^\star w  \\ & \quad = \sum_{i=1}^{d_w} w_i\left[ \ell(Ce_i/ \bar w_i,e_i) +(\Lambda_i^\star)^\top \left( v- Ce_i/\bar w_i \right) \right] \geq \sum_{i=1}^{d_w} w_i \, \ell(v,e_i) \geq \ell(v,w) 
\end{align*}
for all~$v\in \cV$ and~$w\in \cW$. The first inequality follows from the concavity of~$\ell(v,w)$ in~$v$ and the definition of~$\Lambda_i^\star$ as a supergradient, while the second inequality follows from the convexity of $\ell(v,w)$ in~$w$ and Jensen's inequality. Hence, $(\lambda^\star_0,\lambda_v^\star, \lambda_w^\star, \Lambda^\star)$ is indeed feasible in~\eqref{eq:dual-barycentric}. A similar calculation reveals that the objective function value of~$(\lambda^\star_0,\lambda_v^\star, \lambda_w^\star, \Lambda^\star)$ in~\eqref{eq:dual-barycentric} is given by the formula in the proposition statement. Consequently, by weak duality as established in Theorem~\ref{thm:duality:moment}, we have shown that~$\P^\star$ is primal optimal and~$(\lambda^\star_0,\lambda_v^\star, \lambda_w^\star, \Lambda^\star)$ is dual optimal. 
\end{proof}

Proposition~\ref{prop:barycentric-bound} remains valid with obvious minor modifications if~$\cW$ is defined as an arbitrary regular simplex in~$\R^{d_w}$ \citep{frauendorfer:92}. If $z=(v,w)$ and the loss function~$\ell(x,z)=\ell(x,v,w)$ in~\eqref{eq:primal:dro} is concave in~$v$ and convex in~$w$ for any fixed~$x\in\cX$, then Proposition~\ref{prop:barycentric-bound} implies that the DRO problem~\eqref{eq:primal:dro} is equivalent to the stochastic program $\inf_{x\in\cX}\E_{\P^\star} [\ell(x,U,V)]$, where~$\P^\star$ is {\em independent} of~$x$. As~$\P^\star$ is a discrete distribution with~$d_w$ atoms, this stochastic program is usually easy to solve. Traditionally, the distribution~$\P^\star$ is used to approximate hard {\em stochastic} optimization problems of the form $\inf_{x\in\cX} \E_\P[\ell(x,V,W)]$, where~$\P$ is a {\em known} continuous distribution of~$(V,W)$. Proposition~\ref{prop:barycentric-bound} implies that if~$\ell(x,v,w)$ is concave in~$v$ and convex in~$w$ for any~$x\in\cX$, then replacing~$\P$ with~$\P^\star$ leads to a conservative approximation, which is termed the {\em upper barycentric approximation} of the original stochastic program \citep{frauendorfer:92}. Barycentric approximations for more general stochastic programs involving loss functions that may fail to be convex and/or concave are derived by \citet{kuhn2005generalized-bounds}.

\subsection{Ben-Tal and Hochman Bound}
Consider the worst-case expectation problem
\begin{subequations}
\begin{align}
	\label{eq:primal-ben-tal-hochman}
	\sup_{\P \in \cP(\cZ)} \left\{ \E_\P \left[ \ell (Z) \right] \;:\;  \E_\P \left[ Z \right] = \mu,~ \E_\P \left[ |Z-\mu| \right] = \sigma \right\},
\end{align}
which maximizes the expected value of~$\ell(Z)$ over the family of all univariate distributions supported on~$\cZ$ with mean~$\mu$ and mean absolute deviation~$\sigma$. Note that problem~\eqref{eq:primal-ben-tal-hochman} optimizes over a moment ambiguity set of the form~\eqref{eq:moment-ambiguity-set} with~$f(z)=(z,|z-\mu|)$ and~$\cF=\{ \mu\}\times \{\sigma\}$. By Theorem~\ref{thm:duality:moment} and as the support function of~$\cF$ is linear, the problem dual to~\eqref{eq:primal-ben-tal-hochman} is given by
\begin{align}
	\label{eq:dual-ben-tal-hochman}
	\inf_{\lambda_0,\lambda_1,\lambda_2\in\R} \left\{ \lambda_0+\lambda_1 \mu +\lambda_2\sigma\;:\; \lambda_0+\lambda_1 z +\lambda_2|z-\mu|\geq \ell(z) ~~ \forall z\in\cZ \right\}.
\end{align}
\end{subequations}
Intuitively, this dual problem aims to approximate the loss function from above with a piecewise linear continuous function that has a kink at~$\mu$. The problems~\eqref{eq:primal-ben-tal-hochman} and~\eqref{eq:dual-ben-tal-hochman} can be solved in closed form if~$\ell$ is convex.

\begin{proposition}[Ben-Tal and Hochman Bound]
\label{prop:ben-tal-hochman}
Assume that~$\cZ=[0,1]$, $\mu\in(0,1)$ and $\sigma\in [0,2\mu(1-\mu)]$. Suppose also that $\ell$ is a real-valued convex function. Then, the primal problem~\eqref{eq:primal-ben-tal-hochman} is solved by
\[
\P^\star= \frac{\sigma}{2\mu}\,\delta_0+\left(1-\frac{\sigma}{2\mu}- \frac{\sigma}{2(1-\mu)}\right) \delta_\mu + \frac{\sigma}{2(1-\mu)} \, \delta_1,
\]
and the dual problem~\eqref{eq:dual-ben-tal-hochman} is solved by
\begin{align*}
	\lambda^\star_0&= \frac{(1-\mu)\ell(0)+\ell(\mu) -\mu\ell(1)}{2(1-\mu)}, \\
	\lambda^\star_1 & = \frac{(\mu-1)\ell(0)+(1-2\mu)\ell(\mu)+\mu \ell(1)}{2\mu(1-\mu)}, \\
	\lambda^\star_2&= \frac{(1-\mu)\ell(0) -\ell(\mu) +\mu\ell(1)}{2\mu(1-\mu)}.
\end{align*}
In addition, the optimal values of~\eqref{eq:primal-ben-tal-hochman} and~\eqref{eq:dual-ben-tal-hochman} coincide and are both equal to
\[
\frac{\sigma}{2\mu}\,\ell(0)+\left(1-\frac{\sigma}{2\mu}- \frac{\sigma}{2(1-\mu)}\right) \ell(\mu) + \frac{\sigma}{2(1-\mu)} \, \ell(1).
\]
\end{proposition}

\begin{proof}
The assumptions about~$\mu$ and~$\sigma$ imply that~$\P^\star$ is supported on~$\cZ$ and that the probabilities of the three atoms are non-negative and sum to~$1$. Also, we have
\begin{align*}
	\E_{\P^\star}\left[ Z\right] = \left(\mu-\frac{\sigma}{2}- \frac{\sigma \mu}{2(1-\mu)}\right) + \frac{\sigma }{2(1-\mu)} =\mu \quad\text{and}\quad \E_{\P^\star}\left[ | Z - \mu | \right] =\sigma.
\end{align*}
Thus, $\P^\star$ is feasible in~\eqref{eq:primal-ben-tal-hochman}. In addition, one readily verifies that the objective function value of~$\P^\star$ in~\eqref{eq:primal-ben-tal-hochman} is given by the formula in the proposition statement. 

Next, note that the piecewise linear function $\lambda_0^\star+\lambda^\star_1 z+\lambda_2^\star |z-\mu|$ coincides with the loss function~$\ell(z)$ for every~$z\in\{0,\mu,1\}$. As the loss function is convex, we may thus conclude that $\lambda_0^\star+\lambda^\star_1 z+\lambda_2^\star |z-\mu|$ majorizes~$\ell(z)$ for every~$z\in [0,1]=\cZ$. This shows that $(\lambda_0^\star, \lambda^\star_1,\lambda^\star_2)$ is feasible in~\eqref{eq:dual-ben-tal-hochman}. An elementary calculation further reveals that the objective function value of $(\lambda_0^\star, \lambda^\star_1,\lambda^\star_2)$ in~\eqref{eq:dual-ben-tal-hochman} is given by the formula in the proposition statement. Weak duality as established in Theorem~\ref{thm:duality:moment} thus implies that~$\P^\star$ is primal optimal and that $(\lambda_0^\star, \lambda^\star_1,\lambda^\star_2)$ is dual optimal. 
\end{proof}

Proposition~\ref{prop:ben-tal-hochman} readily extends to support sets of the form $\cZ=[a,b]$ for any~$a,b\in\R$ with $a<\mu<b$ by applying a linear coordinate transformation. If~$\ell(x,z)$ in~\eqref{eq:primal:dro} is convex in~$z$ for any fixed~$x\in\cX$, then Proposition~\ref{prop:ben-tal-hochman} implies that the DRO problem~\eqref{eq:primal:dro} is equivalent to the stochastic program $\inf_{x\in\cX}\E_{\P^\star} [\ell(x,Z)]$, where the three-point distribution $\P^\star$ is {\em independent} of~$x$. Traditionally, this stochastic program is used as a conservative approximation for
a stochastic program of the form $\inf_{x\in\cX}\E_{\P} [\ell(x,Z)]$, where~$\P$ is a known continuous distribution \citep{ben-talhochman:72}. Unlike the Jensen and Edmundson-Madansky bounds, which only use information about the location of~$\P$, and unlike the barycentric approximation, which only uses information about the location and certain cross-moments of~$\P$, the Ben-Tal and Hochman bound uses information about the location as well as the dispersion of~$\P$. Thus, it provides a tighter approximation.

If~$Z$ is a $d$-dimensional random vector with {\em independent} components~$Z_i$, $i\in[d]$, each of which has a known mean and mean absolute deviation, then one can show that the worst-case expected value of a convex loss function is attained by $\P^\star=\otimes_{i=1}^d \P_i^\star$, where each~$\P_i^\star$ is a three-point distribution constructed as in Proposition~\ref{prop:ben-tal-hochman} \citep{ben-talhochman:72}. In this case, $\P^\star$ is a discrete distribution with~$3^d$ atoms. Hence, evaluating expected values with respect to~$\P^\star$ is generically hard but becomes tractable for a class of exponential loss functions that offer safe approximations for chance constraints \citep{postek2018robust}.

\subsection{Scarf's Bound}
\label{sec:scarf-bound}
Consider the worst-case expectation problem
\begin{subequations}
\begin{align}
	\label{eq:primal-Chebyshev}
	\sup_{\P \in \cP(\cZ)} \left\{ \E_\P \left[ \ell (Z) \right] \;:\;  \E_\P \left[ Z \right] = 0,~ \E_\P \left[ Z^2 \right] = \sigma^2 \right\},
\end{align}
which maximizes the expected value of $\ell(Z)$ over the Chebyshev ambiguity set of all univariate distributions supported on~$\cZ$ with mean~$0$ and variance~$\sigma^2$. This Chebyshev ambiguity set is a moment ambiguity set of the form~\eqref{eq:moment-ambiguity-set} with~$f(z)=(z, z^2)$ and~$\cF=\{0\}\times\{\sigma^2\}$. By Theorem~\ref{thm:duality:moment} and as the support function of~$\cF$ is linear, the problem dual to~\eqref{eq:primal-Chebyshev} is given by
\begin{align}
	\label{eq:dual-Chebyshev}
	\inf_{\lambda_0,\lambda_1,\lambda_2\in\R} \left\{ \lambda_0+\lambda_2\sigma^2\;:\; \lambda_0+\lambda_1 z +\lambda_2 (z-\mu)^2\geq \ell(z) ~~ \forall z\in\cZ \right\}.
\end{align}
\end{subequations}
This dual problem seeks a quadratic function $q(z)=\lambda_0+\lambda_1z+\lambda_2 z^2$ that majorizes the loss function~$\ell(z)$ throughout~$\cZ$ and has minimal expectation $\E_\P[q(Z)]$ under any distribution~$\P$ with mean~$0$ and variance~$\sigma^2$. The problems~\eqref{eq:primal-Chebyshev} and~\eqref{eq:dual-Chebyshev} can be solved in closed form if~$\ell$ is a ramp function.

\begin{proposition}[Scarf's Bound]
\label{prop:scarf}
If~$\cZ=\R$, $\sigma^2\in\R_+$ and~$\ell(z)=\max\{z-a,0\}$ is a ramp function with a kink at~$a\in\R$, then the primal problem~\eqref{eq:primal-Chebyshev} is solved by
\[
\P^\star = \half\left(1+\frac{a}{\sqrt{a^2+\sigma^2}} \right)\delta_{a-\sqrt{a^2+\sigma^2}} + \half\left(1-\frac{a}{\sqrt{a^2+\sigma^2}} \right)\delta_{a+\sqrt{a^2+\sigma^2}},
\]
and the dual problem~\eqref{eq:dual-Chebyshev} is solved by
\begin{align*}
	\lambda^\star_0 = \frac{\left( a -\sqrt{a^2+\sigma^2} \right)^2}{4\sqrt{a^2+\sigma^2}}, \quad \lambda^\star_1 = -\frac{a -\sqrt{a^2+\sigma^2}}{2\sqrt{a^2+\sigma^2}} \quad \text{and} \quad \lambda^\star_2 = \frac{1}{4\sqrt{a^2+\sigma^2}}.
\end{align*}
The optimal values of~\eqref{eq:primal-Chebyshev} and~\eqref{eq:dual-Chebyshev} are both equal to $\half( \sqrt{a^2+\sigma^2} -a )$.
\end{proposition}

\begin{proof}
Note that the two-point distribution~$\P^\star$ is well-defined, that is, its atoms have non-negative probabilities that sum to~$1$. By the definition of~$\P^\star$, we also have
\begin{align*}
	\E_{\P^\star}\left[ Z\right] & = \half\left(1+\frac{a}{\sqrt{a^2+\sigma^2}} \right)\left(a-\sqrt{a^2+\sigma^2}\right) \\
	& \qquad + \half\left(1-\frac{a}{\sqrt{a^2+\sigma^2}} \right)\left(a+\sqrt{a^2+\sigma^2}\right) =0.
\end{align*}
Similarly, it is easy to verify that $\E_{\P^\star}[ Z^2]=\sigma^2$. This shows that~$\P^\star$ is feasible in~\eqref{eq:primal-Chebyshev}. The objective function value of~$\P^\star$ is
\[
\E_{\P^\star}\left[ \ell(Z) \right] = \E_{\P^\star}\left[ \max\{Z-a,0\} \right] = \half\left( \sqrt{a^2+\sigma^2} -a \right).
\]
Next, observe that the dual variables $(\lambda^\star_0, \lambda^\star_1, \lambda^\star_2)$ defined in the proposition statement give rise to the quadratic function
\[
q^\star(z)=\lambda_0^\star+\lambda_1^\star z+\lambda_2^\star z^2= \frac{1}{4\sqrt{a^2+\sigma^2}}\left( z-a+\sqrt{a^2+\sigma^2}\right)^2.
\]
We will now show that $q^\star(z)\geq \max\{z-a,0\}=\ell(z)$ for all~$z\in\cZ$. Clearly, $q^\star$ is non-negative and evaluates to~$0$ at~$a-\sqrt{a^2+\sigma^2}$. In addition, $q^\star$ touches the affine function~$z-a$ at~$a+\sqrt{a^2+\sigma^2}$. To see this, note that
\[
q^\star\left(a+\sqrt{a^2+\sigma^2} \right) = \sqrt{a^2+\sigma^2} \quad \text{and}\quad \frac{\diff}{\diff z} q^\star\left(a+\sqrt{a^2+\sigma^2} \right) = 1.
\]
Hence, $q^\star$ majorizes the ramp function~$\ell(z)$, implying that $(\lambda_0^\star, \lambda_1^\star, \lambda_2^\star)$ is dual feasible. Also, the objective function value of $(\lambda_0^\star, \lambda_1^\star, \lambda_2^\star)$ is given by
\[
\lambda_0^\star+\lambda_2^\star \sigma^2 = \frac{1}{4\sqrt{a^2+\sigma^2}}\left( \sigma^2 +\left( a -\sqrt{a^2+\sigma^2} \right)^2\right) = \half\left( \sqrt{a^2+\sigma^2} -a \right).
\]
As the objective function values of~$\P^\star$ and~$(\lambda_0^\star, \lambda_1^\star, \lambda_2^\star)$ match, weak duality as established in Theorem~\ref{thm:duality:moment} thus implies that~$\P^\star$ is primal optimal and that $(\lambda_0^\star, \lambda^\star_1,\lambda^\star_2)$ is dual optimal. This observation completes the proof.
\end{proof}

Proposition~\ref{prop:scarf} was first derived by~\citet{Scarf:58} in his pioneering treatise on the distributionally robust newsvendor problem; see also \cite[Theorem~1]{jagannathan:77}. Note that if the mean of~$Z$ is known to equal~$\mu\neq 0$ instead of~$0$, then Scarf's bound remains valid if we replace~$a$ with~$a-\mu$. \citet{gallegomoon:93} extend Scarf's bound to more general loss functions such as wedge functions or ramp functions with a discontinuity, whereas \citet{natarajan2018asymmetry} extend Scarf's bound to more general ambiguity sets that not only contain information about the mean and variance of~$Z$ but also about its {\em semi}variance. In addition, \citet{das2021heavy} discuss variants of Scarf’s bound that rely on information about the mean and the $\alpha$-th moment of~$Z$ for any~$\alpha>1$.

Proposition~\ref{prop:scarf} is often used to reformulate DRO problems of the form~\eqref{eq:primal:dro} whose objective function is given by the expected value of a ramp function. Examples include distributionally robust newsvendor, support vector machine or mean-CVaR portfolio selection problems. In most of these applications, the location~$a$ of the kind of the ramp function is a decision variable or a function of the decision variables. Thus, the worst-case distribution~$\P^\star$ is decision-dependent, which means that Proposition~\ref{prop:scarf} does {\em not} enable us to reduced the DRO problem~\eqref{eq:primal:dro} to a stochastic program with a single fixed worst-case distribution.

\subsection{Marshall and Olkin Bound}
\label{sec:marshall&olkin}
Consider the worst-case probability problem
\begin{subequations}
\begin{align}
	\label{eq:primal-marshall-olkin}
	\sup_{\P \in \cP(\cZ)} \left\{ \P \left(Z \in\cC \right) \;:\;  \E_\P \left[ Z \right] = 0,~ \E_\P \left[ ZZ^\top \right] = I_d \right\},
\end{align}
which maximizes the probability of the event~$Z\in\cC$ over the Chebyshev ambiguity set of all distributions on~$\cZ=\R^d$ with mean~$0$ and covariance matrix~$I_d$. This Chebyshev ambiguity set is a moment ambiguity set of the form~\eqref{eq:moment-ambiguity-set} with~$f(z)=(z, zz^\top)$ and~$\cF=\{0\}\times\{I_d\}$. If we set~$\ell$ to the characteristic function of~$\cC$ defined through $\ell(z)=\ds{1}_{z\in\cC}$ for all~$z\in\cZ$, then the worst-case probability problem~\eqref{eq:primal-marshall-olkin} can be recast as a worst-case expectation problem. By Theorem~\ref{thm:duality:moment} and as the support function of~$\cF$ is linear, the corresponding dual problem is thus given by
\begin{align}
	\label{eq:dual-marshall-olkin}
	\inf_{\lambda_0\in\R,\lambda\in\R^d,\Lambda\in\S^d} \left\{ \lambda_0+ \inner{\Lambda}{I_d} \;:\; \lambda_0+\lambda^\top z +z^\top\Lambda z\geq \ell(z) ~~ \forall z\in\cZ \right\}.
\end{align}
\end{subequations}
The problems~\eqref{eq:primal-marshall-olkin} and~\eqref{eq:dual-marshall-olkin} can be solved analytically if~$\cC$ is convex and closed.

\begin{proposition}[Marshall and Olkin Bound]
\label{prop:marshall-olkin}
Suppose that~$\cZ=\R^d$, $\cC\subseteq\R^d$ is convex and closed, and~$\ell$ is the characteristic function of~$\cC$. Set~$\Delta=\min_{z\in\cC}\|z\|_2$, and let~$z_0\in\R^d$ be the unique minimizer of this problem. Then, the optimal values of~\eqref{eq:primal-marshall-olkin} and~\eqref{eq:dual-marshall-olkin} are both equal to $(1+\Delta^2)^{-1}$. If~$\Delta=0$, then the supremum of~\eqref{eq:primal-marshall-olkin} may not be attained. However, if~$\Delta>0$, then~\eqref{eq:primal-marshall-olkin} is solved by
\[
\P^\star = \frac{1}{1+\Delta^2}\,\delta_{z_0} + \frac{\Delta^2}{1+\Delta^2}\Q,
\]
where $\Q\in\cP(\cZ)$ is an arbitrary distribution with mean~$-z_0/\Delta^2$ and covariance matrix $\frac{1+\Delta^2}{\Delta^2} \,(I_d-z_0 z_0^\top/\Delta^2)$. For any~$\Delta\geq 0$, problem~\eqref{eq:dual-Chebyshev} is solved by
\[
\lambda_0^\star = \frac{1}{(1+\Delta^2)^2},\quad \lambda^\star = \frac{2z_0}{(1+\Delta^2)^2}\quad \text{and}\quad \Lambda^\star=\frac{z_0z_0^\top}{(1+\Delta^2)^2}.
\]
\end{proposition}

\begin{proof}
Assume first that~$\Delta=0$, that is, $0\in\cC$. For every~$j\in\N$, let~$\Q_j\in\cP(\cZ)$ be any distribution with mean~$0$ and covariance matrix~$j I_d$, and set
\[
\P_j=(1-1/j)\,\delta_0+(1/j)\,\Q_j.
\]
We thus have $\E_{\P_j}[Z]=0$ and $\E_{\P_j}[ZZ^\top] = I_d$, which implies that~$\P_j$ is feasible in~\eqref{eq:primal-marshall-olkin}. In addition, the objective function value of~$\P_j$ in~\eqref{eq:primal-marshall-olkin} satisfies
\[
\P_j(Z\in\cC) =1-1/j+\Q_j(Z\in\cZ)/j\geq 1-j^{-1}.
\]
Driving~$j$ to infinity reveals that problem~\eqref{eq:primal-marshall-olkin} is trivial for~$\Delta=0$ and that its supremum equals~$1$. Assume now that~$\Delta>0$, and let~$\Q\in\cP(\cZ)$ be an arbitrary distribution with mean~$-z_0/\Delta^2$ and covariance matrix $\frac{1+\Delta^2}{\Delta^2} \,(I_d-z_0 z_0^\top/\Delta^2)$. Such a distribution is guaranteed to exist because $I_d\succeq z_0 z_0^\top/\Delta^2$. In addition, define~$\P^\star$ as in the proposition statement. By construction, we have $\E_{\P^\star}[ Z] =0$ and
\begin{align*}
	\E_{\P^\star}\left[ ZZ^\top \right] & = \frac{z_0z_0^\top}{1+\Delta^2} + \frac{\Delta^2}{1+\Delta^2} \, \E_{\Q}\left[ ZZ^\top \right] \\
	& = \frac{z_0z_0^\top}{1+\Delta^2} + I_d- \frac{z_0 z_0^\top}{\Delta^2} + \frac{\Delta^2}{1+\Delta^2} \, \frac{z_0 z_0^\top}{\Delta^4}= I_d.
\end{align*}
Also, the objective function value of~$\P^\star$ in~\eqref{eq:primal-marshall-olkin} is given by $\P^\star(Z\in\cC) =  (1+\Delta^2)^{-1}$. Next, use $(\lambda^\star_0, \lambda^\star, \Lambda^\star)$ defined in the proposition to 
construct the quadratic function
\[
q^\star(z)=\lambda^\star_0+(\lambda^\star)^\top z+ z^\top \Lambda^\star z=\frac{(z_0^\top z + 1)^2}{(1+\Delta^2)^2}.
\]
Note that $q^\star$ is non-negative and constant on any hyperplane perpendicular to~$z_0$. If $\Delta>0$, we have~$q^\star(z_0)=1$ as well as~$q^\star(-z_0/\Delta^2)=0$. 
Thus, at every~$z\in\cZ$ with~$z_0^\top z\geq -1$, the quadratic function~$q^\star(z)$ is non-decreasing in the direction of~$z_0$. As~$z_0$ minimizes the differentiable convex function~$\|z\|_2^2$ over the convex closed set~$\cC$, we have~$z_0^\top (z-z_0)\geq 0$ for all~$z\in\cC$. By the monotonicity properties of~$q^\star$, this implies that~$q^\star(z)\geq 1$ for every~$z\in\cC$. Hence, the quadratic function~$q^\star$ majorizes the indicator function~$\ell$ on~$\cZ$, which implies that~$(\lambda_0^\star, \lambda^\star, \Lambda^\star)$ is dual feasible. If $\Delta=0$, then $q^\star (z)=1$ for all~$z\in\cZ$, and~$(\lambda_0^\star, \lambda^\star, \Lambda^\star)$ is also dual feasible. In any case, one readily verifies that its objective function value is given~by
\[
\lambda^\star_0+ \inner{\Lambda^\star}{I_d} = \left(1+\Delta^2\right)^{-1}.
\]
As the objective function values of~$\P^\star$ and~$(\lambda_0^\star, \lambda^\star, \Lambda^\star)$ for $\Delta>0$ match, weak duality as established in Theorem~\ref{thm:duality:moment} implies that~$\P^\star$ is primal optimal and that~$(\lambda_0^\star, \lambda^\star, \Lambda^\star)$ is dual optimal. If $\Delta=0$, then the optimal value~$1$ of the primal problem also matches the objective function value of~$(\lambda_0^\star, \lambda^\star, \Lambda^\star)$ in~\eqref{eq:dual-marshall-olkin}. Hence, $(\lambda_0^\star, \lambda^\star, \Lambda^\star)$ remains dual optimal even though the supremum of the primal problem may not be attained. This observation completes the proof.
\end{proof}

\subsection{Chebyshev Risk}
\label{sec:chebyshev-risk}
Analytical solutions of worst-case {\em expectation} problems sometimes enable us to evaluate the worst-case {\em risk} of a random variable if the underlying risk measure is law-invariant, translation-invariant as well as scale-invariant; see Definition~\ref{def:invariance-properties}. 
%
%
For example, it is elementary to verify that the $\beta$-VaR and $\beta$-CVaR constitute law-invariant, translation-invariant as well as scale-invariant risk measures for every fixed~$\beta \in(0,1)$. If the distribution of~$Z$ is unknown except for its mean $\mu\in\R^d$ and covariance matrix~$\Sigma\in\S_+^d$, then it is natural to quantify the riskiness of an uncertain loss~$\ell(Z)$ under a law-invariant risk measure~$\varrho$ by the corresponding {\em Chebyshev risk}. Specifically, the Chebyshev risk of~$\ell(Z)$ is defined as the worst-case risk
\begin{align*}
\sup_{\P\in \cP(\mu, \Sigma)} \varrho_\P[\ell(Z)],
\end{align*}
where $\cP(\mu,\Sigma)$ denotes the Chebyshev ambiguity set that contains all probability distributions on~$\R^d$ with mean~$\mu\in\R^d$ and covariance matrix~$\Sigma\in\S^d_+$. 

We now describe a powerful tool for analyzing the Chebyshev risk with respect to any law-, translation- and scale-invariant risk measure. To this end, recall that if~$Z$ follows some distribution~$\P$ on~$\R^d$, then~$L=\ell(Z)$ follows the pushforward distribution~$\P\circ \ell^{-1}$ on~$\R$. If~$\P$ is uncertain and only known to belong to some ambiguity set~$\cP$, then the distribution of~$L=\ell(Z)$ is also uncertain and only known to belong to the pushforward ambiguity set $\cP\circ \ell^{-1}=\{\P\circ \ell^{-1}:\P\in\cP\}$. The following proposition due to \citet{popescu2007robust} shows that linear pushforwards of Chebyshev ambiguity sets are again Chebyshev ambiguity sets.

\begin{proposition}[Pushforwards of Chebyshev Ambiguity Sets] 
\label{prop:projection-chebyshev}
\hspace{-2mm} If $\mu\in\R^d$, $\Sigma\in\S_+^d$, $\theta\in\R^d$, and $\ell:\R^d\to\R$ is the linear transformation defined through $\ell(z)=\theta^\top z$, then the pushforward of the Chebyshev ambiguity set $\cP(\mu,\Sigma)$ is the Chebyshev ambiguity set of all distributions on~$\R$ with mean~$\theta^\top\mu$ and variance~$\theta^\top\Sigma\theta$, that is,
\[
\cP(\mu,\Sigma)\circ \ell^{-1} = \cP(\theta^\top \mu, \theta^\top\Sigma\theta).
\]
\end{proposition}
\begin{proof}
Select first any distribution~$\P\in\cP(\mu,\Sigma)$. If the random vector~$Z$ follows~$\P$, then the random variable $L=\ell(Z)$ follows~$\P\circ \ell^{-1}$. Thus, we have
\[
\E_{\P\circ \ell^{-1}}[L] = \E_\P[\ell(Z)]=\E_\P[\theta^\top Z]=\theta^\top \mu,
\]
where the first equality follows from the measure-theoretic change of variables formula. Similarly, one can show that $\E_{\P\circ \ell^{-1}}[(L-\theta^\top\mu)^2] =\theta^\top\Sigma\theta$. Thus, we find
\[
\cP(\mu,\Sigma)\circ \ell^{-1} \subseteq \cP(\theta^\top \mu, \theta^\top\Sigma\theta).
\]
Next, select any $\Q_L\in \cP(\theta^\top \mu, \theta^\top\Sigma\theta)$. If~$\theta^\top\Sigma\theta = 0$, then~$\Q_L=\delta_{\theta^\top \mu}$, which coincides with the pushforward distribution~$\P\circ \ell^{-1}$ for any~$\P\in\cP(\mu,\Sigma)$. In the remainder of the proof we may thus assume that~$\theta^\top\Sigma\theta \neq 0$. Let now~$L$ be a random variable governed by~$\Q_L$, and let~$M$ be a $d$-dimensional random vector governed by an arbitrary distribution~$\Q_M\in\cP(\R^d)$ with mean~$\mu$ and covariance matrix~$\Sigma$. For example, we can set~$\Q_M$ to the normal distribution $\cN(\mu,\Sigma)$. Assume~$L$ and~$M$ are independent. Then, the distribution~$\P$ of the $d$-dimensional random vector
\[
\textstyle Z=\frac{1}{\theta^\top\Sigma\theta}\Sigma \theta\, L + \left(I_d-\frac{1}{\theta^\top\Sigma\theta}\Sigma \theta\theta^\top \right) M
\]
belongs to~$\cP(\mu, \Sigma)$. By the construction of~$L$ and~$M$, we have indeed
\[
\E_\P[Z]= \textstyle \frac{1}{\theta^\top\Sigma\theta}\Sigma \theta \,  \E_{\Q_L}[L] + \left(I_d-\frac{1}{\theta^\top\Sigma\theta}\Sigma \theta\theta^\top \right) \E_{\Q_M}[M] =\mu
\]
and
\begin{align*}
	& \E_\P[(Z-\mu)(Z-\mu)^\top] \\
	&\qquad =  \textstyle \frac{1}{\theta^\top\Sigma\theta}\Sigma \theta \theta^\top\Sigma + \left(I_d-\frac{1}{\theta^\top\Sigma\theta}\Sigma \theta\theta^\top \right)\Sigma \left(I_d-\frac{1}{\theta^\top\Sigma\theta}\theta\theta^\top \Sigma \right) = \Sigma.
\end{align*}
The first equality in the above expression holds because~$L$ and~$M$ are independent, $L$ has variance~$\theta^\top\Sigma\theta$ and~$M$ has covariance matrix~$\Sigma$. By construction, we further have $\ell(Z)=\theta^\top Z=L$, which implies that $\P\circ \ell^{-1}=\Q_L$. We have thus shown that for every $\Q_L\in \cP(\theta^\top \mu, \theta^\top\Sigma\theta)$ there exists $\P\in \cP(\mu, \Sigma)$ with $\P\circ \ell^{-1}=\Q_L$, that is,
\[
\cP(\mu,\Sigma)\circ \ell^{-1} \supseteq \cP(\theta^\top \mu, \theta^\top\Sigma\theta).
\]
This observation completes the proof.
\end{proof}

Generalizations of Proposition~\ref{prop:projection-chebyshev} to multi-dimensional affine transformations and to subfamilies of the Chyebyshev ambiguity set that contain only distributions with certain structural properties (such as symmetry, linear unimodality or log-concavity etc.) are presented in \citep{yu2009projection}; see also~\citep{chen2011tight}.

We now show that if the risk measure~$\varrho$ is law-, translation- and scale-invariant and the loss function~$\ell$ is linear, then the Chebyshev risk reduces to a mean-standard deviation risk measure, which involves the standard risk coefficient of~$\varrho$. 

\begin{definition}[Standard Risk Coefficient]
\label{def:standard-risk-coefficient}
The standard risk coefficient of a law-invariant risk measure~$\varrho$ is given by $\alpha = \sup_{\Q \in \cP(0, 1)} \varrho_\Q [L]$.
\end{definition} 

Thus, the standard risk coefficient of~$\varrho$ is defined as the worst-case risk of an uncertain loss~$L$ whose distribution~$\Q$ is only known to have mean~$0$ and variance~$1$.

\begin{proposition}[Chebyshev Risk]
\label{prop:chebyshev-risk-mean-sdev}
If~$\varrho$ is a law-, translation- and scale-invariant risk measure with standard risk coefficient~$\alpha$, there is~$\theta\in\R^d$ with $\ell(z)=\theta^\top z$ for all~$z\in\R^d$, and $\cP(\mu, \Sigma)$ is the Chebyshev ambiguity set of all distributions on~$\R^d$ with mean~$\mu\in\R^d$ and covariance matrix~$\Sigma\in\S^d_+$, then the Chebyshev risk satisfies
\begin{align*}
	\sup_{\P\in \cP(\mu, \Sigma)} \varrho_\P[\ell(Z)] = \theta^\top \mu + \alpha \sqrt{\theta^\top \Sigma \theta}.
\end{align*}
\end{proposition}
\begin{proof}
If $\theta^\top\Sigma\theta=0$, then 
\begin{align*}
	\sup_{\P \in \cP(\mu,\Sigma)} \varrho_\P \left[\theta^\top Z\right]& =  \theta^\top \mu + \sup_{\P \in \cP(\mu,\Sigma)} \varrho_\P \left[\theta^\top (Z-\mu)\right]\\
	& =  \theta^\top \mu + \sup_{\P \in \cP(\mu,\Sigma)} \varrho_\P \left[0\right] = \theta^\top \mu ,
\end{align*}
where the first equality holds because~$\varrho$ is translation invariant, whereas the second equality holds because $\theta^\top(Z-\mu)$ equals $0$ in law under any~$\P\in\cP(\mu,\Sigma)$ and because~$\varrho$ is law-invariant. Finally, the third equality follows from the scale-invariance of~$\varrho$. If~$\theta^\top\Sigma\theta>0$, on the other hand, then we have
\begin{align*}
	\sup_{\P \in \cP(\mu,\Sigma)} \varrho_\P \left[\theta^\top Z\right]& =  \theta^\top \mu + \sup_{\P \in \cP(\mu,\Sigma)} \varrho_\P \left[\theta^\top (Z-\mu)\right] \\
	& = \theta^\top \mu + \sup_{\P \in \cP(\mu,\Sigma)} \varrho_\P \left[\frac{\theta^\top (Z-\mu)}{\sqrt{\theta^\top \Sigma \theta}}\right] \sqrt{\theta^\top \Sigma \theta}\\
	& = \theta^\top \mu + \alpha \sqrt{\theta^\top \Sigma \theta},
\end{align*}
where the first two equalities follow from the translation- and scale-invariance of~$\varrho$, respectively. The third equality follows from Proposition~\ref{prop:projection-chebyshev}, the law-invariance of~$\varrho$ and the definition of~$\alpha$. Indeed, the pushforward of the multivariate Chebyshev ambiguity set~$\cP(\mu,\Sigma)$ under the transformation $\ell(z)=\theta^\top(z-\mu)/\sqrt{\theta^\top\Sigma\theta}$ coincides with the univariate standard Chebyshev ambiguity set~$\cP(0,1)$.
\end{proof}

The standard risk coefficient of a generic law-invariant risk measure may be difficult to compute. We now show, however, that the standard risk coefficients of the VaR and the CVaR match and are available in closed form.

\begin{proposition}[Standard Risk Coefficients of VaR and CVaR] 
\label{prop:var=cvar}
For any $\beta\in(0,1)$, the standard risk coefficients of the $\beta$-VaR and the $\beta$-CVaR coincide, that is,
\begin{align*}
	\sup_{\Q \in \cP(0,1)} \beta\CVaR_{\Q}[L] 
	= \sup_{\Q \in \cP(0,1)} \beta\VaR_{\Q}[L] = \sqrt{\frac{1-\beta}{\beta}}.
	\end{align*}\end{proposition}
	
	\begin{proof} As $\beta\CVaR_{\Q}[L]$ upper bounds $\beta\VaR_{\Q}[L]$ for every $\Q\in \cP(0,1)$, we have
\begin{align}
	\label{eq:wcvar-wccvar}
	\sup_{\Q \in \cP(0,1)} \beta\CVaR_{\Q}[L] 
	\geq \sup_{\Q \in \cP(0,1)} \beta\VaR_{\Q}[L].
\end{align}
The rest of the proof proceeds as follows. We first derive an analytical formula for the worst-case $\beta$-VaR on the right hand side ({\em Step~1}). Next, we prove that the same analytical formula provides an upper bound on the worst-case $\beta$-CVaR on the left hand side ({\em Step~2}). The claim then follows from the above inequality.

{\em Step~1.} We first express the worst-case $\beta$-VaR as its smallest upper bound to find
\begin{align*}
	\sup_{\Q \in \cP(0,1)} \beta\VaR_{\Q}[L]
	& = \inf_{\tau\in\R} \left\{ \tau : \beta\VaR_{\Q}(L) \leq \tau \; \forall \Q\in \cP(0,1) \right\} \\
	& = \inf_{\tau\in\R} \left\{ \tau : \Q(L \geq \tau) \leq \beta \; \forall \Q\in \cP(0,1) \right\} \\
	& = \inf_{\tau\in\R} \bigg\{ \tau : \frac{1}{1+\tau^2} \leq \beta \bigg\} = \sqrt{\frac{1-\beta}{\beta}}.
\end{align*}
The second equality in the above derivation follows from~\eqref{eq:var-refomulation}, and the third equality follows from the Marshall and Olkin bound of Proposition~\ref{prop:marshall-olkin}. The final formula is obtained by analytically solving the minimization problem over~$\tau$.

{\em Step~2.} The max-min inequality\footnote{The Chebyshev ambiguity set~$\cP(0,1)$ is {\em not} weakly compact (see Example~\ref{ex:chebyshev-not-compact}). Therefore, Sion's minimax theorem does not allow us to interchange the infimum over~$\tau$ and the supremum over~$\Q$. While we could instead invoke Theorem~\ref{thm:duality:regular}, this is actually not needed to prove Proposition~\ref{prop:var=cvar}.} and the definition of the $\beta$-CVaR imply that
\begin{align*}
	\sup_{\Q \in \cP(0,1)} \beta\CVaR_{\Q}[L]
	&\leq \inf_{\tau \in \R} \sup_{\Q \in \cP(0,1)}  \tau + \frac{1}{\beta} \E_\Q \left[\max\{L - \tau,0\} \right] \\
	&= \inf_{\tau \in \R} \tau + \frac{1}{2\beta}\left( \sqrt{1+\tau^2} -\tau \right) = \sqrt{\frac{1-\beta}{\beta}},
\end{align*}
where the first equality follows from Scarf's bound derived in Proposition~\ref{prop:scarf}, and the last equality is obtained by analytically solving the convex minimization problem over~$\tau$. The unique minimizer is given~by
\[
\tau^\star = \frac{1-2\beta}{2\sqrt{\beta(1-\beta)}}.
\]
This completes {\em Step~2}. The claim then follows by combining the analytical formula for the worst-case $\beta$-VaR found in {\em Step~1} and the analytical upper bound on the worst-case $\beta$-CVaR found in {\em Step~2} with the elementary inequality~\eqref{eq:wcvar-wccvar}.
\end{proof}

Propositions~\ref{prop:chebyshev-risk-mean-sdev} and~\ref{prop:var=cvar} provide an analytical formula for the Chebyshev risk of a linear loss function provided that the underlying risk measure is the VaR or the CVaR. The formula for the worst-case VaR was first derived in \citep{lanckriet2001minimax, lanckriet2002robust,ghaoui2003worst}; see also~\citep{calafiore2006distributionally}. The equality of the worst-case VaR and the worst-case CVaR was discovered in~\citep{zymler2013distributionally}. It not only holds for linear but also for arbitrary concave and arbitrary quadratic (not necessarily concave) loss functions. Proposition~\ref{prop:chebyshev-risk-mean-sdev} follows from \citep{nguyen2021mean}. The standard risk coefficient can be characterized in closed form for a wealth of law-, translation- and scale-invariant risk measures other than the VaR and the CVaR. It is available, for instance, for all spectral risk measures and all risk measures that admit a Kusuoka representation \citep{li2018closed} as well as all distortion risk measures \citep{cai2023distributionally}; see also~\citep{nguyen2021mean}.

\subsection{Gelbrich Risk}
\label{sec:gelbrich-risk}
Denote by~$\cG_r (\hat\mu, \hat\Sigma)$ the Gelbrich ambiguity set of all distributions~$\P\in\cP(\R^d)$ whose mean-covariance pairs $(\mu,\Sigma)\in\R^d\times\S_+^d$ reside in a ball of radius~$r\geq 0$ around~$(\hat \mu,\hat \Sigma)\in\R^d\times\S_+^d$ with respect to the Gelbrich distance; see Definition~\ref{def:Gelbrich}. Recall from Section~\ref{sec:chebyshev-with-moment-uncertainty} that the Gelbrich ambiguity set accounts for moment ambiguity and thus often provides a more realistic account of uncertainty than a na\"ive Chebyshev ambiguity set. If the distribution of~$Z$ is only known to have a mean-covariance pair close to~$(\hat \mu,\hat \Sigma)$, then it is natural to quantify the riskiness of an uncertain loss~$\ell(Z)$ under a law-invariant risk measure~$\varrho$ by the {\em Gelbrich risk}
\begin{align*}
\sup_{\P\in \cG_r (\hat\mu, \hat\Sigma)} \varrho_\P[\ell(Z)].
\end{align*}
By construction, $\cG_r (\hat\mu, \hat\Sigma)$ is the union of all Chebyshev ambiguity sets~$\cP(\mu,\Sigma)$ corresponding to a mean-covariance pair~$(\mu,\Sigma)$ with $\G ( (\mu, \Sigma), (\hat \mu, \hat \Sigma)) \leq r$. This decomposition of the Gelbrich ambiguity set into Chebyshev ambiguity sets allows us via Proposition~\ref{prop:chebyshev-risk-mean-sdev} to derive an analytical formula for the Gelbrich risk.

\begin{proposition}[Gelbrich Risk]
\label{prop:gelbrich-risk-mean-sdev}
Assume that~$\varrho$ is a law-, translation- and scale-invariant risk measure with standard risk coefficient~$\alpha\in\R_+$, there is~$\theta\in\R^d$ with $\ell(z)=\theta^\top z$ for all~$z\in\R^d$, and $\cG_r(\hat \mu, \hat\Sigma)$ is the Gelbrich ambiguity set of all distributions on~$\R^d$ whose mean-covariance pairs have a Gelbrich distance of at most~$r\geq 0$ from~$(\hat\mu, \hat\Sigma)\in\R^d\times \S^d_+$. Then, the Gelbrich risk satisfies
\begin{align}
	\label{eq:gelbrich-risk-regularization}
	\sup_{\P \in \cG_r (\hat\mu, \hat\Sigma)} \; \varrho_\P \left[ \theta^\top Z \right] 
	= \hat \mu^\top \theta + \alpha \sqrt{\theta^\top \hat \Sigma \theta } + r \sqrt{1+ \alpha^2}\, \| \theta \|_2.
\end{align}
\end{proposition}

\begin{proof}
Assume first that~$\hat \Sigma \succ 0$. If $\theta=0$, then the claim holds trivially because~$\varrho$ is law- and scale-invariant. If $r=0$, then the claim follows immediately from Proposition~\ref{prop:chebyshev-risk-mean-sdev}. We may thus assume that~$\theta \neq 0$ and~$r > 0$. In this case, we have
\begin{align*}
	\sup_{\P \in \cG_r (\hat\mu, \hat\Sigma)} \; \varrho_\P \left[ \theta^\top Z \right] 
	&= \left\{ 
	\begin{array}{cl}
		\sup & \displaystyle \sup_{\P \in \cP(\mu, \Sigma)} \, \varrho_\P \left[ \theta^\top Z \right] \\[2ex]
		\st & \mu \in \R^d, \ \Sigma \in \S_+^d, \ \G \big( (\mu, \Sigma), (\hat \mu, \hat \Sigma) \big) \leq r
	\end{array}
	\right. \\
	&= \left\{ 
	\begin{array}{cl}
		\sup & \mu^\top \theta + \alpha \sqrt{\theta^\top \Sigma \theta} \\[0.5ex]
		\st & \mu \in \R^d, \ \Sigma \in \S_+^d \\
		& \|\mu - \hat \mu \|^2 + \Tr \left(\Sigma + \hat \Sigma - 2 \big( \hat \Sigma^\half \Sigma \hat \Sigma^\half \big)^\half \right) \leq r^2,
	\end{array}
	\right.
\end{align*}
where the first equality exploits the decomposition of the Gelbrich ambiguity set into Chebyshev ambiguity sets. The second equality follows from Proposition~\ref{prop:chebyshev-risk-mean-sdev} and Definition~\ref{def:Gelbrich}. By dualizing the resulting convex optimization problem, we~find
\begin{equation}
	\label{eq:gelbrich-dual}
	\begin{aligned}
		\sup_{\P \in \cG_r (\hat\mu, \hat\Sigma)} \; \varrho_\P \left[ \theta^\top Z \right] = \inf_{\gamma \in \R_+} & \Bigg\{ \gamma \big( r^2 - \Tr(\hat \Sigma) \big) + \sup_{\mu \in \R^d} \Big\{ \mu^\top \theta - \gamma \| \mu - \hat \mu \|^2  \Big\} \\[-0.5ex]
		& +\sup_{\Sigma \in \S_+^d} \Big\{ \alpha \sqrt{\theta^\top \Sigma \theta} - \gamma \Tr \big(\Sigma - 2 \big( \hat \Sigma^\half \Sigma \hat \Sigma^\half \big)^\half \big) \Big\} \Bigg\}.
	\end{aligned}
\end{equation}
Strong duality holds because~$r > 0$, which implies that $(\hat\mu,\hat\Sigma)$ constitutes a Slater point for the primal maximization problem. If $\gamma = 0$, then the maximization problems over~$\mu$ and~$\Sigma$ in~\eqref{eq:gelbrich-dual} are unbounded. We may thus restrict~$\gamma$ to be strictly positive. For any fixed~$\gamma > 0$, the maximization problem over~$\mu$ can be solved in closed form. Its optimal value is given by~$\hat \mu^\top \theta + \| \theta \|^2 / (4 \gamma)$.
By introducing an auxiliary variable~$t$, the maximization problem over~$\Sigma$ can be reformulated as
\begin{align}
	\label{eq:Sigma-problem-dual-objective}
	\begin{array}{cl}
		\sup & \alpha t - \gamma \Tr \left(\Sigma - 2 \big( \hat \Sigma^\half \Sigma \hat \Sigma^\half \big)^\half \right) \\
		\st & t \in \R_+, \ \Sigma \in \S_+^d, \; t^2 - \theta^\top \Sigma \theta \leq 0.
	\end{array}
\end{align} 
Note that $t=0$ and $\Sigma= \theta\theta^\top$ form a Slater point for~\eqref{eq:Sigma-problem-dual-objective} because~$\theta\neq 0$. Thus, problem~\eqref{eq:Sigma-problem-dual-objective} admits a strong dual. The variable substitution~$B \gets (\hat \Sigma^{\frac{1}{2}} \Sigma \hat \Sigma^{\frac{1}{2}})^{\frac{1}{2}}$ allows us to reformulate this dual problem more concisely as
\begin{align}
	\label{eq:Sigma-problem-dual-objective2}
	\inf_{\lambda \in \R_+} ~ \sup_{t \in\R_+} ~ \alpha t - \lambda t^2  + \sup_{B \in\S^d_+} ~ \Tr \left(B^2 \Delta_\lambda \right) + 2 \gamma \Tr(B),
\end{align}
where~$\Delta_\lambda = \hat \Sigma^{-\half}(\lambda \theta \theta^\top - \gamma I_d) \hat \Sigma^{-\half}$ for any~$\lambda \geq 0$. Note that~$\Delta_\lambda$ is well-defined because~$\hat \Sigma\succ 0$. Recall now that the standard risk coefficient~$\alpha$ was assumed to be non-negative. If~$\lambda>0$, then the supremum over~$t$ in~\eqref{eq:Sigma-problem-dual-objective2} evaluates to~$\alpha^2/(4\lambda)$. Otherwise, if~$\lambda=0$, then this supremum evaluates to~$+\infty$. From now on, we may thus restrict the outer minimization problem in~\eqref{eq:Sigma-problem-dual-objective2} to strictly positive~$\lambda$. Similarly, if~$\Delta_\lambda \not \prec 0$, then the supremum over~$B$ in~\eqref{eq:Sigma-problem-dual-objective2} evaluates to~$+\infty$. From now on, we may thus restrict the outer minimization problem in~\eqref{eq:Sigma-problem-dual-objective2} to~$\lambda$ that satisfy $\gamma I_d - \lambda \theta \theta^\top \succ 0$. This constraint is equivalent to~$\lambda < \gamma \| \theta \|^{-2}$ and guarantees that~$\Delta_\lambda \prec 0$. As~$\lambda>0$, this in turn implies that~$B^\star=-\gamma\Delta_\lambda^{-1}$ is positive definite and satisfies the first-order optimality condition $B \Delta_\lambda + \Delta_\lambda B + 2\gamma I_d=0$. Note that this optimality condition can be interpreted as a continuous Lyapunov equation, and therefore its solution~$B^\star$ is in fact unique; see, {\em e.g.}, \cite[Theorem~12.5]{hespanha2009linear}. By making the implicit constraints on~$\lambda$ explicit and by evaluating the two suprema over~$t$ and~$B$ analytically, problem~\eqref{eq:Sigma-problem-dual-objective2} can finally be reformulated as
\begin{align*}
	& \inf_{0 < \lambda < \gamma \| \theta \|^{-2}} \;
	\frac{\alpha^2}{4\lambda} + \gamma^2 \Tr \left( \hat \Sigma^\half (\gamma I_d- \lambda \theta \theta^\top)^{-1} \hat \Sigma^{\half} \right)
	\\ &= \inf_{0 < \lambda < \gamma \| \theta \|^{-2}} \; \frac{\alpha^2}{4 \lambda} + \gamma \Tr \left(\hat \Sigma \right) + \frac{\theta^\top \hat \Sigma \theta}{\lambda^{-1} -\| \theta \|^2 / \gamma}=
	\gamma \Tr \left(\hat \Sigma \right) + \frac{\alpha^2}{4} \frac{\| \theta \|^2}{\gamma} + \alpha \sqrt{\theta^\top \hat \Sigma \theta}.        
\end{align*}
Here, the first equality exploits the Sherman-Morrison formula \cite[Corollary~2.8.8]{bernstein2009matrix} to rewrite the inverse matrix, and the second equality is obtained by solving the minimization problem over~$\lambda$ analytically. Indeed, the infimum is attained at the unique solution~$\lambda^\star$ of the first-order condition 
\[
\frac{1}{\lambda}  = \frac{\| \theta \|^2}{\gamma} + \frac{2}{\alpha} \sqrt{\theta^\top \hat \Sigma \theta}
\]
in the interior of the feasible set. In summary, we have solved both embedded subproblems in~\eqref{eq:gelbrich-dual} analytically. Substituting their optimal values into~\eqref{eq:gelbrich-dual}~yields
\begin{align*}
	\sup_{\P \in \cG_r(\hat \mu, \hat \Sigma)} \; \varrho_\P [\theta^\top Z]  
	&= \inf_{\gamma \geq 0} ~ \hat \mu^\top \theta + \alpha \sqrt{\theta^\top \hat \Sigma \theta} +  \gamma r^2 + \frac{1 + \alpha^2}{4} \frac{\| \theta \|^2}{\gamma} \\
	&= \hat \mu^\top \theta + \alpha \sqrt{\theta^\top \hat \Sigma \theta} + r \sqrt{1+ \alpha^2}\, \| \theta \|.
\end{align*}
Here, the second equality is obtained by solving the minimization problem over~$\gamma$ in closed form. We have thus established the desired formula~\eqref{eq:gelbrich-risk-regularization} for~$\hat \Sigma \succ 0$. 

It remains to be shown that~\eqref{eq:gelbrich-risk-regularization} remains valid even if~$\hat \Sigma$ is singular. To this end, use~$J(\hat \Sigma)$ as a shorthand for the Gelbrich risk as a function of~$\hat \Sigma$. By leveraging Berge's maximum theorem~\cite[pp.~115--116]{berge1963topological} and the continuity of the Gelbrich distance (see the discussion after Proposition~\ref{prop:Gelbrich:SDP}), it is easy to show that~$J(\hat \Sigma)$ is continuous on~$\S_+^d$. The claim thus follows by noting that~\eqref{eq:gelbrich-risk-regularization} holds for all $\hat\Sigma\succ 0$, that both sides of~\eqref{eq:gelbrich-risk-regularization} are continuous in~$\hat\Sigma$ and that every~$\hat\Sigma\in\S^d_+$ can be expressed as a limit of positive definite matrices.
\end{proof}

Proposition~\ref{prop:gelbrich-risk-mean-sdev} is due to~\citet{nguyen2021mean}. It shows that, for a broad class of risk measures, the worst-case risk over a Gelbrich ambiguity set reduces to a Markowitz-type mean-variance risk functional with a 2-norm regularization term. We emphasize that the risk measure~$\rho$ enters the resulting optimization model only indirectly through the standard risk coefficient~$\alpha$.

\subsection{Worst-Case Expectations over Kullback-Leibler Ambiguity Sets}
\label{sec:KL-risk}

Consider the worst-case expectation problem
\begin{subequations}
\begin{align}
	\label{eq:primal-KL-risk}
	\sup_{\P \in \cP(\cZ)} \left\{ \E_\P \left[ \ell (Z) \right] \;:\;  \KL(\P, \hat{\P}) \leq r \right\},
\end{align}
which maximizes the expected value of~$\ell(Z)$ over the Kullback-Leibler ambiguity set of all distributions supported on~$\cZ$ whose Kullback-Leibler divergence with respect to~$\hat\P\in\cP(\cZ)$ is at most~$r\geq 0$. The Kullback-Leibler ambiguity set is a $\phi$-divergence ambiguity set of the form~\eqref{eq:phi-divergence-ambiguity-set}, where~$\phi$ satisfies~$\phi(s)=s\log(s)-s+1$ for all~$s\geq 0$.  As~$\phi^\infty(1)=+\infty$, we have~$\KL(\P, \hat{\P})=\infty$ unless~$\P\ll\hat \P$. Hence, problem~\eqref{eq:primal-KL-risk} maximizes only over distributions~$\P$ that are absolutely continuous with respect to~$\hat\P$. Note that~$\phi^*(t)=e^{t}-1$ for all~$t\in\R$. By Theorem~\ref{thm:duality:phi} and the definition of the perspective function, the problem dual to~\eqref{eq:primal-KL-risk} is thus given~by
\begin{align}
	\label{eq:dual-KL-risk}
	\inf_{\lambda_0 \in\R,\,\lambda \in \R_+} ~ \lambda_0+ \lambda (r-1) + \E_{\hat \P} \left[ \lambda \exp \left( \frac{\ell(Z) - \lambda_0}{\lambda} \right) \right].
\end{align}
\end{subequations}
The problems~\eqref{eq:primal-KL-risk} and~\eqref{eq:dual-KL-risk} can be solved in closed form if the loss function~$\ell$ is linear and the nominal distribution~$\hat\P$ is Gaussian.

\begin{proposition}[Worst-Case Expectations over KL Ambiguity Sets]
\label{prop:KL-analytical}
Suppose that~$\cZ=\R^d$, $\hat\P\in\cP(\cZ)$ is a normal distribution with mean~$\hat\mu\in\R^d$ and covariance matrix~$\hat\Sigma\in\S^d_{++}$, and~$r> 0$. Suppose also that~$\ell$ is linear, that is, there exists~$\theta\in\R^d$ with $\ell(z)=\theta^\top z$ for all~$z\in\cZ$. Then, the primal problem~\eqref{eq:primal-KL-risk} is solved by the normal distribution~$\P^\star$ with mean~$\hat\mu +(2r)^\half \hat \Sigma\theta/ (\theta^\top\hat\Sigma\theta)^\half$ and covariance matrix~$\hat\Sigma$. The dual problem~\eqref{eq:dual-KL-risk} is solved by~$(\lambda_0^\star,\lambda^\star)$, where~$\lambda^\star= (\theta^\top\hat\Sigma\theta)^\half/(2r)^\half$ and
\[
\lambda_0^\star = \lambda^\star \log \E_{\hat\P} \left[ \exp \left( \ell(Z)/\lambda^\star\right) \right].
\]
The optimal values of~\eqref{eq:primal-KL-risk} and~\eqref{eq:dual-KL-risk} are both equal to~$\hat\mu^\top\theta + (2r)^\half  (\theta^\top\hat\Sigma\theta)^\half$.    
\end{proposition}

\begin{proof}
Focus first on the dual problem~\eqref{eq:dual-KL-risk}, and fix any~$\lambda\geq 0$. Then, the partial minimization problem over~$\lambda_0$ is solved by
\[
\lambda_0^\star(\lambda) = \lambda \log \E_{\hat \P} \left[ \exp \left( \ell(Z)/\lambda \right) \right].
\]
Substituting this parametric minimizer back into~\eqref{eq:dual-KL-risk} shows that the optimal value of the dual problem~\eqref{eq:dual-KL-risk} is given by
\begin{align*}
	\inf_{\lambda \in \R_+} ~ \lambda r + \lambda \log \E_{\hat \P} \left[ \exp \left( \frac{\ell(Z)}{\lambda}\right) \right] & = \inf_{\lambda \in\R_+} ~ \lambda r + \hat\mu^\top\theta + \frac{1}{2\lambda} \theta^\top\hat\Sigma\theta \\ & = \hat\mu^\top\theta +  (2r)^\half (\theta^\top\hat\Sigma\theta)^\half,
\end{align*}
where the first equality exploits the linearity of~$\ell$, the normality of~$\hat\P$ and the formula for the expected value of a log-normal distribution. The second equality holds because the minimization problem over~$\lambda\geq 0$ is solved by~$\lambda^\star= (\theta^\top\hat\Sigma\theta)^\half/(2r)^\half$. Next, define~$\P^\star\in\cP(\cZ)$ as the normal distribution with mean 
$\mu^\star=\hat\mu +\hat \Sigma\theta/\lambda^\star$ and covariance matrix~$\Sigma^\star = \hat\Sigma$. Comparing the density functions of~$\hat\P$ and~$\P^\star$ shows~that
\[
\frac{\diff\P^\star}{\diff\hat\P}(z) =\exp\left( \frac{\theta^\top(z-\hat \mu)}{\lambda^\star} - \frac{\theta^\top\hat\Sigma\theta}{2(\lambda^\star)^2}\right) \quad \forall z\in\cZ.
\]
By Definition~\ref{def:KL}, we thus obtain
\begin{align*}
	\KL(\P^\star, \hat{\P}) = \int_\cZ \log\left( \frac{\diff\P^\star}{\diff\hat\P}(z) \right) \diff\P^\star(z) =\frac{\theta^\top\hat\Sigma\theta}{2(\lambda^\star)^2}=r,
\end{align*}
where the second and the third equalities follow readily from our formula for the Radon-Nikodym derivative $\diff\P^\star/\diff\hat\P$ and from basic algebra, respectively. Hence, $\P^\star$ is feasible in~\eqref{eq:dual-KL-risk}. In addition, its objective function value is given by
\[
\E_{\P^\star}[\ell(Z)] = \theta^\top\mu^\star = \hat\mu^\top\theta +  (2r)^\half (\theta^\top\hat\Sigma\theta)^\half.
\]
As the objective function values of~$\P^\star$ and~$(\lambda_0^\star, \lambda^\star)$ with $\lambda^\star_0=\lambda^\star_0(\lambda^\star)$ match, weak duality as established in Theorem~\ref{thm:duality:phi} implies that~$\P^\star$ is primal optimal and that~$(\lambda_0^\star, \lambda^\star)$ is dual optimal. This observation completes the proof.
\end{proof}

Proposition~\ref{prop:KL-analytical} is due to \citet{hu2013kullback}. It is also reminiscent of risk-sensitive control theory \citep{hansen2008robustness}. In this stream of literature, a fictitious adversary may perturb the distribution of the exogenous noise terms of an optimal control problem {\em arbitrarily} but incurs a penalty equal to the Kullback-Leibler divergence with respect to a Gaussian baseline model.

\subsection{Worst-Case Expectations over Total Variation Balls}
\label{sec:TV-bound}
Consider the worst-case expectation problem
\begin{subequations}
\begin{align}
	\label{eq:primal-TV-risk}
	\sup_{\P \in \cP(\cZ)} \left\{ \E_\P \left[ \ell (Z) \right] \;:\;  \TV(\P,\hat\P)\leq r \right\},
\end{align}
which maximizes the expected value of~$\ell(Z)$ over a total variation ball of radius~$r\in [0,1]$ around~$\hat\P\in\cP(\cZ)$. Recall from Section~\ref{section:TV} that the total variation distance is a $\phi$-divergence and that the underlying entropy function satisfies $\phi(s)=\half|s-1|$ for all~$s\geq 0$ and $\phi(s)=\infty$ for all~$s<0$. Recall also that the total variation distance between two distributions is bounded above by~$1$ and that this bound is attained if the two distributions are mutually singular. An elementary calculation reveals that the conjugate entropy function satisfies $\phi^*(t)=\max\{t+\half,0\}-\half$ if $t\le \half$ and $\phi^*(t)=+\infty$ if $t> \half$. By Theorem~\ref{thm:duality:phi}, the problem dual to~\eqref{eq:primal-TV-risk} is thus given~by 
\begin{align}
	\label{eq:dual-TV-risk}
	\begin{array}{cl}
		\displaystyle \inf_{\lambda_0\in\R,\, \lambda\in\R_+} & \displaystyle \lambda_0+\lambda \left(r-\half\right) + \E_{\hat\P} \left[\max \left\{\ell(Z) - \lambda_0+\frac{\lambda}{2}, 0 \right\} \right] \\[2ex]
		\st &  \displaystyle \lambda_0+ \lambda/2 \geq \sup_{z\in\cZ} \ell(z).
	\end{array}
\end{align}
\end{subequations}
The problems~\eqref{eq:primal-TV-risk} and~\eqref{eq:dual-TV-risk} can be solved in closed form if~$\cZ$ is compact.

\begin{proposition}[Worst-Case Expectations over Total Variation Balls]
\label{prop:TV-analytical}
Suppose that~$\cZ\subseteq \R^d$ is compact, $\hat\P\in\cP(\cZ)$ and~$r\in(0,1)$, and define~$\beta_r=1-r$. In addition, assume that $\E_{\hat\P}[\ell(Z)]>-\infty$ and $\ell$ is upper semicontinuous. Then, the optimal values of~\eqref{eq:primal-TV-risk} and~\eqref{eq:dual-TV-risk} are both equal to
\begin{align}
	\label{eq:wctv-sup-cvar}
	(1-\beta_r)\cdot \sup_{z\in\cZ} \ell(z) + \beta_r\cdot \beta_r\CVaR_{\hat \P}\left[\ell(Z) \right]. 
\end{align}
\end{proposition}

The proof of Proposition~\ref{prop:TV-analytical} will reveal that~\eqref{eq:primal-TV-risk} and~\eqref{eq:dual-TV-risk} are both solvable. Indeed, we will construct optimal solutions~$\P^\star$ and~$(\lambda_0^\star, \lambda^\star)$ for~\eqref{eq:primal-TV-risk} and~\eqref{eq:dual-TV-risk}, respectively. A precise description of these optimizers is cumbersome and thus omitted from the proposition statement. If the loss~$\ell(Z)$ has a continuous distribution under~$\hat\P$, however, then~$\P^\star$ admits a simpler and more intuitive description. Indeed, in this case, $\P^\star$ is obtained from~$\hat \P$ by shifting the probability mass of all outcomes~$z\in\cZ$ associated with a high loss~$\ell(z)\geq \beta_r\VaR_{\hat \P}[\ell(Z)]$ to some outcome $z\in\cZ$ associated with the highest possible loss $\ell(z)=\max_{z'\in\cZ}\ell(z')$.


\begin{proof}[Proof of Proposition~\ref{prop:TV-analytical}]
For ease of notation, set $\overline\ell=\sup_{z\in\cZ}\ell(z)$. Focus first on the dual problem~\eqref{eq:dual-TV-risk}, and fix any~$\lambda\geq 0$. Note that the dual objective function is non-decreasing in~$\lambda_0$. The partial minimization problem over~$\lambda_0$ is therefore solved by $\lambda_0^\star(\lambda)=\overline\ell-\lambda / 2$. Substituting this parametric minimizer back into~\eqref{eq:dual-TV-risk} shows that the optimal value of the dual problem is given by
\begin{align*}
	&\overline\ell + \inf_{\lambda\in\R_+} ~ \lambda (r-1) + \E_{\hat\P} \big[\max \big\{\ell( Z) - \overline\ell +\lambda, 0 \big\} \big] \\
	& \quad = r\,\overline\ell + (1-r) \inf_{\tau\leq\overline\ell} ~ \tau + (1-r)^{-1} \E_{\hat\P} \big[\max \big\{\ell(Z) - \tau, 0 \big\} \big],
\end{align*}
where the equality follows from the substitution $\tau\leftarrow \overline \ell-\lambda$. By Definition~\ref{def:cvar}, the infimum over~$\tau$ evaluates to $\beta_r\CVaR_{\hat \P}[\ell(Z)]$ with $\beta_r=1-r$. Recall that this infimum is attained by $\tau^\star=\beta_r\VaR_{\hat \P}[\ell(Z)]$, which is bounded above by~$\overline\ell$. In summary, we have thus shown that the optimal value of problem~\eqref{eq:dual-TV-risk} equals 
\begin{align*}
	(1-\beta_r)\cdot \overline \ell + \beta_r\cdot \beta_r\CVaR_{\hat \P}\left[\ell(Z) \right]. 
\end{align*}
To construct a primal maximizer, assume first that $\hat\P(\ell(Z)<\overline\ell)\leq r$, which implies that $\beta_r\CVaR_{\hat \P}[\ell(Z)] = \overline \ell$. Thus, the optimal value of the dual problem~\eqref{eq:dual-TV-risk} simplifies to~$\overline\ell$, which is attained by any distribution~$\P^\star$ that is obtained from~$\hat\P$ by moving all probability mass from $\{z\in\cZ:\ell(z)<\overline\ell\}$ to $\{z\in\cZ:\ell(z)=\overline\ell\}$.

Next, assume that $\hat\P(\ell(Z)<\overline\ell)> r$, which implies that $\beta_r\VaR_{\hat \P}[\ell(Z)]<\overline\ell$. In this case, we partition~$\cZ$ into the following four subsets.
\begin{align*}
	\cZ_1&=\big\{z\in\cZ: \beta_r\VaR_{\hat \P}[\ell(Z)]>\ell(z)\big\}\\
	\cZ_2& =\big\{z\in\cZ: \overline\ell>\ell(z)=\beta_r\VaR_{\hat \P}[\ell(Z)]\big\}\\
	\cZ_3& =\big\{z\in\cZ: \overline\ell>\ell(z)>\beta_r\VaR_{\hat \P}[\ell(Z)]\big\}\\
	\cZ_4&=\big\{z\in\cZ: \overline\ell=\ell(z)\big\}
\end{align*}
Note that~$\cZ_1$ and~$\cZ_3$ can be empty, whereas~$\cZ_2$ and $\cZ_4$ must be non-empty. We also define~$\hat\P_i$ as the nominal distribution~$\hat\P$ conditioned on the event~$Z\in\cZ_i$ for all~$i\in[4]$, and we define $\U_{\cZ_4}$ as the uniform distribution on~$\cZ_4$. Next, we set
\begin{align*}
	\P^\star & = \left(\beta_r-\hat\P(Z\in\cZ_3)-\hat\P(Z\in\cZ_4)\right)\cdot \hat\P_2 + \\ & \hspace{2cm}\hat\P(Z\in\cZ_3)\cdot \hat \P_3 + \hat\P(Z\in\cZ_4)\cdot \hat \P_4+ (1-\beta_r)\cdot \U_{\cZ_4}.
\end{align*}
Thus, $\P^\star$ is a mixture of four probability distributions. As the non-negative mixture probabilities sum to~$1$, $\P^\star$ is a probability distribution. Using~$\rho=\hat \P+\U_{\cZ_4}$ as a dominating measure for $\hat\P$ and $\P^\star$ and recalling that~$\phi(s)=\half|s-1|$ if~$s\geq 0$, we~find
\begin{align*}
	& \TV(\P^\star,\hat\P)= \D_\phi(\P^\star,\hat\P) 
	= \half\sum_{i=1}^4\int_{\cZ_i} \left| \frac{\diff \P^\star}{\diff \rho}(z)-\frac{\diff \hat{\P}}{\diff \rho}(z)\right| \diff \rho(z) \\
	&= \hat\P(Z\in\cZ_1) + \left(\hat\P(Z\in\cZ_2) +\hat\P(Z\in\cZ_3) +\hat\P(Z\in\cZ_4) -\beta_r\right) + 0 + (1-\beta_r) = r,
\end{align*}
where the third equality follows from the definition of~$\P^\star$ and the relation
\[
\hat\P(Z\in\cZ_2) +\hat\P(Z\in\cZ_3) +\hat\P(Z\in\cZ_4)= \hat\P \left(\ell(Z)\geq \beta_r\VaR_{\hat \P}[\ell(Z)] \right)\geq \beta_r,
\]
and the last equality follows from the definition of~$\beta_r$. Thus, $\P^\star$ is feasible in~\eqref{eq:primal-TV-risk}. In addition, the objective function value of~$\P^\star$ in~\eqref{eq:primal-TV-risk} amounts to 
\begin{align*}
	& \E_{\P^\star} \left[ \ell (Z) \right] = (1-\beta_r)\cdot \overline \ell \\[-0.5ex]
	& \qquad + \E_{\hat\P} \left[ \left. \ell (Z) \right| \ell(Z) > \beta_r\VaR_{\hat \P}[\ell(Z)]  \right] \cdot \hat\P\left( \ell(Z) > \beta_r\VaR_{\hat \P}[\ell(Z)] \right)\\
	& \qquad + \E_{\hat\P} \left[ \left. \ell (Z) \right| \ell(Z) = \beta_r\VaR_{\hat \P}[\ell(Z)]  \right] \cdot \left(\beta_r- \hat\P\left( \ell(Z) > \beta_r\VaR_{\hat \P}[\ell(Z)] \right)\right) \\
	& = (1-\beta_r)\cdot \overline \ell + \beta_r\cdot \beta_r\CVaR_{\hat \P}\left[\ell(Z) \right].
\end{align*}
Here, the second equality follows from \cite[Theorem~4.47 \& Remark~4.48]{follmer2008stochastic}. Note that if the marginal distribution of~$\ell(Z)$ is continuous under~$\hat \P$, then the above derivation simplifies. Indeed, in this case we have
\[
\hat\P\left( \ell(Z) > \beta_r\VaR_{\hat \P}[\ell(Z)] \right)=\beta_r
\]
and
\[
\E_{\hat\P} \left[ \left. \ell (Z) \right| \ell(Z) > \beta_r\VaR_{\hat \P}[\ell(Z)]  \right] = \beta_r\CVaR_{\hat \P}\left[\ell(Z) \right].
\]
Irrespective of~$\hat \P$, the objective function value of~$\P^\star$ in~\eqref{eq:primal-TV-risk} matches the optimal value of~\eqref{eq:dual-TV-risk}. Weak duality as established in Theorem~\ref{thm:duality:phi} thus implies that~$\P^\star$ solves the primal problem~\eqref{eq:primal-TV-risk}. This observation completes the proof. 
\end{proof}

\citet{jiang2018risk} and \citet{shapiro2017distributionally} study a variant of problem~\eqref{eq:primal-TV-risk} that maximizes over a {\em restricted} total variation ball. Thus, they additionally impose~$\P\ll\hat\P$ in~\eqref{eq:primal-TV-risk}. The supremum of the resulting restricted problem amounts to
\[
(1-\beta_r)\cdot \text{ess\,sup}_{\hat\P}[\ell(Z)] + \beta_r\cdot \beta_r\CVaR_{\hat \P}\left[\ell(Z) \right],
\]
which may be strictly smaller than~\eqref{eq:wctv-sup-cvar}. If additionally~$\ell(Z)$ has a continuous marginal distribution under~$\hat\P$, then the supremum is no longer attained.

\subsection{Worst-Case Expectations over L\'evy-Prokhorov Balls}
\label{sec:levyprokhorov}
Consider the worst-case expectation problem
\begin{subequations}
\begin{align}
	\label{eq:primal-LP-risk}
	\sup_{\P \in \cP(\cZ)} \left\{ \E_\P \left[ \ell (Z) \right] \;:\;  \LP(\P,\hat\P)\leq r \right\},
\end{align}
which maximizes the expected value of~$\ell(Z)$ over a L\'evy-Prokhorov ball of radius~$r\in [0,1]$ around~$\hat\P\in\cP(\cZ)$. We assume here that the L\'evy-Prokhorov distance is induced by a norm~$\|\cdot\|$ on~$\R^d$. By Proposition~\ref{prop:dual-levy-prokhorov}, the L\'evy-Prokhorov ball of radius~$r\in(0,1)$ coincides with the optimal transport ambiguity set
\[
\cP = \left\{ \P \in \cP(\cZ) \, : \, \OT_{c_r}(\P, \hat{\P}) \leq r \right\},
\]
where the transportation cost function~$c_r$ is defined through $c_r(z,\hat z)= \ds 1_{\|z - \hat z\|> r}$. Theorem~\ref{thm:duality:OT} thus implies that the problem dual to~\eqref{eq:primal-LP-risk} is given by
\begin{align}
	\label{eq:dual-LP-risk}
	\inf_{\lambda\in\R_+} \left\{ \lambda r+ \E_{\hat \P}\left[\sup_{z\in\cZ} \ell(z) -\lambda c_r(z,\hat Z) \right]\right\}
\end{align}
\end{subequations}
whenever~$\ell$ is upper semicontinuous. If~$\cZ$ is compact, then we can leverage Proposition~\ref{prop:TV-analytical} to solve the problems~\eqref{eq:primal-LP-risk} and~\eqref{eq:dual-LP-risk} in closed form.

\begin{proposition}[Worst-Case Expectations over L\'evy-Prokhorov Balls]
\label{prop:LP-analytical}
Suppose that~$\cZ\subseteq \R^d$ is compact, $\hat\P\in\cP(\cZ)$ and~$r\in(0,1)$, and define~$\beta_r=1-r$. In addition, assume that $\E_{\hat\P}[\ell(Z)]>-\infty$ and $\ell$ is upper semicontinuous. Then, the optimal values of~\eqref{eq:primal-LP-risk} and~\eqref{eq:dual-LP-risk} are both equal to
\begin{align}
	\label{eq:wclp-sup-cvar}
	(1-\beta_r)\cdot \sup_{z\in\cZ} \ell(z) + \beta_r\cdot \beta_r\CVaR_{\hat \P}\left[\ell_r(\hat Z) \right],
\end{align}
where $\ell_r(\hat z)=\sup_{z\in\cZ}\{\ell(z):\|z-\hat z\|\leq r\}$ is an adversarial loss function that assigns each~$\hat z\in\cZ$ the worst-case loss in the $r$-neighborhood of~$\hat z$.
\end{proposition}

The proof of Proposition~\ref{prop:LP-analytical} will reveal that~\eqref{eq:primal-LP-risk} and~\eqref{eq:dual-LP-risk} are both solvable. However, a precise description of the respective optimizers is cumbersome and thus omitted from the proposition statement. Note that the adversarial loss function~$\ell_r$ inherits upper semicontinuity from~$\ell$ thanks to \cite[Theorem~2, p.~116]{berge1963topological}. The following lemma is needed in the proof of Proposition~\ref{prop:LP-analytical}.

\begin{lemma}
\label{lem:PL-to-TV}
Assume that~$\cZ\subseteq\R^d$ is compact, $\ell$ is upper semicontinuous, $\hat z\in\cZ$ and~$r,\lambda\geq 0$. Then, the following identity holds.
\begin{align*}
	\sup_{z\in\cZ} \left\{\ell(z)-\lambda \cdot \ds 1_{\|z-\hat z\|> r}\right\} \;=\; \sup_{z\in\cZ}\left\{ \ell_r(z)-\lambda \cdot \ds 1_{z\neq \hat z} \right\}.
\end{align*}
\end{lemma}
\begin{proof}
For ease of notation we introduce two auxiliary functions~$f$ and~$g$ from~$\cZ$ to~$\overline\R$, which are defined through $f(z)=\ell(z)-\lambda \cdot \ds 1_{\|z-\hat z\|> r}$ and $g(z)=\ell_r(z)-\lambda \cdot \ds 1_{z\neq \hat z}$ for all~$z\in\cZ$. Note that both~$f$ and~$g$ are upper semicontinuous.

First, select~$z^\star\in\arg\max_{z\in\cZ} f(z)$, which exists because~$\cZ$ is compact and~$f$ is upper semicontinuous. If~$\|z^\star-\hat z\|> r$, then the definition of~$\ell_r$ implies that
\[
\sup_{z\in\cZ}f(z)=f(z^\star)=\ell(z^\star)-\lambda\leq \ell_r(z^\star)-\lambda =g(z^\star)\leq\sup_{z\in\cZ} g(z).
\]
On the other hand, if~$\|z-\hat z\|\leq r$, then
\[
\sup_{z\in\cZ}f(z)=f(z^\star)=\ell(z^\star)\leq \ell_r(\hat z)=g(\hat z)\leq\sup_{z\in\cZ} g(z).
\]
Next, select~$\tilde z\in\arg\max_{z\in\cZ} g(z)$. If~$\tilde z\neq\hat z$, then with $z^\star\in\arg\max_{z\in\cZ} \ell(z)$ we have
\[
\sup_{z\in\cZ}g(z)=g(\tilde z)=\ell_r(\tilde z)-\lambda\leq \ell(z^\star)-\lambda\leq f(z^\star) = \sup_{z\in\cZ}f(z),
\]
where the inequalities follow from the definition of~$z^\star$ and the non-negativity of~$\lambda$. Conversely, if~$\tilde z=\hat z$, then with $z_r^\star\in\arg\max_{z'\in\cZ} \{\ell(z'):\|z'-\hat z\|\leq r\}$ we have
\[
\sup_{z\in\cZ}g(z)=g(\tilde z)=\ell_r(\hat z) =\ell(z^\star_r) = f(z_r^\star)= \sup_{z\in\cZ}f(z)
\]
Thus, the claim follows.
\end{proof}

\begin{proof}[Proof of Proposition~\ref{prop:LP-analytical}]
Lemma~\ref{lem:PL-to-TV} allows us to reformulate the dual problem~\eqref{eq:dual-LP-risk} in terms of the adversarial loss function~$\ell_r$ as
\begin{align}
	\label{eq:dual-LP-risk-reformulated}
	\inf_{\lambda\in\R_+} \left\{ \lambda r+ \E_{\hat \P}\left[\sup_{z\in\cZ} \ell_r(z)-\lambda \cdot \ds 1_{z\neq \hat z} \right]\right\}.
\end{align}
As $r>0$, $\cZ$ is compact and $\ell_r$ is upper semicontinuous, Theorem~\ref{thm:duality:OT} implies that~\eqref{eq:dual-LP-risk-reformulated} is the strong dual of a problem that maximizes the expected value of the adversarial loss function~$\ell_r$ over an optimal transport ambiguity set corresponding to the transportation cost function~$c_0(z,\hat z)=\ds 1_{z\neq\hat z}$. Its optimal value thus matches
\begin{align*}
	\sup_{\P \in \cP(\cZ)} \left\{ \E_\P \left[ \ell_r (Z) \right] \;:\;  \OT_{c_0} (\P,\hat\P)\leq r \right\} \;=\; \sup_{\P \in \cP(\cZ)} \left\{ \E_\P \left[ \ell_r (Z) \right] \;:\;  \TV(\P,\hat\P)\leq r \right\},
\end{align*}
where the equality holds because $\TV = \OT_{c_0}$ as shown in Proposition~\ref{prop:TV=OT}. Since $\sup_{z\in\cZ} \ell_r(z)=\sup_{z\in\cZ} \ell(z)=\overline\ell$, Proposition~\ref{prop:TV-analytical} readily implies that the supremum of the resulting maximization problem over a total variation ball is given by
\begin{align*}
	(1-\beta_r)\cdot \overline \ell + \beta_r\cdot \beta_r\CVaR_{\hat \P}\left[\ell_r(\hat Z) \right], 
\end{align*}
Assume now that $\psi:\cZ\to\cZ$ is a Borel measurable function satisfying
\[
\psi(\hat z) \in \arg\max_{z\in\cZ} \left\{ \ell(z) : \|z - \hat z\|\leq r\right\} \quad \forall \hat z\in\cZ,
\]
which exists thanks to \cite[Corollary~14.6 and Theorem~14.37]{rockafellar2009variational}, and define~$\hat\P_\psi=\hat\P\circ \psi^{-1}$ as the pushforward distribution of~$\hat\P$ under~$\psi$. Next, we construct a primal maximizer under the assumption that $\hat\P_\psi(\ell(Z)<\overline\ell)> r$. To this end, we partition~$\cZ$ into the following four subsets.
\begin{align*}
	\cZ_1&=\big\{z\in\cZ: \beta_r\VaR_{\hat \P_\psi}[\ell(\hat Z)]>\ell(z)\big\}\\
	\cZ_2& =\big\{z\in\cZ: \overline\ell>\ell(z)=\beta_r\VaR_{\hat \P_\psi}[\ell(\hat Z)]\big\}\\
	\cZ_3& =\big\{z\in\cZ: \overline\ell>\ell(z)>\beta_r\VaR_{\hat \P_\psi}[\ell(\hat Z)]\big\}\\
	\cZ_4&=\big\{z\in\cZ: \overline\ell=\ell(z)\big\}
\end{align*}
We also define~$\hat\P_{i}$ as the distribution~$\hat\P_\psi$ conditioned on the event~$\hat Z\in\cZ_i$ for all~$i\in[4]$, and we define $\U_{\cZ_4}$ as the uniform distribution on~$\cZ_4$. Next, we set
\begin{align*}
	\P^\star & = \left(\beta_r-\hat\P_\psi(\hat Z\in\cZ_3)-\hat\P_\psi(\hat Z\in\cZ_4)\right)\cdot \hat\P_{2} + \\ & \hspace{2cm}\hat\P_\psi(\hat Z\in\cZ_3)\cdot \hat \P_{3} + \hat\P_\psi(\hat Z\in\cZ_4)\cdot \hat \P_{4}+ (1-\beta_r)\cdot \U_{\cZ_4}.
\end{align*}
Note that~$\P^\star$ is constructed as in the proof of Proposition~\ref{prop:TV-analytical}, the only difference being that~$\hat\P$ is now replaced with its pushforward distribution~$\hat\P_\psi$. We then find
\begin{align*}
	\LP(\P^\star,\hat\P) & \leq \max \left\{\OT_{c_r}(\P^\star,\hat\P), r \right\}  \\ 
	& \leq \max \left\{ \OT_{c_r}(\P^\star,\hat\P_\psi) + \OT_{c_r}(\hat\P_\psi,\hat \P) , r \right\} \\
	& \leq  \max \left\{\TV(\P^\star,\hat\P_\psi) , r \right\}=r,
\end{align*}
where the first inequality follows from Proposition~\ref{prop:dual-levy-prokhorov}, and the second inequality holds because~$c_r$ is a pseudo-metric on~$\cZ$, which implies that~$\OT_{c_r}$ is a pseudo-metric on~$\cP(\cZ)$ and thus satisfies the triangle inequality. The third inequality holds because $\OT_{c_r}(\hat\P_\psi,\hat \P)=0$ and because~$c_0(z,\hat z)\geq c_r(z,\hat z)$ for all~$z,\hat z\in\cZ$, which implies that $\OT_{c_r}(\P^\star,\hat\P_\psi)\leq \TV(\P^\star,\hat\P_\psi)$. Finally, the equality follows from the proof of Proposition~\ref{prop:TV-analytical}, which ensures that $\TV(\P^\star,\hat\P_\psi)=r$. We also have
\begin{align*}
	\E_{\P^\star}\left[ \ell(Z) \right] & = (1-\beta_r)\cdot \overline \ell + \beta_r\cdot \beta_r\CVaR_{\hat \P_\psi }\left[\ell(\hat Z) \right] \\
	& = (1-\beta_r)\cdot \overline \ell + \beta_r\cdot \beta_r\CVaR_{\hat \P}\left[\ell(\psi(\hat Z)) \right].
\end{align*}
where the two equalities follow again from the proof of Proposition~\ref{prop:TV-analytical} and from the measure-theoretic change of variables formula, respectively. As $\ell(\psi(\hat z))=\ell_r(\hat z)$ for every~$\hat z\in\cZ$, the objective function value of~$\P^\star$ in~\eqref{eq:primal-LP-risk} matches the optimal value of the dual problem~\eqref{eq:dual-LP-risk}. Weak duality as established in Theorem~\ref{thm:duality:OT} thus implies that~$\P^\star$ solves the primal problem~\eqref{eq:primal-LP-risk}. If $\hat\P_\psi(\ell(Z)<\overline\ell)\leq r$, the construction of a primal maximizer is simpler and thus omitted for brevity. 
\end{proof}

The results of this section were first obtained by \citet{bennouna2023holistic} under the assumption that the nominal distribution~$\hat\P$ is discrete.

\subsection{Worst-Case Expectations over $\infty$-Wasserstein Balls}
\label{sec:infty:wasserstein-risk}
Consider the worst-case expectation problem
\begin{subequations}
\begin{align}
	\label{eq:primal-infty-wasserstein-risk}
	\sup_{\P \in \cP(\cZ)} \left\{ \E_\P \left[ \ell (Z) \right] \;:\;  \W_\infty(\P,\hat\P) \leq r \right\},
\end{align}
which maximizes the expected value of~$\ell(Z)$ over an $\infty$-Wasserstein ball of radius~$r\in \R_+$ around~$\hat\P\in\cP(\cZ)$. We assume here that the $\infty$-Wasserstein distance is induced by a given norm~$\|\cdot\|$ on~$\R^d$. Recall from Proposition~\ref{prop:dual-W_infty} that the $\infty$-Wasserstein ambiguity set coincides with the optimal transport ambiguity set
\begin{align*}
	\cP = \left\{ \P \in \cP(\cZ) \, : \, \OT_{c_r}(\P, \hat{\P}) \leq 0 \right\},
\end{align*}
where the transportation cost function~$c_r$ is defined through $c_r(z,\hat z)= \ds 1_{\|z - \hat z\|> r}$. We emphasize that, while the radius of the $\infty$-Wasserstein ball under consideration is~$r$, the radius of the corresponding optimal transport ambiguity set~$\cP$ is~$0$. Theorem~\ref{thm:duality:OT} thus implies that the problem dual to~\eqref{eq:primal-infty-wasserstein-risk} is given by
\begin{align}
	\label{eq:dual-infty-wasserstein-risk}
	\inf_{\lambda\in \R_+} \, \E_{\hat \P}\left[\sup_{z\in\cZ} \ell(z) -\lambda c_r(z,\hat Z) \right]
\end{align}
\end{subequations}
whenever $\ell$ is upper semicontinuous. If~$\cZ$ is compact, then the problems~\eqref{eq:primal-infty-wasserstein-risk} and~\eqref{eq:dual-infty-wasserstein-risk} can be solved in closed form.

\begin{proposition}[Worst-Case Expectations over $\infty$-Wasserstein Balls]
\label{prop:infty-wasserstein-analytical}
Suppose that~$\cZ\subseteq \R^d$ is compact, $\hat\P\in\cP(\cZ)$, $r\in\R_+$, $\E_{\hat\P}[\ell(\hat Z)]>-\infty$ and~$\ell$ is upper semicontinuous. Define the adversarial loss function $\ell_r(\hat z)=\sup_{z\in\cZ}\{\ell(z):\|z-\hat z\|\leq r\}$ as in Proposition~\ref{prop:LP-analytical}, and let $\psi:\cZ\to\cZ$ be a Borel function that satisfies
\[
\psi(\hat z) \in \arg\max_{z\in\cZ} \left\{ \ell(z) : \|z - \hat z\|\leq r\right\} \quad \forall \hat z\in\cZ.
\]
Then, the primal problem~\eqref{eq:primal-infty-wasserstein-risk} is solved by $\P^\star=\hat\P\circ\psi^{-1}$. In addition, the optimal values of~\eqref{eq:primal-infty-wasserstein-risk} and~\eqref{eq:dual-infty-wasserstein-risk} are both equal to $\E_{\hat \P} [ \ell_r(\hat Z)]$.
\end{proposition}

\begin{proof}
Note that the Borel function~$\psi$ exists thanks to \cite[Corollary~14.6 and Theorem~14.37]{rockafellar2009variational}. This ensures that the pushforward distribution $\P^\star=\hat\P\circ \psi^{-1}$ is well-defined. Note also that~$\P^\star$ is feasible in~\eqref{eq:primal-infty-wasserstein-risk} because
\begin{align*}
	\W_\infty(\P^\star,\hat\P) = \inf \left\{ r' \geq 0: \OT_{c_{r'}}(\P^\star, \hat \P) \leq 0 \right\} \leq r,
\end{align*}
where the equality follows from Proposition~\ref{prop:dual-W_infty} with $d(z,\hat z)=\|z-\hat z\|$, and the inequality holds because $\OT_{c_r}(\P^\star, \hat \P) = 0$. We also have
\begin{align*}
	\E_{\P^\star} \left[ \ell(Z) \right] 
	= \E_{\hat \P} \left[ \ell(\psi(Z)) \right]
	= \E_{\hat \P} \left[ \ell_{r}(Z) \right].
\end{align*}
Next, note that $\sup_{z\in\cZ} \ell(z) -\lambda c_r(z,\hat z)$ is non-increasing in~$\lambda$ for any fixed~$\hat z\in\cZ$. Also, it is uniformly bounded above by $\sup_{z\in\cZ}\ell(z)$, which is a finite constant thanks to the compactness of~$\cZ$ and the upper semicontinuity of~$\ell$. By the monotone convergence theorem, the optimal value of the dual problem~\eqref{eq:dual-infty-wasserstein-risk} thus satisfies
\begin{align*}
	\inf_{\lambda\in\R_+} \, \E_{\hat \P}\left[\sup_{z\in\cZ} \ell(z) -\lambda c_r(z,\hat Z) \right] = \E_{\hat \P}\left[\inf_{\lambda\in\R_+} \,  \sup_{z\in\cZ} \ell(z) -\lambda c_r(z,\hat Z) \right] = \E_{\hat \P} \left[ \ell_{r}(\hat Z) \right],
\end{align*}
where the second equality holds because~$\cZ$ is compact. Weak duality as established in Theorem~\ref{thm:duality:OT} thus implies that~$\P^\star$ solves the primal problem~\eqref{eq:primal-infty-wasserstein-risk}.
\end{proof}

Proposition~\ref{prop:infty-wasserstein-analytical} shows that the worst-case expectation of the original loss~$\ell(Z)$ with respect to an $\infty$-Wasserstein ball coincides with the crisp expectation of the adversarial loss $\ell_r(\hat Z)$ with respect to the nominal distribution~$\hat \P$. This result was first discovered by \citet{gao2017wasserstein} for discrete nominal distributions and later extend by \citet{gao2024wasserstein} to general nominal distributions. The loss function~$\ell_r$ is routinely used in machine learning for the adversarial training of neural networks \citep{szegedy2014intriguing,ian15adversarial}. Proposition~\ref{prop:infty-wasserstein-analytical} thus reveals an intimate connection between adversarial training and distributionally robust optimization with respect to an $\infty$-Wasserstein ambiguity set. This connection has been further explored in the context of adversarial classification by \citet{garcia2022regularized,trillos2022adversarial,garcia2023analytical,bungert2023geometry,bungert2024mean,pydi2024many,frank2024adversarial} and \citet{frank2024existence}.

\subsection{Worst-Case Expectations over $1$-Wasserstein Balls}
\label{sec:wc-wasserstein}
Consider the worst-case expectation problem
\begin{subequations}
\begin{align}
	\label{eq:Wasserstein-1-Lipschitz}
	\sup_{\P \in \cP(\cZ)} \left\{ \E_\P \left[ \ell (Z) \right] \;:\;  \W_1(\P, \hat{\P}) \leq r \right\},
\end{align}
which maximizes the expected value of~$\ell(Z)$ over a $1$-Wasserstein ball of radius~$r\in \R_+$ around~$\hat\P\in\cP(\cZ)$. We assume here that the $1$-Wasserstein distance is induced by a given norm~$\|\cdot\|$ on~$\R^d$. Thus, the $1$-Wasserstein ambiguity set coincides with the optimal transport ambiguity set $\cP = \left\{ \P \in \cP(\cZ) : \OT_c(\P, \hat{\P}) \leq r \right\}$ corresponding to the transportation cost function~$c$ is defined through $c(z,\hat z)= \|z - \hat z\|$. Theorem~\ref{thm:duality:OT} thus implies that the problem dual to~\eqref{eq:Wasserstein-1-Lipschitz} is given by
\begin{align}
	\label{eq:Wasserstein-1-Lipschitz-dual}
	\inf_{\lambda\ge 0} \lambda r+ \E_{\hat \P}\left[\sup_{z\in\cZ} \ell(z) -\lambda \|z-\hat Z\| \right]
\end{align}
\end{subequations}
whenever $\ell$ is upper semicontinuous. If~$\cZ=\R^d$ and $\ell$ is convex and Lipschitz continuous, then the problems~\eqref{eq:Wasserstein-1-Lipschitz} and~\eqref{eq:Wasserstein-1-Lipschitz-dual} can be solved in closed form.

\begin{proposition}[Worst-Case Expectations over $1$-Wasserstein Balls]
\label{prop:1-wasserstein-analytical}
Suppose that~$\cZ= \R^d$, $\hat\P\in\cP(\cZ)$ and $r\in\R_+$. If $\E_{\hat\P}[\ell(\hat Z)]>-\infty$ and~$\ell$ is convex and Lipschitz continuous, then the optimal values of~\eqref{eq:Wasserstein-1-Lipschitz} and~\eqref{eq:Wasserstein-1-Lipschitz-dual} are equal to 
\[
\E_{\hat \P} [ \ell(\hat Z)] +r\lip(\ell).
\]
\end{proposition}

Under the conditions of Proposition~\ref{prop:1-wasserstein-analytical}, the supremum of the primal problem~\eqref{eq:Wasserstein-1-Lipschitz} is usually {\em not} attained. The proof constructs a sequence of distributions that attain the supremum asymptotically. These distributions move an increasingly small portion of~$\hat\P$ increasingly far along the direction of steepest increase of~$\ell$. Intuitively, the amount of probability mass transported over a distance~$\Delta$ must decay as $\cO(r/\Delta)$ as~$\Delta$ grows. The dual problem~\eqref{eq:Wasserstein-1-Lipschitz-dual} is solved by~$\lambda^\star= \lip(\ell)$.

\begin{proof}[Proof of Proposition~\ref{prop:1-wasserstein-analytical}]
As the convex function~$\ell$ is Lipschitz continuous, it is in particular proper and closed. By the Fenchel-Moreau theorem (Lemma~\ref{lem:bi:coincidence}) $\ell$ thus admits the dual representation 
\begin{align*}
	\ell(z)=\sup_{y\in\dom(\ell^*)} z^\top y-\ell^*(y),
\end{align*}
where~$\ell^*$ denotes the convex conjugate of~$\ell$. Put differently, $\ell$ coincides with the pointwise supremum of the affine functions~$f_y(z)=y^\top z-\ell^*(y)$ parametrized by~$y\in\dom(\ell^*)$. H\"older's inequality then implies that
\[
\left| f_y(z)-f_y(\hat z)\right|=\left| y^\top(z-\hat z)\right|\leq \|y\|_*\|z-\hat z\|,
\]
where $\|\cdot\|_*$ denotes the norm dual to~$\|\cdot\|$. As H\"older's inequality is tight, $f_y$ is Lipschitz continuous with Lipschitz modulus $\lip(f_y)=\|y\|_*$. In addition, as the Lipschitz modulus of a supremum of affine functions coincides with the supremum of the corresponding Lipschitz moduli, the Lipschitz modulus of~$\ell$ is given by 
\[
\lip(\ell)=\sup_{y\in\dom(\ell^*)} \|y\|_*= \max_{y\in\cl(\dom(\ell^*))} \|y\|_*.
\]
The maximum in the last expression is attained by some~$y^\star\in\R^d$ because $\lip(\ell)<\infty$ by assumption. Next, define~$z^\star$ as any optimal solution of $\max_{\|z\|\leq 1} (y^\star)^\top z$. By construction, we thus have $(y^\star)^\top z^\star=\|y^\star\|_*$. We also introduce a sequence~$\{y_i\}_{i\in\N}$ in~$\dom(\ell^*)$ that converges to~$y^\star$, and we set $q_i=i^{-1} (1+|\ell^*(y_i)|)^{-1}$ for every~$i\in\N$. In addition, we define $f_i:\R^d\to\R^d$ through~$f_i(z)=z+rz^\star/q_i$ for any~$i\in\N$. Thus, $f_i$ represents the translation that shifts each point in~$\R^d$ along the direction~$z^\star$ by a distance equal to~$r/q_i$. We further define
\[
\P_i=(1-q_i)\,\hat\P + q_i\,\hat \P\circ f^{-1}_i,
\]
where $\hat \P\circ f^{-1}_i$ stands for the pushforward distribution of~$\hat \P$ under~$f_i$. Intuitively, $\P_i$ is obtained by decomposing~$\hat\P$ into two parts~$(1-q_i)\hat\P$ and~$q_i\hat\P$ and then translating the second part by~$rz^\star/q_i$. By construction, we thus have $\OT_c(\P_i, \hat{\P}) \leq r$ and
\begin{align*}
	\E_{\P_i}[\ell(Z)] & = (1-q_i)\, \E_{\hat \P}[\ell(Z)] + q_i\, \E_{\hat \P}[\ell(Z+rz^\star/q_i)] \\
	&\geq (1-q_i)\, \E_{\hat \P}[\ell(Z)] + q_i\, \E_{\hat \P}[(y_i)^\top(Z+rz^\star/q_i) - \ell^*(y_i)].
\end{align*}
Here, the inequality follows from the representation of~$\ell$ in terms of its conjugate~$\ell^*$. As~$i$ tends to infinity, $q_i$ as well as $q_i\ell^*(y_i)$ converge to~$0$, and~$y_i$ converges to~$y^\star$. Recall also that $(y^\star)^\top z^\star=\|y^\star\|_*=\lip(\ell)$. This shows that the supremum of the worst-case expectation problem~\eqref{eq:Wasserstein-1-Lipschitz} is bounded below by $\E_{\hat \P}[\ell(Z)] + r\lip(\ell)$. 

Next, define $\lambda^\star= \lip(\ell)$, and note that
\begin{align*}
	\ell(\hat z)\leq \sup_{z\in\cZ} \ell(z)-\lambda^\star\|z-\hat z\|\leq \sup_{z\in\cZ} \ell(\hat z)+\lip(\ell)\|z-\hat z\|-\lambda^\star\|z-\hat z\|=\ell(\hat z)
\end{align*}
for all~$\hat z\in\cZ$, where the second inequality follows from the Lipschitz continuity of~$\ell$, and the equality holds thanks to the definition of~$\lambda^\star$. Thus, the objective function value of~$\lambda^\star$ in the dual problem~\eqref{eq:Wasserstein-1-Lipschitz-dual} is given by
\[
\lambda^\star r+ \E_{\hat \P}\left[\sup_{z\in\cZ} \ell(z) -\lambda \|z-\hat Z\| \right] = \E_{\hat \P}\left[\ell(\hat{Z})\right] +r\lip(\ell).
\]
In summary, we have shown that---asymptotically for large~$i$---the objective function value of~$\P_i$ in~\eqref{eq:Wasserstein-1-Lipschitz} matches that of~$\lambda^\star$ in~\eqref{eq:Wasserstein-1-Lipschitz-dual}. By weak duality as established in Theorem~\ref{thm:duality:OT}, the supremum of the primal problem~\eqref{eq:Wasserstein-1-Lipschitz} thus coincides with the Lipschitz-regularized nominal loss $\E_{\hat \P}[\ell(Z)] + r\lip(\ell)$ and is asymptotically attained by the distribution~$\P_i$, which moves a fraction~$q_i$ of the total probabilty mass by a distance~$r/q_i$ along the direction~$z^\star$. 
\end{proof}

The connection between robustificaton and Lipschitz regularization was discovered by \citet{mohajerin2018data}. It offers a probabilistic interpretation for regularization techniques commonly used in statistics and machine learning \citep{shafieezadeh2015distributionally, shafieezadeh2019regularization}. Further extensions to nonconvex loss functions has been established in~\citep{blanchet2019multivariate,ho2020adversarial,shafiee2023new,gao2024wasserstein,zhang2024optimal}.

\subsection{$1$-Wasserstein Risk}
\label{sec:1-wasserstein-risk}
Consider a law-invariant risk measure~$\varrho$ that can be expressed as a superposition of CVaRs with different risk levels~$\beta\in[0,1]$. Specifically, assume that
\begin{align}
\label{eq:kusuoka-representation}
\varrho_\P \left[\ell( Z)\right] = \int_0^1 \beta\CVaR_{\P}[\ell(Z)]\, \diff \sigma(\beta)
\end{align}
for all $\P\in\cP(\cZ)$, where~$\sigma$ is a probability distribution on~$[0,1]$ with $\int_0^1\beta^{-1}\diff\sigma(\beta)<\infty$. Any~$\varrho$ with these properties is called a {\em spectral} risk measure \citep{ref:acerbi2002spectral}, and~\eqref{eq:kusuoka-representation} is termed a Kusuoka representation of~$\varrho$ \citep{ref:kusuoka2001coherent, ref:shapiro2013kusuoka}.

If the distribution of~$Z$ is only known to be close to~$\hat\P\in\cP(\cZ)$, then it is natural to quantify the riskiness of an uncertain loss~$\ell(Z)$ under a spectral risk measure~$\varrho$ by the {\em $1$-Wasserstein risk}, that is, the supremum of $\varrho_\P [\ell( Z)]$ over all distributions~$\P$ in a $1$-Wasserstain ball around~$\hat\P$. The $1$-Wasserstein risk is available in closed form whenever~$\cZ=\R^d$ and $\ell$ is convex and Lipschitz continuous.

\begin{proposition}[1-Wasserstein Risk]
\label{prop:worst-case-1-wasserstein-risk}
Let~$\varrho$ be a spectral risk measure satisfying~\eqref{eq:kusuoka-representation} with $\int_0^1\beta^{-1}\diff\sigma(\beta)<\infty$. Assume that~$\hat\P\in\cP(\R^d)$ with~$\E_{\hat\P} [\|Z\|]<\infty$ for some norm~$\|\cdot\|$ on $\R^d$. Define $\cP=\{\P\in\cP(\R^d): \W_1(\P, \hat{\P}) \leq r\}$, where $r\geq 0$ and $\W_1$ is the 1-Wasserstein distance with transportation cost function~$c(z,\hat z)=\|z-\hat z\|$. If~$\ell$ is convex and Lipschitz continuous with~$\lip(\ell)<\infty$, then we have
\[
\sup_{\P\in\cP}  \varrho_\P \left[\ell( Z)\right] = \varrho_{\hat \P} \left[\ell( Z)\right] + r \lip(\ell) \int_0^1 \beta^{-1}\diff\sigma(\beta).
\]
\end{proposition}

\begin{proof}
The assumption $\int_0^1\beta^{-1}\diff\sigma(\beta)<\infty$ ensures that~$\sigma(\{0\})=0$, and the assumption~$\E_{\hat\P} [\|Z\|]<\infty$ ensures via the Lipschitz continuity of~$\ell$ that~$\E_{\hat\P} [\ell(Z)]$ is finite. We first bound the worst-case risk from above. To this end, note that
\begin{align*}
	\sup_{\P\in\cP}  \varrho_\P \left[\ell( Z)\right] &\leq \int_0^1 \sup_{\P\in\cP} \beta\CVaR_{\P}[\ell(Z)]\, \diff \sigma(\beta) \\
	& \leq \int_0^1 \inf_{\tau\in\R} \tau+\frac{1}{\beta} \sup_{\P\in\cP} \E_\P\left[ \max \left\{\ell(Z)-\tau,0 \right\} \right] \, \diff \sigma(\beta) \\
	& = \int_0^1 \inf_{\tau\in\R} \tau+\frac{1}{\beta} \left( \E_{\hat \P}\left[ \max \left\{\ell(Z)-\tau,0 \right\} \right] + r\lip(\ell) \right)\, \diff \sigma(\beta)\\
	&= \varrho_{\hat \P} \left[\ell( Z)\right] + r \lip(\ell) \int_0^1 \beta^{-1}\diff\sigma(\beta) <+\infty,
\end{align*}
where the first inequality holds because~$\P$ may adapt to~$\beta$ when the supremum is evaluated inside the integral, and the second inequality follows from the standard max-min inequality. The first equality follows from the results on worst-case expectations over 1-Wasserstein balls in Section~\ref{sec:wc-wasserstein}. 

To derive the converse inequality, we assume first that~$\sigma(\{1\})=0$. The general case will be addressed later. Note that $\mu = \inf_{\P\in\cP}\E_\P[\ell(Z)]$ is finite because~$\ell$ is Lipschitz continuous and because~$\E_{\hat\P} [\|Z\|]<\infty$, which implies via the proof of Theorem~\ref{thm:wasserstein-compactness} that all distributions in~$\cP$ have uniformly bounded first moment.
We may assume without loss of generality that~$\mu\geq 0$. Otherwise, we may replace~$\ell(z)$ with~$\ell(z)-\mu$, which simply increases the worst-case risk by~$-\mu$ because any spectral risk measure is translation invariant. The assumption that~$\mu\geq 0$ then implies that
\[
\beta\CVaR_{\P}[\ell(Z)]\geq \E_\P[\ell(Z)] \geq 0\quad \forall \beta\in[0,1],\; \forall \P\in\cP.
\]
Thus, we have
\begin{align*}
	\sup_{\P\in\cP}\varrho_\P \left[\ell( Z)\right] & = \sup_{\P\in\cP}\sup_{\delta> 0}\int_\delta^{1-\delta} \beta\CVaR_{\P}[\ell(Z)]\, \diff \sigma(\beta) \\
	& = \sup_{\delta> 0} \sup_{\P\in\cP}\int_\delta^{1-\delta} \beta\CVaR_{\P}[\ell(Z)]\, \diff \sigma(\beta),
\end{align*}
where the first equality follows from the monotone convergence theorem and the assumption that~$\sigma(\{0\})=\sigma(\{1\})=0$. Hence, for any~$\varepsilon>0$ there is~$\delta>0$ with
\begin{subequations}
	\label{eq:wcvar-approximation}
	\begin{align}
		\left|\sup_{\P\in\cP}\varrho_\P \left[\ell( Z)\right]- \sup_{\P\in\cP}\int_\delta^{1-\delta} \beta\CVaR_{\P}[\ell(Z)]\, \diff \sigma(\beta) \right|\leq\varepsilon
	\end{align}
	and
	\begin{align}
		\left|\int_0^1\beta^{-1}\diff\sigma(\beta) - \int_\delta^{1-\delta} \beta^{-1}\diff\sigma(\beta)\right|\leq \varepsilon.
	\end{align}
\end{subequations}
Recall now from Theorem~\ref{thm:wasserstein-compactness} that~$\mc P$ is weakly compact and thus tight. Hence, there exists a compact set~$\cC\subseteq\R^d$ with $\P(Z\notin\cC)\leq \delta/2$ for every~$\P\in\cP$. As~$\cC$ is compact, $\underline\tau =\min_{z\in\cC}\ell(z)$ and $\overline\tau =\max_{z\in\cC}\ell(z)$ are both finite. Using the trivial bounds $\P(\ell(Z)\geq\underline\tau)\geq \P(Z\in\cC)$ and $\P(\ell(Z)\leq\overline\tau)\geq \P(Z\in\cC)$ and noting that $\P(Z\in\cC)\geq 1-\delta/2$ for every~$\P\in\cP$, one can then readily show that
\begin{align*}
	\underline\tau\leq (1-\delta)\VaR_{\P}[\ell(Z)]\leq \beta\VaR_{\P}[\ell(Z)] \leq \delta\VaR_{\P}[\ell(Z)] \leq \overline \tau
\end{align*}
for all $\beta\in[\delta,1-\delta]$ and for all~$\P\in\cP$. Next, define $y_i\in\dom(\ell^*)$, $q_i\in[0,1]$, the function $f_i:\R^d\to\R^d$ and the distribution $\P_i=(1-q_i)\,\hat\P + q_i\,\hat \P\circ f^{-1}_i$ for $i\in\N$ as in Section~\ref{sec:wc-wasserstein}.
We then obtain
\begin{align}
	\label{eq:worst-case-spectral-w1}
	\begin{aligned}
		\sup_{\P\in\cP}  \varrho_\P \left[\ell( Z)\right] &\geq \int_\delta^{1-\delta} \inf_{\tau\in\R} \tau+\frac{1}{\beta} \E_{\P_i} \left[ \max \left\{\ell(Z)-\tau,0 \right\} \right] \, \diff \sigma(\beta) \\
		& = \int_\delta^{1-\delta} \inf_{\tau \in[\underline\tau, \overline\tau]} \tau+\frac{1-q_i}{\beta} \E_{\hat\P} \left[ \max \left\{\ell(Z)-\tau,0 \right\} \right] \\
		& \hspace{2cm} + \frac{q_i}{\beta} \E_{\hat\P} \left[ \max \left\{\ell(Z+rz^\star/q_i)-\tau,0 \right\} \right] \, \diff \sigma(\beta).
	\end{aligned}
\end{align}
The inequality in~\eqref{eq:worst-case-spectral-w1} holds because $\beta\CVaR_{\P}[\ell(Z)]\geq 0$ for all $\beta\in[0,1]$ by assumption and because~$\P_i\in\cP$ as shown in Section~\ref{sec:wc-wasserstein}. The equality follows from the definition of~$\P_i$ and from \cite[Theorem~10]{rockafellar2002cvar-general}, which ensures that the minimization problem over~$\tau$ is solved by~$\beta\VaR_{\P}[\ell(Z)]\in [\underline\tau, \overline\tau]$. As~$\ell$ is proper, convex and lower semicontinuous, and as~$y_i$ belongs to the domain of~$\ell^*$, the Fenchel-Moreau theorem further implies that
\begin{align*}
	\ell(z+rz^\star/q_i) &= \sup_{y\in\dom(\ell^*)} (z+rz^\star/q_i)^\top y-\ell^*(y) \geq (z+rz^\star/q_i)^\top y_i -\ell^*(y_i).
\end{align*}
The last expectation in~\eqref{eq:worst-case-spectral-w1} thus admits the lower bound
\begin{align*}
	\E_{\hat\P} \left[ \max \left\{\ell(Z+rz^\star/q_i)-\tau,0 \right\} \right] & \geq \E_{\hat\P} \left[ \ell(Z+rz^\star/q_i)-\tau \right] \\
	& \geq \E_{\hat \P}\left[y_i^\top Z\right] +r y_i^\top z^\star/q_i -\ell^*(y_i)-\overline \tau.
\end{align*}
Substituting this estimate into~\eqref{eq:worst-case-spectral-w1} and letting~$i$ tend to infinity yields
\begin{align*}
	\begin{aligned}
		\sup_{\P\in\cP}  \varrho_\P \left[\ell( Z)\right] &\geq \lim_{i\to\infty} \int_\delta^{1-\delta} \inf_{\tau\in[ \underline\tau, \overline\tau]} \tau+\frac{1-q_i}{\beta} \E_{\hat\P} \left[ \max \left\{\ell(Z)-\tau,0 \right\} \right] \, \diff \sigma(\beta) \\
		& \hspace{1cm} + r \lip(\ell) \int_\delta^{1-\delta} \beta^{-1}\, \diff \sigma(\beta) \\
		& = \int_\delta^{1-\delta} \beta\CVaR_{\P}[\ell(Z)] \, \diff \sigma(\beta) + r \lip(\ell) \int_\delta^{1-\delta} \beta^{-1}\, \diff \sigma(\beta),
	\end{aligned}
\end{align*}
where we have used that~$q_i$ as well as~$q_i\ell^*(y_i)$ converge to~$0$ and that~$y_i^\top z^\star$ converges to~$(y^\star)^\top z^\star=\lip(\ell)$ as~$i$ tends to infinity; see also Section~\ref{sec:wc-wasserstein}. The equality follows from the monotone convergence theorem, which applies because~$q_i$ is monotonically decreasing with~$i$. Letting~$\varepsilon$ tend to~$0$ thus implies via~\eqref{eq:wcvar-approximation} that
\begin{align*}
	\sup_{\P\in\cP}  \varrho_\P \left[\ell( Z)\right] &\geq \int_0^1 \beta\CVaR_{\P}[\ell(Z)] \, \diff \sigma(\beta) + r \lip(\ell) \int_0^1 \beta^{-1}\, \diff \sigma(\beta).
\end{align*}
This lower bound matches the upper bound derived in the first part of the proof, and thus the claim follows, provided that~$\sigma(\{1\})=0$. If the probability distribution~$\sigma$ has an atom at~$1$, then it can be decomposed as $\sigma=\hat\sigma+\sigma(\{1\})\cdot \delta_1$, where~$\hat\sigma$ is a non-negative measure on~$(0,1)$. We can thus decompose the risk under~$\P$ as
\[
\varrho_\P[\ell(Z)] = \int_0^1 \beta\CVaR_{\P}[\ell(Z)] \, \diff \hat \sigma(\beta) + \sigma(\{1\}) \cdot \E_\P[\ell(Z)].
\]
The first term in this decomposition can then be handled as above, and the second term can be handled as in Section~\ref{sec:wc-wasserstein}. Details are omitted for brevity.
\end{proof}

Proposition~\ref{prop:worst-case-1-wasserstein-risk} shows that the $1$-Wasserstein risk of a Lipschitz continuous convex loss function coincides with the sum of the nominal risk and a Lipschitz regularization term. It is asymptotically attained by the distribution~$\P_i$, which moves a fraction~$q_i$ of the total probability mass by a distance~$r/q_i$ along the direction~$z^\star$. Proposition~\ref{prop:1-wasserstein-analytical} emerges as a special case of Proposition~\ref{prop:worst-case-1-wasserstein-risk} when $\sigma=\delta_1$. The worst-case risk over $p$-Wasserstein balls for~$p\geq 1$ was first studied by \citet{pflug2012}, and a result akin to Proposition~\ref{prop:worst-case-1-wasserstein-risk} was obtained for linear loss functions. Extensions to more general risk measures were studied by \citet{pichler:2013} and \citet{wozabal2014robustifying}. The extension to convex loss functions is new.

\subsection{$p$-Wasserstein Risk}
\label{sec:p-wasserstein-risk}

We now show that if the loss function~$\ell(z)$ is linear, then the worst-case risk over a $p$-Wasserstein ball may be available in closed form even if~$p\in(1,\infty)$. The results of this section depend on the following lemma, which characterizes the conjugates of powers of norms; see also
\cite[Lemma~C.9]{zhen2023unification}.

\begin{lemma}[Conjugates of Powers of Norms]
\label{lem:p-norm}
Assume that $\|\cdot\|$ and $\|\cdot\|_*$ are mutually dual norms on~$\R^{d}$ and that $p,q \in (1,\infty)$ are conjugate exponents with $\frac{1}{p} + \frac{1}{q} = 1$. Define $\varphi(q)=(q-1)^{(q-1)}/q^q$. Then, the following statements hold.
\begin{itemize}
	\item[(i)] If $f(z) = \frac{1}{p} \|z\|^p $, then $f^*(y) = \frac{1}{q} \| y \|^q_*$. 
	\item[(ii)] If $g(z) = \| z-\hat z \|^p$, then $g^{*}(y) = y^\top\hat z +   \varphi(q) \left\| y \right\|^q_*$. 
\end{itemize}
\end{lemma}
\begin{proof}
As for assertion~(i), fix any~$z,y \in\R^{d}$. We then have
\begin{align*}
	z^\top y - \frac{1}{p} \|z\|^p \leq \|z\| \|y\|_* - \frac{1}{p} \|z\|^p \leq \max_{t\ge 0}~ t\|y\|_*-\frac{1}{p}t^p = \frac{1}{q} \|y\|^{q}_*,
\end{align*}
where the first inequality follows from the construction of the dual norm, and the second inequality is obtained by maximizing over~$t=\|z\|$. The equality holds because the maximization problem is solved by~$\tau=\|y\|^{1/(p-1)}_*$. Both inequalities collapse to equalities if~$z\in\arg\max_{\|z\|=\tau} z^\top y$. This allows us to conclude that
\[ 
f^*(y) = \sup_{z\in\R^d} z^\top y - \frac{1}{p} \|z\|^p  = \frac{1}{q} \|y\|^{q}_*. 
\]
As for assertion~(ii), note that
\begin{align*}
	g^*(y) & = \sup_{z\in\R^d} y^\top z  - \| z - \hat z \|^p =  y^\top \hat z +  p \cdot \sup_{z\in\R^d} (y/p)^\top z  - \frac{1}{p} \| z\|^p \\
	& = y^\top \hat z +   \frac{p}{q} \left\| y/p\right\|_*^q = y^\top \hat z +   \varphi(q) \left\|y\right\|^q_* ,
\end{align*}
where the last two equalities exploit assertion~(i) and the definition of~$\varphi(q)$.
\end{proof}

We now show that the worst-case CVaR of a linear loss function~$\ell(z)=\theta^\top z$ over a $p$-Wasserstein ball of radius~$r$ around~$\hat\P$ equals the sum of the nominal CVaR under~$\hat\P$ and a regularization term that scales with the norm of~$\theta$ and with~$r$.

\begin{proposition}[$p$-Wasserstein Risk]
\label{prop:worst-case-p-wasserstein-risk}
Assume that~$\hat\P\in\cP(\R^d)$ with~$\E_{\hat\P} [\|Z\|^p]<\infty$  for some~$p\in(1,\infty)$ and for some norm~$\|\cdot\|$ on $\R^d$. Define $\cP=\{\P\in\cP(\R^d): \W_p(\P, \hat{\P}) \leq r\}$, where~$r\geq 0$ and~$\W_p$ is the $p$-Wasserstein distance with transportation cost function~$c(z,\hat z)=\|z-\hat z\|^p$. If~$\theta\in\R^d$ and~$\beta\in(0,1)$, then
\[
\sup_{\P\in\cP}  \beta\CVaR_{\P}[\theta^\top Z] = \beta\CVaR_{\hat \P}[\theta^\top Z] + r \beta^{-1/p} \|\theta\|_*.
\]
\end{proposition}

\begin{proof}
By the definition of the CVaR by \citet{rockafellar2000optimization}, we have
\begin{align}
	\label{eq:wc-cvar-p-wasserstein}
	\sup_{\P\in\cP}  \beta\CVaR_{\P}[\theta^\top Z] \leq \inf_{\tau\in\R} \tau + \frac{1}{\beta} \sup_{\P\in\cP} \E_{\P}\left[ \max\left\{ \theta^\top Z-\tau, 0\right\}\right],
\end{align}
where the inequality is obtained by interchanging the supremum over~$\P$ and the infimum over~$\tau$. The underlying worst-case expectation problem satisfies
\begin{align*}
	&\sup_{\P\in\cP} \E_{\P}\left[ \max\left\{ \theta^\top Z-\tau, 0\right\}\right] \\
	& \leq \inf_{\lambda\ge 0} \lambda r^p+ \E_{\hat \P}\left[\sup_{z\in\R^d} \max\left\{ \theta^\top z-\tau, 0\right\} -\lambda \|z-\hat Z\|^p \right]\\
	& = \inf_{\lambda\ge 0} \lambda r^p+ \E_{\hat \P}\left[ \max\left\{ \sup_{z\in\R^d} \theta^\top z-\tau -\lambda \|z-\hat Z\|^p , \sup_{z\in\R^d} -\lambda \|z-\hat Z\|^p \right\}  \right] \\
	& = \inf_{\lambda\ge 0} \lambda r^p+ \E_{\hat \P}\left[ \max\left\{ \theta^\top \hat Z-\tau +\varphi(q)\lambda  \left\| \theta / \lambda \right\|_*^q , 0 \right\}  \right]
\end{align*}
where the inequality exploits weak duality, and the first equality is obtained by interchanging the order of the two maximization operations. The second equality follows from Lemma~\ref{lem:p-norm}(ii). Substituting the resulting formula into~\eqref{eq:wc-cvar-p-wasserstein} and interchanging the infimum over~$\tau$ with the infimum over~$\lambda$ then yields
\begin{align*}
	& \sup_{\P\in\cP}  \beta\CVaR_{\P}[\theta^\top Z] \\
	& \leq \inf_{\lambda\ge 0} \frac{\lambda r^p}{\beta} + \inf_{\tau\in\R} \tau + \frac{1}{\beta}  \E_{\hat \P}\left[ \max\left\{ \theta^\top \hat Z-\tau +\varphi(q)\lambda  \left\| \theta / \lambda \right\|_*^q , 0 \right\}  \right]\\
	& = \inf_{\lambda\ge 0} \frac{\lambda r^p}{\beta} + \beta\CVaR_{\hat \P}[\theta^\top \hat Z +\varphi(q)\lambda  \left\| \theta / \lambda \right\|_*^q ] \\
	& = \beta\CVaR_{\hat \P}[\theta^\top \hat Z ] + \inf_{\lambda\ge 0} \frac{\lambda r^p}{\beta} +\varphi(q)\lambda  \left\| \theta / \lambda \right\|_*^q,
\end{align*}
where the equalities follow from the definition and the translation invariance of the CVaR, respectively. Solving the minimization problem over~$\lambda$ analytically yields
\begin{align*}
	\sup_{\P\in\cP}  \beta\CVaR_{\P}[\theta^\top Z] \leq  \beta\CVaR_{\hat \P}[\theta^\top \hat Z ] + r\beta^{-1/p} \|\theta\|_*.
\end{align*}

To derive the converse inequality, we use~$\tau_\beta$ as a shorthand for $\beta\VaR_{\hat \P}[\theta^\top \hat Z ]$, which is finite because~$\beta\in(0,1)$, and we select any $z^\star\in\arg\max_{\|z\|=1} \theta^\top z$. In addition, we decompose the nominal distribution as $\hat\P=\beta \, \hat\P_+ + (1-\beta)\, \hat\P_-$, where~$\hat\P_+$ and~$\hat\P_-$ are probability distributions supported on $\cZ_+=\{z\in\R^d:\theta^\top z\geq \tau_\beta\}$ and $\cZ_-=\{z\in\R^d:\theta^\top z\leq \tau_\beta\}$, respectively. Such a decomposition always exists thanks to the definition of~$\tau_\beta$. For example, if~$\hat\P(\theta^\top Z=\tau_\beta)=0$, as would be the case if~$\hat\P$ was absolutely continuous with respect to Lebesgue measure, then~$\hat\P_-$ and~$\hat\P_+$ can simply be obtained by conditioning~$\hat\P$ on~$\cZ_-$ and~$\cZ_+$, respectively. We also define $f:\R^d\to\R^d$ through $f(z)=z + rz^\star/\beta^{1/p}$. Thus, $f$ shifts all points in~$\R^d$ along the direction~$z^\star$ by a distance equal to~$r/\beta^{1/p}$. Finally, we set~$\P^\star=\beta\,\hat \P_+\circ f^{-1}+ (1-\beta)\, \hat\P_-$. Hence, $\P^\star$ is obtained by decomposing~$\hat\P$ into two parts~$\beta\,\hat\P_+$ and~$(1-\beta)\,\hat\P_-$ and then translating the first part by~$r z^\star/\beta^{1/p}$. We thus have $\W_p(\P^\star,\hat\P)\leq r$, and $\beta\VaR_{\P}[\theta^\top Z]=\tau_\beta$. This in turn implies that
\begin{align*}
	\sup_{\P\in\cP} \, &\beta\CVaR_{\P}[\theta^\top Z]  \geq  \beta\CVaR_{\P^\star}[\theta^\top Z]\\
	& = \tau_\beta+\frac{1}{\beta} \E_{\P^\star}\left[\max\left\{ \theta^\top Z-\tau_\beta,0\right\} \right] \\
	& = \tau_\beta + \E_{\hat\P_+}\left[\max\left\{ \theta^\top f(Z)-\tau_\beta,0\right\} \right]+\frac{1-\beta}{\beta} \E_{\hat\P_-}\left[\max\left\{ \theta^\top Z-\tau_\beta,0\right\} \right]\\
	& = \E_{\hat\P_+}\left[ \theta^\top Z\right]+ r\beta^{-1/p} \|\theta\|_* = \beta\CVaR_{\hat \P}[\theta^\top \hat Z ] + r\beta^{-1/p} \|\theta\|_*.
\end{align*}
Here, the first equality follows from the definition of the CVaR and from \cite[Theorem~10]{rockafellar2002cvar-general}, which ensures~$\tau$ matches~$\beta\VaR_{\P^\star}[\ell(Z)]=\tau_\beta$ at optimality. The second equality exploits the definition of~$\P^\star$, and the third equality holds because~$\theta^\top z^\star=\|\theta\|_*$ and because~$\theta^\top z\geq \tau_\beta$ for all~$z\in\cZ_+$ and~$\theta^\top z\leq \tau_\beta$ for all~$z\in\cZ_-$. Finally, the fourth equality follows from the construction of~$\hat\P_+$ and from \cite[Proposition~5]{rockafellar2002cvar-general}. This completes the proof.
\end{proof}

\section{Finite Convex Reformulations of Nature's Subproblem}
\label{sec:finite-convex-reformulations}

Although nature's subproblem admits analytical solutions in important special cases (\emph{cf.}~Section~\ref{sec:analytical-wc}), it can usually only be solved numerically. Sometimes, nature's subproblem can be reformulated as an equivalent convex optimization problem. In these cases, it can be addressed with off-the-shelf solvers. In other cases, however, it may be necessary or preferable to develop customized solution algorithms.

This section focuses on finite convex reductions. That is, we will describe conditions under which the dual worst-case expectation problems derived in Section~\ref{sec:duality-wc-expectation} can be reformulated as finite convex {\em minimization} problems. These finite reformulations are significant because they can be combined with the outer minimization problem over~$x \in \mathcal{X}$ to construct a reformulation of the overall DRO problem~\eqref{eq:primal:dro} as a classical minimization problem amenable to standard optimization software. We subsequently dualize the finite convex reformulations of nature's subproblem to obtain equivalent finite convex {\em maximization} problems. These finite bi-dual maximization problems are significant because their optimal solutions allow us to construct worst-case distributions that (asymptotically) attain the supremum of nature's subproblem~\eqref{eq:worst-case:expectation}. Even though we only address worst-case expectations, all results of this section readily extend to worst-case optimized certainty equivalents thanks to Theorem~\ref{thm:duality:regular}. For the sake of brevity, however, we will not elaborate on these extensions. To simplify notation, we will always suppress the dependence of the loss function~$\ell$ on the decision variables~$x$.

The remainder of this section develops as follows. In Section~\ref{sec:proof-strategy-tractability}, we first outline a general strategy for deriving finite convex dual and bi-dual reformulations of nature's subproblem~\eqref{eq:worst-case:expectation}. We subsequently exemplify this strategy for worst-case expectation problems over Chebyshev ambiguity sets (Section~\ref{sec:wc-distribution:chebyshev}), $\phi$-divergence ambiguity sets (Section~\ref{sec:wc-distribution:phi}) and optimal transport ambiguity sets (Section~\ref{sec:wc-distribution:transport}).

\subsection{General Proof Strategy}
\label{sec:proof-strategy-tractability}
The worst-case expectation problem~\eqref{eq:worst-case:expectation} constitutes a semi-infinite program that involves infinitely many decision variables (because it optimizes over a subset of an infinite-dimensional measure space) but only finitely many constraints ({\em e.g.}, moment conditions and/or bounds on the divergence or discrepancy to a reference distribution). The duality results of Section~\ref{sec:duality-wc-expectation} enable us to recast this semi-infinite maximization problem as a semi-infinite minimization problem with finitely many variables and infinitely many constraints. We then leverage reformulation techniques from robust optimization to recast the dual semi-infinite program as a finite-dimensional convex minimization problem. These techniques exploit standard results from convex analysis as well as the $\cS$-Lemma, which we review next. Throughout this discussion we adopt the convention that $0\cdot\infty=\infty$.

We first show that scaling and perspectivication constitute dual operations.

\begin{lemma}[Duality of Scaling and Perspectivication]
\label{lem:conjugate-perspective}
If $f:\R^d\to\overline\R$ is a proper, closed and convex function and $\alpha\in\R_+$ a fixed constant, then the following hold.
\begin{enumerate}[label=(\roman*)]
	\item \label{lem:conjugate-perspective-1} If $g(z)=\alpha f(z)$, then $g^*(y)= (f^*)^\pi(y, \alpha)$ for all~$y\in\R^d$.
	\item \label{lem:conjugate-perspective-2} If $g(z)=f^\pi(z,\alpha)$, then $g^*(y)= \cl(\alpha f^*)(y)$ for all~$y\in\R^d$.
\end{enumerate}
\end{lemma}

\begin{proof}
We prove assertion~\ref{lem:conjugate-perspective-1} by case distinction. First, if~$\alpha > 0$, then we have
\begin{align*}
	g^*(y) 
	&= \sup_{z \in \R^d} y^\top z - \alpha f(z) 
	= \alpha \sup_{z \in \R^d} (y / \alpha)^\top z - f(z) \\
	&= \alpha f^*(y/\alpha)
	= (f^*)^\pi(y, \alpha).
\end{align*}
If $\alpha= 0$, on the other hand, then a similar reasoning shows that
\begin{align*}
	g^*(y) 
	&= \sup_{z \in \R^d} y^\top z - \delta_{\dom(f)}(z)
	= \delta^*_{\dom(f)}(y)= \delta^*_{\dom(f^{**})}(y)  \\
	&= (f^*)^\infty(y)
	= (f^*)^\pi(y, \alpha),
\end{align*}
where the first equality follows from our convention that $0\cdot\infty=\infty$, which implies that $0 f(z)=\delta_{\dom(f)}(z)$. The second equality follows from the definition of the support function, and the third equality holds because~$f$ is convex and closed, which implies via Lemma~\ref{lem:bi:coincidence} that $f=f^{**}$. Finally, the fourth equality follows from \citep[Theorem~13.3]{rockafellar1970convex}, and the last equality exploits the definition of the perspective function for~$\alpha=0$. This completes the proof of assertion~\ref{lem:conjugate-perspective-1}.

As for assertion~\ref{lem:conjugate-perspective-2}, assume first that $\alpha>0$, and note that
\begin{align*}
	g^*(y) & = \sup_{z\in\R^d} y^\top z -f^\pi(z,\alpha) = \alpha\; \sup_{z\in\R^d} y^\top (z/\alpha) -f(z/\alpha) =\alpha f^*(y)= \cl(\alpha f^*)(y),
\end{align*}
where the last equality holds because~$f^*$ is closed. If $\alpha=0$, then we have
\begin{align*}
	g^*(y) & = \sup_{z\in\R^d} y^\top z -f^\infty(z) = \sup_{z\in\R^d} y^\top (z) -\delta^*_{\dom(f^*)} =\delta_{\cl(\dom(f^*))}(y) = \cl(\alpha f^*)(y).
\end{align*}
Here, the first equality exploits the definition of the perspective. The second and the third equalities follow from \citep[Theorem~13.3]{rockafellar1970convex} and \citep[Theorem~13.2]{rockafellar1970convex}, respectively. The last equality, finally, holds because $0 f^*=\delta_{\dom(f^*)}$ by our conventions of extended arithmetic. This proves assertion~\ref{lem:conjugate-perspective-2}.
\end{proof}

The following lemma derives a formula for the conjugate of a sum of functions.

\begin{lemma}[Conjugates of Sums]
\label{lem:conjugate-sums}
If $f_k:\R^d\to\overline\R$, $k\in[K]$, are proper, convex and closed functions, then the conjugate of $f=\sum_{k\in[K]} f_k$ satisfies
\begin{equation}
	\label{eq:infimal-convolution}
	f^*(y)\leq \inf_{y_1,\ldots,y_K\in\R^d}\bigg\{ \sum_{k\in[K]} f_k^*(y_k) : \sum_{k\in[K]} y_k=y \bigg\} \quad \forall y\in\R^d.  
\end{equation}
If there exists $\bar z\in\cap_{k\in[K]} \rint(\dom(f_k))$, then the inequality in the above expression reduces to an equality, and the minimum is attained for every~$y\in\R^d$.
\end{lemma}

The infimum on the right hand side of~\eqref{eq:infimal-convolution} defines a function of~$y$. This function is called the {\em infimal convolution} of the functions $f_k^*$, $k\in[K]$. Thus, Lemma~\ref{lem:conjugate-sums} asserts that, under a mild Slater-type condition, the conjugate of a sum of functions coincides with the infimal convolution of the conjugates of these functions.

\begin{proof}[Proof of Lemma~\ref{lem:conjugate-sums}]
By using a standard variable splitting trick and the max-min inequality, one can show that the conjugate of~$f$ admits the following upper bound.
\begin{align*}
	f^*(y) 
	& = \sup_{z,z_1,\ldots, z_K\in\R^d} \bigg\{ y^\top z- \sum_{k\in[K]} f(z_k)\;:\; z_k =z ~~ \forall k \in [K] \bigg\} \\
	&= \sup_{z,z_1,\ldots, z_K\in\R^d} \;  \inf_{y_1, \ldots,y_K \in \R^d} y^\top z-\sum_{k \in[K]} f(z_k) - y_k^\top (z - z_k) \\
	&\leq \inf_{y_1, \ldots,y_K \in \R^d} \; \sup_{z,z_1,\ldots, z_K\in\R^d} y^\top z- \sum_{k \in[K]} f(z_k) - y_k^\top (z - z_k) \\
	&= \inf_{y_1, \ldots,y_K \in \R^d} \; \sup_{z\in\R^d} \bigg\{ y^\top z- \sum_{k\in[K]} y_k^\top z \bigg\} + \sum_{k\in[K]} \sup_{z_k\in\R^d} \big\{y_k^\top z_k - f(z_k) \big\}
\end{align*}
The supremum over~$z$ in the resulting expression evaluates to~$0$ if $\sum_{k\in[K]}y_k=y$ and to $\infty$ otherwise. In addition, the supremum over~$z_k$ evaluates to $f_k^*(y_k)$ for every $k\in[K]$. Substituting these analytical formulas into the last expression yields
\[
f^*(y)\leq \inf_{y_1,\ldots,y_K\in\R^d}\bigg\{ \sum_{k\in[K]} f^*(y_k) : \sum_{k\in[K]} y_k=y \bigg\}.
\]
If $\cap_{k\in[K]}\rint(\dom(f_k))$ is non-empty, then the above inequality becomes an equality, and the infimum is attained thanks to \citep[Theorem~16.4]{rockafellar1970convex}. 
\end{proof}

Consider now a classical optimization problem
\begin{align}
\label{eq:primal:convex} \tag{P}
\inf_{z \in \R^d} \big\{f(z): ~ g_k(z) \leq 0 ~~ \forall k \in [K] \big\}
\end{align}
with objective function~$f:\R^d\to\overline\R$ and constraint functions~$g_k:\R^d\to\overline\R$, $k \in [K]$. Below we will show that the problem dual to~\eqref{eq:primal:convex} is given by
\begin{align}
\label{eq:dual:convex} \tag{D}
\sup_{\substack{ \alpha_1,\ldots,\alpha_K \in \R_+ \\ \beta_0,\ldots,\beta_K \in \R^d}} \left\{ -f^*(\beta_0) - \sum_{k=1}^K (g^*_k)^\pi(\beta_k, \alpha_k) \; : \; \sum_{k =0}^K \beta_k = 0\right\}.
\end{align}

To this end, we adopt the following definition of a Slater point.

\begin{definition}[Slater Point]
\label{def:slater}
A Slater point of the set $\cZ = \{ z \in \R^d : g_k(z) \leq 0 ~ \forall k \in [K] \}$ is any vector $\bar z \in \cZ$ with $\bar z \in \rint(\dom(g_k))$ for all $k \in [K]$ and $g_k (\bar z) < 0$ for all $k \in[K]$ such that $g_k$ is nonlinear. A Slater point $\bar z$ of the set $\cZ$ is a Slater point of the minimization problem $\inf \{ f(z): z \in \cZ \}$ if $\bar z \in \rint(\dom(f))$.
\end{definition}

Slater points of maximization problems are defined in the obvious way. We simply replace the requirement $\bar z \in \rint(\dom(f))$ with $\bar z \in \rint(\dom(-f))$. Using Lemmas~\ref{lem:conjugate-perspective} and~\ref{lem:conjugate-sums}, we can now prove that~\eqref{eq:primal:convex} and~\eqref{eq:dual:convex} are indeed duals.

\begin{theorem}[Convex Duality]
\label{thm:convex:duality}
Assume that the functions~$f$ and $g_k$, $k \in [K]$, are proper, closed and convex. Then, the infimum of~\eqref{eq:primal:convex} is larger or equal to the supremum of~\eqref{eq:dual:convex}. In addition, the following strong duality relations hold.
\begin{enumerate}[label=(\roman*)]
	\item \label{lem:slater:duality} If~\eqref{eq:primal:convex} or~\eqref{eq:dual:convex} admits a Slater point, then the infimum of~\eqref{eq:primal:convex} matches the supremum of~\eqref{eq:dual:convex}, and~\eqref{eq:dual:convex} or~\eqref{eq:primal:convex} is solvable, respectively.
	\item \label{lem:compact:duality} If the feasible set of~\eqref{eq:primal:convex} or~\eqref{eq:dual:convex} is non-empty and bounded, then the infimum of~\eqref{eq:primal:convex} matches the supremum of~\eqref{eq:dual:convex}, and~\eqref{eq:primal:convex} or~\eqref{eq:dual:convex} is solvable, respectively.
\end{enumerate}
\end{theorem}

\begin{proof}
The max-min inequality readily implies that the infimum of~\eqref{eq:primal:convex} is bounded below by the optimal value of its Lagrangian dual, that is, we have
\begin{align*}
	\inf\text{\eqref{eq:primal:convex}} & = \inf_{z \in \R^d}\sup_{\alpha \in \R_+^K}  ~ f(z) + \sum_{k \in [K]} \alpha_k g_k(z) \\
	& \geq \sup_{\alpha \in \R_+^K} \inf_{z \in \R^d} ~ f(z) + \sum_{k \in [K]} \alpha_k g_k(z)\\
	& = \sup_{\alpha \in \R_+^K} - \sup_{z\in\R^d} 0^\top z - f(z)- \sum_{k\in[K]} \alpha_k g_k(z) \\
	& = \sup_{\alpha \in \R_+^K} - \bigg( f+ \sum_{k\in[K]} \alpha_k g_k \bigg)^*(0).
\end{align*}
The resulting lower bound involves the conjugate of a sum of several functions. By Lemma~\ref{lem:conjugate-sums}, the conjugate of this sum is bounded below by the infimal convolution of the conjugates of all functions in the sum. Consequently, we obtain
\begin{align}
	\label{eq:lb-convex-dual}
	\inf\text{\eqref{eq:primal:convex}} \geq \sup_{\substack{ \alpha_1,\ldots,\alpha_K \in \R_+ \\ \beta_0,\ldots,\beta_K \in \R^d}} \left\{ - f^*(\beta_0) - \sum_{k=1}^K (\alpha_k g_k)^*(\beta_k) \;:\; \sum_{k =0}^K \beta_k = 0 \right\}.
\end{align}
By Lemma~\ref{lem:conjugate-perspective}\,\ref{lem:conjugate-perspective-1}, we further have $(\alpha_k g_k)^*(\beta_k)= (g^*_k)^\pi(\beta_k, \alpha_k)$ for all $\beta_k\in\R^d$ and $\alpha_k\in\R_+$. Thus, the lower bound in~\eqref{eq:lb-convex-dual} matches the supremum of~\eqref{eq:dual:convex}. This proves weak duality. For a proof of strong duality and solvability under the conditions~(i) and~(ii), we refer to \citep[Theorem~2]{zhen2023unification}.
\end{proof}

Armed with Theorem~\ref{thm:convex:duality}, we can now show that the semi-infinite constraints appearing in the dual worst-case expectation problems derived in Section~\ref{sec:duality-wc-expectation} can systematically be reformulated in terms of finitely many convex constraints.

\begin{proposition}[Semi-Infinite Constraints I]
\label{prop:dual-reformulation}
Assume that the functions $f:\R^d \to \overline \R$ and $g_k:\R^d \to \overline \R$, $k \in [K]$, are proper, closed and convex, and that there is $\bar z \in \R^d$ with $\bar z \in \rint(\dom(g_k))$, $k \in [K]$, $\bar z \in \rint(\dom(f))$ and $g_k (\bar z) < 0$ for all $k \in [K]$ such that $g_k$ is nonlinear. Then the semi-infinite constraint
\begin{align*}
	f(z) \geq 0 \quad \forall z \in \R^d:~ g_k(z) \leq 0 ~~ \forall k \in [K]
\end{align*}
holds if and only if there exist $\alpha_1,\ldots,\alpha_K \in \R_+$ and $\beta_0, \ldots, \beta_K \in \R^d$ with
\begin{align*}
	f^*(\beta_0) + \sum_{k=1}^K (g^*_k)^\pi(\beta_k, \alpha_k) \leq 0 \quad\text{and} \quad \sum_{k=0}^K \beta_k = 0 .
\end{align*}
\end{proposition}

\begin{proof}
The semi-infinite constraint in the statement of the proposition is satisfied if and only if the infimum of~\eqref{eq:primal:convex} is non-negative. Under the stated assumptions, Theorem~\ref{thm:convex:duality} implies that this is the case precisely when the supremum of~\eqref{eq:dual:convex} is non-negative. Since~\eqref{eq:primal:convex} admits a Slater point, the supremum of~\eqref{eq:dual:convex} is attained. Thus, the supremum of~\eqref{eq:dual:convex} is non-negative if and only if there are $\alpha_1,\ldots,\alpha_K \in \R_+$ and $\beta_0, \ldots, \beta_K \in \R^d$ satisfying the constraints in the statement of the proposition.
\end{proof}

Proposition~\ref{prop:dual-reformulation} enables us to derive finite convex reformulations of the semi-infinite constraints that appear in the dual of the worst-case expectation problem~\eqref{eq:worst-case:expectation} whenever the relevant objective and constraint functions are convex in~$z$.

Another similar reformulation technique relies on the $\cS$-Lemma (see, \emph{e.g.}, \citealt{polik2007survey}), which we present without a proof.

\begin{lemma}[{$\cS$-Lemma~\citep{yakubovich1971s}}]
\label{lem:S-lemma}
Assume that $f: \R^d \to \R$ and $g: \R^d \to \R$ are quadratic functions. If there exists a Slater point $\bar z \in \R^d$ such that $g(\bar z) < 0$, then the following two statements are equivalent.
\begin{enumerate}[label=(\roman*)]
	\item There is no $z \in \R^d$ such that $f(z) < 0$ and $g(z) \leq 0$.
	\item There exists $\alpha \in\R_+$ such that $f (z) + \alpha g(z) \geq 0$ for all $z \in \R^d$.
\end{enumerate}
\end{lemma}

The $\cS$-Lemma allows us to derive a finite convex reformulations of semi-infinite constraints that require a (possibly indefinite) quadratic function to be non-negative over the feasible set of a single quadratic constraint. Note in particular that the involved functions~$f$ and~$g$ are \emph{not} required to be convex in~$z$.

\begin{proposition}[Semi-Infinite Constraints II]
\label{prop:SDP:reformulation}
Assume that $Q_0, Q_1 \in \S^d$, $q_0, q_1 \in \R^d$, and $r_0, r_1 \in \R$. In addition, assume that there exists a Slater point $\bar z \in \R^d$ such that $\bar z^\top Q_0 \bar z + 2 q_0^\top \bar z + r_0 < 0$. Then, the semi-infinite constraint
\begin{align*}
	z^\top Q_1 z + 2 q_1^\top z + r_1 \geq 0 \quad \forall z \in\R^d: ~z^\top Q_0 z + 2 q_0^\top z + r_0 \leq 0
\end{align*}
holds if and only if there exists $\alpha \in \R_+$ with
\begin{align*}
	\begin{bmatrix} Q_1 + \alpha Q_0 & q_1 + \alpha q_0 \\ q_1^\top + \alpha q_0^\top & r_1 + \alpha r_0  \end{bmatrix} \succeq 0.
\end{align*}
\end{proposition}

\begin{proof}
We observe that
\begin{align*}
	& ~ z^\top Q_1 z + 2 q_1^\top z + r_1 \geq 0 \quad \forall z \in \R^d:~ z^\top Q_0 z + 2 q_0^\top z + r_0 \leq 0 \\
	\iff & ~ \exists \alpha \in \R_+ ~ \text{with} ~ z^\top (Q_1 + \alpha Q_0) \, z + 2 (q_1 + \alpha q_0)^\top z + r_1 + \alpha r_0 \geq 0 \quad \forall z \in \R^d \\
	\iff & ~ \exists \alpha \in \R_+ ~ \text{with} ~ \begin{bmatrix}
		z \\ 1
	\end{bmatrix}^\top \begin{bmatrix} Q_1 + \alpha Q_0 & q_1 + \alpha q_0 \\ q_1^\top + \alpha q_0^\top & r_1 + \alpha r_0  \end{bmatrix} \begin{bmatrix}
		z \\ 1
	\end{bmatrix}\geq 0\quad \forall z\in\R^d,
\end{align*}
where the first equivalence applies Lemma~\ref{lem:S-lemma} to $f(z) = z^\top Q_1 z + 2 q_1^\top z + r_1$ and $g(z) = z^\top Q_0 z + 2 q_0^\top z + r_0$. As quadratic forms are homogeneous of degree~$2$ as well as continuous, the last statement is equivalent to the desired positive semidefiniteness condition. This observation concludes the proof.
\end{proof}

Proposition~\ref{prop:SDP:reformulation} is particularly useful for deriving finite convex reformulations of the dual worst-case expectation problems over Chebyshev or Gelbrich ambiguity sets; see~\eqref{eq:weak-duality-chebyshev} and~\eqref{eq:weak-duality-gelbrich}. As we will see, the corresponding semi-infinite constraints fail to be convex in $z$, which implies that Proposition~\ref{prop:dual-reformulation} is not applicable.

Finite convex reformulations of the dual worst-case expectation problem~\eqref{eq:worst-case:expectation} are key to solving the DRO problem~\eqref{eq:primal:dro}. They allow us to combine the outer minimization over $x \in \mathcal{X}$ with the inner minimization over the auxiliary decision variables of the dual worst-case expectation problem to obtain a finite convex reformulation of~\eqref{eq:primal:dro}. However, the finite dual reformulations of~\eqref{eq:worst-case:expectation} do not allow us to readily identify worst-case distributions that (asymptotically) attain the supremum of~\eqref{eq:worst-case:expectation}. Such worst-case distributions enable decision-makers to evaluate how a given candidate decision performs under the most challenging conditions, which is the essence of stress testing and contamination experiments; see, {\em e.g.}, \citep{dupacova06contamination}. They also play a pivotal role in optimal uncertainty quantification, where they are used to determine the sharpest possible probabilistic bounds on quantities of interest, given limited information about the underlying probability distributions. We direct the readers to \citep{owhadi2013optimal,ghanem2017handbook} for more details.

To identify a worst-case distribution that attains the supremum of~\eqref{eq:worst-case:expectation}, or to identify a sequence of distributions that attain this supremum asymptotically, we consider the \emph{bi-dual} reformulation of the worst-case expectation problem~\eqref{eq:worst-case:expectation} that results from dualizing the finite convex dual of~\eqref{eq:worst-case:expectation}. The bi-dual can often be interpreted as a restriction of the worst-case expectation problem~\eqref{eq:worst-case:expectation} to a subset of distributions $\P \in \cP$ that are parametrized by finitely many decision variables. Strong duality between problem~\eqref{eq:worst-case:expectation}, its dual and its bi-dual then allows us to conclude that any optimal solution to this bi-dual problem represents a (sequence of) distribution(s) that attains the supremum of~\eqref{eq:worst-case:expectation} (asymptotically).

The idea of extracting worst-case distributions from the finite bi-dual of problem~\eqref{eq:worst-case:expectation} was formalized by~\citet[\ts~4.2]{delage2010distributionally} for Chebyshev ambiguity sets and later extended to optimal transport ambiguity sets by \citet{mohajerin2018data}. In Section~\ref{sec:wc-distribution:chebyshev} we will see that, for the Chebyshev ambiguity set~\eqref{eq:chebyshev-with-moments-in-F} with uncertain moments, the worst-case distributions constitute mixtures of distributions with first and second moments that are determined by the optimal solution of the finite bi-dual problem. For $\phi$-divergence ambiguity sets centered at a discrete distribution~$\hat \P$, Section~\ref{sec:wc-distribution:phi} will show that the worst-case distributions are supported on the atoms of~$\hat \P$ and (if $\phi$ grows at most linearly) on $\argmax_{z\in\cZ}\ell(z)$ with probability weights determined by the optimal solution to the finite bi-dual problem. Similarly, for the optimal transport ambiguity set~\eqref{eq:OT-ambiguity-set} centered at a discrete distribution~$\hat \P$, Section~\ref{sec:wc-distribution:transport} will show that the worst-case distributions constitute mixtures of discrete distributions, with the locations and probability weights of their atoms determined by the optimal solution to the finite bi-dual problem.

\subsection{Chebyshev Ambiguity Sets with Uncertain Moments}
\label{sec:wc-distribution:chebyshev}

Recall that the Chebyshev ambiguity set~\eqref{eq:chebyshev-with-moments-in-F} with uncertain moments is defined~as
\begin{equation*}
\cP = \left\{ \P \in \cP_2(\cZ): \E_\P[Z] = \mu, ~~\E_\P[Z Z^\top] = M ~~ \forall (\mu, M) \in \cF \right\},
\end{equation*}
where $\cF \subseteq \R^d \times \S_+^d$ represents a closed moment uncertainty set and $\cP_2(\cZ)$ stands for the family of all probability distributions on~$\cZ$ with finite second moments. This section combines the duality result for Chebyshev ambiguity sets (\emph{cf.}~Theorem~\ref{thm:duality:Chebyshev}) with the finite dual reformulation of the ensuing semi-infinite program (\emph{cf.}~Proposition~\ref{prop:SDP:reformulation}) to derive an equivalent reformulation of nature's subproblem~\eqref{eq:worst-case:expectation} as a finite-dimensional minimization problem. We also show how the corresponding bi-dual allows us to extract worst-case distributions $\mathbb{P}^\star \in \cP$ that attain the optimal value of~\eqref{eq:worst-case:expectation}. Since the support-only ambiguity sets (\emph{cf.}~Section~\ref{sec:support}), the Markov ambiguity sets (\emph{cf.}~Section~\ref{sec:Markov}), the Chebychev ambiguity sets with known moments (\emph{cf.}~Section~\ref{sec:Chebyshev}) and the mean-dispersion ambiguity sets (\emph{cf.}~Section~\ref{sec:mean-dispersion}) can all be viewed as special instances of the Chebyshev ambiguity set with uncertain moments, our results immediately extend to those ambiguity sets as well, and we do not re-derive the corresponding statements for the sake of brevity. Due to its recent applications in statistics \citep{nguyen2020distributionally}, signal processing \citep{nguyen2019bridging} and control \citep{taskesen2024distributionally}, however, we report the finite dual and bi-dual reformulations of the Gelbrich ambiguity set with moment uncertainty set~\eqref{eq:Gelbrich:F}. All reformulations derived in this section leverage Lemma~\ref{lem:S-lemma}. Thus, they require quadratic representations of the loss function~$\ell$ and the support set $\mathcal{Z}$ as detailed in the following assumption.

\begin{assumption}[Regularity Conditions for Chebyshev Ambiguity Sets]~
\label{assmp:quadratic}
\begin{enumerate}[label=(\roman*)]
	\item \label{assmp:quadratic:loss} The loss function $\ell$ is a point-wise maximum of quadratic functions,
	\begin{align}
		\label{eq:quadratic:loss}
		\ell(z) = \max_{j \in [J]} \ell_j(z) \quad \text{with} \quad \ell_j(z) = z^\top Q_j z + 2 q_j^\top z + q_j^0,
	\end{align}
	where $J \in \N$, $Q_j \in \S^d, q_j \in \R^d$, and $q_j^0 \in \R$ for all $j\in[J]$.
	\item \label{assmp:quadratic:support} The support set $\cZ$ is an ellipsoid of the form
	\begin{align}
		\label{eq:ellipsoid}
		\cZ = \{ z \in \R^d : (z - z_0)^\top Q_0 (z - z_0) \leq 1 \},
	\end{align}
	where $Q_0 \in \S_+^d$ and $z_0 \in \R^d$.
\end{enumerate}
\end{assumption}

Note that Assumption~\ref{assmp:quadratic} does {\em not} impose any convexity conditions on the quadratic component functions $z^\top Q_j z + 2 q_j^\top z + q_j^0$ that make up the loss function~$\ell$.

\begin{theorem}[Finite Dual Reformulation for Chebyshev Ambiguity Sets]
\label{thm:finite:convex:Chebyshev:I}
If~$\cP$ is the Chebyshev ambiguity set~\eqref{eq:chebyshev-with-moments-in-F} with any $\cF \subseteq \R^d \times \S_+^d$ and Assumption~\ref{assmp:quadratic} holds, then the worst-case expectation problem~\eqref{eq:worst-case:expectation} satisfies the weak duality relation
\begin{align*}
	& \sup_{\P \in \cP}~ \E_\P \left[ \ell(Z) \right] \\
	&\leq \left\{
	\begin{array}{cll}
		\inf & \lambda_0 + \delta_\cF^*(\lambda, \Lambda) \\
		\st & \lambda_0 \in \R, \, \lambda \in \R^d, \, \Lambda \in \S^d, \, \alpha \in \R^J_+ \\
		& \begin{bmatrix} \Lambda - Q_j + \alpha_j Q_0 & \frac{1}{2} \lambda - q_j - \alpha_j Q_0 z_0 \\ \left( \frac{1}{2} \lambda - q_j - \alpha_j Q_0 z_0 \right)^\top & \lambda_0 - q_j^0 + \alpha_j (z_0^\top Q_0 z_0 -1) \end{bmatrix} \succeq 0 & \forall j \in [J].
	\end{array}
	\right.
\end{align*}
If $\cF$ is a convex and compact set with $M\succ\mu\mu^\top$ for all~$(\mu,M)\in\rint(\cF)$, then strong duality holds, that is, the above inequality becomes an equality.
\end{theorem}

\begin{proof}
Weak duality follows from Theorem~\ref{thm:duality:Chebyshev} and from the following equivalent reformulation of the semi-infinite constraint in the dual problem~\eqref{eq:weak-duality-chebyshev}.
\begin{align*}
	& ~ \lambda_0 + \lambda^\top z + z^\top \Lambda z \geq \ell(z) \quad \forall z \in \cZ \\
	\iff & ~ \lambda_0 + \lambda^\top z + z^\top \Lambda z \geq z^\top Q_j z + 2 q_j^\top z + q_j^0 \quad \forall z \in \cZ, \forall j \in [J] \\
	\iff & ~ \exists \alpha \in \R^J_+~ \text{with} \\
	& \begin{bmatrix} \Lambda - Q_j + \alpha_j Q_0 & \frac{1}{2} \lambda - q_j - \alpha_j Q_0 z_0 \\ \left( \frac{1}{2} \lambda - q_j - \alpha_j Q_0 z_0 \right)^\top & \lambda_0 - q_j^0 + \alpha_j (z_0^\top Q_0 z_0 -1) \end{bmatrix} \succeq 0 ~~\forall j \in [J]
\end{align*}
Here, the first equivalence holds thanks to Assumption~\ref{assmp:quadratic}\,\ref{assmp:quadratic:loss}, and the second equivalence follows from Proposition~\ref{prop:SDP:reformulation}, which applies because $z_0 \in \rint(\cZ)$ constitutes a Slater point thanks to Assumption~\ref{assmp:quadratic}\,\ref{assmp:quadratic:support}. In addition, as the loss function is quadratic, strong duality follows readily from Theorem~\ref{thm:duality:Chebyshev}.
\end{proof}

Recall next that the Gelbrich ambiguity set~\eqref{eq:Gelbrich:F} is defined in as an instance of the Chebyshev ambiguity set~\eqref{eq:chebyshev-with-moments-in-F} with moment uncertainty set
\begin{align*}
\cF = \left\{(\mu, M) \in \R^d \times \S_+^d~:\, \begin{array}{l}
	\exists \Sigma\,\in\S^d_+ \text{ with } M = \Sigma + \mu \mu^\top, \\ \G \left( (\mu, \Sigma), (\hat \mu, \hat \Sigma) \right) \leq r
\end{array}
\right\}.
\end{align*}
Here, $(\hat \mu, \hat \Sigma)$ is a nominal mean-covariance pair, and~$r \geq 0$ is a size parameter. The next result follows directly from Theorems~\ref{thm:duality:Gelbrich} and~\ref{thm:finite:convex:Chebyshev:I}. We thus omit its proof.

\begin{theorem}[Finite Dual Reformulation for Gelbrich Ambiguity Sets]
\label{thm:finite:convex:Gelbrich:I}
If~$\cP$ is the Chebyshev ambiguity set~\eqref{eq:chebyshev-with-moments-in-F} with~$\cF$ given by~\eqref{eq:Gelbrich:F} and Assumption~\ref{assmp:quadratic} holds, then the worst-case expectation problem~\eqref{eq:worst-case:expectation} satisfies the weak duality relation
\begin{align*}
	&\sup_{\P \in \cP}~ \E_\P \left[ \ell(Z) \right] \\
	&\leq \left\{
	\begin{array}{cll}
		\inf & \lambda_0 + \gamma \big( r^2 - \| \hat \mu \|^2 - \Tr(\hat \Sigma) \big) + \Tr(A_0) + \alpha_0 \\
		\st & \lambda_0 \in \R, \, \alpha_0, \gamma \in \R_+, \, \alpha \in \R^J_+, \, \lambda \in \R^d, \, \Lambda \in \S^d, \, A_0 \in \S_+^d \\
		& \begin{bmatrix} \Lambda - Q_j + \alpha_j Q_0 & \frac{1}{2} \lambda - q_j - \alpha_j Q_0 z_0 \\ \left( \frac{1}{2} \lambda - q_j - \alpha_j Q_0 z_0 \right)^\top & \lambda_0 - q_j^0 + \alpha_j (z_0^\top Q_0 z_0 -1) \end{bmatrix} \succeq 0 & \forall j \in [J] \\[2.5ex]
		& \begin{bmatrix} \gamma I_d- \Lambda & \gamma \hat \Sigma^\half \\[1ex] \gamma \hat \Sigma^\half & A_0 \end{bmatrix} \succeq 0, \; \begin{bmatrix} \gamma I_d- \Lambda & \gamma \hat \mu + \frac{\lambda}{2} \\[1ex] (\gamma \hat \mu + \frac{\lambda}{2})^\top & \alpha_0 \end{bmatrix} \succeq 0.
	\end{array}
	\right.
\end{align*}
If $r>0$, then strong duality holds, that is, the above inequality becomes an equality.
\end{theorem}

In order to characterize the extremal distributions that attain the supremum in the worst-case expectation problem~\eqref{eq:worst-case:expectation} over Chebyshev and Gelbrich ambiguity sets, we first derive the corresponding bi-duals of~\eqref{eq:worst-case:expectation}.

\begin{theorem}[Finite Bi-Dual Reformulation for Chebyshev Ambiguity Sets]
\label{thm:finite:convex:Chebyshev:II}
If~$\cP$ is the Chebyshev ambiguity set~\eqref{eq:chebyshev-with-moments-in-F} with $\cF \subseteq \R^d \times \S_+^d$ and Assumption~\ref{assmp:quadratic} holds, then the worst-case expectation problem~\eqref{eq:worst-case:expectation} satisfies the weak duality relation
\begin{align}
	&\sup_{\P \in \cP} ~~ \E_\P \left[ \ell(Z) \right] \notag \\
	&\hspace{-1em}\leq \left\{
	\begin{array}{c@{~}l@{~}l}
		\sup & \displaystyle \sum_{j \in [J]} \Tr(Q_j \Theta_j) + 2 q_j^\top \theta_j + q_j^0 p_j \\ [2ex]
		\st & \mu \in \R^d, \, M \in \S_+^d, \, p_j \in \R_+, \, \theta_j \in \R^d,\, \Theta_j \in \S_+^d & \forall j \in [J] \\ [1ex]
		& \begin{bmatrix} \Theta_j & ~\theta_j \\ \theta_j^\top & ~p_j \end{bmatrix} \succeq 0, \, \Tr(Q_0 \Theta_j) - 2 z_0^\top Q_0 \theta_j + z_0^\top Q_0 z_0 p_j \leq p_j & \forall j \in [J] \\[2ex]
		& \displaystyle \sum_{j \in [J]} p_{j} = 1, ~ \mu= \sum_{j \in [J]} \theta_j,~ M = \sum_{j \in [J]} \Theta_j, ~ (\mu, M) \in \cF.
	\end{array}
	\right. \hspace{-1ex}
	\label{eq:finite:chebyshev:II}
\end{align}
If $\cF$ is a convex and compact set with $M\succ\mu\mu^\top$ for all~$(\mu,M)\in\rint(\cF)$, then strong duality holds, that is, the inequality~\eqref{eq:finite:chebyshev:II} becomes an equality.
\end{theorem}

\begin{proof}
By decomposing the Gelbrich ambiguity set into Chebyshev ambiguity sets of the form $\cP(\mu, M) = \{\P \in \cP(\cZ) : \E_\P[Z] = \mu, \E_\P[Z Z^\top] = M \}$, we obtain
\begin{align}
	\label{eq:sup:sup}
	\sup_{\P \in \cP} \, \E_\P[\ell(Z)] = \sup_{(\mu, M) \in \cF} \, \sup_{\P \in \cP(\mu, M)} \, \E_\P[\ell(Z)].
\end{align}
The inner maximization problem on the right hand side of~\eqref{eq:sup:sup} represents a worst-case expectation problem over an instance of the ambiguity set~\eqref{eq:chebyshev-with-moments-in-F} with the moment uncertainty set being the singleton $\{ (\mu, M) \}$. The support function of this singleton is given by $\delta^*_{\{ (\mu, M) \}} (\lambda, \Lambda) = \lambda^\top \mu + \Tr(\Lambda M)$. Thus, Theorem~\ref{thm:finite:convex:Chebyshev:I} implies that the inner supremum on the right hand side of~\eqref{eq:sup:sup} is bounded above by
\begin{align*}
	\begin{array}{cll}
		\inf & \lambda_0 + \lambda^\top \mu + \Tr(\Lambda M) \\[1ex]
		\st & \lambda_0 \in \R, \, \lambda \in \R^d, \, \Lambda \in \S^d, \, \alpha \in \R^J_+ \\ [1ex]
		& \begin{bmatrix} \Lambda - Q_j + \alpha_j Q_0 & \frac{1}{2} \lambda - q_j - \alpha_j Q_0 z_0 \\ \left( \frac{1}{2} \lambda - q_j - \alpha_j Q_0 z_0 \right)^\top & \lambda_0 - q_j^0 + \alpha_j (z_0^\top Q_0 z_0 -1) \end{bmatrix} \succeq 0 & \forall j \in [J].
	\end{array}
\end{align*}
The dual of this semidefinite program can be represented as
\begin{align*}
	\begin{array}{cll}
		\sup & \displaystyle \sum_{j \in [J]} \Tr(Q_j \Theta_j) + 2 q_j^\top \theta_j + q_j^0 p_j \\ [1ex]
		\st & p_j \in \R_+, \; \theta_j \in \R^d,\; \Theta_j \in \S_+^d & \forall j \in [J] \\ [1ex]
		& \displaystyle \begin{bmatrix} \Theta_j & \theta_j \\ \theta_j^\top & p_j \end{bmatrix} \succeq 0, \; \Tr(Q_0 \Theta_j) - 2 z_0^\top Q_0 \theta_j + z_0^\top Q_0 z_0 p_j \leq p_j & \forall j \in [J] \\[2ex]
		& \displaystyle \sum_{j \in [J]} p_{j} = 1 , \, \sum_{j \in [J]} \theta_j = \mu, \sum_{j \in [J]} \Theta_j = M.
	\end{array}
\end{align*}
Strong duality holds because the primal minimization problem admits a Slater point. Indeed, by defining~$\Lambda=\lambda_0 I_d$ and setting~$\lambda_0$ to a large value, one can ensure that the linear matrix inequality in the primal problem holds strictly. Replacing the inner supremum on the right hand side of~\eqref{eq:sup:sup} with the above dual semidefinite program yields the upper bound in~\eqref{eq:finite:chebyshev:II}. If $\cF$ is convex and compact with $M\succ\mu\mu^\top$ for all~$(\mu,M)\in\rint(\cF)$, then~\eqref{eq:finite:chebyshev:II} becomes an equality thanks to Theorem~\ref{thm:finite:convex:Chebyshev:I}.
\end{proof}

Note that the bi-dual reformulation in~\eqref{eq:finite:chebyshev:II} is solvable whenever~$\cF$ is compact. Indeed, its objective function is ostensibly continuous. In addition, it is easy to verify that its feasible region is compact provided that~$\cF$ is compact.

\begin{theorem}[Finite Bi-Dual Reformulation for Gelbrich Ambiguity Sets]
\label{thm:finite:convex:Gelbrich:II}
If~$\cP$ is the Chebyshev ambiguity set~\eqref{eq:chebyshev-with-moments-in-F} with~$\cF$ given by~\eqref{eq:Gelbrich:F} and Assumption~\ref{assmp:quadratic} holds, then the worst-case expectation problem~\eqref{eq:worst-case:expectation} satisfies the weak duality relation
\begin{align}
	&\sup_{\P \in \cP} \E_\P \left[ \ell(Z) \right] \notag \\
	&\leq \left\{
	\begin{array}{cl@{~~}l}
		\max & \displaystyle \sum_{j \in [J]} \Tr(Q_j \Theta_j) + 2 q_j^\top \theta_j + q_j^0 p_j \\ [2ex]
		\st & \mu \in \R^d, ~ M, U \in \S_+^d, ~ C \in \R^{d \times d} \\
		& p_j \in \R_+, ~ \theta_j \in \R^d,~ \Theta_j \in \S_+^d & \forall j \in [J] \\[0.5ex]
		& \begin{bmatrix} M - U & C \\ C^\top & \hat \Sigma \end{bmatrix} \succeq 0, \, \begin{bmatrix} U & \mu \\ \mu^\top & 1 \end{bmatrix} \succeq 0, \, \begin{bmatrix} \Theta_j & \theta_j \\ \theta_j^\top & p_j \end{bmatrix} \succeq 0 & \forall j \in [J] \\ [2ex]
		& \Tr(Q_0 \Theta_j) - 2 z_0^\top Q_0 \theta_j + z_0^\top Q_0 z_0 p_j \leq p_j &   \forall j \in [J] \\[0.5ex]
		& \displaystyle \sum_{j \in [J]} p_{j} = 1, ~ \mu= \sum_{j \in [J]} \theta_j, ~ M = \sum_{j \in [J]} \Theta_j \\[3ex]
		& \| \hat \mu\|_2^2 - 2 \mu^\top \hat \mu + \Tr (M + \hat \Sigma - 2C) \leq r^2.
	\end{array}
	\right.
	\label{eq:finite:Gelbrich:II}
\end{align}
If $r>0$, then strong duality holds, that is, the above inequality becomes an equality.
\end{theorem}

The proof of Theorem~\ref{thm:finite:convex:Gelbrich:II} follows from Proposition~\ref{prop:U:Gelbrich:SDP} and Theorem~\ref{thm:finite:convex:Chebyshev:II} and is thus omitted. We are now ready to construct extremal distributions $\P^\star \in \cP(\cZ)$ that attain the supremum of the worst-case expectation problem~\eqref{eq:worst-case:expectation} over the Chebyshev ambiguity set~\eqref{eq:chebyshev-with-moments-in-F}. To this end, fix any maximizer $(\mu^\star, M^\star, p^\star, \theta^\star, \Theta^\star)$ of the bi-dual problem~\eqref{eq:finite:chebyshev:II}, which exists if $\cF$ is compact. Next, define the index sets
\begin{align*}
\mspace{-2mu}
\cJ^\infty = \big\{ j \in [J]: p_j^\star = 0, \, \Theta_j^\star \neq 0 \big\} \quad\text{and}\quad \cJ^+ = \big\{ j \in [J]: p_j^\star > 0 \big\},
\end{align*}
and define $\cJ = \cJ^+ \cup \cJ^\infty$. The extremal distributions $\P^\star$ will be constructed as mixtures of constituent distributions $\mathbb{P}_j$, $j\in \cJ$, corresponding to different pieces of the loss function~$\ell$. In the following, we use~$\P \sim (\mu, M)$ to indicate that the distribution~$\P$ has mean~$\mu$ and second-order moment matrix~$M$. Note that if~$\cZ= \{ z \in \R^d : (z - z_0)^\top Q_0 (z - z_0) \leq 1 \}$ is the ellipsoid from Assumption~\ref{assmp:quadratic}\,\ref{assmp:quadratic:support} and~$\P\in\cP(\cZ)$ is a distribution supported on~$\cZ$ with~$\P \sim (\mu, M)$, then we have
\[
1\geq \E_\P\left[ (Z - z_0)^\top Q_0 (Z - z_0)\right] = \Tr(Q_0 M) + 2 z_0^\top \mu + z_0^\top Q_0 z_0.
\]
The inequality in the above expression holds because $\P\in\cP(\cZ)$, and the equality holds because $\P \sim (\mu, M)$. The following lemma by
\citet[Proposition~6.1]{hanasusanto2015distributionally} shows the reverse implication. That is, if $\mu$ and $M$ satisfy the above inequality, then there is a (discrete) distribution $\P \sim (\mu, M)$ supported on~$\cZ$.

\begin{lemma}[Distributions on Ellipsoids with Given Moments]
\label{lem:distribution-on-ellipsoid}
If~$\cZ$ is the ellipsoid from Assumption~\ref{assmp:quadratic}\,\ref{assmp:quadratic:support}, and if $\Tr(Q_0 M) + 2 z_0^\top \mu + z_0^\top Q_0 z_0\leq 1$ for some $M\in\S_+^d$ and~$\mu\in\R^d$ with $M\succeq \mu\mu^\top$, then there exists a discrete distribution $\P\in\cP(\cZ)$ with at most $2d$ atoms that satisfies $\P \sim (\mu, M)$.
\end{lemma}

The proof of Lemma~\ref{lem:distribution-on-ellipsoid} is simple but tedious and thus omitted.

\begin{theorem}[Extremal Distributions of Chebyshev Ambiguity Sets]
\label{thm:finite:convex:Chebyshev:III}
If all conditions of Theorem~\ref{thm:finite:convex:Chebyshev:II} for weak as well as strong duality are satisfied and that $(\mu^\star, M^\star, p^\star, \theta^\star, \Theta^\star )$ solves~\eqref{eq:finite:chebyshev:II}, then the following hold.
\begin{enumerate}[label=(\roman*)]
	\item \label{thm:chebyshev:extremal} If $\cJ^\infty = \emptyset$, then there exist discrete distributions $\P_j^\star \sim (\theta_j^\star / p_j^\star, \Theta_j^\star / p_j^\star)$ supported on~$\cZ$ for all $j \in \cJ^+$, and~\eqref{eq:worst-case:expectation} is solved by $\P^\star = \sum_{j \in \cJ^+} p_j^\star\, \P^\star_j$. In addition, we have $\P^\star \sim (\mu^\star, M^\star)$, and $\P^\star$ is supported on~$\cZ$.
	\item \label{thm:chebyshev:assymptotic} If $\cJ^\infty \neq \emptyset$, then there exist discrete distributions $\P_j^m \sim (\theta_j^\star / p_j^m, \Theta_j^\star / p_j^m)$ supported on~$\cZ$ for all~$j\in\cJ$, where $p_j^m= (1 - |\cJ^\infty| /m) p_j^\star$ for $j\in\cJ^+$ and $p_j^m= 1/m$ for $j\in\cJ^\infty$, and where~$m$ is any integer with~$m\geq |\cJ^\infty|$. In addition, \eqref{eq:worst-case:expectation} is asymptotically solved by $\P^m = \sum_{j \in \cJ} p_j^m \, \P_j^m$ as $m$ grows.
\end{enumerate}
\end{theorem}

\begin{proof}
As for assertion~\ref{thm:chebyshev:extremal}, the constraints of problem~\eqref{eq:finite:Gelbrich:II} imply that
\[
\Tr(Q_0 \Theta_j^\star/p_j^\star) - 2 z_0^\top Q_0 \theta_j^\star/p_j^\star + z_0^\top Q_0 z_0 \leq 1 
\]
and
\[
\begin{bmatrix} \Theta^\star_j & \theta^\star_j \\ (\theta^\star_j)^\top & p^\star_j \end{bmatrix} \succeq 0 \quad\iff\quad \Theta^\star_j/p^\star_j \succeq (\theta^\star_j/p^\star_j)(\theta^\star_j/p^\star_j)^\top
\]
for all $j\in\cJ^+$. Lemma~\ref{lem:distribution-on-ellipsoid} thus guarantees that there exist discrete distributions~$\P_j^\star \sim (\theta_j^\star / p_j^\star, \Theta_j^\star / p_j^\star)$, $j \in \cJ^+$, all of which are supported on~$\cZ$. Consequently, $\P^\star = \sum_{j \in \cJ^+} p_j^\star\, \P^\star_j$ is also supported on~$\cZ$. In addition, we have
\begin{align*}
	& \mathbb{E}_{\P^\star} [ Z ] = \sum_{j\in \cJ^+} p_j^\star \cdot \mathbb{E}_{\P^\star_j} [ Z ] = \sum_{j\in \cJ^+} p_j^\star \cdot \theta_j^\star / p_j^\star = \sum_{j\in \cJ^+} \theta_j^\star = \mu^\star 
\end{align*}
and
\begin{align*}
	\mathbb{E}_{\P^\star} [ Z Z^\top ] = \sum_{j\in \cJ^+} p_j^\star \cdot \mathbb{E}_{\P^\star_j} [ Z Z^\top ] = \sum_{j\in \cJ^+} p_j^\star \cdot \Theta^\star_j/p^\star_j = \sum_{j\in \cJ^+} \Theta_j^\star = M^\star,
\end{align*}
that is, $\P^\star \sim (\mu^\star, M^\star)$. As $(\mu^\star, M^\star) \in \mathcal{F}$, it is now clear that $\P^\star \in \cP$ and that
\begin{align*}
	\E_{\P^\star}[\ell(Z)] \leq \sup_{\P \in \cP} \, \E_\P[\ell(Z)] = \sum_{j \in [J]} \Tr(Q_j \Theta_j^\star) + 2 q_j^\top \theta_j^\star + q_j^0 p_j^\star,
\end{align*}
where the equality follows from strong duality as established in Theorem~\ref{thm:finite:convex:Chebyshev:II}. At the same time, the definition of~$\P^\star$ as a mixture distribution and the definition of~$\ell$ in~\eqref{eq:quadratic:loss} as a pointwise maximum of quadratic component functions implies that
\begin{align*}
	\mspace{-8mu}
	\E_{\P^\star}[\ell(Z)] 
	\geq \sum_{j \in \cJ^+} p_j^\star \cdot \E_{\P_j^\star}[\ell_j (Z)]  
	= \sum_{j \in [J]} \Tr(Q_j \Theta_j^\star) + 2 q_j^\top \theta_j^\star + q_j^0 p_j^\star.
\end{align*}
Specifically, the inequality holds because $\ell\geq \ell_j$ for every~$j\in[J]$, and the equality holds because $\theta^\star_j = 0$ and $\Theta_j^\star = 0$ whenever $p^\star_j = 0$. Indeed, if $p_j^\star = 0$, then $\Theta_j^\star = 0$ because the index set~$\cJ^\infty$ is empty, and     
the linear matrix inequality in~\eqref{eq:finite:chebyshev:II} implies that $\theta_j^\star = 0$ whenever $\Theta_j^\star = 0$. The above inequalities thus ensure that $\P^\star$ solves the worst-case expectation problem~\eqref{eq:worst-case:expectation}. This completes the proof of assertion~\ref{thm:chebyshev:extremal}.

Next, we address assertion~\ref{thm:chebyshev:assymptotic}. Similar arguments as in the proof of assertion~\ref{thm:chebyshev:extremal} can be used to show that $\P^m \in \cP$ for every $m \geq |\cJ^\infty|$. This implies that $\E_{\P^m} [\ell(Z)] \leq \sup_{\P \in \cP} \, \E_\P[\ell(Z)]$ whenever $m \geq |\cJ^\infty|$. In addition, we observe that
\begin{align*}
	\lim_{m \to \infty} \E_{\P^m} [\ell(Z)] 
	&\geq \lim_{m \to \infty} \sum_{j \in \cJ} p_j^m \cdot \E_{\P^m} [\ell_j (Z)]
	= \sum_{j \in \cJ} \lim_{m \to \infty} p_j^m \cdot \E_{\P^m} [\ell_j (Z)] \\
	&= \sum_{j \in [J]} \Tr(Q_j \Theta_j^\star) + 2 q_j^\top \theta_j^\star + q_j^0 p_j^\star
	= \sup_{\P \in \cP} \, \E_\P[\ell(Z)],
\end{align*}
where the second equality exploits the definition of~$\P^m$ and the third equality follows from strong duality as established in Theorem~\ref{thm:finite:convex:Chebyshev:II}. This completes the proof.
\end{proof}

Theorem~\ref{thm:finite:convex:Chebyshev:III} also applies to the Gelbrich ambiguity set, which constitutes a Chebyshev ambiguity set of the form~\eqref{eq:chebyshev-with-moments-in-F} with~$\cF$ given by~\eqref{eq:Gelbrich:F}. The extremal distribution~$\P^\star$ identified in Theorem~\ref{thm:finite:convex:Chebyshev:III}\,\ref{thm:chebyshev:extremal} constitutes a mixture of different distributions~$\P^\star_j$, each of which corresponds to a component~$\ell_j$ of the loss function~$\ell$; see Assumption~\ref{assmp:quadratic}\,\ref{assmp:quadratic:loss}. The mixture components~$\P^\star_j$ may be set to {\em any} distributions on~$\cZ$ that satisfy the prescribed moment conditions. Note that {\em discrete} distributions consistent with these requirements are guaranteed to exist thanks to Lemma~\ref{lem:distribution-on-ellipsoid}. However, if $\cZ = \R^d$, say, then one could also set~$\P^\star_j$ to the Gaussian distribution with the given first and second moments. From the proof of Theorem~\ref{thm:finite:convex:Chebyshev:III} it becomes clear that~$\P^\star_j$ must be supported on $\{z\in\cZ:\ell_j(z)\geq \ell_{j'}(z)~\forall j'\neq j\}$, which is generically nonconvex. Therefore, \citet[\ts~2.2]{kuhn2019wasserstein} conjectured that the construction of~$\P^\star_j$ is NP-hard. From the proof of Lemma~\ref{lem:distribution-on-ellipsoid} in \citep[\ts~6]{hanasusanto2015distributionally} it becomes clear, however, that~$\P^\star_j$ can be constructed efficiently. Similar comments are in order for the distributions $\P^m_j$ appearing in Theorem~\ref{thm:finite:convex:Chebyshev:III}\,\ref{thm:chebyshev:assymptotic}.

If $\cJ^\infty\neq \emptyset$, then the extremal distributions constructed in Theorem~\ref{thm:finite:convex:Chebyshev:III} contain diverging mixture components whose covariance matrices explode along certain recession directions of the support set~$\cZ$ (that is, along the eigenvectors of~$\Theta_j^\star$, $j\in \cJ^\infty$, corresponding to non-zero eigenvalues). However, these diverging mixture components are assigned weights that decay with their variances such that the covariance matrix of the entire mixture distribution remains bounded.

The following lemma establishes a sufficient condition for $\cJ^\infty$ to be empty, which ensures via Theorem~\ref{thm:finite:convex:Chebyshev:III}\,\ref{thm:chebyshev:extremal} that problem~\eqref{eq:worst-case:expectation} is solvable. 

\begin{lemma}
\label{lem:sufficient:chebyshev}
If all conditions of Theorem~\ref{thm:finite:convex:Chebyshev:III} are satisfied and the support set~$\cZ$ defined in~\eqref{eq:ellipsoid} is compact, then~$\cJ^\infty = \emptyset$, and thus problem~\eqref{eq:worst-case:expectation} is solvable.
\end{lemma}

\begin{proof}
If $p_j^\star = 0$ for some $j \in [J]$, then the linear matrix inequality in~\eqref{eq:finite:chebyshev:II} implies that $\theta_j^\star = 0$. Consequently, the $j$-th trace inequality simplifies to $\Tr(Q_0 \Theta_j^\star) \leq 0$. As~$Q_0 \succ 0$ because~$\cZ$ is compact, we thus find that $\Theta_j^\star = 0$. In summary, we have shown that~$p^\star_j=0$ implies $\Theta_j^\star = 0$, and therefore $\cJ^\infty$ is empty as desired.
\end{proof}

We conclude this section with some remarks on worst-case expectation problems with more generic moment ambiguity sets. Translated into our terminology,  \citet{richter1957parameterfreie} and \citet{rogosinski1958moments} show that if $\cP = \{ \P \in \cP(\cZ): \E_\P[f(Z)] =\mu \}$ for some $f:\cZ\to\R^m$ and~$\mu\in\R^m$, and if~\eqref{eq:worst-case:expectation} is solvable, then the supremum in~\eqref{eq:worst-case:expectation} is attained by a {\em discrete} distribution with at most~$m+2$ atoms. See \citep[Theorem~7.32]{shapiro2009lectures} for modern proof of this result. Note also that, under the given assumptions, the worst-case expectation problem~\eqref{eq:worst-case:expectation} can be recast as 
\begin{equation}
\label{eq:standard-LP}
\sup_{\varrho\in \cM_+(\cZ)} \left\{ \int_{\cZ} \ell(z)\,\diff\rho(z) \;:\; \int_\cZ \diff\rho(z)=1,~~\int_\cZ f(z)\, \diff\rho(z)=\mu \right\}.
\end{equation}
Problem~\eqref{eq:standard-LP} constitutes an {\em infinite}-dimensional linear program over the non-negative Borel measures on~$\cZ$ with $m+1$ linear equality constraints. Every {\em finite}-dimensional linear program with non-negative variables and $m+1$ equality constraints is known to admit an optimal basic feasible solution with at most $m+1$ non-zero entries. The infinite-dimensional analog of a basic feasible solution is a discrete measure with at most $m+1$ atoms. Accordingly, one can prove that if~\eqref{eq:standard-LP} is solvable, then its supremum is attained by a measure with at most $m+1$ atoms \citep[Corollary~5 and Proposition~6(v)]{pinelis2016extreme}. This result strengthens the Richter-Rogosinski theorem. However, the minimum number of atoms required for an optimal measure cannot be reduced beyond $m+1$ without additional assumptions.

The above reasoning implies that the worst-case expectation problem~\eqref{eq:worst-case:expectation} and its reformulation~\eqref{eq:standard-LP} as a semi-infinite linear program can be reduced to a finite-dimensional optimization problem over the locations and probabilities of the $m+1$ atoms of a discrete measure. Finite reductions of this type are routinely studied in optimal uncertainty quantification \citep{owhadi2013optimal}. However, they generically represent {\em nonconvex} optimization problems. Indeed, even the integral of a linear function with respect to a discrete measure involves products of the probabilities and the coordinates of the measure's atoms. If~\eqref{eq:standard-LP} is solvable and~$\ell$ is representable as a pointwise maximum of~$J$ concave functions, then the $m+1$ atoms of an extremal measure can be further condensed. That is, using an induction argument and an iterative application of Jensen's inequality, one can show that~\eqref{eq:standard-LP} is solved by a discrete measure with at most~$J$ atoms \citep[Lemma~3.1]{han2015convexUQ}. This result is significant even though~$J$ is not necessarily smaller than~$m+1$. It implies that~\eqref{eq:standard-LP} admits a finite reduction that optimizes over discrete measures with~$J$ atoms. And this (nonconvex) finite reduction is intimately related to the dual problem~\eqref{eq:weak-duality-moments} derived in Theorem~\ref{thm:duality:moment} through a `{\em primal-worst-equals-dual-best}' duality scheme for robust optimization problems \citep{beck2009duality}. Specifically, \eqref{eq:weak-duality-moments} can be viewed as a `{\em primal-worst}' robust optimization problem, and the finite reduction corresponding to discrete measures with~$J$ atoms can be viewed as the corresponding `{\em dual-best}' optimization problem \citep{zhen2023unification}. These problems share the same optimal value under mild regularity conditions. In addition, the (dual best) finite reduction can be convexified by applying a variable transformation and a perspectification trick \citep[Theorem~1.1]{han2015convexUQ}. The same convex reformulation can also be obtained by dualizing the finite dual reformulation of the (primal worst) problem~\eqref{eq:weak-duality-moments} as outlined in Section~\ref{sec:proof-strategy-tractability}. For further details we refer to \citep{zhen2023unification}.

\subsection{$\phi$-Divergence Ambiguity Sets}
\label{sec:wc-distribution:phi}
Recall that the $\phi$-divergence ambiguity set~\eqref{eq:phi-divergence-ambiguity-set} is defined~as
\begin{align*}
\cP = \left\{ \P \in \cP(\cZ) \, : \, \D_\phi(\P, \hat \P) \leq r \right\}, 
\end{align*}
where~$r$ is a size parameter, $\phi$ is an entropy function in the sense of Definition~\ref{def:phi}, $\D_\phi$ is the corresponding $\phi$-divergence in the sense of Definition~\ref{def:D_phi}, and~$\hat\P\in\cP(\cZ)$ is a reference distribution. In the following, we first demonstrate that the worst-case expectation problem~\eqref{eq:worst-case:expectation} over a $\phi$-divergence ambiguity sets can be reformulated as a finite convex program whenever~$\hat \P$ is discrete and $\ell$ is real-valued. 

\begin{assumption}[Discrete Reference Distribution]
\label{assmp:discrete:reference}
We have $\hat{\P} = \sum_{i \in [N]} \hat p_i \delta_{\hat z_i}$ for some $N \in \N$, where the probabilities~$\hat p_i$, $i\in[N]$, are strictly positive and sum to~$1$, and where~$\hat z_i \in \cZ$ for every $i\in[N]$. In addition, $\ell(z)\in\R$ for all $z\in\cZ$.
\end{assumption} 

The requirement that~$\hat p_i$ be positive for every $i\in[N]$ is non-restrictive because atoms with zero probability can simply be eliminated without changing~$\hat\P$. 

\begin{theorem}[Finite Dual Reformulation for $\phi$-Divergence Ambiguity Sets]
\label{thm:finite:convex:phi:I}
If~$\cP$ is the $\phi$-divergence ambiguity set~\eqref{eq:phi-divergence-ambiguity-set} and Assumption~\ref{assmp:discrete:reference} holds, then the worst-case expectation problem~\eqref{eq:worst-case:expectation} satisfies the weak duality relation
\begin{align}
	\label{eq:finite-dual-phi}
	\sup_{\P \in \cP} \; \E_\P[\ell(Z)] \leq \left\{ \begin{array}{cll} \displaystyle \inf_{\lambda_0\in\R,\lambda\in\R_+} & \displaystyle \lambda_0  + \lambda r + \sum_{i \in [N]} \hat p_i \cdot (\phi^*)^\pi \left( \ell(\hat z_i) - \lambda_0, \lambda \right)  \\[3ex]
		\st & \displaystyle  \lambda_0+ \lambda\,\phi^\infty(1) \geq \sup_{z\in\cZ} \ell(z),
	\end{array}\right.
\end{align}
where the product $\lambda\,\phi^\infty(1)$ is assumed to evaluate to~$\infty$ if~$\lambda=0$ and~$\phi^\infty(1) =\infty$. 
If~$r>0$ and~$\phi$ is continuous at~$1$, then strong duality holds, that is, the above inequality becomes an equality.
\end{theorem}

Theorem~\ref{thm:finite:convex:phi:I} is an immediate corollary of Theorem~\ref{eq:weak-duality-phi-divergence}. Indeed, problem~\eqref{eq:finite-dual-phi} is obtained from~\eqref{eq:weak-duality-phi-divergence} by re-expressing the integral with respect to the discrete reference distribution~$\hat \P$ as a weighted sum. Thus, no proof is required. Recall now that the {\em restricted} $\phi$-divergence ambiguity set is defined as the set of all distributions~$\P\in \cP$ with~$\P\ll \hat \P$. It is straightforward to verify that if~$\cP$ is discrete, then the corresponding worst-case expectation problem~\eqref{eq:worst-case:expectation} admits a finite convex reformulation that is given by a relaxation of~\eqref{eq:finite-dual-phi} without constraints. Details are omitted for brevity. Next, we derive a finite convex program dual to~\eqref{eq:finite-dual-phi} that allows us to construct an extremal distribution.

\begin{theorem}[Finite Bi-Dual Reformulations for $\phi$-Divergence Ambiguity Sets]
\label{thm:finite:convex:phi:II}
If~$\cP$ is the $\phi$-divergence ambiguity set~\eqref{eq:phi-divergence-ambiguity-set}, Assumption~\ref{assmp:discrete:reference} holds, $r>0$ and~$\phi$ is continuous at~$1$, then problem~\eqref{eq:worst-case:expectation} satisfies the strong duality relation
\begin{align}
	\label{eq:finite:phi:II}
	\sup_{\P \in \cP} \; \E_\P[\ell(Z)] 
	= \left\{ \begin{array}{cll} \displaystyle \max_{p_0,\ldots, p_N\in\R_+} & \displaystyle p_0 \overline \ell + \sum_{i \in [N]}  p_i \ell(\hat z_i) \\[3ex]
		\st & \displaystyle  p_0+ \sum_{i \in [N]}  p_i =1\\
		& \displaystyle p_0\phi^\infty(1) + \sum_{i \in [N]}  \hat p_i \phi\left(\frac{p_i}{\hat p_i}\right) \leq r,
	\end{array}\right.
\end{align}
where $\overline\ell$ is a shorthand for $\sup_{z\in\cZ} \ell(z)$. The product $p_0\phi^\infty(1)$ is assumed to equal~$0$ if $p_0=0$ and $\phi^\infty(1)=\infty$. Similarly, $p_0\overline\ell$ is assumed to equal~$0$ if~$p_0=0$ and~$\overline\ell=\infty$.
\end{theorem}

The finite bi-dual reformulation~\eqref{eq:finite:phi:II} can readily be derived from the primal worst-case expectation problem~\eqref{eq:worst-case:expectation} or from its finite dual reformulation~\eqref{eq:finite-dual-phi}. We find it insightful to derive~\eqref{eq:finite:phi:II} from~\eqref{eq:finite-dual-phi}. This is also more consistent with the general proof strategy outlined in Section~\ref{sec:proof-strategy-tractability}. We will briefly touch on the derivation of~\eqref{eq:finite:phi:II} from the primal problem~\eqref{eq:worst-case:expectation} after the proof.

\begin{proof}[Proof of Theorem~\ref{thm:finite:convex:phi:II}]
Assume first that $\phi^\infty(1)<\infty$. Under the assumptions stated in the theorem, the worst-case expectation problem~\eqref{eq:worst-case:expectation} and its dual~\eqref{eq:finite-dual-phi} share the same optimal value thanks to Theorem~\ref{thm:finite:convex:phi:I}. By dualizing the single explicit constraint in~\eqref{eq:weak-duality-phi-divergence} and using Lemma~\ref{lem:conjugate-perspective}\,\ref{lem:conjugate-perspective-1}, we thus find
\begin{align*}
	\sup_{\P \in \cP} \; \E_\P[\ell(Z)] &=  \inf_{\lambda_0\in\R,\lambda\in\R_+} \lambda_0  + \lambda r + \sum_{i \in [N]} \hat p_i \bigg(\sup_{y_i\in\R_+} y_i\left( \ell(\hat z_i)-\lambda_0\right)-\lambda\phi(y_i)\bigg) \\
	&\qquad + \sup_{p_0\in\R_+} \left( \overline \ell -\lambda_0 -\lambda\phi^\infty(1)\right)\,p_0.
\end{align*}
Interchanging the infima and suprema and rearranging terms further yields
\begin{align*}
	& \sup_{\P \in \cP} \; \E_\P[\ell(Z)] \\[-1ex] & =  \sup_{p_0,y_1,\ldots,y_N\in\R_+} p_0\overline \ell + \sum_{i\in[N]} \hat p_i y_i \ell(\hat z_i) + \inf_{\lambda_0\in\R} \bigg(1-p_0- \sum_{i\in[N]} \hat p_i y_i \bigg) \lambda_0 \\
	&\hspace{2cm} +\inf_{\lambda\in\R_+} \bigg(r-p_0\phi^\infty(1)- \sum_{i\in[N]} \phi(y_i)  \bigg)  \lambda \\
	& = \left\{ \begin{array}{cll} \displaystyle \sup_{p_0,y_0\ldots, y_N\in\R_+} & \displaystyle p_0\overline \ell + \sum_{i \in [N]}  \hat p_i y_i \ell(\hat z_i) \\[3ex]
		\st & \displaystyle  p_0+ \sum_{i \in [N]}  \hat p_i y_i=1,~~ p_0\phi^\infty(1) + \sum_{i \in [N]}  \hat p_i \phi(y_i) \leq r.
	\end{array} \right.
\end{align*}
The first equality in the above expression follows from strong duality, which holds because~$r>0$ and~$\phi$ is continuous at~$1$. Indeed, these conditions ensure that the resulting maximization problem admits a Slater point with $p_0=0$ and $y_i=1$ for all $i\in[N]$. The substitution $p_i\leftarrow \hat p_i y_i$, $i\in[N]$, finally shows that the obtained problem is equivalent to~\eqref{eq:finite:phi:II}. This proves the claim for $\phi^\infty(1)<\infty$.

Suppose next that $\phi^\infty(1) = \infty$ in which case $0 \, \phi^\infty(1)$ evaluates to~$\infty$. Hence, the constraint in~\eqref{eq:weak-duality-phi-divergence} is satisfied for any~$(\lambda_0,\lambda)\in\R\times\R_+$ and is thus redundant. By repeating the steps from the first part of the proof with obvious minor modifications shows that~\eqref{eq:finite:phi:II} still holds if we assume that~$p_0\phi^\infty(1)$ and $p_0\overline\ell$ evaluate to~$0$ when~$p_0=0$. Indeed, this means that~$p_0=0$ is the only feasible solution in~\eqref{eq:finite:phi:II}, and problem~\eqref{eq:finite:phi:II} can be simplified by eliminating~$p_0$ altogether.
\end{proof}

The finite bi-dual reformulation on the right hand side of~\eqref{eq:finite:phi:II} has a linear objective function and a compact convex feasible region. Therefore, it is solvable thanks to Weierstrass' maximum theorem. In particular, note that the feasible region is a subset of the probability simplex in~$\R^{N+1}$. If there exists a worst-case scenario $\hat z_0\in\argmax_{z\in\cZ} \ell(\hat z)$ (which must satisfy $\ell(z_0)=\overline\ell$), then any maximizer $p^\star$ of the bi-dual can be used to construct an extremal distribution $\P^\star= \sum_{i=0}^N p^\star_i \delta_{\hat z_i}$ for the worst-case expectation problem~\eqref{eq:worst-case:expectation}. Indeed, the constraints of problem~\eqref{eq:finite:phi:II} ensure that~$p_0^\star,\ldots,p_N^\star$ are non-negative probabilities that sum to~$1$. Thus,~$\P^\star$ is a valid distribution supported on~$\cZ$. Setting $\rho=\sum_{i=0}^N \delta_{\hat z_i}$, we also find 
\begin{align*}
\D_\phi(\P^\star,\hat\P) & = \int_{\cZ}
\phi^\pi \left( \frac{\diff \P}{\diff \rho}(z), \frac{\diff \hat{\P}}{\diff \rho}(z) \right) \diff \rho(z)\\
&= \phi^\pi \left(p_0^\star,0\right) + \sum_{i\in[N]} \phi^\pi \left( p^\star_i, \hat p_i\right) \leq r,
\end{align*}
where the first equality exploits the definition of~$\D_\phi$, and the second equality exploits our choice of the reference distribution~$\rho$. In addition, the inequality follows from the constraints of problem~\eqref{eq:finite:phi:II} and the observation that
\[
\phi^\pi (p_0^\star,0) = \phi^\infty(p_0^\star) = p_0^\star \phi^\infty(1).
\]
This confirms that~$\P^\star$ is feasible in~\eqref{eq:worst-case:expectation}. Also, its objective function value equals
\[
\E_{\P^\star}[\ell(Z)] = \sum_{i=0}^N p_i^\star \ell(\hat z_i).
\]
As~$\ell(\hat z_0)=\overline\ell$, we may conclude that $\E_{\P^\star}[\ell(Z)]$ coincides with the maximum of the bi-dual reformulation in~\eqref{eq:finite:phi:II}, which in turn matches the supremum of~\eqref{eq:worst-case:expectation} by virtue of Theorem~\ref{thm:finite:convex:phi:II}. Hence, $\P^\star$ is indeed a maximizer of problem~\eqref{eq:worst-case:expectation}.

Recall that if~$\phi^\infty(1)=\infty$, then~$\D_\phi(\P,\hat\P)=\infty$ unless~$\P\ll\hat\P$. Therefore, every distribution~$\P$ in a $\phi$-divergence ambiguity set around~$\hat\P$ must be absolutely continuous with respect to~$\hat\P$. If~$\phi^\infty(1)<\infty$, on the other hand, then~$\P$ can assign a positive probability to points in~$\cZ$ that have zero probability under~$\hat\P$. Note that $\D_\phi(\P,\hat\P)$ only depends on {\em how much} probability mass~$\P$ removes from the support of~$\hat\P$, but it does not depend on {\em where} that probability mass is moved. As nature aims to maximize the expected loss, it will move all of this probability mass to a point with maximal loss within~$\cZ$ ({\em i.e.}, to some point~$\hat z_0 \in\argmax_{z\in\cZ} \ell(\hat z)$). 

If~$\cP$ is the {\em restricted} $\phi$-divergence ambiguity set~\eqref{eq:restricted-phi-divergence-ambiguity-set}, Assumption~\ref{assmp:discrete:reference} holds, $r>0$ and~$\phi$ is continuous at~$1$, then Theorem~\ref{thm:finite:convex:phi:II} remains valid with a minor modification. That is, one must append the constraint~$p_0=0$ to the finite bi-dual reformulation on the right hand side of~\eqref{eq:finite:phi:II}. Details are omitted for brevity.

\subsection{Optimal Transport Ambiguity Sets}
\label{sec:wc-distribution:transport}

Recall that the optimal transport ambiguity set~\eqref{eq:OT-ambiguity-set} is defined~as
\begin{align*}
\cP = \left\{ \P \in \cP(\cZ) \, : \, \OT_c(\P, \hat \P) \leq r \right\},
\end{align*}
where~$r\geq 0$ is a size parameter, $c$ is a transportation cost function in the sense of Definition~\ref{def:cost}, $\OT_c$ is the corresponding optimal transport discrepancy in the sense of Definition~\ref{def:OT}, and~$\hat\P\in\cP(\cZ)$ is a reference distribution. We will first show that the worst-case expectation problem~\eqref{eq:worst-case:expectation} over an optimal transport ambiguity set can often be reformulated as a finite convex minimization problem. To this end, we restrict attention to discrete reference distributions as in Assumption~\ref{assmp:discrete:reference}, and we impose convexity conditions on the transportation cost function, the loss function, and the support set $\cZ$. In addition, we impose a mild technical condition on the support points of the discrete reference distribution~$\hat \P$.

\begin{assumption}[Regularity Conditions for Optimal Transport Ambiguity Sets] ~\vspace{-3.5ex}
\label{assmp:convexity}
\begin{enumerate}[label=(\roman*)]
	\item \label{assmp:loss} The loss function $\ell$ is a point-wise maximum of $J \in \N$ concave functions, that is, $\ell(z) = \max_{j \in [J]} \ell_{j}(z)$, where $-\ell_{j}:\cZ\to\R$ is proper, convex and closed.
	\item \label{assmp:convex:support} 
	The support set is representable as $\cZ = \{z \in \R^d: g_k(z) \leq 0 ~ \forall k \in [K] \}$ for some $K \in \N$, where $g_k:\cZ\to\overline\R$ is proper, convex and closed.
	\item \label{assmp:cost} The transportation cost function~$c(z,\hat z)$ is convex in~$z$ for every fixed~$\hat z \in \mathcal{Z}$.
	\item \label{assmp:slater} The support point $\hat z_i$ belongs to $\rint(\dom(c(\cdot, \hat z_i)))$ and constitutes a Slater point for~$\cZ$ in the sense of Definition~\ref{def:slater} for every $i \in [N]$.
\end{enumerate}
\end{assumption}

Assumption~\ref{assmp:convexity}\,\ref{assmp:loss} is non-restrictive because any continuous function~$\ell$ on a compact set~$\cZ$ can be uniformly approximated by a pointwise maximum of finitely many concave functions~$\ell_j$, $j\in[J]$, albeit maybe at the expense of requiring large numbers~$J$ of pieces. Assumptions~\ref{assmp:convexity}\,\ref{assmp:convex:support} and~\ref{assmp:cost} are restrictive but satisfied by support sets and transportation cost functions commonly encountered in applications. Finally, Assumption~\ref{assmp:convexity}\,\ref{assmp:slater} is of a purely technical nature and can always be enforced by slightly perturbing the problem data.

\begin{theorem}[Finite Dual Reformulation for Optimal Transport Ambiguity Sets]
\label{thm:finite:convex:OT:I}
If~$\cP$ is the optimal transport ambiguity set~\eqref{eq:OT-ambiguity-set} and Assumptions~\ref{assmp:discrete:reference} and~\ref{assmp:convexity} hold, then the worst-case expectation problem~\eqref{eq:worst-case:expectation} obeys the weak duality relation
\begin{align}
	\nonumber
	&\sup_{\P \in \cP} ~ \E_\P[\ell(Z)] \\
	&\leq \left\{
	\begin{array}{cl@{\;}l}
		\inf & \displaystyle \lambda r + \sum_{i \in [N]} \hat p_i s_i \\[2.5ex]
		\st & \lambda\in\R_+,~ \alpha_{ijk} \in \R_+, ~ s_i \in \R & \forall i \in [N], j \in [J], k \in [K]\\
		& \zeta^{\ell}_{ij}, \zeta^{c}_{ij}, \zeta^{g}_{ijk} \in \R^d & \forall i \in [N], j \in [J], k \in [K] \\[1ex]
		& (-\ell_j)^*(\zeta^{\ell}_{ij}) + (c_i^*)^\pi( \zeta^{c}_{ij}, \lambda) \\[0.5ex]
		& \hspace{5mm} +\displaystyle \sum_{k \in [K]} (g_k^*)^\pi( \zeta^{g}_{ijk}, \alpha_{ijk}) \leq s_i & \forall i \in [N], j \in [J] \\[3ex]
		& \displaystyle \zeta^{\ell}_{ij} + \zeta^{c}_{ij} + \sum_{k \in [K]} \zeta^{g}_{ijk} = 0 & \forall i \in [N], j \in [J],
	\end{array} \right.
	\label{eq:finite-dual-ot}
\end{align}
where $c_i : \cZ \to \overline \R$ is defined through $c_i(z) = c(z, \hat z_i)$ for every $i \in [N]$. If~$r>0$, then strong duality holds, that is, the above inequality becomes an equality.
\end{theorem}

The dual minimization problem of Theorem~\eqref{thm:finite:convex:OT:I} constitutes a finite convex program because the conjugates~$(-\ell_j)^*$, $c_i^*$ and~$g_k^*$ and their perspectives are convex functions. It accommodates~$\cO(NJK)$ decision variables and~$\cO(NJ)$ constraints.

\begin{proof}[Proof of Theorem~\ref{thm:finite:convex:OT:I}]
By Theorem~\ref{thm:duality:OT}, we have
\begin{align*}
	\sup_{\P \in \cP} ~ \E_\P \left[ \ell (Z) \right]
	\leq \left\{
	\begin{array}{cll}
		\inf & \displaystyle \lambda r + \sum_{i \in [N]} \hat p_i s_i \\
		\st & \lambda \in \R_+, ~ s_i \in \R & \forall i \in [N] \\
		& \displaystyle \sup_{z \in \cZ}~ \ell(z) - \lambda c(z, \hat z_i) \leq s_i & \forall i \in [N],
	\end{array}
	\right.
\end{align*}
where~$s_i$ represents an auxiliary epigraphical decision variable for any $i\in[N]$. By Assumption~\ref{assmp:convexity}\,\ref{assmp:loss} and the definition of the functions~$c_i$, $i\in[N]$, the above minimization problem is equivalent to the following robust convex program.
\begin{align}
	\label{eq:strong:duality:robust}
	\begin{array}{cll}
		\inf & \displaystyle \lambda r + \sum_{i \in [N]} \hat p_i s_i \\[1ex]
		\st & \lambda \in \R_+, ~ s_i \in \R \\
		& \displaystyle \sup_{z \in \cZ} \, \ell_j(z) - \lambda c_i(z) \leq s_i ~~ \forall i \in [N], \, j \in[J]
	\end{array}
\end{align}
For any fixed~$i\in[N]$ and~$j\in[J]$, Assumptions~\ref{assmp:convexity}\,\ref{assmp:loss} and~\ref{assmp:convexity}\,\ref{assmp:convex:support} imply that the embedded maximization problem over~$z$ constitutes a convex program. In addition, this problem admits a Slater point~$\hat z_i$ thanks to Assumptions~\ref{assmp:convexity}\,\ref{assmp:loss} and~\ref{assmp:convexity}\,\ref{assmp:slater}. In order to dualize this convex program, we first recall from Lemma~\ref{lem:conjugate-sums} that the conjugate of $f(z)=-\ell_j(z)+\lambda c_i(z)$ at~$\zeta\in\R^d$ can be represented as 
\[
f^*(\zeta) = \min_{\zeta_{ij}^\ell,\zeta_{ij}^c \in\R^d} \left\{ (-\ell_j)^*(\zeta_{ij}^\ell) + (c_i^*)^\pi(\zeta^c_{ij}, \lambda) : \zeta_{ij}^\ell + \zeta_{ij}^c = \zeta\right\}.
\]
By Theorem~\ref{thm:convex:duality}, we thus obtain
\begin{align*}
	\sup_{z \in \cZ} ~ \ell_j(z) - \lambda c_i(z) 
	= \left\{ 
	\begin{array}{cl}
		\min & \displaystyle (-\ell_j)^*(\zeta^{\ell}_{ij}) \!+\! (c_i^*)^\pi( \zeta^{c}_{ij}, \lambda) + \sum_{k \in [K]} (g_k^*)^\pi( \zeta^{g}_{ijk}, \alpha_{ijk}) \\[2.5ex]
		\st & \alpha_{ijk} \in \R_+, ~ \zeta^{\ell}_{ij}, \zeta^{c}_{ij}, \zeta^{g}_{ijk} \in \R^d \quad \forall k \in [K] \\[1.5ex]
		& \displaystyle \zeta^{\ell}_{ij} + \zeta^{c}_{ij} + \sum_{k \in [K]} \zeta^{g}_{ijk} = 0.
	\end{array}
	\right.
\end{align*}
Next, we replace each embedded maximization problem in~\eqref{eq:strong:duality:robust} with its equivalent dual minimization problem, and we eliminate the corresponding minimization operators, which is allowed because all minima are attained. This yields the desired finite convex reformulation of the problem dual to~\eqref{eq:worst-case:expectation}, and it establishes weak duality. If~$r > 0$, then strong duality follows from Theorem~\ref{thm:duality:OT}.
\end{proof}

The finite convex reformulation of Theorem~\ref{thm:finite:convex:OT:I} was first derived under the more restrictive assumption that $c(z, \hat z) = \| z - \hat z \|$ by \citet[Theorem~4.2]{mohajerin2018data} and later generalized to arbitrary convex transportation cost functions by \citet[\textsection~6]{zhen2023unification}. We next derive a finite convex bi-dual for the worst-case expectation problem~\eqref{eq:worst-case:expectation} over the optimal transport ambiguity set~\eqref{eq:OT-ambiguity-set}, which forms the basis for identifying extremal distributions that (asymptotically) attain the supremum in~\eqref{eq:worst-case:expectation}. Our derivation will rely on the following two lemmas.

First, we derive a formula for the conjugate of a scaled perspective function. 

\begin{lemma}[Conjugates of Scaled Perspectives I]
\label{lem:perspective-multiplication}
If $f:\R^d\to\overline\R$ is proper, convex and closed, and if $\alpha\in\R_+$, then, for all $y\in\R^d$ and $y_0\in\R$, we have 
\begin{align*}
	(\alpha f^\pi)^*(y, y_0) = \begin{cases}
		0 & \text{if~} (f^*)^\pi(y, \alpha) \leq - y_0 ,\\
		\infty & \text{otherwise.}
	\end{cases}
\end{align*}
\end{lemma}
\begin{proof}
Assume first that $\alpha > 0$. If $\lambda>0$, then we have
\[
\alpha f^\pi(z,\lambda) = \alpha\lambda f(z/\lambda) = (\alpha f)^\pi(z,\lambda)\quad\forall z\in\R^d.
\]
Similarly, if $\lambda = 0$, then $\alpha f^\pi(z,\lambda) = \alpha f^\infty(z) = (\alpha f)^\infty(z) = (\alpha f)^\pi(z,\lambda)$ for all $z\in\R^d$.
We thus have shown that $\alpha f^\pi = (\alpha f)^\pi$. Next, define the set
\begin{align*}
	\cC & = \big\{ (y, y_0) \in \R^d\times \R: (\alpha f)^*(y) \leq -y_0 \big\}\\
	& = \big\{ (y, y_0)\in \R^d\times \R: (f^*)^\pi (y, \alpha) \leq -y_0 \big\},
\end{align*}
where the second equality follows from the definition of the perspective function. By~\citep[Corollary~13.5.1]{rockafellar1970convex}, we have $(\alpha f)^\pi =  \delta^*_\cC$. As~$\cC$ is closed, this implies that $(f^\pi)^* = \delta_{\cC}^{**} = \delta_{\cC}$, and thus the claim follows for~$\alpha>0$.

Assume next that~$\alpha=0$. In this case we have $\alpha f^\pi = \delta_{\dom(f^\pi)} $ thanks to our rules of extended arithmetic. This observation implies that 
\begin{align*}
	(\alpha f^\pi)^* (y,y_0)
	&= \delta^*_{\dom(f^\pi)}(y,y_0) = \sup_{\lambda\in\R_{++}} \sup_{ z\in\R^d} \left\{ y^\top z + y_0 \lambda : (z, \lambda) \in \dom(f^\pi) \right\} \\
	& = \sup_{\lambda \in\R_{++}}  \,\lambda \sup_{ z\in\R^d } \left\{y^\top (z/\lambda)  : z/\lambda \in \dom(f) \right\}+ y_0\lambda  \\
	&= \sup_{\lambda \in\R_{++}} \, \lambda \delta^*_{\dom(f)}(y) + \lambda y_0 = \begin{cases}
		0 & \text{if~} \delta_{\dom(f)}^*(y) + y_0 \leq 0, \\
		\infty & \text{otherwise}.
	\end{cases}
\end{align*}
Note that it is sufficient to optimize only over~$\lambda>0$ because $\dom(f^\pi) \subseteq \R^d\times \R_+$. As~$f$ is convex and closed, we have~$f=f^{**}$ thanks to Lemma~\ref{lem:bi:coincidence}, and thus we find
\[
\delta^*_{\dom(f)}(y)= \delta^*_{\dom(f^{**})}(y) = (f^*)^\infty (y) = (f^*)^\pi (y, 0),
\]
where the second and the third equalities follow from \citep[Theorem~13.3]{rockafellar1970convex} and from the definition of the perspective, respectively. Combining the above observations proves the claim for~$\alpha=0$.
\end{proof}

The next lemma derives a formula for the conjugate of a sum of scaled prespectives. It thus generalizes Lemma~\ref{lem:perspective-multiplication}, which addresses only one single scaled perspective, and it is also related to Lemma~\ref{lem:conjugate-sums}, which characterizes the conjugate of a sum of arbitrary convex functions---not necessarily scaled perspectives.

\begin{lemma}[Conjugates of Perspective Functions II]
\label{lem:conjugate-of-perspective:II}
Suppose that $f_i : \R^d \to \overline \R$, $i \in [m]$, are proper, convex and closed and that there is $\bar z \in \cap_{i \in [m]} \rint(\dom(f_i))$. Let $f(z_1, \dots, z_m, \lambda) = \sum_{i \in [m]} \alpha_i f_i^\pi(z_i, \lambda) $ be a weighted sum of the corresponding perspective functions with weight vector $\alpha \in \R_+^m$. Then, the conjugate of~$f$ satisfies
\begin{align*}
	f^*(y_1, \dots, y_m, y_0)
	&= \begin{cases}
		0 & 
		\left\{\begin{array}{l}
			\text{if } \exists \beta \in \R^m \text{ with }\sum_{i \in [m]} \beta_i = y_0 \text{ and}\\ (f_i^*)^\pi(y_i, \alpha_i) \leq - \beta_i ~~ \forall i \in [m],
		\end{array} \right. \\[2ex]
		\infty & \text{otherwise.}
	\end{cases}
\end{align*}
\end{lemma}

\begin{proof}
By using a variable splitting trick as in the proof of Lemma~\ref{lem:conjugate-sums}, we find
\begin{align*}
	f^*(y_1, \dots, y_m, y_0)
	&= \sup_{z_1,\ldots, z_m\in\R^d} \sup_{\lambda\in\R_+} \, y_0 \lambda +\sum_{i \in [m]} y_i^\top z_i  - \sum_{i \in [m]} \alpha_i f_i^\pi(z_i, \lambda) \\[1ex]
	&= \left\{ \begin{array}{cl}
		\displaystyle \sup_{\substack{ z_1,\ldots,z_m \in \R^d \\ \lambda \in \R, \, \lambda_1,\ldots,\lambda_m \in \R_+}} & \displaystyle y_0 \lambda + \sum_{i \in [m]} y_i^\top z_i - \alpha_i f_i^\pi(z_i, \lambda_i)  \\
		\st & \lambda_i = \lambda \quad i \in [m]
	\end{array} \right.
\end{align*} 
The resulting convex maximization problem admits a Slater point. To see this, recall that there exists $\bar z \in \cap_{i \in [m]} \rint(\dom(f_i))$. As $\dom(f^\pi)$ is contained in the cone generated by $\dom(f)\times\{1\}$, we may thus conclude that the solution with~$\lambda=1$, $\lambda_i=1$ and~$z_i=\bar z$ for all $i\in[m]$ constitutes a Slater point. Therefore, the above maximization problem admits a strong Lagrangian dual, that is, we have
\begin{align*}
	& f^*(y_1, \dots, y_m, y_0) \\
	&= \min_{\beta_1,\ldots,\beta_m \in \R} \sup_{\substack{ z_1,\ldots,z_m \in \R^d \\ \lambda \in \R, \, \lambda_1,\ldots,\lambda_m \in \R_+}} y_0 \lambda + \sum_{i \in [m]} y_i^\top z_i - \alpha_i f_i^\pi(z_i, \lambda_i) + \beta_i (\lambda_i - \lambda) \\
	&= \min_{\beta_1,\ldots,\beta_m \in \R} \left\{ \sum_{i \in [m]} (\alpha_i f_i^\pi)^* (y_i, \beta_i) : \sum_{i \in [m]} \beta_i = y_0 \right\},
\end{align*}
see also Theorem~\ref{thm:convex:duality}. By Lemma~\ref{lem:perspective-multiplication}, we further have $(\alpha_i f_i^\pi)^* =\delta_{\cC_i}$, where
\[
\cC_i = \big\{ (y, y_0) \in \R^d\times\R : (f_i^*)^\pi (y, \alpha_i) \leq -y_0 \big\}
\]
for all $i\in[m]$. Substituting this alternative expression for~$(\alpha_i f_i^\pi)^*$ into the above dual problem yields the desired formula. Thus, the claim follows.
\end{proof}

We emphasize that Lemmas~\ref{lem:perspective-multiplication} and~\ref{lem:conjugate-of-perspective:II} are complementary to Lemma~\ref{lem:conjugate-of-perspective}. Indeed, while Lemma~\ref{lem:conjugate-of-perspective} evaluates the conjugate only with respect to the first argument of a perspective function, Lemmas~\ref{lem:perspective-multiplication} and~\ref{lem:conjugate-of-perspective:II} do so with respect to {\em both} arguments. We are now ready to derive a finite bi-dual reformulation of the worst-case expectation problem over an optimal transport ambiguity set.

\begin{theorem}[Finite Bi-Dual Reformulation for Optimal Transport Ambiguity Sets]
\label{thm:finite:convex:OT:II}
If~$\cP$ is the optimal transport ambiguity set~\eqref{eq:OT-ambiguity-set} and Assumptions~\ref{assmp:discrete:reference} and~\ref{assmp:convexity} hold, then the worst-case expectation problem~\eqref{eq:worst-case:expectation} satisfies the weak duality relation
\begin{align}
	&\sup_{\P \in \cP} ~ \E_\P[\ell(Z)] \notag \\
	&\hspace{-1ex}\leq \left\{
	\begin{array}{cl@{\hspace{-1em}}l}
		\sup & \displaystyle \sum_{i \in [N]} \sum_{j \in [J]} - (-\ell_j)^\pi ( p_{ij} \hat z_i + z_{ij}, p_{ij}) \\[3ex]
		\st & p_{ij} \in \R_+, ~ z_{ij} \in \R^d & \forall i \in [N], \, j \in [J] \\[1ex]
		& \displaystyle g_k^\pi ( p_{ij} \hat z_i + z_{ij}, p_{ij}) \leq 0 & \forall i \in [N], \, j \in [J], \, k \in [K] \\[1ex]
		& \displaystyle \sum_{j \in [J]} p_{ij} = \hat p_i & \forall i \in [N] \\[3ex]
		& \displaystyle \sum_{i \in [N]} \sum_{j \in [J]} c_i^\pi ( p_{ij} \hat z_i + z_{ij},p_{ij} ) \leq r,
	\end{array} 
	\right.
	\label{eq:finite:OT:II}
\end{align}
where $c_i : \cZ \to \overline \R$ is defined through $c_i(z) = c(z, \hat z_i)$ for every $i \in [N]$. If~$r>0$, then strong duality holds, that is, the above inequality becomes an equality.
\end{theorem}
\begin{proof}
We will show that~\eqref{eq:finite:OT:II} is obtained by dualizing the finite dual reformulation~\eqref{eq:finite-dual-ot} of problem~\eqref{eq:worst-case:expectation}. To see this, we assign Lagrange multipliers $p_{ij}\in \R_+$ and $z_{ij}\in\R^d$, $i\in[N]$, $j\in[J]$, to the first and second constraint groups in~\eqref{eq:finite-dual-ot}, respectively. The Lagrangian dual of~\eqref{eq:finite-dual-ot} can then be represented compactly as
\begin{align*}
	\sup_{p \geq 0,z} \;\inf_{\substack{\lambda\geq 0, \alpha\geq 0 \\ s, \zeta^\ell, \zeta^c, \zeta^g}} L_1 (s; p, z) + L_2 (\zeta^\ell; p, z) + L_3 (\lambda, \zeta^c; p, z, \lambda) + L_4 (\alpha, \zeta^g; p, z),
\end{align*}
where the Lagrangian is additively separable with respect to four disjoint groups of primal decision variables, namely, $s$, $\zeta^\ell$, $(\lambda, \zeta^c)$ and $(\alpha,\zeta^g)$. The corresponding partial Lagrangians are defined as follows.
\begin{align*}
	L_1(s; p, z) &= \sum_{i\in[N]} \hat p_i s_i - \sum_{i\in[N]} \sum_{j\in[J]} p_{ij} s_i \\
	L_2(\zeta^\ell; p, z) &= \sum_{i\in[N]} \sum_{j\in[J]} p_{ij} \cdot (-\ell_j)^*(\zeta_{ij}^\ell) - z_{ij}^\top \zeta_{ij}^\ell \\
	L_3(\lambda, \zeta^c; p, z) & = \lambda r + \sum_{i\in[N]} \sum_{j\in[J]} p_{ij} \cdot (c_i^*)^\pi(\zeta_{ij}^c, \lambda) - z_{ij}^\top \zeta_{ij}^c \\
	L_4(\alpha, \zeta^g; p,z) & = \sum_{i\in[N]} \sum_{j\in[J]} \sum_{k\in[K]} p_{ij} \cdot (g_k^*)^\pi(\zeta_{ijk}^g, \alpha_{ijk}) - z_{ij}^\top \zeta_{ijk}^g
\end{align*}
These partial Lagrangians can be minimized separately with respect to the primal decision variables. For example, an elementary calcucation shows that
\begin{align*}
	\inf_{s} L_1(s; p, z) 
	= \begin{cases} 0 & \displaystyle \text{if} ~ \sum_{j\in[J]} p_{ij} = \hat p_i ~~ \forall i \in [N], \\ -\infty & \text{otherwise.} \end{cases}
\end{align*}
Recall now that $-\ell_j$ is proper, convex and closed, which implies via Lemma~\ref{lem:bi:coincidence} that $(-\ell_j)^{**}=-\ell_j$. Note also that minimizing $L_2(\zeta^\ell; p, z)$ with respect to~$\zeta^\ell$ amounts to evaluating the conjugate of a sum of conjugates with mutually different arguments. By using Lemma~\ref{lem:conjugate-perspective}\,\ref{lem:conjugate-perspective-1} and applying a few elementary manipulations we thus find
\begin{align*}
	\displaystyle \inf_{\zeta^\ell} L_2(\zeta^\ell; p, z) = \sum_{i \in [N]} \sum_{j \in [J]} - (-\ell_j)^\pi(z_{ij}, p_{ij}).
\end{align*}
Similarly, recall that~$c_i$ is proper, convex and closed such that $c_i^{**}=c_i$. Note also that minimizing $L_3(\lambda, \zeta^c; p, z)$ with respect to~$\lambda$ and~$\zeta^c$ amounts to evaluating the conjugate of a sum of perspective functions with one common argument. By using Lemma~\ref{lem:conjugate-of-perspective:II} and applying a few elementary manipulations we thus find
\begin{align*}
	\inf_{\lambda \geq 0,  \zeta^c} L_3(\lambda, \zeta^c; p, z) = \begin{cases}
		0 & 
		\left\{\begin{array}{l}
			\text{if } \exists \beta_{ij} \in \R^m \text{ with }\sum_{i \in [N]}\sum_{j\in[J]} \beta_{ij} = r \text{ and}\\ c_i^\pi(z_{ij}, p_{ij}) \leq \beta_{ij} ~~ \forall i \in [N],\, j\in[J],
		\end{array} \right. \\[2ex]
		-\infty & \text{otherwise.}
	\end{cases}
\end{align*}
Finally, recall that~$g_k$ is proper, convex and closed such that $g_k^{**}=g_k$. Note also that minimizing $L_4(\alpha, \zeta^g; p, z)$ with respect to~$\alpha$ and~$\zeta^g$ amounts to evaluating the conjugate of a sum of perspective functions with mutually different arguments. By using Lemma~\ref{lem:perspective-multiplication} and applying a few elementary manipulations we thus find
\begin{align*}
	\inf_{\alpha\geq 0, \zeta^g} L_4(\alpha, \zeta^g; p,z) 
	= \begin{cases} 0 & \text{if} ~ g_k^\pi \big( z_{ij}, p_{ij}) \leq 0 ~ \forall i \in [N],\, j \in [J],\, k \in [K], \\
		-\infty & \text{otherwise.}\end{cases}
\end{align*}
Substituting the infima of the partial Lagrangians into the dual objective yields the following equivalent reformulation for the problem dual to~\eqref{eq:finite-dual-ot}.
\begin{align}
	\label{eq:finite:OT:II-preliminary}
	\begin{array}{cll}
		\sup & \displaystyle \sum_{i \in [N]} \sum_{j \in [J]} - (-\ell_j)^\pi (z_{ij}, p_{ij}) \\
		\st & \displaystyle p_{ij} \in \R_+, ~ \beta_{ij} \in \R, ~ z_{ij} \in \R^d & \forall i \in [N], \, j \in [J] \\[1mm]
		& \displaystyle g_k^\pi(z_{ij}, p_{ij}) \leq 0 & \forall i \in [N], \, j \in [J], \, k \in [K] \\[1mm]
		& \displaystyle \sum_{j \in [J]} p_{ij} = \hat p_i & \forall i \in [N] \\[1mm]
		& \displaystyle c_i^\pi ( z_{ij}, p_{ij}) \leq \beta_{ij} & \forall i \in [N], \, j \in [J] \\[1mm]
		& \displaystyle \sum_{i \in [N]} \sum_{j \in [J]} \beta_{ij} = r
	\end{array}
\end{align}
Note that if the finite dual reformulation~\eqref{eq:finite-dual-ot} of the worst-case expectation problem is viewed as an instance of the primal convex program~\eqref{eq:primal:convex}, then problem~\eqref{eq:finite:OT:II-preliminary} represents the corresponding instance of the dual convex program~\eqref{eq:dual:convex}. By Assumptions~\ref{assmp:discrete:reference} and~\ref{assmp:convexity}, problem~\eqref{eq:finite:OT:II-preliminary} admits a Slater point with $p_{ij}=\hat p_i/J$ and $z_{ij}=\hat z_i$ for all $i\in[N]$ and~$j\in[J]$. Thus, strong duality holds thanks to Theorem~\ref{thm:convex:duality}\,\ref{lem:slater:duality}. It remains to be shown that~\eqref{eq:finite:OT:II-preliminary} is equivalent to~\eqref{eq:finite:OT:II}. To this end, note first that the last constraint in~\eqref{eq:finite:OT:II-preliminary} can be relaxed to a less-than-or-equal-to inequality without increasing the problem's supremum such that $\beta_{ij}= c_i^\pi ( z_{ij}, p_{ij})$ at optimality. This allows us to eliminate the $\beta_{ij}$ variables from~\eqref{eq:finite:OT:II-preliminary}. Problem~\eqref{eq:finite:OT:II} is then obtained by applying the substitution $z_{ij} \leftarrow z_{ij} - p_{ij} \hat z_i$. 
\end{proof}

The finite bi-dual reformulation~\eqref{eq:finite:OT:II} is guaranteed to be solvable provided that the transportation cost function satisfies the following additional assumption.

\begin{assumption}[Identity of Indiscernibles]
\label{assmp:cost:zero}
The transportation cost function is real-valued and satisfies $c(z, \hat z) = 0$ if and only if~$z = \hat z$.
\end{assumption}

\begin{lemma}[Solvability of the Finite Bi-Dual Reformulation]
\label{lem:bi-dual-solvability-OT}
Suppose that Assumptions~\ref{assmp:discrete:reference}, \ref{assmp:convexity} and~\ref{assmp:cost:zero} hold. Then, problem~\eqref{eq:finite:OT:II} is solvable.
\end{lemma}

\begin{proof}
Under the stated assumptions, problem~\eqref{eq:finite:OT:II} maximizes an upper semicontinuous function over a compact feasible region, and thus the claim follows from Weierstrass' maximum theorem. To see that the objective function of~\eqref{eq:finite:OT:II} is upper semicontinuous, note that the functions $-\ell_j$ are proper, convex and closed for all~$j\in[J]$ thanks to Assumption~\ref{assmp:convexity}\,\ref{assmp:loss}. By \citep[pages~35 and~67]{rockafellar1970convex}, their perspectives are proper, convex and closed, too; see also \citep[Proposition~C.2]{zhen2023unification}. Thus, the {\em negative} perspective functions appearing in the objective function of problem~\eqref{eq:finite:OT:II} are indeed upper semicontinuous. Similarly, one can show that the feasible region of problem~\eqref{eq:finite:OT:II} is closed. Indeed, $g_k$ and~$c_i$ are proper, convex and closed for all~$k\in[K]$ and~$i\in[N]$ thanks to Assumption~\ref{assmp:convexity} and Definition~\ref{def:cost}. This readily implies that their perspectives are lower semicontinuous, and thus the feasible region of~\eqref{eq:finite:OT:II} is indeed closed. To see that the feasible region is also bounded, note first that $p_{ij} \in [0,1]$ for all~$i \in [N]$ and~$j \in [J]$. Indeed, these variables must be non-negative and compatible with the probabilities $\hat p_i$, $i\in[N]$, of the discrete reference distribution. Next, we show that the variables~$z_{ij}$ for~$i \in [N]$ and~$j \in [J]$ are restricted to a bounded set, as well. Indeed, by \citep[Lemma~C.10]{zhen2023unification}, which applies thanks to Assumption~\ref{assmp:cost:zero} and Definition~\ref{def:cost}, there exists~$\delta > 0$ such that $c_i(\hat z_i + z) \geq \delta \| z \|_2 - 1$ for all~$z \in \R^d$ and~$i\in[N]$. The last constraint of problem~\eqref{eq:finite:OT:II} therefore implies that
\begin{align*}
	\sum_{i \in [N]}  \sum_{j \in [J]} c_i^\pi( p_{ij} \hat z_i + z_{ij}, p_{ij}) \leq r 
	\quad \implies \quad 
	\sum_{i \in [N]}  \sum_{j \in [J]} \| z_{ij} \|_2 \leq \frac{1 + r}{\delta},
\end{align*}
where we used the identity $\sum_{i \in [N]} \sum_{j \in [J]} p_{ij} = \sum_{i\in[N]} \hat p_i= 1$ and the definition of the perspective function. Thus, the feasible region of~\eqref{eq:finite:OT:II} is indeed bounded. 
\end{proof}

We are now ready to construct extremal distributions~$\P^\star\in\cP(\cZ)$ that attain the supremum of the worst-case expectation problem~\eqref{eq:worst-case:expectation} over the optimal transport ambiguity set~\eqref{eq:OT-ambiguity-set}. To this end, fix any maximizer $(p^\star, z^\star)$ of the bi-dual problem~\eqref{eq:finite:OT:II}, which exists thanks to Lemma~\ref{lem:bi-dual-solvability-OT}. Next, define the index sets
\begin{align*}
\mspace{-2mu}
\cJ^\infty_i = \big\{ j \in [J]: p_{ij}^\star = 0, \, z_{ij}^\star \neq 0 \big\}\quad \text{and} \quad \cJ^+_i = \big\{ j \in [J]: p_{ij}^\star > 0 \big\} ,
\end{align*}
and define~$\cJ_i = \cJ_i^+ \cup \cJ_i^\infty$ for any~$i \in [N]$. The following theorem uses the maximizer~$(p^\star,z^\star)$ and the corresponding index sets to construct~$\P^\star$.

\begin{theorem}[Extremal Distributions of Optimal Transport Ambiguity Sets]
\label{thm:finite:convex:OT:III}
Suppose that all conditions of Theorem~\ref{thm:finite:convex:OT:II} for weak and strong duality are satisfied, Assumption~\ref{assmp:cost:zero} holds, and~$(p^\star, z^\star)$ solves~\eqref{eq:finite:OT:II}. Then, the following~hold.
\begin{enumerate}[label=(\roman*)]
	\item \label{thm:OT:extremal} If $\cJ_i^\infty = \emptyset$ for all~$i \in [N]$, then problem~\eqref{eq:worst-case:expectation} is solved by 
	\[
	\P^\star = \sum_{i \in [N]} \sum_{j \in \cJ^+_i} p_{ij}^\star \, \delta_{\hat z_i + z_{ij}^\star / p_{ij}^\star }.
	\]
	\item \label{thm:OT:assymptotic} If $\cJ^\infty_i \neq \emptyset$ for some~$i \in [N]$, then problem~\eqref{eq:worst-case:expectation} is asymptotically solved by $\P^m = \sum_{i \in [N]} \sum_{j \in \cJ_i} p_{ij}^m \, \delta_{z_{ij}^m}$ as~$m \in \N$, $m\geq \max_{i\in[N]} |\cJ^\infty_i|$, grows, where
	\begin{align*}
		\hspace{-4ex}
		p_{ij}^m 
		= \begin{cases}
			\left(1 - \tfrac{|\cJ^\infty_i|}{m} \right) p_{ij}^\star & \text{if }j \in \cJ^+_i, \\[1ex]
			\frac{\hat p_i}{m} & \text{if }j \in \cJ^\infty_i,
		\end{cases} 
		~~ \text{and} ~~~
		z_{ij}^m 
		= \begin{cases}
			\hat z_i + \frac{z_{ij}^\star}{p_{ij}^\star} & \text{if } j \in \cJ^+_i, \\[1ex]
			\hat z_i + \frac{z_{ij}^\star}{p_{ij}^m} & \text{if } j \in \cJ^\infty_i.
		\end{cases}
	\end{align*}
\end{enumerate}
\end{theorem}

\begin{proof}
In view of assertion~\ref{thm:OT:extremal}, we first show that $\P^\star$ defined in the statement of the theorem is feasible in the worst-case expectation problem~\eqref{eq:worst-case:expectation}. To this end, observe first that feasibility of $(p^\star, z^\star)$ in~\eqref{eq:finite:OT:II} implies that $p_{ij}^\star \geq 0$ for all $i \in [N]$ and $j \in [J]$, and that $\sum_{i \in [N]} \sum_{j \in \cJ^+_i} p_{ij}^\star = 1$. Note also that $\hat z_i + z_{ij}^\star / p_{ij}^\star \in \mathcal{Z}$ for all $i \in [N]$ and $j \in \cJ^+_i$ due to the second constraint in~\eqref{eq:finite:OT:II}. This confirms that $\P^\star\in\cP(\cZ)$. The penultimate constraint group of problem~\eqref{eq:finite:OT:II} also implies that
\[
\sum_{i\in[N]} \sum_{j\in\cJ_i^+} p_{ij}^\star \, \delta_{\big(\hat z_i + z_{ij}^\star / p_{ij}^\star, \hat z_i\big)} \in \Gamma(\P^\star, \hat \P)
\]
constitutes a valid transportation plan for morphing~$\hat\P$ into~$\P^\star$. Thus, we find
\begin{align*}
	\OT_c(\P^\star, \hat \P) 
	& \leq  \sum_{i \in [N]} \sum_{j \in \cJ_i^+} p_{ij}^\star \cdot c(\hat z_i + z_{ij}^\star / p_{ij}^\star, \hat z_i) \\
	& = \sum_{i \in [N]} \sum_{j \in [J]} c_i^\pi(p_{ij}^\star \hat z_i + z_{ij}^\star, p_{ij}^\star) \leq r.
\end{align*}
Here, the equality holds because all terms corresponding to~$i\in[N]$ and~$j\notin\cJ_i^+$ vanish. Indeed, if~$j\notin\cJ_i^+$, then~$p^\star_{ij}=0$. As~$\cJ_i^\infty=\emptyset$, this implies that~$z_{ij}^\star = 0$. Thus, we have $c_i^\pi(p_{ij}^\star \hat z_i + z_{ij}^\star, p_{ij}^\star) = c_i^\pi(0, 0) = c_i^\infty(0) = 0$ by the definitions of the perspective and the recession function. The second inequality in the above expression follows from the last constraint in~\eqref{eq:finite:OT:II}. In summary, we have shown that~$\P^\star$ is feasible in~\eqref{eq:worst-case:expectation}. As for the objective function value of~$\P^\star$, note that
\begin{align*}
	\E_{\P^\star} [\ell(Z)] 
	\leq \sup_{\P \in \cP} \, \E_{\P} [\ell(Z)] 
	\leq \sum_{i \in [N]} \sum_{j \in [J]} - (-\ell_j)^\pi (p_{ij}^\star \hat z_i + z_{ij}^\star, p_{ij}^\star),
\end{align*}
where the second inequality follows from the weak duality relation established in Theorem~\ref{thm:finite:convex:OT:II}. At the same time, however, the expected loss under $\P^\star$ satisfies
\begin{align*}
	\mspace{-4mu}
	\E_{\P^\star} [\ell(Z)] 
	& = \sum_{i \in [N]} \sum_{j \in \cJ^+} \max_{j' \in [J]} p_{ij}^\star \ell_{j'} \big(\hat z_i + \tfrac{z_{ij}^\star}{p_{ij}^\star} \big) \\
	&\geq  \sum_{i \in [N]} \sum_{j \in \cJ^+} - (-\ell_j)^\pi (p_{ij}^\star \hat z_i + z_{ij}^\star, p_{ij}^\star) \\
	&= \sum_{i \in [N]} \sum_{j \in [J]} - (-\ell_j)^\pi (p_{ij}^\star \hat z_i + z_{ij}^\star, p_{ij}^\star),
\end{align*}
where the inequality uses the definition of the perspective function and the trivial observation that~$j \in \mathcal{J}^+$ is a feasible choice for~$j' \in [J]$. The last equality holds once more because $p_{ij}^\star = 0$ implies $z_{ij}^\star = 0$ and $(-\ell_j)^\pi(0,0) = (-\ell_j)^\infty(0) = 0$ by the definition of the perspective and the recession function. In summary, the above inequalities imply that~$\P^\star$ is optimal in~\eqref{eq:worst-case:expectation}. Hence, assertion~\ref{thm:OT:extremal} follows.

As for assertion~\ref{thm:OT:assymptotic}, we first show that $\P^m \in \cP$ for any fixed $m \geq \max_{i \in [N]} |\cJ_i^\infty|$. The constraints of problem~\eqref{eq:finite:OT:II} imply that $p_{ij}^m \ge 0$ for all~$j \in \cJ_i$ and~$i \in [N]$ and that $\sum_{i \in [N]} \sum_{j \in \cJ} p_{ij}^m = 1$. They also imply that~$z_{ij}^m \in \cZ$ for every~$j \in \cJ_i$ and~$i \in [N]$. This is easy to see if~$j\in\cJ_i^+$. If~$j \in \cJ_i^\infty$, on the other hand, then~$p^\star_{ij}=0$, $z^\star_{ij}\neq 0$ and~$g_k^\pi(z^\star_{ij},0)\leq 0$ for all $k\in[K]$, which implies via \citep[Theorem~8.6]{rockafellar1970convex} that~$z^\star_{ij}$ is a recession direction of~$\cZ$. Geometrically, this means that the ray emanating from any point in~$\cZ$ along the direction~$z^\star_{ij}$ never leaves~$\cZ$. Thus, $z_{ij}^m = \hat z_i+m \, z_{ij}^\star/\hat p_i \in\cZ$ for all~$i\in[N]$ and~$j\in\cJ_i^\infty$. In addition, one verifies that
\[
\sum_{i\in[N]} \sum_{j\in\cJ_i} p_{ij}^m \, \delta_{\big(z_{ij}^m, \hat z_i\big)} \in \Gamma(\P^\star, \hat \P)
\]
constitutes a valid transportation plan for morphing~$\hat\P$ into~$\P^m$. Thus, we find
\begin{align*}
	&\OT_c(\P^m, \hat \P)\\
	& \leq  \sum_{i \in [N]} \sum_{j \in \cJ_i} p_{ij}^m \, c(z_{ij}^m, \hat z_i) \\
	& = \sum_{i \in [N]}  \sum_{j \in \cJ^+_i} p_{ij}^\star \left(1 - \tfrac{ |\cJ^\infty_i|}{m} \right) \, c \big( \hat z_i + \tfrac{z^\star_{ij}}{p_{ij}^\star}, \hat z_i \big) + \sum_{i \in [N]}  \sum_{j \in \cJ^\infty_i} \frac{\hat p_i}{m} \, c \big(  \hat z_i + m \tfrac{z^\star_{ij}}{\hat p_i}, \hat z_i \big) \\
	&\leq \sum_{i \in [N]} \sum_{j \in \cJ^+_i} p_{ij}^\star \, c\big(  \hat z_i + \tfrac{z_{ij}^\star}{p_{ij}^\star}, \hat z_i \big) + \sum_{i \in [N]} \sum_{j \in \cJ^\infty_i} \lim_{m \rightarrow \infty} \frac{\hat p_i}{m} \, c \big( \hat z_i + m \tfrac{z_{ij}^\star}{\hat p_i}, \hat z_i \big) \\
	&= \sum_{i \in [N]} \sum_{j \in \cJ^+_i} p_{ij}^\star \, c\big(  \hat z_i + \tfrac{z_{ij}^\star}{p_{ij}^\star}, \hat z_i \big) + \sum_{i \in [N]} \sum_{j \in \cJ^\infty_i} \lim_{m \rightarrow \infty} \frac{\hat p_i}{m} \, c \big(m \tfrac{z_{ij}^\star}{\hat p_i}, \hat z_i \big) \\
	&  = \sum_{i \in [N]}  \sum_{j \in [J]} c_i^\pi ( p_{ij}^\star \hat z_i + z_{ij}^\star, p_{ij}^\star)  \leq r,
\end{align*}
where the first equality follows from the definitions of~$p_{ij}^m$ and~$ z_{ij}^m$. The second inequality holds because the transportation cost function~$c(z, \hat z)$ is non-negative and convex in~$z$, which implies that both terms in the third line are non-decreasing in~$m$. The second equality follows from Assumption~\ref{assmp:cost:zero}, which ensures that~$c(z,\hat z)$ is real-valued such that the reference point in the definition of the recession function of $c(\cdot,\hat z_i)$ can be chosen freely. The third equality exploits the definition of the perspective function~$c_i^\pi$ and the observation that $c_i^\pi(0, 0) = c_i^\infty(0) = 0$. Finally, the last inequality follows from the last constraint of problem~\eqref{eq:finite:OT:II}. We have thus shown that~$\P^m$ is feasible in~\eqref{eq:worst-case:expectation}. In analogy to analysis for~$\P^\star$, one can show that the asymptotic expected loss $\lim_{m \rightarrow \infty} \E_{\P^m}[\ell(Z)]$ is at least as large as the optimal value~$\sum_{i \in [N]} \sum_{j \in [J]} - (-\ell_j)^\pi ( p^\star_{ij} \hat z_i + z_{ij}^\star, p^\star_{ij})$ of the finite bi-dual reformulation~\eqref{eq:finite:OT:II}. However, as the suprema of~\eqref{eq:worst-case:expectation} and~\eqref{eq:finite:OT:II} match, it is clear that the distributions~$\P^m$, $m \in \N$, must be asymptotically optimal in~\eqref{eq:worst-case:expectation}.
\end{proof}

If $\cJ^\infty_i \neq \emptyset$ for some~$i \in [N]$, then the extremal distributions constructed in Theorem~\ref{thm:finite:convex:OT:III} send atoms with decaying probabilities to infinity along specific
recession directions $z_{ij}^\star$, $j\in \cJ^\infty_i$, of the support set~$\cZ$. Moving atoms to infinity is possible even when only a finite transportation budget~$r$ is available provided that the probability mass transported scales inversely with the transportation cost. The following lemma establishes sufficient conditions for~$\cJ_i^\infty$ to be empty for every~$i\in[N]$, which ensures via Theorem~\ref{thm:finite:convex:OT:III}\,\ref{thm:OT:extremal} that problem~\eqref{eq:worst-case:expectation} is solvable. 

\begin{lemma}
\label{lem:sufficient:OT}
If all assumptions of Theorem~\ref{thm:finite:convex:OT:III} are satisfied and either of the following conditions holds, then~$\cJ_i^\infty = \emptyset$ for every~$i\in [N]$, and~\eqref{eq:worst-case:expectation} is solvable.
\begin{enumerate}[label=(\roman*)]
	\item \label{lem:sufficient:cost} The transportation cost function grows superlinearly in its first argument. By this we mean that $c_i^\infty(z)=\infty$ for any~$z\neq 0$ and for any~$i\in[N]$. 
	\item \label{lem:sufficient:support} The support set $\cZ$ is bounded.
\end{enumerate}
\end{lemma}

\begin{proof}
As usual, let~$(p^\star, z^\star)$ be a maximizer of problem~\eqref{eq:finite:OT:II}, which exists thanks to Lemma~\ref{lem:bi-dual-solvability-OT}. As for assertion~\ref{lem:sufficient:cost}, assume that the transportation cost function grows superlinearly. For the sake of argument, assume also that there exists~$i \in [N]$ with $\cJ^\infty_i \neq \emptyset$. For every $j \in \cJ^\infty_i$ we thus have $p^\star_{ij}=0$ and $z^\star_{ij}\neq 0$. Hence, we find
\begin{align*}
	c_i^\pi ( p_{ij}^\star \hat z_i + z_{ij}^\star, p_{ij}^\star)
	= c_i^\infty(z^\star_{ij}) = \infty,
\end{align*}
where the first equality uses the definition of the perspective function, and the second equality holds because the transportation cost function grows superlinearly. Thus, $(p^\star, z^\star)$ violates the last constraint of problem~\eqref{eq:finite:OT:II}, which contradicts its assumed feasibility. We may thus conclude that~$\cJ^\infty_i = \emptyset$ and that~\eqref{eq:worst-case:expectation} is solvable.

As for assertion~\ref{lem:sufficient:support}, assume now that~$\cZ$ is bounded. Without loss of generality, we may also assume that $p_{ij}^\star = 0$ for some $i \in [N]$ and $j \in [J]$ for otherwise $\cJ^\infty_i$ is trivially empty. The constraints of problem~\eqref{eq:finite:OT:II} then ensure that~$g_k^\pi(z^\star_{ij},0)\leq 0$ for all $k\in[K]$, which implies via \citep[Theorem~8.6]{rockafellar1970convex} that~$z^\star_{ij}$ is a recession direction of~$\cZ$. As $\cZ$ is compact, however, this implies that $z_{ij}^\star = 0$. We may thus again conclude that $\cJ^\infty_i = \emptyset$  and that~\eqref{eq:worst-case:expectation} is solvable.
\end{proof}

Condition~\ref{lem:sufficient:cost} of Lemma~\ref{lem:sufficient:OT} is satisfied whenever~$\cP$ is a $p$-Wasserstein ball and the transportation cost function is of the form $c(z,\hat z)=\|z-\hat z\|^p$ for some~$p>1$.

The structural properties of the distributions that solve the worst-case expectation problem~\eqref{eq:worst-case:expectation} over an optimal transport ambiguity set, as well as necessary and sufficient conditions for their existence, were studied by \citet{wozabal2012framework,owhadi2017extreme,yue2020linear} and \citet{gao2016distributionally}. In particular, significant efforts were spent on characterizing the extremal distributions of a Wasserstein ball centered at a discrete reference distributions with~$N$ atoms. The earliest result in this domain is due to \citet[Theorem~3.3]{wozabal2012framework} who showed that the worst-case expectation of a continuous bounded loss function is attained by a discrete distribution with at most $N+3$~atoms. Later, \citet[Theorem~2.3]{owhadi2017extreme} and \citet[Corollary~1]{gao2016distributionally} managed to sharpen this result by showing that the worst-case expectation is in fact attained by a discrete distribution with at most $N+2$ or even only $N+1$ atoms, respectively; see also \citep[Theorem~4]{yue2020linear}. Theorem~\ref{thm:finite:convex:OT:III}\,\ref{thm:OT:extremal} and Lemma~\ref{lem:sufficient:OT} reveal that if~$\cZ$ is bounded and the loss function~$\ell$ is concave, thus satisfying Assumption~\ref{assmp:convexity}\,\ref{assmp:loss} with~$J=1$, then the worst-case expected loss is attained by an $N$-point distribution. For more general loss functions, however, every $N$-point distributions can be strictly suboptimal even if problem~\eqref{eq:worst-case:expectation} is solvable; see \citep[Example~5]{kuhn2019wasserstein}. The results in this section are based on \citep[\textsection~6]{zhen2023unification}.

\subsection{Nash Equilibria and Adversarial Attacks}

The DRO problem~\eqref{eq:primal:dro} can be viewed as a zero-sum game in which the decision-maker first chooses a decision~$x \in \cX$, and nature subsequently responds with a distribution~$\P \in \cP$ that adapts to~$x$. Throughout this section we will refer to~\eqref{eq:primal:dro}  as the \emph{primal} DRO problem. In addition, one can study the \emph{dual} DRO problem
\begin{align}
\label{eq:dual:dro}
\sup_{\P \in \cP} \inf_{x \in \cX} ~ \E_{\P} \left[ \ell(x, Z) \right],
\end{align}
where nature first selects a distribution $\P \in \cP$, and the decision-maker subsequently responds with a decision~$x \in \mathcal{X}$ that adapts to~$\P$. In contrast to the primal DRO problem~\eqref{eq:primal:dro}, whose objective function is linear in~$\P$, the objective function of the dual DRO problem~\eqref{eq:dual:dro} is concave in~$\P$. This difference makes the dual DRO problem more challenging to solve. It is now natural to seek conditions that imply strong duality and thus ensure that the infimum of the primal DRO problem~\eqref{eq:primal:dro} coincides with the supremum of the dual DRO problem~\eqref{eq:dual:dro}. One readily verifies that strong duality is implied, for example, by the existence of a Nash equilibrium $(x^\star, \mathbb{P}^\star) \in \mathcal{X} \times \mathcal{P}$ satisfying the saddle point condition
\begin{align}
\label{eq:Nash}
\E_\P \left[ \ell(x^\star, Z) \right] 
\leq \E_{\P^\star} \left[ \ell(x^\star, Z) \right] 
\leq \E_{\P^\star} \left[ \ell(x, Z) \right] 
\quad \forall x \in \cX,~\P \in \cP.
\end{align}
We emphasize that the reverse implication is false, that is, strong duality does not necessarily imply the existence of a Nash equilibrium. The primal DRO problem naturally arises in many applications. The practical usefulness of the dual DRO problem, on the other hand, is less evident because this problem assumes somewhat unrealistically that the decision-maker observes the distribution that governs~$Z$. Nevertheless, the dual DRO problem has deep connections to robust statistics, machine learning as well as several other disciplines as we explain below.

From the perspective of robust statistics, a minimizer~$x^\star$ of the primal DRO problem~\eqref{eq:primal:dro} can be interpreted as a \emph{robust estimator} for the minimizer of the stochastic program $\min_{x \in \cX}~\E_{\P_0} \left[ \ell(x, Z) \right]$ corresponding to an unknown distribution~$\P_0$. When $x^\star$ and $\P^\star$ satisfy the saddle point condition~\eqref{eq:Nash}, then the robust estimator~$x^\star$ constitutes a best response to~$\P^\star$. Hence, it solves the stochastic program corresponding to~$\P^\star$; see also \citep[Chapter~5]{lehmann2006theory}. For this reason, $\P^\star$ is often referred to as the \emph{least favorable distribution}. The existence of~$\P^\star$ makes~$x^\star$ a plausible estimator because it ensures that~$x^\star$ is the minimizer of a stochastic program corresponding to {\em some} distribution in the ambiguity set.

Algorithms for computing Nash equilibria of DRO problems are also relevant for applications in machine learning. To see this, recall that adversarial training aims to immunize machine learning models against adversarial perturbations of the input data \citep{szegedy2014intriguing,ian15adversarial,madry2018towards,wang2019convergence,kurakin2022adversarial}. In this context, it is of interest to generate adversarial examples, that is, maliciously designed inputs that mislead prediction models encoded by parameters~$x \in \cX$. As a na\"ive approach to construct adversarial examples, one could simply solve the worst-case expectation problem
\begin{align}
\label{eq:wc}
\sup_{\P \in \cP} \, \E_\P[\ell(\hat x, Z)],
\end{align}
which seeks a test distribution that maximizes the expected prediction loss of one particular model encoded by~$\hat x$. Thus, any solution~$\P^\star$ of~\eqref{eq:wc} can be viewed as an adversarial attack, and samples drawn from~$\P^\star$ are naturally interpreted as adversarial examples. In order to develop efficient strategies for attacking as well as defending prediction models, however, it is desirable to construct adversarial attacks that fool a broad spectrum of different models. Such attacks are called \emph{transferable} in the machine learning literature \citep{tramer2017space,demontis2019adversarial,kurakin2022adversarial}. The dual DRO problem~\eqref{eq:dual:dro} can be used to construct transferable attacks in a systematic manner. Indeed, the solutions of~\eqref{eq:dual:dro} are {\em not} tailored to a particular model~$\hat x\in\cX$. Instead, they aim to attack \emph{all} models~$x \in \cX$ simultaneously. If the primal DRO problem~\eqref{eq:primal:dro} has a unique minimizer~$x^\star$, then this minimizer can be recovered by solving the stochastic program corresponding to the adversary's Nash strategy~$\P^\star$.

To date, dual DRO problems have only been investigated in the context of specific applications. For example, it is known that the least favorable distributions in distributionally robust estimation and Kalman filtering problems with a $2$-Wasserstein ambiguity set centered at a Gaussian reference distribution are themselves Gaussian and can be computed efficiently via semidefinite programming \citep{shafieezadeh2018wasserstein, nguyen2019bridging}. Several recent studies describe similar results for distributionally robust optimal control problems with a $2$-Wasserstein ambiguity set \citep{al2023distributionally, hajar2023wasserstein, kargin2024lqr, kargin2024distributionally, kargin2024infinite, kargin2024wasserstein, taskesen2024distributionally}. When the Wasserstein ambiguity set is replaced with a Kullback-Leibler ambiguity set around a Gaussian reference distribution, then the least favorable distributions remain Gaussian and can be determined in quasi-closed form \citep{levy2004robust,levy2012robust}. In fact, these results even extend to generalized $\tau$-divergence ambiguity sets \citep{zorzi2016robust,zorzi2017robustness}. Gaussian distributions also solve several other minimax games reminiscent of DRO problems, which are relevant for applications in statistics, control and information theory \citep{bacsar1972minimax,bacsar1973minimax,bacsar1973multistage,bacsar1977optimum,bacsar1982optimum,bacsar1983gaussian,bacsar1984bandwidth,bacsar1985complete,bacsar1986solutions}. Furthermore, it is possible to characterize the Nash equilibria of distributionally robust pricing and auction design problems with support-only and Markov ambiguity sets in closed form \citep{bergemann2008pricing,koccyiugit2020distributionally,koccyiugit2021robust,anunrojwong2024robustness,chen2023screening}. Minimax theorems establishing strong duality between primal and dual DRO problems involving more general optimal transport ambiguity sets are reported in \citep{blanchet2019confidence,shafiee2023new,frank2024existence,pydi2024many}.

\section{Regularization by Robustification}
\label{sec:approximations-of-nature}

Classical stochastic optimization seeks decisions that perform well under a probability distribution~$\hat\P$ estimated from training data. By ignoring any information about estimation errors in~$\hat\P$, however, stochastic optimization tends to output overfitted decisions that incur a low expected loss under~$\hat\P$ but may perform poorly under the unknown population distribution~$\P$. This problem becomes more acute if training data is scarce. A key advantage of DRO {\em vis-\`a-vis} stochastic optimization is that it has access to information about estimation errors. DRO uses this information to prevent overfitting. Robustifying a stochastic optimization problem against distributional uncertainty can thus be viewed as a form of implicit regularization.

We now show that there is often a deep connection between {\em implicit} regularization (achieved by robustifying a problem against distributional uncertainty) and {\em explicit} regularization (achieved by adding a penalty term to the problem's objective function). This discussion complements and extends several results from Section~\ref{sec:analytical-wc}. For example, in Section~\ref{sec:KL-risk} we have seen that the worst-case expected value of a linear loss function with respect to a Kullback-Leibler ambiguity set centered at a Gaussian distribution coincides with the nominal expected loss and a variance regularization term. Similarly, in Section~\ref{sec:wc-wasserstein} we have seen that the worst-case expected value of a convex loss function with respect to a $1$-Wasserstein ambiguity set coincides with the nominal expected loss and a Lipschitz regularization term. See also Sections~\ref{sec:1-wasserstein-risk} and~\ref{sec:p-wasserstein-risk} for some variants and generalizations of this result.

In Section~\ref{sec:phi-divergence-regularization} we will show---in broad generality---that the worst-case expected loss over a $\phi$-divergence ambiguity set is closely related to the nominal expected loss with a {\em variance} regularization term. Similarly, in Section~\ref{sec:wasserstein-regularization} we will show that the worst-case expected loss over a Wasserstein ambiguity set is closely related to the nominal expected loss with {\em variation} and {\em Lipschitz} regularization terms. In Section~\ref{sec:Lipschitz-continuous-risk-measures} we will further show that many popular risk measures are Lipschitz continuous in the distribution of the relevant risk factors with respect to a Wasserstein distance. This implies that the worst-case risk over a Wasserstein ambiguity set is closely related to the nominal risk and a {\em Lipschitz} regularization term. We remark that the connections between robustification and regularization are less well understood for moment ambiguity sets. From Section~\ref{sec:gelbrich-risk} we know that the worst-case risk of a linear loss function over a Gelbrich ambiguity set often coincides with the nominal risk and a $2$-norm regularization term. However, it is unclear whether similar results can be obtained for nonlinear loss functions or other moment ambiguity sets. Therefore, we will not touch on moment ambiguity sets in this section. We emphasize that the connections between robustification and regularization often enable statistical analyses of DRO problems; see Section~\ref{sec:statistics}.


\subsection{$\phi$-Divergence Ambiguity Sets}
\label{sec:phi-divergence-regularization}
As a motivating example, we show that robustification with respect to a Pearson $\chi^2$-divergence ambiguity set is closely related to variance regularization. To see this, recall first that the Pearson $\chi^2$-divergence ambiguity set~\eqref{eq:pearson-chi-squared-ambiguity-set} is defined as
\begin{align*}
\cP = \{ \P \in \mc \P(\cZ): \chi^2(\P, \hat \P) \leq r \}.
\end{align*}
If $\ell$ is a bounded Borel function, Proposition~\ref{prop:variance:chi} readily implies that
\begin{align*}
\cP \subseteq \left\{ \P \in \mc \P(\cZ): \E_\P [\ell(Z)] \leq \E_{\hat \P}[\ell(Z)] + r^\half \V_{\hat \P}[\ell(Z)]^\half\right\},
\end{align*}
and thus we may conclude that
\begin{align*}
\sup_{\P \in \cP} \, \E_\P[\ell(Z)] \leq \E_{\hat \P}[\ell(Z)] + r^\half \V_{\hat \P}[\ell(Z)]^\half.
\end{align*}
Hence, the worst-case expected loss with respect to a Pearson $\chi^2$-divergence ambiguity set of radius~$r$ around~$\hat\P$ is bounded above by the mean-standard deviation risk measure with risk-aversion coefficient~$r^\half$ evaluated under~$\hat\P$. By slight abuse of terminology, the scaled standard deviation $r^\half \V_{\hat \P}[\ell(Z)]^\half$ is commonly referred to as a variance regularizer. By leveraging Theorem~\ref{thm:duality:restricted:phi}, the above bound can be extended to arbitrary (possibly unbounded) Borel loss functions. This extension critically relies on the following lemma. 

\begin{lemma}[Variance Formula]
\label{lem:variance}
For any reference distribution~$\hat \P \in \cP(\cZ)$, size parameter~$r > 0$ and Borel function $\ell\in\cL(\R^d)$ with $\E_{\hat \P}[|\ell(Z)|] < \infty$, we have
\begin{align}
	\label{eq:variance-formula}
	\inf_{\lambda_0 \in \R, \lambda \in \R_+} \, \lambda r + \frac{\E_{\hat \P}[(\ell(Z) - \lambda_0)^2]}{4 \lambda} = r^\half \V_{\hat \P}[\ell(Z)]^\half.
\end{align}
\end{lemma}

\begin{proof}
If $\E_{\hat \P}[\ell(Z)^2] = \infty$, then both sides of~\eqref{eq:variance-formula} evaluate to~$\infty$, and thus the claim follows. In the remainder of the proof, we may thus assume that $\E_{\hat \P}[\ell(Z)^2] < \infty$. In this case, one readily verifies that the partial minimization problem over $\lambda_0$ is solved by $\lambda_0^\star = \E_{\hat \P}[\ell(Z)]$. Substituting $\lambda_0^\star$ back into the objective function reveals that the infimum on the left hand side of~\eqref{eq:variance-formula} equals $\inf_{\lambda \in \R_+} \, \lambda r + {\V_{\hat \P}[\ell(Z)]}/{4 \lambda}$. In order to prove~\eqref{eq:variance-formula}, it suffices to realize that this minimization problem over~$\lambda$ is solved by $\lambda^\star = \sqrt{\V_{\hat \P}[\ell(Z)] / (4r)}$. This observation completes the proof.
\end{proof}

\begin{theorem}[Variance Regularization]
If~$\cP$ is the Pearson $\chi^2$-divergence ambiguity set~\eqref{eq:pearson-chi-squared-ambiguity-set} and $\E_{\hat\P}[|\ell(Z)|]<\infty$, then we have
\begin{align*}
	\sup_{\P \in \cP} \, \E_\P[\ell(Z)] \leq \E_{\hat \P}[\ell(Z)] + r^\half \V_{\hat \P}[\ell(Z)]^\half.
\end{align*}
\end{theorem}

\begin{proof}
The claim trivially holds if~$r = 0$. We may thus assume that~$r > 0$. Recall now that the entropy function~$\phi$ inducing the Pearson $\chi^2$-divergence satisfies $\phi(s)=(s-1)^2$ if~$s\geq 0$ and $\phi(s)=\infty$ if~$s<0$. Hence, the conjugate entropy function~$\phi^*$ satisfies $\phi^*(t)=\frac{1}{4} t^2 +t$ if~$t \geq -2$ and $\phi^*(t)=-1$ if~$t < -2$, and its domain is given by $\dom(\phi^*)=\R$. As $\phi^\infty(1) = \infty$, all distributions $\P\in\cP$ are absolutely continuous with respect to~$\hat\P$. Thus, Theorem~\ref{thm:duality:restricted:phi} applies, and we find
\begin{align*}
	\sup_{\P \in \cP} \; \E_\P[\ell(Z)] 
	&= \inf_{\lambda_0\in\R, \lambda \in \R_+} \lambda_0  + \lambda r + \E_{\hat \P} \left[ (\phi^*)^\pi \left( \ell(Z) - \lambda_0, \lambda \right) \right] \\
	&\leq \inf_{\lambda_0\in\R, \lambda \in \R_+} \E_{\hat \P}[\ell(Z)] + \lambda r + \frac{\E_{\hat \P}[(\ell(Z) - \lambda_0)^2]}{4 \lambda} \\
	&= \E_{\hat \P}[\ell(Z)] + r^\half \V_{\hat \P}[\ell(Z)]^\half,
\end{align*}
where the inequality holds because $\phi^*(t) \leq \frac{1}{4} t^2 + t$, and the second equality follows from Lemma~\ref{lem:variance}. Thus, the claim follows.
\end{proof}

Most $\phi$-divergences are smooth and non-negative and thus resemble the Pearson $\chi^2$-divergence  locally around~$1$ \citep[\ts~7.10]{polyanskiy2024information}. Accordingly, one can use a Taylor expansion to show that robustification over a $\phi$-divergence ambiguity set of sufficiently small size~$r$ is often equivalent to variance regularization. To formalize this result, we assume from now on that~$\phi$ is differentiable.

\begin{assumption}[Differentiability]
\label{assmp:smooth:phi}
The entropy function $\phi$ is twice continuously differentiable on a neighborhood of $1$ with $\phi(1) = \phi'(1) = 0$ and~$\phi''(1) = 2$.
\end{assumption}

The assumption that $\phi'(1)=0$ incurs no loss of generality. Indeed, any entropy function~$\phi$ is equivalent to a transformed entropy function~$\tilde\phi$ defined through $\tilde \phi(t)= \phi(t) - \phi'(1) \cdot t + \phi'(1)$ with $\tilde\phi'(1)=0$. That is, both~$\phi$ and~$\tilde\phi$ induce the same divergence. Note that all entropy functions listed in Table~\ref{tab:phi-divergence}---except for the one associated with the total variation---satisfy~$\phi'(1)=0$. The assumption that $\phi''(1) = 2$ serves as an arbitrary normalization but will simplify calculations. 

Recall now that the restricted $\phi$-divergence ambiguity set~\eqref{eq:restricted-phi-divergence-ambiguity-set} is defined~as
\begin{align*}
\cP = \left\{ \P \in \cP(\cZ) \, : \, \P \ll \hat \P, ~ \D_\phi(\P, \hat \P) \leq r \right\}.
\end{align*} 
Here, $\cZ$ is a closed support set, $r\in\R_+$ is a size parameter, $\phi$ is an entropy function in the sense of Definition~\ref{def:phi}, $\D_\phi$ is the corresponding $\phi$-divergence in the sense of Definition~\ref{def:D_phi}, and~$\hat\P\in\cP(\cZ)$ is a reference distribution. The following theorem provides a leading-order Taylor expansion of the worst-case expectation over~$\cP$.

\begin{theorem}[Taylor Expansion of Worst-Case Expectation]
\label{thm:phi:taylor}
If~$\cP$ is the restricted $\phi$-divergence ambiguity set~\eqref{eq:restricted-phi-divergence-ambiguity-set}, the entropy function~$\phi$ satisfies Assumption~\ref{assmp:smooth:phi} and the loss~$\ell(Z)$ is $\hat\P$-almost surely bounded, then we have
\begin{align}
	\label{eq:phi-Taylor}
	\sup_{\P \in \cP} \, \E_\P[\ell(Z)] = \E_{\hat \P}[\ell(Z)] + r^\half \V_{\hat \P}[\ell(Z)]^\half + o( r^\half ).  
\end{align}
\end{theorem}

\begin{proof}
Note that~\eqref{eq:phi-Taylor} trivially holds if~$r=0$. Similarly, if $\V_{\hat\P}[\ell(Z)]=0$, then~$\ell(Z)$ coincides $\hat\P$-almost surely with $\E_{\hat\P}[\ell(Z)]$. As~$\cP$ is a restricted $\phi$-divergence ambiguity set, this readily implies that $\E_{\P}[\ell(Z)]= \E_{\hat\P}[\ell(Z)]$ for all~$\P\in\cP$. Indeed, any~$\P\in\cP$ satisfies~$\P\ll\hat\P$. Hence, \eqref{eq:phi-Taylor} is again trivially satisfied. In the remainder of the proof we my therefore assume that~$r>0$ and that~$\V_{\hat\P}[\ell(Z)]>0$.

Assumption~\ref{assmp:smooth:phi} implies that~$\phi(s)=(s-1)^2 +o(s^2)$. By Taylor's theorem with Peano remainder, $\phi$ can thus be bounded from below (or above) locally around~$1$ by a quadratic function whose second derivative is slightly smaller (or larger) than $\phi''(1)=2$. Thus, there exists a function $\kappa:\R_+\to\R_+$ with $\lim_{\varepsilon\downarrow 0}\kappa(\varepsilon)=0$ and 
\begin{align}
	\label{eq:bounds-on-phi}
	\frac{1}{1+\kappa(\varepsilon)} \cdot s^2\leq \phi(1+s)\leq (1+\kappa(\varepsilon))\cdot s^2 \quad \forall s\in[-\varepsilon, +\varepsilon]
\end{align}
for all sufficiently small~$\varepsilon$. The rest of the proof proceeds in two steps, both of which exploit~\eqref{eq:bounds-on-phi}. First, we show that the right hands side of~\eqref{eq:phi-Taylor} provides a {\em lower} bound on the worst-case expected loss over~$\cP$ (Step~1). Next, we show that the right hands side of~\eqref{eq:phi-Taylor} also provides an {\em upper} bound on the worst-case expected loss over~$\cP$ (Step~2). Taken together, Steps~1 and~2 will imply the claim.

\paragraph{Step~1.} Every distribution~$\P$ in the restricted $\phi$-divergence ambiguity set~$\cP$ satisfies~$\P\ll\hat\P$ and has thus a density function~$f\in \cL^1(\hat\P)$ with respect to~$\hat\P$. Here, $\cL^1(\hat \P)$ denotes as usual the family of all Borel functions from~$\cZ$ to~$\R$ that are integrable with respect to~$\hat\P$. As~$\P\ll\hat\P$, we have $\D_\phi(\P,\hat\P)=\E_{\hat\P}[\phi(f(Z))]$ (see also Section~\ref{sec:phi-divergence-amgibuity-sets}). Thus, the worst-case expectation problem over~$\cP$ can be recast as
\begin{align*}
	\sup_{\P\in\cP} \E_\P[\ell(Z)] = \left\{ \begin{array}{cl}
		\displaystyle \sup_{f\in \cL^1(\hat\P) } & \displaystyle \E_{\hat\P} \left[ \ell(Z)f(Z)\right]\\
		\st  & \hat\P(f(Z)\geq 0)=1 \\
		& \E_{\hat \P}\left[ f(Z)\right]=1 \\
		& \E_{\hat \P} \left[ \phi(f(Z))\right] \leq r.
	\end{array}\right.
\end{align*}
Renaming~$f(z)+1$ as~$f(z)$ further yields
\begin{align}
	\label{eq:shifted-wc-phi}
	\sup_{\P\in\cP} \E_\P[\ell(Z)] = \E_{\hat \P}[\ell(Z)] + \left\{ \begin{array}{cl}
		\displaystyle \sup_{f\in \cL^1(\hat\P) } & \displaystyle \E_{\hat\P} \left[ \ell(Z)f(Z)\right]\\
		\st  & \hat\P(f(Z)\geq -1)=1 \\
		& \E_{\hat \P}\left[ f(Z)\right]=0 \\
		& \E_{\hat \P} \left[ \phi(1+ f(Z))\right] \leq r.
	\end{array}\right.
\end{align}
Next, introduce an auxiliary function $\varepsilon:\R_+\to\R_+$ satisfying
\[
\varepsilon(r) = 2r^{\half}\cdot \frac{\esssup_{\hat\P}[|\ell(Z)-\E_{\hat \P}[\ell(Z)]|]}{\V_{\hat \P}[\ell(Z)]^\half}.
\]
In addition, for every $r\in\R_+$, define the function $f^\star_r\in\cL^1(\hat\P)$ through
\[
f^\star_r (z)=\frac{r^{\half}}{(1+\kappa(\varepsilon(r)))^\half} \cdot \frac{\ell(z)- \E_{\hat\P}[\ell(Z)]}{\V_{\hat \P}[\ell(Z)]^\half}.
\]
By construction, we may thus conclude that
\begin{align}
	\label{eq:bounded-f_r}
	\left| f^\star_r (Z) \right| & \leq r^{\half} \cdot \frac{ \left| \ell(Z) - \E_{\hat\P}[\ell(Z)] \right|}{\V_{\hat \P}[\ell(Z)]^\half} \leq \varepsilon(r) \quad \hat\P\text{-a.s.}
\end{align}
for every~$r\in\R_+$, where the two inequalities follow from the definitions of~$f^\star_r(z)$ and~$\varepsilon(r)$, respectively. In addition, we have $\E_{\hat\P}[f^\star_r (Z)]=0$ and
\begin{align*}
	\E_{\hat\P} \left[ \phi(1+f^\star_r (Z))\right] & \leq (1+\kappa(\varepsilon(r))) \cdot \E_{\hat\P} \left[ f^\star_r (Z)^2)\right] = r
\end{align*}
for all sufficiently small~$r$. The inequality in the above expression follows from~\eqref{eq:bounded-f_r} and from the upper bound on~$\phi$ in~\eqref{eq:bounds-on-phi}, which holds for all sufficiently small~$\varepsilon$. The equality exploits the definition of~$f^\star_r$. This shows that~$f^\star_r$ constitutes a feasible solution for the maximization problem in~\eqref{eq:shifted-wc-phi} if~$r$ is sufficiently small. Substituting~$f^\star_r$ into~\eqref{eq:shifted-wc-phi} then yields the desired lower bound. Indeed, we have
\begin{align*}
	\sup_{\P\in\cP} \E_\P[\ell(Z)] & \geq \E_{\hat \P}[\ell(Z)] + \E_{\hat\P} \left[ \ell(Z)f^\star_r(Z)\right] \\[-1ex]
	& = \E_{\hat \P}[\ell(Z)] +\frac{r^{\half}}{(1+\kappa(\varepsilon(r)))^\half} \cdot \frac{\E_{\hat \P} \left[ \ell(Z)(\ell(Z)-\E_{\hat\P}[\ell(Z)]) \right] }{\V_{\hat \P}[\ell(Z)]^\half} \\
	& = \E_{\hat \P}[\ell(Z)] + r^\half \V_{\hat \P}[\ell(Z)]^\half+o(r^\half),
\end{align*}
for all sufficiently small~$r$, where the first equality follows from the definition of~$f^\star_r$. The second equality exploits the Taylor expansion of the inverse square root function around~$1$ and the elementary observation that $\lim_{r\downarrow 0}\kappa(\varepsilon(r))=0$. 

\paragraph{Step~2.} The Huber loss $h_\varepsilon:\R\to \R$ with tuning parameter $\varepsilon>0$ is defined through
\[
h_\varepsilon(s)=\left\{ \begin{array}{ll}
	\half s^2 & \text{if } |s|\leq \varepsilon, \\
	\varepsilon |s| -\half\varepsilon^2 & \text{otherwise}.
\end{array} \right.
\]
By construction, $h_\varepsilon$ is continuously differentiable, depends quadratically on~$s$ if $|s|\leq\varepsilon$ and depends linearly on~$s$ if~$|s|>\varepsilon$. Its conjugate $h_\varepsilon^*:\R\to\overline\R$ satisfies 
\[
h_\varepsilon^*(t)=\left\{ \begin{array}{ll}
	\half t^2 & \text{if } |t|\leq \varepsilon, \\
	\infty & \text{otherwise.}
\end{array} \right.
\]
The lower bound on~$\phi$ in~\eqref{eq:bounds-on-phi} and the convexity of~$\phi$ imply that
\[
\phi(s) \geq \frac{2}{1+\kappa(\varepsilon)} h_\varepsilon(s-1)\quad \forall s\in\R
\]
whenever~$\varepsilon$ is sufficiently small. This uniform lower bound on~$\phi$ in terms of~$h_\varepsilon$  gives rise to a uniform upper bound on~$\phi^*$ in terms of~$h^*_\varepsilon$. Indeed, we have
\begin{align}
	\nonumber
	\phi^*(t) & \leq \sup_{s\in\R}~ st-\frac{2}{1+\kappa(\varepsilon)} h_\varepsilon(s-1) \\
	& = t+ \frac{2}{1+\kappa(\varepsilon)} \textstyle h^*_\varepsilon \left(\half t(1+\kappa(\varepsilon))\right) = t+\left\{ \begin{array}{ll}
		\frac{(1+\kappa(\varepsilon))t^2}{4} & \text{if } |t| \leq  \frac{2\varepsilon}{1+\kappa(\varepsilon)}, \\
		\infty & \text{otherwise,}
	\end{array}\right.
	\label{eq:phi-conjugate-ub}
\end{align}
for all sufficiently small~$\varepsilon$. The first equality in~\eqref{eq:phi-conjugate-ub} is obtained by applying the variable transformation $s\leftarrow s-1$ and by extracting the constant $2/(1+\kappa(\varepsilon))$ from the supremum. The second equality follows from the definition of~$h^*_\varepsilon$. By weak duality as established in Theorem~\ref{thm:duality:restricted:phi}, we then find
\begin{align}
	\sup_{\P\in\cP} \E_\P[\ell(Z)] & \leq \inf_{\lambda_0\in\R, \lambda \in \R_+} \lambda_0  + \lambda r + \E_{\hat \P} \left[ (\phi^*)^\pi \left( \ell(Z) - \lambda_0, \lambda \right) \right] \nonumber \\
	& \leq \left\{ \begin{array}{cl}
		\displaystyle \inf_{\lambda_0\in\R, \lambda \in \R_+}& \E_{\hat \P} [\ell(Z)]  + \lambda r + \frac{1+\kappa(\varepsilon(r))}{4 \lambda} \E_{\hat \P} \left[ \left( \ell(Z) - \lambda_0\right)^2 \right]\\
		\st & \hat\P\left( |\ell(Z)-\lambda_0|\leq \frac{2\varepsilon(r)\lambda}{1+\kappa(\varepsilon(r))}\right) =1, \end{array}\right.
	\label{eq:variance-reg-ub1}
\end{align}
where the second inequality follows from the definition of the perspective function and from~\eqref{eq:phi-conjugate-ub}, which holds for all sufficiently small~$\varepsilon$. Here, we have re-used the function~$\varepsilon(r)$ introduced in Step~1. Next, we set $\lambda_0^\star= \E_{\hat \P}[\ell(Z)]$ and define
\[
\lambda^\star_r = \frac{\big(1+\kappa(\varepsilon(r)) \big)^\half}{2r^\half} \cdot \V_{\hat \P}[\ell(Z)]^\half
\]
for any~$r>0$. Note that $(\lambda^\star_0, \lambda_r^\star)$ is feasible in~\eqref{eq:variance-reg-ub1} provided that~$r$ is sufficiently small; in particular, $r$ must be small enough to satisfy $\kappa(\varepsilon(r))\leq 3$. Indeed, we have 
\begin{align*}
	& \hat\P\left( |\ell(Z)-\lambda^\star_0|\leq \frac{2\varepsilon(r) \lambda_r^\star}{1+\kappa(\varepsilon(r))}\right) = \hat\P\left( |\ell(Z)-\E_{\hat\P}[\ell(Z)]|\leq \frac{\varepsilon(r) \V_{\hat \P}[\ell(Z)]^\half}{r^\half \big(1+\kappa(\varepsilon(r))\big)^\half}\right) \\
	& = \hat\P\left( |\ell(Z)-\E_{\hat\P}[\ell(Z)]|\leq \frac{2}{\big(1+\kappa(\varepsilon(r))\big)^\half} \esssup_{\hat\P}\big[|\ell(Z)-\E_{\hat \P}[\ell(Z)]|\big] \right) =1,
\end{align*}
where the first equality follows from the definitions of~$\lambda_0^\star$ and~$\lambda_r^\star$, the second equality follows from the definition of~$\varepsilon(r)$, and the last equality holds because $\kappa(\varepsilon(r))\leq 3$. Substituting $(\lambda^\star_0, \lambda_r^\star)$ into~\eqref{eq:variance-reg-ub1} then yields the desired upper bound.
\begin{align*}
	\sup_{\P\in\cP} \E_\P[\ell(Z)] & \leq \E_{\hat \P} [\ell(Z)]  + \lambda^\star_r r + \frac{1+\kappa(\varepsilon(r))}{4 \lambda^\star_r} \E_{\hat \P} \left[ \left( \ell(Z) - \lambda^\star_0\right)^2 \right] \\
	& = \E_{\hat \P} [\ell(Z)] + \frac{(1+\kappa(\varepsilon(r)))^\half}{2} \cdot r^\half \V_{\hat \P}[\ell(Z)]^\half \\
	&\qquad + \frac{(1+\kappa(\varepsilon(r)))^{\frac{1}{2}}}{2 \V_{\hat \P}[\ell(Z)]^\half} r^\half \E_{\hat \P} \left[ \left( \ell(Z) - \E_{\hat \P}[\ell(Z)]\right)^2 \right] \\
	&= \E_{\hat \P} [\ell(Z)] + r^\half \V_{\hat \P}[\ell(Z)]^\half + o(r^\half)
\end{align*}
Here, the first equality follows from the definitions of~$\lambda_0^\star$ and~$\lambda_r^\star$, and the second equality holds because $\lim_{r\downarrow 0}\kappa(\varepsilon(r))=0$. Hence, the claim follows.
\end{proof}

Theorem~\ref{thm:phi:taylor} reveals that, up to leading order in~$r$, robustification with respect to a restricted divergence ambiguity set is equivalent to variance regularization. The requirement that the loss must be almost surely bounded is restrictive but necessary. However, it can be relaxed if the entropy function~$\phi$ grows superlinearly. As an example, recall from Proposition~\ref{prop:KL-analytical} that the worst-case expectation of a linear loss function with respect to a Kullback-Leibler ambiguity set centered at a Gaussian distribution equals precisely $\E_{\hat \P} [\ell(Z)] + (2r)^\half \V_{\hat \P}[\ell(Z)]^\half$ without any higher-order error terms. This formula is consistent with Theorem~\ref{thm:phi:taylor} because the entropy function of the Kullback-Leibler divergence satisfies~$\phi''(1)=1$. Thus, it must be scaled by~2 to satisfy Assumption~\ref{assmp:smooth:phi}. Note that any (non-constant) linear loss functions fails to be $\hat\P$-almost surely bounded with respect to any (non-degenerate) Gaussian distribution~$\hat\P$. However, the conclusions of Theorem~\ref{thm:phi:taylor} hold nevertheless because the underlying entropy function grows faster than linearly.

A Taylor expansion akin to~\eqref{eq:phi-Taylor} for {\em empirical} reference distributions and for the Kullback-Leibler divergence ambiguity set~\eqref{eq:KL-ambiguity-set} is due to \citet{lam2019recovering}. \citet{duchi2021statistics} generalize this result to other $\phi$-divergences. Similar results for empirical reference distributions are also reported by \citet{lam2016robust, lam2018sensitivity, duchi2019variance} and \citet{blanchet2023statistical} in different contexts. In a parallel line of research, \citet{gotoh2018robust,gotoh2021calibration} derive a Taylor expansion of the penalty-based worst-case expected loss $\sup_{\P\in\cP(\cZ)} \E_\P[\ell(Z)] -\frac{1}{r}\D_\phi(\P,\hat\P)$. They focus again on discrete empirical reference distributions and provide both first- as well as higher-order terms of the corresponding Taylor expansion. 

\citet{maurer2009empirical} show that variance-regularized empirical risk minimization may provide faster rates of convergence to the expected loss under the population distribution compared to standard empirical risk minimization. This improved convergence highlights the potential benefits of incorporating variance regularization in the learning process. Unfortunately, simple stochastic optimization problems with a mean-variance objective are NP-hard even if the underlying loss function is convex in the decision variables \citep{ahmed2006convexity}. In contrast, the worst-case expectation with respect to any ambiguity set preserves the convexity of the underlying loss function. Theorem~\ref{thm:phi:taylor} thus suggests that the worst-case expected loss over a restricted $\phi$-divergence ambiguity set provides a convex surrogate for the nonconvex---but statistically attractive---variance-regularized empirical loss.

\subsection{Wasserstein Ambiguity Sets}
\label{sec:wasserstein-regularization}
As a motivating example, we show that robustification with respect to a $1$-Wasser\-stein ambiguity set is closely related to Lipschitz regularization. To see this, recall first that the $p$-Wasserstein ambiguity set~\eqref{eq:p-wasserstein-ball} for~$p \in [1, \infty)$ is defined as
\begin{align*}
\cP = \left\{ \P \in \cP(\cZ): \W_p(\P, \hat \P) \leq r \right\}.
\end{align*}
Here, $\cZ$ is a closed support set, $r\in\R_+$ is a size parameter, $\W_p$ is the $p$-Wasserstein distance induced by a norm~$\| \cdot \|$ on~$\R^d$ (see Definition~\ref{def:p-Wassertein}), and~$\hat\P\in\cP(\cZ)$ is a reference distribution. If the loss function~$\ell$ is piecewise concave, then the worst-case expectation problem~\eqref{eq:worst-case:expectation} over~$\cP$ can be reformulated as a finite convex program (see Theorem~\ref{thm:finite:convex:OT:I}). For more general loss functions, however, exact reformulations of~\eqref{eq:worst-case:expectation} are unavailable. We now show that if~$p=1$ and~$\ell$ is Lipschitz continuous as well as $\hat \P$-integrable, then the worst-case expectation problem~\eqref{eq:worst-case:expectation} admits a simple upper bound that involves the Lipschitz modulus of~$\ell$.

\begin{proposition}[Lipschitz Regularization]
\label{prop:KR-bound}
Suppose that~$\cP$ is the $1$-Wasserstein ambiguity set of radius~$r\in\R_+$ around $\hat\P\in \cP(\cZ)$, and~$\W_1$ is induced by a norm~$\| \cdot \|$ on~$\R^d$. In addition, suppose that~$\ell$ is Lipschitz continuous on~$\cZ$ with respect to the same norm~$\|\cdot\|$ and that~$\E_{\hat \P}[|\ell(Z)|]<\infty$. Then, we have
\begin{align}
	\label{eq:lipschitz-regularization}
	\sup_{\P \in \cP} \, \E_\P[\ell(Z)] \leq \E_{\hat \P}[\ell(Z)] + r \cdot \lip(\ell).
\end{align}
\end{proposition}

We emphasize that evaluating the Lipschitz modulus of a generic loss function is computationally challenging. For example, one can show that computing $\lip(\ell)$ is NP-hard even if $\|\cdot\|$ is the $\infty$-norm and even if~$\ell$ is a (convex) conic quadratic loss function; see, {\em e.g.}, \citep[Remark~3]{kuhn2019wasserstein} for a simple proof.

\begin{proof}[Proof of Proposition~\ref{prop:KR-bound}]
The Kantorovich-Rubinstein duality implies that
\begin{align*}
	\sup_{\P \in \cP} \, \E_\P[\ell(Z)] & = \E_{\hat \P}[\ell(Z)]+ \lip(\ell)\cdot \left( \sup_{\P \in \cP} \, \E_\P\left[\frac{\ell(Z)}{\lip(\ell)}\right] - \E_{\hat \P}\left[ \frac{\ell(Z)}{\lip(\ell)}\right] \right) \\
	& \leq \E_{\hat \P}[\ell(Z)] + r \cdot \lip(\ell).
\end{align*}
Indeed, the normalized function $\ell/\lip(\ell)$ is Lipschitz continuous and has Lipschitz modulus at most~$1$. By Corollary~\ref{cor:kanorovich-rubinstein-duality}, we thus have for every~$\P\in\cP$ that
\[
\E_\P\left[\frac{\ell(Z)}{\lip(\ell)}\right] - \E_{\hat \P}\left[ \frac{\ell(Z)}{\lip(\ell)}\right] \leq \W_1(\P,\hat\P)\leq r.
\]
Therefore, the claim follows.
\end{proof}

Close connections between Wasserstein distributionally robust optimization and Lipschitz regularization have been discovered in different contexts \citep{mohajerin2018data,shafieezadeh2015distributionally,shafieezadeh2019regularization,gao2024wasserstein}. Recall that the upper bound in~\eqref{eq:lipschitz-regularization} is tight. Indeed, Proposition~\ref{prop:1-wasserstein-analytical} implies that~\eqref{eq:lipschitz-regularization} collapses to an equality if~$\ell$ is convex and~$\cZ=\R^d$. The Lipschitz modulus of the loss function encodes its variability. Thus, the Lipschitz regularization term in~\eqref{eq:lipschitz-regularization} penalizes loss functions that display a high degree of variability. In the following we will derive generalized variation regularization bounds akin to~\eqref{eq:lipschitz-regularization} for worst-case expectation problems over $p$-Wasserstein ambiguity sets for~$p\in\N$.

Toward this goal, for any~$k \in \Z_+$ we use~$D^k \ell(\hat z)$, to denote the totally symmetric tensor of all $k$-th order partial derivatives of~$\ell(z)$ at $z=\hat z$. Accordingly, $D^k \ell(\hat z)[z_1, \ldots, z_k]$ stands for the directional derivative of $\ell(z)$ along the directions~$z_i\in\R^d$ for~$i \in [k]$. If~$z_i=z$ for all~$i\in[k]$, then we use~$D^k \ell(\hat z)[z]^k$ as a shorthand for $D^k \ell(\hat z)[z, \ldots, z]$. Any norm~$\|\cdot\|$ on~$\R^d$ induces a norm on the space of totally symmetric $k$-th order tensors through
\begin{align*}
\|D^k \ell(\hat z)\| 
&= \sup_{z_1, \ldots, z_k\in\R^d} \left\{ \left| D^k \ell(\hat  z)[z_1, \ldots, z_k] \right|: \|z_i\| \leq 1 ~ \forall i \in [k] \right\} \\
&= \sup_{z\in\R^d} \left\{ \left| D^k \ell(\hat z)[z]^k \right|: \| z \| \leq 1 \right\} ,
\end{align*}
where the second equality exploits the symmetry of~$D^k \ell(\hat z)$ \citep[Satz~1]{Banach1938}. By slight abuse of notation, we use the same symbol~$\|\cdot\|$ for the tensor norm as for the underlying vector norm~$\|\cdot\|$. The following theorem generalizes Proposition~\ref{prop:KR-bound} to any~$p\in\N$. This result is due to \citet[Theorem~3.2]{shafiee2023new}.

\begin{theorem}[Variation and Lipschitz Regularization]
\label{thm:OT:safe:approximation}
If~$\cP$ is the $p$-Wasserstein ambiguity set~\eqref{eq:p-wasserstein-ball} for some~$p \in \N$, where $\W_p$ is induced by a norm~$\| \cdot \|$ on~$\R^d$, $\cZ$ is convex and~$\ell$ is~$p-1$ times continuously differentiable, then we have 
\begin{align*}
	\sup_{\P \in \cP} \E_\P \left[ \ell(Z) \right] 
	\leq \E_{\hat \P} [\ell(\hat Z)] + \sum_{k=1}^{p-1} \frac{r^k}{k!}\, \E_{\hat \P} \left[ \| D^{k} \ell(\hat Z)\|^{q_k} \right]^{\frac{1}{q_k}} + \frac{r^p}{p!} \, \lip( D^{p-1}\ell), 
\end{align*}
where~$p_k = p/k$ and~$q_k = p / (p-k)$ for all $k\in[p-1]$. 
\end{theorem}

\begin{proof}
Select any~$\P\in \cP$ and any optimal coupling~$\gamma^\star\in\Gamma(\P,\hat\P)$ with $\W_p(\P, \hat{\P})=\E_{\gamma^\star}[\|Z-\hat Z\|^p]^{1/p}$, which exists by Lemma~\ref{lem:OT-solvability}. As~$\gamma^\star\in\Gamma(\P,\hat\P)$, we have
\begin{align*}
	\E_\P [ \ell(Z)] - \E_{\hat \P} [ \ell(\hat Z)] 
	= \E_{\gamma^\star} \left[ \ell(Z) - \ell(\hat Z)\right].
\end{align*}
By \citep[Theorem~2.2.5]{krantz2002primer}, we can expand $\ell(z) - \ell(\hat z)$ as a Taylor series with Lagrange remainder. Thus, there exists a Borel function $f:\cZ\times\cZ \to\cZ$ that maps any pair~$(z, \hat z)$ to a point on the line segment between~$z$ and~$\hat z$ such that
\begin{align}
	\ell(z) - \ell(\hat z) 
	&= \sum_{k=1}^{p-1} \frac{1}{k!} D^{k} \ell(\hat z)\left[ z - \hat z \right]^k + \frac{1}{p!} D^{p} \ell(f(z,\hat z)) \left[ z - \hat z \right]^p \notag \\ 
	&\leq \sum_{k=1}^{p-1} \frac{1}{k!} \| D^{k} \ell(\hat z)\| \| z - \hat z \|^k + \frac{1}{p!} \|D^{p} \ell(f(z, \hat z)) \| \| z - \hat z \|^p.
	\label{eq:upper:bound}
\end{align}
The inequality in~\eqref{eq:upper:bound} follows from the  definition of the tensor norm. By H\"older's inequality, the expected value of the $k$-th term in~\eqref{eq:upper:bound} with respect to~$\gamma^\star$ satisfies
\begin{align*}
	\E_{\gamma^\star} \left[ \| D^{k} \ell(\hat Z)\| \| Z-\hat Z \|^k \right] 
	&\leq \E_{\gamma^\star}\left[ \| Z-\hat Z\|^{k p_k} \right]^\frac{1}{p_k} \E_{\gamma^\star}\left[ \| D^{k} \ell(\hat Z)\|^{q_k} \right]^{\frac{1}{q_k}}  \\ 
	&\leq r^k \E_{\hat \P}\left[ \| D^{k} \ell(\hat Z)\|^{q_k} \right]^{\frac{1}{q_k}},
\end{align*}
where $p_k = p/k$ and $q_k = p / (p-k)$ represent conjugate exponents. The second inequality in the above expression holds because~$\gamma^\star\in\Gamma(\P,\hat\P)$, which implies that
\[
\E_{\gamma^\star} [\|Z-\hat Z\|^{k p_k}]^{\frac{1}{p_k}} = \E_{\gamma^\star} [\|Z-\hat Z\|^p]^{\frac{k}{p}} = \W_p(\P, \hat{\P})^k \leq r^k.
\]
As~$\cZ$ is convex, we may conclude that~$f(z,\hat z)\in\cZ$ for all~$z,\hat z\in\cZ$. Thus, the expected value of the Lagrange remainder in~\eqref{eq:upper:bound} with respect to~$\gamma^\star$ satisfies
\begin{align*}
	& \E_{\gamma^\star} \left[ \|D^{p}\ell(f(Z, \hat Z))\| \| Z-\hat Z \|^p \right] \\
	&\leq \sup_{\hat z \in \cZ} \| D^{p} \ell(\hat z)\| \; \E_{\gamma^\star} \left[ \| Z-\hat Z \|^p \right] \leq r^p \sup_{\hat z \in \cZ} \| D^{p} \ell(\hat z)\|\leq r^p \lip( D^{p-1}\ell),
\end{align*}
where the second inequality exploits again H\"older's inequality and the properties of the optimal coupling~$\gamma^\star$. The third inequality follows from the mean value theorem. The desired inequality is finally obtained by combining the upper bounds on the expected values of all terms in~\eqref{eq:upper:bound} with respect to~$\gamma^\star$.
\end{proof}

Theorem~\ref{thm:OT:safe:approximation} shows that the worst-case expected loss over a $p$-Wasser\-stein ball is bounded above by the sum of the expected loss under the reference distribution, $p-1$ variation regularization terms, and a Lipschitz regularization term. Note that~$p_1=p$ and $q=q_1=p/(p-1)$ are H\"older conjugates and that $D^1\ell=\nabla\ell$. Thus, the term corresponding to~$k=1$ in the upper bound of Theorem~\ref{thm:OT:safe:approximation} can be expressed more explicitly as $\E_{\hat \P}[\|\nabla\ell(\hat Z)\|^q]]^{1/q}$. The next theorem, which is adapted from \citep{bartl2020robust, gao2024wasserstein}, reveals that this variation regularizer matches the leading term of a Taylor expansion of the worst-case expected loss in the radius~$r$ of the $p$-Wasserstein ball for any~$p>1$.

\begin{theorem}[Taylor Expansion of Worst-Case Expectation]
\label{thm:OT:taylor}
Suppose that~$\cP$ is the $p$-Wasserstein ambiguity set~\eqref{eq:p-wasserstein-ball} for some~$p > 1$, where~$\W_p$ is induced by a norm~$\| \cdot \|$ on~$\R^d$, and $\cZ$ is convex. Suppose also that the following hold.
\begin{itemize}
	\item[(i)] {\bf Growth Condition.} There exist $g,\delta_0>0$ such that $\ell(z)-\ell(\hat z)\leq g\|z-\hat z\|^p$ for all $z,\hat z\in\cZ$ with $\|z-\hat z\|>\delta_0$.
	\item[(ii)] {\bf Smoothness Condition.} There exists $L>0$ such that $\|\nabla\ell(z)-\nabla\ell(\hat z)\|_*\leq L \|z-\hat z\|$ for all $z,\hat z\in\cZ$, where $\|\cdot\|_*$ is the norm dual to $\|\cdot\|$.
	\item[(iii)] {\bf Integrability Condition.} Both $\E_{\hat \P} [ \|\nabla \ell(\hat Z)\|_*^q]$ and $\E_{\hat \P} [ \|\nabla \ell(\hat Z)\|_*^{2q-2}]$ are finite, where $q=p/(p-1)$ is the H\"older conjugate of~$p$.
\end{itemize}
Then, we have
\begin{align}
	\label{eq:OT-Taylor}
	\sup_{\P \in \cP} \; \E_\P[\ell(Z)] 
	= \E_{\hat\P} \left[ \ell(Z) \right] + r \cdot \E_{\hat \P} \left[ \| \nabla \ell(Z)\|_*^{q} \right]^{\frac{1}{q}} + o(r).
\end{align}
\end{theorem}

Recall that all norms on~$\R^d$ are topologically equivalent. Thus, in the smoothness condition we could equivalently use the primal norm instead of the dual norm to measure differences between gradients. However, working with the dual norm is more convenient and will simplify the proof of Theorem~\ref{thm:OT:taylor}.

\begin{proof}[Proof of Theorem~\ref{thm:OT:taylor}]
For any fixed~$\delta\in\R_+$ and~$\hat z\in\cZ$, we define the variation of the loss function~$\ell$ over a norm ball of radius~$\delta$ around~$\hat z$ as
\[
V_\delta(\hat z) = \sup_{z\in\cZ} \big\{ \ell(z)-\ell(\hat z) : \|z-\hat z\|\leq \delta \big\}.
\]
Note that~$V_\delta(\hat z)$ is finite because~$\ell$ is continuous thanks to the smoothness condition. As a preparation to prove the theorem, we first establish simple upper and lower bounds on~$V_\delta(\hat z)$. As~$\cZ$ is convex, the line segment from~$\hat z$ to any~$z\in\cZ$ is contained in~$\cZ$. The mean value theorem then implies that there exists a point $\bar z\in\cZ$ on this line segment that satisfies $\ell(z)-\ell(\hat z)=\nabla\ell(\bar z)^\top(z-\hat z)$. Thus, we have
\begin{equation*}
	\begin{aligned}
		\big|\ell(z)-\ell(\hat z) -\nabla\ell(\hat z)^\top (z-\hat z) \big| & = \big| \nabla\ell(\bar z)^\top (z-\hat z) - \nabla\ell(\hat z)^\top (z-\hat z)\big| \\
		& \leq \| \nabla\ell(\bar z)-\nabla\ell(\hat z)\|_* \|z-\hat z\| \leq L\|z-\hat z\|^2,
	\end{aligned}
\end{equation*}
where the two inequalities follow from the definition of the dual norm and from the smoothness condition, respectively. This implies that
\begin{equation}
	\label{eq:gradient-estimate}
	\nabla\ell(\hat z)^\top (z-\hat z) - L\|z-\hat z\|^2 \leq \ell(z)-\ell(\hat z)\leq \nabla\ell(\hat z)^\top (z-\hat z) + L\|z-\hat z\|^2.
\end{equation}
The first inequality in~\eqref{eq:gradient-estimate} gives rise to a lower bound on $V_\delta(\hat z)$. Indeed, we find
\begin{align}
	\nonumber
	V_\delta(\hat z) &\geq \sup_{z\in\cZ} \left\{ \nabla\ell(\hat z)^\top (z-\hat z) - L\|z-\hat z\|^2 : \|z-\hat z\|\leq \delta \right\} \\
	& \geq \sup_{z\in\cZ} \left\{ \nabla\ell(\hat z)^\top (z-\hat z) : \|z-\hat z\|\leq \delta \right\} -L\delta^2 = \|\nabla \ell(\hat z)\|_*\delta-L\delta^2,
	\label{eq:g-lower-bound}
\end{align}
where the equality follows from the definition of the dual norm. Similarly, the second inequality in~\eqref{eq:gradient-estimate} gives rise to the following upper bound on $V_\delta(\hat z)$.
\begin{align}
	\label{eq:g-upper-bound-plarger2}
	V_\delta(\hat z) &\leq \|\nabla \ell(\hat z)\|_*\delta + L\delta^2
	\quad \forall \delta \in \mathbb{R}_+
\end{align}
This upper bound grows quadratically with~$\delta$ and is therefore too loose for our purposes if~$p < 2$. In this case, we must establish an alternative upper bound that grows only as~$\delta^p$. This is possible thanks to the growth condition on~$\ell$. To see this, define the worst-case variation of~$\ell$ over any ball of radius~$\delta_0$ as
\[
\overline V=\sup \big\{ \ell(z)-\ell(\hat z):z,\hat z\in\cZ, ~\|z-\hat z\|\leq\delta_0 \big\}.
\]
One can show that~$\overline V$ is finite. If~$\cZ$ is compact, then this is a consequence of Weierstrass' maximum theorem, which applies because~$\ell$ is continuous. If~$\cZ$ is unbounded, on the other hand, then this is a consequence of the convexity of~$\cZ$ and the growth condition on~$\ell$. In this case, there exists a recession direction~$d$ of~$\cZ$ with $\|d\|=2\delta_0$. Thus, for all $z,\hat z\in\cZ$ with $\|z-\hat z\|\leq \delta$ we have
\begin{align*}
	\ell(z)-\ell(\hat z) & \leq |\ell(z)-\ell(z+d)| + |\ell(z + d)-\ell(\hat z)| \\ & \leq g\|d\|^p+g\|z+d-\hat z\|^p \leq g \big( (2\delta_0)^p + (3\delta_0)^p \big).
\end{align*}
The second inequality follows from the growth condition on~$\ell$ and the estimates $\|z -(z+d)\|=\delta_0$ and $\|(z+d)-\hat z\|\geq \|d\|-\|z-\hat z\|\geq \delta_0$. Thus, $\ell(z)-\ell(\hat z)$ admits a finite upper bound independent of~$z$ and~$\hat z$, which confirms that~ $\overline V$ is finite.

The growth condition on~$\ell$ ensures that $V_\delta(\hat z) \leq \max\{\overline V, g\delta^p\}$. Combining this estimate with~\eqref{eq:g-upper-bound-plarger2} and defining $u(\delta) = \min \{ \max\{ \overline V , g\delta^p\}, L\delta^2 \}$ yields
\begin{align*}
	V_\delta(\hat z) & \leq \min \left\{ \max\{ \overline V , g\delta^p\}, \|\nabla \ell(\hat z)\|_*\delta+L\delta^2 \right\} \leq \|\nabla \ell(\hat z)\|_*\delta+ u(\delta) .
\end{align*}
Note that $u(\delta)=g\delta^p$ for all sufficiently large~$\delta$ and $u(\delta)=L\delta^2$ for all sufficiently small~$\delta$. In between there is a (possibly empty) interval on which~$u(\delta)=\overline V$ is constant. Since $p \leq 2$, in all three regimes, $u(\delta)$ can be bounded above by~$g'\delta^p$ for some growth parameter~$g'\in\R_+$. Setting~$G$ to the largest of these three growth parameters, we may thus conclude that
\begin{equation}
	\label{eq:g-upper-bound-psmaller2}
	V_\delta(\hat z) \leq \|\nabla \ell(\hat z)\|_*\delta+ G\delta^p \quad \forall\delta\in\R_+.
\end{equation}
Thus, if~$p\leq 2$, then $V_\delta(\hat z)$ admits an upper bound that grows only as~$\delta^p$.

The remainder of the proof proceeds in two steps. First, we show that the right hand side of~\eqref{eq:OT-Taylor} provides a {\em lower} bound on the worst-case expected loss over~$\cP$ (Step~1). Next, we show that the right hand side of~\eqref{eq:OT-Taylor} also provides an {\em upper} bound on the worst-case expected loss over~$\cP$ (Step~2). This will prove the claim.

\paragraph{Step~1.} Define~$\cF$ as the family of all Borel functions $f:\cZ \to\cZ$. Any~$f\in\cF$ induces a pushforward distribution~$\cP=\hat\P\circ f^{-1}$ supported on~$\cZ$. By restricting the Wasserstein ball around~$\hat\P$ to contain only such pushforward distributions, we find 
\begin{subequations}
	\begin{align}
		\label{eq:OT-LB-a}
		\sup_{\P\in\cP} \E_\P[\ell(Z)] & \geq \sup_{f\in\cF} \left\{ \E_{\hat\P} \left[\ell(f(\hat Z)) \right] : \E_{\hat \P} \left[\|f(\hat Z)-\hat Z\|^p \right]\leq r^p \right\} \\
		& \geq \E_{\hat\P} \left[\ell(\hat Z) \right] + \sup_{\delta\in\Delta} \left\{ \E_{\hat\P} \left[ V_{\delta(\hat Z)}(\hat Z) \right] : \E_{\hat \P} \left[\delta(\hat Z)^p \right]\leq r^p \right\},
		\label{eq:OT-LB-b}
	\end{align}
\end{subequations}
where the set~$\Delta$ in~\eqref{eq:OT-LB-b} represents the family of all Borel functions~$\delta:\cZ\to\R_+$. The second inequality in the above expression can be justified as follows. Select any $\delta\in\Delta$ feasible in~\eqref{eq:OT-LB-b}, and define $f\in\cF$ as any Borel function satisfying
\[
f(\hat z)\in\arg\max_{z\in\cZ} \big\{ \ell(z) : \|z-\hat z\|\leq\delta(\hat z) \big\} \quad\forall \hat z\in\cZ.
\]
Such a Borel function exists thanks to \citep[Corollary~14.6 and Theorem~14.37]{rockafellar2009variational}. As~$\delta$ is feasible in~\eqref{eq:OT-LB-b}, this function~$f$ satisfies
\[
\E_{\hat \P} \left[\|f(\hat Z)-\hat Z\|^p \right]\leq \E_{\hat \P} \left[ \delta(\hat Z)^p \right]\leq r^p
\]
and is thus feasible in~\eqref{eq:OT-LB-a}. Its objective function value in~\eqref{eq:OT-LB-a} satisfies
\[
\E_{\hat\P} \left[\ell(f(\hat Z)) \right] = 
\E_{\hat\P} \left[\ell(\hat Z)\right] +\E_{\hat\P} \left[ V_{\delta(\hat Z)}(\hat Z) \right] .
\]
Hence, any feasible solution in~\eqref{eq:OT-LB-b} gives rise to a feasible solution in~\eqref{eq:OT-LB-a} with the same objective function value. This proves the inequality in~\eqref{eq:OT-LB-a}. Substituting the lower bound~\eqref{eq:g-lower-bound} on $V_\delta(\hat z)$ into~\eqref{eq:OT-LB-b} then yields the estimate
\begin{align}
	\label{eq:lower-bound-wcloss}
	& \sup_{\P\in\cP} \E_\P[\ell(Z)]  \geq \E_{\hat\P} \left[\ell(\hat Z) \right] + \left\{ \begin{array}{cl} \displaystyle \sup_{\delta\in\Delta} & \E_{\hat\P} \left[ \|\nabla\ell(\hat Z)\|_* \delta(\hat Z)-L \delta(\hat Z)^2 \right] \\
		\st & \E_{\hat \P} \left[\delta(\hat Z)^p \right]\leq r^p. \end{array}\right. 
\end{align}
If $\|\nabla\ell(\hat Z)\|_*=0$ $\hat\P$-almost surely, then we have established the desired lower bound. From now on we may thus assume that $\E_{\hat\P}[\|\nabla\ell(\hat Z)\|_*]>0$. Next, we construct a function~$\delta^\star\in\Delta$ feasible in the maximization problem in~\eqref{eq:lower-bound-wcloss} and use its objective function value as a lower bound on the problem's supremum. Specifically, we set
\[
\delta^\star(\hat z) = \frac{\|\nabla \ell(\hat z)\|_*^{q-1} r}{\E_{\hat \P} [ \|\nabla \ell(\hat Z)\|_*^q]^{1/p}} \quad\forall \hat z\in\cZ,
\]
which is well-defined by the integrability condition. As $q-1=q/p$, we find
\[
\E_{\hat \P} \left[ \delta^\star(\hat Z)^p \right] = r^p \quad\text{and} \quad  \E_{\hat \P} \left[ \|\nabla \ell(\hat Z)\|_* \delta^\star (\hat Z)^p \right] = r\cdot \E_{\hat \P} \left[ \|\nabla \ell(\hat Z)\|_*^q\right]^{\frac{1}{q}}.
\]
Hence, $\delta^\star$ is feasible in~\eqref{eq:lower-bound-wcloss}, and its objective function value amounts to
\[
\E_{\hat\P} \left[ \|\nabla\ell(\hat Z)\|_* \delta^\star(\hat Z)-L\delta^\star (\hat Z)^2 \right] = r\cdot \E_{\hat \P} \left[ \|\nabla \ell(\hat Z)\|_*^q\right]^{\frac{1}{q}} - L r^2\cdot \frac{ \E_{\hat \P} [ \|\nabla \ell(\hat Z)\|_*^{2q-2}]}{\E_{\hat \P} [ \|\nabla \ell(\hat Z)\|_*^q]^{2/p}}.
\]
Note that the last term is again finite thanks to the integrability condition. Substituting this expression back into~\eqref{eq:lower-bound-wcloss} yields the desired lower bound
\begin{align*}
	\sup_{\P\in\cP} \E_\P[\ell(Z)]  \geq \E_{\hat\P} \left[\ell(\hat Z) \right] + r \cdot \E_{\hat\P} \left[ \|\nabla\ell(\hat Z)\|_*^q\right]^{\frac{1}{q}} +\cO(r^2).
\end{align*}

\paragraph{Step~2.} By strong duality as established in Theorem~\ref{thm:duality:OT}, we have
\begin{align}
	\nonumber
	\sup_{\P\in\cP} \E_\P[\ell(Z)] & = \inf_{\lambda\in\R_+} \lambda r^p + \E_{\hat \P} \left[ \sup_{z\in\cZ} \ell(z)-\lambda \|z-\hat Z\|^p \right]\\
	& = \inf_{\lambda\in\R_+} \lambda r^p + \E_{\hat \P} \left[ \ell(\hat Z) + \sup_{\delta\in\R_+} V_\delta(\hat Z)-\lambda \delta^p \right],
	\label{eq:upper-bound-taylor}
\end{align}
where the second equality follows from the observation that 
\begin{align*}
	\sup_{z\in\cZ} \ell(z)-\lambda \|z-\hat z\|^p & = \sup_{z\in\cZ} \sup_{\delta\in\R_+}\big\{ \ell(z)-\lambda \delta^p: \|z-\hat z\|\leq\delta\big\}\\
	& = \ell (\hat z) + \sup_{\delta\in\R_+} V_\delta(\hat z)-\lambda \delta^p.
\end{align*}
Next, we construct an upper bound on~\eqref{eq:upper-bound-taylor}. In fact, we need separate constructions for~$p>2$ and~$p\leq 2$. Assume first that~$p>2$. In this case, we have
\begin{subequations}
	\begin{align}
		\nonumber
		& \sup_{\P\in\cP} \E_\P[\ell(Z)] - \E_{\hat \P} \left[ \ell (\hat Z) \right] \\
		\nonumber \leq & \inf_{\lambda_1,\lambda_2\in\R_+} (\lambda_1+\lambda_2) r^p + \E_{\hat \P} \left[ \sup_{\delta\in\R_+} \|\nabla \ell(\hat Z)\|_*\delta+L\delta^2 -(\lambda_1+\lambda_2) \delta^p \right] \\
		\leq & \inf_{\lambda_1 \in\R_+} \lambda_1r^p + \E_{\hat \P} \left[ \sup_{\delta\in\R_+} \|\nabla \ell(\hat Z)\|_*\delta -\lambda_1 \delta^p \right]
		\label{eq:OT-UB-a} \\ 
		& \qquad + \inf_{\lambda_2\in\R_+} \lambda_2 r^p + \sup_{\delta\in\R_+} L\delta^2 -\lambda_2\delta^p,
		\label{eq:OT-UB-b}
	\end{align}
\end{subequations}
where the first inequality follows from the estimate~\eqref{eq:g-upper-bound-plarger2}, and the second inequality holds because the supremum over~$\delta$ is duplicated. The resulting upper bound on the worst-case expected loss thus coincides with the sum of two infima. One readily verifies that the maximization problem over~$\delta$ in~\eqref{eq:OT-UB-a} is solved by $\delta^\star= (p\lambda_1)^{-q/p} \|\nabla \ell(\hat Z)\|_*^{q/p}$. Thus, the infimum in~\eqref{eq:OT-UB-a} equals
\begin{subequations}
	\begin{align}
		\label{eq:OT-UB-a-explicit}
		\inf_{\lambda_1 \in\R_+} \lambda_1r^p + \frac{1}{q} (\lambda_1 p)^{-\frac{q}{p}}  \E_{\hat \P} \left[ \|\nabla \ell(\hat Z)\|_*^q \right] = r \cdot \E_{\hat \P} \left[ \|\nabla \ell(\hat Z)\|_*^q \right]^{\frac{1}{q}},
	\end{align}
	where the equality holds because the resulting minimization problem over~$\lambda_1$ is solved by $\lambda_1^\star = pr^{-p/q}\E_{\hat \P} \left[ \|\nabla \ell(\hat Z)\|_*^q \right]^{1/q}$. Similarly, the maximization problem over~$\delta$ in~\eqref{eq:OT-UB-b} is solved by $\delta^\star=C_1 \lambda_2^{-1/(p-2)}$, where~$C_1$ represents a positive constant that only depends on~$p$ and~$L$. Thus, the infimum in~\eqref{eq:OT-UB-b} equals
	\begin{align}
		\label{eq:OT-UB-b-explicit}
		\inf_{\lambda_2\in\R_+} \lambda_2 r^p + C_2 \lambda_2^{-\frac{2}{p-2}} = C_3 r^2,
	\end{align}
\end{subequations}
where $C_2$ and $C_3$ are other positive constants depending on~$p$ and~$L$. The equality in~\eqref{eq:OT-UB-b-explicit} is obtained by solving the minimization problem over~$\lambda_2$ in closed form. Replacing~\eqref{eq:OT-UB-a} with~\eqref{eq:OT-UB-a-explicit} and~\eqref{eq:OT-UB-b} with~\eqref{eq:OT-UB-b-explicit} finally yields
\begin{align*}
	\sup_{\P\in\cP} \E_\P[\ell(Z)]  \leq \E_{\hat\P} \left[\ell(\hat Z) \right] + r \cdot \E_{\hat\P} \left[ \|\nabla\ell(\hat Z)\|_*^q\right]^{\frac{1}{q}} +\cO(r^2).
\end{align*}

Assume next that $p\leq 2$. In this case, we have
\begin{subequations}
	\begin{align}
		\nonumber
		\sup_{\P\in\cP} \E_\P[\ell(Z)]& \leq \inf_{\lambda_1,\lambda_2\in\R_+} (\lambda_1+\lambda_2) r^p + \E_{\hat \P} \left[ \sup_{\delta\in\R_+} \|\nabla \ell(\hat Z)\|_*\delta+ G\delta^p -(\lambda_1+\lambda_2) \delta^p \right] \\
		& \leq \inf_{\lambda_1 \in\R_+} \lambda_1r^p + \E_{\hat \P} \left[ \sup_{\delta\in\R_+} \|\nabla \ell(\hat Z)\|_*\delta -\lambda_1 \delta^p \right]
		\label{eq:OT-UB-a2} \\ 
		& \qquad + \inf_{\lambda_2\in\R_+} \lambda_2 r^p + \sup_{\delta\in\R_+} G\delta^p -\lambda_2\delta^p,
		\label{eq:OT-UB-b2}
	\end{align}
\end{subequations}
where the first inequality follows from the estimate~\eqref{eq:g-upper-bound-psmaller2}. Note that the infimum in~\eqref{eq:OT-UB-a2} is identical to that in~\eqref{eq:OT-UB-a} and thus simplifies to~\eqref{eq:OT-UB-a-explicit}. Next, note that the maximization problem over~$\delta$ in~\eqref{eq:OT-UB-b2} is unbounded unless $\lambda_2\geq G$. This condition thus constitutes an implicit constraint for the minimization problem over~$\lambda_2$. Whenever $\lambda_2$ satisfies this constraint, however, the supremum over~$\delta$ evaluates to~$0$, and therefore the infimum over~$\lambda_2$ evaluates to $Gr^p$. Replacing~\eqref{eq:OT-UB-a2} with~\eqref{eq:OT-UB-a-explicit} and~\eqref{eq:OT-UB-b2} with~$Gr^p$ finally yields
\begin{align*}
	\sup_{\P\in\cP} \E_\P[\ell(Z)]  \leq \E_{\hat\P} \left[\ell(\hat Z) \right] + r \cdot \E_{\hat\P} \left[ \|\nabla\ell(\hat Z)\|_*^q\right]^{\frac{1}{q}} +\cO(r^p).
\end{align*}
As both $\cO(r^2)$ and~$\cO(r^p)$ for $0<p\leq 2$ are of the order $o(r)$, the claim follows.
\end{proof}

The proof of Theorem~\ref{thm:OT:taylor} reveals that the variation~$V_\delta(\hat z)$ equals~$\|\nabla\ell(\hat z)\|_*\delta$ to first order in~$\delta$. Hence, it is natural to refer to the regularization term $\E_{\hat \P}[ \| \nabla \ell(\hat Z)\|_*^{q}]^{1/q}$ appearing in~\eqref{eq:OT-Taylor} as the {\em total variation}. 

Regularizers penalizing the Lipschitz moduli, gradients, Hessians or tensors of higher-order partial derivatives are successfully used in the adversarial training of neural networks \citep{lyu2015unified, jakubovitz2018improving, finlay2021scaleable, bai2024wasserstein} and in the stabilizing training of generative adversarial networks \citep{roth2017stabilizing, nagarajan2017gradient, gulrajani2017improved}. However, these regularizers introduce nonconvexity into an otherwise convex optimization problem. Theorems~\ref{thm:OT:safe:approximation} and~\ref{thm:OT:taylor} thus suggest that the worst-case expected loss with respect to a Wasserstein ambiguity set provides a convex surrogate for the  empirical loss with Lipschitz and/or variation regularizers.

\subsection{Lipschitz Continuity of Law-Invariant Convex Risk Measures}
\label{sec:Lipschitz-continuous-risk-measures}

Let~$\varrho$ be a law-invariant convex risk measure as introduced in Section~\ref{sec:duality-wc-risk}. Recall that all convex risk measures are translation invariant, monotone and convex. Assume also that~$\varrho$ is an $\cL_p$-risk measure for some~$p\geq 1$. By this we mean that~$\varrho_{\P}[\ell(Z)]$ is finite whenever $\ell\in\cL_p(\P)$ and~$\P\in\cP(\R^d)$, that is, whenever~$\E_{\P}[|\ell(Z)|^p]<+\infty$. The aim of this section is to derive interpretable and easily computable upper bounds on the worst case of $\varrho_\P[\ell(Z)]$ with respect to all distributions~$\P$ of~$Z$ in a $p$-Wasserstein ball. To this end, we first recall the definition of a subgradient.

\begin{definition}[Subgradient]
\label{def:subgradient}
If $\varrho$ is a law-invariant convex $\cL_p$-risk measure for some $p\geq 1$, then $h\in\cL_q(\P)$ is a subgradient of~$\varrho_\P$ at~$\ell_0\in\cL_p(\P)$ if $\frac{1}{p}+\frac{1}{q}=1$ and
\[
\varrho_\P[\ell(Z)] \geq \varrho_\P[ \ell_0(Z)] + \E_\P \left[ h(Z)\cdot \left(\ell(Z)-\ell_0(Z) \right) \right] \quad \forall \ell\in\cL_p(\P).
\]
\end{definition}

We say that~$\varrho_\P$ is subdifferentiable at~$\ell_0$ if it has at least one subgradient at~$\ell_0$. 

\begin{definition}[Lipschitz Continuity]
\label{def:Lipschitz-risk-measure}
Let~$\varrho$ be a law-invariant convex $\cL_p$-risk measure for some~$p\geq 1$. Then, $\varrho$ is Lipschitz continuous if there exists~$L\geq 0$ with
\[
\left| \varrho_\P[\ell(Z)] - \varrho_\P[\ell_0(Z)] \right|  \leq L\cdot \E_\P \left[ \left|\ell(Z)-\ell_0(Z) \right|^p \right]^{\frac{1}{p}} ~~ \forall \ell, \ell_0 \in\cL_p(\P),~\forall \P\in\cP(\R^d).
\]
We use  $\lip(\varrho)$ to denote the Lipschitz modulus, i.e., the smallest $L$ with this property.
\end{definition}

\begin{lemma}[Subgradient Bounds]
\label{lem:subgradient-bound}
Let $\varrho$ be a law-invariant convex $\cL_p$-risk measure and~$h\in\cL_q(\P)$ a subgradient of~$\varrho_\P$ at~$\ell_0\in\cL_p(\P)$ for some~$\P\in\cP(\R^d)$, where $\frac{1}{p}+\frac{1}{q}=1$. If~$\varrho$ is Lipschitz continuous, then $\E_\P[|h(Z)|^q]^{1/q}\leq \lip(\varrho)$.
\end{lemma}
\begin{proof}
By the Lipschitz continuity of $\varrho$ and the definition of subgradients, we have
\begin{align*}
	& \varrho_\P[ \ell_0(Z)] + \lip(\varrho) \cdot \E_\P \left[ \left|\ell(Z)-\ell_0(Z) \right|^p \right]^{\frac{1}{p}} \\ & \qquad \geq  \varrho_\P[ \ell(Z)] 
	\geq \varrho_\P[ \ell_0(Z)] + \E_\P \left[ h(Z)\cdot \left(\ell(Z)-\ell_0(Z) \right) \right]
\end{align*}
for every $\ell\in\cL_p(\P)$. This inequality is equivalent to
\begin{align*}
	\lip(\varrho) \geq \sup_{\substack{\ell\in\cL_p(\P) \\ \ell\neq\ell_0}} \E_\P \left[ h(Z)\cdot \frac{\ell(Z)-\ell_0(Z)}{\E_\P \left[ \left|\ell(Z)-\ell_0(Z) \right|^p \right]^{\frac{1}{p}}} \right] =\E_\P[|h(Z)|^q]^{1/q},
\end{align*}
where the equality holds because the $\cL_q$-norm is dual to the $\cL_p$-norm.
\end{proof}

The results of this section also rely on the fundamentals of comonotonicity theory, which we review next. For any Borel measurable function $f:\R^d\to\R$ the distribution function $F:\R\to[0,1]$ of the random variable~$f(Z)$ under~$\P$ is defined through $F(\tau)=\P(f(Z)\leq \tau)$ for every~$\tau\in\R$, and the corresponding (left) quantile function $F^\leftarrow :[0,1]\to \overline\R$ is defined through $F^\leftarrow (q)=\inf \{\tau\in\R :F_1(\tau)\geq q\}$ for every $q\in[0,1]$. Note that if~$F$ is invertible, then~$F^\leftarrow = F^{-1}$. Note also that~$F$ is generally right-continuous, whereas~$F^\leftarrow$ is generally left-continuous. The definition of the quantile function $F^\leftarrow$ also readily implies the equivalence
\begin{equation}
\label{eq:cdf-versus-quantile-function}
F(\tau)\geq q~~ \iff ~~ \tau \geq F^\leftarrow(q) \quad\forall \tau\in\R,~ \forall q\in[0,1].
\end{equation}

\begin{definition}[Comonotonicity]
\label{def:comonotonicity}
Two random variables $f(Z)$ and $g(Z)$ induced by Borel measurable functions~$f,g:\R^d\to \R$ are comonotonic under~$\P$ if
\[
\P \left(f(Z)\leq\tau_1 \wedge g(Z)\leq \tau_2 \right) = \min\left\{F(\tau_1), G(\tau_2) \right\}\quad \forall \tau_1,\tau_2\in\R,
\]
where~$F$ and~$G$ denote the distribution functions of~$f(Z)$ and~$g(Z)$ under~$\P$.
\end{definition}

The following proposition sheds more light on Definition~\ref{def:comonotonicity}. It shows that comonotonic random variables can essentially always be expressed as functions of each other \cite[Corollary~5.17]{mcneil2015quantitative}.

\begin{proposition}[Comonotonicity]
\label{prop:comonotonicity-alternative}
Let $f(Z)$ and $g(Z)$ be two random variables with respective distribution functions~$F$ and~$G$ under~$\P$ as in Definition~\ref{def:comonotonicity}. If~$F$ is continuous, then $f(Z)$ and $g(Z)$  are comonotonic under~$\P$ if and only if 
\[
g(Z)= G^\leftarrow (F (f(Z)))\quad \P\text{-a.s.}
\] 
\end{proposition}
\begin{proof}
Note first that~$F(f(Z))$ follows the standard uniform distribution on~$[0,1]$ under~$\P$. To see this, note that for any $q\in[0,1]$ we have
\[
\P\left( F(f(Z))\leq q \right) = \P\left( f(Z)\leq F^\leftarrow (q) \right) = F\left( F^\leftarrow(q) \right) =q,
\]
where the first two equalities follow from the definitions of~$F^\leftarrow$ and~$F$, respectively, while the last equality holds because~$F$ is continuous.

Assume now that $f(Z)$ and $g(Z)$ are comonotonic under~$\P$. Hence, we have
\begin{align*}
	\P\left( f(Z)\leq\tau_1 \wedge g(Z)\leq \tau_2 \right) & = \min\left\{F(\tau_1), G(\tau_2) \right\} \\
	& = \P\left( F(f(Z))\leq \min\left\{F(\tau_1), G(\tau_2) \right\} \right) \\
	& = \P\left( F(f(Z))\leq F(\tau_1) \wedge  F(f(Z))\leq G(\tau_2) \right) \\
	& = \P\left( F^\leftarrow(F(f(Z)))\leq \tau_1 \wedge  G^\leftarrow(F(f(Z))) \leq \tau_2 \right)
\end{align*}
for all $\tau_1,\tau_2\in\R$. Here, the second equality holds because~$F(f(Z))$ follows the standard uniform distribution under~$\P$. The last equality holds thanks to~\eqref{eq:cdf-versus-quantile-function}. As $F^\leftarrow(F(f(Z)))$ is $\P$-almost surely equal to~$f(Z)$, we thus have
\begin{align*}
	\P\left( f(Z)\leq\tau_1 \wedge g(Z)\leq \tau_2 \right)  = \P\left( f(Z)\leq \tau_1 \wedge  G^\leftarrow(F(f(Z))) \leq \tau_2 \right).
\end{align*}
for all $\tau_1,\tau_2\in\R$. Hence, $( f(Z), g(Z))$ and $(f(Z), G^\leftarrow(F(f(Z))))$ are equal in law under~$\P$. This implies in particular that the distribution of $g(Z)$ conditional on $f(Z)$ coincides with the distribution of $G^\leftarrow(F(f(Z)))$ conditional on $f(Z)$ under~$\P$. As the latter distribution is given by the Dirac point mass at $G^\leftarrow(F(f(Z)))$, we may conclude that $g(Z)$ is $\P$-almost surely equal to $G^\leftarrow(F(f(Z)))$.

Assume now that $g(Z)= G^\leftarrow (F (f(Z)))$  $\P$-almost surely. Thus, we have
\begin{align*}
	\P\left( f(Z)\leq\tau_1 \wedge g(Z)\leq \tau_2 \right) & = \P\left( f(Z)\leq \tau_1 \wedge  G^\leftarrow(F(f(Z))) \leq \tau_2 \right) \\
	& = \min\left\{F(\tau_1), G(\tau_2) \right\},
\end{align*}
where the second equality follows from the first part of the proof.
\end{proof}


Next, we show that the correlation of two random variables with fixed marginals is maximal if they are comonotonic \citep[Theorem~5.25]{mcneil2015quantitative}.

\begin{theorem}[Attainable Correlations]
\label{thm:attainable-correlations}
Let $f$, $f^\star$, $g$ and $g^\star$ be real-valued Borel measurable functions on~$\R^d$. Assume that, if~$Z$ is governed by~$\P$, then~$f(Z)$ and~$f^\star(Z)$ have the same distribution function~$F$, whereas~$g(Z)$ and~$g^\star(Z)$ have the same distribution function~$G$. If~$f^\star(Z)$ and~$g^\star(Z)$ are comonotonic, then
\[
\E_\P\left[f(Z)\cdot g(Z) \right] \leq \E_\P\left[f^\star(Z)\cdot g^\star(Z) \right].
\]
\end{theorem}

\begin{proof}
Define the joint distribution function~$H:\R^2\to[0,1]$ of~$f(Z)$ and~$g(Z)$ under~$\P$ via $H(\tau_1,\tau_2)=\P(f(Z)\leq\tau_1 \wedge g(Z)\leq\tau_2)$ for all $\tau_1,\tau_2\in \R$. By \cite[Lemma~5.24]{mcneil2015quantitative}, the covariance of~$f(Z)$ and~$g(Z)$ under~$\P$ satisfies
\begin{align}
	\label{eq:covariance-formula}
	\text{cov}_\P(f(Z),g(Z)) & = \int_{-\infty}^{+\infty}\int_{-\infty}^{+\infty} \left( H(\tau_1,\tau_2)- F(\tau_1)\,G(\tau_2) \right) \diff\tau_1 \,\diff\tau_2.
\end{align}
In addition, by the classical Fr\'echet bounds for copulas \cite[Remark~5.8]{mcneil2015quantitative}, we know that  $H(\tau_1,\tau_2)\leq \min\{F(\tau_1), G(\tau_2)\}$ for all~$\tau_1,\tau_2\in\R$. As the marginal distribution functions~$F$ and~$G$ are fixed, it is evident from~\eqref{eq:covariance-formula} that the covariance of the random variables~$f(Z)$ and~$g(Z)$ is maximized if their joint distribution function~$H(\tau_1, \tau_2)$ coincides with its Fr\'echet upper bound. This, however, happens if and only if~$f(Z)$ and~$g(Z)$ are comonotonic under~$\P$. We have thus shown that $\text{cov}_\P(f(Z),g(Z)) \leq \text{cov}_\P(f^\star(Z),g^\star(Z))$, which in turn implies that 
\begin{align*}
	\E_\P\left[ f(Z)\cdot g(Z) \right]  & = \text{cov}_\P(f(Z),g(Z)) + \E_\P\left[ f(Z) \right]\cdot \E_\P\left[g(Z) \right]  \\
	&  \leq \text{cov}_\P(f^\star(Z),g^\star(Z)) + \E_\P\left[ f^\star(Z) \right]\cdot \E_\P\left[g^\star(Z) \right] \\
	& = \E_\P\left[ f^\star(Z)\cdot g^\star(Z) \right].
\end{align*}
Here, the inequality exploits the assumption that~$f(Z)$ equals~$f^\star(Z)$ in law and that~$g(Z)$ equals~$g^\star(Z)$ in law under~$\P$. Hence, the claim follows.
\end{proof}

We are now ready to show that if~$\varrho$ is a Lipschitz continuous $\cL_p$-risk measure and~$\ell$ is a Lipschitz continuous loss function, then the risk $\varrho_\P[\ell(Z)]$ is Lipschitz continuous in the distribution~$\P$ with respect to the $p$-Wasserstein distance.

\begin{theorem}[Lipschitz Continuity of Risk Measures]
\label{thm:lipschitz-risk}
If~$\ell:\R^d\to\R$ is a Lipschitz continuous loss function with respect to some norm~$\|\cdot\|$ on~$\R^d$, $p\geq 1$ and $\varrho$ a Lipschitz continuous and law-invariant convex $\cL_p$-risk measure, then 
\[
\left| \varrho_\P[\ell(Z)] - \varrho_{\hat\P}[\ell(\hat Z)] \right| \leq  \lip(\varrho)\cdot \lip(\ell)\cdot \W_p(\P,\hat\P)
\]
for all $\P,\hat\P\in\cP(\R^d)$. Here, $\W_p$ is defined with respect to~$\|\cdot\|$, and $\frac{1}{p}+\frac{1}{q}=1$.
\end{theorem}

\begin{proof}
Consider an arbitrary~$\P\in\cP(\R^d)$. By \citep[Corollary~3.1]{ref:ruszczynski2006risk},  $\varrho_\P$ is continuous and subdifferentiable on the whole Banach space~$\cL_p(\P)$ equipped with its norm topology.
The Fenchel-Moreau theorem thus implies that
\begin{subequations}
	\begin{equation}
		\label{eq:risk-measure-fenchel}
		\varrho_\P[\ell'(Z)] = \sup_{h'\in\cL_q(\P)} \E_\P[h'(Z)\cdot \ell'(Z)] - \varrho_\P^*[h'(Z)],
	\end{equation}
	for all $\ell'\in\cL_p(\P)$, where
	\begin{equation}
		\label{eq:conjugate-risk-measure}
		\varrho^*_\P[h'(Z)] = \sup_{\ell'\in\cL_p(\P)} \E_\P[h'(Z)\cdot \ell'(Z)] - \varrho_\P[\ell'(Z)]
	\end{equation}
\end{subequations}
for all $h'\in\cL_q(\P)$ \citep[Theorem~5]{rockafellar1974conjugate}. The relation~\eqref{eq:conjugate-risk-measure} defines a law-invariant convex risk measure~$\varrho^*$. Indeed, $\varrho^*$ is convex because pointwise suprema of affine functions are convex. In addition, $\varrho^*$ inherits law-invariance from~$\varrho$. Note that~$h\in\cL_q(\P)$ attains the supremum in~\eqref{eq:risk-measure-fenchel} at $\ell'=\ell$ if and only~if
\begin{align*}
	& \varrho_\P[\ell(Z)] = \E_\P[h(Z)\cdot\ell(Z)] - \varrho^*_\P[h(Z)] \\
	\iff ~& \varrho^*_\P[h(Z)] = \E_\P[h(Z)\cdot\ell(Z)] - \varrho_\P[\ell(Z)]  \\
	\iff ~& \E_\P[h(Z)\cdot\ell'(Z)] - \varrho_\P[\ell'(Z)] \\
	& \hspace{2cm}\leq \E_\P[h(Z)\cdot\ell(Z)] - \varrho_\P[\ell(Z)]  \quad \forall \ell'\in\cL_p(\P),
\end{align*}
where the last equivalence follows from the definition of~$\varrho^*_\P[h(Z)]$ in~\eqref{eq:conjugate-risk-measure}. By rearranging terms, we then find that the last inequality is equivalent to
\begin{align*}
	\varrho_\P[ \ell(Z)] + \E_\P \left[ h(Z)\cdot(\ell'(Z)-\ell(Z))\right] \leq \varrho_\P[\ell'(Z)]  \quad \forall \ell'\in\cL_p(\P).
\end{align*}
Thus, $h$ attains the supremum in~\eqref{eq:risk-measure-fenchel} at~$\ell$ if and only if it represents a subgradient of~$\varrho_\P$ at~$\ell$. As~$\varrho_\P$ is subdifferentiable throughout~$\cL_p(\P)$, the above reasoning implies that the supremum in~\eqref{eq:risk-measure-fenchel} is always attained.

Select now any $\P,\hat\P\in\cP(\R^d)$ with~$\W_p(\P,\hat\P)<+\infty$. We assume temporarily that~$\P$ and~$\hat\P$ are non-atomic, that is, $\P(Z=z) =\hat\P(Z=z)=0$ for all~$z\in\R^d$. Thus, for any admissible distribution function~$F$ there exists a Borel measurable function~$f:\R^d\to\R$ such that~$\P(f(Z)\leq\tau)=F(\tau)$ for all~$\tau\in\R$; see, {\em e.g.}, \cite[Lemma~1]{delage2019diceisions}. Note that non-atomicity will later be relaxed.
Select now also any~$h\in\cL_q(\P)$ that attains the supremum in~\eqref{eq:risk-measure-fenchel} at~$\ell'=\ell$, which is guaranteed to exist. The representation~\eqref{eq:risk-measure-fenchel} then implies that
\begin{align*}
	& \hspace{-1mm} \varrho_\P[\ell(Z)] - \varrho_{\hat\P}[\ell(\hat Z)] \\
	& = \E_\P\left[ h(Z)\cdot\ell(Z)\right] -\varrho^*_\P[h(Z)] -  \sup_{\hat h\in\cL_q(\hat \P)} \left\{ \E_{\hat\P} \left[ \hat h(\hat Z)\cdot\ell(\hat Z)\right] - \varrho^*_{\hat\P}[\hat h( \hat Z)] \right\}. 
\end{align*}
In the following, we use~$F$ to denote the distribution function of~$h(Z)$ under~$\P$ and~$\hat F$ to denote the distribution function of~$\ell(\hat Z)$ under~$\hat \P$. In addition, we restrict the above maximization problem to functions~$\hat h$ for which the distribution function of the random variable $\hat h(\hat Z)$ coincides with~$F$. As restricting the feasible set of a maximization problem leads to a lower bound on its optimal value, we find
\begin{align}
	& \varrho_\P[\ell(Z)] - \varrho_{\hat\P}[\ell(\hat Z)]  \nonumber \\
	& \hspace{5mm}   \leq \E_\P\left[ h(Z)\cdot\ell(Z)\right]-  \left\{ \begin{array}{c@{~~\,}l} \displaystyle \sup_{\hat h\in\cL_q(\hat \P)} &
		\E_{\hat\P} \left[ \hat h(\hat Z)\cdot\ell(\hat Z)\right]  \\ \text{s.t.} & \hat\P\left(\hat h (\hat Z)\leq \tau\right) = F(\tau)\quad \forall\tau \in\R.
	\end{array} \right.
	\label{eq:risk-bound1}
\end{align}
Here, we have exploited the law-invariance of the risk measure~$\varrho^*$, which implies that $\varrho^*_\P[h(Z)]$ and $\varrho^*_{\hat\P}[\hat h( \hat Z)]$ match. Next, define the function~$\hat h^\star: \R^d\to\R$ through 
\[
\hat h^\star(\hat z)=F^\leftarrow(\hat F(\ell(\hat z))) \quad \forall \hat z\in\R^d.
\] 
Note that~$\hat F$ is continuous because $\hat\P$ is non-atomic and $\ell$ is (Lipschitz) continuous. By Proposition~\ref{prop:comonotonicity-alternative}, the random variables $\hat h^\star(\hat Z)$ and $\ell(\hat Z)$ are thus comonotonic  and have distribution functions $F$ and $\hat F$ under~$\hat \P$, respectively. Hence, $\hat h^\star$ is feasible in the maximization problem in~\eqref{eq:risk-bound1}. In addition, by Theorem~\ref{thm:attainable-correlations}, $\hat h^\star$ is optimal.

Next, select any transportation plan $\gamma\in\Gamma(\P,\hat\P)$. As the marginal distributions of~$Z$ and~$\hat Z$ under~$\gamma$ are given by~$\P$ and~$\hat \P$, respectively, the above implies that
\begin{align}
	& \varrho_\P[\ell(Z)] - \varrho_{\hat\P}[\ell(\hat Z)] \nonumber \\
	& \hspace{5mm}   \leq \E_\gamma \left[ h(Z)\cdot\ell(Z)\right]-  \left\{ \begin{array}{c@{~~\,}l} \displaystyle \sup_{\hat h\in\cL_q(\gamma)} &
		\E_{\gamma} \left[ \hat h(Z, \hat Z)\cdot\ell(\hat Z)\right]  \\ \text{s.t.} & \gamma\left(\hat h (Z, \hat Z)\leq \tau\right) = F(\tau)\quad \forall\tau \in\R.
	\end{array} \right.
	\label{eq:risk-bound2}
\end{align}
Note that we have relaxed the maximization problem in~\eqref{eq:risk-bound2} by allowing the function~$\hat h$ to depend both on~$Z$ and~$\hat Z$. However, this extra flexibility does not result in a higher optimal value. Indeed, Theorem~\ref{thm:attainable-correlations} ensures that the supremum is attained by any function~$\hat h$ for which the random variables $\hat h(Z,\hat Z)$ and $\ell(\hat Z)$ are comonotonic and for which $\hat h(Z,\hat Z)$ has distribution function~$F$. As we have seen before, there exists a function with these properties that does not depend on~$Z$. Hence, the right to select a function~$\hat h$ that depends on~$Z$ is worthless.

Observe now that the function $\hat h(Z,\hat Z)=h(Z)$ is feasible in~\eqref{eq:risk-bound2}. Thus, we~find
\begin{align*}
	\varrho_\P[\ell(Z)] - \varrho_{\hat\P}[\ell(\hat Z)] &\leq \E_\gamma \left[ h(Z)\cdot\ell(Z)\right]- \E_{\gamma} \left[ h(Z)\cdot\ell(\hat Z)\right] \\
	& \leq \E_\gamma \left[ h(Z)\cdot \left| \ell(Z)- \ell(\hat Z) \right| \right] \\
	& \leq \E_\gamma \left[ h(Z)\cdot \lip(\ell) \cdot \|Z - \hat Z\| \right] \\
	& \leq \lip(\ell) \cdot \E_\gamma \left[ \|Z - \hat Z\|^p\right]^{\frac{1}{p}} \cdot \E_\P\left[h(Z)^q \right]^{\frac{1}{q}}
\end{align*}
where the second inequality holds because all convex risk measures are monotonic, which implies that the subgradient $h(Z)$ is $\P$-almost surely non-negative. The third inequality exploits the Lipschitz continuity of the loss function, and the fourth inequality follows from H\"older's inequality. As the resulting inequality holds for all couplings $\gamma\in\Gamma(\P,\hat\P)$, the definition of the $p$-Wasserstein distance implies that
\begin{align*}
	\varrho_\P[\ell(Z)] - \varrho_{\hat\P}[\ell(\hat Z)] & \leq \lip(\ell) \cdot \W_p(\P,\hat\P) \cdot  \E_\P\left[h(Z)^q \right]^{\frac{1}{q}} \\&  \leq \lip(\varrho)\cdot \lip(\ell) \cdot \W_p(\P,\hat\P),
\end{align*}
where the second inequality follows from Lemma~\ref{lem:subgradient-bound}. The claim then follows by interchanging the roles of~$\P$ and~$\hat\P$.

Recall now that we assumed~$\P$ and~$\hat \P$ are non-atomic. This assumption was needed to show that the supremum in~\eqref{eq:risk-bound1} is attained. In general, one can extend~$\P$ to a distribution~$\P'$ on~$\R^{d+1}$ under which~$(Z_1,\ldots, Z_d)$ and~$Z_{d+1}$ are independent and have marginal distributions equal to~$\P$ and to the uniform distribution on~$[0,1]$, respectively. In the same way, $\hat\P$ can be extended to a distribution~$\hat\P'$ on~$\R^{d+1}$. By construction, $\P'$ and~$\hat\P'$ are non-atomic. As~$\varrho$ is law-invariant, we further have
\[
\left|\varrho_\P[\ell(Z)] - \varrho_{\hat\P}[\ell(\hat Z)]\right| = \left|\varrho_{\P'}[\ell(Z)] - \varrho_{\hat\P'}[\ell(\hat Z)]\right|.
\]
The right hand side of this equation can now be bounded as above. 
\end{proof}

Theorem~\ref{thm:lipschitz-risk} immediately implies the following worst-case risk bound.

\begin{corollary}
\label{cor:worst-case-risk-lipschitz}
If all assumptions of Theorem~\ref{thm:lipschitz-risk} hold and $\cP=\{\P\in\cP(\R^d): \W_p(\P, \hat\P)\leq r\}$ is a $p$-Wasserstein ball of radius~$r\geq 0$ for any~$p\geq 1$, then
\[
\sup_{\P\in \cP} \varrho_\P[\ell(Z)] \leq  \varrho_{\hat \P} [\ell(Z)]+ r\cdot \lip(\varrho)\cdot \lip(\ell).
\]

\end{corollary}

Theorem~\ref{thm:lipschitz-risk} and Corollary~\ref{cor:worst-case-risk-lipschitz} are due to \citet{pichler:2013}. Corollary~\ref{cor:worst-case-risk-lipschitz} shows that the worst-case risk over all distributions in a $p$-Wasserstein ball is upper bounded by the sum of the nominal risk and a Lipschitz regularization term for a broad spectrum of law-invariant convex risk measures. If the loss function~$\ell$ is linear, that is, if $\ell(z)=\theta^\top z$ for some~$\theta\in\R^d$, then this upper bound is often tight \citep{pflug2012, wozabal2014robustifying}. In this case the Lipschitz modulus of~$\ell$ simplifies to $\|\theta\|_*$. For example, the CVaR at level~$\beta\in(0,1]$ is a law-invariant convex $\cL_p$-risk measure, and it is Lipschitz continuous with Lipschitz modulus~$\beta^{-1/p}$. Thus, Corollary~\ref{cor:worst-case-risk-lipschitz} applies. From Proposition~\ref{prop:worst-case-p-wasserstein-risk} we know, however, that the upper bound is exact in this case. If additionally~$p=1$, then Proposition~\ref{prop:worst-case-1-wasserstein-risk} implies that the upper bound remains exact whenever~$\ell$ is convex and Lipschitz continuous. 

\section{Numerical Solution Methods for DRO Problems}
\label{sec:numerical-solution-methods}

The finite convex reformulations of the worst-case expectation problem~\eqref{eq:worst-case:expectation} presented in Section~\ref{sec:finite-convex-reformulations} are often susceptible to standard optimization software, that is, they obviate the need for tailored algorithms. However, these reformulations can have two significant drawbacks. First, the corresponding monolithic optimization problems can become large and hence challenging to solve. Second, depending on the chosen ambiguity set, the emerging reformulations may belong to a class of optimization problems that are more difficult to solve than a deterministic version of the original problem. For instance, even if the loss function $\ell$ in the worst-case expectation~\eqref{eq:worst-case:expectation} is piecewise affine and the support set $\mathcal{Z}$ is an ellipsoid, the finite dual reformulation over Chebyshev ambiguity sets, as provided by Theorem~\ref{thm:finite:convex:Chebyshev:I}, results in a semidefinite program, as opposed to a numerically favorable quadratically constrained quadratic program. Both disadvantages can be alleviated by resorting to tailored algorithms, which we discuss in this section.

Most numerical methods for solving the DRO problem~\eqref{eq:primal:dro} address an equivalent reformulation of~\eqref{eq:primal:dro} obtained by dualizing the inner worst-case expectation problem. This reformulation is usually constructed by leveraging one of the strong duality theorems from Section~\ref{sec:duality-wc-expectation}. The resulting reformulation of~\eqref{eq:primal:dro} is thus representable as a semi-infinite program of the form
\begin{align}
\label{eq:robust}
\inf \left\{ f(y) : y \in \cY, ~ g_j(y, z_j) \leq 0 ~~ \forall z_j \in \cZ, \, j \in [m] \right\}.
\end{align}
Note that~\eqref{eq:robust} is naturally interpreted as a classical robust optimization problem.
%

As an example, assume that~$\cP$ is the generic moment ambiguity set~\eqref{eq:moment-ambiguity-set} and that some mild regularity conditions hold. In this case, Theorem~\ref{thm:duality:moment} implies that
\begin{align*}
\inf_{x \in \cX} \, \sup_{\P \in \cP} \, \E_\P[\ell(x, Z)]
=
\left\{ \begin{array}{cl}
	\inf & \lambda_0 + \delta_\cF^*(\lambda) \\[1ex]
	\st & x \in \cX, \lambda_0 \in \R, \, \lambda \in \R^m \\ [1ex]
	& \lambda_0 + f(z)^\top \lambda \geq \ell(x, z) \quad \forall z \in \cZ.
\end{array} \right.
\end{align*}
If the support function $\delta_\cF^*(\lambda)$ is known in closed form, then the resulting minimization problem becomes an instance of~\eqref{eq:robust} with $y=(x,\lambda_0,\lambda)$, $\cY=\cX\times\R\times \R^m$, $f(y)=\lambda_0+\delta^*_\cF(\lambda)$, $m=1$ and $g_1(y,z_1) = \lambda_0+ f(z_1)^\top \lambda - \ell(x,z_1)$. Alternatively, $\delta_\cF^*(\lambda)$ can be recast as the optimal value of a dual minimization problem, and the underlying decision variables can be appended to~$y$. As another example, if $\cP$ is the $\phi$-divergence ambiguity set~\eqref{eq:phi-divergence-ambiguity-set} centered at a discrete distribution~$\hat \P = \sum_{i \in [N]} \hat p_i \delta_{\hat z_i}$ and if mild regularity conditions hold, then Theorem~\ref{thm:duality:phi} implies that
\begin{align*}
\inf_{x \in \cX} \sup_{\P \in \cP} \; \E_\P[\ell(x,Z)] 
= \left\{ \begin{array}{cl} \inf & \displaystyle \lambda_0  + \lambda r + \sum_{i \in [N]} \hat p_i \cdot (\phi^*)^\pi \left( \ell(\hat z_i) - \lambda_0, \lambda \right)  \\[3ex]
	\st & x \in \cX, \, \lambda_0 \in \R, \, \lambda \in \R_+ \\[1ex]
	& \displaystyle  \lambda_0 + \lambda\,\phi^\infty(1) \geq \ell(x,z) \quad \forall z \in \cZ.
\end{array}\right.
\end{align*}
This minimization problem is readily recognized as an instance of~\eqref{eq:robust}. Note also that if~$\cP$ is the {\em restricted} $\phi$-divergence ambiguity set~\eqref{eq:phi-divergence-ambiguity-set} and~$\hat\P$ is discrete, then, under mild regularity conditions, Theorem~\ref{thm:duality:restricted:phi} implies that the above reformulation remains valid provided that~$\cZ$ is replaced with $\{\hat z_i:i\in[N]\}$. Finally, when $\cP$ is the optimal transport ambiguity set~\eqref{eq:OT-ambiguity-set} centered at a discrete reference distribution and if mild regularity conditions hold, then Theorem~\ref{thm:duality:OT} implies that
\begin{align*}
\inf_{x \in \cX} \sup_{\P \in \cP} \; \E_\P[\ell(x, Z)] 
= \left\{ 
\begin{array}{cl}
	\inf & \displaystyle \lambda r + \sum_{i \in [N]} \hat p_i s_i \\[1ex]
	\st & x \in \cX, \, \lambda \in \R_+, \, s \in \R^N \\[1ex]
	& \ell(x,z) - \lambda c(z,\hat z_i) \leq s_i \quad \forall z_i \in \cZ, \, i \in [N]. 
\end{array} \right.
\end{align*}
This minimization problem is again an instance of~\eqref{eq:robust}. 

In the remainder of this section we discuss various numerical methods for solving the semi-infinite program~\eqref{eq:robust}.
%
%
%
Some of these methods solve one or several relaxations of~\eqref{eq:robust} that enforce the uncertainty-affected constraint only for a finite subset~$\tilde\cZ$ of~$\cZ$. Hence, these methods assume access to a scenario oracle.

\begin{definition}[Scenario Oracle]
\label{def:scenario-oracle}
Given any finite scenario set $\tilde \cZ \subseteq \cZ$, a scenario oracle outputs a solution to the scenario problem
\begin{align}
	\label{eq:scenario}
	\begin{array}{cl}
		\inf \left\{ f(y) : y \in \cY, ~ g_j(y, z_j) \leq 0 ~~ \forall z_j \in \tilde \cZ, \forall j \in [m] \right\}.
	\end{array}
\end{align}
\end{definition}


As we will see below, cutting plane algorithms refine scenario relaxations of the semi-infinite program~\eqref{eq:robust} by iteratively adding those parameter realizations $z \in \cZ \setminus \tilde \cZ$ for which the constraint violation is maximal. Identifying such realizations requires a noise oracle as per the following definition.

\begin{definition}[Noise Oracle]
\label{def:noise-oracle}
Given any fixed decision $\tilde y \in \cY$, a noise oracle outputs a solution to the noise problem 
\begin{align}
	\label{eq:noise}
	\sup_{z \in \cZ} \max_{j \in [m]} g_j(\tilde y, z).
\end{align}
\end{definition}



In the following, we first survey the scenario approach, which replaces the semi-infinite program~\eqref{eq:robust} with a finite scenario problem that offers stochastic approximation guarantees. This approach calls the scenario oracle only {\em once}. We then review cutting plane techniques that iteratively call scenario and noise oracles to generate a solution sequence that attains the optimal value of problem~\eqref{eq:robust}, either within finitely many iterations or asymptotically. Next, we study online convex optimization algorithms, which do not require expensive scenario and/or noise oracles and instead solve only {\em deterministic} versions of problem~\eqref{eq:robust} and use cheap first-order updates of the candidate decisions and/or incumbent worst-case parameter realizations. We close with an overview of specialized numerical solution methods that are tailored to specific ambiguity sets.

\subsection{The Scenario Approach}
\label{sec:scenario}

The scenario approach was pioneered by \citet{de2004constraint} in the context of robust Markov decision processes and by \citet{calafiore2005uncertain,calafiore2006scenario} and \citet{campi2008exact,campi2011sampling} in the context of generic robust optimization problems of the form~\eqref{eq:robust}. The scenario approach replaces the semi-infinite constraint in~\eqref{eq:robust} with a collection of finitely many constraints corresponding to uncertainty realizations sampled from some fixed distribution $\Q \in \cP(\cZ)$. \\[1.5ex]
\textbf{Algorithm 1: Scenario Approach} 
\begin{enumerate}
\item[1.] \textbf{Initialization.} Fix a distribution $\Q \in \cP(\cZ)$. 
\item [2.] \textbf{Sampling.} Draw $N$ independent samples $Z_1, \dots, Z_N$ from $\Q$ and construct the scenario set $\tilde \cZ = \{ Z_1, \dots, Z_N \}$. \\[-2.25ex]
\item[2.] \textbf{Termination.} Return the output $\tilde Y$ of the scenario oracle~\eqref{eq:scenario} with input~$\tilde\cZ$.
\end{enumerate}



Note that, as the input to the scenario oracle~\eqref{eq:scenario} is a random scenario set governed by the $N$-fold product distribution~$\Q^N$, its output~$\tilde Y$ is also random. Fix now a constraint violation probability $\delta \in (0, 1)$, a significance level $\eta \in (0, 1)$, and ensure that the sample size $N$ in Step~2 of Algorithm~1 satisfies $N \geq N(d_y, \delta, \eta)$, where $d_y$ is the dimension of the decision vector $y$ and
\begin{align*}
N(d_y, \delta, \eta) = \min \left\{ N \in \N: \sum_{i = 0}^{d_y-1} \begin{pmatrix} N \\ i \end{pmatrix} \delta^i (1 - \delta)^{N - i} \leq \eta \right\}.
\end{align*}
Assuming that the objective and constraint functions of problem~\ref{eq:robust} are convex in~$y$ for any fixed~$z_j$, $j\in[m]$, and that the optimal solution to~\eqref{eq:scenario} exists and is unique for any fixed scenario set~$\tilde \cZ$,
Algorithm~1 then guarantees that
\begin{align*}
\Q^N \left( \Q \left(  g_j(\tilde Y,Z) \leq 0 \; \forall j \in [m] \right) \geq 1 - \delta \right) \geq 1 - \eta,
\end{align*}
where $Z$ follows~$\Q$ and $\tilde Y$ is governed by~$\Q^N$; see \cite[Theorem~1]{campi2008exact}. In other words, the output $\tilde Y$ of the scenario oracle~\eqref{eq:scenario} constitutes a feasible solution of the chance constrained program
\begin{align*}
\inf \left\{ f(y) : y \in \cY, ~ \Q \left( g_j(y, Z) \leq 0 \; \forall j \in [m] \right) \geq 1 - \delta \right\}.
\end{align*}
with probability at least $1 - \eta$, where $\eta$ can be interpreted as the (small) chance of poorly approximating $\Q$ in Step~2 of Algorithm~1. We emphasize that the convexity of~\eqref{eq:robust} plays a crucial role in the derivation of this probabilistic guarantee.

Two remarks on the scenario approach are in order. First, its performance guarantee is stochastic as it relates to a chance constrained program that relaxes the semi-infinite program~\eqref{eq:robust}. Second, the sample size $N(d_y, \delta, \eta)$ needed for a probabilistic guarantee is of the order $\tilde\cO((d_y + \log(1/\eta)) / \delta)$, that is, it grows linearly with the dimension $d_y$ of the decision vector $y$. This dependence limits the problem dimensions that can be handled in practice. Robust performance guarantees for the scenario approach have been studied by \citet{MESL15:performance-bounds}. The dependence of the probabilistic performance guarantees on the dimension of the decision vector~$y$ can be improved by using regularization \citep{campi2013random}, one-off calibration schemes \citep{care2014fast} and sequential validation \citep{calafiore2011research} or by exploiting limited support ranks  \citep{schildbach2013randomized} and solution-dependent numbers of support constraints \citep{campi2018wait}. In general, however, the number of sampled constraints may remain large, which can be an impediment to the adoption of the scenario approach in large-scale problems.

\subsection{Cutting Plane Algorithms}
\label{sec:cutting:plane}

\citet{mutapcic2009cutting} propose an iterative method for solving the semi-infinite program~\eqref{eq:robust}, which is based on Kelley's cutting-plane algorithm \citep{kelley1960cutting}. Their method can be described as follows. \\[1.5ex]
\textbf{Algorithm 2: Cutting-Plane Algorithm}
\begin{enumerate}
\item[1.] \textbf{Initialization.} Select a non-empty finite scenario set $\tilde \cZ \subseteq \cZ$, and set the feasibility threshold parameter $\varepsilon$ to a small value.
\item[2.] \textbf{Master Problem.} Solve the scenario oracle problem~\eqref{eq:scenario} to find $\tilde y$.
\item[3.] \textbf{Sub-Problem.} Solve the noise oracle problem~\eqref{eq:noise} to find $\tilde z$.
\item[4.] \textbf{Termination.} If $g_j(\tilde y, \tilde z) \leq \varepsilon$ for all $j \in [m]$, terminate with $\tilde y$ as a $\varepsilon$-feasible solution to~\eqref{eq:robust}. Otherwise, update $\tilde \cZ \gets \tilde \cZ \cup \{\tilde z\}$ and return to Step~2.
\end{enumerate}

Algorithm~2 alternates between two steps. Step~2 solves the scenario oracle problem~\eqref{eq:scenario} for a finite scenario set~$\tilde \cZ$ and outputs a candidate solution $\tilde y$. Step~3 then solves the noise oracle problem~\eqref{eq:noise} for the given candidate solution~$\tilde y$ and outputs a most violated scenario~$\tilde z$. If the optimal value of~\eqref{eq:noise} exceeds $\varepsilon$, then the scenario $\tilde z$ is added to the scenario set $\tilde \cZ$ and the process is repeated. Otherwise, $\tilde{y}$ from Step~2 is returned as an $\varepsilon$-feasible solution to the semi-infinite program~\eqref{eq:robust}, that is, a solution $\tilde{y} \in \cY$ that satisfies $g_j (\tilde{y}, z_j) \leq \varepsilon$ for all $z_j \in \cZ$ and $j \in [m]$.



Cutting plane algorithms replace the sampling procedure of the scenario approach with a noise oracle, but they still require access to a scenario oracle that solves the master problems. In contrast to the scenario approach, however, the number of constraints in the master problem increases with each iteration, which can limit scalability. If the constraint functions $g_j(y,z_j)$ are Lipschitz continuous in~$y$, then Algorithm~2 terminates after $\cO(\delta^{-d_y})$ iterations, which however is exponential in the dimension of $y$ \citep[\textsection~5.2]{mutapcic2009cutting}. Despite this, in practice, cutting plane algorithms often converge to near-optimal solutions in very few iterations, which has contributed to their widespread use in robust optimization.

\subsection{Online Convex Optimization Algorithms}
\label{sec:online}

Cutting plane algorithms can become computationally expensive due to their reliance on scenario and noise oracles for the solution of the master and sub-problems, respectively. In the following, we review how ideas from online convex optimization can help to alleviate these shortcomings~\citep{shalev2012online,hazan2022introduction}.

In their seminal work on this topic, \citet{ben2015oracle} employ a bisection search to reduce the solution of problem~\eqref{eq:robust} to the solution of a sequence of robust feasibility problems of the form
\begin{align}
\label{eq:feasibility}
\inf \left\{ 0: y \in \cY, ~ f(y) \leq c, ~g_j(y, z_j) \leq 0 ~~ \forall z_j \in \cZ, ~ \forall j \in [m] \right\}.
\end{align}
More precisely, the following bisection algorithm can be used to solve~\eqref{eq:robust}.
\\[1.5ex]
\textbf{Algorithm 3: Bisection Algorithm}
\begin{enumerate}
\item[1.] \textbf{Initialization.}
Find an interval $[a, b]$ that contains the optimal value of~\eqref{eq:robust}, and select an arbitrary feasible solution~$\tilde y$.
\item[2.] \textbf{Decision Problem.} Set $c = (a+b)/2$, and check if~\eqref{eq:feasibility} is feasible or not.
\item[3.] \textbf{Update.} If~\eqref{eq:feasibility} is feasible, update~$\tilde y$ to a solution of the feasibility problem~\eqref{eq:feasibility}, and set $b \gets c$; otherwise, set $a \gets c$.
\item[4.] \textbf{Termination.} If $b - a \leq \delta$, terminate with $\tilde y$ as an approximately optimal solution to~\eqref{eq:robust}. Otherwise, return to Step~2.
\end{enumerate}

\citet{ben2015oracle} use techniques from online convex optimization to solve the robust feasibility problem~\eqref{eq:feasibility}. In particular, they develop a method similar to Algorithm~2 that approximately solves a nominal feasibility problem instead of calling the scenario oracle and that uses a first-order update rule instead of calling the noise oracle. Accordingly, they require the constraint functions $g_j$, $j\in[m]$, to be differentiable. Their algorithm can be summarized as follows.
\\[1.5ex]
\textbf{Algorithm 4: Dual-Subgradient Meta Algorithm}
\begin{enumerate}
\item[1.] \textbf{Initialization:} Choose initial uncertainty realizations $z_j \in \cZ$, $j \in [m]$.
\item[2.] \textbf{Nominal Problem:} Find $\tilde y$ that solves the \emph{approximate} feasibility problem
\begin{equation*}
	\inf \left\{ 0: f(y) \leq c, ~g_j(y, z_j) \leq \varepsilon ~~ \forall j \in [m] \right\}
\end{equation*}
corresponding to the current uncertainty realizations and corresponding to some $\varepsilon>0$. If no such $\tilde y$ exists, terminate and report that~\eqref{eq:feasibility} is infeasible.
\item[3.] \textbf{Update Parameters:} Update $z_j$, $j \in [m]$, using the gradient rule
\begin{align*}
	z_j \gets \Proj_{\cZ}[z_j + \eta \nabla_z g_j(\tilde y, z_j)] ~~ \forall j \in [m],
\end{align*}
where $\eta$ is a given stepsize and $\Proj_{\cZ}$ denotes the Euclidean projection onto~$\cZ$.
\item[4.] \textbf{Termination:} Once a termination condition is met, return the average of all candidate solutions $\tilde y$ found in Step~2.
\end{enumerate}

Algorithms~3 and~4 can be combined to a single algorithm that finds a $\delta$-optimal and $\varepsilon$-feasible solution to the semi-infinite program~\eqref{eq:robust} in  $\cO(\varepsilon^{-2} \log(1/\delta))$ iterations. This convergence guarantee holds under the following assumptions. The feasible sets~$\cY$ and~$\cZ$ are closed and convex, the objective function $f : \cY \rightarrow \mathbb{R}$ is convex and Lipschitz continuous, and the constraint functions $g_j : \cY \times \cZ \rightarrow \mathbb{R}$, $j\in[m]$, constitute saddle functions. Specifically, $g(y,z_j)$ is convex and Lipschitz continuous in~$y$ as well as concave and upper semicontinuous in~$z_j$ for every $j\in[m]$. We refer to \citep{ben2015oracle} for further implementation details.

Algorithm~4 still solves multiple nominal feasibility problems in Step~2, which can be expensive. As a remedy, \citet{ho2018online} reduce the solution of the feasibility problem~\eqref{eq:feasibility} to the verification of the inequality
\begin{align}
\label{eq:min:max:max}
\min_{y \in \cY_c} \max_{z \in \cZ} \max_{j \in [m]} g_j(y, z) \leq \varepsilon
\end{align}
for a given tolerance $\varepsilon>0$, where $\cY_c =\{ y \in \cY: f(y) \leq c \}$. Checking~\eqref{eq:min:max:max} requires the solution of a saddle point problem. Note that the objective function of this saddle point problem is convex in~$y$ but but fails to be concave in~$z$ when~$m > 1$. Therefore, standard primal-dual algorithms from online convex optimization do not apply. Nevertheless, \citet{ho2018online} construct an online algorithm that outputs a trajectory of candidate solutions~$\tilde y$ and uncertainty realizations~$\tilde z$ that converge to a saddle point. This method uses a first-order algorithm~$\cA_y$ for solving the (parametric) minimization problem $\min_{y\in\cY_c} \max_{j\in[m]} g_j(y,z)$ as well as a first-order algorithm $\cA_j$ for solving the (parametric) maximization problem $\max_{z\in \cZ} g_j(y,z)$ for each $j\in[m]$ as subroutines. Specifically, $\cA_y$ is assumed to map any history of candidate solutions $\tilde y^1, \dots, \tilde y^t$ and uncertainty realizations $z^1_j, \ldots z^t_j \in \cZ$ for $j \in [m]$ and $t\in\N$ to a new candidate solution~$\tilde y^{t+1}$ such that
\begin{align*}
\sum_{t \in [T]} \max_{j \in [m]} g_j(\tilde y^t, \tilde z^t_j) - \min_{y \in \cY_c} \sum_{t \in [T]} \max_{j \in [m]} g_j(y, \tilde z^t_j) \leq \cR_y(T)\quad\forall T\in\N,
\end{align*}
where $\cR_y(T)$ is a sublinear regret bound. Similarly, it is assume that $\cA_j$ maps any history of candidate solutions and uncertainty realizations of length $t\in\N$ to a new uncertainty realization~$\tilde z_j^{t+1}$ such that
\begin{align*}
\max_{z \in \cZ} \sum_{t \in [T]} g_j(\tilde y^t, z) - \sum_{t \in [T]} g_j(\tilde y^t, \tilde z^t_j) \leq \cR_j(T) \quad \forall T\in\N,
\end{align*}
where $\cR_j(T)$ is a sublinear regret bound for every $j \in [m]$.
The algorithm by \citet{ho2018online} can now be summarized as follows.
\\[1.5ex]
\textbf{Algorithm 5: Online Learning Framework}
\begin{enumerate}
\item[1.] \textbf{Initialization:} Initialize the solution history to $\cH \gets \emptyset$.
\item[2.] \textbf{Find Candidate Solution:} Use algorithm $\cA_y$ with input $\cH$ to construct a new candidate solution, that is, set $\tilde y \gets \cA_y(\cH)$.
\item[3.] \textbf{Find Uncertainty Realizations:} Use algorithm $\cA_j$ with input $\cH$ to construct a new uncertainty realization, that is, set $\tilde z_j \gets \cA_j(\cH)$ for all $j \in [m]$.
\item[4.] \textbf{Update History}: Update the solution history to $\cH \gets \cH \cup \{ ( \tilde y, (\tilde z_j)_{j} ) \}.$
\item[4.] \textbf{Termination:} Once a termination condition is met, return the average of all candidate solutions $\tilde y$ found in Step~2.
\end{enumerate}
\citet{ho2018online} show that Algorithm~5 solves the saddle point problem on the left hand side of~\eqref{eq:min:max:max} approximately with regret guarantee
\[
\sum_{t\in[T]} \max_{z \in \cZ} \max_{j \in [m]} g_j(\tilde y^t, z) - \min_{y \in \cY_c} \max_{j \in [m]} g_j(y, \tilde z^t) \leq \cR_y(T) + \max_{j\in[m]} \cR_j(T)\quad \forall T\in\N.
\]
The total regret bound in the above expression grows sublinearly with~$T$. Under the usual convexity assumptions, Algorithms~3 and~5 can be combined to a joint algorithm that finds a $\delta$-optimal and $\varepsilon$-feasible solution to the semi-infinite program~\eqref{eq:robust} in  $\cO(\varepsilon^{-2} \log(1/\delta))$ iterations. Thus, the iteration complexity did not improve {\em vis-\`a-vis} the algorithm by \citet{ben2015oracle}. However, the computational effort per iteration is significantly lower for Algorithm~5 than for Algorithm~4. Indeed, Algorithm~4 solves a feasibility problem with~$m$ uncertainty realizations in each iteration, whereas Algorithm~5 only calls the algorithms~$\cA_y$ and $\cA_j$, $j\in[m]$, which compute cheap first-order updates. For further details, we refer to \cite{ho2018online}. In addition, \citet{ho2019exploiting} exploit structural properties of the objective and constraint functions to reduce the overall iteration complexity to $\cO(\varepsilon^{-1} \log(1/\delta))$. 

Up until now, all the algorithms discussed in this section relied on the bisection method to reduce the semi-infinite program~\eqref{eq:robust} to a sequence of robust feasibility problems~\eqref{eq:feasibility}. This introduces unnecessary computational overhead. As a remedy, \citet{postek2024first} use primal-dual saddle point algorithms that address the following perspective reformulation of problem~\eqref{eq:robust}, which was initially introduced in \citep[Appendix~A]{ho2018online}.
\begin{align*}
\min_{y \in \cY} \max_{z \in \cZ, \lambda \in \Delta^m} \, f(y) + \sum_{j \in [m]} \lambda_j g_j (y, z / \lambda_j)
\end{align*}
Here, $\Delta^m = \{ \lambda \in \R_+^m : \sum_{j \in [m]} \lambda_j = 1 \}$, and $0 \cdot g_j(y, z / 0)$ is interpreted as the negative recession function of the convex function $-g_j(y, \cdot)$. By construction, the objective function of this reformulation is convex in~$y$ and jointly concave in $Z$ and $\lambda$. While the primal-dual saddle point algorithm of \citet{postek2024first} typically enjoys an iteration complexity of $\cO(\varepsilon^{-2})$, where $\varepsilon$ now represents the primal-dual gap in the saddle point formulation, it requires more sophisticated oracles than those used by \citet{ho2018online,ho2019exploiting}. This is primarily because the perspective transformation eliminates favorable properties such as strong convexity and smoothness, and it also significantly degrades the Lipschitz constants. To address this challenge while still relying on standard oracles as in~\citep{ho2018online,ho2019exploiting}, \citet{tu2024max} propose a two-layer algorithm based on the following Lagrangian formulation of~\eqref{eq:robust}.
\begin{align*}
\max_{\lambda \in \R^m_+} \min_{y \in \cY} \max_{z \in \cZ} \, f(y) + \sum_{j \in [m]} \lambda_j g_j(y, z)
\end{align*}
\citet{tu2024max} show that their algorithm has an iteration complexity of $\cO(\varepsilon^{-3})$ or $\cO(\varepsilon^{-2})$, depending on the smoothness of the objective and constraint functions.

\subsection{Tailored Numerical Solution Methods for Specific Ambiguity Sets}

With the exception of some of the online optimization algorithms, the numerical solution methods discussed thus far still rely on general-purpose solvers to solve auxiliary nominal, scenario, master and/or sub-problems. General-purpose solvers are typically based on second-order interior-point methods that may fail to offer scalability to large-scale problem instances. To alleviate this concern, several first-order methods have been developed for specific classes of ambiguity sets.

\subsubsection{Gelbrich Ambiguity Sets}

Gelbrich ambiguity sets naturally emerge in signal processing and control applications. The standard reformulations of DRO problems over Gelbrich ambiguity sets, however, constitute semidefinite programs (\emph{cf}.~Theorem~\ref{thm:finite:convex:Gelbrich:I}), which significantly limits their scalability. To circumvent this shortcoming, \citet{shafieezadeh2018wasserstein} develop a Frank-Wolfe algorithm whose direction-finding subproblem admits a quasi-closed form solution. This algorithm enjoys a sublinear convergence rate. Leveraging the strong convexity of the Gelbrich ambiguity set, \citet{nguyen2019bridging} improve this Frank-Wolfe algorithm to achieve a linear convergence rate whenever the loss function satisfies the Levitin–Polyak condition \citep{polyak1966constrained}. Using frequency-domain techniques, \citet{kargin2024distributionally,kargin2024infinite} introduce a Frank-Wolfe algorithm for infinite-horizon robust control problems that involve infinite-dimensional moment matrices. Finally, \citet{mcallister2023distributionally} propose a Newton method for solving a class of DRO problems over Gelbrich ambiguity sets that converges superlinearly in numerical experiments.

\subsubsection{$\phi$-Divergence Ambiguity Sets}

The existing literature largely focuses on DRO problems over the restricted $\phi$-divergence ambiguity set~\eqref{eq:restricted-phi-divergence-ambiguity-set}, including the group DRO formulation introduced by \citet{sagawa2019distributionally} as a special case. Unfortunately, stochastic gradient methods applied directly to the dual minimization problem~\eqref{eq:weak-duality-restricted-phi-divergence} are known to be unstable. This challenge motivated \citet{namkoong2016stochastic} to adopt a direct saddle point formulation of the DRO problem with a discrete reference distribution $\hat \P$, which they solve iteratively with a bandit mirror descent algorithm. Several other algorithms address the saddle point formulation, including customized multi-level Monte Carlo methods \citep{levy2020large,hu2021bias,hu2024multi}, accelerated methods that query ball optimization oracles \citep{carmon2022distributionally}, and biased stochastic gradient methods \citep{ghosh2021efficient, wang2024regularization,azizian2023regularization}. \citet{gurbuzbalaban2022stochastic} and \citet{zhu2023distributionally} solve nonconvex DRO problems over classes of $\phi$-divergence ambiguity sets. Specifically, \citet{gurbuzbalaban2022stochastic} introduce a subgradient algorithm for non-smooth and nonconvex loss functions, while \citet{zhu2023distributionally} establish convergence rates and finite-sample guarantees for a subgradient method targeted at weakly convex loss functions. Both works build on the foundational results of \citet{ruszczynski2021stochastic}, which laid the groundwork for efficient first-order methods for multilevel optimization problems.

\subsubsection{Optimal Transport Ambiguity Sets}

\citet{li2019first} develop a first-order iterative method for distributionally robust logistic regression problems over $1$-Wasserstein balls. This method is based on a variant of the proximal alternating direction method of multipliers (ADMM). Numerical experiments demonstrate that the proposed algorithm is several orders of magnitude faster than general-purpose solvers. A similar conclusion is drawn by \citet{li2020fast}, who introduce epigraphical projection-based algorithms to solve distributionally robust support vector machine problems. When the loss function $\ell(x, z)$ is either convex-concave or convex-convex in~$x$ and~$z$, respectively, the reformulation of the DRO problem~\eqref{eq:primal:dro} reveals a structure that is conducive to distributed implementation. Consequently, \citet{cherukuri2020cooperative} use saddle point algorithms related to the augmented Lagrangian method to solve the reformulated problem over a network of agents. For convex-concave loss functions, \citet{li2020data} propose a hybrid algorithm that combines Frank-Wolfe and subgradient methods. For any fixed $x \in \mathcal{X}$, their approach solves the inner maximization problem in~\eqref{eq:primal:dro} with a variant of the Frank-Wolfe algorithm. The resulting maximizer is then used to construct an approximate subgradient for the outer minimization problem. All of these algorithms crucially rely on the reference distribution~$\hat \P$ being discrete. \citet{blanchet2020semi} and \citet{blanchet2018optimal} propose a stochastic gradient descent algorithm to solve DRO problems over optimal transport ambiguity sets with generic reference distributions. Other stochastic optimization schemes leverage variance reduction techniques \citep{yu2022fast} and zeroth-order random reshuffling algorithms \citep{maheshwari2022zeroth}. These works typically rely on the duality results introduced in Section~\ref{sec:duality-wc-expectation} and subsequently apply stochastic subgradient descent, using subgradients of the regularized loss function $\ell_{c}(x, \hat z) = \sup_{z \in \cZ} \ell(x, z) - \lambda c(z, \hat z)$ with respect to~$x$ and~$\lambda$. \citet{ho2020adversarial} extend this approach to nonconvex robust binary classification problems. \citet{sinha2018certifying} examine relaxed distributionally robust neural network training problems, assuming that the required level of robustness against adversarial perturbations is sufficiently small. This is tantamount to forcing~$\lambda$ to exceed a sufficiently large threshold. If $c(z,\hat z)=\|z-\hat z\|_2^2$, this in turn ensures that the maximization problem over~$z$ that defines $\ell_c(x, \hat z)$ has a strongly concave objective function and is thus efficiently solvable. Stochastic subgradients of $\ell_c(x,\hat z)$ are therefore readily available thanks to Danskin's theorem. \citet{shafiee2023new} leverage nonconvex duality theorems, such as Toland's duality principle, to solve distributionally robust portfolio selection problems. Algorithms that minimize the variation-regularized nominal loss, which is known to approximate the worst-case expected loss thanks to Theorem~\ref{thm:OT:taylor}, are explored by \citet{li2022tikhonov} and \citet{bai2024wasserstein}. Finally, \citet{wang2021sinkhorn,wang2024regularization} and \citet{azizian2023regularization} introduce entropy and $\phi$-divergence regularizers to improve the efficiency of algorithms for Wasserstein DRO problems, and \citet{vincent2024skdro} provide a Python library for training related distributionally robust machine learning models.

\section{Statistical Guarantees}
\label{sec:statistics}

Despite ample empirical evidence that distributionally robust decisions can outperform those provided by alternative methodologies for decision-making under uncertainty, the statistical properties of DRO remain underexplored. This section aims to survey some of the key techniques and methods employed in the literature to analyze the statistical aspects of DRO, while at the same time acknowledging that numerous questions remain open in this domain.

The statistical guarantees of moment-based ambiguity sets are relatively weak in the sense that the optimal value of problem~\eqref{eq:primal:dro} under a moment-based ambiguity set~$\cP$ does not match the optimal value of the corresponding stochastic program~\eqref{eq:sp} under the unknown true distribution $\mathbb{P} = \mathbb{P}_0$ even if $\mathbb{P}_0$ was known exactly when~$\cP$ is constructed. The reason for this is that exact knowledge of lower-order moments of $\mathbb{P}_0$ is not sufficient to uniquely characterize $\mathbb{P}_0$ itself. For this reason, our review focuses on $\phi$-divergence and optimal transport ambiguity sets, which offer asymptotic consistency guarantees as the number of samples available from $\mathbb{P}_0$ grows, and we refer to \citet{delage2010distributionally} and \citet{nguyen2021mean} for statistical analyses of Chebyshev and Gelbrich ambiguity sets, respectively. 


Section~\ref{sec:statistics:data-driven} introduces the data-driven optimization framework that we will be interested in, as well as the two key performance criteria of \emph{excess risk} and \emph{out-of-sample disappointment}. Subsequently, Section~\ref{sec:statistics:asymptotic} surveys asymptotic analyses, which are based on the laws of large numbers, the central limit theorem, the empirical likelihood approach as well as the large and moderate deviations principles. Finally, Section~\ref{sec:statistics:non-asymptotic} reviews non-asymptotic analyses, which rely on measure concentration bounds as well as generalization bounds.

Our review of the statistical properties of DRO omits several important topics. For example, we do not cover domain adaptation guarantees \citep{farnia2016minimax, volpi2018generalizing, lee2018minimax,lee2020learning, sutter2021robust,taskesen2021sequential, rychener2024wasserstein}, which ensure that a DRO model trained on data from some source distribution generalizes to a different target distribution.
We also omit discussions of adversarial generalization bounds \citep{sinha2018certifying, wang2019convergence, tu2019theoretical, kwon2020principled,an2021generalization}, which use DRO to analyze model robustness against adversarial perturbations, as well as applications in high-dimensional statistical learning \citep{aolaritei2022performance}. Finally, we do not cover Bayesian guarantees~\citep{gupta2019near,shapiro2023bayesian,liu2024bayesian}, which focus on average-case rather than worst-case performance guarantees.

\subsection{Excess Risk and Out-of-Sample Disappointment}\label{sec:statistics:data-driven}

Consider the idealized scenario in which the uncertainty underlying a decision problem follows a known probability distribution $\P_0 \in \cP(\cZ)$. In this case, we aim to determine a decision $x_0$ that minimizes the expected value of a loss function $\ell: \cX \times \cZ \to \R$ with respect to $\P_0$. That is, we seek an element~$x_0$ of
\begin{align}
\label{eq:stochastic:program}
\cX_0 = \argmin_{x \in \cX} ~ \E_{\P_0} [ \ell(x, Z) ].
\end{align}
Note that problem~\eqref{eq:stochastic:program} constitutes a classical stochastic program. While~\eqref{eq:stochastic:program} is theoretically sound, it faces two significant practical limitations. First, the distribution $\P_0$ underlying a decision problem is rarely known in practice. Second, even if $\P_0$ was known, evaluating the objective function of~\eqref{eq:stochastic:program} requires the computation of an integral, which is intractable in high dimensions even for simple nonlinear loss functions \citep{dyer2006computational, hanasusanto2016comment}.

In practice, we often observe the true probability distribution $\P_0$ indirectly through historical data. From now on we thus assume to have access to~$N$ independent training samples from $\P_0$, denoted as $Z_1, \dots, Z_N$. The goal of data-driven optimization is to construct a decision from the training samples. This decision should perform well not just on the training data, but also on unseen test samples from~$\P_0$. The performance of a data-driven decision on test data is also called its \emph{out-of-sample performance}. Formally, data-driven optimization aims to learn a decision rule $\cT_N : \cZ^N \rightrightarrows \cX$ that maps training samples from the product space~$\cZ^N$ to a set of candidate decisions in the decision space~$\cX$. Note that~$\cT_N$ constitutes a set-valued mapping because it is usually constructed as the set of minimizers of an optimization problem depending on the training samples. A data-driven decision is then any (Borel measurable) function~$\hat X_N$ of the training samples that satisfies
\begin{align*}
\hat X_N \in \cT_N(Z_1, \dots, Z_N).
\end{align*}
Note that $\hat X_N$ inherits the randomness of the training samples and is therefore a random vector. However, we notationally suppress its dependence of on the training samples in order to avoid clutter. Instead, we us the superscript ‘$\hat{\phantom{v}}$’ together with the subscript `$N$' to designate any random objects that are defined as functions of $Z_1,\ldots,Z_N$ and are thus governed by the product distribution~$\P_0^N$.


Arguably the simplest approach to data-driven optimization is the sample average approximation (SAA), which is also known as empirical risk minimization in statistics. The idea of SAA is to replace the unobservable true distribution~$\P_0$ in~\eqref{eq:stochastic:program} with the observable empirical distribution
\begin{align}
\label{eq:emprical:distribution}
\hat \P_N = \frac{1}{N} \sum_{i \in [N]} \delta_{Z_i}
\end{align}
formed from the training samples $Z_1, \ldots, Z_N$ and to construct the decision rule
\begin{align}
\label{eq:saa-decision-rule}
\cT_N(Z_1, \dots, Z_N) = \argmin_{x \in \cX} ~ \E_{\hat \P_N} [ \ell(x, Z) ].
\end{align}
As the empirical distribution is discrete, the SAA approach obviates the need to evaluate high-dimensional integrals and is thus computationally attractive. Nevertheless, the performance of its optimal solutions on test data can be disappointing even when the test data are independently sampled from the true distribution~$\P_0$. This phenomenon has been observed across various application domains and has been given different names depending on the context. In finance, \citet{michaud1989markowitz} identifies this issue as the \emph{error maximization effect} of portfolio optimization. Statistics and machine learning recognizes it as \emph{overfitting}, a well-known challenge where models perform well on training data but fail to generalize to new, unseen test data. In the stochastic programming literature, \citet{shapiro2003monte} refers to this phenomenon as the \emph{optimization bias}, and in decision analysis the effect has been described as the \emph{optimizer's curse} \citep{smith2006optimizer}.

The disappointing out-of-sample performance of the SAA decisions prompted statisticians and machine learners to add a regularization term to the objective function in~\eqref{eq:saa-decision-rule}. The regularization term serves two purposes. It not only combats overfitting to the training data, but it also encourages simpler decisions. Such simplicity aligns with the principle of parsimony and reflects nature's inherent tendency towards simplicity. As \citet{jeffreys1921certain} aptly noted, \\[1em]
\begin{minipage}{0.95\textwidth}
\centering
\emph{``The existence of simple laws is, then, apparently, to be regarded as a quality of nature; and accordingly we may infer that it is justifiable to prefer a simple law to a more complex one that fits our observations slightly better.''}
\end{minipage} \\[1em]
Formally, the regularized SAA approach provides the decision rule
\begin{align*}
\cT_N(Z_1, \dots, Z_N) = \argmin_{x \in \cX} ~ \E_{\hat \P_N} [ \ell(x, Z) ] + R(x),
\end{align*}
where the regularization function $R: \cX \to \R_+$ penalizes the complexity of decision~$x$. In the classical statistics literature, the regularization function is mostly \emph{data independent}, that is, it only depends on the decision $x$ and not on the observed training data $Z_1, \dots, Z_N$. The most prominent examples include norm regularization, where $R(x) = \| x \|$, and Tikhonov regularization, where $R(x) = \| x \|^2$. These regularization techniques balance the conflicting goals of computing decisions that are optimal for the observed training data and maintaining model simplicity, thereby improving the generalization capability of the derived decisions to unseen data.

Recall from Sections~\ref{sec:analytical-wc} and~\ref{sec:approximations-of-nature} that regularization and distributional robustness are closely intertwined. Assume that we use the empirical distribution $\hat \P_N$ as the center of a $\phi$-divergence ambiguity set~\eqref{eq:phi-divergence-ambiguity-set} or optimal transport ambiguity set~\eqref{eq:OT-ambiguity-set}. Then, the DRO approach provides the decision rule
\begin{align*}
\cT_N(Z_1, \dots, Z_N) = \argmin_{x \in \cX} \, \sup_{\P \in \hat \cP_N} \, \E_\P[\ell(x, Z)],
\end{align*}
which can be viewed as a variant of the regularized SAA decision rule. The corresponding {\em data-dependent} regularization function is called the {\em DRO regularizer} and is given by
\begin{align}
\label{eq:DRO:regularizer}
\hat R_N(x) = \sup_{\P \in \hat \cP_N} \, \E_\P[\ell(x, Z)] - \E_{\hat \P_N}[\ell(x, Z)].
\end{align}
Thus, it depends on both the decision $x$ and the observed training data $Z_1, \dots, Z_N$. The regularizer~\eqref{eq:DRO:regularizer} quantifies how much the worst-case expected loss across all distributions $\P \in \hat \cP_N$ can exceed the in-sample expected loss $\E_{\hat \P_N}[\ell(x, Z)]$.

The performance of decision rules in data-driven optimization is primarily measured by two criteria, each of which is aligned with a different field of study and addresses a different set of practical concerns. The first criterion, \emph{excess risk}, is predominantly used in statistics. It quantifies the distance of a data-driven decision~$\hat X_N$ to an optimal decision $x_0$. The second criterion, \emph{out-of-sample disappointment}, is more commonly employed in operations research. It provides a measure of how much the out-of-sample risk of a data-driven decision $\hat X_N$ exceeds the in-sample risk of $\hat X_N$. In the following, we formally define both criteria.

\paragraph{Excess Risk.} Let $\eta \in (0,1)$ be a significance level, $\cT_N$ be a decision rule, and $\Delta: \cX \times \cX_0 \to \R_+$ be a \emph{performance function}. Suppose that $\hat X_N \in \cT_N(Z_1, \dots, Z_N)$. The excess risk criterion offers the guarantee that for any size $N \geq N(\cX, \cZ, \eta)$ of the training set, we have
\begin{align}
\label{eq:excess:risk}
\P_0^N [\Delta(\hat X_N, x_0)\leq \hat \delta_N ] \geq 1 - \eta
\end{align}
for some (possibly data-dependent) error certificate $\hat \delta_N$. In statistical learning theory, performance functions often measure the regret in terms of the loss function $\ell$ under the true distribution $\P_0$. Specifically, for any feasible candidate decisions $x \in \cX$ and any optimal decision $x_0 \in \cX_0$, the \emph{regret} takes the form
\begin{align*}
\Delta(x, x_0) 
= \E_{\P_0} [\ell(x, Z)] - \E_{\P_0} [\ell(x_0, Z)]
= \E_{\P_0} [\ell(x, Z)] - \min_{x \in \cX} \E_{\P_0} [\ell(x, Z)] \geq 0.
\end{align*}
In compressed sensing and M-estimation problems with linear models, performance is often defined as the \emph{estimation error} in the decision space, and it takes the form
\begin{align*}
\Delta(x, x_0) = \| x - x_0 \|_2^2.
\end{align*}
Here, we assume for simplicity that the minimizer~$x_0$ is unique.
We refer to~\citep{mendelson2003few, bousquet2004introduction} for an introduction to statistical learning theory. For more advanced treatments, we refer to~\citep{anthony1999neural,koltchinskii2011oracle,vapnik2013nature,shalev2014understanding,vershynin2018high,wainwright2019high}.

\paragraph{Out-of-Sample Disappointment.}
Let $\eta \in (0,1)$ be a significance level and $\cT_N$ be a decision rule. Suppose that $\hat X_N \in \cT_N(Z_1, \dots, Z_N)$. The out-of-sample disappointment criterion offers the guarantee that for any size $N \geq N(\cX, \cZ, \eta)$ of the training set, we have
\begin{align}
\label{eq:disappointment}
\P_0^N \left[ \E_{\P_0} [\ell( \hat X_N, Z ) ] \leq \hat L_N \right] \geq 1 - \eta
\end{align}
for some (possibly data-dependent) loss certificate $\hat L_N$. Alternatively, one can express~\eqref{eq:disappointment} as a probabilistic bound on the difference between the out-of-sample performance and the in-sample performance,
\begin{align*}
\P_0^N \left[ \E_{\P_0} [\ell( \hat X_N, Z ) ] - \E_{\hat \P_N} [\ell( \hat X_N, Z ) ] \leq \hat \delta_N \right] \geq 1 - \eta,
\end{align*}
for some error certificate $\hat \delta_N$. Both criteria become equivalent when we set $\hat \delta_N = \E_{\hat \P_N} [\ell( \hat X_N, Z ) ] + \hat L_N$. Unlike the excess risk bound~\eqref{eq:excess:risk}, the out-of-sample disappointment bound~\eqref{eq:disappointment} does not require explicit knowledge of an optimal decision $x_0$ and solely leverages the statistical properties of $\P_0$. As we will see in the following sections, $\hat L_N$ and $\hat \delta_N$ typically correspond to the optimal value of the DRO problem~\eqref{eq:primal:dro} and the DRO regularizer~\eqref{eq:DRO:regularizer}, respectively.

\medskip

The next sections focus on ambiguity sets that are centered at the empirical distribution $\hat \P_N$ defined in~\eqref{eq:emprical:distribution}. Specifically, we consider ambiguity sets constructed using a discrepancy measure $\D:\cP(\cZ) \times \cP(\cZ) \to [0,\infty]$:
\begin{align}
\label{eq:ambiguity-set-N}
\hat \cP_N = \{ \P \in \cP(\cZ) : \D(\P, \hat \P_N) \leq r_N \}.
\end{align}
The discrepancy measure $\D$ could be a $\phi$-divergence or a Wasserstein distance. We will explain how the radius $r_N$ should scale with the training sample size~$N$ to obtain the least conservative statistical guarantees.

\subsection{Asymptotic Analyses}
\label{sec:statistics:asymptotic}

The laws of large numbers and the central limit theorem provide foundational insights into the statistical properties of the SAA approach. Under appropriate regularity conditions, the laws of large numbers guarantee that the empirical loss $\E_{\hat\P_N}[\ell(x,Z)]$ converges $\P_0$-almost surely to the true expected loss $\E_{\P_0}[\ell(x,Z)]$, uniformly on $\cX$ (see, \emph{e.g.}, \citealt[\ts~7.2.5]{shapiro2009lectures}). This implies that the optimal value and the set of optimal solutions of the SAA problem exhibit \emph{asymptotic consistency}, that is, they both converge to their counterparts in the stochastic program under $\P_0$ as the sample size $N$ approaches infinity. The central limit theorem, on the other hand hand, implies that the scaled difference between the empirical loss (under $\hat \P_N$) and true expected loss (under $\P_0$) converges weakly to a normal distribution with mean zero and variance equal to the true variance of the loss under $\P_0$ (see, \emph{e.g.}, \citealt[\ts~5.1.2]{shapiro2009lectures}). Thus, the optimal value of the SAA problem also exhibits \emph{asymptotic normality}. The asymptotic properties of the SAA decision rule have been studied extensively, see, \emph{e.g.},~\citep{cramer1946mathematical,huber1967behavior,dupacova1988asymptotic,shapiro1989asymptotic,shapiro1990differential,shapiro1991asymptotic,shapiro1993asymptotic,king1991epi,king1993asymptotic,van1998asymptotic,lam2021impossibility}.

Building on these foundations, we will next review the asymptotic consistency and normality of DRO decision rules. While studying these asymptotic behaviors, different theoretical frameworks provide distinct insights. The central limit theorem and empirical likelihood approaches characterize the typical fluctuations around the mean under an appropriate scaling. The central limit theorem establishes Gaussian convergence, whereas the empirical likelihood theory provides a nonparametric framework for constructing likelihood ratio tests with asymptotic $\chi^2$-limits, enabling hypothesis testing without specific parametric assumptions. In contrast, large deviations theory examines the tail behavior of distribution sequences. Rather than focusing on typical fluctuations, it characterizes the exponential decay rate of probabilities associated with rare events far from the mean. Moderate deviations theory bridges the gap between the typical and rare event analyses provided by the aforementioned frameworks. It studies the asymptotic behavior of distribution sequences at intermediate scales, thus investigating larger deviations than the central limit theorem but smaller deviations than large deviations theory.


\subsubsection{Asymptotic Consistency and Normality}

\citet[Theorem~6]{lam2019recovering} establishes the asymptotic (uniform strong) consistency of the optimal value of DRO decision rules over likelihood ambiguity sets. The proof relies on the preservation theorem of Glivenko-Cantelli classes~\citep[Theorem~3]{van2000preservation}, which intuitively says that function classes maintain their uniform convergence properties when combined through continuous operations, assuming that the original classes are well-behaved. \citet[Theorem~6]{duchi2021statistics} extend the analysis to more general $\phi$-divergence ambiguity~sets. 

\citet[Theorem~3.6]{mohajerin2018data} establish the asymptotic consistency of the optimal value and the optimal solutions of DRO decision rules over $1$-Wasserstein balls of the form~\eqref{eq:ambiguity-set-N}. Their proof combines the Borel–Cantelli Lemma~\citep[Theorem~2.18]{kallenberg1997foundations} with measure concentration results by \citet[Theorem~2]{fournier2015rate}. Intuitively, the Borel–Cantelli lemma asserts that if probabilities of an infinite sequence of events $(\cE_N)_{N \in \N}$ have a finite sum, then the probability of infinitely many occurrences of these events is zero. Leveraging its contraposition, \citet{mohajerin2018data} consider the events $\cE_N = { \W_1(\P_0, \hat \P_N) \leq r_N }$, where $\hat \P_N$ is the empirical distribution over $N$ independent samples from~$\P_0$; see~\eqref{eq:emprical:distribution}. By selecting a converging sequence of radii $(r_N)_{N \in \N}$ that decay according to a scaling law informed by \citep[Theorem~2]{fournier2015rate}, they prove that $\P_0^\infty( \lim_{N \to \infty} \W_1(\P_0, \hat \P_N) = 0 ) = 1$. This enables them to show that the optimal value of the DRO problem~\eqref{eq:primal:dro} over the $1$-Wasserstein ball~\eqref{eq:ambiguity-set-N} converges asymptotically \emph{from above} to the optimal value of the stochastic program~\eqref{eq:stochastic:program}. They also establish asymptotic convergence of the optimal solutions under an additional continuity assumption. This result can be extended to general $p$-Wasserstein ambiguity sets \citep[Theorem~20]{kuhn2019wasserstein}. Similar asymptotic convergence results have been established by \citet[Proposition~1]{gao2024wasserstein}, albeit through a different approach. Their proof does not rely on measure concentration results or an explicit characterization of the radius $r_N$. Instead, it leverages Theorem~\ref{thm:duality:OT} together with the reverse Fatou lemma and the monotone convergence theorem. This approach, however, does not explicitly determine whether the asymptotic convergence occurs from above or below. 

\citet[Theorem~4]{lam2019recovering} establishes the asymptotic normality of the optimal values of DRO problems over likelihood ambiguity sets. In a similar fashion, \citet[Theorem~10]{duchi2019variance} establish the asymptotic normality of the optimal solutions of DRO problems over Pearson $\chi^2$-divergence ambiguity sets. \citet[Theorem~11]{duchi2021learning} extend this result to Cressie-Read ambiguity sets. The asymptotic normality of DRO decision rules over $p$-Wasserstein balls, finally, is established by \citet{blanchet2019robust,blanchet2018distributionally,blanchet2019confidence}.

More recently, \citet{blanchet2023statistical} have developed a comprehensive framework for analyzing statistical limit theorems for DRO decision rules over both $\phi$-divergences and Wasserstein ambiguity sets of the form~\eqref{eq:ambiguity-set-N}. By connecting data-driven DRO formulations to their regularized counterparts (\emph{cf.}~Section~\ref{sec:approximations-of-nature}), their framework provides insights into how DRO decision rules behave depending on the rate at which the radius $r_N$ decreases with the sample size $N$. Specifically, \citet[\ts~2.2]{blanchet2023statistical} show that, under suitable regularity conditions, DRO formulations typically exhibit three distinct asymptotic behaviors.
\begin{enumerate}[label=(\roman*)]
\item When $r_N$ decreases faster than the critical statistical rate of $N^{-1/2}$, the DRO effect becomes negligible compared to the sampling error, and the asymptotic behavior of DRO mirrors that of standard empirical risk minimization.
\item When $r_N$ decreases at precisely the critical rate $N^{-1/2}$, the DRO effect manifests itself as a quantifiable asymptotic bias term that acts as a regularizer, and its interaction with the statistical noise results in a shifted normal limiting distribution.
\item When $r_N$ decreases slower than $N^{-1/2}$, the DRO effect dominates the statistical noise, leading to a limiting behavior governed primarily by the geometry of the ambiguity set.
\end{enumerate}
The analysis employs the functional central limit theorem alongside careful Taylor expansions of the worst-case expectation akin to those presented in Section~\ref{sec:approximations-of-nature}. In particular, the authors establish that, under appropriate regularity conditions, the limiting distributions are normal with explicitly characterized means and variances. 

\subsubsection{Empirical Likelihood Approach}

The (generalized) empirical likelihood theory introduced by \citet{owen1988empirical,owen1990empirical,owen1991empirical,owen2001empirical} provides a powerful nonparametric analogue to parametric maximum likelihood theory. At its core, the empirical distribution $\hat \P_N$ serves as a nonparametric maximum likelihood estimator for the unknown true distribution $\hat \P_0$, and statistical relies on empirical likelihood ratios. Under suitable conditions, the empirical likelihood ratio statistic converges to a $\chi^2$-distribution. Unlike the central limit theorem, which yields normal approximations and thus symmetric confidence intervals, the empirical likelihood theory typically produces asymmetric confidence regions. A key advantage of this approach is that the resulting data-driven confidence regions automatically adapt to the geometry of the underlying distribution and naturally respect constraints such as boundedness or non-negativity, without requiring explicit transformations or variance estimation. However, this theoretical elegance and flexibility comes at the computational overhead of computing both the lower and upper bounds of the confidence interval separately.

In the following, we briefly review the empirical likelihood approach and its application to DRO decision rules. Let $Z_1, \dots, Z_N$ be independent samples from~$\P_0$, and let $\theta : \cP(\cZ) \to \R$ be a statistical quantity of interest (\emph{e.g.}, the expected value of $Z_i$). Empirical likelihood confidence regions for $\theta(\P_0)$ can be constructed as
\begin{align}
\label{eq:formula-for-C-N}
\hat \cC_N = \left\{ \theta(\P) ~:~ \D_\phi(\P, \hat \P_N) \leq \tfrac{r}{N} \right\}
\end{align}
For some $r \in \R_+$.  Thus, the set $\hat \cC_N$ is the image of a $\phi$-divergence neighborhood around the empirical distribution $\hat \P_N$ under $\theta$. The key tool to establish probabilistic bounds is the so called \emph{profile divergence} $\pi_N: \R \to \R_+$, which is defined through
\begin{align}
\label{eq:profile}
\pi_N (\tau) = \inf_{\P \in \cP(\cZ)} \left\{ \D_\phi(\P, \hat \P_N) ~: ~ \theta(\P) = \tau \right\}.
\end{align}
For a functional $\theta$ satisfying suitable smoothness conditions, the empirical likelihood method provides asymptotically exact coverage guarantees of the form
\begin{align*}
\lim_{N \to \infty} \P_0^N \big( \theta(\P_0) \in \hat \cC_N \big)
= \lim_{N \to \infty} \P_0^N \big( \pi_N(\theta(\P_0)) \leq \tfrac{r}{N} \big) \
= 1 - \eta,
\end{align*}
where $\eta$ represents a significance level determined by $r$ and $\theta$.

The classical empirical likelihood approach~\citep{owen1988empirical,owen2001empirical} relies on the empirical likelihood divergence with entropy function $\phi(s) = - \log(s) + s - 1$ if $s\geq 0$ and $\phi(s)=\infty$ if $s<0$ (see Table~\ref{tab:phi-divergence}). In this case, $\pi_N$ is called \emph{profile likelihood}. Assume that $Z$ is a $d$-dimensional random vector that is governed by the distribution $\P_0$ and whose covariance matrix has rank $d_0 \leq d$. For the expected value $\theta(\P_0) := \E_{\P_0}[Z]$, 
\citet{owen1990empirical} proves that, as $N \to \infty$, we have
$$
\pi_N(\E_{\P_0}[Z]) \stackrel{d}{\rightarrow} \chi^2_{d_0},
$$
where $\chi^2_{d_0}$ denotes the $\chi^2$-distribution with $d_0$ degrees of freedom. Thus, $\hat \cC_N$ constitutes an asymptotically exact $(1-\eta)$-confidence interval for $\theta(\P_0)$ if we set $r$ in~\eqref{eq:formula-for-C-N} to the $(1-\eta)$-quantile of a $\chi^2$-distribution with $d_0$ degrees of freedom.

In the context of stochastic programming problems, the statistical quantity of interest is typically the optimal value of the stochastic program, that is, $\theta(\P) = \inf_{x \in \cX} \, \E_\P[\ell(x, Z)]$. In this case, the set $\hat \cC_N$ becomes the interval
\begin{align*}
\hat \cC_N = \Big[ \inf_{\P \in \hat \cP_N} \inf_{x \in \cX} \, \E_\P[\ell(x, Z)], \;\; \sup_{\P \in \hat \cP_N} \inf_{x \in \cX} \, \E_\P[\ell(x, Z)]  \Big],
\end{align*}
where $\hat \cP_N$ is the $\phi$-divergence ambiguity set of the form~\eqref{eq:ambiguity-set-N} around~$\hat\P_N$. 
If $\hat \cP_N$ is a likelihood ambiguity set, \citet{lam2019recovering} investigates the asymptotic coverage probability of this interval by leveraging asymptotic guarantees for the SAA decision rule by \citet{lam2017empirical}. In particular, he shows that if suitable regularity conditions hold and $r_N = r / N$, where $r$ is the $(1-\eta)$-quantile of a $\chi^2$-distribution with a single degree of freedom, then $\hat \cC_N$ becomes an asymptotically exact $(1-\eta)$-confidence interval. One can thus show that the resulting confidence bounds achieve the asymptotically exact coverage at the parametric rate $\cO(N^{-\frac{1}{2}})$. \citet{duchi2021statistics} further generalize these results to DRO decision rules over broader classes of $\phi$-divergence ambiguity sets. Additionally, \citet{he2021higher} examine higher-order coverage errors and introduce a correction term similar to the Bartlett correction. The authors derive higher-order correction terms for general von Mises differentiable functionals and thus move beyond the approximately smooth functions previously studied in the empirical likelihood literature.

In a parallel line of research, \citet{blanchet2019robust,blanchet2018distributionally,blanchet2019confidence,blanchet2019sample} and \citet{lin2024small} introduce the \emph{Wasserstein profile} function as a Wasserstein analogue to the profile divergence~\eqref{eq:profile}. This approach replaces the $\phi$-divergence with the 2-Wasserstein distance, and it offers a geometric perspective on uncertainty quantification. This approach yields confidence bounds that achieve asymptotic parametric rate $\cO(N^{-\frac{1}{2}})$. For more details, we direct the readers to the recent survey by \citet{blanchet2021statistical}.


\subsubsection{Large and Moderate Deviations Principles}

Unlike the central limit theorem and the empirical likelihood approach, which characterize limits of distribution sequences, large and moderate deviations theory study the asymptotic tail behavior of distribution sequences. Specifically, they prove exponential decay rates of probabilities of rare events over sequences of random variables. The foundations of large deviations theory trace back to two seminal developments in physics and mathematics. The first is Boltzmann's groundbreaking works on statistical mechanics and entropy. The second is Cram\'er's pioneering paper on the asymptotic behavior of sums of random variables \citep{cramer1938nouveau}. Despite these early advances, the field lacked a unified mathematical framework until Varadhan's seminal paper~\citep{varadhan1966asymptotic}, which introduces a formal definition of a \emph{large deviation principle}. We refer to the textbooks by \citet{ellis2006entropy} and \citet{dembo2009large} for a modern treatment of the topic. 

Assume now that the unknown true distribution $\P_0$ is known to belong to a parametric distribution family $\{\P_\theta:\theta\in\Theta\}\subseteq \cP(\cZ)$, where~$\theta$ ranges over a prescribed parameter space~$\Theta$. In this case, estimating~$\P_0$ is tantamount to estimating the unknown true parameter vector~$\theta_0\in\Theta$ that satisfies~$\P_0=\P_{\theta_0}$. A {\em statistic}~$\hat\theta_N$ is a random variable valued in~$\Theta$ and constructed from $(Z_1,\ldots, Z_N)\sim\P_\theta^N$ that converges in probability to~$\theta$ as $N$ grows, for any $\theta\in\Theta$. Formally, we say that the statistic $\hat \theta_N$ satisfies a {\em large deviations principle} with speed $b_N$ and with lower semicontinuous rate function $I : \Theta\times\Theta \to [ 0, \infty]$ if
\begin{equation}
\label{eq:ldp}
\begin{aligned}
-\inf_{\theta' \in \inter(\cB)} I(\theta',\theta) 
&\leq \liminf_{N \to \infty} \frac{1}{b_N} \log \P_\theta(\hat \theta_N \in \cB ) \\
&\leq \limsup_{N \to \infty} \frac{1}{b_N} \log \P_\theta(\hat \theta_N \in \cB ) 
\leq -\inf_{\theta' \in \cl(\cB)} I(\theta',\theta)
\end{aligned}
\end{equation}
for all $\theta\in\Theta$ and for all Borel sets $\cB \subseteq \Theta$. Here, we assume that the sequence $b_N$, $N\in\N$, tends monotonically towards infinity. If~\eqref{eq:ldp} holds, one can show under mild conditions that $I(\theta,\theta)=0$ because $\hat\theta_N$ converges to~$\theta$ in probability under~$\P_\theta$. It is therefore natural to interpret~$I(\theta',\theta)$ as a discrepancy function that quantifies the dissimilarity between the estimator realization~$\theta'$ and the probabilistic model~$\theta$. As~$I$ is lower semicontinuous, the minimization problems on the left and on the right hand side of~\eqref{eq:ldp} share the same infimum $r=\inf_{\theta' \in \inter(\cB)} I(\theta',\theta)=\inf_{\theta' \in \cl(\cB)} I(\theta',\theta)$ for most Borel sets~$\cB$ of interest. In these cases, the inequalities in~\eqref{eq:ldp} collapse to equalities, and~\eqref{eq:ldp} simplifies to the more intuitive statement
\[
\P_\theta(\hat \theta_N \in \cB )=\exp\left(-r b_N +o(b_N) \right).
\]
That is, the probability of the estimator~$\hat\theta_N$ falling into the set~$\cB$ decays exponentially at rate~$r$ with speed~$b_N$, where $r$ can be interpreted as the $I$-distance from~$\theta$ to~$\cB$.

Several statistics of practical interest satisfy large deviations principles. For example, if~$\cZ$ is finite and $\{\P_\theta:\theta\in\Theta\}$ is the family of {\em all} distributions on~$\cZ$ encoded by the corresponding probability vectors~$\theta\in\Theta$, where $\Theta$ is the probability simplex of appropriate dimension, then the empirical distribution~$\hat\P_N$ corresponding to the empirical probability vector $\hat \theta_N$ is an estimator for the data-generating distribution~$\P_\theta$. In this case, Sanov's theorem \citep[Theorem~11.4.1]{cover2006elements} asserts that~$\hat\theta_N$ satisfies a large deviations principle with rate function $I(\theta',\theta)=\KL(\P_{\theta'},\P_\theta)$ and speed $b_N=N$. Similarly, if $\{\P_\theta:\theta\in\Theta\}$ is any distribution family parametrized by its unknown mean vector $\theta=\E_{\P_\theta}[Z]$ and if the log-moment generating function $\Lambda_\theta(t) = \log(\E_{\P_\theta}[\exp(t^\top Z)])$ is finite for all~$t,\theta\in\R^d$, then the sample mean $\hat \theta_N = \tfrac{1}{N} \sum_{i \in [N]} Z_i$ is an estimator for~$\theta$. In this case, Cram\'er's theorem~\citep{cramer1938nouveau} asserts that $\hat \theta_N$ satisfies a large deviations principle with rate function $I(\theta',\theta) = \Lambda_\theta^*(\theta')$ and speed $b_N = N$. Note that the log-moment generating function~$\Lambda_\theta$ as well as its conjugate $\Lambda_\theta^*$ are both convex. We remark that a large deviations principle with sublinear speed ($\lim_{N\to\infty}b_N/N=0$) is sometimes referred to as a {\em moderate deviations principle}. For an example of a moderate deviations principle we refer to \citep{jongeneel2022efficient}.

\citet{van2021data} leverage Sanov's theorem to show that the optimal value of the DRO problem with a likelihood ambiguity set of radius~$r$ around the empirical distribution~$\hat\P_N$ yields the least conservative confidence bound on the optimal value of the true stochastic program, asymptotically as the sample size~$N$ grows large, with significance level $\eta$ decaying exponentially as $e^{-rN}$. More generally, \citet{sutter2024pareto} assume that~$\P_0$ is known to belong to a parametric distribution family $\{\P_\theta:\theta\in\Theta\}$ and that~$\theta$ admits an estimator~$\hat\theta_N$ that satisfies a large deviations principle with rate function~$I$ and speed~$b_N=N$. Under some regularity conditions, they then show that the optimal value of the DRO problem with ambiguity set $\hat\cP_N = \{\P_\theta:\theta\in\Theta,\; I(\hat\theta_N,\theta)\leq r\}$ yields again the least conservative confidence bound on the optimal value of the true stochastic program with significance level $\eta\propto e^{-rN}$. Similar statistical optimality results can sometimes be obtained even when the training samples are serially dependent, {\em e.g.}, when they are generated by a Markov process with unknown transition probability matrix or certain autoregressive processes \citep{sutter2024pareto}.

The DRO estimators by \citet{van2021data} and \citet{sutter2024pareto} lack asymptotic consistency because they exploit large deviations principles with linear speed $b_N=N$. \citet{bennouna2021learning} show that asymptotic consistency can be recovered by relying on moderate deviations principles with sublinear speed. This line of research has seen significant recent developments. The use of large and moderate deviations principles has also been extended to various learning and control settings such as distributionally robust Markov decision processes \citep{li2021distributionally}, bandit problems \citep{van2024optimal}, bootstrap-based methods \citep{bertsimas2022bootstrap}, optimal learning \citep{ganguly2023optimal,liu2023smoothed}, control \citep{jongeneel2021topological,jongeneel2022efficient}, contextual learning \citep{srivastava2021data}, and robust statistics \citep{chan2024distributional}. 

\subsection{Non-Asymptotic Analyses}
\label{sec:statistics:non-asymptotic}

Non-asymptotic statistics seeks finite-sample guarantees that hold regardless of the sample size. This is in contrast to the asymptotic methods described in Section~\ref{sec:statistics:asymptotic}, which rely on properties that emerge as sample size tends infinity. Non-asymptotic methods allow for a rigorous control over error rates, which makes them robust in situations where asymptotic approximations might produce misleading results. In the following, we review two major classes of non-aymptotic analyses, that is, measure concentration bounds and generalization bounds.


\subsubsection{Measure Concentration Bounds}
The most elementary approach to obtain finite sample guarantees is to design the ambiguity set $\hat \cP_N$ such that it contains the unknown true probability distribution $\P_0$ with high probability. This requires an analysis of the convergence rate of $\hat \P_N$ towards $\P_0$, and it leads to out-of-sample disappointment bounds that depend only on $\hat \cP_N$ and not on the complexity of the loss function $\ell$ or the decision space $\cX$.

\begin{theorem}[Out-of-Sample Disappointment]
\label{thm:finite-sample}
Suppose that the ambiguity set $\hat \cP_N$ defined in~\eqref{eq:ambiguity-set-N} satisfies
\begin{equation}
\label{eq:measure:concentration}
\P_0^N \left( \P_0 \in \hat \cP_N \right) \geq 1 - \eta.
\end{equation}
We then have
\begin{subequations}
\begin{equation}
	\label{eq:generalization-1}
	\P_0^N \Big( \E_{\P_0}[\ell(x, Z)] \leq \sup_{\P \in \hat \cP_N} \E_{\P}[\ell(x, Z)] ~~ \forall x \in \cX \Big) \geq 1 - \eta.
\end{equation}
Moreover, if $\hat X_N$ is an optimizer of the distributionally robust decision problem with respect to the ambiguity set $\hat \cP_N$, then we have
\begin{equation}
	\label{eq:generalization-2}
	\P_0^N \Big( \E_{\P_0}[\ell(\hat X_N, Z)] \leq \min_{x \in \cX} \sup_{\P \in \hat \cP_N} \E_{\P}[\ell(x, Z)]  \Big) \geq 1 - \eta.
\end{equation}
\end{subequations}
\end{theorem}

The proof of~\eqref{eq:generalization-1} and~\eqref{eq:generalization-2} readily follows from the measure concentration bound~\eqref{eq:measure:concentration} and is therefore omitted. Theorem~\ref{thm:finite-sample} asserts that the worst-case expected loss provides an upper confidence bound on the true expected loss under the unknown data-generating distribution uniformly across all loss functions. Moreover, it also asserts that the optimal value of the DRO problem~\eqref{eq:primal:dro} provides an upper confidence bound on the out-of-sample performance of its optimizers. 

When using $\phi$-divergences to construct $\hat \cP_N$ as in~\eqref{eq:ambiguity-set-N}, the probabilistic requirement~\eqref{eq:measure:concentration} only applies to underlying distributions $\P_0$ that are discrete~\citep[\ts~7]{polyanskiy2024information}. In contrast, the Wasserstein distance applies to generic distributions $\P_0$. This area of study has a rich history, with seminal contributions from \citet{dudley1969speed}, \citet{ajtai1984optimal}, and \citet{dobric1995asymptotics}. More recent advancements have been made by \citet{bolley2007quantitative}, \citet{boissard2014mean}, \citet{dereich2013constructive}, and \citet{fournier2015rate}. Of particular importance to our discussion is the following measure concentration result, which serves as the foundation for finite sample guarantees in DRO over $p$-Wasserstein ambiguity sets. 

\begin{theorem}[{Measure Concentration \citep[Theorem~2]{fournier2015rate}}]
\label{theorem:concentration-empirical}
Suppose that $\hat \P_N$ is the empirical distribution constructed from $N$ independent samples from $\P_0$, $p \neq d/2$, and that $\P_0$ is light-tailed in the sense that there exist $\alpha > p$ and $A>0$ such that $\E_{\P_0} ( \exp( \| Z \|^\alpha) ) \leq A$. Then, there are constants $c_1, c_2>0$ that depend on $\P_0$ only through $\alpha$, $A$, and $d$ such that for any $\eta \in (0, 1]$, the concentration inequality $\P_0^N (W_p(\P_0, \hat \P) \leq r_N) \geq 1 - \eta$ holds whenever~$r$ exceeds
\begin{equation}
\label{eq:opt-radius-empirical}
r(d,N,\eta) = \left\{ \begin{array}{ll} \displaystyle \left(\frac{\log (c_1 /\eta)}{c_2N} \right)^{\min\{{1}/{d} ,{1}/{2}\}} & \displaystyle \text{if } N \ge \frac{\log(c_1 /\eta)}{c_2}, \\[2ex]
	\displaystyle \left(\frac{\log (c_1/ \eta) }{c_2 N} \right)^{{1}/{\alpha}} & \displaystyle \text{if } N < \frac{\log(c_1 /\eta) }{c_2}.        \end{array}\right.
	\end{equation}
\end{theorem}
The result remains valid for $p = d/2$ but with a more complicated formula for $r(d,N,\eta)$ \cite[Theorem~2]{fournier2015rate}. Intuitively, Theorem~\ref{theorem:concentration-empirical} asserts that any $p$-Wasserstein ball $\hat\cP_N$ of~$r_N \ge r(d,N,\eta)$ around~$\hat\P_N$ represents a $(1-\eta)$-confidence set for the unknown data-generating distribution~$\P_0$. For uncertainty dimensions $d>2$, the critical radius~$r(d,N,\eta)$ of this confidence set decays as $\cO(N^{-\frac{1}{d}})$. In other words, to reduce the critical radius by~$50\%$, the sample size must increase by $2^d$. Unfortunately, this curse of dimensionality is fundamental, and the decay rate of $r(d,N,\eta)$ is essentially optimal \cite[\ts~1.3]{fournier2015rate}. Explicit constants $c_1$ and $c_2$ are provided by \citet{fournier2022convergence}. 

Generic measure concentration bounds suffer from a curse of dimensionality. \citet{shafieezadeh2019regularization} and \citet{wu2022generalization} show that this curse can be overcome in the context of linear prediction models by projecting~$Z$ to a one-dimensional random variable, yielding the parametric convergence rate $\cO(N^{-\frac{1}{2}})$. \citet{nietert2024outlier} develop a similar approach for rank-$k$ linear models, where $2 < k < d$, and achieve an improved rate of $\cO(N^{-\frac{1}{k}})$ based on $k$-sliced Wasserstein distances. The $1$-sliced Wasserstein distance is also used by~\citet{olea2022out} to obtain the parametric rate $\cO(N^{-\frac{1}{2}})$ for a class of regression problems.

We conclude this section by highlighting that the DRO approach admits instance-dependent regret bounds, which essentially depend on no complexity measures of the decision space or the loss function. Instead, they only depend on the complexity of the optimal solution $x_0$ through the DRO regularizer $\hat R_N(x_0)$. \citet[Theorem~4.1]{zeng2022generalization} and \citet[Theorem~1]{nietert2024outlier} establish such bounds for DRO problems over the ambiguity set~\eqref{eq:ambiguity-set-N} when $\D$ is the maximum mean discrepancy and the (outlier-robust) Wasserstein distance, respectively. Similar instance-dependent guarantees for DRO problems with Wasserstein ambiguity sets are developed by \citet{hou2023instance}.

\subsubsection{Generalization Bounds}

An alternative approach to obtain statistical guarantees leverages the union bound from probability theory and covering numbers or complexity measures from statistical learning theory. The first step consists in deriving an inequality of the~form
\begin{equation}
\label{eq:for:all}
\P_0^N \Big( \E_{\P_0}[\ell(x, Z)] \leq \hat L_N(x) \Big) \geq 1 - \eta \quad \forall x \in \cX,
\end{equation}
where the loss certificate~$\hat L_N(x)$ depends on the decision $x \in \cX$. For example, a guarantee of the form~\eqref{eq:for:all} can be obtained by combining empirical Bernstein inequalities \citep{maurer2009empirical} and a DRO model with a $\chi^2$-divergence ambiguity set \citep[Theorem~2]{duchi2019variance}. In this case, the certificate~$\hat L_N(x)$ reduces to the sum of the  expected loss under~$\hat \P_N$ and a variance regularizer under~$\P_0$. Alternatively, a guarantee of the form~\eqref{eq:for:all} can also be obtained by combining transport inequalities \citep{marton1986simple,talagrand1996transportation} and a DRO model with a Wasserstein ambiguity set \citep[Theorem~1]{gao2020finite}. In this case, $\hat L_N(x)$ reduces to the sum of the  expected loss under~$\hat \P_N$ and a variation regularizer under~$\P_0$. The second step consists in converting the individual guarantee~\eqref{eq:for:all} to a uniform guarantee. For example, if $\cX$ is finite, this can easily be achieved by using the union bound. If~$\cX$ is uncountable, one may use one of several standard techniques. If the loss function is Lipschitz continuous in~$x\in\cX$ uniformly across all $z\in\cZ$ and $\cX$ is compact, then one can discretize~$\cX$ by uniform gridding. In this case, the loss at an arbitrary point is uniformly approximated by the loss at the nearest grid point, and a uniform guarantee can again be obtained by using the union bound. However, the number of grid points needed for an $\varepsilon$-approximation is of the order $\cO((1/\varepsilon)^d)$, which is impractical in high dimensions~$d$. A more sophisticated approach to discretize~$\cX$ exploits structural knowledge of the loss function at multiple scales. However, obtaining tight approximation in high dimensions remains challenging. In order to mitigate the computational burden related to discretization, one may exploit several complexity measures that quantify the expressiveness of the functions $\ell(x,\cdot)$ for all $x\in\cX$ such as the VC dimension or the Rademacher complexity as well as its local version. Nonetheless, Rademacher complexities can be computationally challenging to compute. For full details we refer to \citep{boucheron2013concentration, vershynin2018high,wainwright2019high}. 

The last step consists in approximating the certificate~$\hat L_N(x)$ by the worst-case expected loss over a data-driven ambiguity set $\hat\cP_N$ based on the $\chi^2$-divergence or a Wasserstein distance. The corresponding approximation error can be controlled by leveraging Taylor approximations as in Theorems~\ref{thm:phi:taylor} and \ref{thm:OT:taylor} together with appropriate concentration inequalities. In summary, this procedure shows that the optimal value of a data-driven DRO problem over a $\chi^2$-divergence or a Wasserstein ambiguity set provides a finite-sample upper confidence bound on the corresponding stochastic program under the unknown true distribution~$\P_0$.

\citet{duchi2019variance} and \citet{gao2020finite} derive generalization bounds of this kind for $\chi^2$-divergence and Wasserstein ambiguity sets, respectively, while \citet{azizian2023exact} extend their analysis to entropic regularized optimal transport ambiguity sets. All these bounds exhibit the parametric rate $\cO(N^{-\frac{1}{2}})$. In addition, \citet{duchi2019variance} demonstrate that, under certain curvature conditions, $\chi^2$-divergence decision rules can achieve the fast rate $\cO(N^{-1})$.


\paragraph{Acknowledgments.} This research was supported by the Swiss National Science Foundation under the NCCR Automation (grant agreement 51NF40\_180545). The authors thank Nicolas Lanzetti, Mengmeng Li, Karthik Natarajan, Jakob Nyl\"of, Yves Rychener, Philipp Schneider, Buse Sen, Ehsan Sharifian, Bradley Sturt and Man-Chung Yue for their valuable feedback on the paper. We are responsible for all remaining errors.

\addcontentsline{toc}{section}{References}
\bibliography{bibliography}

\begin{thebibliography}{xx}

\bibitem[Acerbi(2002)Acerbi]{ref:acerbi2002spectral}
C.~Acerbi  (2002), Spectral measures of risk: A coherent representation of
  subjective risk aversion, {\em Journal of Banking \& Finance} {\bf 26}(7),
  1505--1518.

\bibitem[Ahmadi-Javid(2012)Ahmadi-Javid]{ahmadi2012entropic}
A.~Ahmadi-Javid  (2012), Entropic value-at-risk: {A} new coherent risk measure,
  {\em Journal of Optimization Theory and Applications} {\bf 155}(3),
  1105--1123.

\bibitem[Ahmed(2006)Ahmed]{ahmed2006convexity}
S.~Ahmed  (2006), Convexity and decomposition of mean-risk stochastic programs,
  {\em Mathematical Programming} {\bf 106}(3), 433--446.

\bibitem[Ajtai {\em et~al.}(1984)Ajtai, Koml{\'o}s and
  Tusn{\'a}dy]{ajtai1984optimal}
M.~Ajtai, J.~Koml{\'o}s and G.~Tusn{\'a}dy  (1984), On optimal matchings, {\em
  Combinatorica} {\bf 4}(4), 259--264.

\bibitem[Al~Taha {\em et~al.}(2023)Al~Taha, Yan and
  Bitar]{al2023distributionally}
F.~Al~Taha, S.~Yan and E.~Bitar  (2023), A distributionally robust approach to
  regret optimal control using the {W}asserstein distance, in {\em IEEE
  Conference on Decision and Control}, pp.~2768--2775.

\bibitem[Ali and Silvey(1966)Ali and Silvey]{ali1966general}
S.~M. Ali and S.~D. Silvey  (1966), A general class of coefficients of
  divergence of one distribution from another, {\em Journal of the Royal
  Statistical Society: Series B} {\bf 28}(1), 131--142.

\bibitem[Altschuler and Boix-Adsera(2023)Altschuler and
  Boix-Adsera]{altschuler2023polynomial}
J.~M. Altschuler and E.~Boix-Adsera  (2023), Polynomial-time algorithms for
  multimarginal optimal transport problems with structure, {\em Mathematical
  Programming} {\bf 199}(1), 1107--1178.

\bibitem[Ambrosio {\em et~al.}(2008)Ambrosio, Gigli and
  Savar{\'e}]{ambrosio2008gradient}
L.~Ambrosio, N.~Gigli and G.~Savar{\'e}  (2008), {\em Gradient Flows: {I}n
  Metric Spaces and in the Space of Probability Measures}, Springer.

\bibitem[An and Gao(2021)An and Gao]{an2021generalization}
Y.~An and R.~Gao  (2021), Generalization bounds for ({W}asserstein) robust
  optimization, in {\em Advances in Neural Information Processing Systems},
  pp.~10382--10392.

\bibitem[Analui and Pflug(2014)Analui and Pflug]{analui2014distributionally}
B.~Analui and G.~C. Pflug  (2014), On distributionally robust multiperiod
  stochastic optimization, {\em Computational Management Science} {\bf 11},
  197--220.

\bibitem[Anthony and Bartlett(1999)Anthony and Bartlett]{anthony1999neural}
M.~Anthony and P.~L. Bartlett  (1999), {\em Neural Network Learning:
  Theoretical Foundations}, Cambridge University Press.

\bibitem[Anunrojwong {\em et~al.}(2024)Anunrojwong, Balseiro and
  Besbes]{anunrojwong2024robustness}
J.~Anunrojwong, S.~R. Balseiro and O.~Besbes  (2024), On the robustness of
  second-price auctions in prior-independent mechanism design, {\em Operations
  Research (Forthcoming)}.

\bibitem[Aolaritei {\em et~al.}(2022{\em a})Aolaritei, Lanzetti, Chen and
  D{\"o}rfler]{aolaritei2022uncertainty}
L.~Aolaritei, N.~Lanzetti, H.~Chen and F.~D{\"o}rfler  (2022{\em a}),
  Uncertainty propagation via optimal transport ambiguity sets, {\em
  arXiv:2205.00343}.

\bibitem[Aolaritei {\em et~al.}(2022{\em b})Aolaritei, Shafiee and
  D{\"o}rfler]{aolaritei2022performance}
L.~Aolaritei, S.~Shafiee and F.~D{\"o}rfler  (2022{\em b}), Wasserstein
  distributionally robust estimation in high dimensions: {P}erformance analysis
  and optimal hyperparameter tuning, {\em arXiv:2206.13269}.

\bibitem[Artzner {\em et~al.}(1999)Artzner, Delbaen, Eber and
  Heath]{artzner1999coherent}
P.~Artzner, F.~Delbaen, J.-M. Eber and D.~Heath  (1999), Coherent measures of
  risk, {\em Mathematical Finance} {\bf 9}(3), 203--228.

\bibitem[Atkinson and Mitchell(1981)Atkinson and Mitchell]{atkinson1981rao}
C.~Atkinson and A.~F. Mitchell  (1981), Rao's distance measure, {\em
  Sankhy{\=a}: The Indian Journal of Statistics, Series A} {\bf 43}(3),
  345--365.

\bibitem[Azizian {\em et~al.}(2023{\em a})Azizian, Iutzeler and
  Malick]{azizian2023exact}
W.~Azizian, F.~Iutzeler and J.~Malick  (2023{\em a}), Exact generalization
  guarantees for (regularized) {W}asserstein distributionally robust models, in
  {\em Advances in Neural Information Processing Systems}, pp.~14584--14596.

\bibitem[Azizian {\em et~al.}(2023{\em b})Azizian, Iutzeler and
  Malick]{azizian2023regularization}
W.~Azizian, F.~Iutzeler and J.~Malick  (2023{\em b}), Regularization for
  {W}asserstein distributionally robust optimization, {\em ESAIM: Control,
  Optimisation and Calculus of Variations} {\bf 29}(31), 1--33.

\bibitem[Bach(2013)Bach]{bach2013learning}
F.~Bach  (2013), Learning with submodular functions: A convex optimization
  perspective, {\em Foundations and Trends in Machine Learning} {\bf 6}(2-3),
  145--373.

\bibitem[Bach(2019)Bach]{bach2019submodular}
F.~Bach  (2019), Submodular functions: {F}rom discrete to continuous domains,
  {\em Mathematical Programming} {\bf 175}(1-2), 419--459.

\bibitem[Bai {\em et~al.}(2024)Bai, He, Jiang and Obloj]{bai2024wasserstein}
X.~Bai, G.~He, Y.~Jiang and J.~Obloj  (2024), {W}asserstein distributional
  robustness of neural networks, in {\em Advances in Neural Information
  Processing Systems}, pp.~26322--26347.

\bibitem[Bai {\em et~al.}(2023)Bai, Lam and Zhang]{bai2023distributionally}
Y.~Bai, H.~Lam and X.~Zhang  (2023), A distributionally robust optimization
  framework for extreme event estimation, {\em arXiv:2301.01360}.

\bibitem[Baire(1905)Baire]{baire1905semicontinuity}
R.~Baire  (1905), {\em Le{\c{c}}ons sur les Fonctions Discontinues},
  {Gauthier}-{Villars}.

\bibitem[Banach(1938)Banach]{Banach1938}
S.~Banach  (1938), {\"Uber homogene Polynome in ($L^2$)}, {\em Studia
  Mathematica} {\bf 7}(1), 36--44.

\bibitem[Bandi and Bertsimas(2014)Bandi and Bertsimas]{bandi2014optimal}
C.~Bandi and D.~Bertsimas  (2014), Optimal design for multi-item auctions: A
  robust optimization approach, {\em Mathematics of Operations Research} {\bf
  39}(4), 1012--1038.

\bibitem[Bartl {\em et~al.}(2021)Bartl, Drapeau, Obl{\'o}j and
  Wiesel]{bartl2020robust}
D.~Bartl, S.~Drapeau, J.~Obl{\'o}j and J.~Wiesel  (2021), Sensitivity analysis
  of {W}asserstein distributionally robust optimization problems, {\em
  Proceedings of the Royal Society A} {\bf 477}(2256), 20210176.

\bibitem[Ba{\c{s}}ar(1977)Ba{\c{s}}ar]{bacsar1977optimum}
T.~Ba{\c{s}}ar  (1977), Optimum {F}isherian information for multivariate
  distributions, {\em The Annals of Statistics} {\bf 5}(6), 1240--1244.

\bibitem[Ba{\c{s}}ar(1983)Ba{\c{s}}ar]{bacsar1983gaussian}
T.~Ba{\c{s}}ar  (1983), The {G}aussian test channel with an intelligent jammer,
  {\em IEEE Transactions on Information Theory} {\bf 29}(1), 152--157.

\bibitem[Ba{\c{s}}ar and Ba{\c{s}}ar(1984)Ba{\c{s}}ar and
  Ba{\c{s}}ar]{bacsar1984bandwidth}
T.~Ba{\c{s}}ar and T.~{\"U}. Ba{\c{s}}ar  (1984), A bandwidth expanding scheme
  for communication channels with noiseless feedback in the presence of unknown
  jamming noise, {\em Journal of the Franklin Institute} {\bf 317}(2), 73--88.

\bibitem[Ba{\c{s}}ar and Bernhard(1995)Ba{\c{s}}ar and Bernhard]{bacsar1995h}
T.~Ba{\c{s}}ar and P.~Bernhard  (1995), {\em $\mathcal H_\infty$-optimal
  Control and Related Minimax Design Problems: A Dynamic Game Approach},
  Springer.

\bibitem[Ba{\c{s}}ar and Max(1973)Ba{\c{s}}ar and Max]{bacsar1973multistage}
T.~Ba{\c{s}}ar and M.~Max  (1973), A multistage pursuit-evasion game that
  admits a {G}aussian random process as a maximin control policy, {\em
  Stochastics} {\bf 1}(1-4), 25--69.

\bibitem[Ba{\c{s}}ar and Mintz(1972)Ba{\c{s}}ar and Mintz]{bacsar1972minimax}
T.~Ba{\c{s}}ar and M.~Mintz  (1972), Minimax terminal state estimation for
  linear plants with unknown forcing functions, {\em International Journal of
  Control} {\bf 16}(1), 49--69.

\bibitem[Ba{\c{s}}ar and Mintz(1973)Ba{\c{s}}ar and Mintz]{bacsar1973minimax}
T.~Ba{\c{s}}ar and M.~Mintz  (1973), On a minimax estimate for the mean of a
  normal random vector under a generalized quadratic loss function, {\em The
  Annals of Statistics} {\bf 1}(1), 127--134.

\bibitem[Ba{\c{s}}ar and Wu(1985)Ba{\c{s}}ar and Wu]{bacsar1985complete}
T.~Ba{\c{s}}ar and Y.~W. Wu  (1985), A complete characterization of minimax and
  maximin encoder-decoder policies for communication channels with incomplete
  statistical description, {\em IEEE Transactions on Information Theory} {\bf
  31}(4), 482--489.

\bibitem[Ba{\c{s}}ar and Wu(1986)Ba{\c{s}}ar and Wu]{bacsar1986solutions}
T.~Ba{\c{s}}ar and Y.~W. Wu  (1986), Solutions to a class of minimax decision
  problems arising in communication systems, {\em Journal of Optimization
  Theory and Applications} {\bf 51}(3), 375--404.

\bibitem[Ba{\c{s}}ar and Ba{\c{s}}ar(1982)Ba{\c{s}}ar and
  Ba{\c{s}}ar]{bacsar1982optimum}
T.~{\"U}. Ba{\c{s}}ar and T.~Ba{\c{s}}ar  (1982), Optimum coding and decoding
  schemes for the transmission of a stochastic process over a continuous-time
  stochastic channel with partially unknown statisticst, {\em Stochastics} {\bf
  8}(3), 213--237.

\bibitem[Bayrak {\em et~al.}(2025)Bayrak, Ko{\c{c}}yi{\u{g}}it, Kuhn and
  P{\i}nar]{bayrak2022distributionally}
H.~I. Bayrak, {\c{C}}.~Ko{\c{c}}yi{\u{g}}it, D.~Kuhn and M.~C. P{\i}nar
  (2025), Distributionally robust optimal allocation with costly verification,
  {\em Operations Research (in press)}.

\bibitem[Bayraksan and Love(2015)Bayraksan and Love]{bayraksan2015data}
G.~Bayraksan and D.~K. Love  (2015), Data-driven stochastic programming using
  phi-divergences, {\em {INFORMS} Tutorials in Operations Research} pp.~1--19.

\bibitem[Beale(1955)Beale]{beale1955minimizing}
E.~M.~L. Beale  (1955), On minimizing a convex function subject to linear
  inequalities, {\em Journal of the Royal Statistical Society: Series B} {\bf
  17}(2), 173--184.

\bibitem[Beck and Ben-Tal(2009)Beck and Ben-Tal]{beck2009duality}
A.~Beck and A.~Ben-Tal  (2009), Duality in robust optimization: {P}rimal worst
  equals dual best, {\em Operations Research Letters} {\bf 37}(1), 1--6.

\bibitem[Belbasi {\em et~al.}(2023)Belbasi, Selvi and Wiesemann]{belbasi2023s}
R.~Belbasi, A.~Selvi and W.~Wiesemann  (2023), It's all in the mix:
  {W}asserstein machine learning with mixed features, {\em arXiv:2312.12230}.

\bibitem[Ben-Tal and Hochman(1972)Ben-Tal and Hochman]{ben-talhochman:72}
A.~Ben-Tal and E.~Hochman  (1972), More bounds on the expectation of a convex
  function of a random variable, {\em Journal of Applied Probability} {\bf
  9}(4), 803--812.

\bibitem[Ben-Tal and Nemirovski(1998)Ben-Tal and Nemirovski]{ben1998robust}
A.~Ben-Tal and A.~Nemirovski  (1998), Robust convex optimization, {\em
  Mathematics of Operations Research} {\bf 23}(4), 769--805.

\bibitem[Ben-Tal and Nemirovski(1999{\em a})Ben-Tal and
  Nemirovski]{ben1999robust}
A.~Ben-Tal and A.~Nemirovski  (1999{\em a}), Robust solutions of uncertain
  linear programs, {\em Operations Research Letters} {\bf 25}(1), 1--13.

\bibitem[Ben-Tal and Nemirovski(1999{\em b})Ben-Tal and
  Nemirovski]{ben-tal1999robust}
A.~Ben-Tal and A.~Nemirovski  (1999{\em b}), Robust truss topology design via
  semidefinite programming, {\em SIAM Journal on Optimization} {\bf 7}(4),
  991--1016.

\bibitem[Ben-Tal and Nemirovski(2000)Ben-Tal and Nemirovski]{ben2000robust}
A.~Ben-Tal and A.~Nemirovski  (2000), Robust solutions of linear programming
  problems contaminated with uncertain data, {\em Mathematical Programming}
  {\bf 88}(4), 411--424.

\bibitem[Ben-Tal and Nemirovski(2001)Ben-Tal and Nemirovski]{ben2001lectures}
A.~Ben-Tal and A.~Nemirovski  (2001), {\em Lectures on Modern Convex
  Optimization: Analysis, Algorithms, and Engineering Applications}, SIAM.

\bibitem[Ben-Tal and Nemirovski(2002)Ben-Tal and Nemirovski]{ben2002robust}
A.~Ben-Tal and A.~Nemirovski  (2002), Robust optimization--methodology and
  applications, {\em Mathematical Programming} {\bf 92}(3), 453--480.

\bibitem[Ben-Tal and Teboulle(1986)Ben-Tal and Teboulle]{ben1986expected}
A.~Ben-Tal and M.~Teboulle  (1986), Expected utility, penalty functions, and
  duality in stochastic nonlinear programming, {\em Management Science} {\bf
  32}(11), 1445--1466.

\bibitem[Ben-Tal and Teboulle(2007)Ben-Tal and Teboulle]{ben2007old}
A.~Ben-Tal and M.~Teboulle  (2007), An old-new concept of convex risk measures:
  {T}he optimized certainty equivalent, {\em Mathematical Finance} {\bf 17}(3),
  449--476.

\bibitem[Ben-Tal {\em et~al.}(1991)Ben-Tal, Ben-Israel and
  Teboulle]{ben1991certainty}
A.~Ben-Tal, A.~Ben-Israel and M.~Teboulle  (1991), Certainty equivalents and
  information measures: duality and extremal principles, {\em Journal of
  Mathematical Analysis and Applications} {\bf 157}(1), 211--236.

\bibitem[Ben-Tal {\em et~al.}(2015{\em a})Ben-Tal, den Hertog and
  Vial]{ben2015deriving}
A.~Ben-Tal, D.~den Hertog and J.-P. Vial  (2015{\em a}), Deriving robust
  counterparts of nonlinear uncertain inequalities, {\em Mathematical
  Programming} {\bf 149}(1), 265--299.

\bibitem[Ben-Tal {\em et~al.}(2013)Ben-Tal, den Hertog, De~Waegenaere,
  Melenberg and Rennen]{ben2013robust}
A.~Ben-Tal, D.~den Hertog, A.~De~Waegenaere, B.~Melenberg and G.~Rennen
  (2013), Robust solutions of optimization problems affected by uncertain
  probabilities, {\em Management Science} {\bf 59}(2), 341--357.

\bibitem[Ben-Tal {\em et~al.}(2009)Ben-Tal, El~Ghaoui and
  Nemirovski]{ben2009robust}
A.~Ben-Tal, L.~El~Ghaoui and A.~Nemirovski  (2009), {\em Robust Optimization},
  Princeton University Press.

\bibitem[Ben-Tal {\em et~al.}(2015{\em b})Ben-Tal, Hazan, Koren and
  Mannor]{ben2015oracle}
A.~Ben-Tal, E.~Hazan, T.~Koren and S.~Mannor  (2015{\em b}), Oracle-based
  robust optimization via online learning, {\em Operations Research} {\bf
  63}(3), 628--638.

\bibitem[Bennouna and Van~Parys(2021)Bennouna and
  Van~Parys]{bennouna2021learning}
A.~Bennouna and B.~P. Van~Parys  (2021), Learning and decision-making with
  data: {O}ptimal formulations and phase transitions, {\em arXiv:2109.06911}.

\bibitem[Bennouna and Van~Parys(2023)Bennouna and
  Van~Parys]{bennouna2023holistic}
A.~Bennouna and B.~P. Van~Parys  (2023), Holistic robust data-driven decisions,
  {\em arXiv:2207.09560}.

\bibitem[Bennouna {\em et~al.}(2023)Bennouna, Lucas and
  Van~Parys]{bennouna2023certified}
A.~Bennouna, R.~Lucas and B.~P. Van~Parys  (2023), Certified robust neural
  networks: {G}eneralization and corruption resistance, in {\em International
  Conference on Machine Learning}, pp.~2092--2112.

\bibitem[Berge(1963)Berge]{berge1963topological}
C.~Berge  (1963), {\em Topological Spaces: Including a Treatment of
  Multi-Valued Functions, Vector Spaces, and Convexity}, Courier Corporation.

\bibitem[Bergemann and Schlag(2008)Bergemann and Schlag]{bergemann2008pricing}
D.~Bergemann and K.~H. Schlag  (2008), Pricing without priors, {\em Journal of
  the European Economic Association} {\bf 6}(2-3), 560--569.

\bibitem[Bernstein(2009)Bernstein]{bernstein2009matrix}
D.~S. Bernstein  (2009), {\em Matrix Mathematics: Theory, Facts, and Formulas},
  Princeton University Press.

\bibitem[Bertsimas and den Hertog(2022)Bertsimas and den
  Hertog]{bertsimas2022robust}
D.~Bertsimas and D.~den Hertog  (2022), {\em Robust and Adaptive Optimization},
  Dynamic Ideas.

\bibitem[Bertsimas and Popescu(2002)Bertsimas and
  Popescu]{bertsimas2002relation}
D.~Bertsimas and I.~Popescu  (2002), On the relation between option and stock
  prices: {A} convex optimization approach, {\em Operations Research} {\bf
  50}(2), 358--374.

\bibitem[Bertsimas and Popescu(2005)Bertsimas and
  Popescu]{bertsimas2005optimal}
D.~Bertsimas and I.~Popescu  (2005), Optimal inequalities in probability
  theory: {A} convex optimization approach, {\em SIAM Journal on Optimization}
  {\bf 15}(3), 780--804.

\bibitem[Bertsimas and Sethuraman(2000)Bertsimas and
  Sethuraman]{bertsimas2000moment}
D.~Bertsimas and J.~Sethuraman  (2000), Moment problems and semidefinite
  optimization, in {\em Handbook of Semidefinite Programming: Theory,
  Algorithms, and Applications} (H.~Wolkowicz, R.~Saigal and L.~Vandenberghe,
  eds), Springer, pp.~469--509.

\bibitem[Bertsimas and Sim(2004)Bertsimas and Sim]{bertsimas2004price}
D.~Bertsimas and M.~Sim  (2004), The price of robustness, {\em Operations
  Research} {\bf 52}(1), 35--53.

\bibitem[Bertsimas and Van~Parys(2022)Bertsimas and
  Van~Parys]{bertsimas2022bootstrap}
D.~Bertsimas and B.~P. Van~Parys  (2022), Bootstrap robust prescriptive
  analytics, {\em Mathematical Programming} {\bf 195}(1), 39--78.

\bibitem[Bertsimas {\em et~al.}(2011)Bertsimas, Brown and
  Caramanis]{bertsimas2011theory}
D.~Bertsimas, D.~B. Brown and C.~Caramanis  (2011), Theory and applications of
  robust optimization, {\em {SIAM} Review} {\bf 53}(3), 464--501.

\bibitem[Bertsimas {\em et~al.}(2021)Bertsimas, den Hertog and
  Pauphilet]{bertsimas2021probabilistic}
D.~Bertsimas, D.~den Hertog and J.~Pauphilet  (2021), Probabilistic guarantees
  in robust optimization, {\em SIAM Journal on Optimization} {\bf 31}(4),
  2893--2920.

\bibitem[Bertsimas {\em et~al.}(2010)Bertsimas, Doan, Natarajan and
  Teo]{bertsimas2010models}
D.~Bertsimas, X.~V. Doan, K.~Natarajan and C.-P. Teo  (2010), Models for
  minimax stochastic linear optimization problems with risk aversion, {\em
  Mathematics of Operations Research} {\bf 35}(3), 580--602.

\bibitem[Bertsimas {\em et~al.}(2018{\em a})Bertsimas, Gupta and
  Kallus]{bertsimas2018data}
D.~Bertsimas, V.~Gupta and N.~Kallus  (2018{\em a}), Data-driven robust
  optimization, {\em Mathematical Programming} {\bf 167}(2), 235--292.

\bibitem[Bertsimas {\em et~al.}(2018{\em b})Bertsimas, Gupta and
  Kallus]{bertsimas2018robust}
D.~Bertsimas, V.~Gupta and N.~Kallus  (2018{\em b}), Robust sample average
  approximation, {\em Mathematical Programming} {\bf 171}(1-2), 217--282.

\bibitem[Bertsimas {\em et~al.}(2004)Bertsimas, Natarajan and
  Teo]{bertsimas2004probabilistic}
D.~Bertsimas, K.~Natarajan and C.-P. Teo  (2004), Probabilistic combinatorial
  optimization: {M}oments, semidefinite programming, and asymptotic bounds,
  {\em SIAM Journal on Optimization} {\bf 15}(1), 185--209.

\bibitem[Bertsimas {\em et~al.}(2006{\em a})Bertsimas, Natarajan and
  Teo]{bertsimas2006persistence}
D.~Bertsimas, K.~Natarajan and C.-P. Teo  (2006{\em a}), Persistence in
  discrete optimization under data uncertainty, {\em Mathematical Programming}
  {\bf 108}(2-3), 251--274.

\bibitem[Bertsimas {\em et~al.}(2006{\em b})Bertsimas, Natarajan and
  Teo]{bertsimas2006tight}
D.~Bertsimas, K.~Natarajan and C.-P. Teo  (2006{\em b}), Tight bounds on
  expected order statistics, {\em Probability in the Engineering and
  Informational Sciences} {\bf 20}(4), 667--686.

\bibitem[Bertsimas {\em et~al.}(2022)Bertsimas, Shtern and
  Sturt]{bertsimas2020two}
D.~Bertsimas, S.~Shtern and B.~Sturt  (2022), Two-stage sample robust
  optimization, {\em Operations Research} {\bf 70}(1), 624--640.

\bibitem[Bertsimas {\em et~al.}(2023)Bertsimas, Shtern and
  Sturt]{bertsimas2023data}
D.~Bertsimas, S.~Shtern and B.~Sturt  (2023), A data-driven approach to
  multistage stochastic linear optimization, {\em Management Science} {\bf
  69}(1), 51--74.

\bibitem[Bhatia {\em et~al.}(2018)Bhatia, Jain and Lim]{bhatia2018strong}
R.~Bhatia, T.~Jain and Y.~Lim  (2018), Strong convexity of sandwiched entropies
  and related optimization problems, {\em Reviews in Mathematical Physics} {\bf
  30}(9), 1850014.

\bibitem[Bhatia {\em et~al.}(2019)Bhatia, Jain and Lim]{bhatia2019bures}
R.~Bhatia, T.~Jain and Y.~Lim  (2019), On the {B}ures--{W}asserstein distance
  between positive definite matrices, {\em Expositiones Mathematicae} {\bf
  37}(2), 165--191.

\bibitem[Bhattacharyya(2004)Bhattacharyya]{bhattacharyya2004second}
C.~Bhattacharyya  (2004), Second order cone programming formulations for
  feature selection, {\em Journal of Machine Learning Research} {\bf 5},
  1417--1433.

\bibitem[Billingsley(2013)Billingsley]{billingsley2013convergence}
P.~Billingsley  (2013), {\em Convergence of Probability Measures}, Wiley.

\bibitem[Birge and Wets(1986)Birge and Wets]{birgewets:86}
J.~Birge and R.-B. Wets  (1986), Designing approximation schemes for stochastic
  optimization problems, in particular for stochastic programs with recourse,
  {\em Mathematical Programming Study} {\bf 27}, 54--102.

\bibitem[Birge and Louveaux(2011)Birge and Louveaux]{birge2011introduction}
J.~R. Birge and F.~Louveaux  (2011), {\em Introduction to Stochastic
  Programming}, Springer.

\bibitem[Bishop(2006)Bishop]{bishop2006pattern}
C.~M. Bishop  (2006), {\em Pattern Recognition and Machine Learning}, Springer.

\bibitem[Blanchet and Kang(2020)Blanchet and Kang]{blanchet2020semi}
J.~Blanchet and Y.~Kang  (2020), Semi-supervised learning based on
  distributionally robust optimization, in {\em Data Analysis and Applications
  3} (A.~Makrides, A.~Karagrigoriou and C.~H. Skiadas, eds), Wiley, pp.~1--33.

\bibitem[Blanchet and Kang(2021)Blanchet and Kang]{blanchet2019sample}
J.~Blanchet and Y.~Kang  (2021), Sample out-of-sample inference based on
  {W}asserstein distance, {\em Operations Research} {\bf 69}(3), 985--1013.

\bibitem[Blanchet and Murthy(2019)Blanchet and Murthy]{blanchet2019quantifying}
J.~Blanchet and K.~Murthy  (2019), Quantifying distributional model risk via
  optimal transport, {\em Mathematics of Operations Research} {\bf 44}(2),
  565--600.

\bibitem[Blanchet and Shapiro(2023)Blanchet and
  Shapiro]{blanchet2023statistical}
J.~Blanchet and A.~Shapiro  (2023), Statistical limit theorems in
  distributionally robust optimization, in {\em Winter Simulation Conference},
  pp.~31--45.

\bibitem[Blanchet {\em et~al.}(2022{\em a})Blanchet, Chen and
  Zhou]{blanchet2018distributionally}
J.~Blanchet, L.~Chen and X.~Y. Zhou  (2022{\em a}), Distributionally robust
  mean-variance portfolio selection with {W}asserstein distances, {\em
  Management Science} {\bf 68}(9), 6382--6410.

\bibitem[Blanchet {\em et~al.}(2019{\em a})Blanchet, Glynn, Yan and
  Zhou]{blanchet2019multivariate}
J.~Blanchet, P.~W. Glynn, J.~Yan and Z.~Zhou  (2019{\em a}), Multivariate
  distributionally robust convex regression under absolute error loss, in {\em
  Advances in Neural Information Processing Systems}, pp.~11817--11826.

\bibitem[Blanchet {\em et~al.}(2020)Blanchet, He and
  Murthy]{blanchet2020distributionally}
J.~Blanchet, F.~He and K.~Murthy  (2020), On distributionally robust extreme
  value analysis, {\em Extremes} {\bf 23}(2), 317--347.

\bibitem[Blanchet {\em et~al.}(2019{\em b})Blanchet, Kang and
  Murthy]{blanchet2019robust}
J.~Blanchet, Y.~Kang and K.~Murthy  (2019{\em b}), Robust {W}asserstein profile
  inference and applications to machine learning, {\em Journal of Applied
  Probability} {\bf 56}(3), 830--857.

\bibitem[Blanchet {\em et~al.}(2023)Blanchet, Kuhn, Li and
  Ta{\c{s}}kesen]{blanchet2023unifying}
J.~Blanchet, D.~Kuhn, J.~Li and B.~Ta{\c{s}}kesen  (2023), Unifying
  distributionally robust optimization via optimal transport theory, {\em
  arXiv:2308.05414}.

\bibitem[Blanchet {\em et~al.}(2024{\em a})Blanchet, Lam, Liu and
  Wang]{blanchet2020convolution}
J.~Blanchet, H.~Lam, Y.~Liu and R.~Wang  (2024{\em a}), Convolution bounds on
  quantile aggregation, {\em Operations Research (Forthcoming)}.

\bibitem[Blanchet {\em et~al.}(2024{\em b})Blanchet, Li, Lin and
  Zhang]{jose2024favor}
J.~Blanchet, J.~Li, S.~Lin and X.~Zhang  (2024{\em b}), Distributionally robust
  optimization and robust statistics, {\em arxiv:2401.14655}.

\bibitem[Blanchet {\em et~al.}(2021)Blanchet, Murthy and
  Nguyen]{blanchet2021statistical}
J.~Blanchet, K.~Murthy and V.~A. Nguyen  (2021), Statistical analysis of
  {W}asserstein distributionally robust estimators, {\em {INFORMS} Tutorials in
  Operations Research} pp.~227--254.

\bibitem[Blanchet {\em et~al.}(2022{\em b})Blanchet, Murthy and
  Si]{blanchet2019confidence}
J.~Blanchet, K.~Murthy and N.~Si  (2022{\em b}), Confidence regions in
  {W}asserstein distributionally robust estimation, {\em Biometrika} {\bf
  109}(2), 295--315.

\bibitem[Blanchet {\em et~al.}(2022{\em c})Blanchet, Murthy and
  Zhang]{blanchet2018optimal}
J.~Blanchet, K.~Murthy and F.~Zhang  (2022{\em c}), Optimal transport-based
  distributionally robust optimization: Structural properties and iterative
  schemes, {\em Mathematics of Operations Research} {\bf 47}(2), 1500--1529.

\bibitem[Blankenstein {\em et~al.}(2016)Blankenstein, Crone, van~den Bos and
  van Duijvenvoorde]{blankenstein2016adolescent}
N.~E. Blankenstein, E.~A. Crone, W.~van~den Bos and A.~C.~K. van Duijvenvoorde
  (2016), Adolescents display distinctive tolerance to ambiguity and to
  uncertainty during risky decision making, {\em Developmental Neuropsychology}
  {\bf 41}(1--2), 77--92.

\bibitem[Boissard and Le~Gouic(2014)Boissard and Le~Gouic]{boissard2014mean}
E.~Boissard and T.~Le~Gouic  (2014), On the mean speed of convergence of
  empirical and occupation measures in {W}asserstein distance, {\em Annales de
  l'IHP Probabilit{\'e}s et Statistiques} {\bf 50}(2), 539--563.

\bibitem[Bolley {\em et~al.}(2007)Bolley, Guillin and
  Villani]{bolley2007quantitative}
F.~Bolley, A.~Guillin and C.~Villani  (2007), Quantitative concentration
  inequalities for empirical measures on non-compact spaces, {\em Probability
  Theory and Related Fields} {\bf 137}(3-4), 541--593.

\bibitem[Boole(1854)Boole]{boole1854investigation}
G.~Boole  (1854), {\em An Investigation of the Laws of Thought}, Walton and
  Maberly.

\bibitem[Bose and Daripa(2009)Bose and Daripa]{bose2009dynamic}
S.~Bose and A.~Daripa  (2009), A dynamic mechanism and surplus extraction under
  ambiguity, {\em Journal of Economic theory} {\bf 144}(5), 2084--2114.

\bibitem[Boskos {\em et~al.}(2020)Boskos, Cort{\'e}s and
  Mart{\'\i}nez]{boskos2020data}
D.~Boskos, J.~Cort{\'e}s and S.~Mart{\'\i}nez  (2020), Data-driven ambiguity
  sets with probabilistic guarantees for dynamic processes, {\em IEEE
  Transactions on Automatic Control} {\bf 66}(7), 2991--3006.

\bibitem[Bossaerts {\em et~al.}(2010)Bossaerts, Ghirardato, Guarnaschelli and
  Zame]{bossaerts2010ambiguity}
P.~Bossaerts, P.~Ghirardato, S.~Guarnaschelli and W.~R. Zame  (2010), Ambiguity
  in asset markets: Theory and experiment, {\em The Review of Financial
  Studies} {\bf 23}(4), 1325--1359.

\bibitem[Boucheron {\em et~al.}(2013)Boucheron, Lugosi and
  Massart]{boucheron2013concentration}
S.~Boucheron, G.~Lugosi and P.~Massart  (2013), {\em Concentration
  Inequalities: A Nonasymptotic Theory of Independence}, Oxford University
  Press.

\bibitem[Bousquet {\em et~al.}(2004)Bousquet, Boucheron and
  Lugosi]{bousquet2004introduction}
O.~Bousquet, S.~Boucheron and G.~Lugosi  (2004), Introduction to statistical
  learning theory, in {\em Advanced Lectures on Machine Learning} (O.~Bousquet,
  U.~von Luxburg and G.~R{\"a}tsch, eds), Springer, pp.~169--207.

\bibitem[Box(1953)Box]{box1953non}
G.~E. Box  (1953), Non-normality and tests on variances, {\em Biometrika} {\bf
  40}(3-4), 318--335.

\bibitem[Box(1979)Box]{box1979robustness}
G.~E. Box  (1979), Robustness in the strategy of scientific model building, in
  {\em Robustness in statistics} (R.~L. Launer and G.~N. Wilkinson, eds),
  Academic Press, pp.~201--236.

\bibitem[Brenier(1991)Brenier]{brenier1991polar}
Y.~Brenier  (1991), Polar factorization and monotone rearrangement of
  vector-valued functions, {\em Communications on Pure and Applied Mathematics}
  {\bf 44}(4), 375--417.

\bibitem[Brezis(2011)Brezis]{brezis2011functional}
H.~Brezis  (2011), {\em Functional Analysis, Sobolev Spaces and Partial
  Differential Equations}, Springer.

\bibitem[Brugman {\em et~al.}(2022)Brugman, Van~Leeuwaarden and
  Stegehuis]{brugman2022sharpest}
J.~Brugman, J.~S. Van~Leeuwaarden and C.~Stegehuis  (2022), Sharpest possible
  clustering bounds using robust random graph analysis, {\em Physical Review E}
  {\bf 106}(6), 064311.

\bibitem[Buckert {\em et~al.}(2014)Buckert, Schwieren, Kudielka and
  Fiebach]{buckert14stress}
M.~Buckert, C.~Schwieren, B.~M. Kudielka and C.~J. Fiebach  (2014), Acute
  stress affects risk taking but not ambiguity aversion, {\em Frontiers in
  Neuroscience} {\bf 8}, 82.

\bibitem[Bui {\em et~al.}(2022)Bui, Nguyen and Nguyen]{bui2021counterfactual}
N.~Bui, D.~Nguyen and V.~A. Nguyen  (2022), Counterfactual plans under
  distributional ambiguity, in {\em International Conference on Learning
  Representations}.

\bibitem[Bungert {\em et~al.}(2023)Bungert, Garc{\'\i}a~Trillos and
  Murray]{bungert2023geometry}
L.~Bungert, N.~Garc{\'\i}a~Trillos and R.~Murray  (2023), The geometry of
  adversarial training in binary classification, {\em Information and
  Inference: A Journal of the IMA} {\bf 12}(2), 921--968.

\bibitem[Bungert {\em et~al.}(2024)Bungert, Laux and Stinson]{bungert2024mean}
L.~Bungert, T.~Laux and K.~Stinson  (2024), A mean curvature flow arising in
  adversarial training, {\em Journal de Mathématiques Pures et Appliquées}
  {\bf 192}, 103625.

\bibitem[Cabantous(2007)Cabantous]{cabantous2007ambiguity}
L.~Cabantous  (2007), Ambiguity aversion in the field of insurance: Insurers’
  attitude to imprecise and conflicting probability estimates, {\em Theory and
  Decision} {\bf 62}(3), 219--240.

\bibitem[Cai {\em et~al.}(2023)Cai, Li and Mao]{cai2023distributionally}
J.~Cai, J.~Y.-M. Li and T.~Mao  (2023), Distributionally robust optimization
  under distorted expectations, {\em Operations Research (Forthcoming)}.

\bibitem[Calafiore(2007)Calafiore]{calafiore2007ambiguous}
G.~C. Calafiore  (2007), Ambiguous risk measures and optimal robust portfolios,
  {\em SIAM Journal on Optimization} {\bf 18}(3), 853--877.

\bibitem[Calafiore and Campi(2005)Calafiore and Campi]{calafiore2005uncertain}
G.~C. Calafiore and M.~C. Campi  (2005), Uncertain convex programs:
  {R}andomized solutions and confidence levels, {\em Mathematical Programming}
  {\bf 102}(1), 25--46.

\bibitem[Calafiore and Campi(2006)Calafiore and Campi]{calafiore2006scenario}
G.~C. Calafiore and M.~C. Campi  (2006), The scenario approach to robust
  control design, {\em IEEE Transactions on Automatic Control} {\bf 51}(5),
  742--753.

\bibitem[Calafiore and El~Ghaoui(2006)Calafiore and
  El~Ghaoui]{calafiore2006distributionally}
G.~C. Calafiore and L.~El~Ghaoui  (2006), On distributionally robust
  chance-constrained linear programs, {\em Journal of Optimization Theory and
  Applications} {\bf 130}(1), 1--22.

\bibitem[Calafiore {\em et~al.}(2011)Calafiore, Dabbene and
  Tempo]{calafiore2011research}
G.~C. Calafiore, F.~Dabbene and R.~Tempo  (2011), Research on probabilistic
  methods for control system design, {\em Automatica} {\bf 47}(7), 1279--1293.

\bibitem[Campi and Car{\'e}(2013)Campi and Car{\'e}]{campi2013random}
M.~C. Campi and A.~Car{\'e}  (2013), Random convex programs with
  {$L_1$}-regularization: sparsity and generalization, {\em SIAM Journal on
  Control and Optimization} {\bf 51}(5), 3532--3557.

\bibitem[Campi and Garatti(2008)Campi and Garatti]{campi2008exact}
M.~C. Campi and S.~Garatti  (2008), The exact feasibility of randomized
  solutions of uncertain convex programs, {\em SIAM Journal on Optimization}
  {\bf 19}(3), 1211--1230.

\bibitem[Campi and Garatti(2011)Campi and Garatti]{campi2011sampling}
M.~C. Campi and S.~Garatti  (2011), A sampling-and-discarding approach to
  chance-constrained optimization: {F}easibility and optimality, {\em Journal
  of Optimization Theory and Applications} {\bf 148}(2), 257--280.

\bibitem[Campi and Garatti(2018)Campi and Garatti]{campi2018wait}
M.~C. Campi and S.~Garatti  (2018), Wait-and-judge scenario optimization, {\em
  Mathematical Programming} {\bf 167}(1), 155--189.

\bibitem[Car{\'e} {\em et~al.}(2014)Car{\'e}, Garatti and Campi]{care2014fast}
A.~Car{\'e}, S.~Garatti and M.~C. Campi  (2014), {FAST}—fast algorithm for
  the scenario technique, {\em Operations Research} {\bf 62}(3), 662--671.

\bibitem[Carmon and Hausler(2022)Carmon and
  Hausler]{carmon2022distributionally}
Y.~Carmon and D.~Hausler  (2022), Distributionally robust optimization via ball
  oracle acceleration, in {\em Advances in Neural Information Processing
  Systems}, pp.~35866--35879.

\bibitem[Carroll(2017)Carroll]{carroll2017robustness}
G.~Carroll  (2017), Robustness and separation in multidimensional screening,
  {\em Econometrica} {\bf 85}(2), 453--488.

\bibitem[Champion {\em et~al.}(2008)Champion, De~Pascale and
  Juutinen]{champion2008wasserstein}
T.~Champion, L.~De~Pascale and P.~Juutinen  (2008), The $\infty$-{W}asserstein
  distance: {L}ocal solutions and existence of optimal transport maps, {\em
  SIAM Journal on Mathematical Analysis} {\bf 40}(1), 1--20.

\bibitem[Chan {\em et~al.}(2024)Chan, Van~Parys and
  Bennouna]{chan2024distributional}
G.~Chan, B.~Van~Parys and A.~Bennouna  (2024), From distributional robustness
  to robust statistics: {A} confidence sets perspective, {\em
  arXiv:2410.14008}.

\bibitem[Chebyshev(1874)Chebyshev]{tchebichef1874valeurs}
P.~Chebyshev  (1874), Sur les valeurs limites des int{\'e}grales, {\em Journal
  de Math{\'e}matiques Pures et Appliqu{\'e}es} {\bf 19}, 157--160.

\bibitem[Chen and Sim(2024)Chen and Sim]{chen2024robust}
L.~Chen and M.~Sim  (2024), Robust {CARA} optimization, {\em Operations
  Research (Forthcoming)}.

\bibitem[Chen {\em et~al.}(2023)Chen, Fu, Si, Sim and Xiong]{chen2023robust}
L.~Chen, C.~Fu, F.~Si, M.~Sim and P.~Xiong  (2023), Robust optimization with
  moment-dispersion ambiguity, {\em SSRN preprint 4525224}.

\bibitem[Chen {\em et~al.}(2011)Chen, He and Zhang]{chen2011tight}
L.~Chen, S.~He and S.~Zhang  (2011), Tight bounds for some risk measures, with
  applications to robust portfolio selection, {\em Operations Research} {\bf
  59}(4), 847--865.

\bibitem[Chen {\em et~al.}(2022)Chen, Ma, Natarajan, Simchi-Levi and
  Yan]{chen2022distributionally}
L.~Chen, W.~Ma, K.~Natarajan, D.~Simchi-Levi and Z.~Yan  (2022),
  Distributionally robust linear and discrete optimization with marginals, {\em
  Operations Research} {\bf 70}(3), 1822--1834.

\bibitem[Chen {\em et~al.}(2020)Chen, Padmanabhan, Lim and
  Natarajan]{chen2020correlation}
L.~Chen, D.~Padmanabhan, C.~C. Lim and K.~Natarajan  (2020), Correlation robust
  influence maximization, in {\em Advances in Neural Information Processing
  Systems}, pp.~7078--7089.

\bibitem[Chen and Paschalidis(2018)Chen and Paschalidis]{chen2018robust}
R.~Chen and I.~C. Paschalidis  (2018), A robust learning approach for
  regression models based on distributionally robust optimization, {\em Journal
  of Machine Learning Research} {\bf 19}(1), 517--564.

\bibitem[Chen and Paschalidis(2019)Chen and Paschalidis]{chen2019selecting}
R.~Chen and I.~C. Paschalidis  (2019), Selecting optimal decisions via
  distributionally robust nearest-neighbor regression, in {\em Advances in
  Neural Information Processing Systems}, pp.~749--759.

\bibitem[Chen {\em et~al.}(2010)Chen, Sim, Sun and Teo]{chen2010cvar}
W.~Chen, M.~Sim, J.~Sun and C.-P. Teo  (2010), From {CVaR} to uncertainty set:
  {I}mplications in joint chance-constrained optimization, {\em Operations
  Research} {\bf 58}(2), 470--485.

\bibitem[Chen {\em et~al.}(2024{\em a})Chen, Hu and Wang]{chen2023screening}
Z.~Chen, Z.~Hu and R.~Wang  (2024{\em a}), Screening with limited information:
  {A} dual perspective, {\em Operations Research} {\bf 72}(4), 1487--1504.

\bibitem[Chen {\em et~al.}(2024{\em b})Chen, Kuhn and Wiesemann]{chen2018data}
Z.~Chen, D.~Kuhn and W.~Wiesemann  (2024{\em b}), Data-driven chance
  constrained programs over {W}asserstein balls, {\em Operations Research} {\bf
  72}(1), 410--424.

\bibitem[Chen {\em et~al.}(2019)Chen, Sim and Xu]{chen2019distributionally}
Z.~Chen, M.~Sim and H.~Xu  (2019), Distributionally robust optimization with
  infinitely constrained ambiguity sets, {\em Operations Research} {\bf 67}(5),
  1328--1344.

\bibitem[Cheng {\em et~al.}(2014)Cheng, Delage and
  Lisser]{cheng2014distributionally}
J.~Cheng, E.~Delage and A.~Lisser  (2014), Distributionally robust stochastic
  knapsack problem, {\em SIAM Journal on Optimization} {\bf 24}(3), 1485--1506.

\bibitem[Cherukuri and Cort{\'e}s(2019)Cherukuri and
  Cort{\'e}s]{cherukuri2020cooperative}
A.~Cherukuri and J.~Cort{\'e}s  (2019), Cooperative data-driven
  distributionally robust optimization, {\em IEEE Transactions on Automatic
  Control} {\bf 65}(10), 4400--4407.

\bibitem[Chizat(2022)Chizat]{chizat2022sparse}
L.~Chizat  (2022), Sparse optimization on measures with over-parameterized
  gradient descent, {\em Mathematical Programming} {\bf 194}(1), 487--532.

\bibitem[Chizat and Bach(2018)Chizat and Bach]{chizat2018global}
L.~Chizat and F.~Bach  (2018), On the global convergence of gradient descent
  for over-parameterized models using optimal transport, in {\em Advances in
  Neural Information Processing Systems}.

\bibitem[Cl{\'e}ment and Desch(2008)Cl{\'e}ment and Desch]{PhilippeClement2008}
P.~Cl{\'e}ment and W.~Desch  (2008), {W}asserstein metric and subordination,
  {\em Studia Mathematica} {\bf 189}(1), 35--52.

\bibitem[Coulson {\em et~al.}(2021)Coulson, Lygeros and
  D{\"o}rfler]{coulson2021distributionally}
J.~Coulson, J.~Lygeros and F.~D{\"o}rfler  (2021), Distributionally robust
  chance constrained data-enabled predictive control, {\em IEEE Transactions on
  Automatic Control} {\bf 67}(7), 3289--3304.

\bibitem[Cover and Thomas(2006)Cover and Thomas]{cover2006elements}
T.~Cover and J.~Thomas  (2006), {\em Elements of Information Theory}, Wiley.

\bibitem[Cram{\'e}r(1938)Cram{\'e}r]{cramer1938nouveau}
H.~Cram{\'e}r  (1938), Sur un nouveau th{\'e}oreme-limite de la th{\'e}orie des
  probabilit{\'e}s, {\em Actualit{\'e}s Scientifiques et Industrielles} {\bf
  736}, 5--23.

\bibitem[Cram{\'e}r(1946)Cram{\'e}r]{cramer1946mathematical}
H.~Cram{\'e}r  (1946), {\em Mathematical Methods of Statistics}, Princeton
  University Press.

\bibitem[Csisz{\'a}r(1963)Csisz{\'a}r]{csiszar1964informationstheoretische}
I.~Csisz{\'a}r  (1963), {Eine informationstheoretische Ungleichung und ihre
  Anwendung auf den Beweis der Ergodizit\"at von Markoffschen Ketten}, {\em
  Publications of the Mathematical Institute of the Hungarian Academy of
  Sciences} {\bf 8}, 85--108.

\bibitem[Csisz{\'a}r(1967)Csisz{\'a}r]{csiszar1967information}
I.~Csisz{\'a}r  (1967), Information-type measures of difference of probability
  distributions and indirect observation, {\em Studia Scientiarum
  Mathematicarum Hungarica} {\bf 2}, 229--318.

\bibitem[Dantzig(1955)Dantzig]{dantzig1955linear}
G.~B. Dantzig  (1955), Linear programming under uncertainty, {\em Management
  Science} {\bf 1}(3-4), 197--206.

\bibitem[Dantzig(1956)Dantzig]{dantzig1949programming}
G.~B. Dantzig  (1956), {\em The Simplex Method}, RAND Corporation.

\bibitem[Das {\em et~al.}(2021)Das, Dhara and Natarajan]{das2021heavy}
B.~Das, A.~Dhara and K.~Natarajan  (2021), On the heavy-tail behavior of the
  distributionally robust newsvendor, {\em Operations Research} {\bf 69}(4),
  1077--1099.

\bibitem[De~Farias and Van~Roy(2004)De~Farias and Van~Roy]{de2004constraint}
D.~P. De~Farias and B.~Van~Roy  (2004), On constraint sampling in the linear
  programming approach to approximate dynamic programming, {\em Mathematics of
  Operations Research} {\bf 29}(3), 462--478.

\bibitem[Delage and Iancu(2015)Delage and Iancu]{DI15:robust_multistage}
E.~Delage and D.~A. Iancu  (2015), Robust multistage decision making, {\em
  {INFORMS} Tutorials in Operations Research} pp.~20--46.

\bibitem[Delage and Ye(2010)Delage and Ye]{delage2010distributionally}
E.~Delage and Y.~Ye  (2010), Distributionally robust optimization under moment
  uncertainty with application to data-driven problems, {\em Operations
  Research} {\bf 58}(3), 595--612.

\bibitem[Delage {\em et~al.}(2019)Delage, Kuhn and
  Wiesemann]{delage2019diceisions}
E.~Delage, D.~Kuhn and W.~Wiesemann  (2019), ``{D}ice''-sion–making under
  uncertainty: {W}hen can a random decision reduce risk?, {\em Management
  Science} {\bf 65}(7), 3282--3301.

\bibitem[Delbaen(2002)Delbaen]{delbaen2002coherent}
F.~Delbaen  (2002), Coherent risk measures on general probability spaces, in
  {\em Advances in Finance and Stochastics: Essays in Honour of Dieter
  Sondermann} (K.~Sandmann and P.~J. Sch{\"o}nbucher, eds), Springer,
  pp.~1--37.

\bibitem[Dembo and Zeitouni(2009)Dembo and Zeitouni]{dembo2009large}
A.~Dembo and O.~Zeitouni  (2009), {\em Large Deviations Techniques and
  Applications}, Springer.

\bibitem[DeMiguel and Nogales(2009)DeMiguel and Nogales]{demiguel2009portfolio}
V.~DeMiguel and F.~J. Nogales  (2009), Portfolio selection with robust
  estimation, {\em Operations Research} {\bf 57}(3), 560--577.

\bibitem[DeMiguel {\em et~al.}(2009)DeMiguel, Garlappi and
  Uppal]{demiguel-1-over-n}
V.~DeMiguel, L.~Garlappi and R.~Uppal  (2009), Optimal versus naive
  diversification: How inefficient is the $1/n$ portfolio strategy?, {\em The
  Review of Financial Studies} {\bf 22}(5), 1915--1953.

\bibitem[Demontis {\em et~al.}(2019)Demontis, Melis, Pintor, Jagielski, Biggio,
  Oprea, Nita-Rotaru and Roli]{demontis2019adversarial}
A.~Demontis, M.~Melis, M.~Pintor, M.~Jagielski, B.~Biggio, A.~Oprea,
  C.~Nita-Rotaru and F.~Roli  (2019), Why do adversarial attacks transfer?
  {E}xplaining transferability of evasion and poisoning attacks, in {\em USENIX
  Security Symposium}, pp.~321--338.

\bibitem[Dereich {\em et~al.}(2013)Dereich, Scheutzow and
  Schottstedt]{dereich2013constructive}
S.~Dereich, M.~Scheutzow and R.~Schottstedt  (2013), Constructive quantization:
  Approximation by empirical measures, {\em Annales de l'IHP Probabilit{\'e}s
  et Statistiques} {\bf 49}(4), 1183--1203.

\bibitem[Dharmadhikari and Joag-Dev(1988)Dharmadhikari and
  Joag-Dev]{dharmadhikari1988unimodality}
S.~Dharmadhikari and K.~Joag-Dev  (1988), {\em Unimodality, Convexity, and
  Applications}, Elsevier.

\bibitem[Diakonikolas and Kane(2023)Diakonikolas and
  Kane]{diakonikolas2023algorithmic}
I.~Diakonikolas and D.~M. Kane  (2023), {\em Algorithmic High-Dimensional
  Robust Statistics}, Cambridge University Press.

\bibitem[Diakonikolas {\em et~al.}(2019)Diakonikolas, Kamath, Kane, Li, Moitra
  and Stewart]{diakonikolas2019robust}
I.~Diakonikolas, G.~Kamath, D.~Kane, J.~Li, A.~Moitra and A.~Stewart  (2019),
  Robust estimators in high-dimensions without the computational
  intractability, {\em SIAM Journal on Computing} {\bf 48}(2), 742--864.

\bibitem[Diao {\em et~al.}(2023)Diao, Balasubramanian, Chewi and
  Salim]{diao2023forward}
M.~Z. Diao, K.~Balasubramanian, S.~Chewi and A.~Salim  (2023), Forward-backward
  {G}aussian variational inference via {JKO} in the {B}ures-{W}asserstein
  space, in {\em International Conference on Machine Learning}, pp.~7960--7991.

\bibitem[Dimmock {\em et~al.}(2016)Dimmock, Kouwenberg and
  Wakker]{dimmock2016ambiguity}
S.~G. Dimmock, R.~Kouwenberg and P.~P. Wakker  (2016), Ambiguity attitudes in a
  large representative sample, {\em Management Science} {\bf 62}(5),
  1363--1380.

\bibitem[Doan and Natarajan(2012)Doan and Natarajan]{doan2012complexity}
X.~V. Doan and K.~Natarajan  (2012), On the complexity of nonoverlapping
  multivariate marginal bounds for probabilistic combinatorial optimization
  problems, {\em Operations Research} {\bf 60}(1), 138--149.

\bibitem[Doan {\em et~al.}(2015)Doan, Li and Natarajan]{doan2015robustness}
X.~V. Doan, X.~Li and K.~Natarajan  (2015), Robustness to dependency in
  portfolio optimization using overlapping marginals, {\em Operations Research}
  {\bf 63}(6), 1468--1488.

\bibitem[Dobri{\'c} and Yukich(1995)Dobri{\'c} and
  Yukich]{dobric1995asymptotics}
V.~Dobri{\'c} and J.~E. Yukich  (1995), Asymptotics for transportation cost in
  high dimensions, {\em Journal of Theoretical Probability} {\bf 8}(1),
  97--118.

\bibitem[Dokov and Morton(2005)Dokov and Morton]{dokov2005second}
S.~P. Dokov and D.~P. Morton  (2005), Second-order lower bounds on the
  expectation of a convex function, {\em Mathematics of Operations Research}
  {\bf 30}(3), 662--677.

\bibitem[Donoho and Liu(1988)Donoho and Liu]{donoho1988automatic}
D.~L. Donoho and R.~C. Liu  (1988), The "automatic" robustness of minimum
  distance functionals, {\em The Annals of Statistics} {\bf 16}(2), 552--586.

\bibitem[Donsker and Varadhan(1983)Donsker and Varadhan]{donsker1983asymptotic}
M.~D. Donsker and S.~S. Varadhan  (1983), Asymptotic evaluation of certain
  {M}arkov process expectations for large time. {IV}, {\em Communications on
  Pure and Applied Mathematics} {\bf 36}(2), 183--212.

\bibitem[Dowson and Landau(1982)Dowson and Landau]{dowson1982frechet}
D.~Dowson and B.~Landau  (1982), The {F}r{\'e}chet distance between
  multivariate normal distributions, {\em Journal of Multivariate Analysis}
  {\bf 12}(3), 450--455.

\bibitem[Doyle {\em et~al.}(1989)Doyle, Glover, Khargonekar and
  Francis]{doyle1989robust}
J.~C. Doyle, K.~Glover, P.~Khargonekar and B.~Francis  (1989), Robust control
  of time-delay systems, {\em IEEE Transactions on Automatic Control} {\bf
  34}(6), 674--683.

\bibitem[Duchi and Namkoong(2019)Duchi and Namkoong]{duchi2019variance}
J.~C. Duchi and H.~Namkoong  (2019), Variance-based regularization with convex
  objectives, {\em Journal of Machine Learning Research} {\bf 20}(68), 1--55.

\bibitem[Duchi and Namkoong(2021)Duchi and Namkoong]{duchi2021learning}
J.~C. Duchi and H.~Namkoong  (2021), Learning models with uniform performance
  via distributionally robust optimization, {\em The Annals of Statistics} {\bf
  49}(3), 1378--1406.

\bibitem[Duchi {\em et~al.}(2021)Duchi, Glynn and
  Namkoong]{duchi2021statistics}
J.~C. Duchi, P.~W. Glynn and H.~Namkoong  (2021), Statistics of robust
  optimization: {A} generalized empirical likelihood approach, {\em Mathematics
  of Operations Research} {\bf 46}(3), 946--969.

\bibitem[Duchi {\em et~al.}(2023)Duchi, Hashimoto and
  Namkoong]{duchi2023distributionally}
J.~Duchi, T.~Hashimoto and H.~Namkoong  (2023), Distributionally robust losses
  for latent covariate mixtures, {\em Operations Research} {\bf 71}(2),
  649--664.

\bibitem[Dudley(1969)Dudley]{dudley1969speed}
R.~M. Dudley  (1969), The speed of mean {G}livenko-{C}antelli convergence, {\em
  The Annals of Mathematical Statistics} {\bf 40}(1), 40--50.

\bibitem[Dul{\'a} and Murthy(1992)Dul{\'a} and Murthy]{dula1992tchebysheff}
J.~H. Dul{\'a} and R.~V. Murthy  (1992), A {T}chebysheff-type bound on the
  expectation of sublinear polyhedral functions, {\em Operations Research} {\bf
  40}(5), 914--922.

\bibitem[Dullerud and Paganini(2001)Dullerud and
  Paganini]{dullerud-2001-robust-control}
G.~E. Dullerud and F.~Paganini  (2001), {\em A Course in Robust Control Theory:
  A Convex Approach}, Springer.

\bibitem[Dupa{\v{c}}ov{\'a}(2006)Dupa{\v{c}}ov{\'a}]{dupacova06contamination}
J.~Dupa{\v{c}}ov{\'a}  (2006), Stress testing via contamination, in {\em Coping
  with Uncertainty: Modeling and Policy Issues} (K.~Marti, Y.~Ermoliev,
  M.~Makowski and G.~Pflug, eds), Springer, pp.~29--46.

\bibitem[Dupacov{\'a} and Wets(1988)Dupacov{\'a} and
  Wets]{dupacova1988asymptotic}
J.~Dupacov{\'a} and R.~Wets  (1988), Asymptotic behavior of statistical
  estimators and of optimal solutions of stochastic optimization problems, {\em
  The Annals of Statistics} {\bf 16}(4), 1517--1549.

\bibitem[Dupačová(1966)Dupačová]{dupavcova1966minimax}
J.~Dupačová  (1966), On minimax solutions of stochastic linear programming
  problems, {\em {\v{C}}asopis pro p{\v{e}}stov{\'a}n{\'\i} matematiky} {\bf
  91}(4), 423--430.

\bibitem[Dupačová(1987)Dupačová]{dupacova1987minimax}
J.~Dupačová  (1987), The minimax approach to stochastic programming and an
  illustrative application, {\em Stochastics} {\bf 20}(1), 73--88.

\bibitem[Dupačová(1994)Dupačová]{dupacova1994applications}
J.~Dupačová  (1994), Applications of stochastic programming under incomplete
  information, {\em Journal of Computational and Applied Mathematics} {\bf
  56}(1--2), 113--125.

\bibitem[Dupuis and Mao(2022)Dupuis and Mao]{dupuis2022formulation}
P.~Dupuis and Y.~Mao  (2022), Formulation and properties of a divergence used
  to compare probability measures without absolute continuity, {\em ESAIM:
  Control, Optimisation and Calculus of Variations} {\bf 28}, Article 10.

\bibitem[Duque and Morton(2020)Duque and Morton]{duque2020distributionally}
D.~Duque and D.~P. Morton  (2020), Distributionally robust stochastic dual
  dynamic programming, {\em SIAM Journal on Optimization} {\bf 30}(4),
  2841--2865.

\bibitem[Dyer and Stougie(2006)Dyer and Stougie]{dyer2006computational}
M.~Dyer and L.~Stougie  (2006), Computational complexity of stochastic
  programming problems, {\em Mathematical Programming} {\bf 106}(3), 423--432.

\bibitem[Edmundson(1956)Edmundson]{edmundson:56}
H.~Edmundson  (1956), Bounds on the expectation of a convex function of a
  random variable, Technical report, The Rand Corporation Paper 982, Santa
  Monica, California.

\bibitem[{El Ghaoui} and Lebret(1998{\em a}){El Ghaoui} and
  Lebret]{elghaoui1998robust}
L.~{El Ghaoui} and H.~Lebret  (1998{\em a}), Robust optimization of control
  systems: A convex approach, {\em IEEE Transactions on Automatic Control} {\bf
  43}(3), 309--319.

\bibitem[{El Ghaoui} and Lebret(1998{\em b}){El Ghaoui} and
  Lebret]{elghaoui1998least}
L.~{El Ghaoui} and H.~Lebret  (1998{\em b}), Robust solutions to least-squares
  problems with uncertain data, {\em SIAM Journal on Matrix Analysis and
  Applications} {\bf 18}(4), 1035--1064.

\bibitem[El~Ghaoui {\em et~al.}(2003)El~Ghaoui, Oks and
  Oustry]{ghaoui2003worst}
L.~El~Ghaoui, M.~Oks and F.~Oustry  (2003), Worst-case value-at-risk and robust
  portfolio optimization: A conic programming approach, {\em Operations
  Research} {\bf 51}(4), 543--556.

\bibitem[{El Ghaoui} {\em et~al.}(1998){El Ghaoui}, Oustry and
  Lebret]{elghaoui1997robust}
L.~{El Ghaoui}, F.~Oustry and H.~Lebret  (1998), Robust solutions to uncertain
  semidefinite programs, {\em SIAM Journal on Optimization} {\bf 9}(1), 33--52.

\bibitem[Ellis(2007)Ellis]{ellis2006entropy}
R.~S. Ellis  (2007), {\em Entropy, Large Deviations, and Statistical
  Mechanics}, Springer.

\bibitem[Ellsberg(1961)Ellsberg]{ellsberg1961risk}
D.~Ellsberg  (1961), Risk, ambiguity, and the {S}avage axioms, {\em Quarterly
  Journal of Economics} {\bf 75}(4), 643--669.

\bibitem[Embrechts and Puccetti(2006)Embrechts and
  Puccetti]{embrechts2006bounds}
P.~Embrechts and G.~Puccetti  (2006), Bounds for functions of multivariate
  risks, {\em Journal of Multivariate Analysis} {\bf 97}(2), 526--547.

\bibitem[Epstein and Miao(2003)Epstein and Miao]{epstein2003equilibrium}
L.~G. Epstein and J.~Miao  (2003), A two-person dynamic equilibrium under
  ambiguity, {\em Journal of Economic Dynamics and Control} {\bf 27}(7),
  1253--1288.

\bibitem[Erdo{\u{g}}an and Iyengar(2006)Erdo{\u{g}}an and
  Iyengar]{erdougan2006ambiguous}
E.~Erdo{\u{g}}an and G.~Iyengar  (2006), Ambiguous chance constrained problems
  and robust optimization, {\em Mathematical Programming} {\bf 107}(1-2),
  37--61.

\bibitem[Ermoliev {\em et~al.}(1985)Ermoliev, Gaivoronski and
  Nedeva]{ermoliev1985stochastic}
Y.~Ermoliev, A.~A. Gaivoronski and C.~Nedeva  (1985), Stochastic optimization
  problems with incomplete information on distribution functions, {\em SIAM
  Journal on Control and Optimization} {\bf 23}(5), 697--716.

\bibitem[Esteban-P{\'e}rez and Morales(2022)Esteban-P{\'e}rez and
  Morales]{esteban2022distributionally}
A.~Esteban-P{\'e}rez and J.~M. Morales  (2022), Distributionally robust
  stochastic programs with side information based on trimmings, {\em
  Mathematical Programming} {\bf 195}(1), 1069--1105.

\bibitem[Farnia and Tse(2016)Farnia and Tse]{farnia2016minimax}
F.~Farnia and D.~Tse  (2016), A minimax approach to supervised learning, in
  {\em Advances in Neural Information Processing Systems}, pp.~4240--4248.

\bibitem[Fenchel(1953)Fenchel]{fenchel1953convex}
W.~Fenchel  (1953), {\em Convex Cones, Sets, and Functions}, Princeton
  University Press.

\bibitem[Finlay and Oberman(2021)Finlay and Oberman]{finlay2021scaleable}
C.~Finlay and A.~M. Oberman  (2021), Scaleable input gradient regularization
  for adversarial robustness, {\em Machine Learning with Applications} {\bf 3},
  Article 100017.

\bibitem[Folland(1999)Folland]{folland1999real}
G.~B. Folland  (1999), {\em Real Analysis: Modern Techniques and Their
  Applications}, John Wiley \& Sons.

\bibitem[F{\"o}llmer and Schied(2008)F{\"o}llmer and
  Schied]{follmer2008stochastic}
H.~F{\"o}llmer and A.~Schied  (2008), {\em Stochastic Finance. An Introduction
  in Discrete Time}, de Gruyter.

\bibitem[Fournier(2023)Fournier]{fournier2022convergence}
N.~Fournier  (2023), Convergence of the empirical measure in expected
  {W}asserstein distance: {N}on-asymptotic explicit bounds in {$\R^d$}, {\em
  ESAIM: Probability and Statistics} {\bf 27}, 749--775.

\bibitem[Fournier and Guillin(2015)Fournier and Guillin]{fournier2015rate}
N.~Fournier and A.~Guillin  (2015), On the rate of convergence in {W}asserstein
  distance of the empirical measure, {\em Probability Theory and Related
  Fields} {\bf 162}(3), 707--738.

\bibitem[Frank and Niles-Weed(2024{\em a})Frank and
  Niles-Weed]{frank2024adversarial}
N.~Frank and J.~Niles-Weed  (2024{\em a}), The adversarial consistency of
  surrogate risks for binary classification, in {\em Advances in Neural
  Information Processing Systems}, pp.~41343--41354.

\bibitem[Frank and Niles-Weed(2024{\em b})Frank and
  Niles-Weed]{frank2024existence}
N.~S. Frank and J.~Niles-Weed  (2024{\em b}), Existence and minimax theorems
  for adversarial surrogate risks in binary classification, {\em Journal of
  Machine Learning Research} {\bf 25}(58), 1--41.

\bibitem[Frauendorfer(1992)Frauendorfer]{frauendorfer:92}
K.~Frauendorfer  (1992), {\em Stochastic Two-Stage Programming}, Springer.

\bibitem[Fr{\'e}chet(1935)Fr{\'e}chet]{frechet1935generalisation}
M.~Fr{\'e}chet  (1935), G{\'e}n{\'e}ralisation du th{\'e}oreme des
  probabilit{\'e}s totales, {\em Fundamenta Mathematicae} {\bf 25}(1),
  379--387.

\bibitem[Gaivoronski(1991)Gaivoronski]{gaivoronski1991numerical}
A.~A. Gaivoronski  (1991), A numerical method for solving stochastic
  programming problems with moment constraints on a distribution function, {\em
  Annals of Operations Research} {\bf 31}(1), 347--370.

\bibitem[Gallego and Moon(1993)Gallego and Moon]{gallegomoon:93}
G.~Gallego and I.~Moon  (1993), The distribution free newsboy problem: Review
  and extensions, {\em The Journal of the Operational Research Society} {\bf
  44}(8), 825--834.

\bibitem[Ganguly and Sutter(2023)Ganguly and Sutter]{ganguly2023optimal}
A.~Ganguly and T.~Sutter  (2023), Optimal learning via moderate deviations
  theory, {\em arXiv:2305.14496}.

\bibitem[Gao(2023)Gao]{gao2020finite}
R.~Gao  (2023), Finite-sample guarantees for {W}asserstein distributionally
  robust optimization: {B}reaking the curse of dimensionality, {\em Operations
  Research} {\bf 71}(6), 2291--2306.

\bibitem[Gao and Kleywegt(2023)Gao and Kleywegt]{gao2016distributionally}
R.~Gao and A.~J. Kleywegt  (2023), Distributionally robust stochastic
  optimization with {W}asserstein distance, {\em Mathematics of Operations
  Research} {\bf 48}(2), 603--655.

\bibitem[Gao {\em et~al.}(2024{\em a})Gao, Arora and Huang]{aroradata2022data}
R.~Gao, R.~Arora and Y.~Huang  (2024{\em a}), Data-driven multistage
  distributionally robust linear optimization with nested distance, {\em arXiv
  preprint arXiv:2407.16346}.

\bibitem[Gao {\em et~al.}(2017)Gao, Chen and Kleywegt]{gao2017wasserstein}
R.~Gao, X.~Chen and A.~J. Kleywegt  (2017), {W}asserstein distributional
  robustness and regularization in statistical learning, {\em
  arXiv:1712.06050}.

\bibitem[Gao {\em et~al.}(2024{\em b})Gao, Chen and
  Kleywegt]{gao2024wasserstein}
R.~Gao, X.~Chen and A.~J. Kleywegt  (2024{\em b}), {W}asserstein
  distributionally robust optimization and variation regularization, {\em
  Operations Research} {\bf 72}(3), 1177--1191.

\bibitem[Gao {\em et~al.}(2018)Gao, Xie, Xie and Xu]{gao2018robust}
R.~Gao, L.~Xie, Y.~Xie and H.~Xu  (2018), Robust hypothesis testing using
  {W}asserstein uncertainty sets, in {\em Advances in Neural Information
  Processing Systems}, pp.~7902--7912.

\bibitem[Garc{\'\i}a~Trillos and Garc{\'\i}a~Trillos(2022)Garc{\'\i}a~Trillos
  and Garc{\'\i}a~Trillos]{garcia2022regularized}
C.~A. Garc{\'\i}a~Trillos and N.~Garc{\'\i}a~Trillos  (2022), On the
  regularized risk of distributionally robust learning over deep neural
  networks, {\em Research in the Mathematical Sciences} {\bf 9}(3), 54.

\bibitem[Garc{\'\i}a~Trillos and Jacobs(2023)Garc{\'\i}a~Trillos and
  Jacobs]{garcia2023analytical}
N.~Garc{\'\i}a~Trillos and M.~Jacobs  (2023), An analytical and geometric
  perspective on adversarial robustness, {\em Notices of the American
  Mathematical Society} {\bf 70}(8), 1193--1204.

\bibitem[Garc{\'\i}a~Trillos and Murray(2022)Garc{\'\i}a~Trillos and
  Murray]{trillos2022adversarial}
N.~Garc{\'\i}a~Trillos and R.~Murray  (2022), Adversarial classification:
  {N}ecessary conditions and geometric flows, {\em Journal of Machine Learning
  Research} {\bf 23}(187), 1--38.

\bibitem[Garc{\'\i}a~Trillos {\em et~al.}(2023)Garc{\'\i}a~Trillos, Jacobs and
  Kim]{trillos2023multimarginal}
N.~Garc{\'\i}a~Trillos, M.~Jacobs and J.~Kim  (2023), The multimarginal optimal
  transport formulation of adversarial multiclass classification, {\em Journal
  of Machine Learning Research} {\bf 24}(45), 1--56.

\bibitem[Gassmann and Ziemba(1986)Gassmann and Ziemba]{gassmannziemba:86}
H.~Gassmann and W.~Ziemba  (1986), A tight upper bound for the expectation of a
  convex function of a multivariate random variable, in {\em Stochastic
  Programming 84 Part I} (A.~Pr{\'e}kopa and R.~J.-B. Wets, eds), Vol.~27,
  Springer, pp.~39--53.

\bibitem[Gelbrich(1990)Gelbrich]{gelbrich1990formula}
M.~Gelbrich  (1990), On a formula for the ${L}^2$ {W}asserstein metric between
  measures on {E}uclidean and {H}ilbert spaces, {\em Mathematische Nachrichten}
  {\bf 147}(1), 185--203.

\bibitem[Georgakopoulos {\em et~al.}(1988)Georgakopoulos, Kavvadias and
  Papadimitriou]{georgakopoulos1988probabilistic}
G.~Georgakopoulos, D.~Kavvadias and C.~H. Papadimitriou  (1988), Probabilistic
  satisfiability, {\em Journal of Complexity} {\bf 4}(1), 1--11.

\bibitem[Ghanem {\em et~al.}(2017)Ghanem, Higdon and
  Owhadi]{ghanem2017handbook}
R.~Ghanem, D.~Higdon and H.~Owhadi  (2017), {\em Handbook of Uncertainty
  Quantification}, Springer.

\bibitem[Ghosh {\em et~al.}(2021)Ghosh, Squillante and
  Wollega]{ghosh2021efficient}
S.~Ghosh, M.~Squillante and E.~Wollega  (2021), Efficient stochastic gradient
  descent for learning with distributionally robust optimization, in {\em
  Advances in Neural Information Processing Systems}, pp.~28310--28322.

\bibitem[Gilboa and Schmeidler(1989)Gilboa and Schmeidler]{gilboa1989maxmin}
I.~Gilboa and D.~Schmeidler  (1989), Maxmin expected utility with a non-unique
  prior, {\em Journal of Mathematical Economics} {\bf 18}(2), 141--153.

\bibitem[Givens and Shortt(1984)Givens and Shortt]{givens1984class}
C.~Givens and R.~Shortt  (1984), A class of {W}asserstein metrics for
  probability distributions, {\em The Michigan Mathematical Journal} {\bf
  31}(2), 231--240.

\bibitem[Goerigk and Kurtz(2023)Goerigk and Kurtz]{goerigk2023data}
M.~Goerigk and J.~Kurtz  (2023), Data-driven robust optimization using deep
  neural networks, {\em Computers \& Operations Research} {\bf 151}, Article
  106087.

\bibitem[Goodfellow {\em et~al.}(2015)Goodfellow, Shlens and
  Szegedy]{ian15adversarial}
I.~J. Goodfellow, J.~Shlens and C.~Szegedy  (2015), Explaining and harnessing
  adversarial examples, in {\em International Conference on Learning
  Representations}.

\bibitem[Gotoh {\em et~al.}(2018)Gotoh, Kim and Lim]{gotoh2018robust}
J.-y. Gotoh, M.~J. Kim and A.~E. Lim  (2018), Robust empirical optimization is
  almost the same as mean--variance optimization, {\em Operations Research
  Letters} {\bf 46}(4), 448--452.

\bibitem[Gotoh {\em et~al.}(2021)Gotoh, Kim and Lim]{gotoh2021calibration}
J.-y. Gotoh, M.~J. Kim and A.~E. Lim  (2021), Calibration of distributionally
  robust empirical optimization models, {\em Operations Research} {\bf 69}(5),
  1630--1650.

\bibitem[Gravin and Lu(2018)Gravin and Lu]{gravin2018separation}
N.~Gravin and P.~Lu  (2018), Separation in correlation-robust monopolist
  problem with budget, in {\em SIAM Symposium on Discrete Algorithms},
  pp.~2069--2080.

\bibitem[Green and Limebeer(1995)Green and Limebeer]{green1995tutorial}
M.~Green and D.~J.~N. Limebeer  (1995), H-infinity control theory: A tutorial,
  {\em Automatica} {\bf 31}(2), 213--222.

\bibitem[G{\"u}l and Zoubir(2017)G{\"u}l and Zoubir]{gul2017minimax}
G.~G{\"u}l and A.~M. Zoubir  (2017), Minimax robust hypothesis testing, {\em
  IEEE Transactions on Information Theory} {\bf 63}(9), 5572--5587.

\bibitem[Gulrajani {\em et~al.}(2017)Gulrajani, Ahmed, Arjovsky, Dumoulin and
  Courville]{gulrajani2017improved}
I.~Gulrajani, F.~Ahmed, M.~Arjovsky, V.~Dumoulin and A.~Courville  (2017),
  Improved training of {W}asserstein {GAN}s, in {\em Advances in Neural
  Information Processing Systems}, pp.~5769--5779.

\bibitem[Gupta(2019)Gupta]{gupta2019near}
V.~Gupta  (2019), Near-optimal {B}ayesian ambiguity sets for distributionally
  robust optimization, {\em Management Science} {\bf 65}(9), 4242--4260.

\bibitem[G{\"u}rb{\"u}zbalaban {\em et~al.}(2022)G{\"u}rb{\"u}zbalaban,
  Ruszczy{\'n}ski and Zhu]{gurbuzbalaban2022stochastic}
M.~G{\"u}rb{\"u}zbalaban, A.~Ruszczy{\'n}ski and L.~Zhu  (2022), A stochastic
  subgradient method for distributionally robust non-convex and non-smooth
  learning, {\em Journal of Optimization Theory and Applications} {\bf 194}(3),
  1014--1041.

\bibitem[Hajar {\em et~al.}(2023)Hajar, Kargin and
  Hassibi]{hajar2023wasserstein}
J.~Hajar, T.~Kargin and B.~Hassibi  (2023), {W}asserstein distributionally
  robust regret-optimal control under partial observability, in {\em Allerton
  Conference on Communication, Control, and Computing}, pp.~1--6.

\bibitem[Hakobyan and Yang(2024)Hakobyan and Yang]{hakobyan2024wasserstein}
A.~Hakobyan and I.~Yang  (2024), {W}asserstein distributionally robust control
  of partially observable linear stochastic systems, {\em IEEE Transactions on
  Automatic Control} {\bf 69}(9), 6121--6136.

\bibitem[Hamburger(1920)Hamburger]{hamburger1920moment}
H.~Hamburger  (1920), {\"U}ber eine {E}rweiterung des {S}tieltjesschen
  {M}omentenproblems, {\em Mathematische Annalen} {\bf 81}(2), 235--319.

\bibitem[Hampel(1968)Hampel]{hampel1968contributions}
F.~R. Hampel  (1968), Contributions to the theory of robust estimation,
  Technical report, University of California, Berkeley.

\bibitem[Hampel(1971)Hampel]{hampel1971general}
F.~R. Hampel  (1971), A general qualitative definition of robustness, {\em The
  Annals of Mathematical Statistics} {\bf 42}(6), 1887--1896.

\bibitem[Han {\em et~al.}(2021)Han, Shang and Huang]{han2021multiple}
B.~Han, C.~Shang and D.~Huang  (2021), Multiple kernel learning-aided robust
  optimization: {L}earning algorithm, computational tractability, and usage in
  multi-stage decision-making, {\em European Journal of Operational Research}
  {\bf 292}(3), 1004--1018.

\bibitem[Han {\em et~al.}(2015)Han, Tao, Topcu, Owhadi and
  Murray]{han2015convexUQ}
S.~Han, M.~Tao, U.~Topcu, H.~Owhadi and R.~M. Murray  (2015), Convex optimal
  uncertainty quantification, {\em SIAM Journal on Optimization} {\bf 25}(3),
  1368--1387.

\bibitem[Hanasusanto and Kuhn(2013)Hanasusanto and Kuhn]{hanasusanto2013robust}
G.~A. Hanasusanto and D.~Kuhn  (2013), Robust data-driven dynamic programming,
  in {\em Advances in Neural Information Processing Systems}, pp.~827--835.

\bibitem[Hanasusanto and Kuhn(2018)Hanasusanto and Kuhn]{hanasusanto2018conic}
G.~A. Hanasusanto and D.~Kuhn  (2018), Conic programming reformulations of
  two-stage distributionally robust linear programs over {W}asserstein balls,
  {\em Operations Research} {\bf 66}(3), 849--869.

\bibitem[Hanasusanto {\em et~al.}(2016)Hanasusanto, Kuhn and
  Wiesemann]{hanasusanto2016comment}
G.~A. Hanasusanto, D.~Kuhn and W.~Wiesemann  (2016), A comment on
  ``{C}omputational complexity of stochastic programming problems'', {\em
  Mathematical Programming} {\bf 159}(1-2), 557--569.

\bibitem[Hanasusanto {\em et~al.}(2015{\em a})Hanasusanto, Kuhn, Wallace and
  Zymler]{hanasusanto2015distributionally}
G.~A. Hanasusanto, D.~Kuhn, S.~W. Wallace and S.~Zymler  (2015{\em a}),
  Distributionally robust multi-item newsvendor problems with multimodal demand
  distributions, {\em Mathematical Programming} {\bf 152}(1), 1--32.

\bibitem[Hanasusanto {\em et~al.}(2015{\em b})Hanasusanto, Roitch, Kuhn and
  Wiesemann]{hanasusanto2015perspective}
G.~A. Hanasusanto, V.~Roitch, D.~Kuhn and W.~Wiesemann  (2015{\em b}), A
  distributionally robust perspective on uncertainty quantification and chance
  constrained programming, {\em Mathematical Programming} {\bf 151}(1), 35--62.

\bibitem[Hansen and Sargent(2008)Hansen and Sargent]{hansen2008robustness}
L.~P. Hansen and T.~J. Sargent  (2008), {\em Robustness}, Princeton University
  Press.

\bibitem[Hansen and Sargent(2010)Hansen and Sargent]{HANSEN20101097}
L.~P. Hansen and T.~J. Sargent  (2010), Wanting robustness in macroeconomics,
  in {\em Handbook of Monetary Economics} (B.~M. Friedman and M.~Woodford,
  eds), Vol.~3, Elsevier, chapter~20, pp.~1097--1157.

\bibitem[Hartley and Somerville(2015)Hartley and
  Somerville]{hartley2015adolescent}
C.~A. Hartley and L.~H. Somerville  (2015), The neuroscience of adolescent
  decision-making, {\em Current Opinion in Behavioral Sciences} {\bf 5},
  108--115.

\bibitem[Hartung(1982)Hartung]{hartung1982extension}
J.~Hartung  (1982), An extension of {S}ion's minimax theorem with an
  application to a method for constrained games, {\em Pacific Journal of
  Mathematics} {\bf 103}(2), 401--408.

\bibitem[Hastie {\em et~al.}(2009)Hastie, Tibshirani and
  Friedman]{hastie2009elements}
T.~Hastie, R.~Tibshirani and J.~Friedman  (2009), {\em The Elements of
  Statistical Learning: Data Mining, Inference, and Prediction}, Springer.

\bibitem[Hausdorff(1923)Hausdorff]{hausdorff1923moment}
F.~Hausdorff  (1923), Momentprobleme f{\"u}r ein endliches {I}ntervall, {\em
  Mathematische Zeitschrift} {\bf 16}(1), 220--248.

\bibitem[Hayden {\em et~al.}(2010)Hayden, Heilbronner and Platt]{monkeys}
B.~Hayden, S.~Heilbronner and M.~Platt  (2010), Ambiguity aversion in rhesus
  macaques, {\em Frontiers in Neuroscience \textbf{\emph{4}}}.

\bibitem[Hazan(2022)Hazan]{hazan2022introduction}
E.~Hazan  (2022), {\em Introduction to Online Convex Optimization}, MIT Press.

\bibitem[He {\em et~al.}(2010)He, Xue, Chen, Lu, Dong, Lei, Ding, Li, Li, Chen,
  Li, Moyzis and Bechara]{he2010serotonin}
Q.~He, G.~Xue, C.~Chen, Z.~Lu, Q.~Dong, X.~Lei, N.~Ding, J.~Li, H.~Li, C.~Chen,
  J.~Li, R.~K. Moyzis and A.~Bechara  (2010), Serotonin transporter gene-linked
  polymorphic region (5-{HTTLPR}) influences decision making under ambiguity
  and risk in a large {C}hinese sample, {\em Neuropharmacology} {\bf 59}(6),
  518--526.

\bibitem[He and Lam(2021)He and Lam]{he2021higher}
S.~He and H.~Lam  (2021), Higher-order expansion and {B}artlett correctability
  of distributionally robust optimization, {\em arXiv:2108.05908}.

\bibitem[Hespanha(2019)Hespanha]{hespanha2009linear}
J.~P. Hespanha  (2019), {\em Linear Systems Theory}, Princeton University
  Press.

\bibitem[Ho-Nguyen and K{\i}l{\i}n{\c{c}}-Karzan(2018)Ho-Nguyen and
  K{\i}l{\i}n{\c{c}}-Karzan]{ho2018online}
N.~Ho-Nguyen and F.~K{\i}l{\i}n{\c{c}}-Karzan  (2018), Online first-order
  framework for robust convex optimization, {\em Operations Research} {\bf
  66}(6), 1670--1692.

\bibitem[Ho-Nguyen and K{\i}l{\i}n{\c{c}}-Karzan(2019)Ho-Nguyen and
  K{\i}l{\i}n{\c{c}}-Karzan]{ho2019exploiting}
N.~Ho-Nguyen and F.~K{\i}l{\i}n{\c{c}}-Karzan  (2019), Exploiting problem
  structure in optimization under uncertainty via online convex optimization,
  {\em Mathematical Programming} {\bf 177}(1), 113--147.

\bibitem[Ho-Nguyen and Wright(2023)Ho-Nguyen and Wright]{ho2020adversarial}
N.~Ho-Nguyen and S.~J. Wright  (2023), Adversarial classification via
  distributional robustness with {W}asserstein ambiguity, {\em Mathematical
  Programming} {\bf 198}(2), 1411--1447.

\bibitem[Ho-Nguyen {\em et~al.}(2022)Ho-Nguyen, K{\i}l{\i}n{\c{c}}-Karzan,
  K{\"u}{\c{c}}{\"u}kyavuz and Lee]{ho2020distributionally}
N.~Ho-Nguyen, F.~K{\i}l{\i}n{\c{c}}-Karzan, S.~K{\"u}{\c{c}}{\"u}kyavuz and
  D.~Lee  (2022), Distributionally robust chance-constrained programs with
  right-hand side uncertainty under {W}asserstein ambiguity, {\em Mathematical
  Programming} {\bf 196}(1--2), 641--672.

\bibitem[Honeyman {\em et~al.}(1980)Honeyman, Ladner and
  Yannakakis]{honeyman1980testing}
P.~Honeyman, R.~E. Ladner and M.~Yannakakis  (1980), Testing the universal
  instance assumption, {\em Information Processing Letters} {\bf 10}(1),
  14--19.

\bibitem[Hong {\em et~al.}(2021)Hong, Huang and Lam]{hong2021learning}
L.~J. Hong, Z.~Huang and H.~Lam  (2021), Learning-based robust optimization:
  {P}rocedures and statistical guarantees, {\em Management Science} {\bf
  67}(6), 3447--3467.

\bibitem[Horn and Johnson(1985)Horn and Johnson]{horn1985optimal}
R.~A. Horn and C.~R. Johnson  (1985), H$\infty$-optimal control and related
  minimax design problems, {\em IEEE Transactions on Automatic Control} {\bf
  30}(10), 1057--1069.

\bibitem[Hou {\em et~al.}(2023)Hou, Kassraie, Kratsios, Krause and
  Rothfuss]{hou2023instance}
S.~Hou, P.~Kassraie, A.~Kratsios, A.~Krause and J.~Rothfuss  (2023),
  Instance-dependent generalization bounds via optimal transport, {\em Journal
  of Machine Learning Research} {\bf 24}(1), 16815--16865.

\bibitem[Hsu {\em et~al.}(2005)Hsu, Bhatt, Adolphs, Tranel and
  Camerer]{hsu2005neural}
M.~Hsu, M.~Bhatt, R.~Adolphs, D.~Tranel and C.~F. Camerer  (2005), Neural
  systems responding to degrees of uncertainty in human decision-making, {\em
  Science} {\bf 310}(5754), 1680--1683.

\bibitem[Hu {\em et~al.}(2021)Hu, Chen and He]{hu2021bias}
Y.~Hu, X.~Chen and N.~He  (2021), On the bias-variance-cost tradeoff of
  stochastic optimization, in {\em Advances in Neural Information Processing
  Systems}, pp.~22119--22131.

\bibitem[Hu {\em et~al.}(2024)Hu, Wang, Chen and He]{hu2024multi}
Y.~Hu, J.~Wang, X.~Chen and N.~He  (2024), Multi-level {M}onte-{C}arlo gradient
  methods for stochastic optimization with biased oracles, {\em
  arXiv:2408.11084}.

\bibitem[Hu and Hong(2013)Hu and Hong]{hu2013kullback}
Z.~Hu and L.~J. Hong  (2013), Kullback-{L}eibler divergence constrained
  distributionally robust optimization, {\em Available from Optimization
  Online}.

\bibitem[Hu {\em et~al.}(2013)Hu, Hong and So]{hu2013ambiguous}
Z.~Hu, L.~J. Hong and A.~M.-C. So  (2013), Ambiguous probabilistic programs,
  {\em Available from Optimization Online}.

\bibitem[Huang {\em et~al.}(2004)Huang, Yang, King, Lyu and
  Chan]{huang2004minimum}
K.~Huang, H.~Yang, I.~King, M.~R. Lyu and L.~Chan  (2004), The minimum error
  minimax probability machine, {\em Journal of Machine Learning Research} {\bf
  5}, 1253--1286.

\bibitem[Huber(1981)Huber]{huber1981robust}
P.~Huber  (1981), {\em Robust Statistics}, Wiley.

\bibitem[Huber(1964)Huber]{huber1964robust}
P.~J. Huber  (1964), Robust estimation of a location parameter, {\em The Annals
  of Mathematical Statistics} {\bf 35}(1), 73--101.

\bibitem[Huber(1967)Huber]{huber1967behavior}
P.~J. Huber  (1967), The behavior of maximum likelihood estimates under
  nonstandard conditions, in {\em Proceedings of the Fifth Berkeley Symposium
  on Mathematical Statistics and Probability}, pp.~221--233.

\bibitem[Huber(1968)Huber]{huber1968robust}
P.~J. Huber  (1968), Robust confidence limits, {\em Zeitschrift f{\"u}r
  Wahrscheinlichkeitstheorie und verwandte Gebiete} {\bf 10}(4), 269--278.

\bibitem[Husain(2020)Husain]{husain2020distributional}
H.~Husain  (2020), Distributional robustness with {IPM}s and links to
  regularization and {GAN}s, in {\em Advances in Neural Information Processing
  Systems}, pp.~11816--11827.

\bibitem[Isii(1960)Isii]{isii1960extrema}
K.~Isii  (1960), The extrema of probability determined by generalized moments
  ({I}) {B}ounded random variables, {\em Annals of the Institute of Statistical
  Mathematics} {\bf 12}(2), 119--134.

\bibitem[Isii(1962)Isii]{isii1962sharpness}
K.~Isii  (1962), On sharpness of {T}chebycheff-type inequalities, {\em Annals
  of the Institute of Statistical Mathematics} {\bf 14}(1), 185--197.

\bibitem[Iyengar {\em et~al.}(2023)Iyengar, Lam and Wang]{iyengar2022hedging}
G.~Iyengar, H.~Lam and T.~Wang  (2023), Hedging against complexity:
  {D}istributionally robust optimization with parametric approximation,
  pp.~9976--10011.

\bibitem[Jagannathan(1977)Jagannathan]{jagannathan:77}
R.~Jagannathan  (1977), Minimax procedure for a class of linear programs under
  uncertainty, {\em Operations Research} {\bf 25}(1), 173--177.

\bibitem[Jakubovitz and Giryes(2018)Jakubovitz and
  Giryes]{jakubovitz2018improving}
D.~Jakubovitz and R.~Giryes  (2018), Improving {DNN} robustness to adversarial
  attacks using {J}acobian regularization, in {\em European Conference on
  Computer Vision}, pp.~514--529.

\bibitem[Janak {\em et~al.}(2007)Janak, Lin and Floudas]{janak2007new}
S.~L. Janak, X.~Lin and C.~A. Floudas  (2007), A new robust optimization
  approach for scheduling under uncertainty: {II.} {U}ncertainty with known
  probability distribution, {\em Computers \& Chemical Engineering} {\bf
  31}(3), 171--195.

\bibitem[Jeffreys and Wrinch(1921)Jeffreys and Wrinch]{jeffreys1921certain}
H.~Jeffreys and D.~Wrinch  (1921), On certain fundamental principles of
  scientific enquiry, {\em Philosophical Magazine} {\bf 42}, 269--298.

\bibitem[Jensen(1906)Jensen]{jensen1906fonctions}
J.~L. W.~V. Jensen  (1906), Sur les fonctions convexes et les
  in{\'e}galit{\'e}s entre les valeurs moyennes, {\em Acta Mathematica} {\bf
  30}(1), 175--193.

\bibitem[Jiang and Xie(2024)Jiang and Xie]{xie2024favor}
N.~Jiang and W.~Xie  (2024), Distributionally favorable optimization: A
  framework for data-driven decision-making with endogenous outliers, {\em
  {SIAM} Journal on Optimization} {\bf 34}(1), 419--458.

\bibitem[Jiang and Guan(2016)Jiang and Guan]{jiang2016data}
R.~Jiang and Y.~Guan  (2016), Data-driven chance constrained stochastic
  program, {\em Mathematical Programming} {\bf 158}(1), 291--327.

\bibitem[Jiang and Guan(2018)Jiang and Guan]{jiang2018risk}
R.~Jiang and Y.~Guan  (2018), Risk-averse two-stage stochastic program with
  distributional ambiguity, {\em Operations Research} {\bf 66}(5), 1390--1405.

\bibitem[Jiang and Obloj(2024)Jiang and Obloj]{jiang2024sensitivity}
Y.~Jiang and J.~Obloj  (2024), Sensitivity of causal distributionally robust
  optimization, {\em arXiv:2408.17109}.

\bibitem[Jiang {\em et~al.}(2024)Jiang, Chewi and
  Pooladian]{jiang2024algorithms}
Y.~Jiang, S.~Chewi and A.-A. Pooladian  (2024), Algorithms for mean-field
  variational inference via polyhedral optimization in the {W}asserstein space,
  in {\em Conference on Learning Theory}, pp.~2720--2721.

\bibitem[Jongeneel {\em et~al.}(2021)Jongeneel, Sutter and
  Kuhn]{jongeneel2021topological}
W.~Jongeneel, T.~Sutter and D.~Kuhn  (2021), Topological linear system
  identification via moderate deviations theory, {\em IEEE Control Systems
  Letters} {\bf 6}, 307--312.

\bibitem[Jongeneel {\em et~al.}(2022)Jongeneel, Sutter and
  Kuhn]{jongeneel2022efficient}
W.~Jongeneel, T.~Sutter and D.~Kuhn  (2022), Efficient learning of a linear
  dynamical system with stability guarantees, {\em IEEE Transactions on
  Automatic Control} {\bf 68}(5), 2790--2804.

\bibitem[Jylh{\"a}(2015)Jylh{\"a}]{jylha2015optimal}
H.~Jylh{\"a}  (2015), The {$L^\infty$} optimal transport: infinite cyclical
  monotonicity and the existence of optimal transport maps, {\em Calculus of
  Variations and Partial Differential Equations} {\bf 52}, 303--326.

\bibitem[Kallenberg(1997)Kallenberg]{kallenberg1997foundations}
O.~Kallenberg  (1997), {\em Foundations of Modern Probability}, Springer.

\bibitem[Kargin {\em et~al.}(2024{\em a})Kargin, Hajar, Malik and
  Hassibi]{kargin2024lqr}
T.~Kargin, J.~Hajar, V.~Malik and B.~Hassibi  (2024{\em a}), The
  distributionally robust infinite-horizon {LQR}, {\em arXiv:2408.06230}.

\bibitem[Kargin {\em et~al.}(2024{\em b})Kargin, Hajar, Malik and
  Hassibi]{kargin2024distributionally}
T.~Kargin, J.~Hajar, V.~Malik and B.~Hassibi  (2024{\em b}), Distributionally
  robust {K}alman filtering over finite and infinite horizon, {\em
  arXiv:2407.18837}.

\bibitem[Kargin {\em et~al.}(2024{\em c})Kargin, Hajar, Malik and
  Hassibi]{kargin2024infinite}
T.~Kargin, J.~Hajar, V.~Malik and B.~Hassibi  (2024{\em c}), Infinite-horizon
  distributionally robust regret-optimal control, in {\em International
  Conference on Machine Learning}, pp.~23187--23214.

\bibitem[Kargin {\em et~al.}(2024{\em d})Kargin, Hajar, Malik and
  Hassibi]{kargin2024wasserstein}
T.~Kargin, J.~Hajar, V.~Malik and B.~Hassibi  (2024{\em d}), {W}asserstein
  distributionally robust regret-optimal control over infinite-horizon, in {\em
  Learning for Dynamics \& Control Conference}, pp.~1688--1701.

\bibitem[Karlin and Studden(1966)Karlin and Studden]{karlin1966tchebycheff}
S.~Karlin and W.~J. Studden  (1966), {\em Tchebycheff Systems: With
  Applications in Analysis and Statistics}, Interscience Publishers.

\bibitem[Karmarkar(1984)Karmarkar]{karmarkar1984new}
N.~Karmarkar  (1984), A new polynomial-time algorithm for linear programming,
  {\em Combinatorica} {\bf 4}(4), 373--395.

\bibitem[Kelley(1960)Kelley]{kelley1960cutting}
J.~E. Kelley, Jr  (1960), The cutting-plane method for solving convex programs,
  {\em Journal of the Society for Industrial and Applied Mathematics} {\bf
  8}(4), 703--712.

\bibitem[Kent {\em et~al.}(2021)Kent, Li, Blanchet and Glynn]{kent2021modified}
C.~Kent, J.~Li, J.~Blanchet and P.~W. Glynn  (2021), Modified {F}rank {W}olfe
  in probability space, in {\em Advances in Neural Information Processing
  Systems}, pp.~14448--14462.

\bibitem[Keynes(1921)Keynes]{keynes1921treatise}
J.~M. Keynes  (1921), {\em A Treatise on Probability}, Macmillan.

\bibitem[Khachiyan(1979)Khachiyan]{khachiyan1979polynomial}
L.~G. Khachiyan  (1979), A polynomial algorithm in linear programming, {\em
  Doklady Akademii Nauk} {\bf 244}(5), 1093--1096.

\bibitem[Khalil(1996)Khalil]{khalil1996control}
H.~K. Khalil  (1996), {\em Control System Analysis and Design with Advanced
  Design Tools}, Prentice Hall.

\bibitem[King and Rockafellar(1993)King and Rockafellar]{king1993asymptotic}
A.~J. King and R.~T. Rockafellar  (1993), Asymptotic theory for solutions in
  statistical estimation and stochastic programming, {\em Mathematics of
  Operations Research} {\bf 18}(1), 148--162.

\bibitem[King and Wets(1991)King and Wets]{king1991epi}
A.~J. King and R.~J.-B. Wets  (1991), Epi-consistency of convex stochastic
  programs, {\em Stochastics and Stochastic Reports} {\bf 34}(1-2), 83--92.

\bibitem[Klabjan {\em et~al.}(2013)Klabjan, Simchi-Levi and
  Song]{klabjan2013robust}
D.~Klabjan, D.~Simchi-Levi and M.~Song  (2013), Robust stochastic lot-sizing by
  means of histograms, {\em Production and Operations Management} {\bf 22}(3),
  691--710.

\bibitem[Knight(1921)Knight]{knight1921risk}
F.~H. Knight  (1921), {\em Risk, Uncertainty and Profit}, Houghton Mifflin.

\bibitem[Ko{\c{c}}yi{\u{g}}it {\em et~al.}(2020)Ko{\c{c}}yi{\u{g}}it, Iyengar,
  Kuhn and Wiesemann]{koccyiugit2020distributionally}
{\c{C}}.~Ko{\c{c}}yi{\u{g}}it, G.~Iyengar, D.~Kuhn and W.~Wiesemann  (2020),
  Distributionally robust mechanism design, {\em Management Science} {\bf
  66}(1), 159--189.

\bibitem[Ko{\c{c}}yi{\u{g}}it {\em et~al.}(2022)Ko{\c{c}}yi{\u{g}}it,
  Rujeerapaiboon and Kuhn]{koccyiugit2021robust}
{\c{C}}.~Ko{\c{c}}yi{\u{g}}it, N.~Rujeerapaiboon and D.~Kuhn  (2022), Robust
  multidimensional pricing: {S}eparation without regret, {\em Mathematical
  Programming} {\bf 196}(1--2), 841--874.

\bibitem[Koltchinskii(2011)Koltchinskii]{koltchinskii2011oracle}
V.~Koltchinskii  (2011), {\em Oracle Inequalities in Empirical Risk
  Minimization and Sparse Recovery Problems}, Springer.

\bibitem[Kouvelis and Yu(1997)Kouvelis and Yu]{kouvelis1997robust}
P.~Kouvelis and G.~Yu  (1997), {\em Robust Discrete Optimization and Its
  Applications}, Springer.

\bibitem[Krain {\em et~al.}(2006)Krain, Wilson, Arbuckle, Castellanos and
  Milham]{krain2006distinct}
A.~L. Krain, A.~M. Wilson, R.~Arbuckle, X.~F. Castellanos and M.~P. Milham
  (2006), Distinct neural mechanisms of risk and ambiguity: A meta-analysis of
  decision-making, {\em NeuroImage} {\bf 32}(1), 477--484.

\bibitem[Krantz and Parks(2002)Krantz and Parks]{krantz2002primer}
S.~G. Krantz and H.~R. Parks  (2002), {\em A Primer of Real Analytic
  Functions}, Springer.

\bibitem[Kuhn(2005)Kuhn]{kuhn2005generalized-bounds}
D.~Kuhn  (2005), {\em Generalized Bounds for Convex Multistage Stochastic
  Programs}, Springer.

\bibitem[Kuhn {\em et~al.}(2019)Kuhn, Mohajerin~Esfahani, Nguyen and
  Shafieezadeh-Abadeh]{kuhn2019wasserstein}
D.~Kuhn, P.~Mohajerin~Esfahani, V.~A. Nguyen and S.~Shafieezadeh-Abadeh
  (2019), {{W}asserstein distributionally robust optimization: Theory and
  applications in machine learning}, {\em {INFORMS} Tutorials in Operations
  Research} pp.~130--166.

\bibitem[Kullback(1959)Kullback]{kullback1959information}
S.~Kullback  (1959), Information theory and statistics, {\em Wiley}.

\bibitem[Kupper and Schachermayer(2009)Kupper and
  Schachermayer]{kupper2009representation}
M.~Kupper and W.~Schachermayer  (2009), Representation results for law
  invariant time consistent functions, {\em Mathematics and Financial
  Economics} {\bf 2}(3), 189--210.

\bibitem[Kurakin {\em et~al.}(2022)Kurakin, Goodfellow and
  Bengio]{kurakin2022adversarial}
A.~Kurakin, I.~J. Goodfellow and S.~Bengio  (2022), Adversarial machine
  learning at scale, in {\em International Conference on Learning
  Representations}.

\bibitem[Kusuoka(2001)Kusuoka]{ref:kusuoka2001coherent}
S.~Kusuoka  (2001), On law invariant coherent risk measures, in {\em Advances
  in Mathematical Economics} (S.~Kusuoka and T.~Maruyama, eds), Springer,
  pp.~83--95.

\bibitem[Kwon {\em et~al.}(2020)Kwon, Kim, Won and Paik]{kwon2020principled}
Y.~Kwon, W.~Kim, J.-H. Won and M.~C. Paik  (2020), Principled learning method
  for {W}asserstein distributionally robust optimization with local
  perturbations, in {\em International Conference on Machine Learning},
  pp.~5567--5576.

\bibitem[Lal(1955)Lal]{lal1955note}
D.~N. Lal  (1955), A note on a form of {T}chebycheff's inequality for two or
  more variables, {\em Sankhy{\=a}: The Indian Journal of Statistics} {\bf
  15}(3), 317--320.

\bibitem[Lam(2016)Lam]{lam2016robust}
H.~Lam  (2016), Robust sensitivity analysis for stochastic systems, {\em
  Mathematics of Operations Research} {\bf 41}(4), 1248--1275.

\bibitem[Lam(2018)Lam]{lam2018sensitivity}
H.~Lam  (2018), Sensitivity to serial dependency of input processes: {A} robust
  approach, {\em Management Science} {\bf 64}(3), 1311--1327.

\bibitem[Lam(2019)Lam]{lam2019recovering}
H.~Lam  (2019), Recovering best statistical guarantees via the empirical
  divergence-based distributionally robust optimization, {\em Operations
  Research} {\bf 67}(4), 1090--1105.

\bibitem[Lam(2021)Lam]{lam2021impossibility}
H.~Lam  (2021), On the impossibility of statistically improving empirical
  optimization: {A} second-order stochastic dominance perspective, {\em
  arXiv:2105.13419}.

\bibitem[Lam and Mottet(2017)Lam and Mottet]{lam2017tail}
H.~Lam and C.~Mottet  (2017), Tail analysis without parametric models: {A}
  worst-case perspective, {\em Operations Research} {\bf 65}(6), 1696--1711.

\bibitem[Lam and Zhou(2017)Lam and Zhou]{lam2017empirical}
H.~Lam and E.~Zhou  (2017), The empirical likelihood approach to quantifying
  uncertainty in sample average approximation, {\em Operations Research
  Letters} {\bf 45}(4), 301--307.

\bibitem[Lam {\em et~al.}(2024)Lam, Liu and Singham]{lam2024shape}
H.~Lam, Z.~Liu and D.~I. Singham  (2024), Shape-constrained distributional
  optimization via importance-weighted sample average approximation, {\em
  arXiv:2406.07825}.

\bibitem[Lam {\em et~al.}(2021)Lam, Liu and Zhang]{lam2021orthounimodal}
H.~Lam, Z.~Liu and X.~Zhang  (2021), Orthounimodal distributionally robust
  optimization: {R}epresentation, computation and multivariate extreme event
  applications, {\em arXiv:2111.07894}.

\bibitem[Lambert {\em et~al.}(2022)Lambert, Chewi, Bach, Bonnabel and
  Rigollet]{lambert2022variational}
M.~Lambert, S.~Chewi, F.~Bach, S.~Bonnabel and P.~Rigollet  (2022), Variational
  inference via {W}asserstein gradient flows, in {\em Advances in Neural
  Information Processing Systems}, pp.~14434--14447.

\bibitem[Lanckriet {\em et~al.}(2001)Lanckriet, El~Ghaoui, Bhattacharyya and
  Jordan]{lanckriet2001minimax}
G.~R. Lanckriet, L.~El~Ghaoui, C.~Bhattacharyya and M.~I. Jordan  (2001),
  Minimax probability machine, in {\em Advances in Neural Information
  Processing Systems}, pp.~801--807.

\bibitem[Lanckriet {\em et~al.}(2002)Lanckriet, El~Ghaoui, Bhattacharyya and
  Jordan]{lanckriet2002robust}
G.~R. Lanckriet, L.~El~Ghaoui, C.~Bhattacharyya and M.~I. Jordan  (2002), A
  robust minimax approach to classification, {\em Journal of Machine Learning
  Research} {\bf 3}, 555--582.

\bibitem[Lanzetti {\em et~al.}(2022)Lanzetti, Bolognani and
  D{\"o}rfler]{lanzetti2022first}
N.~Lanzetti, S.~Bolognani and F.~D{\"o}rfler  (2022), First-order conditions
  for optimization in the {W}asserstein space, {\em arXiv:2209.12197}.

\bibitem[Lanzetti {\em et~al.}(2024)Lanzetti, Terpin and
  D{\"o}rfler]{lanzetti2024variational}
N.~Lanzetti, A.~Terpin and F.~D{\"o}rfler  (2024), Variational analysis in the
  {W}asserstein space, {\em arXiv:2406.10676}.

\bibitem[Lasserre(2001)Lasserre]{lasserre2001global}
J.~B. Lasserre  (2001), Global optimization with polynomials and the problem of
  moments, {\em SIAM Journal on Optimization} {\bf 11}(3), 796--817.

\bibitem[Lasserre(2002)Lasserre]{lasserre2002bounds}
J.~B. Lasserre  (2002), Bounds on measures satisfying moment conditions, {\em
  The Annals of Applied Probability} {\bf 12}(3), 1114--1137.

\bibitem[Lasserre(2008)Lasserre]{lasserre2008semidefinite}
J.~B. Lasserre  (2008), A semidefinite programming approach to the generalized
  problem of moments, {\em Mathematical Programming} {\bf 112}(1), 65--92.

\bibitem[Lasserre(2009)Lasserre]{lasserre2009moments}
J.~B. Lasserre  (2009), {\em Moments, Positive Polynomials and Their
  Applications}, World Scientific.

\bibitem[Lasserre and Weisser(2021)Lasserre and
  Weisser]{lasserre2021distributionally}
J.~B. Lasserre and T.~Weisser  (2021), Distributionally robust polynomial
  chance-constraints under mixture ambiguity sets, {\em Mathematical
  Programming} {\bf 185}(1-2), 409--453.

\bibitem[Lau and Liu(2022)Lau and Liu]{lau2022wasserstein}
T.~T.-K. Lau and H.~Liu  (2022), {W}asserstein distributionally robust
  optimization with {W}asserstein barycenters, {\em arXiv:2203.12136}.

\bibitem[Lee and Raginsky(2018)Lee and Raginsky]{lee2018minimax}
J.~Lee and M.~Raginsky  (2018), Minimax statistical learning with {W}asserstein
  distances, in {\em Advances in Neural Information Processing Systems},
  pp.~2687--2696.

\bibitem[Lee {\em et~al.}(2020)Lee, Park and Shin]{lee2020learning}
J.~Lee, S.~Park and J.~Shin  (2020), Learning bounds for risk-sensitive
  learning, in {\em Advances in Neural Information Processing Systems},
  pp.~13867--13879.

\bibitem[Lehmann and Casella(2006)Lehmann and Casella]{lehmann2006theory}
E.~L. Lehmann and G.~Casella  (2006), {\em Theory of Point Estimation},
  Springer.

\bibitem[Levitin and Polyak(1966)Levitin and Polyak]{polyak1966constrained}
E.~S. Levitin and B.~T. Polyak  (1966), Constrained minimization methods, {\em
  USSR Computational Mathematics and Mathematical Physics} {\bf 6}(5), 1--50.

\bibitem[Levy(2008)Levy]{levy2008robust}
B.~C. Levy  (2008), Robust hypothesis testing with a relative entropy
  tolerance, {\em IEEE Transactions on Information Theory} {\bf 55}(1),
  413--421.

\bibitem[Levy and Nikoukhah(2004)Levy and Nikoukhah]{levy2004robust}
B.~C. Levy and R.~Nikoukhah  (2004), Robust least-squares estimation with a
  relative entropy constraint, {\em IEEE Transactions on Information Theory}
  {\bf 50}(1), 89--104.

\bibitem[Levy and Nikoukhah(2012)Levy and Nikoukhah]{levy2012robust}
B.~C. Levy and R.~Nikoukhah  (2012), Robust state space filtering under
  incremental model perturbations subject to a relative entropy tolerance, {\em
  IEEE Transactions on Automatic Control} {\bf 58}(3), 682--695.

\bibitem[Levy {\em et~al.}(2020)Levy, Carmon, Duchi and Sidford]{levy2020large}
D.~Levy, Y.~Carmon, J.~C. Duchi and A.~Sidford  (2020), Large-scale methods for
  distributionally robust optimization, in {\em Advances in Neural Information
  Processing Systems}, pp.~8847--8860.

\bibitem[Li {\em et~al.}(2016)Li, Jiang and Mathieu]{li2016distributionally}
B.~Li, R.~Jiang and J.~L. Mathieu  (2016), Distributionally robust
  risk-constrained optimal power flow using moment and unimodality information,
  in {\em IEEE Conference on Decision and Control}, pp.~2425--2430.

\bibitem[Li {\em et~al.}(2019{\em a})Li, Jiang and Mathieu]{li2019ambiguous}
B.~Li, R.~Jiang and J.~L. Mathieu  (2019{\em a}), Ambiguous risk constraints
  with moment and unimodality information, {\em Mathematical Programming} {\bf
  173}(1-2), 151--192.

\bibitem[Li {\em et~al.}(2019{\em b})Li, Turmunkh and Wakker]{li2019trust}
C.~Li, U.~Turmunkh and P.~P. Wakker  (2019{\em b}), Trust as a decision under
  ambiguity, {\em Experimental Economics} {\bf 22}(1), 51--75.

\bibitem[Li and Mart{\'\i}nez(2020)Li and Mart{\'\i}nez]{li2020data}
D.~Li and S.~Mart{\'\i}nez  (2020), Data assimilation and online optimization
  with performance guarantees, {\em IEEE Transactions on Automatic Control}
  {\bf 66}(5), 2115--2129.

\bibitem[Li {\em et~al.}(2020)Li, Chen and So]{li2020fast}
J.~Li, C.~Chen and A.~M.-C. So  (2020), Fast epigraphical projection-based
  incremental algorithms for {W}asserstein distributionally robust support
  vector machine, in {\em Advances in Neural Information Processing Systems},
  pp.~4029--4039.

\bibitem[Li {\em et~al.}(2019{\em c})Li, Huang and So]{li2019first}
J.~Li, S.~Huang and A.~M.-C. So  (2019{\em c}), A first-order algorithmic
  framework for {W}asserstein distributionally robust logistic regression, in
  {\em Advances in Neural Information Processing Systems}, pp.~3937--3947.

\bibitem[Li {\em et~al.}(2022)Li, Lin, Blanchet and Nguyen]{li2022tikhonov}
J.~Li, S.~Lin, J.~Blanchet and V.~A. Nguyen  (2022), Tikhonov regularization is
  optimal transport robust under martingale constraints, in {\em Advances in
  Neural Information Processing Systems}, pp.~17677--17689.

\bibitem[Li(2018)Li]{li2018closed}
J.~Y.-M. Li  (2018), Closed-form solutions for worst-case law invariant risk
  measures with application to robust portfolio optimization, {\em Operations
  Research} {\bf 66}(6), 1533--1541.

\bibitem[Li and Mao(2022)Li and Mao]{li2022general}
J.~Y.-M. Li and T.~Mao  (2022), A general {W}asserstein framework for
  data-driven distributionally robust optimization: {T}ractability and
  applications, {\em arXiv:2207.09403}.

\bibitem[Li {\em et~al.}(2021)Li, Sutter and Kuhn]{li2021distributionally}
M.~Li, T.~Sutter and D.~Kuhn  (2021), Distributionally robust optimization with
  {M}arkovian data, in {\em International Conference on Machine Learning},
  pp.~6493--6503.

\bibitem[Li {\em et~al.}(2011)Li, Ding and Floudas]{li2011comparative}
Z.~Li, R.~Ding and C.~A. Floudas  (2011), A comparative theoretical and
  computational study on robust counterpart optimization: {I.} {R}obust linear
  optimization and robust mixed integer linear optimization, {\em Industrial \&
  Engineering Chemistry Research} {\bf 50}(18), 10567--10603.

\bibitem[Liese and Vajda(1987)Liese and Vajda]{liese1987convex}
F.~Liese and I.~Vajda  (1987), {\em Convex Statistical Distances}, Teubner.

\bibitem[Lin {\em et~al.}(2024)Lin, Blanchet, Glynn and Nguyen]{lin2024small}
S.~Lin, J.~Blanchet, P.~Glynn and V.~A. Nguyen  (2024), Small sample behavior
  of {W}asserstein projections, connections to empirical likelihood, and other
  applications, {\em arXiv:2408.11753}.

\bibitem[Liu {\em et~al.}(2024{\em a})Liu, Chen, Wang and
  Wang]{liu2024newsvendor}
F.~Liu, Z.~Chen, R.~Wang and S.~Wang  (2024{\em a}), Newsvendor under
  mean-variance ambiguity and misspecification, {\em arXiv:2405.07008}.

\bibitem[Liu {\em et~al.}(2024{\em b})Liu, Su and Xu]{liu2024bayesian}
J.~Liu, Z.~Su and H.~Xu  (2024{\em b}), Bayesian distributionally robust {N}ash
  equilibrium and its application, {\em arXiv:2410.20364}.

\bibitem[Liu and Loh(2023)Liu and Loh]{liu2023robust}
Z.~Liu and P.-L. Loh  (2023), Robust {W-GAN}-based estimation under
  {W}asserstein contamination, {\em Information and Inference: A Journal of the
  IMA} {\bf 12}(1), 312--362.

\bibitem[Liu {\em et~al.}(2023)Liu, Van~Parys and Lam]{liu2023smoothed}
Z.~Liu, B.~P. Van~Parys and H.~Lam  (2023), Smoothed $f$-divergence
  distributionally robust optimization: Exponential rate efficiency and
  complexity-free calibration, {\em arXiv:2306.14041}.

\bibitem[Long {\em et~al.}(2024)Long, Qi and Zhang]{long2024supermodularity}
D.~Z. Long, J.~Qi and A.~Zhang  (2024), Supermodularity in two-stage
  distributionally robust optimization, {\em Management Science} {\bf 70}(3),
  1394--1409.

\bibitem[Lyu {\em et~al.}(2015)Lyu, Huang and Liang]{lyu2015unified}
C.~Lyu, K.~Huang and H.-N. Liang  (2015), A unified gradient regularization
  family for adversarial examples, in {\em International Conference on Data
  Mining}, pp.~301--309.

\bibitem[Madansky(1959)Madansky]{madansky:59}
A.~Madansky  (1959), Bounds on the expectation of a convex function of a
  multivariate random variable, {\em The Annals of Mathematical Statistics}
  {\bf 30}(3), 743--746.

\bibitem[M{\k{a}}dry {\em et~al.}(2018)M{\k{a}}dry, Makelov, Schmidt, Tsipras
  and Vladu]{madry2018towards}
A.~M{\k{a}}dry, A.~Makelov, L.~Schmidt, D.~Tsipras and A.~Vladu  (2018),
  Towards deep learning models resistant to adversarial attacks, in {\em
  International Conference on Learning Representations}.

\bibitem[Maheshwari {\em et~al.}(2022)Maheshwari, Chiu, Mazumdar, Sastry and
  Ratliff]{maheshwari2022zeroth}
C.~Maheshwari, C.-Y. Chiu, E.~Mazumdar, S.~Sastry and L.~Ratliff  (2022),
  Zeroth-order methods for convex-concave min-max problems: {A}pplications to
  decision-dependent risk minimization, in {\em International Conference on
  Artificial Intelligence and Statistics}, pp.~6702--6734.

\bibitem[Mak {\em et~al.}(2015)Mak, Rong and Zhang]{mak2015appointment}
H.-Y. Mak, Y.~Rong and J.~Zhang  (2015), Appointment scheduling with limited
  distributional information, {\em Management Science} {\bf 61}(2), 316--334.

\bibitem[Markov(1884)Markov]{markov1884certain}
A.~Markov  (1884), On certain applications of algebraic continued fractions,
  PhD thesis, St Petersburg (in Russian).

\bibitem[Marshall and Olkin(1960)Marshall and Olkin]{marshall1960one}
A.~W. Marshall and I.~Olkin  (1960), A one-sided inequality of the {C}hebyshev
  type, {\em The Annals of Mathematical Statistics} {\bf 31}(2), 488--491.

\bibitem[Marton(1986)Marton]{marton1986simple}
K.~Marton  (1986), A simple proof of the blowing-up lemma, {\em IEEE
  Transactions on Information Theory} {\bf 32}(3), 445--446.

\bibitem[Maurer and Pontil(2009)Maurer and Pontil]{maurer2009empirical}
A.~Maurer and M.~Pontil  (2009), Empirical {B}ernstein bounds and sample
  variance penalization, in {\em Conference on Learning Theory}.

\bibitem[McAllister and Esfahani(2024)McAllister and
  Esfahani]{mcallister2023distributionally}
R.~D. McAllister and P.~M. Esfahani  (2024), Distributionally robust model
  predictive control: {C}losed-loop guarantees and scalable algorithms, {\em
  IEEE Transactions on Automatic Control (in press)} pp.~1--16.

\bibitem[McNeil {\em et~al.}(2015)McNeil, Frey and
  Embrechts]{mcneil2015quantitative}
A.~McNeil, R.~Frey and P.~Embrechts  (2015), {\em Quantitative Risk Management:
  Concepts, Techniques and Tools}, Princeton University Press.

\bibitem[Mendelson(2003)Mendelson]{mendelson2003few}
S.~Mendelson  (2003), A few notes on statistical learning theory, in {\em
  Advanced Lectures on Machine Learning} (S.~Mendelson and A.~J. Smola, eds),
  Springer, pp.~1--40.

\bibitem[Michaud(1989)Michaud]{michaud1989markowitz}
R.~O. Michaud  (1989), The {M}arkowitz optimization enigma: Is ‘optimized’
  optimal?, {\em Financial Analysts Journal} {\bf 45}(1), 31--42.

\bibitem[Milz and Ulbrich(2020)Milz and Ulbrich]{milz2020approximation}
J.~Milz and M.~Ulbrich  (2020), An approximation scheme for distributionally
  robust nonlinear optimization, {\em SIAM Journal on Optimization} {\bf
  30}(3), 1996--2025.

\bibitem[Milz and Ulbrich(2022)Milz and Ulbrich]{milz2022approximation}
J.~Milz and M.~Ulbrich  (2022), An approximation scheme for distributionally
  robust {PDE}-constrained optimization, {\em SIAM Journal on Control and
  Optimization} {\bf 60}(3), 1410--1435.

\bibitem[Mishra {\em et~al.}(2014)Mishra, Natarajan, Padmanabhan, Teo and
  Li]{mishra2014theoretical}
V.~K. Mishra, K.~Natarajan, D.~Padmanabhan, C.-P. Teo and X.~Li  (2014), On
  theoretical and empirical aspects of marginal distribution choice models,
  {\em Management Science} {\bf 60}(6), 1511--1531.

\bibitem[Mishra {\em et~al.}(2012)Mishra, Natarajan, Tao and
  Teo]{mishra2012choice}
V.~K. Mishra, K.~Natarajan, H.~Tao and C.-P. Teo  (2012), Choice prediction
  with semidefinite optimization when utilities are correlated, {\em IEEE
  Transactions on Automatic Control} {\bf 57}(10), 2450--2463.

\bibitem[Mohajerin~Esfahani and Kuhn(2018)Mohajerin~Esfahani and
  Kuhn]{mohajerin2018data}
P.~Mohajerin~Esfahani and D.~Kuhn  (2018), Data-driven distributionally robust
  optimization using the {W}asserstein metric: {P}erformance guarantees and
  tractable reformulations, {\em Mathematical Programming} {\bf 171}(1),
  115--166.

\bibitem[Mohajerin~Esfahani {\em et~al.}(2018)Mohajerin~Esfahani,
  Shafieezadeh-Abadeh, Hanasusanto and Kuhn]{esfahani2018inverse}
P.~Mohajerin~Esfahani, S.~Shafieezadeh-Abadeh, G.~A. Hanasusanto and D.~Kuhn
  (2018), Data-driven inverse optimization with imperfect information, {\em
  Mathematical Programming} {\bf 167}(1), 191--234.

\bibitem[{Mohajerin Esfahani} {\em et~al.}(2015){Mohajerin Esfahani}, Sutter
  and Lygeros]{MESL15:performance-bounds}
P.~{Mohajerin Esfahani}, T.~Sutter and J.~Lygeros  (2015), Performance bounds
  for the scenario approach and an extension to a class of non-convex programs,
  {\em {IEEE} Transactions on Automatic Control} {\bf 60}(1), 46 -- 58.

\bibitem[Munkres(2000)Munkres]{munkres2000topology}
J.~R. Munkres  (2000), {\em Topology}, Prentice Hall.

\bibitem[Mutapcic and Boyd(2009)Mutapcic and Boyd]{mutapcic2009cutting}
A.~Mutapcic and S.~Boyd  (2009), Cutting-set methods for robust convex
  optimization with pessimizing oracles, {\em Optimization Methods \& Software}
  {\bf 24}(3), 381--406.

\bibitem[Nagarajan and Kolter(2017)Nagarajan and Kolter]{nagarajan2017gradient}
V.~Nagarajan and J.~Z. Kolter  (2017), {Gradient descent GAN optimization is
  locally stable}, in {\em Advances in Neural Information Processing Systems},
  pp.~5591--5600.

\bibitem[Nakao {\em et~al.}(2021)Nakao, Jiang and
  Shen]{nakao2021distributionally}
H.~Nakao, R.~Jiang and S.~Shen  (2021), Distributionally robust partially
  observable {M}arkov decision process with moment-based ambiguity, {\em SIAM
  Journal on Optimization} {\bf 31}(1), 461--488.

\bibitem[Namkoong and Duchi(2016)Namkoong and Duchi]{namkoong2016stochastic}
H.~Namkoong and J.~C. Duchi  (2016), Stochastic gradient methods for
  distributionally robust optimization with $f$-divergences, in {\em Advances
  in Neural Information Processing Systems}, pp.~2216--2224.

\bibitem[Natarajan(2021)Natarajan]{natarajan2021optimization}
K.~Natarajan  (2021), {\em Optimization with Marginals and Moments}, Dynamic
  Ideas.

\bibitem[Natarajan and Linyi(2007)Natarajan and Linyi]{natarajan2007mean}
K.~Natarajan and Z.~Linyi  (2007), A mean--variance bound for a three-piece
  linear function, {\em Probability in the Engineering and Informational
  Sciences} {\bf 21}(4), 611--621.

\bibitem[Natarajan {\em et~al.}(2009{\em a})Natarajan, Pachamanova and
  Sim]{natarajan2009constructing}
K.~Natarajan, D.~Pachamanova and M.~Sim  (2009{\em a}), Constructing risk
  measures from uncertainty sets, {\em Operations Research} {\bf 57}(5),
  1129--1141.

\bibitem[Natarajan {\em et~al.}(2023)Natarajan, Padmanabhan and
  Ramachandra]{natarajan2023distributionally}
K.~Natarajan, D.~Padmanabhan and A.~Ramachandra  (2023), Distributionally
  robust optimization through the lens of submodularity, {\em
  arXiv:2312.04890}.

\bibitem[Natarajan {\em et~al.}(2010)Natarajan, Sim and
  Uichanco]{natarajan2010utility}
K.~Natarajan, M.~Sim and J.~Uichanco  (2010), Tractable robust expected utility
  and risk models for portfolio optimization, {\em Mathematical Finance} {\bf
  20}(4), 695--731.

\bibitem[Natarajan {\em et~al.}(2018)Natarajan, Sim and
  Uichanco]{natarajan2018asymmetry}
K.~Natarajan, M.~Sim and J.~Uichanco  (2018), Asymmetry and ambiguity in
  newsvendor models, {\em Management Science} {\bf 64}(7), 3146--3167.

\bibitem[Natarajan {\em et~al.}(2009{\em b})Natarajan, Song and
  Teo]{natarajan2009persistency}
K.~Natarajan, M.~Song and C.-P. Teo  (2009{\em b}), Persistency model and its
  applications in choice modeling, {\em Management Science} {\bf 55}(3),
  453--469.

\bibitem[Natarajan {\em et~al.}(2011)Natarajan, Teo and
  Zheng]{natarajan2011mixed}
K.~Natarajan, C.~P. Teo and Z.~Zheng  (2011), Mixed 0-1 linear programs under
  objective uncertainty: {A} completely positive representation, {\em
  Operations Research} {\bf 59}(3), 713--728.

\bibitem[Nemirovski and Shapiro(2007)Nemirovski and
  Shapiro]{nemirovski2007convex}
A.~Nemirovski and A.~Shapiro  (2007), Convex approximations of chance
  constrained programs, {\em SIAM Journal on Optimization} {\bf 17}(4),
  969--996.

\bibitem[Nesterov and Nemirovskii(1994)Nesterov and
  Nemirovskii]{doi:10.1137/1.9781611970791}
Y.~Nesterov and A.~Nemirovskii  (1994), {\em Interior-Point Polynomial
  Algorithms in Convex Programming}, SIAM.

\bibitem[Nguyen {\em et~al.}(2022{\em a})Nguyen, Bui and
  Nguyen]{nguyen2022distributionally}
D.~Nguyen, N.~Bui and V.~A. Nguyen  (2022{\em a}), Distributionally robust
  recourse action, in {\em International Conference on Learning
  Representations}.

\bibitem[Nguyen {\em et~al.}(2022{\em b})Nguyen, Kuhn and
  Mohajerin~Esfahani]{nguyen2020distributionally}
V.~A. Nguyen, D.~Kuhn and P.~Mohajerin~Esfahani  (2022{\em b}),
  Distributionally robust inverse covariance estimation: The {W}asserstein
  shrinkage estimator, {\em Operations Research} {\bf 70}(1), 490--515.

\bibitem[Nguyen {\em et~al.}(2021)Nguyen, Shafiee, Filipovi{\'c} and
  Kuhn]{nguyen2021mean}
V.~A. Nguyen, S.~Shafiee, D.~Filipovi{\'c} and D.~Kuhn  (2021), Mean-covariance
  robust risk measurement, {\em arXiv:2112.09959}.

\bibitem[Nguyen {\em et~al.}(2023)Nguyen, Shafieezadeh-Abadeh, Kuhn and
  Mohajerin~Esfahani]{nguyen2019bridging}
V.~A. Nguyen, S.~Shafieezadeh-Abadeh, D.~Kuhn and P.~Mohajerin~Esfahani
  (2023), Bridging bayesian and minimax mean square error estimation via
  {W}asserstein distributionally robust optimization, {\em Mathematics of
  Operations Research} {\bf 48}(1), 1--37.

\bibitem[Nguyen {\em et~al.}(2019)Nguyen, Shafieezadeh-Abadeh, Yue, Kuhn and
  Wiesemann]{nguyen2019optimistic}
V.~A. Nguyen, S.~Shafieezadeh-Abadeh, M.-C. Yue, D.~Kuhn and W.~Wiesemann
  (2019), Optimistic distributionally robust optimization for nonparametric
  likelihood approximation, in {\em Advances in Neural Information Processing
  Systems}, pp.~15872--15882.

\bibitem[Nguyen {\em et~al.}(2020)Nguyen, Zhang, Blanchet, Delage and
  Ye]{nguyen2020distributionallyrobust}
V.~A. Nguyen, F.~Zhang, J.~Blanchet, E.~Delage and Y.~Ye  (2020),
  Distributionally robust local non-parametric conditional estimation, in {\em
  Advances in Neural Information Processing Systems}, pp.~15232--15242.

\bibitem[Nguyen {\em et~al.}(2024)Nguyen, Zhang, Wang, Blanchet, Delage and
  Ye]{nguyen2021robustifying}
V.~A. Nguyen, F.~Zhang, S.~Wang, J.~Blanchet, E.~Delage and Y.~Ye  (2024),
  Robustifying conditional portfolio decisions via optimal transport, {\em
  Operations Research (Forthcoming)}.

\bibitem[Nietert {\em et~al.}(2024{\em a})Nietert, Goldfeld and
  Shafiee]{nietert2024outlier}
S.~Nietert, Z.~Goldfeld and S.~Shafiee  (2024{\em a}), Outlier-robust
  {W}asserstein {DRO}, in {\em Advances in Neural Information Processing
  Systems}, pp.~62792--62820.

\bibitem[Nietert {\em et~al.}(2024{\em b})Nietert, Goldfeld and
  Shafiee]{nietert2024robust}
S.~Nietert, Z.~Goldfeld and S.~Shafiee  (2024{\em b}), Robust distribution
  learning with local and global adversarial corruptions, in {\em Conference on
  Learning Theory}, pp.~4007--4008.

\bibitem[Nishimura and Ozaki(2004)Nishimura and Ozaki]{nishimura2004search}
K.~G. Nishimura and H.~Ozaki  (2004), Search and {K}nightian uncertainty, {\em
  Journal of Economic Theory} {\bf 119}(2), 299--333.

\bibitem[Nishimura and Ozaki(2006)Nishimura and Ozaki]{nishimura2006axiomatic}
K.~G. Nishimura and H.~Ozaki  (2006), An axiomatic approach to-contamination,
  {\em Economic Theory} {\bf 27}(2), 333--340.

\bibitem[Olea {\em et~al.}(2022)Olea, Rush, Velez and Wiesel]{olea2022out}
J.~L.~M. Olea, C.~Rush, A.~Velez and J.~Wiesel  (2022), The out-of-sample
  prediction error of the square-root-{LASSO} and related estimators, {\em
  arXiv:2211.07608}.

\bibitem[Olkin and Pukelsheim(1982)Olkin and Pukelsheim]{olkin1982distance}
I.~Olkin and F.~Pukelsheim  (1982), The distance between two random vectors
  with given dispersion matrices, {\em Linear Algebra and its Applications}
  {\bf 48}, 257--263.

\bibitem[Ordoudis {\em et~al.}(2021)Ordoudis, Nguyen, Kuhn and
  Pinson]{ordoudis2021energy}
C.~Ordoudis, V.~A. Nguyen, D.~Kuhn and P.~Pinson  (2021), Energy and reserve
  dispatch with distributionally robust joint chance constraints, {\em
  Operations Research Letters} {\bf 49}(3), 291--299.

\bibitem[Owen(1988)Owen]{owen1988empirical}
A.~B. Owen  (1988), Empirical likelihood ratio confidence intervals for a
  single functional, {\em Biometrika} {\bf 75}(2), 237--249.

\bibitem[Owen(1990)Owen]{owen1990empirical}
A.~B. Owen  (1990), Empirical likelihood ratio confidence regions, {\em The
  Annals of Statistics} {\bf 18}(1), 90--120.

\bibitem[Owen(1991)Owen]{owen1991empirical}
A.~B. Owen  (1991), Empirical likelihood for linear models, {\em The Annals of
  Statistics} {\bf 19}(4), 1725--1747.

\bibitem[Owen(2001)Owen]{owen2001empirical}
A.~B. Owen  (2001), {\em Empirical Likelihood}, Chapman and Hall.

\bibitem[Owhadi and Scovel(2017)Owhadi and Scovel]{owhadi2017extreme}
H.~Owhadi and C.~Scovel  (2017), Extreme points of a ball about a measure with
  finite support, {\em Communications in Mathematical Sciences} {\bf 15}(1),
  77--96.

\bibitem[Owhadi {\em et~al.}(2013)Owhadi, Scovel, Sullivan, McKerns and
  Ortiz]{owhadi2013optimal}
H.~Owhadi, C.~Scovel, T.~J. Sullivan, M.~McKerns and M.~Ortiz  (2013), Optimal
  uncertainty quantification, {\em {SIAM} Review} {\bf 55}(2), 271--345.

\bibitem[Panaretos and Zemel(2020)Panaretos and Zemel]{panaretos2020invitation}
V.~M. Panaretos and Y.~Zemel  (2020), {\em An Invitation to Statistics in
  {W}asserstein Space}, Springer.

\bibitem[Parrilo(2000)Parrilo]{parrilo2000structured}
P.~A. Parrilo  (2000), Structured Semidefinite Programs and Semialgebraic
  Geometry Methods in Robustness and Optimization, PhD thesis, California
  Institute of Technology.

\bibitem[Parrilo(2003)Parrilo]{parrilo2003semidefinite}
P.~A. Parrilo  (2003), Semidefinite programming relaxations for semialgebraic
  problems, {\em Mathematical Programming} {\bf 96}(2), 293--320.

\bibitem[Pass(2015)Pass]{pass2015multi}
B.~Pass  (2015), Multi-marginal optimal transport: {T}heory and applications,
  {\em ESAIM: Mathematical Modelling and Numerical Analysis} {\bf 49}(6),
  1771--1790.

\bibitem[Peng(1997)Peng]{peng1997backward}
S.~Peng  (1997), Backward {SDE} and related {G}-expectation, in {\em Backward
  Stochastic Differential Equations in Finance} (N.~El~Karoui, S.~Peng and
  M.~C. Quenez, eds), Wiley, pp.~141--160.

\bibitem[Peng(2007{\em a})Peng]{peng2007gB}
S.~Peng  (2007{\em a}), {G}-{B}rownian motion and dynamic risk measure under
  volatility uncertainty, {\em arXiv:0711.2834}.

\bibitem[Peng(2007{\em b})Peng]{peng2007g}
S.~Peng  (2007{\em b}), {G}-expectation, {G}-{B}rownian motion and related
  stochastic calculus of {I}t{\^o} type, in {\em Stochastic Analysis and
  Applications} (F.~E. Benth, G.~Di~Nunno, T.~Lindstrom, B.~Oksendal and
  T.~Zhang, eds), Springer, pp.~541--567.

\bibitem[Peng(2019)Peng]{peng2019nonlinear}
S.~Peng  (2019), {\em Nonlinear Expectations and Stochastic Calculus under
  Uncertainty: With Robust {CLT} and {G}-{B}rownian Motion}, Springer.

\bibitem[Peng(2023)Peng]{peng2011gg}
S.~Peng  (2023), {G}-{G}aussian processes under sublinear expectations and
  q-{B}rownian motion in quantum mechanics, {\em Numerical Algebra, Control and
  Optimization} {\bf 13}(3-4), 583--603.

\bibitem[Perakis and Roels(2008)Perakis and Roels]{perakis2008regret}
G.~Perakis and G.~Roels  (2008), Regret in the newsvendor model with partial
  information, {\em Operations Research} {\bf 56}(1), 188--203.

\bibitem[Pesenti {\em et~al.}(2024)Pesenti, Wang and
  Wang]{pesenti2020optimizing}
S.~Pesenti, Q.~Wang and R.~Wang  (2024), Optimizing distortion riskmetrics with
  distributional uncertainty, {\em Mathematical Programming (Forthcoming)}.

\bibitem[Pflug and Pichler(2014)Pflug and Pichler]{pflug2014multistage}
G.~C. Pflug and A.~Pichler  (2014), {\em Multistage Stochastic Optimization},
  Springer.

\bibitem[Pflug and Wozabal(2007)Pflug and Wozabal]{pflug2007ambiguity}
G.~C. Pflug and D.~Wozabal  (2007), Ambiguity in portfolio selection, {\em
  Quantitative Finance} {\bf 7}(4), 435--442.

\bibitem[Pflug {\em et~al.}(2012)Pflug, Pichler and Wozabal]{pflug2012}
G.~C. Pflug, A.~Pichler and D.~Wozabal  (2012), The {$1/N$} investment strategy
  is optimal under high model ambiguity, {\em Journal of Banking \& Finance}
  {\bf 36}(2), 410--417.

\bibitem[Phelps(1965)Phelps]{phelps2001lectures}
R.~R. Phelps  (1965), {\em Lectures on {C}hoquet’s Theorem}, van Nostrand
  Mathematical Studies.

\bibitem[Philpott {\em et~al.}(2018)Philpott, de~Matos and
  Kapelevich]{philpott2018distributionally}
A.~B. Philpott, V.~L. de~Matos and L.~Kapelevich  (2018), Distributionally
  robust {SDDP}, {\em Computational Management Science} {\bf 15}, 431--454.

\bibitem[Pichler(2013)Pichler]{pichler:2013}
A.~Pichler  (2013), Evaluations of risk measures for different probability
  measures, {\em SIAM Journal on Optimization} {\bf 23}(1), 530--551.

\bibitem[Pinelis(2016)Pinelis]{pinelis2016extreme}
I.~Pinelis  (2016), On the extreme points of moments sets, {\em Mathematical
  Methods of Operations Research} {\bf 83}(3), 325--349.

\bibitem[P{\'o}lik and Terlaky(2007)P{\'o}lik and Terlaky]{polik2007survey}
I.~P{\'o}lik and T.~Terlaky  (2007), A survey of the {S}-lemma, {\em {SIAM}
  Review} {\bf 49}(3), 371--418.

\bibitem[Polyanskiy and Wu(2024)Polyanskiy and Wu]{polyanskiy2024information}
Y.~Polyanskiy and Y.~Wu  (2024), {\em Information Theory: From Coding to
  Learning}, Cambridge University Press.

\bibitem[Popescu(2005)Popescu]{popescu2005semidefinite}
I.~Popescu  (2005), A semidefinite programming approach to optimal-moment
  bounds for convex classes of distributions, {\em Mathematics of Operations
  Research} {\bf 30}(3), 632--657.

\bibitem[Popescu(2007)Popescu]{popescu2007robust}
I.~Popescu  (2007), Robust mean-covariance solutions for stochastic
  optimization, {\em Operations Research} {\bf 55}(1), 98--112.

\bibitem[Postek and Shtern(2024)Postek and Shtern]{postek2024first}
K.~Postek and S.~Shtern  (2024), First-order algorithms for robust optimization
  problems via convex-concave saddle-point {L}agrangian reformulation, {\em
  INFORMS Journal on Computing (Forthcoming)}.

\bibitem[Postek {\em et~al.}(2018)Postek, Ben-Tal, den Hertog and
  Melenberg]{postek2018robust}
K.~Postek, A.~Ben-Tal, D.~den Hertog and B.~Melenberg  (2018), Robust
  optimization with ambiguous stochastic constraints under mean and dispersion
  information, {\em Operations Research} {\bf 66}(3), 814--833.

\bibitem[Postek {\em et~al.}(2016)Postek, den Hertog and
  Melenberg]{postek2016computationally}
K.~Postek, D.~den Hertog and B.~Melenberg  (2016), Computationally tractable
  counterparts of distributionally robust constraints on risk measures, {\em
  {SIAM} Review} {\bf 58}(4), 603--650.

\bibitem[Postek {\em et~al.}(2019)Postek, Romeijnders, den Hertog and van~der
  Vlerk]{postek2019approximation}
K.~Postek, W.~Romeijnders, D.~den Hertog and M.~H. van~der Vlerk  (2019), An
  approximation framework for two-stage ambiguous stochastic integer programs
  under mean-{MAD} information, {\em European Journal of Operational Research}
  {\bf 274}(2), 432--444.

\bibitem[Puccetti and R{\"u}schendorf(2013)Puccetti and
  R{\"u}schendorf]{puccetti2013sharp}
G.~Puccetti and L.~R{\"u}schendorf  (2013), Sharp bounds for sums of dependent
  risks, {\em Journal of Applied Probability} {\bf 50}(1), 42--53.

\bibitem[Pydi and Jog(2021)Pydi and Jog]{pydi2021adversarial}
M.~S. Pydi and V.~Jog  (2021), Adversarial risk via optimal transport and
  optimal couplings, {\em IEEE Transactions on Information Theory} {\bf 67}(9),
  6031--6052.

\bibitem[Pydi and Jog(2024)Pydi and Jog]{pydi2024many}
M.~S. Pydi and V.~Jog  (2024), The many faces of adversarial risk: {A}n
  expanded study, {\em IEEE Transactions on Information Theory} {\bf 70}(1),
  550--570.

\bibitem[Rahimian and Mehrotra(2022)Rahimian and
  Mehrotra]{rahimian2022frameworks}
H.~Rahimian and S.~Mehrotra  (2022), Frameworks and results in distributionally
  robust optimization, {\em Open Journal of Mathematical Optimization} {\bf 3},
  1--85.

\bibitem[Rahimian {\em et~al.}(2019{\em a})Rahimian, Bayraksan and
  {Homem-de-Mello}]{rahimian2019controlling}
H.~Rahimian, G.~Bayraksan and T.~{Homem-de-Mello}  (2019{\em a}), Controlling
  risk and demand ambiguity in newsvendor models, {\em European Journal of
  Operational Research} {\bf 279}(3), 854--868.

\bibitem[Rahimian {\em et~al.}(2019{\em b})Rahimian, Bayraksan and
  {Homem-de-Mello}]{rahimian2019identifying}
H.~Rahimian, G.~Bayraksan and T.~{Homem-de-Mello}  (2019{\em b}), Identifying
  effective scenarios in distributionally robust stochastic programs with total
  variation distance, {\em Mathematical Programming} {\bf 173}(1), 393--430.

\bibitem[Rahimian {\em et~al.}(2022)Rahimian, Bayraksan and
  {Homem-de-Mello}]{rahimian2022effective}
H.~Rahimian, G.~Bayraksan and T.~{Homem-de-Mello}  (2022), Effective scenarios
  in multistage distributionally robust optimization with a focus on total
  variation distance, {\em SIAM Journal on Optimization} {\bf 32}(3),
  1698--1727.

\bibitem[Reid and Williamson(2011)Reid and Williamson]{reid2011information}
M.~D. Reid and R.~C. Williamson  (2011), Information, divergence and risk for
  binary experiments, {\em Journal of Machine Learning Research} {\bf 12}(22),
  731--817.

\bibitem[Richter(1957)Richter]{richter1957parameterfreie}
H.~Richter  (1957), {Parameterfreie Absch{\"a}tzung und Realisierung von
  Erwartungswerten}, {\em Bl{\"a}tter der DGVFM} {\bf 3}(2), 147--162.

\bibitem[Rockafellar(1970)Rockafellar]{rockafellar1970convex}
R.~T. Rockafellar  (1970), {\em Convex Analysis}, Princeton University Press.

\bibitem[Rockafellar(1974)Rockafellar]{rockafellar1974conjugate}
R.~T. Rockafellar  (1974), {\em Conjugate Duality and Optimization}, SIAM.

\bibitem[Rockafellar and Royset(2013)Rockafellar and
  Royset]{rockafellar2013superquantiles}
R.~T. Rockafellar and J.~O. Royset  (2013), Superquantiles and their
  applications to risk, random variables, and regression, {\em {INFORMS}
  Tutorials in Operations Research} pp.~151--167.

\bibitem[Rockafellar and Royset(2014)Rockafellar and
  Royset]{rockafellar2014random}
R.~T. Rockafellar and J.~O. Royset  (2014), Random variables, monotone
  relations, and convex analysis, {\em Mathematical Programming} {\bf
  148}(1-2), 297--331.

\bibitem[Rockafellar and Royset(2015)Rockafellar and
  Royset]{rockafellar2015measures}
R.~T. Rockafellar and J.~O. Royset  (2015), Measures of residual risk with
  connections to regression, risk tracking, surrogate models, and ambiguity,
  {\em SIAM Journal on Optimization} {\bf 25}(2), 1179--1208.

\bibitem[Rockafellar and Uryasev(2000)Rockafellar and
  Uryasev]{rockafellar2000optimization}
R.~T. Rockafellar and S.~Uryasev  (2000), Optimization of conditional
  value-at-risk, {\em Journal of Risk} {\bf 2}(3), 21--41.

\bibitem[Rockafellar and Uryasev(2002)Rockafellar and
  Uryasev]{rockafellar2002cvar-general}
R.~T. Rockafellar and S.~Uryasev  (2002), Conditional value-at-risk for general
  loss distributions, {\em Journal of Banking \& Finance} {\bf 26}(7),
  1443--1471.

\bibitem[Rockafellar and Uryasev(2013)Rockafellar and
  Uryasev]{rockafellar2013fundamental}
R.~T. Rockafellar and S.~Uryasev  (2013), The fundamental risk quadrangle in
  risk management, optimization and statistical estimation, {\em Surveys in
  Operations Research and Management Science} {\bf 18}(1-2), 33--53.

\bibitem[Rockafellar and Wets(2009)Rockafellar and
  Wets]{rockafellar2009variational}
R.~T. Rockafellar and R.~J.-B. Wets  (2009), {\em Variational Analysis},
  Springer.

\bibitem[Rockafellar {\em et~al.}(2006)Rockafellar, Uryasev and
  Zabarankin]{rockafellar2006generalized}
R.~T. Rockafellar, S.~Uryasev and M.~Zabarankin  (2006), Generalized deviations
  in risk analysis, {\em Finance and Stochastics} {\bf 10}(1), 51--74.

\bibitem[Rockafellar {\em et~al.}(2008)Rockafellar, Uryasev and
  Zabarankin]{rockafellar2008risk}
R.~T. Rockafellar, S.~Uryasev and M.~Zabarankin  (2008), Risk tuning with
  generalized linear regression, {\em Mathematics of Operations Research} {\bf
  33}(3), 712--729.

\bibitem[Rogosinski(1958)Rogosinski]{rogosinski1958moments}
W.~W. Rogosinski  (1958), Moments of non-negative mass, {\em Proceedings of the
  Royal Society of London. Series A. Mathematical and Physical Sciences} {\bf
  245}(1240), 1--27.

\bibitem[Rontsis {\em et~al.}(2020)Rontsis, Osborne and
  Goulart]{rontsis2020distributionally}
N.~Rontsis, M.~A. Osborne and P.~J. Goulart  (2020), Distributionally ambiguous
  optimization for batch {B}ayesian optimization, {\em Journal of Machine
  Learning Research} {\bf 21}(149), 1--26.

\bibitem[Roth {\em et~al.}(2017)Roth, Lucchi, Nowozin and
  Hofmann]{roth2017stabilizing}
K.~Roth, A.~Lucchi, S.~Nowozin and T.~Hofmann  (2017), Stabilizing training of
  generative adversarial networks through regularization, in {\em Advances in
  Neural Information Processing Systems}, pp.~2018--2028.

\bibitem[Royset(2022)Royset]{royset2022risk}
J.~O. Royset  (2022), Risk-adaptive approaches to learning and decision making:
  {A} survey, {\em arXiv:2212.00856}.

\bibitem[Ruan {\em et~al.}(2023)Ruan, Li, Murthy and
  Natarajan]{ruan2022nonparametric}
Y.~Ruan, X.~Li, K.~Murthy and K.~Natarajan  (2023), A nonparametric approach
  with marginals for modeling consumer choice, in {\em Conference on Economics
  and Computation}, p.~1078.

\bibitem[Rujeerapaiboon {\em et~al.}(2016)Rujeerapaiboon, Kuhn and
  Wiesemann]{rujeerapaiboon2016robust}
N.~Rujeerapaiboon, D.~Kuhn and W.~Wiesemann  (2016), Robust growth-optimal
  portfolios, {\em Management Science} {\bf 62}(7), 2090--2109.

\bibitem[Rujeerapaiboon {\em et~al.}(2018)Rujeerapaiboon, Kuhn and
  Wiesemann]{rujeerapaiboon2018chebyshev}
N.~Rujeerapaiboon, D.~Kuhn and W.~Wiesemann  (2018), Chebyshev inequalities for
  products of random variables, {\em Mathematics of Operations Research} {\bf
  43}(3), 887--918.

\bibitem[R{\"u}schendorf(1983)R{\"u}schendorf]{ruschendorf1983solution}
L.~R{\"u}schendorf  (1983), Solution of a statistical optimization problem by
  rearrangement methods, {\em Metrika} {\bf 30}(1), 55--61.

\bibitem[R{\"u}schendorf(1991)R{\"u}schendorf]{ruschendorf1991frechet}
L.~R{\"u}schendorf  (1991), Fr{\'e}chet-bounds and their applications, in {\em
  Advances in Probability Distributions with Given Marginals: Beyond the
  Copulas} (G.~Dall'Aglio, S.~Kotz and G.~Salinetti, eds), Springer,
  pp.~151--187.

\bibitem[R{\"u}schendorf(2013)R{\"u}schendorf]{ruschendorf2013mathematical}
L.~R{\"u}schendorf  (2013), {\em Mathematical Risk Analysis: Dependence, Risk
  Bounds, Optimal Allocations and Portfolios}, Springer.

\bibitem[Rustem and Howe(2009)Rustem and Howe]{rh03}
B.~Rustem and M.~Howe  (2009), {\em Algorithms for Worst-Case Design and
  Applications to Risk Management}, Princeton University Press.

\bibitem[Ruszczy{\'n}ski(2021)Ruszczy{\'n}ski]{ruszczynski2021stochastic}
A.~Ruszczy{\'n}ski  (2021), A stochastic subgradient method for nonsmooth
  nonconvex multilevel composition optimization, {\em SIAM Journal on Control
  and Optimization} {\bf 59}(3), 2301--2320.

\bibitem[Ruszczy{\'n}ski and Shapiro(2006)Ruszczy{\'n}ski and
  Shapiro]{ref:ruszczynski2006risk}
A.~Ruszczy{\'n}ski and A.~Shapiro  (2006), Optimization of convex risk
  functions, {\em Mathematics of Operations Research} {\bf 31}(3), 433--452.

\bibitem[Rychener {\em et~al.}(2024)Rychener, Esteban-P{\'e}rez, Morales and
  Kuhn]{rychener2024wasserstein}
Y.~Rychener, A.~Esteban-P{\'e}rez, J.~M. Morales and D.~Kuhn  (2024),
  {W}asserstein distributionally robust optimization with heterogeneous data
  sources, {\em arXiv:2407.13582}.

\bibitem[Sadana {\em et~al.}(2024)Sadana, Delage and Georghiou]{sadana2024data}
U.~Sadana, E.~Delage and A.~Georghiou  (2024), Data-driven decision-making
  under uncertainty with entropic risk measure, {\em arXiv:2409.19926}.

\bibitem[Sagawa {\em et~al.}(2020)Sagawa, Koh, Hashimoto and
  Liang]{sagawa2019distributionally}
S.~Sagawa, P.~W. Koh, T.~B. Hashimoto and P.~Liang  (2020), Distributionally
  robust neural networks for group shifts: {O}n the importance of
  regularization for worst-case generalization, in {\em International
  Conference on Learning Representations}.

\bibitem[Salo and Weber(1995)Salo and Weber]{salo1995auctions}
A.~A. Salo and M.~Weber  (1995), Ambiguity aversion in first-price sealed-bid
  auctions, {\em Journal of Risk and Uncertainty} {\bf 11}(2), 123--137.

\bibitem[Sauldubois and Touzi(2024)Sauldubois and Touzi]{sauldubois2024first}
N.~Sauldubois and N.~Touzi  (2024), First order martingale model risk and
  semi-static hedging, {\em arXiv:2410.06906}.

\bibitem[Savage(2012)Savage]{savage-2012-flaw}
S.~L. Savage  (2012), {\em The Flaw of Averages: Why We Underestimate Risk in
  the Face of Uncertainty}, Wiley.

\bibitem[Savage {\em et~al.}(2006)Savage, Scholtes and
  Zweidler]{savage-2006-probability-management}
S.~L. Savage, S.~Scholtes and D.~Zweidler  (2006), Probability management, {\em
  OR/MS Today}.

\bibitem[Scarf(1958)Scarf]{Scarf:58}
H.~Scarf  (1958), A min-max solution to an inventory problem, in {\em Studies
  in Mathematical Theory of Inventory and Production} (K.~Arrow, S.~Karlin and
  H.~Scarf, eds), Stanford University Press, pp.~201--209.

\bibitem[Schildbach {\em et~al.}(2013)Schildbach, Fagiano and
  Morari]{schildbach2013randomized}
G.~Schildbach, L.~Fagiano and M.~Morari  (2013), Randomized solutions to convex
  programs with multiple chance constraints, {\em SIAM Journal on Optimization}
  {\bf 23}(4), 2479--2501.

\bibitem[Selvi {\em et~al.}(2022)Selvi, Belbasi, Haugh and
  Wiesemann]{selvi2022wasserstein}
A.~Selvi, M.~R. Belbasi, M.~Haugh and W.~Wiesemann  (2022), {W}asserstein
  logistic regression with mixed features, in {\em Advances in Neural
  Information Processing Systems}, pp.~16691--16704.

\bibitem[Shafiee and Kuhn(2024)Shafiee and Kuhn]{shafiee2024general}
S.~Shafiee and D.~Kuhn  (2024), Minimax theorems and {N}ash equilibria in
  distributionally robust optimization problems, {\em Working Paper}.

\bibitem[Shafiee {\em et~al.}(2023)Shafiee, Aolaritei, D{\"o}rfler and
  Kuhn]{shafiee2023new}
S.~Shafiee, L.~Aolaritei, F.~D{\"o}rfler and D.~Kuhn  (2023), New perspectives
  on regularization and computation in optimal transport-based distributionally
  robust optimization, {\em arXiv:2303.03900}.

\bibitem[Shafieezadeh-Abadeh {\em et~al.}(2019)Shafieezadeh-Abadeh, Kuhn and
  Mohajerin~Esfahani]{shafieezadeh2019regularization}
S.~Shafieezadeh-Abadeh, D.~Kuhn and P.~Mohajerin~Esfahani  (2019),
  Regularization via mass transportation, {\em Journal of Machine Learning
  Research} {\bf 20}(103), 1--68.

\bibitem[Shafieezadeh-Abadeh {\em et~al.}(2015)Shafieezadeh-Abadeh,
  Mohajerin~Esfahani and Kuhn]{shafieezadeh2015distributionally}
S.~Shafieezadeh-Abadeh, P.~Mohajerin~Esfahani and D.~Kuhn  (2015),
  Distributionally robust logistic regression, in {\em Advances in Neural
  Information Processing Systems}, pp.~1576--1584.

\bibitem[Shafieezadeh-Abadeh {\em et~al.}(2018)Shafieezadeh-Abadeh, Nguyen,
  Kuhn and Mohajerin~Esfahani]{shafieezadeh2018wasserstein}
S.~Shafieezadeh-Abadeh, V.~A. Nguyen, D.~Kuhn and P.~Mohajerin~Esfahani
  (2018), {W}asserstein distributionally robust {K}alman filtering, in {\em
  Advances in Neural Information Processing Systems}, pp.~8474--8483.

\bibitem[Shalev-Shwartz(2012)Shalev-Shwartz]{shalev2012online}
S.~Shalev-Shwartz  (2012), Online learning and online convex optimization, {\em
  Foundations and Trends in Machine Learning} {\bf 4}(2), 107--194.

\bibitem[Shalev-Shwartz and Ben-David(2014)Shalev-Shwartz and
  Ben-David]{shalev2014understanding}
S.~Shalev-Shwartz and S.~Ben-David  (2014), {\em Understanding Machine
  Learning: From Theory to Algorithms}, Cambridge University Press.

\bibitem[Shapiro(1989)Shapiro]{shapiro1989asymptotic}
A.~Shapiro  (1989), Asymptotic properties of statistical estimators in
  stochastic programming, {\em The Annals of Statistics} {\bf 17}(2), 841--858.

\bibitem[Shapiro(1990)Shapiro]{shapiro1990differential}
A.~Shapiro  (1990), On differential stability in stochastic programming, {\em
  Mathematical Programming} {\bf 47}(1-3), 107--116.

\bibitem[Shapiro(1991)Shapiro]{shapiro1991asymptotic}
A.~Shapiro  (1991), Asymptotic analysis of stochastic programs, {\em Annals of
  Operations Research} {\bf 30}(1), 169--186.

\bibitem[Shapiro(1993)Shapiro]{shapiro1993asymptotic}
A.~Shapiro  (1993), Asymptotic behavior of optimal solutions in stochastic
  programming, {\em Mathematics of Operations Research} {\bf 18}(4), 829--845.

\bibitem[Shapiro(2001)Shapiro]{shapiro2001duality}
A.~Shapiro  (2001), On duality theory of conic linear problems, in {\em
  Semi-Infinite Programming} (M.~{\'A}. Goberna and M.~A. L{\'o}pez, eds),
  Kluwer Academic Publishers, pp.~135--165.

\bibitem[Shapiro(2003)Shapiro]{shapiro2003monte}
A.~Shapiro  (2003), Monte {C}arlo sampling methods, in {\em Stochastic
  Programming} (A.~Ruszczy{\'n}ski and A.~Shapiro, eds), Elsevier,
  pp.~353--425.

\bibitem[Shapiro(2013)Shapiro]{ref:shapiro2013kusuoka}
A.~Shapiro  (2013), On {K}usuoka representation of law invariant risk measures,
  {\em Mathematics of Operations Research} {\bf 38}(1), 142--152.

\bibitem[Shapiro(2017)Shapiro]{shapiro2017distributionally}
A.~Shapiro  (2017), Distributionally robust stochastic programming, {\em SIAM
  Journal on Optimization} {\bf 27}(4), 2258--2275.

\bibitem[Shapiro and Kleywegt(2002)Shapiro and Kleywegt]{shapiro2002minimax}
A.~Shapiro and A.~Kleywegt  (2002), Minimax analysis of stochastic problems,
  {\em Optimization Methods and Software} {\bf 17}(3), 523--542.

\bibitem[Shapiro {\em et~al.}(2009)Shapiro, Dentcheva and
  Ruszczy{\'n}ski]{shapiro2009lectures}
A.~Shapiro, D.~Dentcheva and A.~Ruszczy{\'n}ski  (2009), {\em Lectures on
  Stochastic Programming: Modeling and Theory}, SIAM.

\bibitem[Shapiro {\em et~al.}(2023)Shapiro, Zhou and Lin]{shapiro2023bayesian}
A.~Shapiro, E.~Zhou and Y.~Lin  (2023), Bayesian distributionally robust
  optimization, {\em SIAM Journal on Optimization} {\bf 33}(2), 1279--1304.

\bibitem[Shehadeh(2023)Shehadeh]{shehadeh2023distributionally}
K.~S. Shehadeh  (2023), Distributionally robust optimization approaches for a
  stochastic mobile facility fleet sizing, routing, and scheduling problem,
  {\em Transportation Science} {\bf 57}(1), 197--229.

\bibitem[Shehadeh {\em et~al.}(2020)Shehadeh, Cohn and
  Jiang]{shehadeh2020distributionally}
K.~S. Shehadeh, A.~E. Cohn and R.~Jiang  (2020), A distributionally robust
  optimization approach for outpatient colonoscopy scheduling, {\em European
  Journal of Operational Research} {\bf 283}(2), 549--561.

\bibitem[Shen and Jiang(2023)Shen and Jiang]{shen2020chance}
H.~Shen and R.~Jiang  (2023), Chance-constrained set covering with
  {W}asserstein ambiguity, {\em Mathematical Programming} {\bf 198}(1),
  621--674.

\bibitem[Sheriff and Mohajerin~Esfahani(2024)Sheriff and
  Mohajerin~Esfahani]{sheriff2023nonlinear}
M.~R. Sheriff and P.~Mohajerin~Esfahani  (2024), Nonlinear distributionally
  robust optimization, {\em Mathematical Programming (in press)} pp.~1--60.

\bibitem[Shohat and Tamarkin(1950)Shohat and Tamarkin]{shohat1950problem}
J.~A. Shohat and J.~D. Tamarkin  (1950), {\em The Problem of Moments}, American
  Mathematical Society.

\bibitem[Sinha {\em et~al.}(2018)Sinha, Namkoong and
  Duchi]{sinha2018certifying}
A.~Sinha, H.~Namkoong and J.~Duchi  (2018), Certifying some distributional
  robustness with principled adversarial training, in {\em International
  Conference on Learning Representations}.

\bibitem[Sion(1958)Sion]{sion1958general}
M.~Sion  (1958), On general minimax theorems, {\em Pacific Journal of
  Mathematics} {\bf 8}(1), 171--176.

\bibitem[Smith and Winkler(2006)Smith and Winkler]{smith2006optimizer}
J.~E. Smith and R.~L. Winkler  (2006), The optimizer’s curse: Skepticism and
  postdecision surprise in decision analysis, {\em Management Science} {\bf
  52}(3), 311--322.

\bibitem[Soyster(1973)Soyster]{soyster1973convex}
A.~L. Soyster  (1973), Convex programming with set-inclusive constraints and
  applications to inexact linear programming, {\em Operations Research} {\bf
  21}(5), 1154--1157.

\bibitem[Srivastava {\em et~al.}(2021)Srivastava, Wang, Hanasusanto and
  Ho]{srivastava2021data}
P.~R. Srivastava, Y.~Wang, G.~A. Hanasusanto and C.~P. Ho  (2021), On
  data-driven prescriptive analytics with side information: {A} regularized
  {N}adaraya-{W}atson approach, {\em arXiv:2110.04855}.

\bibitem[Staib and Jegelka(2019)Staib and Jegelka]{staib2019distributionally}
M.~Staib and S.~Jegelka  (2019), Distributionally robust optimization and
  generalization in kernel methods, in {\em Advances in Neural Information
  Processing Systems}, pp.~9134--9144.

\bibitem[Stieltjes(1894)Stieltjes]{stieltjes1894recherches}
T.-J. Stieltjes  (1894), Recherches sur les fractions continues, {\em Annales
  de la Facult{\'e} des sciences de Toulouse pour les sciences
  math{\'e}matiques et les sciences physiques} {\bf 8}(4), 1--122.

\bibitem[Strassen(1965)Strassen]{strassen1965existence}
V.~Strassen  (1965), The existence of probability measures with given
  marginals, {\em The Annals of Mathematical Statistics} {\bf 36}(2), 423--439.

\bibitem[Strohmann and Grudic(2002)Strohmann and
  Grudic]{strohmann2002formulation}
T.~Strohmann and G.~Z. Grudic  (2002), A formulation for minimax probability
  machine regression, in {\em Advances in Neural Information Processing
  Systems}, pp.~785--792.

\bibitem[Stromberg(2015)Stromberg]{stromberg2015introduction}
K.~R. Stromberg  (2015), {\em An Introduction to Classical Real Analysis},
  American Mathematical Society.

\bibitem[Sun {\em et~al.}(2023)Sun, Xie and Witten]{sun2023distributionally}
L.~Sun, W.~Xie and T.~Witten  (2023), Distributionally robust fair transit
  resource allocation during a pandemic, {\em Transportation Science} {\bf
  57}(4), 954--978.

\bibitem[Sutter {\em et~al.}(2021)Sutter, Krause and Kuhn]{sutter2021robust}
T.~Sutter, A.~Krause and D.~Kuhn  (2021), Robust generalization despite
  distribution shift via minimum discriminating information, in {\em Advances
  in Neural Information Processing Systems}, pp.~29754--29767.

\bibitem[Sutter {\em et~al.}(2024)Sutter, Van~Parys and Kuhn]{sutter2024pareto}
T.~Sutter, B.~P. Van~Parys and D.~Kuhn  (2024), A {P}areto dominance principle
  for data-driven optimization, {\em Operations Research} {\bf 72}(5),
  1976--1999.

\bibitem[Szegedy {\em et~al.}(2014)Szegedy, Zaremba, Sutskever, Bruna, Erhan,
  Goodfellow and Fergus]{szegedy2014intriguing}
C.~Szegedy, W.~Zaremba, I.~Sutskever, J.~Bruna, D.~Erhan, I.~J. Goodfellow and
  R.~Fergus  (2014), Intriguing properties of neural networks, in {\em
  International Conference on Learning Representations}.

\bibitem[Talagrand(1996)Talagrand]{talagrand1996transportation}
M.~Talagrand  (1996), Transportation cost for {G}aussian and other product
  measures, {\em Geometric \& Functional Analysis} {\bf 6}(3), 587--600.

\bibitem[Ta{\c{s}}kesen {\em et~al.}(2024)Ta{\c{s}}kesen, Iancu,
  Ko{\c{c}}yi{\u{g}}it and Kuhn]{taskesen2024distributionally}
B.~Ta{\c{s}}kesen, D.~Iancu, {\c{C}}.~Ko{\c{c}}yi{\u{g}}it and D.~Kuhn  (2024),
  Distributionally robust linear quadratic control, in {\em Advances in Neural
  Information Processing Systems}, pp.~18613--18632.

\bibitem[Ta{\c{s}}kesen {\em et~al.}(2023{\em a})Ta{\c{s}}kesen,
  Shafieezadeh-Abadeh and Kuhn]{taskesen2021semi}
B.~Ta{\c{s}}kesen, S.~Shafieezadeh-Abadeh and D.~Kuhn  (2023{\em a}),
  Semi-discrete optimal transport: {H}ardness, regularization and numerical
  solution, {\em Mathematical Programming} {\bf 199}(1), 1033--1106.

\bibitem[Ta{\c{s}}kesen {\em et~al.}(2023{\em b})Ta{\c{s}}kesen,
  Shafieezadeh-Abadeh, Kuhn and Natarajan]{taskesen2023discrete}
B.~Ta{\c{s}}kesen, S.~Shafieezadeh-Abadeh, D.~Kuhn and K.~Natarajan  (2023{\em
  b}), Discrete optimal transport with independent marginals is \#{P}-hard,
  {\em SIAM Journal on Optimization} {\bf 33}(2), 589--614.

\bibitem[Ta{\c{s}}kesen {\em et~al.}(2021)Ta{\c{s}}kesen, Yue, Blanchet, Kuhn
  and Nguyen]{taskesen2021sequential}
B.~Ta{\c{s}}kesen, M.-C. Yue, J.~Blanchet, D.~Kuhn and V.~A. Nguyen  (2021),
  Sequential domain adaptation by synthesizing distributionally robust experts,
  in {\em International Conference on Machine Learning}, pp.~10162--10172.

\bibitem[Tchen(1980)Tchen]{tchen1980inequalities}
A.~H. Tchen  (1980), Inequalities for distributions with given marginals, {\em
  The Annals of Probability} {\bf 8}(4), 814--827.

\bibitem[Terpin {\em et~al.}(2024)Terpin, Lanzetti and
  D{\"o}rfler]{terpin2024dynamic}
A.~Terpin, N.~Lanzetti and F.~D{\"o}rfler  (2024), Dynamic programming in
  probability spaces via optimal transport, {\em SIAM Journal on Control and
  Optimization} {\bf 62}(2), 1183--1206.

\bibitem[Terpin {\em et~al.}(2022)Terpin, Lanzetti, Yardim, D\"orfler and
  Ramponi]{terpin2022trust}
A.~Terpin, N.~Lanzetti, B.~Yardim, F.~D\"orfler and G.~Ramponi  (2022), Trust
  region policy optimization with optimal transport discrepancies: {D}uality
  and algorithm for continuous actions, in {\em Advances in Neural Information
  Processing Systems}, pp.~19786--19797.

\bibitem[Tong(1980)Tong]{tong1980probability}
Y.~L. Tong  (1980), {\em Probability Inequalities in Multivariate
  Distributions}, Academic Press.

\bibitem[Tram{\`e}r {\em et~al.}(2017)Tram{\`e}r, Papernot, Goodfellow, Boneh
  and McDaniel]{tramer2017space}
F.~Tram{\`e}r, N.~Papernot, I.~Goodfellow, D.~Boneh and P.~McDaniel  (2017),
  The space of transferable adversarial examples, {\em arXiv:1704.03453}.

\bibitem[Tsanga and Shehadeha(2024)Tsanga and Shehadeha]{tsanga2024trade}
M.~Y. Tsanga and K.~S. Shehadeha  (2024), On the trade-off between
  distributional belief and ambiguity: {C}onservatism, finite-sample
  guarantees, and asymptotic properties, {\em arXiv:2410.19234}.

\bibitem[Tu {\em et~al.}(2024)Tu, Chen and Yue]{tu2024max}
K.~Tu, Z.~Chen and M.-C. Yue  (2024), A max-min-max algorithm for large-scale
  robust optimization, {\em arXiv:2404.05377}.

\bibitem[Tu {\em et~al.}(2019)Tu, Zhang and Tao]{tu2019theoretical}
Z.~Tu, J.~Zhang and D.~Tao  (2019), Theoretical analysis of adversarial
  learning: {A} minimax approach, in {\em Advances in Neural Information
  Processing Systems}, pp.~12280--12290.

\bibitem[Van Der~Vaart and Wellner(2000)Van Der~Vaart and
  Wellner]{van2000preservation}
A.~Van Der~Vaart and J.~A. Wellner  (2000), Preservation theorems for
  {G}livenko-{C}antelli and uniform {G}livenko-{C}antelli classes, in {\em High
  Dimensional Probability II} (E.~Gin{\'e}, D.~M. Mason and J.~A. Wellner,
  eds), Springer, pp.~115--133.

\bibitem[Van~der Vaart(1998)Van~der Vaart]{van1998asymptotic}
A.~W. Van~der Vaart  (1998), {\em Asymptotic Statistics}, Cambridge University
  Press.

\bibitem[van Eekelen {\em et~al.}(2022)van Eekelen, den Hertog and van
  Leeuwaarden]{van2022mad}
W.~J. van Eekelen, D.~den Hertog and J.~S. van Leeuwaarden  (2022), {MAD}
  dispersion measure makes extremal queue analysis simple, {\em INFORMS Journal
  on Computing} {\bf 34}(3), 1681--1692.

\bibitem[van Eekelen {\em et~al.}(2025)van Eekelen, Hanasusanto, Hasenbein and
  van Leeuwaarden]{van2023second}
W.~J. van Eekelen, G.~A. Hanasusanto, J.~J. Hasenbein and J.~S. van Leeuwaarden
   (2025), Second-order bounds for the {M/M/s} queue with random arrival rate,
  {\em Queueing Systems} {\bf 109}(1), 1--31.

\bibitem[Van~Leeuwaarden and Stegehuis(2021)Van~Leeuwaarden and
  Stegehuis]{van2021robust}
J.~S. Van~Leeuwaarden and C.~Stegehuis  (2021), Robust subgraph counting with
  distribution-free random graph analysis, {\em Physical Review E} {\bf
  104}(4), 044313.

\bibitem[Van~Parys(2024)Van~Parys]{vanparys2024efficient}
B.~P. Van~Parys  (2024), Efficient data-driven optimization with noisy data,
  {\em Operations Research Letters} {\bf 54}, Article 107089.

\bibitem[Van~Parys and Golrezaei(2024)Van~Parys and Golrezaei]{van2024optimal}
B.~P. Van~Parys and N.~Golrezaei  (2024), Optimal learning for structured
  bandits, {\em Management Science} {\bf 70}(6), 3951--3998.

\bibitem[Van~Parys {\em et~al.}(2016{\em a})Van~Parys, Goulart and
  Embrechts]{van2016frechet}
B.~P. Van~Parys, P.~J. Goulart and P.~Embrechts  (2016{\em a}), Fr{\'e}chet
  inequalities via convex optimization, {\em Available from Optimization
  Online}.

\bibitem[Van~Parys {\em et~al.}(2016{\em b})Van~Parys, Goulart and
  Kuhn]{van2016generalized}
B.~P. Van~Parys, P.~J. Goulart and D.~Kuhn  (2016{\em b}), Generalized {G}auss
  inequalities via semidefinite programming, {\em Mathematical Programming}
  {\bf 156}(1-2), 271--302.

\bibitem[Van~Parys {\em et~al.}(2019)Van~Parys, Goulart and
  Morari]{van2019distributionally}
B.~P. Van~Parys, P.~J. Goulart and M.~Morari  (2019), Distributionally robust
  expectation inequalities for structured distributions, {\em Mathematical
  Programming} {\bf 173}(1-2), 251--280.

\bibitem[Van~Parys {\em et~al.}(2015)Van~Parys, Kuhn, Goulart and
  Morari]{van2015distributionally}
B.~P. Van~Parys, D.~Kuhn, P.~J. Goulart and M.~Morari  (2015), Distributionally
  robust control of constrained stochastic systems, {\em IEEE Transactions on
  Automatic Control} {\bf 61}(2), 430--442.

\bibitem[Van~Parys {\em et~al.}(2021)Van~Parys, {Mohajerin Esfahani} and
  Kuhn]{van2021data}
B.~P. Van~Parys, P.~{Mohajerin Esfahani} and D.~Kuhn  (2021), From data to
  decisions: {D}istributionally robust optimization is optimal, {\em Management
  Science} {\bf 67}(6), 3387--3402.

\bibitem[Vapnik(2013)Vapnik]{vapnik2013nature}
V.~Vapnik  (2013), {\em The Nature of Statistical Learning Theory}, Springer.

\bibitem[Varadhan(1966)Varadhan]{varadhan1966asymptotic}
S.~S. Varadhan  (1966), Asymptotic probabilities and differential equations,
  {\em Communications on Pure and Applied Mathematics} {\bf 19}(3), 261--286.

\bibitem[Vershynin(2018)Vershynin]{vershynin2018high}
R.~Vershynin  (2018), {\em High-Dimensional Probability: {A}n Introduction with
  Applications in Data Science}, Cambridge University Press.

\bibitem[Villani(2003)Villani]{villani2003topics}
C.~Villani  (2003), {\em Topics in Optimal Transportation}, American
  Mathematical Society.

\bibitem[Villani(2008)Villani]{villani2008optimal}
C.~Villani  (2008), {\em Optimal Transport: Old and New}, Springer.

\bibitem[Vincent {\em et~al.}(2024)Vincent, Azizian, Malick and
  Iutzeler]{vincent2024skdro}
F.~Vincent, W.~Azizian, J.~Malick and F.~Iutzeler  (2024), $\texttt{skwdro}$:
  {A} library for {W}asserstein distributionally robust machine learning, {\em
  arXiv:2410.21231}.

\bibitem[Volpi {\em et~al.}(2018)Volpi, Namkoong, Sener, Duchi, Murino and
  Savarese]{volpi2018generalizing}
R.~Volpi, H.~Namkoong, O.~Sener, J.~Duchi, V.~Murino and S.~Savarese  (2018),
  Generalizing to unseen domains via adversarial data augmentation, in {\em
  Advances in Neural Information Processing Systems}, pp.~5339--5349.

\bibitem[Vu {\em et~al.}(2021)Vu, Tran, Yue and Nguyen]{vu2021distributionally}
H.~Vu, T.~Tran, M.-C. Yue and V.~A. Nguyen  (2021), Distributionally robust
  fair principal components via geodesic descents, in {\em International
  Conference on Learning Representations}.

\bibitem[Wainwright(2019)Wainwright]{wainwright2019high}
M.~J. Wainwright  (2019), {\em High-Dimensional Statistics: {A} Non-Asymptotic
  Viewpoint}, Cambridge University Press.

\bibitem[Wang and Wang(2011)Wang and Wang]{wang2011complete}
B.~Wang and R.~Wang  (2011), The complete mixability and convex minimization
  problems with monotone marginal densities, {\em Journal of Multivariate
  Analysis} {\bf 102}(10), 1344--1360.

\bibitem[Wang {\em et~al.}(2018)Wang, Gao, Wei, Shafie-khah, Bi and
  Catalao]{wang2018risk}
C.~Wang, R.~Gao, W.~Wei, M.~Shafie-khah, T.~Bi and J.~P. Catalao  (2018),
  Risk-based distributionally robust optimal gas-power flow with {W}asserstein
  distance, {\em IEEE Transactions on Power Systems} {\bf 34}(3), 2190--2204.

\bibitem[Wang {\em et~al.}(2024{\em a})Wang, Becker, Van~Parys and
  Stellato]{wang2022mean}
I.~Wang, C.~Becker, B.~Van~Parys and B.~Stellato  (2024{\em a}), Mean robust
  optimization, {\em Mathematical Programming (in press)} pp.~1--43.

\bibitem[Wang {\em et~al.}(2023)Wang, Becker, Van~Parys and
  Stellato]{wang2023learning}
I.~Wang, C.~Becker, B.~P. Van~Parys and B.~Stellato  (2023), Learning
  decision-focused uncertainty sets in robust optimization, {\em
  arXiv:2305.19225}.

\bibitem[Wang {\em et~al.}(2021)Wang, Gao and Xie]{wang2021sinkhorn}
J.~Wang, R.~Gao and Y.~Xie  (2021), Sinkhorn distributionally robust
  optimization, {\em arXiv:2109.11926}.

\bibitem[Wang {\em et~al.}(2024{\em b})Wang, Gao and
  Xie]{wang2024regularization}
J.~Wang, R.~Gao and Y.~Xie  (2024{\em b}), Regularization for adversarial
  robust learning, {\em arXiv:2408.09672}.

\bibitem[Wang {\em et~al.}(2013)Wang, Peng and Yang]{wang2013bounds}
R.~Wang, L.~Peng and J.~Yang  (2013), Bounds for the sum of dependent risks and
  worst value-at-risk with monotone marginal densities, {\em Finance and
  Stochastics} {\bf 17}(2), 395--417.

\bibitem[Wang(2024)Wang]{wang2024power}
S.~Wang  (2024), The power of simple menus in robust selling mechanisms, {\em
  Management Science (Forthcoming)}.

\bibitem[Wang {\em et~al.}(2020)Wang, Chen and Liu]{wang2020distributionally}
S.~Wang, Z.~Chen and T.~Liu  (2020), Distributionally robust hub location, {\em
  Transportation Science} {\bf 54}(5), 1189--1210.

\bibitem[Wang {\em et~al.}(2024{\em c})Wang, Liu and Zhang]{wang2024minimax}
S.~Wang, S.~Liu and J.~Zhang  (2024{\em c}), Minimax regret robust screening
  with moment information, {\em Manufacturing \& Service Operations Management}
  {\bf 26}(3), 992--1012.

\bibitem[Wang {\em et~al.}(2019)Wang, Ma, Bailey, Yi, Zhou and
  Gu]{wang2019convergence}
Y.~Wang, X.~Ma, J.~Bailey, J.~Yi, B.~Zhou and Q.~Gu  (2019), On the convergence
  and robustness of adversarial training, in {\em International Conference on
  Machine Learning}, pp.~6586--6595.

\bibitem[Wang {\em et~al.}(2024{\em d})Wang, Nguyen and
  Hanasusanto]{wang2024wasserstein}
Y.~Wang, V.~A. Nguyen and G.~A. Hanasusanto  (2024{\em d}), {W}asserstein
  robust classification with fairness constraints, {\em Manufacturing \&
  Service Operations Management} {\bf 26}(4), 1567--1585.

\bibitem[Wang {\em et~al.}(2024{\em e})Wang, Prasad, Hanasusanto and
  Hasenbein]{wang2024distributionally}
Y.~Wang, M.~N. Prasad, G.~A. Hanasusanto and J.~J. Hasenbein  (2024{\em e}),
  Distributionally robust observable strategic queues, {\em Stochastic Systems}
  {\bf 14}(3), 229--361.

\bibitem[Wang {\em et~al.}(2016)Wang, Glynn and Ye]{wang2016likelihood}
Z.~Wang, P.~W. Glynn and Y.~Ye  (2016), Likelihood robust optimization for
  data-driven problems, {\em Computational Management Science} {\bf 13},
  241--261.

\bibitem[Weed and Bach(2019)Weed and Bach]{weed2019sharp}
J.~Weed and F.~Bach  (2019), Sharp asymptotic and finite-sample rates of
  convergence of empirical measures in {W}asserstein distance, {\em Bernoulli}
  {\bf 25}(4A), 2620--2648.

\bibitem[Whittle(1990)Whittle]{whittle1990risk}
P.~Whittle  (1990), {\em Risk-Sensitive Optimal Control}, Wiley.

\bibitem[Wiesemann {\em et~al.}(2013)Wiesemann, Kuhn and Rustem]{WKR13:rmdps}
W.~Wiesemann, D.~Kuhn and B.~Rustem  (2013), Robust {M}arkov decision
  processes, {\em Mathematics of Operations Research} {\bf 38}(1), 153--183.

\bibitem[Wiesemann {\em et~al.}(2014)Wiesemann, Kuhn and
  Sim]{wiesemann2014distributionally}
W.~Wiesemann, D.~Kuhn and M.~Sim  (2014), Distributionally robust convex
  optimization, {\em Operations Research} {\bf 62}(6), 1358--1376.

\bibitem[Wozabal(2012)Wozabal]{wozabal2012framework}
D.~Wozabal  (2012), A framework for optimization under ambiguity, {\em Annals
  of Operations Research} {\bf 193}(1), 21--47.

\bibitem[Wozabal(2014)Wozabal]{wozabal2014robustifying}
D.~Wozabal  (2014), Robustifying convex risk measures for linear portfolios:
  {A} nonparametric approach, {\em Operations Research} {\bf 62}(6),
  1302--1315.

\bibitem[Wu {\em et~al.}(2022)Wu, Li and Mao]{wu2022generalization}
Q.~Wu, J.~Y.-M. Li and T.~Mao  (2022), On generalization and regularization via
  {W}asserstein distributionally robust optimization, {\em arXiv:2212.05716}.

\bibitem[Wu {\em et~al.}(2021)Wu, Sun, Camilleri, Eickhoff and Yu]{wu2021devil}
S.~Wu, S.~Sun, J.~A. Camilleri, S.~B. Eickhoff and R.~Yu  (2021), Better the
  devil you know than the devil you don't: Neural processing of risk and
  ambiguity, {\em NeuroImage} {\bf 236}, 118109.

\bibitem[Xie(2020)Xie]{xie2020tractable}
W.~Xie  (2020), Tractable reformulations of distributionally robust two-stage
  stochastic programs over the type-$\infty$ {W}asserstein ball, {\em
  Operations Research Letters} {\bf 48}(4), 513--523.

\bibitem[Xie(2021)Xie]{xie2019distributionally}
W.~Xie  (2021), On distributionally robust chance constrained programs with
  {W}asserstein distance, {\em Mathematical Programming} {\bf 186}(1),
  115--155.

\bibitem[Xie {\em et~al.}(2022)Xie, Ahmed and Jiang]{xie2022optimized}
W.~Xie, S.~Ahmed and R.~Jiang  (2022), Optimized {B}onferroni approximations of
  distributionally robust joint chance constraints, {\em Mathematical
  Programming} {\bf 191}(1), 79--112.

\bibitem[Xie and Ahmed(2017)Xie and Ahmed]{xie2017distributionally}
W.~Xie and S.~Ahmed  (2017), Distributionally robust chance constrained optimal
  power flow with renewables: {A} conic reformulation, {\em IEEE Transactions
  on Power Systems} {\bf 33}(2), 1860--1867.

\bibitem[Xin and Goldberg(2021)Xin and Goldberg]{xin2021time}
L.~Xin and D.~A. Goldberg  (2021), Time (in)consistency of multistage
  distributionally robust inventory models with moment constraints, {\em
  European Journal of Operational Research} {\bf 289}(3), 1127--1141.

\bibitem[Xin and Goldberg(2022)Xin and Goldberg]{xin2022distributionally}
L.~Xin and D.~A. Goldberg  (2022), Distributionally robust inventory control
  when demand is a martingale, {\em Mathematics of Operations Research} {\bf
  47}(3), 2387--2414.

\bibitem[Xu {\em et~al.}(2024)Xu, Lee, Cheng and Xie]{xu2024flow}
C.~Xu, J.~Lee, X.~Cheng and Y.~Xie  (2024), Flow-based distributionally robust
  optimization, {\em IEEE Journal on Selected Areas in Information Theory} {\bf
  5}, 62--77.

\bibitem[Xu {\em et~al.}(2009)Xu, Caramanis and Mannor]{xu2009regularization}
H.~Xu, C.~Caramanis and S.~Mannor  (2009), Robustness and regularization of
  support vector machines, {\em Journal of Machine Learning Research} {\bf
  10}(51), 1485--1510.

\bibitem[Xu {\em et~al.}(2012{\em a})Xu, Caramanis and
  Mannor]{xu2012distributional}
H.~Xu, C.~Caramanis and S.~Mannor  (2012{\em a}), A distributional
  interpretation of robust optimization, {\em Mathematics of Operations
  Research} {\bf 37}(1), 95--110.

\bibitem[Xu {\em et~al.}(2012{\em b})Xu, Caramanis and
  Mannor]{xu2012optimization}
H.~Xu, C.~Caramanis and S.~Mannor  (2012{\em b}), Optimization under
  probabilistic envelope constraints, {\em Operations Research} {\bf 60}(3),
  682--699.

\bibitem[Yakubovich(1971)Yakubovich]{yakubovich1971s}
V.~A. Yakubovich  (1971), S-procedure in nonlinear control theory, {\em Vestnik
  Leninggradskogo Universiteta (in Russian)} pp.~62--77.

\bibitem[Yang(2018)Yang]{yang2018dynamic}
I.~Yang  (2018), A dynamic game approach to distributionally robust safety
  specifications for stochastic systems, {\em Automatica} {\bf 94}, 94--101.

\bibitem[Yang(2020)Yang]{yang2020wasserstein}
I.~Yang  (2020), {W}asserstein distributionally robust stochastic control: {A}
  data-driven approach, {\em IEEE Transactions on Automatic Control} {\bf
  66}(8), 3863--3870.

\bibitem[Yang {\em et~al.}(2022)Yang, Zhang, Chen, Gao and
  Hu]{yang2022decision}
J.~Yang, L.~Zhang, N.~Chen, R.~Gao and M.~Hu  (2022), Decision-making with side
  information: A causal transport robust approach, {\em Available from
  Optimization Online}.

\bibitem[Yang and Chen(2018)Yang and Chen]{yang2018robust}
P.~Yang and B.~Chen  (2018), Robust {K}ullback-{L}eibler divergence and
  universal hypothesis testing for continuous distributions, {\em IEEE
  Transactions on Information Theory} {\bf 65}(4), 2360--2373.

\bibitem[Yang and Xu(2016)Yang and Xu]{yang2016distributionally}
W.~Yang and H.~Xu  (2016), Distributionally robust chance constraints for
  non-linear uncertainties, {\em Mathematical Programming} {\bf 155}(1-2),
  231--265.

\bibitem[Yanıkoğlu {\em et~al.}(2019)Yanıkoğlu, Gorissen and {den
  Hertog}]{YANIKOGLU2019799}
I.~Yanıkoğlu, B.~L. Gorissen and D.~{den Hertog}  (2019), A survey of
  adjustable robust optimization, {\em European Journal of Operational
  Research} {\bf 277}(3), 799--813.

\bibitem[Yu {\em et~al.}(2009)Yu, Li, Schuurmans and
  Szepesv\'{a}ri]{yu2009projection}
Y.-L. Yu, Y.~Li, D.~Schuurmans and C.~Szepesv\'{a}ri  (2009), A general
  projection property for distribution families, in {\em Advances in Neural
  Information Processing Systems}, pp.~2232--2240.

\bibitem[Yu {\em et~al.}(2022)Yu, Lin, Mazumdar and Jordan]{yu2022fast}
Y.~Yu, T.~Lin, E.~V. Mazumdar and M.~Jordan  (2022), Fast distributionally
  robust learning with variance-reduced min-max optimization, in {\em
  International Conference on Artificial Intelligence and Statistics},
  pp.~1219--1250.

\bibitem[Yue {\em et~al.}(2006)Yue, Chen and Wang]{yue2006expected}
J.~Yue, B.~Chen and M.-C. Wang  (2006), Expected value of distribution
  information for the newsvendor problem, {\em Operations Research} {\bf
  54}(6), 1128--1136.

\bibitem[Yue {\em et~al.}(2022)Yue, Kuhn and Wiesemann]{yue2020linear}
M.-C. Yue, D.~Kuhn and W.~Wiesemann  (2022), On linear optimization over
  {W}asserstein balls, {\em Mathematical Programming} {\bf 195}(1--2),
  1107--1122.

\bibitem[Zames(1966)Zames]{zames1966robust}
G.~Zames  (1966), Robust control theory, {\em Proceedings of the IEEE} {\bf
  54}(9), 1442--1451.

\bibitem[Zeitouni and Gutman(1991)Zeitouni and Gutman]{zeitouni1991universal}
O.~Zeitouni and M.~Gutman  (1991), On universal hypotheses testing via large
  deviations, {\em IEEE Transactions on Information Theory} {\bf 37}(2),
  285--290.

\bibitem[Zeng and Lam(2022)Zeng and Lam]{zeng2022generalization}
Y.~Zeng and H.~Lam  (2022), Generalization bounds with minimal dependency on
  hypothesis class via distributionally robust optimization, in {\em Advances
  in Neural Information Processing Systems}, pp.~27576--27590.

\bibitem[Zhang and Zhou(2020)Zhang and Zhou]{zhang2020theoretical}
A.~Y. Zhang and H.~H. Zhou  (2020), Theoretical and computational guarantees of
  mean field variational inference for community detection, {\em The Annals of
  Statistics} {\bf 48}(5), 2575--2598.

\bibitem[Zhang {\em et~al.}(2024{\em a})Zhang, Yang and Gao]{zhang2024optimal}
L.~Zhang, J.~Yang and R.~Gao  (2024{\em a}), Optimal robust policy for
  feature-based newsvendor, {\em Management Science} {\bf 70}(4), 2315--2329.

\bibitem[Zhang {\em et~al.}(2024{\em b})Zhang, Yang and Gao]{zhang2022simple}
L.~Zhang, J.~Yang and R.~Gao  (2024{\em b}), A short and general duality proof
  for {W}asserstein distributionally robust optimization, {\em Operations
  Research (Forthcoming)}.

\bibitem[Zhang {\em et~al.}(2018)Zhang, Jiang and Shen]{zhang2018ambiguous}
Y.~Zhang, R.~Jiang and S.~Shen  (2018), Ambiguous chance-constrained binary
  programs under mean-covariance information, {\em SIAM Journal on
  Optimization} {\bf 28}(4), 2922--2944.

\bibitem[Zhao and Guan(2018)Zhao and Guan]{zhao2018data}
C.~Zhao and Y.~Guan  (2018), Data-driven risk-averse stochastic optimization
  with {W}asserstein metric, {\em Operations Research Letters} {\bf 46}(2),
  262--267.

\bibitem[Zhao and Jiang(2017)Zhao and Jiang]{zhao2017distributionally}
C.~Zhao and R.~Jiang  (2017), Distributionally robust contingency-constrained
  unit commitment, {\em IEEE Transactions on Power Systems} {\bf 33}(1),
  94--102.

\bibitem[Zhen {\em et~al.}(2023)Zhen, Kuhn and Wiesemann]{zhen2023unification}
J.~Zhen, D.~Kuhn and W.~Wiesemann  (2023), A unified theory of robust and
  distributionally robust optimization via the primal-worst-equals-dual-best
  principle, {\em Operations Research (Forthcoming)}.

\bibitem[Zhou and Doyle(1999)Zhou and Doyle]{zhou-1999-robust-control}
K.~Zhou and J.~C. Doyle  (1999), {\em Essentials of Robust Control}, Prentice
  Hall.

\bibitem[Zhou {\em et~al.}(1996)Zhou, Doyle and Glover]{zhou1996robust}
K.~Zhou, J.~C. Doyle and K.~Glover  (1996), {\em Robust and Optimal Control},
  Prentice Hall.

\bibitem[Zhu {\em et~al.}(2022{\em a})Zhu, Jiao and
  Steinhardt]{zhu2022generalized}
B.~Zhu, J.~Jiao and J.~Steinhardt  (2022{\em a}), Generalized resilience and
  robust statistics, {\em The Annals of Statistics} {\bf 50}(4), 2256--2283.

\bibitem[Zhu {\em et~al.}(2020)Zhu, Jitkrittum, Diehl and
  Sch{\"o}lkopf]{zhu2020mmd}
J.-J. Zhu, W.~Jitkrittum, M.~Diehl and B.~Sch{\"o}lkopf  (2020), Worst-case
  risk quantification under distributional ambiguity using kernel mean
  embedding in moment problem, in {\em IEEE Conference on Decision and
  Control}, pp.~3457--3463.

\bibitem[Zhu {\em et~al.}(2021)Zhu, Jitkrittum, Diehl and
  Sch{\"o}lkopf]{zhu2021kernel}
J.-J. Zhu, W.~Jitkrittum, M.~Diehl and B.~Sch{\"o}lkopf  (2021), Kernel
  distributionally robust optimization: Generalized duality theorem and
  stochastic approximation, in {\em International Conference on Artificial
  Intelligence and Statistics}, pp.~280--288.

\bibitem[Zhu {\em et~al.}(2023)Zhu, G{\"u}rb{\"u}zbalaban and
  Ruszczy{\'n}ski]{zhu2023distributionally}
L.~Zhu, M.~G{\"u}rb{\"u}zbalaban and A.~Ruszczy{\'n}ski  (2023),
  Distributionally robust learning with weakly convex losses: {C}onvergence
  rates and finite-sample guarantees, {\em arXiv:2301.06619}.

\bibitem[Zhu {\em et~al.}(2022{\em b})Zhu, Xie, Zhang, Gao and
  Xie]{zhu2022distributionally}
S.~Zhu, L.~Xie, M.~Zhang, R.~Gao and Y.~Xie  (2022{\em b}), Distributionally
  robust weighted $k$-nearest neighbors, in {\em Advances in Neural Information
  Processing Systems}, pp.~29088--29100.

\bibitem[Zorzi(2014)Zorzi]{zorzi2014multivariate}
M.~Zorzi  (2014), Multivariate spectral estimation based on the concept of
  optimal prediction, {\em IEEE Transactions on Automatic Control} {\bf 60}(6),
  1647--1652.

\bibitem[Zorzi(2016)Zorzi]{zorzi2016robust}
M.~Zorzi  (2016), Robust {K}alman filtering under model perturbations, {\em
  IEEE Transactions on Automatic Control} {\bf 62}(6), 2902--2907.

\bibitem[Zorzi(2017{\em a})Zorzi]{zorzi2017convergence}
M.~Zorzi  (2017{\em a}), Convergence analysis of a family of robust {K}alman
  filters based on the contraction principle, {\em SIAM Journal on Control and
  Optimization} {\bf 55}(5), 3116--3131.

\bibitem[Zorzi(2017{\em b})Zorzi]{zorzi2017robustness}
M.~Zorzi  (2017{\em b}), On the robustness of the {B}ayes and {W}iener
  estimators under model uncertainty, {\em Automatica} {\bf 83}, 133--140.

\bibitem[Zuluaga and Pena(2005)Zuluaga and Pena]{zuluaga2005conic}
L.~F. Zuluaga and J.~F. Pena  (2005), A conic programming approach to
  generalized {T}chebycheff inequalities, {\em Mathematics of Operations
  Research} {\bf 30}(2), 369--388.

\bibitem[Zymler {\em et~al.}(2013{\em a})Zymler, Kuhn and
  Rustem]{zymler2013distributionally}
S.~Zymler, D.~Kuhn and B.~Rustem  (2013{\em a}), Distributionally robust joint
  chance constraints with second-order moment information, {\em Mathematical
  Programming} {\bf 137}(1-2), 167--198.

\bibitem[Zymler {\em et~al.}(2013{\em b})Zymler, Kuhn and
  Rustem]{zymler2013worst}
S.~Zymler, D.~Kuhn and B.~Rustem  (2013{\em b}), Worst-case value at risk of
  nonlinear portfolios, {\em Management Science} {\bf 59}(1), 172--188.

\end{thebibliography}
\label{lastpage}
\end{document}